\documentclass[10pt,reqno]{amsart} 
\numberwithin{equation}{section}
\usepackage{graphicx}

\usepackage[margin=1 in]{geometry}

\usepackage{xcolor}
\usepackage{enumitem}
\usepackage[labelfont=bf,font=sf]{caption}
\usepackage[font={sf,small},labelfont=md]{subfig}

\usepackage[noadjust,nosort]{cite}

\usepackage{amsmath, amsthm, amssymb, tikz, mathtools, array, mathrsfs, tensor}
\usepackage{relsize}

\usetikzlibrary{shapes.geometric}
\tikzset{cross/.style={cross out, draw=black, minimum size=2.5*(#1-\pgflinewidth), inner sep=2pt, outer sep=0.5pt},
	cross/.default={1pt}}
\usetikzlibrary{calc,arrows}
\usepackage{pgfplots}
\pgfplotsset{compat=1.15}
\usepackage{polynom}
\usepackage{esint}
\usepackage{multicol}
\usepackage{dsfont}
\newcommand{\one}{\mathds{1}}
\usepackage{framed}
\usepackage{upref}

\usepackage{hyperref}
\hypersetup{colorlinks=true,linkcolor=blue,citecolor=blue,urlcolor=blue} %hypertexnames=false

\allowdisplaybreaks


\def\eps{\varepsilon}

\newcommand{\E}{\mathbb{E}}

\newcommand{\N}{\mathbb{N}}
\renewcommand{\P}{\mathbb{P}}

\newcommand{\R}{\mathbb{R}}

\newcommand{\Z}{\mathbb{Z}}

\newcommand{\BB}{\mathcal{B}}
\newcommand{\CC}{\mathcal{C}}
\newcommand{\DD}{\mathcal{D}}

\newcommand{\GG}{\mathcal{G}}
\newcommand{\HH}{\mathcal{H}}

\newcommand{\NN}{\mathcal{N}}

\newcommand{\PP}{\mathcal{P}}

\newcommand{\TT}{\mathcal{T}}

\newcommand{\XX}{\mathcal{X}}
\newcommand{\YY}{\mathcal{Y}}

\newcommand{\AAA}{\mathscr{A}}

\newcommand{\CCC}{\mathscr{C}}

\newcommand{\EEE}{\mathscr{E}}
\newcommand{\FFF}{\mathscr{F}}
\newcommand{\GGG}{\mathscr{G}}

\newcommand{\LLL}{\mathscr{L}}
\newcommand{\MMM}{\mathscr{M}}
\newcommand{\NNN}{\mathscr{N}}

\newcommand{\PPP}{\mathscr{P}}

\newcommand{\XXX}{\mathscr{X}}
\newcommand{\YYY}{\mathscr{Y}}
\newcommand{\ZZZ}{\mathscr{Z}}

\newcommand{\ttt}{\mathfrak{t}}
\newcommand{\uuu}{\mathfrak{u}}

\newcommand{\ppp}{\mathfrak{p}}
\newcommand{\upp}{\boldsymbol{\mathfrak{p}}}
\newcommand{\vvv}{\mathfrak{v}}

\newcommand{\fff}{\mathfrak{f}}

\newcommand{\bDelta}{\boldsymbol{\Delta}}

\newcommand{\naesat}{\textsc{nae-sat}}

\newcommand{\bZ}{\textnormal{\textbf{Z}}}
\newcommand{\bN}{\textnormal{\textbf{N}}}

\newcommand{\bq}{\textnormal{\textbf{q}}}
\newcommand{\bF}{\textnormal{\textbf{F}}}

\newcommand{\bB}{\textnormal{\textbf{B}}}
\newcommand{\bH}{\textnormal{\textbf{H}}}

\newcommand{\sig}{\underline{\sigma}}

\newcommand{\uL}{\underline{\texttt{L}}}

\newcommand{\ubx}{\underline{\textbf{x}}}
\newcommand{\ux}{\underline{x}}
\newcommand{\uy}{\underline{y}}
\newcommand{\utau}{\underline{\tau}}
\newcommand{\utheta}{\underline{\theta}}
\newcommand{\uh}{\underline{h}}

\newcommand{\us}{\underline{s}}

\newcommand{\bx}{\mathbf{x}}
\newcommand{\bw}{\mathbf{w}}
\newcommand{\bp}{\mathbf{p}}
\newcommand{\bh}{\mathbf{h}}
\newcommand{\buh}{\underline{\mathbf{h}}}
\newcommand{\ba}{\mathbf{a}}
\newcommand{\bsigma}{\boldsymbol{\sigma}}
\newcommand{\bpi}{\boldsymbol{\pi}}
\newcommand{\bupi}{\underline{\boldsymbol{\pi}}}
\newcommand{\upi}{\underline{\pi}}
\newcommand{\bsig}{\underline{\boldsymbol{\sigma}}}
\newcommand{\boeta}{\boldsymbol{\eta}}
\newcommand{\butau}{\underline{\boldsymbol{\tau}}}
\newcommand{\btau}{\boldsymbol{\tau}}
\newcommand{\bxi}{\boldsymbol{\xi}}
\newcommand{\btheta}{\boldsymbol{\theta}}
\newcommand{\butheta}{\underline{\boldsymbol{\theta}}}

\newcommand{\bmu}{\boldsymbol{\mu}}
\newcommand{\bgamma}{\boldsymbol{\gamma}}

\newcommand{\bLa}{\boldsymbol{\Lambda}}
\newcommand{\bXi}{\boldsymbol{\Xi}}
\newcommand{\dbh}{\mathbf{\dot{h}}}

\newcommand{\dbB}{\mathbf{\dot{B}}}
\newcommand{\hbB}{\mathbf{\hat{B}}}
\newcommand{\bbB}{\mathbf{\bar{B}}}
\newcommand{\dbH}{\mathbf{\dot{H}}}
\newcommand{\hbH}{\mathbf{\hat{H}}}
\newcommand{\bbH}{\mathbf{\bar{H}}}
\newcommand{\dbq}{\mathbf{\dot{q}}}
\newcommand{\hbq}{\mathbf{\hat{q}}}
\newcommand{\ula}{\boldsymbol{\lambda}}

\newcommand{\lit}{\textnormal{lit}}
\newcommand{\sm}{\textnormal{sm}}
\newcommand{\sy}{\textnormal{sy}}
\newcommand{\tr}{\textnormal{tr}}

\newcommand{\ind}{\textnormal{ind}}
\newcommand{\op}{\textnormal{op}}
\newcommand{\lab}{\textnormal{lab}}

\newcommand{\qdot}{\dot{\mathbf{q}}}
\newcommand{\qhat}{\hat{\mathbf{q}}}
\newcommand{\dotq}{\dot{q}}

\newcommand{\BP}{\textnormal{BP}}

\newcommand{\uttt}{\mathfrak{u}}

\newcommand{\comp}{\Omega_{\textnormal{com}}}

\newcommand{\csig}{\underline{\sigma}^{\textnormal{com}}}
\newcommand{\csigma}{\sigma^{\textnormal{com}}}
\newcommand{\bcsig}{\underline{\boldsymbol{\sigma}}^{\textnormal{com}}}
\newcommand{\lsig}{\underline{\sigma}^{\textnormal{lab}}}
\newcommand{\lsigma}{\sigma^{\textnormal{lab}}}

\newcommand{\srr}{{{\small{\texttt{r}}}}}
\newcommand{\rr}{{{\scriptsize{\texttt{R}}}}}
\newcommand{\bb}{{{\scriptsize{\texttt{B}}}}}
\newcommand{\pppp}{{{\scriptsize{\texttt{P}}}}}
\newcommand{\gggg}{{{\scriptsize{\texttt{G}}}}}
\newcommand{\rrrr}{{{\small{\texttt{r}}}}}
\newcommand{\fs}{{\scriptsize{\texttt{S}}}}
\newcommand{\ff}{\textnormal{\small{\texttt{f}}}}
\newcommand{\tz}{\small{\texttt{z}}}
\newcommand{\mm}{{{\texttt{m}}}}
\newcommand{\fF}{\scriptsize{\texttt{F}}}
\newcommand{\bs}{{\textbf{s}}}
\newcommand{\tL}{\texttt{L}}

\newcommand{\tZ}{\bZ^{(L),\textnormal{tr}}}

\newcommand{\la}{\lambda}

\newcommand{\proj}{\textnormal{proj}_n}
\newcommand{\pj}{\textnormal{pj}}

\newcommand{\cyc}{\textnormal{cyc}}
\newcommand{\mult}{\textnormal{mult}}
\newcommand{\uni}{\textnormal{uni}}

\newcommand{\ee}{\mathfrak{E}}
\newcommand{\pee}{\prescript{}{2}{\mathfrak{E}}}
\newcommand{\kF}{\mathfrak{F}}

\newcommand{\sep}{\textnormal{sep}}
\newcommand{\ns}{\textnormal{ns}}
\newcommand{\tns}{\textnormal{s}}
\newcommand{\tsf}{\textnormal{\textsf{f}}}

\newcommand{\pcomp}{\Omega_{\textnormal{com},2}}
\newcommand{\pbDelta}{\prescript{}{2}{\boldsymbol{\Delta}}^{\textnormal{b}}}
\newcommand{\pF}{\mathfrak{F}}
\newcommand{\pHdot}{\dot{H}_\circ}
\newcommand{\pHhat}{\hat{H}_{\textnormal{fc}}}

\newcommand{\ppjw}{\mathbf{w}^{\textnormal{pj}}}

\newcommand{\su}{\textsf{u}}

\newcommand{\up}{\textsf{up}}

\newcommand{\dsigma}{\dot{\sigma}}

\newcommand{\wt}[1]{\widetilde{#1}}

\DeclareMathOperator{\Var}{Var}
\DeclareMathOperator{\Cov}{Cov}

\DeclareMathOperator{\id}{id}

\DeclareMathOperator*{\argmin}{arg\,min} 

\DeclarePairedDelimiter\ceil{\lceil}{\rceil}

\newtheorem{thm}{Theorem}[section]
\newtheorem{prop}[thm]{Proposition}
\newtheorem{cor}[thm]{Corollary}
\newtheorem{lemma}[thm]{Lemma}

\theoremstyle{definition}
\newtheorem{defn}[thm]{Definition}

\newtheorem{remark}[thm]{Remark}

\renewcommand{\thefootnote}{\fnsymbol{footnote}}

\title[1RSB of random regular NAE-SAT]{One-step replica symmetry breaking of random regular NAE-SAT I}

 % Random media, disordered materials (including liquid crystals and spin glasses)


\author{Danny Nam}
\address{\newline Department of Mathematics \newline Princeton University \newline Princeton, NJ 08544 \newline \textup{\tt dhnam@math.princeton.edu}}

\author{Allan Sly}
\address{\newline Department of Mathematics \newline Princeton University \newline Princeton, NJ 08544 \newline \textup{\tt asly@math.princeton.edu}}

\author{Youngtak Sohn}
\address{\newline Department of Mathematics \newline Massachusetts Institute of Technology \newline Cambridge, MA 02139 \newline \textup{\tt youngtak@mit.edu}}

\begin{document}
	\bibliographystyle{acm}
	
	% changes footnote labeling back to numbers
	\renewcommand{\thefootnote}{\arabic{footnote}} \setcounter{footnote}{0}

	%: ABSTRACT
	\begin{abstract}
		In a broad class of sparse random constraint satisfaction problems (\textsc{csp}), deep heuristics from statistical physics predict that there is a \textit{condensation phase transition} before the satisfiability threshold, governed by \textit{one-step replica symmetry breaking} ({\tiny{1}}\textsc{rsb}). In fact, in random regular $k$-\textsc{nae-sat}, which is one of such random \textsc{csp}s, it was  verified~\cite{ssz22} that its free energy is well-defined and the explicit value follows the \textsc{{\tiny{1}}rsb} prediction. However, for any model of sparse random \textsc{csp}, it has been unknown whether the solution space indeed \textit{condenses} on $O(1)$ clusters according to the {\tiny{1}}\textsc{rsb} prediction. In this paper, we give an affirmative answer to this question for the random regular $k$-\textsc{nae-sat} model. 
		% and show that its solution space indeed \textit{condensates} in the condensation regime, 
		Namely, we prove that with probability bounded away from zero, most of the solutions lie inside a bounded number of solution clusters whose sizes are comparable to the scale of the free energy. Furthermore, we establish that the overlap between two independently drawn solutions concentrates precisely at two values. Our proof is based on a detailed moment analysis of a spin system, which has an infinite spin space that encodes the structure of solution clusters. We believe that our method is applicable to a broad range of random \textsc{csp}s in the {\tiny{1}}\textsc{rsb} universality class.

	\end{abstract}
	
	\maketitle
	
	%\tableofcontents
	
	%: INTRODUCTION
	
	\setcounter{tocdepth}{1}
	\tableofcontents

	\section{Introduction}
	A random constraint satisfaction problem (r\textsc{csp}) is defined by a collection of variables whose configuration should satisfy a set of randomly chosen constraints. In the mathematics literature, there are $n$ variables $\underline{x}=\{x_i\}_{i=1}^n\in \mathfrak{X}^n$ taking values in a finite alphabet set $\mathfrak{X}$, and they are subject to $m\equiv \alpha n$ randomly drawn constraints. The major interest is to understand the structure of the solution space of r\textsc{csp}s as $n,m\to \infty$ while $\alpha$ is fixed. Since the early 2000's, statistical physicists developed a deep but non-rigorous theory to study these problems and conjectured that in a wide class of r\textsc{csp}s, there is a fascinating series of phase transitions as $\alpha$ varies (\cite{mpz02,kmrsz07}; cf. \cite{anp05} and  Chapter 19 of \cite{mm09} for a survey). As we detail below, the present paper focuses on investigating the solution space structure when $\alpha$ is in the \textit{condensation regime}, for a r\textsc{csp} model called the \textit{random regular} $k$-\naesat.

	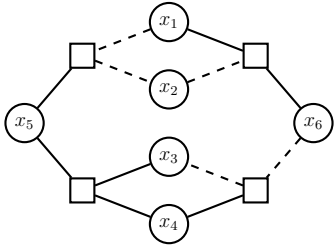
\begin{figure}
		\centering
		\begin{tikzpicture}[square/.style={regular polygon,regular polygon sides=4},thick,scale=0.64, every node/.style={transform shape}]
		\node[circle,draw, minimum size=.1cm] (x4) at  (0,0) {\Large{$x_4$}};
		\node[circle,draw, minimum size=.1cm] (x3) at  (0,1.4) {\Large{$x_3$}};
		\node[circle,draw, minimum size=.1cm] (x2) at  (0,2.8) {\Large{$x_2$}};
		\node[circle,draw, minimum size=.1cm] (x1) at  (0,4.2) {\Large{$x_1$}};
		\node[circle,draw, minimum size=.1cm] (x5) at  (-3,2.1) {\Large{$x_5$}};
		\node[circle,draw, minimum size=.1cm] (x6) at  (3,2.1) {\Large{$x_6$}};
		
		\node[rectangle,draw, minimum size=.5cm] (c1) at  (1.8,3.5) { };
		\node[rectangle,draw, minimum size=.5cm] (c2) at  (1.8,.7) { };
		\node[rectangle,draw, minimum size=.5cm] (c3) at  (-1.8,3.5) { };
		\node[rectangle,draw, minimum size=.5cm] (c4) at  (-1.8,.7) { };
		
		\draw (x1) -- (c1) -- (x6);
		\draw[dashed] (x2) -- (c1);
		\draw[dashed] (x6) -- (c2) -- (x3);
		\draw (c2) -- (x4);
		
		\draw (x5) -- (c4) -- (x3);
		\draw (x4) -- (c4);
		
		\draw[dashed] (x1) -- (c3) -- (x2);
		\draw (x5) -- (c3);

		\end{tikzpicture}
		\caption{The hypergraph illustrating  an instance of a  $2$-regular $3$-\textsc{(nae-)sat} with $6$ variables. Variables and clauses are drawn by the circular and square nodes, respectively, and the dashed edges denote the negated literals. Its \textsc{cnf} formula is given by $(\neg x_1 \vee \neg x_2 \vee x_5) \wedge (x_1 \vee \neg x_2 \vee x_6) \wedge (x_3\vee x_4 \vee x_5) \wedge (\neg x_3 \vee x_4 \vee \neg x_6 )$.} \label{fig1}
	\end{figure}

	The canonical r\textsc{csp} is random $k$-\textsc{sat}, a random Boolean \textsc{cnf} formula formed by taking the \textsc{and} of clauses, each of which is the \textsc{or} of $k$ variables or their negations.  A \textit{not-all-equal-satisfiability} (\textsc{nae-sat}) formula, has the same form as $k$-\textsc{sat} but asks that both $\underline{ x}$ an assignment of the variables and $\neg \underline{ x}$ its negation evaluate to \textsf{true} in the formula.  We call such formula $k$-\textsc{nae-sat} if the clauses appearing in the \textsc{cnf} formula have exactly $k$ literals, and it is called $d$-regular if each variable appears precisely in $d$  clauses (Figure \ref{fig1}). One can then choose a  $d$-regular $k$-\textsc{nae-sat} problem of $n$ variables uniformly at random, which gives the \textit{random d-regular k}-\textsc{nae-sat} problem, with clause density $\alpha=d/k$ (See Section \ref{sec:model} for a formal definition of the model). Compared to the $k$-\textsc{sat} problem, the \textsc{nae-sat} problem possesses extra symmetries that make it more tractable from a mathematical perspective. Nevertheless,  it  is predicted to belong to the same universality class of r\textsc{csp}s as  random $k$-\textsc{sat} and random graph coloring, and hence is expected to share the most interesting qualitative behaviors with them.
	
	Let $Z\equiv Z_n$ denote the number of solutions for a given random $d$-regular $k$-\textsc{nae-sat} instance. Physicists predict that for each fixed $\alpha$, there exists $\textsf{f}(\alpha)$ such that
	$$\frac{1}{n} \log Z \;\longrightarrow\; \textsf{f}(\alpha) \quad \textnormal{in probability}. $$
	A direct computation of the first moment $\E Z$ gives that 
	$$\E Z = 2^n \left(1-2^{-k+1} \right)^m = e^{n\textsf{f}^{\textsf{rs}}(\alpha)}, \quad \textnormal{where}\quad \textsf{f}^{\textsf{rs}}(\alpha)\equiv \log 2+ \alpha \log \left( 1-2^{-k+1}\right), $$
	(the superscript \textsf{rs} refers to the \textit{replica-symmetric} free energy) and we see that $\textsf{f}\leq \textsf{f}^{\textsf{rs}}$, by Markov's inequality. The previous works of Ding-Sly-Sun \cite{dss16} and Sly-Sun-Zhang \cite{ssz22}  established some of the physics conjectures on the description of $Z$ and $\textsf{f}$ given in \cite{zk07,kmrsz07,mrs08}, which can be summarized as follows. 
	\begin{itemize}
		\item (\cite{dss16}) For large enough $k$, there exists the \textit{satisfiability threshold} $\alpha_{\textsf{sat}}\equiv \alpha_{\textsf{sat}}(k)>0$ such that
		\begin{equation*}
		\lim_{n\to\infty} \P(Z>0)
		=
		\begin{cases}
		1 & \textnormal{ for } \alpha\in(0,\alpha_{\textsf{sat}});\\
		0 & \textnormal{ for }\alpha > \alpha_{\textsf{sat}}.
		\end{cases}
		\end{equation*}

		\item (\cite{ssz22}) For large enough $k$, there exists the \textit{condensation threshold} $\alpha_{\textsf{cond}}\equiv \alpha_{\textsf{cond}}(k)\in(0,\alpha_{\textsf{sat}})$  such that 
		\begin{equation}\label{eq:intro:freeEformula}
		\textsf{f}(\alpha)=
		\begin{cases} 
		\textsf{f}^{\textsf{rs}}(\alpha) &
		\textnormal{ for } \alpha \leq \alpha_{\textsf{cond}};\\
		\textsf{f}^{1\textsf{rsb}}(\alpha) 
		& \textnormal{ for } \alpha > \alpha_{\textsf{cond}},
		\end{cases}
		\end{equation}
		where $\textsf{f}^{1\textsf{rsb}}\equiv\textsf{f}^{1\textsf{rsb}}(\alpha)$ is the {\small{1}}\textsc{rsb} \textit{free energy}. Moreover, $\textsf{f}^{\textsf{rs}} > \textsf{f}^{\textsf{1\textsf{rsb}}}$ on $(\alpha_{\textsf{cond}},\alpha_{\textsf{sat}})$. For the explicit formula and derivation of $\textsf{f}^{1\textsf{rsb}}$, we refer to Section 1.6 of \cite{ssz22} for a concise overview.
	\end{itemize}

	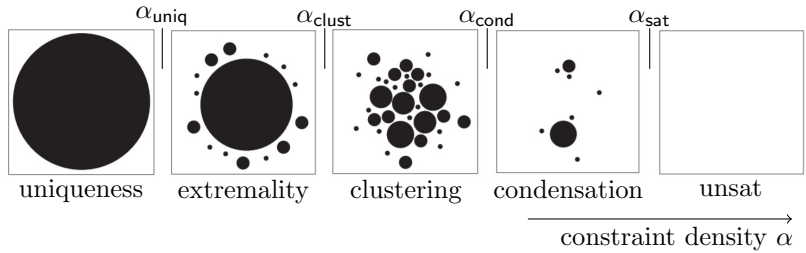
\begin{figure}
		\centering
		\begin{tikzpicture}[scale=0.9]
		\node[inner sep=0pt] (unique) at (0,0)
		{\includegraphics[width=2cm]{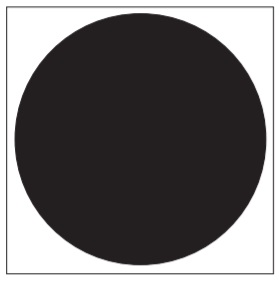}};
		
		\node[inner sep=0pt] (ext) at (2.4,0)
		{\includegraphics[width=2cm]{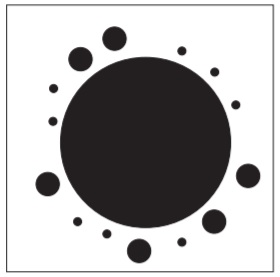}};
		
		\node[inner sep=0pt] (clust) at (4.8,0)
		{\includegraphics[width=2cm]{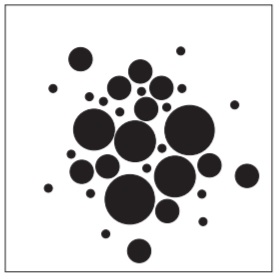}};
		
		\node[inner sep=0pt] (cond) at (7.2,0)
		{\includegraphics[width=2cm]{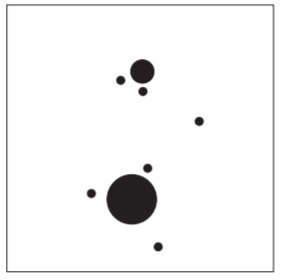}};
		
		\node[inner sep=0pt] (unsat) at (9.6,0)
		{\includegraphics[width=2cm]{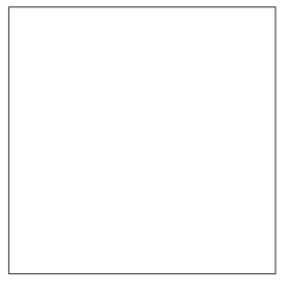}};
		
		\draw (1.2,1.1) -- (1.2,0.5);
		\draw (3.6,1.1) -- (3.6,0.5);
		\draw (6,1.1) -- (6,0.5);
		\draw (8.4,1.1) -- (8.4,0.5);
		\node[anchor=south] (un) at (1.2,1) {$\alpha_{\textsf{uniq}}$};
		\node[anchor=south] (un) at (3.6,1) {$\alpha_{\textsf{clust}}$};
		\node[anchor=south] (un) at (6,1) {$\alpha_{\textsf{cond}}$};
		\node[anchor=south] (un) at (8.4,1) {$\alpha_{\textsf{sat}}$};
		\node[anchor=north] (asd) at (0,-1) {uniqueness};
		\node[anchor=north] (asd) at (2.4,-1) {extremality};
		\node[anchor=north] (asd) at (4.8,-1) {clustering};
		\node[anchor=north] (asd) at (7.2,-1) {condensation};
		\node[anchor=north] (asd) at (9.6,-1) {unsat};
		
		\draw[->] (6.6,-1.72)--(10.5,-1.72);
		\node[anchor=north] (aaas) at (8.8,-1.7) {constraint density $\alpha$};
		\end{tikzpicture}
		\caption{\textit{Figure adapted from} \cite{kmrsz07,dss15ksat}. A pictorial description of the conjectured phase diagram of random regular $k$-\textsc{nae-sat}. In the condensation regime $(\alpha_{\textsf{cond}},\alpha_{\textsf{sat}})$, there remains a bounded number of clusters containing most of the solutions.}\label{fig:pt}
	\end{figure}

	Furthermore, the physics predictions say that the solution space of the random regular $k$-\textsc{nae-sat} is  \textit{condensed} when $\alpha \in(\alpha_{\textsf{cond}},\alpha_{\textsf{sat}})$ into a finite number of \textit{clusters} (Figure \ref{fig:pt}). Here, \textit{clusters} are defined by the connected components of the solution space, where we connect two solutions if they differ in one variable (see Remark \ref{rmk:cluster:definition} for slightly different definition of clusters). Our first main result verifies the prediction for all $k\geq k_0$, where $k_0$ is a universal constant. It is the first to  provide a rigorous \textit{cluster-level} description of the solution space of a sparse r\textsc{csp} in the condensation regime. 
	
	\begin{thm}\label{thm1}
		Let $k\geq k_0$ and $\alpha \in (\alpha_{\textsf{cond}}, \alpha_{\textsf{sat}})$ such that $d\equiv \alpha k\in \mathbb{N}$. For all $\eps >0$, there exists a constant $C\equiv C(\eps,\alpha ,k)>0$ such that with probability at least $1-\eps$, the random $d$-regular $k$-\textsc{nae-sat} instance satisfies the following:
		\begin{enumerate}
			\item[\textnormal{(a)}] The number of solutions is no greater than $\exp (n\textnormal{\textsf{f}}^{1\textnormal{\textsf{rsb}}}(\alpha) -c^\star\log n +C)$, where $\textnormal{\textsf{f}}^{1\textnormal{\textsf{rsb}}}$ is the {\textnormal{{\small{1}}}}\textsc{rsb}  free energy and $c^\star\equiv c^\star(\alpha,k)>0$ is a fixed constant (See Definition \ref{def:la:s:c:star}).
		\end{enumerate}
		Moreover, there exists $\delta\equiv \delta(\alpha,k)>0$ and a constant $C^\prime \equiv C^\prime(M,\alpha,k)\in \R$ for $M \in \N$ such that for every $M\in \N$, with probability at least $\delta$, the random $d$-regular $k$-\textsc{nae-sat} instance satisfies the following:
		\begin{enumerate}[resume]
		    \item[\textnormal{(b)}] For all $\eta>0$, there exists $K\equiv K(\eta,\alpha,k)\in \mathbb{N}$ such that the $K$  largest solution clusters, $\CC_1,\ldots,\CC_K$, occupy at least $1-\eta$ fraction of the solution space;
			
			\item[\textnormal{(c)}] There are at least $\exp(n\textnormal{\textsf{f}}_{1\textnormal{\textsf{rsb}}}(\alpha) -c^\star\log n -C^\prime)$ many solutions in each of $\CC_1,\ldots,\CC_{M}$, the $M$ largest clusters. 
		\end{enumerate} 
	\end{thm}
	\begin{remark}\label{rem:k:adjusted}
	Throughout the paper, we allow for the constant $k_0$ to be adjusted. In particular we take $k_0$ to be a large absolute constant, so that the results of \cite{dss16} and \cite{ssz22} hold. Moreover, it was shown in \cite[Proposition 1.4]{ssz22} that $(\alpha_{\textsf{cond}}, \alpha_{\textsf{sat}})$ is a subset of $[\alpha_{\textsf{lbd}}, \alpha_{\textsf{ubd}}]$, where $\alpha_{\textsf{lbd}}\equiv (2^{k-1}-2)\log 2$ and $\alpha_{\textsf{ubd}}\equiv 2^{k-1}\log 2$, so we will restrict our attention to $\alpha \in  (\alpha_{\textsf{cond}}, \alpha_{\textsf{sat}})\subset [\alpha_{\textsf{lbd}}, \alpha_{\textsf{ubd}}]$.
	\end{remark}
	\begin{remark}\label{rem:whp}
	In the companion paper \cite{nss2}, we show that $(b)$ and $(c)$ of Theorem \ref{thm1} actually holds with probability close to $1$.
	\end{remark}
	\begin{remark}\label{rmk:cluster:definition}
	Although our definition of a cluster in Theorem \ref{thm1} is a connected component of the solution space, our proof shows that Theorem \ref{thm1} also holds for a slightly different definition of clusters, where we merge the connected components if they differ in a small, say $\log n$, number of variables. This follows from the control of the number of the pairs of clusters in the so-called \textit{near-identical regime} (see Lemma \ref{lem:identical regime decay estim}).
	\end{remark}
	In what follows, we briefly discuss the principles underlying the condensation predictions which are helpful to understand the main theorem. As shown in Figure \ref{fig:pt}, the solution space of the random regular $k$-\textsc{nae-sat} is predicted to be \textit{clustered} into exponentially many clusters with each of them occupying an exponentially small mass when $\alpha \in (\alpha_{\textsf{clust}},\alpha_{\textsf{cond}})$. As $\alpha$ gets larger than $\alpha_{\textsf{cond}}(>\alpha_{\textsf{clust}})$ (the \textit{condensation regime}), the solution space becomes \textit{condensed}, which causes the failure of the first moment analysis as seen in~\eqref{eq:intro:freeEformula}.  When $\alpha\in(\alpha_{\textsf{cond}},\alpha_{\textsf{sat}})$, the number of clusters that contribute the most to $\E Z$ is exponentially small in $n$, meaning that those clusters are no longer present in a typical instance of the r\textsc{csp}. Thus, the leading order of $Z$ is given by the largest clusters that can typically exist (which are thus smaller than the main contributors to $\E Z$), and the number of such clusters is believed to be bounded. Moreover, it is expected that the sizes of those clusters are comparable to  the {\small{1}}\textsc{rsb} free energy.

	Theorem \ref{thm1} verifies that the solution space indeed becomes \textit{condensed} in the condensation regime, while the previous works \cite{bchrv16,ssz22} obtained the evidence of the condensation phenomenon in the level of free energy. Furthermore, it is believed that the nature of the condensation is governed by \textit{one-step replica symmetry breaking}, which we detail in the following subsection. 
	
	Compared to the previous related works \cite{dss15ksat,dss16maxis,dss16,ssz22} in similar settings, we interpret the partition function from a different perspective in order to acquire information on the \textit{number of clusters of particular sizes}. Our approach requires a detailed analysis of an auxiliary spin system with an \textit{infinite} spin space, and one of our major accomplishments is to develop new ideas and generalize existing theories to understand such a system.

	\subsection{One-step replica symmetry breaking}\label{subsec:intro:rsb}
	
	In the condensation regime $\alpha\in(\alpha_{\textsf{cond}},\alpha_{\textsf{sat}})$, the random regular $k$-\textsc{nae-sat} model is believed to possess a single layer of hierarchy of clusters in the solution space. The prediction is that the solutions are \textit{well-connected} inside each cluster so that no additional hierarchical structure occurs in the cluster. Such behaviors are conjectured in various other models such as random graph coloring and random $k$-\textsc{sat} \cite{kmrsz07}. However, we remark that there are also other models such as maximum independent set (or high-fugacity hard-core model) in random graphs with small degrees \cite{bkzz13} and Sherrington-Kirkpatrick  model (on the complete graph) \cite{SK75, t06}, which are expected (or proven \cite{acz20}) to undergo \textit{full} \textsc{rsb}, meaning that there are infinitely many levels of hierarchy inside the solution clusters.
	
	One way to characterize {\small{1}}\textsc{rsb} is to look at the \textit{overlap} between two uniformly and independently drawn solutions. In the condensation regime, since there are a bounded number of clusters containing most of the mass, with a non-trivial probability the two solutions belong to the same cluster. According to the description of {\small{1}}\textsc{rsb}, there is no additional structure inside each cluster, and hence the Hamming distance between two independently selected solutions is expected to be concentrated  precisely at  \textit{two values}, depending on whether they came from the same cluster or not.
	
	Our second result verifies that this is indeed the case for the random regular $k$-\textsc{nae-sat} for large enough $k$, which establishes for the first time a rigorous characterization of {\small{1}}\textsc{rsb} in sparse r\textsc{csp}s.

	\begin{defn}
		For $\underline{x}^1,\underline{x}^2 \in \{0,1\}^n$, let $\underline{y}^i = 2\underline{x}^i - \textbf{1}$. The overlap $\rho(\underline{x}^1,\underline{x}^2)$ is defined by
		\begin{equation*}
		\rho(\underline{x}^1,\underline{x}^2) \equiv \frac{1}{n} \uy^1 \cdot \uy^2 = \frac{1}{n} \sum_{i=1}^n y^1_i y^2_i.
		\end{equation*}
		%In words, the overlap is the normalized difference between the number of variables with the same value and the number of those with different values.
	\end{defn}

		\begin{thm}\label{thm2}
		Let $k\geq k_0$, $\alpha \in (\alpha_{\textsf{cond}}, \alpha_{\textsf{sat}})$ such that $d\equiv\alpha k\in \mathbb{N}$. There exists an explicit constant $p^\star\equiv p^\star(\alpha,k)\in(0,1)$ (see Definition \ref{def:pstar}) such that the following holds for random $d$-regular $k$-\textsc{nae-sat} instance $\GGG$ with probability at least $\delta\equiv \delta(\alpha,k)$ (for the formal definition of $\GGG$, see Section \ref{sec:model}):
		 let $\underbar{x}^1, \underbar{x}^2\in \{0,1\}^n$ be independent, uniformly chosen satisfying assignments given $\GGG$.  Then, for some constant $C\equiv C(\alpha,k)$, the absolute value $\rho_{\textnormal{abs}} \equiv |\rho|$ of their overlap $\rho \equiv \rho (\underline{x}^1,\underline{x}^2)$ satisfies
		\begin{enumerate}
			\item[\textnormal{(a)}] $\mathbb{P}(\rho_{\textnormal{abs}}\leq n^{-1/3} |\GGG ) \geq \delta$;
			
			\item[\textnormal{(b)}] $\mathbb{P}( \big|\rho_{\textnormal{abs}}-p^\star \big| \leq n^{-1/3} |\GGG )\geq \delta$;
			
			\item[\textnormal{(c)}] $\mathbb{P}( \min\{ \rho_{\textnormal{abs}}, |\rho_{\textnormal{abs}}-p^\star|\} \geq n^{-1/3}|\GGG )\leq C n^{-1/3}$.
		\end{enumerate}
		
	\end{thm}

	We remark that in (b), $\rho$ can take either $p^\star+ O(n^{-1/3})$ or $-p^\star+O(n^{-1/3})$ with asymptotically equal probability as $n\to\infty$. This is due to the symmetric nature of the \textsc{nae-sat}, where $-\ux$ is also a solution if $\ux$ is. Thus, the clusters of solutions come in pairs as well: if $\mathcal{C}$ is a cluster, then so is $-\mathcal{C}:= \{-\ux: \ux \in \mathcal{C}\} .$ We also remark that the probabilities $\delta$ in Theorems \ref{thm1} and \ref{thm2} are not necessarily the same. We need different technical properties to deduce the conclusions of the two theorems.

	%Physicists have a refined description on the limiting distribution of the cluster sizes in the condensation regime \cite{kmrsz07}.
	According to the physics predictions \cite{kmrsz07}, the relative sizes of the largest clusters of the r\textsc{csp}s with {\small{1}}\textsc{rsb} in the condensation regime are conjectured to converge to a Poisson-Dirichlet process. Here, the conjectured description in \cite{kmrsz07} corresponds to the definition of clusters in Remark \ref{rmk:cluster:definition}, where we merge the connected components if they differ in a small number of variables. Although we provide a cluster-level illustration of the solution space and show that it follows the {\small{1}}\textsc{rsb} prediction, our method is not strong enough to study the limiting distributions of the cluster sizes, and the conjecture is left as an important open problem in the field. 
	%It also infers why the error probabilities in Theorems~\ref{thm1} and \ref{thm2} should be $1-\eps$ rather than $1-o(1)$, since in the above Poisson-Dirichlet process, the largest cluster can make and arbitrarily large or small fraction of the mass with positive probability.

	\subsection{Related works}\label{subsec:intro:related}
	
	Many of the earlier works on r\textsc{csp}s focused on  determining their satisfiability thresholds and verifying the sharpness of \textsc{sat}-\textsc{unsat} transitions. For r\textsc{csp} models that are known not to exhibit \textsc{rsb}, such goals were established. These models include random 2-\textsc{sat} \cite{cr92,bbckw01},  random 1-\textsc{in}-$k$-\textsc{sat} \cite{acim01}, $k$-\textsc{xor-sat} \cite{dm02,dgmm10,ps16}, and random linear equations \cite{acgm17}. On the other hand, for the models which are predicted to display the condensation phenomenon, intensive studies have been conducted to estimate their satisfiability threshold, as shown in \cite{kkks98,ap04,cp16} (random $k$-\textsc{sat}), \cite{am06,cz12,cp12} (random $k$-\textsc{nae-sat}), and \cite{an05,c13,cv13,ceh16} (random graph coloring). 
	
	The satisfiability thresholds for r\textsc{csp}s with \textsc{rsb} have been rigorously determined in several models (random regular $k$-\textsc{nae-sat} \cite{dss16},  maximum independent set \cite{dss16maxis}, random regular $k$-\textsc{sat} \cite{cp16} and  random $k$-\textsc{sat} \cite{dss15ksat}), where they looked at the number of \textit{clusters} instead of the number of solutions and carried out a demanding second moment method. Although determining the location of colorability threshold is left open, the condensation threshold for random graph coloring at the level of the free energy, similar to \eqref{eq:intro:freeEformula}, was settled in \cite{bchrv16}. They conducted a technically challenging analysis based on a clever ``planting" technique, where the results were further generalized to other models in \cite{ckpz18}. Similarly,~\cite{bc16} identified the condensation threshold for random regular $k$-\textsc{sat}, where each variable appears $d/2$-times positive and $d/2$-times negative. 
 
	Further theory was developed in \cite{ssz22} to establish the {\small{1}}\textsc{rsb} free energy prediction for random regular $k$-\textsc{nae-sat} in the condensation regime. %Analogous result was settled in random regular $k$-\textsc{sat} \cite{cw18}, and they were also able to figure out the limiting distribution of the number of solutions.   %%only below condensation
	However, \cite{ssz22} was not able to present a cluster-level description of an r\textsc{csp} instance, nor to explain the nature of the condensation phenomenon. Our main contribution is to illustrate the solution space of the random regular \textsc{nae-sat} instance at the cluster-level and to verify that its condensation is governed by {\small{1}}\textsc{rsb}.
	
	%Besides their intrinsic interest in mathematical perspective, understanding the phase transitions of r\textsc{csp}s has drawn much interest of computer science community. 
    Although we are not aware of previous rigorous analysis of r\textsc{csp} in the condensation regime at cluster-level, the {\small{1}}\textsc{rsb} prediction was established in \cite{Subag17} for the pure $p$-spin spherical spin glass model \cite{CS92}, where \cite{Subag17} verified that the Gibbs measure in low temperature is split into spherical `bands' playing the role of `pure states', which is the analog of clusters in r\textsc{csp}. Similar to Theorem \ref{thm1}, \cite{Subag17} showed an explicit logarithmic correction term for the free energy (see Theorem 2 therein). For the same model, \cite{SZ17} established the limiting distribution of the ground-state. Additionally, \cite{SZ17} showed that the extremal point process of critical points at zero temperature converges in distribution to a Poisson point process, which is analogous to describing the joint law of the largest $O(1)$ clusters in the context of \textsc{nae-sat}. For certain mixed $p$-spin spherical spin glass models, the existence of {\small{2}}\textsc{rsb} at zero temperature was established in \cite{AZ19}.
      
	Lastly, the recent work \cite{bsz19} studied the random $k$-\textsc{max-nae-sat} beyond $\alpha_{\textsf{sat}}$ and verified that the {\small{1}}\textsc{rsb} description breaks down before $\alpha \asymp k^{-3}4^k$. Indeed, the \textit{Gardner transition} from {\small{1}}\textsc{rsb} to \textsc{frsb} is expected at $\alpha_{\textsf{Ga}}\asymp k^{-3}4^k >\alpha_{\textsf{sat}}$ \cite{mr03,kpw04}, and \cite{bsz19} provides evidence of this phenomenon.

	\subsection{Heuristic description of condensation}\label{subsec:intro:cond}
	
	We   briefly overview  what happens in an r\textsc{csp} as the clause density $\alpha=d/k$ varies, as well as a heuristic illustration of condensation. 
	
	Let us denote $0\equiv \textsf{true}$ and $1\equiv\textsf{false}$. When $\alpha$ is fairly small, most of the solutions lie inside a single well-connected cluster. %That is, we can move from one solution to another by flipping  one  variable at once, in such a way that the configurations obtained in the intermediate steps all belong to the solution space. 
	As $\alpha $ becomes larger than $\alpha_{\textsf{clust}}$, the \textit{clustering threshold}, the solution space becomes shattered into exponentially many clusters, each containing exponentially many solutions yet exponentially small compared to the whole solution space. In this regime,  define $\Sigma(s)\equiv\Sigma(s;\alpha)$, the \textit{cluster complexity function}, as
	$$\exp (n\Sigma(s)) \equiv \textnormal{expected number of clusters of size }e^{ns}. $$
	Indeed, the number of size--$e^{ns}$ clusters is believed to concentrate around its mean $e^{n\Sigma(s)}$. Thus, the expected number of solutions can be written as
	$$\E Z = \sum_{s} \exp (n \{s+\Sigma(s)\})   
	\doteq \exp (n\cdot \max \{s+\Sigma(s) : s\geq0 \}  ),$$
	where $\doteq$ denotes the equality up to the leading exponential order. The function $\Sigma(s;\alpha)$ is believed to be smooth and concave in $s$ for each fixed $\alpha$, and indeed physicists predict an explicit formula for $\Sigma(s)$ via the {\small{1}}\textsc{rsb} \textit{cavity method} \cite{kmrsz07,mm09}. Hence, if this is the case, we have that 
	$$\E Z \doteq \exp (n\{s_1 +\Sigma (s_1) \} ),$$
	where $s_1\equiv s_1(\alpha)>0$ is the unique solution of $\Sigma'(s_1;\alpha) = -1$. However, if $\Sigma(s_1;\alpha)<0$, meaning that the expected number of size--$e^{ns_1}$ clusters are exponentially small, those clusters are unlikely to exist in a typical instance and hence the main contribution to $Z$ is given by 
	$$Z \doteq \exp(n\{s_\star + \Sigma(s_\star) \} ), $$
	where $s_\star$ is defined as 
	\begin{equation}\label{eq:def:sstar}
	s_\star \equiv s_\star(\alpha) \equiv \arg \max_s \{s+\Sigma(s): \Sigma(s)\geq 0  \}.
	\end{equation}

	\begin{figure}
		\begin{tikzpicture}[scale=1.1]  
		\begin{axis}[
		% width and height if axis, adjust to your liking
		width=7.2cm,
		height=6cm,
		xtick=\empty, % remove all ticks from x-axis
		ytick=\empty, % ditto for y-axis
		xlabel={$s$},
		ylabel={},
		xmin=0, xmax=5.3,
		ymin=-2.4, ymax = 2.4,
		axis lines=center, % default is to make a box around the axis
		domain=0.2:5.2,
		samples=100]
		\addplot[blue] {-.18*(x-1.5)^2+0.7 };
		\addplot [red] {-.2*(x-2)^2-0.2};
		\addplot [green] {-.16*(x-.8)^2+1.9};
		\addplot[gray] coordinates{ (3.277778-.2,+.2+1-0.688889)  (5.377778, -1.1-0.6888889)};
		\addplot[gray] coordinates {
			(3.925-1.1, 1.1+ 0.3375)  (3.925+1.15,-1.15+ 0.3375)
		};
		\addplot[gray] coordinates {
			(2.5-1.2+2,1.2-1.45)  (2+2.5+1.2,-1.2-1.45)
		};
		\node[] at (axis cs: 0.3,2.1) {\tiny{(A)}};
		
		\node[] at (axis cs: 0.3,0.75) {\tiny{(B)}};
		
		\node[] at (axis cs: 0.3,-.5) {\tiny{(C)}};
		\end{axis}

		\end{tikzpicture}
		\hspace{4mm}
		\begin{tikzpicture}[scale=1.1]  
		\begin{axis}[
		% width and height if axis, adjust to your liking
		width=7.2cm,
		height=6cm,
		xtick=\empty, % remove all ticks from x-axis
		ytick=\empty, % ditto for y-axis
		xlabel={$s$},
		ylabel={},
		xmin=0, xmax=5.3,
		ymin=-2.3, ymax = 2.3,
		axis lines=center, % default is to make a box around the axis
		domain=0.2:5.2,
		samples=100]
		\addplot[blue] {-.18*(x-1.5)^2+0.7 };
		%	\addplot [red] {-.2*(x-2)^2-0.2};
		%	\addplot [green] {-.16*(x-.8)^2+1.9};
		\addplot[gray] coordinates{ (3.277778-.2,+.2+1-0.688889)  (5.377778, -1.1-0.6888889)};
		%	\addplot[gray] coordinates {
		%	(3.925-1.2, 1.2+ 0.3375)  (3.925+1.15,-1.15+ 0.3375)
		%	};
		%\addplot[gray] coordinates {
		%	(2.5-1.2+2,1.2-1.45)  (2+2.5+1.2,-1.2-1.45)
		%	};
		\node[] at (axis cs: 0.3,0.75){\tiny{(B)}} ;
		\node[] at (axis cs: 4.277778,0.25) {{$s_1$}};
		\addplot[densely dashed] coordinates{
			(4.277778,0) (4.277778,-0.688889)};
		
		\addplot[solid] coordinates{(4.277778,0) (4.277778,0.08) };
		\addplot[solid] coordinates{(3.472027,0) (3.472027,0.08) };
		\node[] at (axis cs: 3.472027,-.25){$s_\star$};
		
		\end{axis}

		\end{tikzpicture}
		\caption{A description of $\Sigma(s;\alpha)$ in $s$ for different values of $\alpha$. In the left, the curves correspond to the different values of $\alpha$, with (A) $\alpha\in(\alpha_{\textsf{clust}},\alpha_{\textsf{cond}})$, (B)  $\alpha \in (\alpha_{\textsf{cond}}, \alpha_{\textsf{sat}})$, and (C)  $\alpha>\alpha_{\textsf{sat}}$, with the gray lines depicting the locations of $s_1$. In the right, curve (B) is shown with the values $s_1$ and $s_\star$.}\label{fig:cond}
	\end{figure}
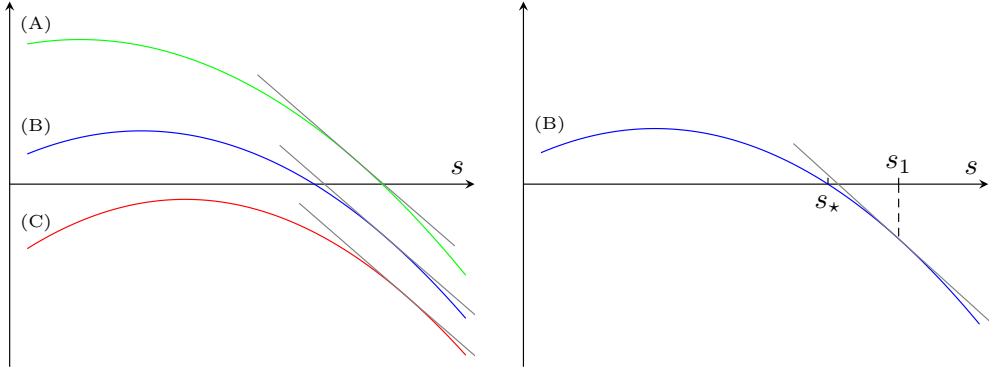

	\noindent This is the regime where the \textit{condensation phenomenon} occurs, and hence the \textit{condensation threshold} $\alpha_{\textsf{cond}}$ is defined by
	$$\alpha_{\textsf{cond}} \equiv \max\{\alpha: \Sigma(s_1(\alpha);\alpha) \geq 0 \} = \max\{\alpha: s_\star(\alpha) \geq s_1(\alpha) \}. $$
	
	For larger $\alpha$, the problem becomes unsatisfiable ($Z=0$) \textsf{whp} beyond the \textit{satisfiability threshold} $\alpha_{\textsf{sat}}$ given by
	$$\alpha_{\textsf{sat}}\equiv \min\{\alpha : \Sigma(s;\alpha)\leq 0 \textnormal{ for all } s \}. $$
	An illustration of the above discussion is given in Figure \ref{fig:cond}. Note that when $\alpha \in ( \alpha_{\textsf{cond}}, \alpha_{\textsf{sat}})$, $\Sigma(s_\star(\alpha);\alpha) =0$, which suggests that the primary contribution to $Z$ should come from a bounded number of clusters of size roughly $e^{ns_\star}$, whereas if $\alpha < \alpha_{\textsf{cond}}$ the leading term consists of the clusters of size roughly $e^{ns_1}$ whose numbers are exponentially large. Indeed, in the latter case $Z$ becomes concentrated around $\E Z$ \cite{am06,cz12,cp12}. As $k\to\infty$, asymptotic values of the thresholds  are known to be 
	$$\alpha_{\textsf{cond}} = \left(2^{k-1}-1\right)\log 2 + o_k(1), \quad \alpha_{\textsf{sat}} = \left(2^{k-1} - \frac{1}{2} - \frac{1}{4\log 2} \right) \log 2 +o_k(1).$$
	The known upper bound for $\alpha_{\textsf{clust}}$ \cite{mr13} tells us that it is relatively much smaller than $\alpha_{\textsf{cond}}$ and $\alpha_{\textsf{sat}}$ if $k$ is large.  Moreover, $\alpha_{\textsf{clust}}$ is believed to coincide with the \textit{reconstruction threshold}. We refer the readers  to \cite{gm07,kmrsz07,mrt11, BS20} for further information.

	\subsection{Tilted cluster partition function and encoding clusters}\label{subsec:intro:encoding}
	
	The main object of study in the present paper shares the same spirit as \cite{ssz22}, and its derivation is based on the ideas discussed in Section \ref{subsec:intro:cond}. We consider the \textit{tilted cluster partition function} $\overline{\bZ}_\lambda$, defined as
	\begin{equation}\label{eq:def:bZbar:elementary}
	\overline{\bZ}_\lambda \equiv \sum_{\Upsilon} |\Upsilon|^\lambda,
	\end{equation} 
	where the sum is taken over all clusters $\Upsilon$.  If we  compute $\E \overline{\bZ}_\lambda$ for $\lambda\equiv\lambda(\alpha) \equiv - \Sigma'(s_\star;\alpha)$ (with $s_\star$ as in \eqref{eq:def:sstar}), then we see that the main contribution comes from the clusters of size $e^{ns_\star}$, following the same reasoning as Section \ref{subsec:intro:cond}. Thus, we expect to have $\overline{\bZ}_\lambda \doteq \E \overline{\bZ}_\lambda$.
% 	, and indeed \cite{ssz22} carried out challenging  moment computations in a similar setting to obtain the {\small{1}}\textsc{rsb} free energy $\textsf{f}^{\textsf{1rsb}}$ for random regular $k$-\textsc{nae-sat}.
	
	The next objective is to obtain a combinatorial representation of a cluster. We follow the \textit{coarsening algorithm}, which is an inductive process starting from  a solution $\underline{x}$ that sets a variable in $\underline{x}$ to be \texttt{f} (free) one by one, if no clause is violated when the variable is flipped (that is, $0\to 1$ or $1\to 0$). We will see in Section \ref{sec:model} that the resulting \textit{frozen configuration} $\underline{y}\equiv\underline{y}(\underline{x})
	\in \{0,1,\texttt{f}\}^n$ obtained by such a procedure serves as a good representation for a cluster.
	
	To study the size of a cluster, we adapt the framework from \cite{ssz22}  to count the number of ways to assign $0/1$-values to free variables in a frozen configuration, which we detail as follows. In the regime of our interest, an important observation is that most of the variables in a solution $\underline{x}$ are \textit{frozen} (so that those variables cannot be flipped in the solution space), while a small constant fraction of them are \textit{free}. Thus, in a frozen configuration $\underline{y}
	\in \{0,1,\texttt{f}\}^n$, the connected structure  among the free variables (and their neighboring clauses) would mostly be trees that are not too large. Heuristically, they can be thought of subcritical branching processes, so the maximal connected free component will have size $O(\log n)$. \cite{ssz22} utilized the idea of \textit{belief propagation} from statistical physics to effectively count the number of \textsc{nae-sat} assignments on a given tree of free variables. These methods will be reviewed in detail in Section \ref{sec:model}. 
	
	The previous work \cite{ssz22} studied the \textit{truncated} partition function $\overline{\bZ}_{\lambda,L}$, which only counts the contributions from the frozen configurations whose free components are \textit{trees} of size at most some finite threshold~$L$.  Again based on the branching process heuristics, there  is always a constant probability for a subcritical branching process to be larger than $L$, and hence we may expect that
	$$\overline{\bZ}_{\lambda,L} \doteq e^{-\delta n} \overline{\bZ}_\lambda, $$
	where $\delta(L) \to 0$ as $L$ tends to infinity. Thus, they investigated the moments of $\overline{\bZ}_{\lambda,L}$ and let $L\to\infty$ to deduce the conclusion on the free energy of the original model. Imposing the finite-size truncation played a crucial role in their work, since it makes the space of \textit{free trees} to be finite so that some of the important methods from the earlier works \cite{dss15ksat,dss16maxis,dss16} are applicable without significant changes. However, to obtain Theorem \ref{thm1}, working with the truncated model is insufficient, since we cannot afford the cost of $e^{-\delta n}$ for any small $\delta>0$. In the following subsection, we describe a brief overview of the ideas to overcome such difficulties and give an outline of the proof.

	\subsection{Proof ideas}\label{subsec:intro:pfapproach} 
	
	The  major difficulties in  understanding the solution space in the cluster-level can be summarized as follows.
	\begin{enumerate}
		\item In addition to investigating $\overline{\bZ}_\lambda$, we need to study the contributions from clusters of sizes in a constant window $[e^{ns},e^{ns+1})$:
		\begin{equation}\label{eq:def:bZsbar:elementary}
		    \overline{\bZ}_{\lambda,s} := \sum_{\Upsilon} |\Upsilon|^\lambda\, \one{\{|\Upsilon| \in [e^{ns}, e^{ns+1}) \}}.
		\end{equation}  
		
		\item As mentioned above, it is required to work with the full space of \textit{free trees} which is infinite.
	\end{enumerate}

	The proof consists of two major parts. We first compute the first and second moments of $\overline{\bZ}_{\lambda,s_\circ} $, with $s_\circ$ defined as
	$$s_\circ \equiv s_\circ(n,\alpha,C) \equiv \textsf{f}^{1\textsf{rsb}}(\alpha) -\frac{ c^\star \log n}{n} + \frac{C}{n}, $$
	where $c^\star$ is the constant introduced in Theorem \ref{thm1} and $C\in \mathbb{Z}$. Let $\overline{\bN}_{s}$ denote the number of clusters whose size is in the interval $[e^{ns}, e^{ns +1} )$:
	\begin{equation}\label{eq:def:bN:elementary}
	    \overline{\bN}_{s} := \sum_{\Upsilon} \one{\{|\Upsilon| \in [e^{ns}, e^{ns+1}) \}}.
	\end{equation}
	Since $e^{-\lambda} \overline{\bZ}_{\lambda,s} \leq e^{n\lambda s} \overline{\bN}_{s} \leq  \overline{\bZ}_{\lambda,s} $,  a successful computation of first and second moments of $\overline{\bZ}_{\lambda,s_\circ}$ will give us information on $\overline{\bN}_{s_\circ}$ based on the second moment method. We show that
	\begin{equation*}
	\P (\overline{\bN}_{s_\circ} >0) 
	\begin{cases}
	\to 0, & \textnormal{ as } C\to \infty;\\
	\geq c>0, & \textnormal{ as } C \to -\infty.
	\end{cases}
	\end{equation*}
	In what follows, we demonstrate the main ideas in computing the first and second moments of $\overline{\bZ}_{\lambda,s_\circ}$ and showing the concentration of the overlaps based on the moment computations.

	\subsubsection{Moment computations}\label{subsubsec:intro:moment}
	
	The previous approaches in \cite{dss15ksat,dss16maxis,dss16,ssz22} to study the moments of $\overline{\bZ}_\lambda$ were to decompose the quantity into the contributions from different types of 
	``local neighborhood profile'' of configurations. However, in our case where we have infinitely many types of free components, such methods do not give a good enough understanding of $\overline{\bZ}_\lambda$, since the Stirling approximations which were crucial in the earlier works are no longer precise.
	
	Instead, we focus on computing the \textit{cost} of containing each type of free component inside a cluster. A major observation we make is that conditioned on the ``boundary profile" of non-free variables and certain type of clauses, the profile of free components is given by the result of \textit{independently throwing in each type of free component with a prescribed probability}. Then, we observe that computing the first and second moments of $\overline{\bZ}_{\lambda,s_\circ}$ conditioned on the boundary profile amounts to computing the probability of certain large deviation event. After making an appropriate choice of $\lambda$ and an exponential tilting, we appeal to the local central limit theorem (\textsc{clt}) to calculate the probability. This method is a well-established technique in large deviations theory \cite{DZ10} and has also been employed in other combinatorial settings \cite{baez97, Harashita21}. Subsequently, to sum up the contribution from different types of boundary profile, we use the \textit{resampling method} which we describe below.
	
	%This also makes it possible to study the moments of $\overline{\bZ}_{\lambda,s_\circ}$ out of $\overline{\bZ}_\lambda$, because the typical size of a cluster becomes concentrated as we include each free component independently.
	\subsubsection{The resampling method}\label{subsubsec:resampling}
	The resampling method was first introduced in \cite{ssz22} to show negative definiteness of the \textit{free energy} of local neighborhood profile around its maximizer. Here, the \textit{free energy} of a local neighborhood profile refers to the exponential growth of first and second moments of $\overline{\bZ}_{\la}$ given the local neighborhood profile. The main idea behind the method can be summarized as follows. Given a \textsc{nae-sat} instance $\GGG$ and a frozen configuration $\underline{y}\in \{0,1,\ff\}^{n}$, sample small, say $\eps$, fraction of variables $Y$. We sample $v\in Y$ far away from each other so that each free tree containing $v\in Y$ do not intersect. Denote by $\NNN(Y)$ the $\frac{3}{2}$ neighborhood of $Y$ in graph distance and let $\GGG_{\partial}:=\GGG\setminus \NNN(Y)$. Next, resample the spins around $Y$ along with literals of $\NNN(Y)$ and the edge connections between $\NNN(Y)$ and $\GGG_{\partial}$ according to certain Gibbs measure conditioned on $\GGG_{\partial}$ and the spin configuration on $\GGG_{\partial}$ (see Definition \ref{def:resampling:markovchain} for the formal definition). Then, we show the empirical profile of the spins becomes closer to the \textit{optimal profile}, which is obtained by solving a fixed point equation of a certain tree recursion called \textit{belief propagation}(BP) (see Proposition \ref{prop:BPcontraction:1stmo} for the definition of BP fixed point). The main issue is to quantify the improvement coming from this \textit{local} update procedure, and it turns out that it is closely related to a convex \textit{tree optimization}.
	 
	However, the techniques from \cite{ssz22} are limited to the analysis of spin systems with bounded number of spins. In the \textit{untruncated} partition function $\overline{\bZ}_{\la}$, the large free trees inevitably appear and we can no longer sample $Y$ so that the free trees around $Y$ are guaranteed to never intersect. In order to overcome this issue, we first show that the frequency of large free trees decays exponentially in the number of variables (cf. Proposition \ref{prop:1stmo:aprioriestimate}). We then appeal to this rareness of the large free trees to show that the free trees around $Y$ do not intersect with good enough probability under uniform sampling of $Y$. We then perform the resampling procedure $O(\frac{1}{\eps})$ times to show the negative definiteness of the free energy of boundary profiles around its maximizer. The details of the proof are given in Section \ref{subsec:resampling}.

	\subsubsection{Concentration of the overlap}
	
	Theorem \ref{thm2} can be obtained based on the ideas and techniques discussed above. For two  uniformly and independently drawn solutions $\underline{x}^1, \underline{x}^2$ from a random regular $k$-\textsc{nae-sat} instance, Theorem \ref{thm1} tells us that they can be contained  either in the same cluster or in different ones, each with strictly positive probability.
	
	If they are from the same cluster, the set of frozen variables in both solutions will be the same. Moreover, from the moment computations, the number of free trees will concentrate around an explicit value. Since the $0/1$-values for the free variables are assigned independently for each free trees, we show that the absolute value of the overlap concentrates on a single value $p^\star$. On the other hand, if the two solutions are from different clusters, the results from the second moment computation show that the corresponding two frozen configurations are near-independent and from the $0/1$ symmetry of \textsc{nae-sat} model, we will conclude that the overlap concentrates around $0$. 
	
	The actual proof is more complicated than the description above, since we need to take account of the free components containing a cycle. Based on our methods, we develop a coupling argument between the clusters containing cyclic free components and those without cyclic free components, which requires an extended analysis on the moment computations. The details of the proof are given in Section \ref{sec:overlap}.

	\subsection{Organization of the article}
	
	In Section \ref{sec:model}, we formally define
	the combinatorial model which gives a mathematical representation of solution clusters. In Section \ref{sec:1stmo}, we compute the first moment and prove Theorem \ref{thm1}-$(a)$. In Section \ref{sec:2ndmo}, we compute the second moment and finish the proof of Theorem \ref{thm1}. In Section \ref{subsec:resampling}, we use the resampling method to analyze the free energy, which is crucially used in Sections \ref{sec:1stmo} and \ref{sec:2ndmo}. Finally, Section \ref{sec:overlap} is devoted to the proof of Theorem \ref{thm2}.
	\subsection{Notational conventions}
	For non-negative quantities $f=f_{d,k, n}$ and $g=g_{d,k,n}$, we use any of the equivalent notations $f=O_{k}(g), g= \Omega_k(f), f\lesssim_{k} g$ and $g \gtrsim_{k} f $ to indicate that there exists a constant $C_k$, which only depends on $k$, such that
	\begin{equation*}
	    \limsup_{n\to\infty} \frac{f}{g} \leq C_k,
	\end{equation*}
	with the convention $0/0\equiv 1$. We drop the subscript $k$ if there exists a universal constant $C$ such that
	\begin{equation*}
	    \limsup_{n\to\infty} \frac{f}{g} \leq C.
	\end{equation*}
	When $f\lesssim_{k} g$ and $g\lesssim_{k} f$, we write $f\asymp_{k} g$. Similarly when $f\lesssim g$ and $g\lesssim f$, we write $f \asymp g$.

	\section{The combinatorial model}\label{sec:model}

	We start by building the rigorous framework to study the clusters of solutions. In Section \ref{subsec:model:sszrev} (resp. Section \ref{subsec:model:pair}), we introduce the combinatorial model to compute the first (resp. second) moment. Most of the definitions in Section \ref{subsec:model:sszrev} are based on \cite[Section 2]{ssz22} except for Sections \ref{subsubsec:model:J} and \ref{subsubsec:model:free:comp}.
	
	An instance of a $d$-regular $k$-\textsc{nae-sat} problem can be represented by a labeled $(d,k)$-regular bipartite graph as follows. Let $V=\{ v_1, \ldots , v_n \}$ and $F=\{a_1, \ldots, a_m \}$ be the sets of variables and clauses, respectively. Connect $v_i$ and $a_j$ by an edge if the variable $v_i$ participates in the clause $a_j$. Denote this  bipartite graph by $\GG=(V,F,E)$, and for $e\in E$, let $\tL_e\in \{0,1\}$ denote the literal assigned to the edge $e$. Then, a \textsc{nae-sat} instance is defined by $\GGG=(V,F,E,\uL)\equiv (V,F,E,\{\tL_{e}\}_{e\in E})$.
	
	For each $e\in E$, we denote the variable (resp. clause) adjacent to it by $v(e)$ (resp. $a(e)$).  Moreover, $\delta v$ (resp. $\delta a$) are the collection of adjacent edges to $v\in V$ (resp. $a \in F$). We denote $\delta v \setminus e := \delta v \setminus \{e\}$ and $\delta a \setminus e := \delta a \setminus \{e\}$ for simplicity. Formally speaking, we regard $E$ as a perfect matching between the set of \textit{half-edges} adjacent to variables and those to clauses which are labeled from $1$ to $nd=mk$, and hence a permutation in $S_{nd}$. 
	
	\begin{defn}
		For an integer $l\ge 1$ and $\ubx  =(\bx_i) \in \{0,1\}^l$, define
		\begin{equation}\label{eq:def:INAE}
		I^{\textsc{nae}}(\ubx) :=  \mathds{1} \{\ubx \textnormal{ is neither identically } 0 \textnormal{ nor }1 \}.
		\end{equation} 
		Let $\GGG = (V,F,E,\uL)$ be a \textsc{nae-sat} instance. An assignment $\ubx \in \{0,1\}^V$ is called a \textbf{solution}  if 
		\begin{equation}\label{eq:def:INAE G}
		I^{\textsc{nae}}(\bx;\GGG) :=
		\prod_{a\in F} I^{\textsc{nae}} \big((\bx_{v(e)} \oplus \tL_e)_{e\in \delta a}\big) =1,
		\end{equation}
		where $\oplus$ denotes the addition mod 2. Denote the set of solutions by $\textsf{SOL}(\GGG)\subset\{0,1\}^V$ and endow a graph structure on $\textsf{SOL}(\GGG)$ by connecting $\ubx \sim \ubx'$ if and only if they have a unit Hamming distance. Also, let $\textsf{CL}(\GGG)$ be the set of \textbf{clusters}, namely the connected components under this adjacency.
	\end{defn}

	\subsection{The frozen configuration, Bethe formula, and colorings}\label{subsec:model:sszrev}
	
	Our interpretation of the clusters is largely based on the ideas used in \cite{ssz22} and we review those concepts in this subsection. Readers who are familiar with \cite[Section 2]{ssz22} may skip this subsection, except for further concepts we introduce in Sections \ref{subsubsec:model:J} and \ref{subsubsec:model:free:comp}.

	\subsubsection{The frozen configuration}
	Our first step is to define \textit{frozen configuration} which is a basic way of encoding  clusters.
	We introduce \textit{free variable} which we denote by $\ff$, whose Boolean addition is defined as $\ff \oplus 0 := \ff =: \ff \oplus 1$. Recalling the definition of $I^{\textsc{nae}}$ \eqref{eq:def:INAE G}, a \textit{frozen configuration} is defined as follows.
	
	\begin{defn}[Frozen configuration]\label{def:frozenconfig}
		For $\GGG = (V,F,E, \uL)$, $\ux\in\{0,1,\ff \}^V$ is called a (valid) \textbf{frozen configuration} if the following conditions are satisfied:
		\begin{itemize}
			\item No \textsc{nae-sat} constraints are violated for $\ux$. That is, $I^{\textsc{nae}}(\ux;\GGG)=1$.
			
			\item For $v\in V$, $x_v\in \{0,1\}$ if and only if it is forced to be so. That is, $x_v\in \{0,1\}$ if and only if there exists $e\in \delta v$ such that $a(e)$ becomes violated if $\tL_e$ is negated, i.e., $I^{\textsc{nae}} (\ux; \GGG\oplus\one_e )=0$ where $\GGG\oplus\one_e$ denotes $\GGG$ with $\tL_e$ flipped. $x_v=\ff$ if and only if no such $e\in \delta v$ exists.
		\end{itemize}
	\end{defn}
	
	We record the following observations about frozen configurations. Details can be found in the previous works (\cite{dss16}, Section 2 and \cite{ssz22}, Section 2).
	
	\begin{enumerate}
		\item We can map a \textsc{nae-sat} solution $\ubx \in \{0,1 \}^V$ to a frozen configuration via the following \textit{coarsening} algorithm: If there is a variable $v$ such that $\bx_v\in\{0,1\}$ and $I^{\textsc{nae}}(\ubx;\GGG) = I^{\textsc{nae}}(\ubx\oplus \one_v ;\GGG)=1$ (i.e., flipping $\bx_v$ does not violate any clause), then set $x_v = \ff$. Iterate this process until additional modifications are impossible.
		
		\item All solutions in a cluster $\Upsilon\in \mathsf{CL}(\GGG)$ are mapped to the same frozen configuration $\ux\equiv \ux[\Upsilon] \in \{0,1,\ff\}^V$ via the coarsening algorithm. Here, two clusters may be mapped to the same frozen configuration $\ux$. However, in the case where the free subgraph of $\ux$ (see Definition \ref{def:freetree:basic} below) consists of connected components which have at most one cycle, there is at most one cluster that is mapped to $\ux$.

		\item Coarsening algorithm is not necessarily surjective. For instance, a typical instance of $\GGG$ does not have a cluster corresponding to all-free ($\ux \equiv \ff$).
	\end{enumerate}

	\subsubsection{Message configurations}\label{subsubsec:msg:config}
	Although the frozen configurations provides a representation of clusters, it does not tell us how to comprehend the size of clusters.  The main obstacle in doing so comes from the connected structure of free variables which can potentially be complicated. We now introduce the notions to comprehend this issue in a tractable way.
	
	\begin{defn}[Separating and forcing clauses]
		Let $\ux$ be a given frozen configuration on $\GGG = (V,F,E,\uL)$.  A clause $a\in F$ is called \textbf{separating} if there exist $e, e^\prime \in \delta a$ such that $\tL_{e}\oplus x_{v(e)} = 0, \quad \tL_{e^\prime} \oplus x_{v(e^\prime)}=1.$
		We say $a\in F$ is \textbf{non-separating} if it is not a separating clause. Moreover, $e\in E$ is called \textbf{forcing} if $\tL_{e}\oplus x_{v(e)} \oplus 1 = \tL_{e'}\oplus x_{v(e')}\in \{0,1\}$ for all $e'\in \delta a \setminus e$. We say $a\in F$ is \textbf{forcing}, if there exists $e\in \delta a$ which is a forcing edge. In particular, a forcing clause is also separating.
	\end{defn}
	
	Observe that a non-separating clause must be adjacent to at least two free variables, which is a fact frequently used throughout the paper.
	
	\begin{defn}[Free cycles]
		Let $\ux$ be a given frozen configuration on $\GGG = (V,F,E,\uL)$. A cycle in $\GGG$ (which should be of an even length) is called a \textbf{free cycle} if 
		\begin{itemize}
			\item Every variable $v$ on the cycle is $x_v = \ff$;
			
			\item Every clause $a$ on the cycle is non-separating.
		\end{itemize}
	\end{defn}
	
	Our primary interest is in the frozen configurations which do not contain any free cycles. If $\ux$ does not have any free cycle, then we can easily extend it to a \textsc{nae-sat} solution in $\ubx$ such that $\bx_v = x_v$ if $x_v\in\{0,1\}$, since \textsc{nae-sat} problem on a tree is always solvable.  %Later in Lemma \ref{lem:freecycles:1stmo} we will see that such a restriction is enough for our purposes.
	
	\begin{defn}[Free trees: basic definition]\label{def:freetree:basic}
		Given a frozen configuration $\ux$ in $\GGG$, the \textbf{free subgraph} $H$ of $\ux$ is defined to be the subgraph of $\GGG$ induced by free variables and non-separating clauses. Each connected component of $H$ is called a \textbf{free piece} of $\ux$ and denoted by $\ttt^{\textnormal{in}}$. When $\ux$ does not contain any free cycles, the {\textbf{free tree}} $\ttt$ is defined by the union of the free piece $\ttt^{\textnormal{in}}$ and the \textit{half-edges} incident to $\ttt^{\textnormal{in}}$.
		
		For the pair $(\ux, \GGG)$, where $\ux$ has no free cycles, we write $\mathscr{F}(\ux,\GGG)$ to denote the collection of free trees inside $(\ux, \GGG)$. We write $V(\ttt)=V(\ttt^{\textnormal{in}})$, $F(\ttt)=F(\ttt^{\textnormal{in}})$ and $E(\ttt)=E(\ttt^{\textnormal{in}})$ to be the collection of variables, clauses and (full-)edges in $\ttt$. Moreover, define $\dot{\partial} \ttt$ (resp. $\hat{\partial} \ttt$) to be the collection of boundary half-edges that are adjacent to $F(\ttt)$ (resp. $V(\ttt)$), and write $\partial\ttt := \dot{\partial}\ttt \sqcup \hat{\partial} \ttt$

	\end{defn} 
	
	\begin{remark}
		%In Section \ref{subsec:model:freecomp}, we will give an extended definition of free trees, called {\textbf{free components}} which includes a labeling scheme of the variables and edges in them.
		In Definition \ref{def:freetree:bdlab}, we extend the definition of free trees by introducing the labeling scheme of boundary half-edges $\partial \ttt$ that characterizes a free tree. We also remark that in \cite{ssz22}, they called $\ttt$  the free piece and $\ttt^{\textnormal{in}}$ the free tree. We decided to swap the two definitions since $\ttt$ plays a more important role than $\ttt^{\textnormal{in}}$ in our paper.	
	\end{remark}

	We now introduce the \textit{message configuration}, which enables us to calculate the size of a free tree (that is, number of \textsc{nae-sat} solutions on $\ttt$ that extends $\ux$) by local quantities. The message configuration is given by $\utau = (\tau_e)_{e\in E} \in \mathscr{M}^E$ ($\mathscr{M}$ is defined below). Here, $\tau_e=(\dot{\tau}_e,\hat{\tau}_e)$, where $\dot{\tau}$ (resp. $\hat{\tau}$) denotes the message from $v(e)$ to $a(e)$ (resp. $a(e)$ to $v(e)$).
	%To simplify our explanation, let $\textsc{e}$ be a directed version of $e$, with $h(\textsc{e}), t(\textsc{e})$ denoting its head and tail. For instance, if $h(\textsc{e})= a(e)$ then $\tau_{\textsc{e}} = \dot{\tau}_e$. The value of $\tau_{\textsc{e}}$ will be either the symbol ``$\star$'' or a bipartite factor tree whose variables and clauses are unlabeled but the edges are labeled with $\{0,1,\fs \}$.
	A message will carry information of the structure of the free tree it belongs to. To this end, we first define the notion of \textit{joining} $l$ trees at a vertex (either variable or clause) to produce a new tree.
	Let $t_1,\ldots , t_l$ be  a collection of rooted bipartite factor trees satisfying the following conditions:
	\begin{itemize}
		\item Their roots $\rho_1,\ldots,\rho_l$ are all of the same type (i.e., either all-variables or all-clauses) and are all degree one.
		
		\item 
		If an edge in $t_i$ is adjacent to a degree one vertex, which is not the root, then the edge is called a \textbf{boundary-edge}. The rest of the edges are called \textbf{internal-edges}. For the case where $t_i$ consists of a single edge and a single vertex, we regard the single edge to be a boundary-edge.
		
		\item $t_1,\ldots,t_l$ are \textbf{boundary-labeled trees}, meaning that their variables, clauses, and internal edges are unlabeled (except we distinguish the root), but the boundary edges  are assigned with values from $\{0,1,\fs \}$, where $\fs$ stands for `separating'. 
	\end{itemize}
	
    Then, the joined tree $t \equiv  \textsf{j}(t_1,\ldots, t_l) $ is obtained by identifying all the roots as a single vertex $o$, and adding an edge which joins $o$ to a new root $o'$ of an opposite type of $o$ (e.g., if $o$ was a variable, then $o'$ is a clause). Note that $t= \textsf{j}(t_1,\ldots,t_l)$ is also a boundary-labeled tree, whose labels at the boundary edges are induced by those of $t_1,\ldots,t_l$.
	
	For the simplest trees that consist of single vertex and a single edge, we use $0$ (resp. $1$) to stand for the ones whose edge is labeled $0$ (resp. $1$): for the case of $\dot{\tau}$, the root is the clause, and for the case of $\hat{\tau}$, the root is the variable. Also,  if its root is a variable and its edge is labeled $\fs$, we write the tree as $\fs$. 
	
	We can also define the Boolean addition to a boundary-labeled tree $t$ as follows. For the trees $0,1$, the Boolean-additions $0\oplus \tL$, $1\oplus\tL$ are defined as above ($t\oplus \tL$), and we define $\fs \oplus \tL = \fs$ for $\tL\in\{0,1\}$. For the rest of the trees, $t \oplus 0 := t$, and $t\oplus 1$ is the boundary-labeled tree with the same graphical structure as $t$ and the labels of the boundary Boolean-added by $1$ (Here, we define $\fs \oplus 1 = \fs$ for the $\fs$-labels).
	
	\begin{defn}[Message configuration]\label{def:model:msg config}
		Let $\dot{\MMM}_0:= \{0,1,\star \}$ and $\hat{\MMM}_0:= \emptyset$. Supposing that $\dot{\MMM}_t, \hat{\MMM}_t$ are defined, we inductively define $\dot{\MMM}_{t+1}, \hat{\MMM}_{t+1}$ as follows. For $\hat{\utau} \in (\hat{\MMM}_t)^{d-1}$, $\dot{\utau} \in (\dot{\MMM}_t)^{k-1}$, we write $\{\hat{\tau}_i \}:= \{\hat{\tau}_1,\ldots,\hat{\tau}_{d-1} \}$ and similarly for $\{\dot{\tau}_i \}$. We define
		\begin{eqnarray}
		\hat{T}\left(\dot{\utau} \right):= 
		\begin{cases}
		0 & \{\dot{\tau}_i \} = \{  1 \};\\
		1 & \{ \dot{\tau }_i \} = \{0  \};\\
		\fs & \{\dot{\tau}_i \} \supseteq \{ 0,1 \};\\
		\star & \star \in \{ \dot{\tau}_i \}, \{0,1 \} \nsubseteq \{\dot{\tau}_i \};\\
		\mathsf{j}\left(\dot{\utau} \right) & \textnormal{otherwise}, 
		\end{cases}
		&  
		\dot{T}(\hat{\utau}) :=
		\begin{cases}
		0 & 0 \in \{\hat{\tau}_i \} \subseteq \hat{\MMM}_t \setminus\{ 1\};\\
		1 &  1\in \{\hat{\tau}_i\} \subseteq \hat{\MMM}_t \setminus \{0 \};\\
		\tz & \{0,1 \} \subseteq \{\hat{\tau}_i \};\\
		\star & \star \in \{\hat{\tau}_i \} \subseteq \hat{\MMM}_t \setminus \{0,1\};\\
		\mathsf{j}\left(\hat{\utau} \right) & \{\hat{\tau}_i \} \subseteq \hat{\MMM}_t\setminus \{0,1,\star \}.
		\end{cases}
		\end{eqnarray}
		Further, we set $\dot{\MMM}_{t+1} := \dot{\MMM}_t \cup \dot{T}( \hat{\MMM}_t^{d-1}  ) \setminus \{\tz \}$, and $\hat{\MMM}_{t+1}:= \hat{\MMM}_t \cup \hat{T}(\dot{\MMM}_t^{k-1} )$, and define $\dot{\MMM}$ (resp. $\hat{\MMM}$) to be the union of all $\dot{\MMM}_t$ (resp. $\hat{\MMM}_t$) and $\MMM:= \dot{\MMM} \times \hat{\MMM}$. Then, a (valid) \textbf{message configuration} on $\GGG=(V,F,E,\uL)$ is a configuration $\utau \in \MMM^E$ that satisfies (i) the local equations given by
		\begin{equation}\label{eq:def:localeq:msg}
		\tau_e = (\dot{\tau}_e, \hat{\tau}_e) = \left(\dot{T}\big(\hat{\utau}_{\delta v(e)\setminus e} \big), \tL_e \oplus  \hat{T} \big((\uL \oplus \dot{\utau})_{\delta a(e)\setminus e} \big) \right),
		\end{equation}
		for all $e\in E$, and $(ii)$ $(ii)$ if one element of $\{\dot{\tau}_e,\hat{\tau}_e\}$ equals $\star$ then the other element is in $\{0,1\}$.
		%Furthermore, for $\tau \in \MMM$, we write $\ttt(\tau)$ to denote the free tree obtained by combining $\dot{\tau}$ and $\hat{\tau}$, by identifying their edges adjacent to their roots  in such a way that the two roots are not identified together. 
	\end{defn} 
	
	In the definition, $\star$ is the symbol introduced to cover  cycles, and $\tz$ is an error message. See \cite[Figure 2]{ssz22} for an example of $\star$ message. 
	
	When a frozen configuration $\ux$ on $\GGG$ with no free cycles is given, we can construct a message configuration $\utau$ via the following procedure:
	\begin{enumerate}
		\item For a forcing edges $e$, set $\hat{\tau}_e=x_{v(e)}$. Also, for an edge $e\in E$, if there exists $e^\prime \in \delta v(e) \setminus e$ such that $\hat{\tau}_{e^\prime} \in \{0,1\}$, then set $\dot{\tau}_e=x_{v(e)}$.
		\item For an edge $e\in E$, if there exists $e_1,e_2\in \delta a(e)\setminus e$ such that $\{\tL_{e_1}\oplus\dot{\tau}_{e_1}, \tL_{e_2}\oplus\dot{\tau}_{e_2}\}=\{0,1\}$, then set $\hat{\tau}_e = \fs$.
		
		\item After these steps, apply the local equations \eqref{eq:def:localeq:msg} recursively to define $\dot{\tau}_e$ and $\hat{\tau}_e$ wherever possible.
		
		\item For the places where it is no longer possible to define their messages until the previous step, set them to be $\star$.
	\end{enumerate}
	
	In fact, the following lemma shows the relation between the frozen and message configurations. We refer to \cite[Lemma 2.7]{ssz22} for its proof.
	
	\begin{lemma}\label{lem:model:bij:frozen and msg}
		The mapping explained above defines a bijection
		\begin{equation}\label{eq:frz msg 1to1}
		\begin{Bmatrix}
		\textnormal{Frozen configurations } \ux \in \{0,1,\ff \}^V\\
		\textnormal{without free cycles}
		\end{Bmatrix}
		\quad
		\longleftrightarrow
		\quad
		\begin{Bmatrix}
		\textnormal{Message configurations}\\
		\utau \in \MMM^E
		\end{Bmatrix}.
		\end{equation}
	\end{lemma}

	Next, we introduce a dynamic programming method based on \textit{belief propagation} to calculate the size of a free tree by local quantities from a message configuration. 
	
	\begin{defn}\label{def:msg:m}
		Let $\mathcal{P}\{0,1\} $ denote the space of probability measures on $\{0,1\}$. We define the mappings $\dot{\mm}:\dot{\MMM} \rightarrow \PP\{0,1\}$ and $\hat{\mm} : \hat{\MMM} \rightarrow \PP \{0,1\}$ as follows. For $\dot{\tau}\in\{0,1\}$ and $\hat{\tau}\in\{0,1\}$, let $\dot{\mm}[\dot{\tau}] =\delta_{\dot{\tau}}$, $\hat{\mm}[\hat{\tau}] = \delta_{\hat{\tau}}$. For $\dot{\tau}\in\dot{\MMM} \setminus \{0,1,\star\}$ and $\hat{\tau}\in\hat{\MMM} \setminus \{0,1,\star \}$, $\dot{\mm}[\dot{\tau}]$ and $\hat{\mm}[\hat{\tau}]$ are recursively defined:
		\begin{itemize}
			\item Let $\dot{\tau} = \dot{T}(\hat{\tau}_1,\ldots,\hat{\tau}_{d-1})$, with $\star \notin \{\hat{\tau}_i \}$. Define
			\begin{equation}\label{eq:def:bethe:bpmsg:dot}
			\dot{z}[\dot{\tau}] := \sum_{\bx\in\{0,1\} }
			\prod_{i=1}^{d-1} \hat{\mm}[\hat{\tau}_i](\bx), \quad 
			\dot{\mm}[\dot{\tau}](\bx) :=
			\frac{1}{\dot{z}[\dot{\tau}]} \prod_{i=1}^{d-1} \hat{\mm}[\hat{\tau}_i](\bx).
			\end{equation}
			Note that $\dot{z}[\dot{\tau}]$ and $\dot{\mm}[\dot{\tau}](\bx)$ are well-defined, since $(\hat{\tau}_1,\ldots, \hat{\tau}_{d-1})$ can be recoved from $\dot{\tau}$ up to permutation.

			\item Let $\hat{\tau} = \hat{T} ( \dot{\tau}_1,\ldots,\dot{\tau}_{k-1})$, with $\star \notin \{\dot{\tau}_i \}$. Define
			\begin{equation}\label{eq:def:bethe:bpmsg:hat}
			\hat{z}[\hat{\tau}] := 2-\sum_{\bx\in\{0,1\}} \prod_{i=1}^{k-1} \dot{\mm}[\dot{\tau}_i](\bx), \quad 
			\hat{\mm}[\hat{\tau}](\bx) :=
			\frac{1}{\hat{z}[\hat{\tau}]}
			\left\{1- \prod_{i=1}^{k-1} \dot{\mm}[\dot{\tau}_i](\bx) \right\} .
			\end{equation}
			Similarly as above, $\hat{z}[\hat{\tau}]$ and $	\hat{\mm}[\hat{\tau}]$ are well-defined.
		\end{itemize}
		Moreover, observe that inductively, $\dot{\mm}[\dot{\tau}], \hat{\mm}[\hat{\tau}] $ are not Dirac measures unless $\dot{\tau}, \hat{\tau}\in \{0,1\}$. 
	\end{defn} 
	It turns out that $\dot{\mm}[\star], \hat{\mm}[\star]$ can be arbitrary measures for our purpose, and hence we assume that they are uniform measures on $\{0,1\}$.
	
	The equations \eqref{eq:def:bethe:bpmsg:dot} and \eqref{eq:def:bethe:bpmsg:hat} are known as \textit{belief propagation} equations. We refer the detailed explanation to \cite[Section 2]{ssz22} where the same notions are introduced, or to \cite[Ch. 14]{mm09} for more fundamental background. From these quantities, we define the following local weights.
	\begin{equation}\label{eq:def:phi}
	\begin{split}
	&\bar{\varphi} (\dot{\tau}, \hat{\tau}) := \bigg\{ \sum_{\bx\in \{0,1\}} \dot{\mm}[\dot{\tau}](\bx) \hat{\mm}[\hat{\tau}](\bx) \bigg\}^{-1}; \quad \hat{\varphi}^{\textnormal{lit}} (\dot{\tau}_1,\ldots, \dot{\tau}_k):= 1-\sum_{\bx \in \{0,1\}} \prod_{i=1}^k \dot{\mm}[\dot{\tau}_i](\bx);\\
	&\dot{\varphi} (\hat{\tau}_1,\ldots,\hat{\tau}_d):=
	\sum_{\bx\in\{0,1\} }\prod_{i=1}^d \hat{\mm}[\hat{\tau}_i](\bx).
	\end{split}
	\end{equation}
	
	These weight factors can be used to derive the size of a free tree. Let $\ttt$ be a free tree in $\mathscr{F}(\ux,\GGG)$, and let $w^{\lit} (\ttt; \ux,\GGG)$ be the number of \textsc{nae-sat} solutions that extend $\ux$ to $\{0,1\}^{V(\ttt)}$.  Further, let $\textsf{size}(\ux,\GGG)$ denote the total number of \textsc{nae-sat} solutions that extend $\ux$ to $\{0,1\}^V.$
	\begin{lemma}[\cite{ssz22}, Lemma 2.9 and Corollary 2.10; \cite{mm09}, Ch. 14]\label{lem:size:msg and trees}
		Let $\ux$ be a frozen configuration on $\GGG=(V,F,E,\uL)$ without any free cycles, and $\utau$ be the corresponding message configuration. For a free tree $\ttt \in \mathscr{F}(\ux;\GGG)$, denote the number of \textsc{nae-sat} extensions of $\ux\mid_{\ttt}$ on $\ttt$ by $\textsf{size}(\ttt,\ux,\GGG)$. Then, we have that
		\begin{equation}\label{eq:free:tree:weight:lit}
		\textsf{size}(\ttt,\ux,\GGG)=w^{\textnormal{lit}}(\ttt,\ux,\GGG):= \prod_{v\in V(\ttt)} \left\{ \dot{\varphi}(\hat{\utau}_{\delta v}) \prod_{e\in \delta v} \bar{\varphi}(\tau_e) \right\} \prod_{a\in F(\ttt)} \hat{\varphi}^{\textnormal{lit}}\big( (\dot{\utau} \oplus \uL)_{\delta a} \big).
		\end{equation}
		Furthermore, let $\Upsilon \in \textsf{CL}(\GGG)$ be the cluster corresponding to $\ux$. Then, we have
		\begin{equation*}
		\textsf{size}(\ux;\GGG):=	|\Upsilon| = \prod_{v\in V} \dot{\varphi} (\hat{\utau}_{\delta v}) \prod_{a\in F} \hat{\varphi}^{\textnormal{lit}}\big((\dot{\utau}\oplus \uL)_{\delta a} \big) \prod_{e \in E} \bar{\varphi} (\tau_e).
		\end{equation*} 
	\end{lemma}
	%Recalling Remark \ref{rmk:bdspins}, we can see that $w^{\lit}(\ttt;\ux,\GGG)$ is well-defined without knowing $\ux$ or $\GGG$ provided $\ttt$ and its labeling given by Definition \ref{def:freetree:bdlab}. Therefore, we write $w^{\lit}(\ttt)=w^{\lit}(\ttt;\ux,\GGG)$.	
	
	\subsubsection{Colorings}\label{subsubsec:model:col}
	In this subsection, we introduce the \textit{coloring configuration}, which is a simplification of the message configuration. We give its definition analogously to in \cite{ssz22}.
	%Moreover, we also introduce the \textit{simplified coloring} whose difference from the \textit{coloring} is in the definition of $\fs$ which will be clear below.

	Recall the definition of $\MMM=\dot{\MMM}\times \hat{\MMM}, $ and let $\{ \fF \} \subset \MMM$ be defined by $\{\fF\}:= \{\tau \in \MMM: \, \dot{\tau} \notin \{ 0,1,\star\} \textnormal{ and } \hat{\tau}\notin \{ 0,1,\star\} \}$.
	%\begin{equation*}
	%\begin{split}
	%\{\fF\}:= \{\tau \in \MMM: \, \dot{\tau} \notin \{ 0,1,\star\} \textnormal{ and } \hat{\tau}\notin \{ 0,1,\star\} \}.
	%\{\fF^{\textnormal{in}} \}& :=
	%\{ 
	%\tau \in \MMM: \, \dot{\tau } \notin \{ 0,1,\star\} \textnormal{ and } \hat{\tau} \notin \{0,1,\fs,\star \}
	%\}.
	%\end{split}
	%\end{equation*}
	Note that $\{\fF \}$ corresponds to the messages on the edges of free trees, except the boundary edges labeled either 0 or 1.
	%, while $\{\fF^{\textnormal{in}} \}$ is the collection of those on the \textit{internal} edges of free trees.
	Define $\Omega :=  \{\rr_0, \rr_1, \bb_0, \bb_1\} \cup \{\fF \}$ and let $\textsf{S}: \MMM \to \Omega$ be the projections given by
	%\begin{equation*}
	%\begin{split}
	%\Omega& :=  \{\rr_0, \rr_1, \bb_0, \bb_1\} \cup \{\fF \};\\
	%\Omega_{\fs} & := \{\rr_0, \rr_1, \bb_0, \bb_1, \fs \} \cup \{\fF^{\textnormal{in}} \},
	%\end{split}
	%\end{equation*}
	\begin{equation*}
	\textsf{S}(\tau) := 
	\begin{cases}
	\rr_0 & \hat{\tau}=0;\\
	\rr_1 & \hat{\tau}=1;\\
	\bb_0 & \hat{\tau } \neq 0, \, \dot{\tau}=0;\\
	\bb_1 & \hat{\tau} \neq 1, \, \dot{\tau}=1;\\
	\tau & \textnormal{otherwise, i.e., } \tau \in \{ \fF \},
	\end{cases}
	%\quad 
	%\textsf{S}_{\textnormal{simp}}(\tau) := 
	%\begin{cases}
	%\rr_0 & \hat{\tau}=0;\\
	%\rr_1 & \hat{\tau}=1;\\
	%\bb_0 & \hat{\tau } \neq 0, \, \dot{\tau}=0;\\
	%\bb_1 & \hat{\tau} \neq 1, \, \dot{\tau}=1;\\
	%\fs &  \dot{\tau} \notin \{0,1 \},\, \hat{\tau} = \fs; \\
	%\tau & \textnormal{otherwise, i.e., } \tau \in \{\fF^{\textnormal{in}} \}. 
	%\end{cases}
	\end{equation*}
	%Note that the projection $\textsf{S}_{\textnormal{simp}}$ has additional simplification on $\fs$.
	For convenience, we abbreviate $\{\rr \}= \{\rr_0, \rr_1 \}$ and $\{\bb \} = \{\bb_0, \bb_1 \}$, and  define the Boolean addition  as $\bb_\bx \oplus \tL := \bb_{\bx \oplus \tL}$, and similarly for $\rr_\bx$. Also, for $\sigma \in \{ \rr,\bb,\fs\}$, we set $\dot{\sigma} :=  \sigma=:\hat{\sigma}$.
	%Furthermore, the inverse $\tau = (\textsf{S})^{-1}(\sigma)$ is partially defined as follows.
	%\begin{itemize}
		%\item For $\sigma\in \{\fs ,\fF \}$, $\tau = (\dot{\tau},\hat{\tau}) = (\dot{\sigma}, \hat{\sigma})$.
		
		%\item For $\sigma = \rr_0, \rr_1$, we set $\hat{\tau} = 0,1$, respectively, and leave $\dot{\tau}$ undefined.
		
		%\item For $\sigma = \bb_0, \bb_1$, we set $\dot{\tau}=0,1,$ respectively, and leave $\hat{\tau} $ undefined. 
		
		%\item Furthermore, $\tau = (\textsf{S}_{\textnormal{simp}})^{-1}(\fs)$ is given by $\hat{\tau}=\fs$ and an undefined $\dot{\tau}$. For the rest, $(\textsf{S}_{\textnormal{simp}})^{-1}=(\textsf{S})^{-1}$.
	%\end{itemize}
	%For the abbreviated notations above such as $\{\rr \}, \{\bb \}$ and $\{\fF \}$, we write e.g. $\{\rr, \bb \} = \{\rr \} \cup \{\bb \}$ for convenience.
	\begin{defn}[Colorings]\label{def:model:col}
		For $\sig \in \Omega ^d$, let
		\begin{equation*}
		\dot{I}(\sig) : =
		\begin{cases}
		1 & \rr_0 \in \{\sigma_i\} \subseteq \{\rr_0, \bb_0 \};\\
		1& \rr_1 \in \{\sigma_i\} \subseteq \{\rr_1,\bb_1 \};\\
		1 & \{\sigma_i \} \subseteq \{ {\fF} \}, \textnormal{ and } \dot{\sigma}_i = \dot{T}\big( (\hat{\sigma}_j)_{j\neq i} \big), \ \forall i;\\
		0 & \textnormal{otherwise}.
		\end{cases}
		\end{equation*}
		Also, define $	\hat{I}^{\textnormal{lit}}: \Omega^k \to \mathbb{R}$ to be
		%and $	\hat{I}^{\textnormal{lit}}_{\textnormal{simp}}: \Omega_{\fs}^k\to \mathbb{R}$ to be
		\begin{equation*}
		\begin{split}
		\hat{I}^{\textnormal{lit}}(\sig)&:=
		\begin{cases}
		1 & \exists i:\, \sigma_i = \rr_0 \textnormal{ and } \{\sigma_j \}_{j\neq i} = \{\bb_1 \};\\
		1 & \exists i:\, \sigma_i = \rr_1 \textnormal{ and } \{\sigma_j \}_{j\neq i} = \{\bb_0 \};\\
		1 & \{\bb \} \subseteq \{\sigma_i \} \subseteq \{\bb\} \cup \{\sigma \in \{\fF \}: \,\hat{\sigma}=\fs \};\\
		1 & \{\sigma_i \} \subseteq \{ \bb_0, {\fF} \}, \, |\{i: \sigma_i\in  \{{\fF} \}\} | \ge 2 , \textnormal{ and } \hat{\sigma}_i = \hat{T}((\dot{\sigma}_j)_{j\neq i}; 0), \ \forall i \textnormal{ s.t. } \sigma_i \neq \bb_0;\\
		1 & \{\sigma_i \} \subseteq \{ \bb_1, {\fF} \}, \, |\{i: \sigma_i\in  \{{\fF} \}\}| \ge 2 , \textnormal{ and } \hat{\sigma}_i = \hat{T}((\dot{\sigma}_j)_{j\neq i}; 0), \ \forall i \textnormal{ s.t. } \sigma_i \neq \bb_1;\\
		0 & \textnormal{otherwise}.
		\end{cases}
		%\\
		%\hat{I}_{\textnormal{simp}}^{\textnormal{lit}}(\sig)&:=
		%\begin{cases}
		%1 & \exists i:\, \sigma_i = \rr_0 \textnormal{ and } \{\sigma_j \}_{j\neq i} = \{\bb_1 \};\\
		%1 & \exists i:\, \sigma_i = \rr_1 \textnormal{ and } \{\sigma_j \}_{j\neq i} = \{\bb_0 \};\\
		%1 & \{\bb \} \subseteq \{\sigma_i \} \subseteq \{\bb,\fs \};\\
		%1 & \{\sigma_i \} \subseteq \{ \bb_0, {\fF^{\textnormal{in}}} \}, \, |\{i: \sigma_i\in  \{{\fF^{\textnormal{in}}} \}\} | \ge 2 , \textnormal{ and } \hat{\sigma}_i = \hat{T}((\dot{\sigma}_j)_{j\neq i}; 0), \ \forall i \textnormal{ s.t. } \sigma_i \neq \bb_0;\\
		%1 & \{\sigma_i \} \subseteq \{ \bb_1, {\fF^{\textnormal{in}}} \}, \, |\{i: \sigma_i\in  \{{\fF^{\textnormal{in}}} \}\}| \ge 2 , \textnormal{ and } \hat{\sigma}_i = \hat{T}((\dot{\sigma}_j)_{j\neq i}; 0), \ \forall i \textnormal{ s.t. } \sigma_i \neq \bb_1;\\
		%0 & \textnormal{otherwise}.
		%\end{cases}
		\end{split}
		\end{equation*}
		On a \textsc{nae-sat} instance $\GGG = (V,F,E,\uL)$, $\sig\in \Omega^E$ is a (valid) \textbf{coloring} if $\dot{I}(\sig_{\delta v})=\hat{I}^{\textnormal{lit}}((\sig\oplus\uL)_{\delta a}) =1 $ for all $v\in V, a\in F$. 
		%\begin{itemize}
			%\item $\sig\in \Omega^E$ is a (valid) \textbf{coloring} if $\dot{I}(\sig_{\delta v})=\hat{I}^{\textnormal{lit}}((\sig\oplus\uL)_{\delta a}) =1 $ for all $v\in V, a\in F$. 
			
			%\item $\sig\in \Omega_\fs^E$ is  a (valid) \textbf{simplified coloring} if $\dot{I}(\sig_{\delta v})=\hat{I}^{\textnormal{lit}}_{\textnormal{simp}}((\sig\oplus\uL)_{\delta a}) =1 $ for all $v\in V, a\in F$. 
		%\end{itemize} 
	\end{defn}
	
	Given \textsc{nae-sat} instance $\GGG$, it was shown in \cite[Lemma 2.12]{ssz22} that there is a bijection
	\begin{equation}\label{eq:msg col 1to1}
	\begin{Bmatrix}
	\textnormal{message configurations}\\
	\utau \in \MMM^E
	\end{Bmatrix}
	\ \longleftrightarrow \
	\begin{Bmatrix}
	\textnormal{colorings} \\
	\sig \in \Omega^E
	\end{Bmatrix}
	%\ \longleftrightarrow \
	%\begin{Bmatrix}
	%\textnormal{simplified colorings} \\
	%\sig \in \Omega_{\fs}^E.
	%\end{Bmatrix}
	\end{equation}
	%Moreover, coloring configurations and free trees are equivalent objects in the sense of Remark \ref{rmk:freetree equiv msg}. For a free tree $\ttt$, we denote $\sig(\ttt) $ to be the coloring on $\ttt$ 
	%\begin{equation}\label{eq:def:col on freetree}
	%\sig(\ttt) = \{\sigma_e(\ttt) \}_{e\in E(\ttt)}
	%\end{equation}
	%induced by $\ttt$, in the sense of Definition \ref{def:freetree:bdlab} and \eqref{eq:def:freetreeequivclass}.
	The weight elements for coloring, denoted by $\dot{\Phi}, \hat{\Phi}^{\textnormal{lit}}, \bar{\Phi}$, are defined as follows. For $\sig \in \Omega^d,$ let
	\begin{equation*}
	\begin{split}
	\dot{\Phi}(\sig) := 
	\begin{cases}
	\dot{\varphi}(\hat{\sig}) & \dot{I}(\sig) =1 \textnormal{ and } \{\sigma_i \} \subseteq \{\fF \};\\
	1 & \dot{I}(\sig) =1 \textnormal{ and } \{\sigma_i \}\subseteq \{\bb, \rr \};\\
	0 & \textnormal{otherwise, i.e., } \dot{I}(\sig)=0.
	\end{cases}
	\end{split}
	\end{equation*}
	%(If $\{\sigma_i \} \subseteq \{ \fF\}$, then $\hat{\utau}=\hat{\sig}$ and $\dot{\varphi}(\hat{\sig})$ is well-defined)
	For $\sig \in \Omega^k$, let
	\begin{equation*}
	\hat{\Phi}^{\textnormal{lit}}(\sig) :=
	\begin{cases}
	\hat{\varphi}^\lit((\dot{\tau}(\sigma_i))_i)
	& \hat{I}^{\textnormal{lit}} (\sig) = 1 \textnormal{ and } \{\sigma_i \} \cap \{\rr \} = \emptyset;\\
	1 & \hat{I}^{\textnormal{lit}}(\sig) = 1 \textnormal{ and } \{\sigma_i \} \cap \{\rr \} \neq \emptyset;\\
	0 & \textnormal{otherwise, i.e., } \hat{I}^{\textnormal{lit}}(\sig)=0.
	\end{cases}
	\end{equation*}
	(If $\sigma \notin \{\rr \}, $ then $\dot{\tau}(\sigma_i)$ is well-defined.)
	%$\dot{\Phi}$ and $\hat{\Phi}^{\lit}$ for simplified coloring are defined analogously, for $\sig\in\Omega_\fs^d$ and $\sig\in\Omega_\fs^k$, respectively, and we use the same notation as those for  coloring.
	Lastly, let
	\begin{equation*}
	\bar{\Phi}(\sigma) := 
	\begin{cases}
	\bar{\varphi} (\sigma) & \sigma \in \{\fF \};\\
	%2 & \sigma \in\Omega, \ \hat{\sigma}=\fs,  \textnormal{ or } \sigma\in \Omega_\fs, \ \sigma = \fs;\\
	1 & \sigma \in \{\rr, \bb \}.
	\end{cases}
	\end{equation*}
	Note that if $\hat{\sigma}=\fs$, then $\bar{\varphi}(\dot{\sigma},\hat{\sigma})=2$ for any $\dot{\sigma}$. %Thus, $\bar{\Phi}$ is a rewriting of $\bar{\varphi},$ and it is well-defined for both the coloring and the simplified coloring.
	The rest of the details explaining the compatibility of $\varphi$ and $\Phi$ can be found in \cite[Section 2.4]{ssz22}.
	Then, the formula for the cluster size we have  seen in Lemma \ref{lem:size:msg and trees} works the same for the coloring configuration.
	
	\begin{lemma}[\cite{ssz22}, Lemma 2.13]\label{lem:model:size:col}
		Let 	$\ux \in \{0,1,\ff \}^V$ be a frozen configuration on $\GGG=(V,F,E,\uL)$, and let $\sig \in \Omega^E$ be the corresponding coloring. Define
		\begin{equation*}
		w_\GGG^{\textnormal{lit}}(\sig):= \prod_{v\in V}\dot{\Phi}(\sig_{\delta v}) \prod_{a\in F} \hat{\Phi}^{\textnormal{lit}} ((\sig \oplus \uL)_{\delta a}) \prod_{e\in E} \bar{\Phi}(\sigma_e).
		\end{equation*}
		Then, we have $\textsf{size}(\ux;\GGG) = w_\GGG^{\textnormal{lit}}(\sig)$.\
		%The same holds true for the simplified coloring.
	\end{lemma}
	Among the valid frozen configurations, we can ignore the contribution from the configurations with too many free or red colors, as observed in the following lemma.
		\begin{lemma}[\cite{dss16}, Proposition 2.2; \cite{ssz22}, Lemma 3.3]\label{lem:free:red}
		For a frozen configuration $\ux \in \{0,1,\ff \}^{V}$, let $\rr(\ux)$ count the number of forcing edges and $\ff(\ux)$ count the number of free variables.
		There exists a constant $c_k>0$ that only depends on $k$ such that for $k\geq k_0$, $\alpha \in [\alpha_{\textsf{lbd}}, \alpha_{\textsf{ubd}}]$, and $\lambda \in(0,1]$, 
		\begin{equation*}
		\sum_{\ux \in \{0,1,\ff\}^V} \E \left[	\textsf{size}(\ux;\GGG)^\lambda\right] \mathds{1}\left\{ \frac{\rr(\ux)}{nd} \vee \frac{\ff(\ux)}{n}> \frac{7}{2^k} \right\}  \le e^{-c_k n},
		\end{equation*}
		where $\textsf{size}(\ux;\GGG)$ is the number of \textsc{nae-sat} solutions $\ubx \in \{0,1\}^{V}$ which extends $\ux\in \{0,1,\ff\}^{V}$.
	\end{lemma}
	
	Thus, our interest is in counting the number of frozen configurations and colorings such that the fractions of red edges and the fraction of free variables are bounded by $7/2^k$. To this end, we define
	\begin{equation}\label{eq:def:Z:lambda}
	\begin{split}
	&\bZ_{\la}:=	\sum_{\ux \in \{0,1,\ff\}^V} 	\textsf{size}(\ux;\GGG)^\lambda \mathds{1}\left\{ \frac{\rr(\ux)}{nd} \vee \frac{\ff(\ux)}{n}\leq \frac{7}{2^k} \right\};\quad \bZ_{\lambda}^{\tr}:= \sum_{\sig \in \Omega^E}w_\GGG^{\textnormal{lit}} (\sig)^\lambda \mathds{1}\left\{ \frac{\rr(\sig)}{nd} \vee \frac{\ff(\sig)}{n} \le \frac{7}{2^k} \right\} ;\\
	&\bZ_{\lambda,s}:= 	\sum_{\ux \in \{0,1,\ff\}^V} 	\textsf{size}(\ux;\GGG)^\lambda \mathds{1}\left\{ \frac{\rr(\ux)}{nd} \vee \frac{\ff(\ux)}{n}\leq \frac{7}{2^k}, ~~~~~e^{ns} \le \textsf{size}(\ux;\GGG) < e^{ns+1}\right\};\\
	&\bZ_{\lambda,s}^{\tr}:= \sum_{\sig\in \Omega^E} w_\GGG^{\textnormal{lit}}(\sig)^{\la} \mathds{1}\left\{ \frac{\rr(\sig)\vee \fs(\hat{\sig})}{nd} \le \frac{7}{2^k},~~~~~e^{ns} \le w_\GGG^{\textnormal{lit}}(\sig) < e^{ns+1} \right\},
	\end{split}
	\end{equation}
	where $\rr(\sig)$ count the number of red edges and $\ff(\sig)$ count the number of free variables of valid $\sig \in \Omega^E$. The superscript $\tr$ is to emphasize that the above quantities count the contribution from frozen configurations whose subgraph consists solely of free trees, i.e. no free cycles (Recall that by Lemma \ref{lem:model:bij:frozen and msg} and \eqref{eq:msg col 1to1}, the space of coloring has a bijective correspondence with the space of frozen configurations without free cycles). Similarly, recalling the definition of $\overline{\bN}_s$ in \eqref{eq:def:bN:elementary}, total number of clusters of size in $[e^{ns},e^{ns+1})$ , $\bN_s$ is defined to be
	\begin{equation*}
	    \bN_s:=\bZ_{0,s}\quad\textnormal{and}\quad \bN_s^{\tr}:=\bZ^{\tr}_{0,s}.
	\end{equation*}
	Hence, $e^{-n\la s-\la}\bZ_{\la,s}\leq \bN_{s}\leq e^{-n\la s} \bZ_{\la,s}$ holds. For the purpose of calculating $\E \bZ_{\la,s}$ up to constant, we will see in Proposition \ref{prop:ratio:uni:1stmo} that it suffices to calculate $\E \bZ_{\lambda,s}^{\tr}$.
	 %However, the following lemma tells us that  this loss is not significant so that we can focus on investigating $\bZ_\lambda^{\tr}, \bZ_{\lambda, s}^{\tr}$. Due to its technicality, the proof is deferred to Section \ref{sec:app:apriori}.
	%\begin{lemma}\label{lem:freecycles:1stmo}
	%	Recall the definitions $\overline{\bZ}_\lambda$ and ${\bZ}_\lambda.$ For any $\lambda\in(0,1]$, we have 
	%	\begin{equation*}
	%	\E \bZ_\lambda 
	%	\le
	%	\E \overline{\bZ}_\lambda
	%	\le (1+2^{-k/3})\E\bZ_\lambda.
	%	\end{equation*}
	%\end{lemma}
	%A major obstacle of studying the coloring configurations is that the space $\Omega$ is infinite. By ignoring the messages carrying large free trees, we can define the \textit{truncated model} which has been studied intensively in \cite{ssz16}.
	\begin{defn}[Truncated colorings]
		Let $1\leq L< \infty$,  $\ux $ be a frozen configuration on $\GGG$ without free cycles and $\sig\in \Omega^E$ be the coloring corresponding to $\ux$. Recalling the notation $\mathscr{F}(\ux;\GGG)$ (Definition \ref{def:freetree:basic}), we say $\sig$ is a (valid) $L$-\textbf{truncated coloring} if $|V(\ttt)| \le L$ for all $\ttt \in \mathscr{F}(\ux;\GGG)$. For an equivalent definition, let $|\sigma|:=v(\dot{\sigma})+v(\hat{\sigma})-1$ for $\sigma \in \{\fF\}$, where $v(\dot{\sigma})$ (resp. $v(\hat{\sigma})$) denotes the number of variables in $\dot{\sigma}$ (resp. $\hat{\sigma}$). Define $\Omega_L := \{\rr,\bb \}\cup\{\fF \}_L$, where $\{\fF \}_L$ be the collection of $\sigma\in \{\fF \}$ such that $|\sigma| \le L$. Then,  $\sig$ is a (valid) $L$-truncated coloring if $\sig \in \Omega_L^E$.
		
		To clarify the names, we often call the original coloring $\sig\in \Omega^E$ the \textbf{untruncated coloring}.
		%The $L$-\textbf{truncated simplified  coloring} $\sig \in \Omega_{\fs,L}^E$ with $\Omega_{\fs,L}:= \{\rr,\bb,\fs \}\cup \{\fF^{\textnormal{in}} \}_L$  is defined analogously.
	\end{defn}
	Analogously to \eqref{eq:def:Z:lambda}, define the truncated partition function
	\begin{equation*}
	\begin{split}
	&\bZ_{\lambda}^{(L),\tr} := \sum_{\sig \in \Omega_L^E}w_\GGG^{\textnormal{lit}} (\sig)^\lambda \mathds{1}\left\{ \frac{\rr(\sig)}{nd} \vee \frac{\ff(\sig)}{n} \le \frac{7}{2^k} \right\} ;\\
	&\bZ_{\lambda,s}^{(L), \tr}:= \sum_{\sig\in \Omega_L^E} w_\GGG^{\textnormal{lit}}(\sig)^{\la} \mathds{1}\left\{ \frac{\rr(\sig)}{nd} \vee \frac{\ff(\sig)}{n} \le \frac{7}{2^k},~~~~~e^{ns} \le w_\GGG^{\textnormal{lit}}(\sig) < e^{ns+1} \right\}.
	\end{split}
	\end{equation*}
	%In principle, studying the truncated model is insufficient for our purpose. However, important some quantities such as the ratio between the first and the second moment of $\bZ_\lambda$ can be understood as that of the truncated partition function upon taking the limit $L\to \infty$. Thus, refined understanding of $\bZ_{\lambda,L}$ will play an important role throughout the paper.
	
	\subsubsection{Averaging over the literals}\label{subsubsec:model:avglit:1stmo}
	Let $\GGG=(V,F,E,\uL)$ be a \textsc{nae-sat} instance and $\GG=(V,F,E)$ be the factor graph without the literal assignment. As the first step towards computing the moment of $\bZ_\lambda^{\tr}$ (or $\bZ_{\lambda,s}^{\tr}$), we attempt to calculate
	$\mathbb{E} [\bZ_\lambda^{\tr} | \GG ]$. That is, taking the average over the literal assignment. To this end, we study $\mathbb{E}^{\textnormal{lit}}[w_\GGG^{\textnormal{lit}}(\sig) ]$ for a given coloring $\sig \in \Omega^E,$ where $\mathbb{E}^{\textnormal{lit}}$ denotes the expectation over the literals $\uL \sim \textnormal{Unif} [\{0,1\}^E]$. From Lemma \ref{lem:model:size:col}, we can write
	\begin{equation}\label{def:w}
	w_{\GGG}(\sig)^\la:=\E^{\textnormal{lit}} [ w_\GGG^{\textnormal{lit}}(\sig)^\lambda] = \prod_{v\in V} \dot{\Phi}(\sig_{\delta v})^\lambda \prod_{a\in F} \E^{\textnormal{lit}} \hat{\Phi}^{\textnormal{lit}}((\sig\oplus \uL)_{\delta a})^\lambda \prod_{e\in E} \bar{\Phi}(\sigma_e)^\lambda.
	\end{equation}
	Define $\hat{\Phi}(\sig_{\delta a})^\lambda := \E^{\textnormal{lit}}[ \hat{\Phi}^{\textnormal{lit}}((\sig\oplus\uL)_{\delta a})^\lambda]. $ To give a more explicit expression of this formula, we recall a property of $\hat{\Phi}^{\lit}$ from \cite{ssz22}.
	\begin{lemma}[\cite{ssz22}, Lemma 2.17]\label{lem:decompose:Phi:hat}
		For $\sig \in \Omega^k$, $\hat{\Phi}^{\lit}$ can be factorized as $\hat{\Phi}^{\lit}(\sig\oplus\uL) = \hat{I}^{\lit}(\sigma \oplus \uL) \hat{\Phi}^{\textnormal{m}}(\sig)$, where 
		\begin{equation}\label{eq:def:Phi:hat:max}
		\hat{\Phi}^{\textnormal{m}}(\sig) := \max\big\{\hat{\Phi}^{\lit}(\sig\oplus\uL): \uL \in \{0,1\}^k \big\}=
		\begin{cases}
		1 & \sig \in \{\rr,\bb\}^{k},\\
		\frac{\hat{z}[\hat{\sigma}_j]}{\bar{\varphi}(\sigma_j)} &\sig \in \Omega^{k}\textnormal{ with } \sigma_j \in \{\ff\}. 
		\end{cases}
		\end{equation}
	\end{lemma}
	As a consequence, we can write $\hat{\Phi}(\sig)^\lambda = \hat{\Phi}^{\textnormal{m}}(\sig)^\lambda \hat{v}(\sig)$, where
	\begin{equation}\label{eq:def:vhat:basic}
	\hat{v}(\sig) := \E^{\lit} [ \hat{I}^{\lit}(\sig\oplus\uL)]. 
	\end{equation}

	\subsubsection{Embedding number of free trees}\label{subsubsec:model:J}
	
	In this subsection, we introduce the notion of \textit{embedding number} of a free tree. Later, we will see that the embedding numbers play a crucial role in quantifying the contribution of each free tree to  $ \E\bZ_\lambda^{\tr}$ (see Proposition \ref{prop:1stmo:B nt decomp}). To this end, we first refine the basic definition of free trees (Definition \ref{def:freetree:basic}) and give the complete definition below.
		\begin{defn}[Free trees: complete definition]\label{def:freetree:bdlab}
		Let $\ux$ be a frozen configuration in $\GGG$ without any free cycles, and $\sig\in \Omega^{E}$ be the corresponding coloring configuration. For each free tree $\ttt\in \mathscr{F}(\ux,\GGG)$, we label each internal edge and boundary half-edge as follows.
		\begin{itemize}
			\item Each internal edge $e\in E(\ttt)$ is labeled with $\tL_e$ endowed from $\GGG$.
			
			\item Each $e\in \dot{\partial}\ttt$, is labeled $(\sigma_e,\tL_e)$. Note that $\sigma_e \in \{\bb\}$ for $e\in \dot{\partial}\ttt$.
			
			\item Each $e\in \hat{\partial} \ttt$ is labeled $\sigma_e=\fs$.
		\end{itemize}
  Observe that for a labeled free tree $\ttt$, we can uniquely determine the coloring configuration 
 \begin{equation}\label{eq:def:col on freetree}
	\sig(\ttt) := \{\sigma_e(\ttt) \}_{e\in E(\ttt)\sqcup \partial \ttt}
 \end{equation}
  using the recursive equation \eqref{eq:def:localeq:msg} and the labels on $\ttt$ defined above. Note that there can be cases where two different free trees $\ttt$, $\ttt'$ give the same $\sig(\ttt) = \sig(\ttt')$; for instance, consider $\ttt$ and $\ttt'$ that have the isomorphic tree structure and $\{\sigma_e(\ttt)\}_{e\in \partial\ttt}=\{\sigma_e(\ttt^\prime)\}_{e\in \partial \ttt^\prime}$, but have opposite literal on each edge. Then, $\sig(\ttt) = \sig(\ttt')$ holds. Since we cannot distinguish the two free trees in the coloring configuration, we define an equivalence relation given by
\begin{equation}\label{eq:def:freetreeequivclass}
\ttt \sim \ttt' \quad \textnormal{if and only if} \quad \sig(\ttt) = \sig(\ttt'). 
\end{equation}
For the rest of the paper, we view a \textbf{free tree}  $\ttt$ as an equivalence class  with respect to this equivalence relation and we denote by $\FFF_{\tr}$ the set of equivalence classes, i.e. free trees. Since $\sig(\ttt)\in \Omega^{E(\ttt)\sqcup \partial \ttt}$ is well-defined for a free tree $\ttt$, the size of $\ttt$, which we denote by $w^{\lit}_{\ttt}$, is determined by \eqref{eq:free:tree:weight:lit}.  
Moreover, we define the (literal-)averaged $\lambda-$tilted weight of $\ttt$ by $w_{\ttt}^{\la}:= \prod_{v\in V(\ttt)} \left\{ \dot{\Phi}(\sig_{\delta v})^{\la} \prod_{e\in \delta v} \bar{\Phi}(\sigma_e)^{\la} \right\} \prod_{a\in F(\ttt)} \hat{\Phi}(\sig_{\delta a})^{\la}$, where $\sig=\sig(\ttt)$. Also, we will often abbreviate $v_{\ttt}\equiv v(\ttt):=|V(\ttt)|,f_{\ttt}\equiv f(\ttt):=|F(\ttt)|$ and $e_{\ttt}\equiv e(\ttt):=|E(\ttt)|$.
	\end{defn}

	%\begin{remark}\label{rmk:bdspins}
	%	Note that for any $e\in \hat{\partial}\ttt$ for a free tree $\ttt\in \mathscr{F}(\ux,\GGG)$, $a(e)$ is a separating clause with $v(e)$ being a free variable, which shows why $e$ gets the label $\fs$.
	%\end{remark}
	%\begin{remark}\label{rmk:freetree equiv msg}
	%	The reason why we label $e\in \dot{\partial}\ttt$ by $\bb_0, \bb_1$, not $0,1$, will become clear in the next subsection where we introduce the coloring model. Moreover, we remark that free trees (under Definition \ref{def:freetree:bdlab}) and the message configurations are equivalent objects:   a free tree $\ttt \in \mathscr{F}_{\tr}$ and a literal assignment $\uL_{E(\ttt)}$ on $\ttt$ uniquely defines a valid message configuration on $\ttt$, and free trees can be uniquely decoded from a given valid message configuraion.
	%\end{remark}
	\begin{defn}[Embedding number of free trees]\label{def:fitnum:trees}
		For a free tree $\ttt \in \mathscr{F}_{\tr}$, let $\sig=\sig(\ttt)$ be the  coloring on $\ttt$ given by \eqref{eq:def:col on freetree}. For each $v\in V(\ttt)$ and $a\in F(\ttt)$, let $\langle \sig_{\delta v} \rangle$, $\langle \sig_{\delta a}\rangle\in \N^{\Omega}$ be  integer-valued vectors defined as follows:
		\begin{equation}\label{eq:def:spin multi index}
		\begin{split}
		\langle \sig_{\delta v} \rangle (\sigma) := \sum_{e\in \delta v} \mathds{1}\{\sigma = \sigma_e \}, \ \ \langle \sig_{\delta a} \rangle (\sigma ) := \sum_{e\in \delta a} \mathds{1} \{\sigma = \sigma_e \}, \ \forall \sigma \in \Omega.
		\end{split}
		\end{equation}
		Note that sum of all coordinates of $\langle \sig_{\delta v} \rangle$ (resp. $\langle \sig_{\delta a}\rangle$) is  $d$ (resp. $k$). Then, the \textbf{embedding number} $J_\ttt$ of $\ttt$ is defined as
		\begin{equation}\label{eq:def:embnum:freetree}
		J_{\ttt} := d^{1-|V(\ttt)|} k^{-|F(\ttt)|} \prod_{v\in V(\ttt)} {d \choose \langle \sig_{\delta v} \rangle } \prod_{a \in F(\ttt)} {k\choose \langle \sig_{\delta a } \rangle}.
		\end{equation}
	\end{defn}
	We remark that the embedding number $J_{\ttt}$ for a free tree $\ttt\in \FFF_{\tr}$ is used in the definition of \textit{optimal free tree profile} $p_{\ttt}^\star$ in Definition \ref{def:opt:bdry:1stmo} below. We also note that the embedding number $J_{\ttt}$ arises from certain \textit{belief propagation} recursions in Appendix \ref{sec:appendix:Properties of BP fixed point} (see equation \eqref{eq:compat:arranging:tree:1stmo}). Another concrete interpretation is given in Lemma \ref{lem:w vs wcom} below. 
	% We can interpret the embedding number $J^{\textnormal{emb}}_\ttt$ of a free tree $\ttt$ as follows: designate a variable $v_0 \in V(\ttt)$ as the root, and suppose that we are embedding $(\ttt, v_0)$ into the infinite $(d,k)$-regular bipartite factor tree $(\mathcal{T},\rho)$ rooted at a variable $\rho$ in such a way that $v_0$ is mapped to $\rho$. Then, we can see that $J^{\textnormal{emb}}_\ttt$ corresponds to the number of different embeddings, since the colorings $\sig_{\delta v}, \sig_{\delta a}$ carry the information on the structure of the free tree around $v$, $a$, along with the labels at the boundary half-edges. One may also consider another scheme, which is to regard a clause $a_0\in F(\ttt)$ as a root of $\ttt$ and embedding it to  $(\mathcal{T},\hat{\rho})$ rooted at a clause $\hat{\rho}$. In this case, the number of different embeddings becomes $\hat{J}^{\textnormal{emb}}_\ttt:= \frac{k}{d}J^{\textnormal{emb}}_\ttt $.
	%The difference comes from the fact that $m= \frac{d}{k}n$, and hence the total  number of choices of embedding $\ttt$ into $\GGG$ are roughly $nJ^{\textnormal{emb}}_\ttt = m \hat{J}^{\textnormal{emb}}_\ttt$, consistent from both perspectives (in the limiting sense).
	\subsubsection{Free components and component colorings} \label{subsubsec:model:free:comp}
	Although our primary interest is in frozen configurations with no free cycles, we will need to show that the contribution from frozen configurations with free cycles to the first moment is comparable to the contribution from such without free cycles (see Proposition \ref{prop:ratio:uni:1stmo}). In order to do so, we introduce the \textit{free components}, which resembles the basic definition of free trees in Definition \ref{def:freetree:basic}.
	
	\begin{defn}[Free components]\label{def:freecomp:basic}
		Let $\ux$ be a frozen configuration on a \textsc{nae-sat} instance $\GGG$, which potentially contains a free cycle.  On the subgraph $H\subset \GGG$ consisting of free variables and non-separating clauses, let $\fff^{\textnormal{in}}$ denote a free piece, which is a connected component of $H$ (Definition~\ref{def:freetree:basic}). A \textbf{free component} is a union of $\fff^{\textnormal{in}}$ and the \textit{half-edges} adjacent to $\fff^{\textnormal{in}}$. Moreover, each free component $\fff$ has a labeling induced by $(\ux, \GGG)$, given by the following notations and explanation:
		\begin{enumerate}
			\item Let $V(\fff)=V(\fff^{\textnormal{in}})$, $F(\fff) = F(\fff^{\textnormal{in}})$ and $E(\fff) = E(\fff^{\textnormal{in}})$ denote the collection of variables, clauses and edges of $\fff$, respectively. Let $\dot{\partial} \fff$ (resp. $\hat{\partial}\fff$) be the collection of boundary half-edges adjacent to $F(\fff)$ (resp. $V(\fff)$), and write $\partial{\fff}:=\dot{\partial}\fff \sqcup \hat{\partial}\fff$.
			
			\item The variables $V(\fff)$ and clauses $F(\fff)$ are unlabeled.
			
			\item Each edge $e\in E(\fff)$ is labeled by $\tL_e$, the literal assignment on $e$ given by $\GGG$.
			
			\item Each $e\in \hat{\partial}\fff$ is labeled by $\fs$, and each $e\in \dot{\partial }\fff$ is labeled by $(\bb_0, \tL_e)$ (resp. $(\bb_1, \tL_e)$) if $x_{v(e)}= 0$ (resp. $x_{v(e)}= 1$). For $e \in \partial\fff$, the first argument of the label, either $\bb_0,\bb_1$, or $\fs$, is called the spin-label.
		\end{enumerate}
		We write $\FFF(\ux,\GGG)$ to be the collection of free components inside $(\ux, \GGG)$. Also, we denote  by $\FFF$ the set of all possible free components. For $\fff\in \FFF$, we denote by $w^\lit(\fff)$ the number of \textsc{nae-sat} solutions of $\fff$. For simplicity, we will often write $v(\fff):=|V(\fff)|, f(\fff):=|F(\fff)|, e(\fff):=|E(\fff)|$ and $\gamma(\fff):=e(\fff)-v(\fff)-f(\fff)$, which is the number of cycles in $\fff$ minus $1$. Moreover, $\eta_{\fff}(\sigma)$ for $\sigma \in \{\bb_0,\bb_1,\fs\}$ denotes the number of boundary half-edges with spin $\sigma$, i.e.
		\begin{equation}\label{eq:def:eta}
		\eta_\fff( \bb_0 ):= \Big|\big\{e\in \dot{\partial}\fff : \, \textnormal{spin-label of }e \textnormal{ is } \bb_0 \big\}\Big| , \ \ 	\eta_\fff( \bb_1 ):= \Big|\big\{e\in \dot{\partial}\fff : \, \textnormal{spin-label of }e \textnormal{ is } \bb_1 \big\}\Big|,\textnormal{ and }\eta_\fff(\fs) := |\hat{\partial} \fff|.
	    \end{equation}
	\end{defn}
	\begin{remark}\label{rmk:FFF:tr:subset:FFF}
		Although $\FFF_{\tr}$ in Definition \ref{def:freetree:bdlab} is defined in terms of equivalence classes of a subset of $\FFF$, with a slight abuse of notations, we will often think $\FFF_{\tr}$ as a subset of $\FFF$ by forgetting this equivalence relation. Also, note that $\eta_{\ttt}(\sigma)$ for $\sigma \in \{\bb_0,\bb_1,\fs\}$ is well-defined by the same equation \eqref{eq:def:eta}.
	\end{remark}
	We now define the \textbf{component colorings}, which is a combinatorial model to include the frozen configurations that contain a free cycle; these colorings are in one-to-one correspondence with the frozen configurations. Such model will be useful when showing the contribution of frozen configurations with a free cycle is of the same order as the entire first moment (see Proposition \ref{prop:ratio:uni:1stmo}). Define $\Omega_{\textnormal{com}}$ as
	\begin{equation}\label{eq:def:Omegcom1}
	\Omega_{\textnormal{com}}:= \{\rr_0, \rr_1, \bb_0, \bb_1, \fs \} \cup \{ (\fff, e): \fff \in \mathscr{F}, \, e\in E(\fff)  \}.
	\end{equation}
	In the above notation $(\fff,e)$, we take the convention that $(\fff,e)=(\fff^\prime,e^\prime)$ if there exists a graph isomorphism which keeps the spin labels and literal labels, and takes $e$ to $e^\prime$. For example for $\fff\in \FFF$ and $e,e'\in E(\fff), e\neq e^\prime$, it might be that $(\fff,e)=(\fff,e^\prime)$.
	\begin{defn}[Component colorings]\label{def:compcol}
		Let $\ux$ be a frozen configuration on $\GGG$. The component coloring $\sig^{\textnormal{com}}\in \Omega_{\textnormal{com}}^E$ corresponding to $\ux$ is defined as follows:
		\begin{enumerate}
			\item For each $v \in V$ such that $x_v\in\{0,1\}$ and $e\in \delta v$, assign $\sigma^{\textnormal{com}}_e = \rr_{x_v}$ if $e$ is forcing, $\sigma^{\textnormal{com}}_e=\bb_{x_v}$ otherwise;
			
			\item For each separating clause $a$, assign $\sigma^{\textnormal{com}}_e = \fs$ for $e\in \delta a$ such that $x_{v(e)}= \ff$.
			
			\item For each $e\in E$ such that $x_{v(e)}=\ff$ and $a(e)$ is non-separating, let $\fff(e)\in \mathscr{F}$ be the free component that contains $e$. We then set $\sigma^{\textnormal{com}}_e = (\fff(e),e)$.
		\end{enumerate}
	\end{defn}
	% 	For a given $\sig^{\textnormal{com}} \in \Omega_{\textnormal{com}}^E$ and $\GGG$, one can check if it is a valid component coloring by the following procedure. For convenience, we write $\{\sig^{\textnormal{com}}_v \}:= \{\sigma^{\textnormal{com}}_e:e\in\delta v \}$, and similarly for $\{\sig^{\textnormal{com}}_a \}$
% 	\begin{enumerate}
% 		\item For each variable $v\in V$, $\{\sig^{\textnormal{com}}_v \}$ satisfies either
% 		\begin{equation*}
% 		\{\rr_0 \} \subset \{\sig^{\textnormal{com}}_v \} \subset \{\rr_0,\bb_0 \}, \ \ 		\{\rr_1 \} \subset \{\sig^{\textnormal{com}}_v \} \subset \{\rr_1,\bb_1 \}, \ \ \textnormal{or }\ \ \{\sig^{\textnormal{com}}_v \} \subset \Omega_{\textnormal{comp}} \setminus \{\rr,\bb \},
% 		\end{equation*}
% 		and from this we can obtain the frozen configuration $\ux\in\{0,1,\ff \}^V$, which should be valid in $\GGG.$
		
% 		\item For each separating clause $a\in V$, 
% 		$\sig^{\textnormal{com}}_a$ should be defined as the message configuration (Definition \ref{def:model:msg config}).
		
% 		\item For each non-separating clause $a\in F$ (with respect to $(\ux, \GGG)$), let $\fff(a)$ be the free component containing $a$. Then, for each $e\in \delta a$ such that $x_{v(e)}=\ff$, it should satisfy $\sigma^{\textnormal{com}}_e = (\fff(a),e)$.
		
% 	\end{enumerate}
	Given $\GGG$, we call component coloring $\csig$ valid on $\GGG$ if there exists a valid frozen configuration such that it maps to $\csig$ with the above procedure. Then, it is easy to see that the above procedure produces a one-to-one correspondence between the frozen configurations (including the ones with free cycles) and the (valid) component colorings.

	Note that we defined $\fff$ so that it records all the literal information of the edges. Thus, we define the expected weight of $\fff$ under averaging over the literal assignment as
	\begin{equation}\label{eq:def:weight:freecomp:avg}
	w^{\textnormal{com}}(\fff)^{\la}:=  w^{\textnormal{lit}}(\fff)^\lambda 2^{-k |F(\fff)|}. 
	\end{equation} 
	The notation $w^{\textnormal{com}}(\fff)$ is introduced to prevent confusion from $w_{\ttt}$ in the case when $\fff$ is a tree (see Lemma \ref{lem:w vs wcom} below for the relationship between the two notions when $\fff$ is a tree).
	Let $\sig^{\textnormal{com}}\in \Omega_{\textnormal{com}}^E$ be a valid component coloring on $\GGG=(\GG,\uL)$ and denote by $w^{\lit}_{\GGG}(\csig)$ the number of \textsc{nae-sat} solutions which extend $\csig$, i.e. $\textsf{size}(\ux,\GGG)$ for $\ux$ corresponding to $\csig$. Then, we can write $\E[ w_\GGG^{\lit} (\sig^{\textnormal{com}})^\lambda | \GG ] $ as follows: writing $\mathscr{F}(\sig^{\textnormal{com}})$ and $F_{\textnormal{sep}}(\sig^{\textnormal{com}})$ to denote the free components and the separating clauses in $\sig^{\textnormal{com}}$ respectively, 
	\begin{equation}\label{eq:sizeformula:freecomp:1stmo}
	w(\sig^{\textnormal{com}})^\lambda := 	\E^{\lit}\left[\left.w_\GGG^{\lit} (\sig^{\textnormal{com}})^\lambda\right| \GG  \right] = \prod_{\fff \in \mathscr{F}(\sig^{\textnormal{com}})} w^{\textnormal{com}}(\fff)^\lambda  \prod_{a\in F_{\textnormal{sep}}(\sig^{\textnormal{com}})} \hat{v}(\sig^{\textnormal{com}}_a).
	\end{equation}
	
\subsection{The pair model}\label{subsec:model:pair}
% A significant difference from \cite{ssz16} is that the analysis for the second moment $\E\bZ_\lambda^2$ (or $\E \bZ_{\lambda,s}^2$) is substantially more involved. This is because we work with the full space $\Omega$ instead of the truncated space $\Omega_L$.
In this subsection, we introduce concepts and notations that are required to understand the second moment of the partition functions.
\subsubsection{Pair-colorings}
To begin with, for $\ula = (\lambda_1,\lambda_2) \in [0,1]^2$ and $\bs=(s_1,s_2)\in [0,\log 2)^2$, define
\begin{equation}\label{eq:def:pair:partition}
\begin{split}
&\bZ_{\ula}^2:= \bZ_{\la_1}^{\tr}\bZ_{\la_2}^{\tr}=\sum_{\sig^1, \sig^2 \in \Omega^E} w_\GGG^{\textnormal{lit}}(\sig^1)^{\lambda_1} w_\GGG^{\textnormal{lit}}(\sig^2)^{\lambda_2}\mathds{1}\left\{ \frac{\rr(\sig^1)}{nd} \vee \frac{\ff(\sig^1)}{n}\vee \frac{\rr(\sig^2)}{nd} \vee \frac{\ff(\sig^2)}{n} \le \frac{7}{2^k} \right\},\\
&\bZ_{\ula}^{2,(L)}:= \bZ_{\la_1}^{(L),\tr}\bZ_{\la_2}^{(L),\tr}=\sum_{\sig^1, \sig^2 \in \Omega_L^E} w_\GGG^{\textnormal{lit}}(\sig^1)^{\lambda_1} w_\GGG^{\textnormal{lit}}(\sig^2)^{\lambda_2}\mathds{1}\left\{ \frac{\rr(\sig^1)}{nd} \vee \frac{\ff(\sig^1)}{n}\vee \frac{\rr(\sig^2)}{nd} \vee \frac{\ff(\sig^2)}{n} \le \frac{7}{2^k} \right\},\\
&\bZ_{\ula,\bs}^2:= \bZ^{\tr}_{\la_1,s_1}\bZ^{\tr}_{\la_2}, \quad \bN^{2}_{\bs}:=\bN_{s_1}^{\tr}\bN_{s_2}^{\tr},\quad\textnormal{and}\quad \bZ_{\ula,\bs}^{2,(L)}:= \bZ^{(L),\tr}_{\la_1,s_1}\bZ^{(L),\tr}_{\la_2}.
\end{split}
\end{equation}
\begin{remark}
Observe that in the above definitions, we restricted our attention to the second moments of $\bZ^{\tr}_{\la}$, i.e. the contribution from the frozen configurations without free cycles. This is because we will apply the second moment method to $\bN_{s}^{\tr}$, rather than $\bN_s$.
\end{remark}
The following estimate was established in \cite{ssz23}, which is the analog of Lemma \ref{lem:free:red} for the second moment. 
    \begin{lemma}[\cite{ssz23}, Corollary F.5]
    \label{lem:free:red:2}
    For $\sig^1,\sig^2\in \Omega^E$, let $\ux^1,\ux^2\in \{0,1,\ff\}^V$ be the corresponding frozen configurations. Let $\zeta(\sig^1,\sig^2)\in [0,1]$ be the fraction of variables such that $\ux^1,\ux^2$ differ (see also Definition \ref{def:overlapofcol} below). Then, there exists a constant $c_k>0$ that only depends on $k$ such that 
    \begin{equation*}
        \E\Bigg[\sum_{\sig^1, \sig^2 : \big|\zeta(\sig^1,\sig^2)-\frac{1}{2}\big|\leq \frac{k^2}{2^{k/2}}} w_\GGG^{\textnormal{lit}}(\sig^1)w_\GGG^{\textnormal{lit}}(\sig^2)\mathds{1}\left\{ \frac{\rr(\sig^1)}{nd} \vee \frac{\ff(\sig^1)}{n}\vee \frac{\rr(\sig^2)}{nd} \vee \frac{\ff(\sig^2)}{n} > \frac{7}{2^k} \right\}\Bigg] \leq e^{-c_k n}\,.
    \end{equation*}
    \end{lemma}
We can consider $\bZ_{\ula}^2$ (resp. $\bZ_{\ula}^{2,(L)}$) as the first moment of the \textbf{pair  coloring} $\bsig =(\sig^1,\sig^2) \in \Omega_{ 2}^E$ (resp. $\Omega_{2,L}$), where $\Omega_{2}:=\Omega\times \Omega$ and $\Omega_{2,L}:=\Omega_L\times \Omega_L$, along with the weight factor $\bw_\GGG^{\textnormal{lit}}(\bsig)^{\ula}:= w_\GGG^{\textnormal{lit}}(\sig^1 )^{\lambda_1}w_\GGG^{\textnormal{lit}}(\sig^2)^{\lambda_2}$. For a pair-coloring $\bsig=(\sig^1,\sig^2)$, we call $\sig^1$ (resp. $\sig^2$) the first (resp. second) copy. $\bsig$ is called a \textbf{valid pair-coloring} if both copies are valid colorings on $\GGG$. Thus, a valid pair-coloring is a special case of a (valid) \textbf{pair frozen configuration} $(\ux^{1},\ux^{2})\in (\{0,1,\ff\}^{2})^{V}$, where each copy has no free cycles. By Lemma \ref{lem:model:size:col} and Lemma \ref{lem:decompose:Phi:hat} we can write
\begin{equation}\label{eq:def:Phi:pair}
\begin{split}
\bw_\GGG(\bsig)^{\ula}:=\E^{\textnormal{lit}} [ \bw_\GGG^{\textnormal{lit}}(\bsig)^{\ula}] = \prod_{v\in V} \dot{\Phi}_2^{\ula}(\sig_{\delta v}) \prod_{a\in F} \hat{\Phi}_2^{\ula}((\bsig\oplus \uL)_{\delta a}) \prod_{e\in E} \bar{\Phi}_2^{\ula}(\bsigma_e),
\end{split}
\end{equation}
where $\dot{\Phi}_2^{\ula} := \dot{\Phi}^{\lambda_1} \otimes \dot{\Phi}^{\lambda_2}, \bar{\Phi}_2^{\ula} := \bar{\Phi}^{\lambda_1}\otimes \bar{\Phi}^{\lambda_2}$ and \begin{equation}
    \hat{\Phi}_2^{\ula}(\bsig) := \E^{\lit} \left[ \hat{\Phi}^{\lit}((\sig^1 \oplus \uL)^{\la_1}\hat{\Phi}^{\lit}((\sig^2 \oplus \uL)^{\la_2} \right]=\hat{v}_2(\bsig)\Big((\hat{\Phi}^{\textnormal{m}})^{\la_1}\otimes(\hat{\Phi}^{\textnormal{m}})^{\la_2}\Big)(\bsig)
\end{equation}
for $\hat{v}_2(\bsigma):=\E^{\textnormal{lit}}[\hat{I}^{\lit}(\sig^1\oplus\uL)\hat{I}^{\lit}(\sig^2\oplus\uL)]$.
\begin{remark}
For the purpose of proving Theorem \ref{thm1} and Theorem \ref{thm2}, we only need to compute the moment of $\bZ^2_{\la}:=\bZ^2_{(\la,\la)}$. That is, we must consider the case where $\la_1=\la_2$. However, we will need to consider the 
$(\la_1,\la_2)$ model to prove Lemma \ref{lem:relation:leading:constants}, which will play a crucial role in the companion paper \cite{nss2}.
\end{remark}
\subsubsection{Union-free components} Studying $\E\bZ_{\ula}^2$ from the pair-coloring perspectives can be difficult: let ${\ttt}_1\in\mathscr{F}(\sig^1;\GGG)$ be a free tree in the first copy, and let $e_1, e_2$ be two distinct edges in $\ttt_1$. Although the free trees induced by $\sigma_{e_1}^1$ and $\sigma_{e_2}^1$ are both $\ttt_1$, $\sigma_{e_1}^2$ and $\sigma_{e_2}^2$ do not necessarily give the same free tree in the second copy. To deal with this issue, we introduce the analogous notion of free components for the pair model.

\begin{defn}[pair-separating clauses and union-free variables]\label{def:sep:pairmo}
Let $\bsig=(\sig^1,\sig^2)\in\Omega_{2}^E$ be a valid pair-coloring in $\GGG=(V,F,E,\uL)$, and let $\ux^1, \ux^2$ be their corresponding frozen configurations. A clause $a\in F$ is \textbf{pair-separating} if $a$ is separating in both $(\ux^1,\GGG)$ and $(\ux^2,\GGG)$. If $a$ is not pair-separating, i.e. $a$ is non-separating in at least one of the copy, we say $a$ is \textbf{non-pair-separating}. Moreover, a variable $v\in V $ is called \textbf{union-free} if at least one of $x_v^1, x_v^2$ is $\ff$.
\end{defn}

\begin{defn}[union-free components in the pair model]\label{def:model:unioncomp}
Let $\bsig=(\sig^1,\sig^2)\in\Omega_{2}^E$ be a valid pair-coloring in $\GGG=(V,F,E,\uL)$. Consider the induced subgraph $H$ of $\GGG$ that consists of the union-free variables and non-pair-separating clauses. Connected components of $H$ are called a \textbf{union-free piece} of $(\bsig,\GGG)$, denoted by $\vvv^{\textnormal{in}}$. For each union-free piece $\vvv^{\textnormal{in}}$, the \textbf{union-free component}, denoted by $\vvv$, is the union of $\vvv^{\textnormal{in}}$ and the boundary \textit{half-edges} incident to $\vvv^{\textnormal{in}}$. Detailed description on $\vvv$ is given by the following.
\begin{enumerate}
\item The variables $V(\vvv)=V(\vvv^{\textnormal{in}})$ and clauses $F(\vvv^{\textnormal{in}})$ of $\vvv$ are unlabeled.

\item The collection of internal edges is denoted by $E(\vvv)=E(\vvv^{\textnormal{in}})$, and the collections of boundary half-edges $\dot{\partial}\vvv$, $\hat{\partial}\vvv$ and $\partial \vvv$ are defined analogously to Definition \ref{def:freecomp:basic}-(1).

\item Define the projection $\textsf{P}:\Omega \to \{\rr_0, \rr_1, \bb_0, \bb_1, \fs,\ff \}$ by
\begin{equation}\label{eq:def:projP}
\textsf{P}(\sigma) :=
\begin{cases}
\sigma & \textnormal{if } \sigma\in \{\rr_0, \rr_1, \bb_0, \bb_1 \};\\
\fs & \textnormal{if } \hat{\sigma}=\fs;\\
\ff & \textnormal{if } \sigma \in \{\fF \} ,
\end{cases}
\end{equation}
and let $\textsf{P}_2:\Omega_{2} \to \{\rr_0, \rr_1,\bb_0,\bb_1,\fs,\ff \}^2$ be $\textsf{P}_2(\bsigma) := (\textsf{P}(\sigma^1), \textsf{P}(\sigma^2) )$. Then, edge $e\in E(\vvv)\sqcup \dot{\partial} \vvv$  is labeled with $(\textsf{P}_2(\bsigma_e), \tL_e)$ while edge $e\in \hat{\partial} \vvv$ is labeled with $\textsf{P}_2(\bsigma_e)$, i.e. no literal information on $\hat{\partial} \vvv$. We will write
$(\textsf{P}_2(\vvv,e), \tL_e)$ to indicate the label at $e$ of the union-free component $\vvv$. We call $\textsf{P}_2(\vvv,e)$ (resp. $\tL_e$) the \textbf{spin label} or \textbf{color} (resp. \textbf{literal label}) at $e$.
\end{enumerate}
We denote the collection of union-free components in $(\bsig,\GGG)$ by $\mathscr{F}_2(\bsig,\GGG)$, and let $\mathscr{F}_2$ be the collection of all free components. Similar to the case of free components in Definition \ref{def:freecomp:basic}, we write $v(\vvv):=|V(\vvv)|, f(\vvv):=|F(\vvv)|,e(\vvv):=|E(\vvv)|$ and $\gamma(\vvv):=e(\vvv)-v(\vvv)-f(\vvv)$, which is the number of cycles in $\vvv$ minus 1.
\end{defn}
Observe that for $\vvv\in \FFF_2$ and $v\in V(\vvv)$, we can recover $(x_v^{1},x_v^{2})\in \{0,1,\ff\}^2$ from the information of spin labels $(\textsf{P}_2(\vvv,e))_{e\in \delta v}$. Thus, we can determine the collection of free trees in first (resp. second) copy in $\vvv$, which we denote by $(\ttt^1_i)$ (resp. $(\ttt^2_j)$). Then, the size of $\vvv$ in the first and the second copy are defined by
\begin{equation}\label{eq:def:size:union-comp:each:copy}
   \bw^{\lit,\ell}(\vvv)\equiv \prod_{i} w^\lit(\ttt^{\ell}_i)\quad\textnormal{for}\quad \ell=1,2.
\end{equation}
Also, define the size of $\vvv$ with respect to the exponent $\ula=(\lambda_1,\lambda_2)$ by
\begin{equation}\label{eq:def:size:freecomp:2ndmo:lit}
\bw^{\textnormal{lit}}(\vvv)^{\ula} \equiv \bw^{\lit,1}(\vvv)^{\la_1}\bw^{\lit,2}(\vvv)^{\la_2}
\end{equation}
Another observation is that for $e\in \hat{\partial}\vvv$, the color at $e$ is $\fs$ in at least for one of the two copies, since $x_{v(e)}^{1}$ or $x_{v(e)}^2$ is free while $a(e)$ is pair-separating. Similarly for $e\in \dot{\partial}\vvv$, the color at $e$ is $\bb$ in at least one copy. Thus,
\begin{itemize}
\item $\textsf{P}_2(\vvv,e)  \in  \{\rr_0,\rr_1, \bb_0, \bb_1 \}^2\setminus \{\rr_0, \rr_1\}^2$ if $e\in \dot{\partial}\vvv$;

\item $\textsf{P}_2(\vvv,e) \in  \{\rr_0,\rr_1, \bb_0, \bb_1, \fs \}^2 \setminus \{\rr_0, \rr_1, \bb_0,\bb_1 \}^2$ if $e\in \hat{\partial} \vvv$.
\end{itemize}
Based on this observation,  we define the space of boundary colors for the pair model by
\begin{equation}\label{eq:def:boundarycolors:pair}
\dot{\partial}_2:= \{\rr_0,\rr_1, \bb_0, \bb_1\}^2\setminus \{\rr_0, \rr_1\}^2, \quad \hat{\partial}_2:= \{\rr_0,\rr_1, \bb_0, \bb_1, \fs \}^2 \setminus \{\rr_0, \rr_1, \bb_0,\bb_1 \}^2,\quad \partial_2:=\dot{\partial}_2\sqcup \hat{\partial}_2.
\end{equation}
Moreover, given $\vvv \in \FFF_2$, we define
\begin{equation*}
\eta_\vvv(\bx):= | \{ e\in \dot{\partial} \vvv\sqcup\hat{\partial}\vvv: \textsf{P}_2(\vvv,e) =\bx \}|,~~~~~~~~\bx\in \dot{\partial}_2\sqcup\hat{\partial}_2
\end{equation*}
\subsubsection{Component coloring for the pair model}\label{subsubsec:model:compcol:pair}
Based on the notion of union-free components, we can define pair component coloring for the pair model analogously to Definition \ref{def:compcol}.
Let $\FFF_2$ denote the collection of all union-free components, and define
\begin{equation}\label{eq:def:Omegcom2}
\Omega_{\textnormal{com},2} := \{\rr_0,\rr_1,\bb_0,\bb_1,\fs \}^2 \cup \{ (\vvv,e): \vvv \in \FFF_2, \ e \in E(\vvv) \}.
\end{equation}
As in the definition of $\Omega_{\textnormal{com}}$ \eqref{eq:def:Omegcom1}, $(\vvv,e)$ is considered as an isomorphism class with respect to the graph isomorphism which keeps spin labels, literal labels and $e$ unchanged.
Then, for a pair frozen configuration $(\ux^1,\ux^2 )\in\{0,1,\ff \}^{2V}$, a \textbf{pair component coloring} $\bsig^{\textnormal{com}}\in \Omega_{\textnormal{com},2}^E$ corresponding to $(\ux^1,\ux^2)$ is defined analogously to Definition \ref{def:compcol}: 
\begin{enumerate}
\item  For each $v\in V$ such that $(x_v^1, x_v^2)\in\{0,1 \}^2$ and $e\in \delta v$, assign $\sigma^{i}_e = \rr_{x_v^i}$ if $e$ is forcing in the $i$-th copy, and $\sigma_e^i = \bb_{x_v^i}$ otherwise for $i=1,2$. Then, we set $\bsigma^{\textnormal{com}}_e = (\sigma_e^1,\sigma_e^2).$
\item For pair-separating clause $a$ and $e\in \delta a$, assign $\sigma_e^{i}=\fs$ if $x_{v(e)}^i=\ff$ for $i=1,2$. Otherwise, assign $\sigma_e^{i}=\bb_{0}$ if $x_{v(e)}^i=0$ and assign $\sigma_e^{i}=\bb_{1}$ if $x_{v(e)}^i=1$. Then, we set $\bsigma^{\textnormal{com}}_e = (\sigma_e^1,\sigma_e^2).$
\item For each $e\in E$ such that $v(e)$ is union-free and $a(e)$ is non-pair-separating, let $\vvv(e)\in \mathscr{F}_2$ be the union-free component that contains $e$. We then set $\bsigma^{\textnormal{com}}_e = (\vvv(e),e)$.
\end{enumerate}
We note from the definition that there is a natural one-to-one correspondence as follows:
\begin{equation}\label{eq:corr:pair-frozen:color:comp}
\begin{Bmatrix}
\textnormal{pair frozen configurations }\\
\textnormal{without free cycles in both copies}\\
(\ux^1,\ux^2) \in (\{0,1,\ff \}^2)^E
\end{Bmatrix}
\ \longleftrightarrow \
\begin{Bmatrix}
\textnormal{pair-colorings} \\
\bsig \in \Omega_2^E
\end{Bmatrix} \ \longleftrightarrow \
\begin{Bmatrix}
\textnormal{pair component colorings} \\
\bsig^{\textnormal{com}} \in \Omega_{\textnormal{com},2}^E.
\end{Bmatrix}
\end{equation}

For a free component $\vvv\in \mathscr{F}_2$, recall the definition of its size $\bw^{\textnormal{lit}}(\vvv)$ \eqref{eq:def:size:freecomp:2ndmo:lit}. Analogously to \eqref{eq:def:weight:freecomp:avg}, we can express the expected size of $\vvv$ under averaging of the literal assignments by
\begin{equation}\label{eq:def:weight:freecomp:avg:2ndmo}
\bw^{\textnormal{com}}(\vvv)^{\ula}:= \bw^{\textnormal{lit}}(\vvv)^{\ula}2^{-k|F(\vvv)|}.
\end{equation}

\subsubsection{Union coloring and union-free trees}\label{subsubsec:model:unifreetree}

In this subsection, we introduce the notions of union coloring, union-free tree and its embedding number. The union coloring configuration will be discussed only briefly, since it will only be used to define the notion of union-free tree, an analog of free tree in Definition \ref{def:freetree:bdlab}.
\begin{defn}[Union coloring]\label{def:unioncol}
Given a \textsc{nae-sat} instance $\GGG$, consider a valid pair-coloring configuration $\bsig \in \Omega_2^E$ such that its union-free components, i.e. $\mathscr{F}_2(\bsig,\GGG)$, consist only of trees. Similar to the coloring configuration, we define \textbf{union-coloring configuration} $\bsig^{\textsf{u}}:= (\bsigma^{\textsf{u}}_e)_{e\in E}$ corresponding to $\bsig \in \Omega_2^E$ as follows.
\begin{enumerate}
\item If $e\in E$ and $\bsigma_e \in \{\rr_0,\rr_1,\bb_0,\bb_1,\fs \}^2$, then $\bsigma_e^{\textsf{u}} := \bsigma_e$.

\item If not, then $e\in E$ must be contained in a union-free component $\vvv$, which is a tree. Recall that in Section \ref{subsubsec:msg:config}, $\textsf{j}(t_1,...,t_\ell)$ denoted the joined tree of rooted bipartite factor trees $t_1,...,t_\ell$ with boundary labels $\{0,1,\fs\}$. Similarly, we now consider the rooted bipartite factor trees $u_1,...,u_{\ell}$, where the inner edges of $u_i$ are labeled with a color (there is no literal label) in $\Omega_2^{\textnormal{in}}:=\{\ff\}\times \{\rr_0,\rr_1,\bb_0,\bb_1,\fs\}\sqcup \{\rr_0,\rr_1,\bb_0,\bb_1,\fs\}\times \{\ff\}$ and the boundary edges are labeled with a color in $\partial_2$. For such $u_i$, $1\oplus u_i$ then denotes the rooted bipartite factor tree with all its boundary colors and inner colors flipped, e.g. $\ff\rr_0$ to $\ff\rr_1$. For $\bsigma \in \Omega_2^{\textnormal{in}}$, denote by $\textsf{j}_2(u_1,...,u_\ell;\bsigma)$ the joined tree of $u_1,...,u_{\ell}$ which has color $\bsigma$ on the unique edge adjacent to its root. Then, analogously to \eqref{eq:def:localeq:msg}, recursively apply $\textsf{j}_2(\cdot)$ from the boundary of $\vvv$: for $e\in \partial \vvv$, $\bsigma^{\textsf{u}}_e$ is determined from the previous step and we take the convention that $\bx \in \dot{\partial}_2$ (resp. $\bx \in \hat{\partial}_2$) is the variable-to-clause tree (resp. clause-to-variable tree) with a single edge whose color is $\bx$. That is, for $e\in E(\vvv)$,
\begin{equation}\label{eq:union:coloring:recursive}
    \bsigma^{\textsf{u}}_e = (\dot{\bsigma}^{\textsf{u}}_e, \hat{\bsigma}^{\textsf{u}}_e) = \left(\textsf{j}_2 \big(\hat{\bsig}^{\textsf{u}}_{\delta v(e)\setminus e} \big), \tL_e \oplus  \textsf{j}_2 \big((\uL \oplus \dot{\bsig}^{\textsf{u}})_{\delta a(e)\setminus e} \big) \right).
\end{equation}
Here, $\dot{\bsigma}^{\textsf{u}}_e$ (resp. $\hat{\bsigma}^{\textsf{u}}_e$) denotes the variable-to-clause tree (resp. clause-to-variable tree). We denote by $\dot{\Omega}^{\textsf{u}}$ (resp. $\hat{\Omega}^{\textsf{u}}$) the set of all possible $\dot{\bsigma}^{\textsf{u}}_e$ (resp. $\hat{\bsigma}^{\textsf{u}}_e$) from \eqref{eq:union:coloring:recursive}. Then, we let $\Omega^{\textsf{u}}:=\dot{\Omega}^\su\times \hat{\Omega}^\su$.
\end{enumerate}
\end{defn}
Observe that for $\dot{\bsigma}^{\textsf{u}} \in \dot{\Omega}^{\textsf{u}}$, we can determine the variable to clause component of the free tree in the $i$'th copy ($i=1,2$) that contains the root clause. This is because we can determine which variables in $\dot{\bsigma}^{\textsf{u}}$ are free in $i$'th copy by the information of the color of each edge in $\dot{\bsigma}^{\textsf{u}}$. Denote by $\dot{\pi}_i(\dot{\bsigma}^{\textsf{u}})\in \dot{\Omega}$ such variable to clause component of the free tree in the $i$'th copy. $\hat{\pi}_i(\hat{\bsigma}^{\textsf{u}})\in \hat{\Omega}$ for $\hat{\bsigma}^{\textsf{u}} \in \hat{\Omega}^{\textsf{u}}$ and $i=1,2$ is analogously defined. Then, let $\pi_i(\bsigma^{\textsf{u}}):=\left(\dot{\pi}_i(\dot{\bsigma}^{\textsf{u}}), \hat{\pi}_i(\hat{\bsigma}^{\textsf{u}})\right)\in \Omega$ for $i=1,2$ and write $\pi(\bsigma^{\textsf{u}}) := (\pi_1(\bsigma^{\textsf{u}}), \pi_2(\bsigma^{\textsf{u}})) \in \Omega_2$. Then, the $\ula\in [0,1]^2$ weight factors of the union coloring $\dot{\Phi}_{\su}^{\ula}, \hat{\Phi}_{\textsf{u}}^{\ula}$ and $\bar{\Phi}_{\textsf{u}}^{\ula}$ are defined by
\begin{equation}
\dot{\Phi}_{\textsf{u}}^{\ula}(\bsig^{\textsf{u}}):= \dot{\Phi}_2\big(\pi(\bsig^{\textsf{u}})\big)^{\ula} ; \quad
\hat{\Phi}_{\textsf{u}}^{\ula}(\bsig^{\textsf{u}}):= \hat{\Phi}_2\big(\pi(\bsig^{\textsf{u}})\big)^{\ula}; \quad \bar{\Phi}_{\textsf{u}}^{\ula}(\bsig^{\textsf{u}}):= \bar{\Phi}_2\big(\pi(\bsig^{\textsf{u}})\big)^{\ula}.
\end{equation}
\begin{defn}[Union-free trees]
\label{def:union:tree}
Let $\vvv\in \FFF_2$ be a union-free component whose underlying graph structure is a tree. Analogously to \eqref{eq:def:col on freetree}, we can determine the union coloring 
\begin{equation}\label{eq:def:col:union:tree}
\bsig^{\textsf{u}}({\vvv}):= (\bsigma_{e}^{\textsf{u}}(\vvv))_{e\in E(\vvv)\sqcup \partial \vvv}
\end{equation}
on the edges of $\vvv$ by recursively applying \eqref{eq:union:coloring:recursive} from the boundary of $\vvv$. As before, we define an equivalence relation for the union-free components which are trees by $\vvv_1 \sim \vvv_2$ if and only if there exists an graph isomorphism from $\vvv_1$ to $\vvv_2$ that preserves $\bsig^{\textsf{u}}({\vvv_1}) = \bsig^{\textsf{u}}({\vvv_2})$. Then, we define the set $\uuu$ of \textbf{union-free trees} as the equivalence class $\uuu=[\vvv]$ under this equivalence relation and denote by $\FFF_2^{\tr}$ the set of union-free trees. Thus, $\bsig^{\textsf{u}}(\uuu):=\big(\bsigma_e^{\textsf{u}}(\uuu)\big)_{e\in E(\uuu)\sqcup \partial \uuu}$ is well-defined. The size of $\uuu$ in the first and second copy are denoted by $\bw_{\uuu}^{\lit,i}$ for $i=1,2$ (see \eqref{eq:def:size:union-comp:each:copy}). Moreover, its averaged weight $\bw_{\uuu}^{\ula}$ is defined as
\begin{equation}
\bw_{\uuu}^{\ula} := \prod_{v\in V(\uuu)} \Big\{\dot{\Phi}_{\textsf{u}}^{\ula}\big(\bsig_{\delta v}^{\textsf{u}}(\uuu)\big) \prod_{e \in \delta v} \bar{\Phi}_{\textsf{u}}^{\ula}\big(\bsigma_e^{\textsf{u}}(\uuu)\big)\Big\} \prod_{a \in F(\uuu)} \hat{\Phi}_{\textsf{u}}^{\ula} \big(\bsig_{\delta a}^{\textsf{u}}(\uuu)\big).
\end{equation}
\end{defn}

We conclude this section with the lemma which shows the relationship $w_{\ttt}^{\la}$ and $w^{\textnormal{com}}(\fff)^{\la}$ in the first moment, and $\bw_{\uuu}^{\ula}$ and $\bw^{\textnormal{com}}(\vvv)^{\ula}$ in the second moment. It will be used in Proposition \ref{prop:1stmo:B nt decomp} and Proposition \ref{prop:2ndmo:B nt decomp}. To do so, we define the set of labeled components for free components and union-free components.

\begin{defn}[labeled components]\label{def:embedding:number}
		For a free tree $\ttt\in \mathscr{F}_{\tr}$, we introduce the set $\LLL(\ttt)$, namely the set of \textbf{labeled components} $\ttt^{\lab}$  corresponding to $\ttt$, where $\ttt^{\lab}$ is obtained from $\ttt$ by adding additional labels on the half-edges and the full edges of $\ttt$ as follows: for each variable $v \in V(\ttt)$ (resp. $a\in F(\ttt)$), arbitrarily label half-edges adjacent to $v$ (resp. a) by $1,...,d$ (resp. $1,...,k$). Then, $\LLL(\ttt)$ is the set of isomorphism classes of resulting $\ttt^\lab$, where an isomorphism is a graph isomorphism which keeps all the labels on the half-edges and full edges consistent. For $\fff\in \FFF\setminus\FFF^{\tr}$, the labeled component $\fff^{\lab}$ corresponding to $\fff$ is defined the same except that we further choose a spanning tree $T$ of $\fff^{\textnormal{in}}$ and add another label to the edges of $T$ by `tree'. Similarly, define $\LLL(\fff)$ to be set of isomorphism classes of resulting $\fff^\lab$. The \textbf{embedding number} $J_\fff$ of a cyclic free component $\fff\in \FFF\setminus \FFF_{\tr}$ is defined as
		\begin{equation*}
		J_\fff := d^{1-v(\fff)} k^{-f(\fff)} \frac{|\mathscr{L}(\fff)|}{T_\fff},,
		\end{equation*}
		where $T_{\fff}$ is the number of spanning trees of $\fff$. The embedding number $J_{\fff}$ of a cyclic free component appears in Proposition \ref{prop:1stmo:B nt decomp}.
	\end{defn}

 \begin{defn}[Embedding number of union-free trees and union-free components]
For a union-tree $\uuu\in \FFF_2^{\tr}$, define its \textbf{embedding number} as
\begin{equation}\label{eq:def:embnum:pairr}
J_{\uuu} := d^{1-v(\uuu)} k^{-f(\uuu)}\prod_{v\in V(\uuu)} {d \choose \big\langle \bsig_{\delta v}^{\textsf{u}}(\uuu) \big\rangle } \prod_{a \in F(\uuu)} {k \choose \big\langle \bsig_{\delta a}^{\textsf{u}}(\uuu) \big\rangle },
\end{equation}
where the notations $\langle \bsig_{\delta v}^{\textsf{u}}(\uuu) \rangle$ and $\langle \bsig_{\delta a}^{\textsf{u}}(\uuu) \rangle$ are defined analogously to \eqref{eq:def:spin multi index}. The embedding number for a union-free component $\vvv\in \mathscr{F}_2$ is defined analogously to Definition \ref{def:embedding:number}: $\mathscr{L}(\vvv)$ denotes the set of all labeled components $\vvv^{\textnormal{lab}}$ corresponding to $\vvv$, where $\vvv^{\textnormal{lab}}$ is obtained from $\vvv$ by putting labels on the half-edges adjacent to variables (resp. clauses) $1,...,d$ (resp. $1,...,k$) and putting extra labels on the inner edges by `tree' according to a spanning tree. Here, the last labeling scheme is redundant if $\vvv$ is a union-free tree. The \textbf{embedding number} $J_\vvv$ of a cyclic union-free component $\vvv\in \FFF_2\setminus \FFF_2^{\tr}$ is defined as 
		\begin{equation*}
		J_\vvv := d^{1-v(\vvv)} k^{-f(\vvv)} \frac{|\mathscr{L}(\vvv)|}{T_\vvv}\,,
		\end{equation*}
  where $T_{\vvv}$ is the number of spanning tress of $\vvv$.
\end{defn}
	\begin{lemma}\label{lem:w vs wcom}
		Recall that $\ttt\in \FFF_{\tr}$ in Definition \ref{def:freetree:bdlab} was defined in terms of an equivalence class. Also recall $\LLL(\ttt)$ from Definition \ref{def:embedding:number}. For any $\la \in [0,1]$, we have
		\begin{equation*}
		d^{v(\ttt)-1} k^{f(\ttt)} J_\ttt w_{\ttt}^\lambda = \sum_{\ttt^\prime \in \FFF_{\tr}: \ttt = [\ttt^\prime]} |\mathscr{L}(\ttt^\prime)| w^{\textnormal{com}}(\ttt^\prime)^\lambda.
		\end{equation*}
		The analog also holds for the pair model. That is, for $\ula \in [0,1]^2$ and $\uuu\in \FFF_2^{\tr}$, we have
		\begin{equation*}
		d^{v(\uuu) -1} k^{f(\uuu)} J_{\uuu} \bw_{\uuu}^{\ula} = 
		\sum_{\uuu^\prime \in \FFF_2^{\tr} : \uuu = [\uuu^\prime]} |\mathscr{L}(\uuu^\prime)| \bw^{\textnormal{com}}(\uuu^\prime)^{\ula}.
		\end{equation*}
	\end{lemma}
	\begin{proof}[Proof of Lemma \ref{lem:w vs wcom}]
    We prove the first identity, since the second one for the union-free tree follows from the same argument. Let $\ttt$ be a free tree and $\fff\in \FFF$ be a free component such that $[\fff] = \ttt$. Note that
		\begin{equation*}
		\frac{w(\ttt)^\lambda }{w^{\textnormal{com}} (\fff)^\lambda }
		= 2^{k f(\ttt)} \prod_{a\in F(\ttt)} \hat{v}\big(\sig_{\delta a}(\ttt)\big),
		\end{equation*}
		For $\uL \in \{0,1\}^{E(\ttt)\sqcup \partial \ttt}$, write $\uL \sim \ttt$ if and only if there exists $\fff \in \FFF$ whose literal-labels are given by $\uL$ and $[\fff]=\ttt$. Then, from the definition of $\hat{v}$ in \eqref{eq:def:vhat:basic}, it is not hard to see that the number of $\uL \in \{0,1\}^{E(\ttt)\sqcup \dot{\partial} \ttt}$ such that $\uL \sim \ttt$ is $2^{k f(\ttt)} \prod_{a\in F(\ttt)} \hat{v}\big(\sig_{\delta a}(\ttt)\big)$. Thus, we have
		\begin{equation*}
		\frac{w(\ttt)^\lambda }{w^{\textnormal{com}}(\fff)^\lambda} =\Big| \big\{\uL\in \{0,1\}^{E(\ttt)\sqcup \dot{\partial} \ttt}: \uL \sim \ttt \big\}\Big|.
		\end{equation*}
		Hence, for the rest of the proof, we aim to show that
		\begin{equation*}
		\prod_{v\in V(\ttt)} {d \choose \big \langle \sig_{\delta v}(\ttt) \big \rangle } \prod_{a\in F(\ttt)} {k \choose \big\langle \sig_{\delta a}(\ttt) \big\rangle} \Big| \big\{\uL\in \{0,1\}^{E(\ttt)\sqcup \dot{\partial} \ttt}: \uL \sim \ttt \big\}\Big| = \sum_{\fff: [\fff]=\ttt} |\LLL(\fff)|.
		\end{equation*}
		For $\utau \in \Omega^l$ for $l\geq 1$, we adopt the notation $\{\utau \}$ to denote the multi-set $\{ \tau_1, \ldots, \tau_l \}$, that respects multiplicities but ignores the ordering. Then, we have the following elementary observation:
		\begin{equation*}
		{d \choose \big \langle \sig_{\delta v}(\ttt) \big \rangle }= \Big|\big\{ \utau \in \Omega^d: \{\utau \} = \{ \sig_{\delta v}(\ttt) \} \big\}\Big|;\qquad {k \choose \big\langle \sig_{\delta a}(\ttt) \big\rangle} = \Big|\big\{ \utau \in \Omega^k: \{\utau \} = \{ \sig_{\delta a}(\ttt) \} \big\}\Big|.
		\end{equation*}
		Thus, our goal is to construct one-to-one correspondence $\Phi$ between the sets
		\begin{equation*}
	\XX:=	\prod_{v \in V(\ttt) } \Big\{ \utau^v: \{ \utau^v \} = \{ \sig_{\delta v}(\ttt) \}  \Big\} \times \prod_{a \in F(\ttt)} \Big\{\utau^a: \{ \utau^a \} = \{ \sig_{\delta a}(\ttt) \} \Big\} \times \{\uL : \uL \sim \ttt \}\quad\textnormal{and}\quad \sqcup_{\fff: [\fff]=\ttt} \LLL(\fff).
		\end{equation*}
		From now on, we adopt the following notation: $e^{\shortparallel}$ denotes a half-edge and $\bar{e}^{\shortparallel}$ denotes the full edge containing $e^{\shortparallel}$. Also, $\delta^{\shortparallel}a$ (resp.~$\delta^{\shortparallel}v$) is the collection of half-edges adjacent to $a$ (resp.~$v$). Moreover, let $T$ be the underlying graph of $\ttt$, with no literal and spin information.
		
		Given $(\{\utau^a \}_{a\in F(\ttt)}, \{\utau^v \}_{v\in V(\ttt)}, \{\tL_e \}_{e\in E(\ttt)\sqcup\dot{\partial}\ttt}) \in \XX,$ choose $i_a : \delta^{\shortparallel} a \to [k]$ and $i_v: \delta^{\shortparallel} v \to [d]$ such that for all $e^{\shortparallel}\in \delta^{\shortparallel}a$ (resp.~$e^{\shortparallel} \in \delta^{\shortparallel} v$), $\tau^a_{i_a(e^{\shortparallel})} = \sigma_{\bar{e}^{\shortparallel}}(\ttt)$ (resp.~$\tau^v_{i_v(e^{\shortparallel})} = \sigma_{\bar{e}^{\shortparallel}}(\ttt)$). If $\big(\sigma_e(\ttt)\big)_{e\in \delta a}$ are all distinct, there exists unique $i_a$, but if some are equal there could be many $i_a$. For the graph $T$, label $e^{\shortparallel}\in \delta^{\shortparallel} a$ (resp.~$e^{\shortparallel}\in \delta^{\shortparallel}v$) with $i_a(e^{\shortparallel})$ (resp.~$i_v(e^{\shortparallel})$), and assign $\tL_e$ as a literal at edge $e$. Note that the spin labels at boundary edge adjacent to a clause of $T$, which are either $\bb_0$ or $\bb_1$, are determined by $\ttt$. This whole procedure leads to a labeled component $\ttt^{\textnormal{lab}} \in \sqcup_{\fff: [\fff]=\ttt} \LLL(\fff)$, and it is not hard to see that the resulting $\ttt^{\textnormal{lab}}$ does not depend on the choice of $\{i_a \}, \{i_v\}$, i.e. results in the same isomorphism class described in Definition \ref{def:embedding:number}. We denote this map by
		\begin{equation*}
		\Phi: (\{\utau^a \}_{a\in F(\ttt)}, \{\utau^v \}_{v\in V(\ttt)}, \{\tL_e \}_{e\in E(\ttt)} ) \mapsto \ttt^{\textnormal{lab}}
		\end{equation*}
 		We show $\Phi$ is a one-to-one correspondence by constructing its inverse $\Psi$: given $\ttt^{\textnormal{lab}}\in \sqcup_{\fff: [\fff]=\ttt} \LLL(\fff)$, define $\utau^a := \big( \sigma_{\bar{e}_1^{\shortparallel}}(\ttt),\ldots, \sigma_{\bar{e}_k^{\shortparallel}} (\ttt)\big)$ for $a\in F(\ttt)$, where $e_j^{\shortparallel}\in \delta^{\shortparallel} a$ is the half-edge labeled $j$ for $1\leq j \leq k$. We can define $\utau^v$ analogously, and the literal $\tL_e$ for $e\in E(\ttt)\sqcup \dot{\partial}\ttt$ can be read off from $\ttt^{\textnormal{lab}}$. This gives rise to $\Psi(\ttt^{\textnormal{lab}}):=(\{\utau^a \}_{a\in F(\ttt)}, \{\utau^v \}_{v\in V(\ttt)}, \{\tL_e \}_{e\in E(\ttt)} )$, and it does not depend on the choice of a representative in $\ttt^{\textnormal{lab}}$. The proof follows from the fact that the maps $\Psi \circ \Phi$ and $\Phi \circ \Psi$ are identities.		
	\end{proof}
\iffalse
\subsubsection{Averaging over the literals for the pair model}\label{subsubsec:model:avglit:2ndmo}

In this subsection, we continue the discussion from Section \ref{subsubsec:model:avglit:1stmo}. For a pair-simplified coloring $\bsig\in \Omega_{\fs, 2}^E$, we can write
\begin{equation*}
\bw_\GGG^{\textnormal{lit}}(\bsig) = \prod_{v\in V}\dot{\Phi}_2(\bsig_{\delta v}) \prod_{a\in F} \hat{\Phi}_2^{\textnormal{lit}}((\bsig + \uL)_{\delta a}) \prod_{e\in E} \bar{\Phi}_2(\bsigma_e),
\end{equation*}
\fi

\section{The first moment}\label{sec:1stmo}
	The goal of this section is to compute the first moment $\E \bZ_{\lambda^\star}$ and $\E \bZ_{\lambda^\star,s_n}$ up to the leading constant for $|s_n-s^\star|\leq n^{-2/3}$ (for the definition of $\la^\star$ and $s^\star$, see \eqref{eq:def:sstarandlambdastar} below). In Section \ref{subsec:1stmo:apriori}, we state \textit{a priori estimates}, which gives the first estimate on the number of large free components and cyclic free components. In Section \ref{subsec:1stmo:optimal}, we show that the number of free trees concentrates on an explicit value, which can be calculated from the \textit{Belief Propagation}(BP) fixed point. In Section \ref{subsec:1stmo:leading constant}, we finish the calculation of the first moment and prove Theorem \ref{thm1}-(a).
\subsection{A priori estimates}\label{subsec:1stmo:apriori}
To begin with, we decompose $\E \bZ_{\la}$ in terms of contributions from each ``local neighborhood profile", which is the same approach done in the previous works \cite{dss15ksat,dss16maxis,dss16,ssz22}. The \textit{coloring profile} is one of such notions which was introduced in \cite{ssz22}. Hereafter, $\mathscr{P}(\mathfrak{X})$ denotes the space of probability measures on $\mathfrak{X}$.
	
	\begin{defn}[coloring profile and the simplex of coloring profile; \cite{ssz22}, Definition 3.1 and 3.2]\label{def:empiricalssz}
		Given a \textsc{nae-sat} instance $\GGG$ and a coloring configuration $\sig \in \Omega^E $, the \textit{coloring profile} of $\underline{\sigma}$ is the triple $H[\underline{\sigma}]\equiv H\equiv (\dot{H},\hat{H},\bar{H}) $ defined as follows. 
		\begin{equation}\label{eq:def:H}
		\begin{split}
		\dot{H}\in \mathscr{P}(\Omega^d), \quad &\dot{H}(\underline{\tau})
		 \equiv
		  \big|\{v\in V: \underline{\sigma}_{\delta v}=\underline{\tau} \}\big| / |V| \quad
		  \textnormal{for all } \underline{\tau}\in \Omega^d;\\
		  \hat{H}\in \mathscr{P}(\Omega^k), \quad &\hat{H}(\underline{\tau})
		  \equiv
		  \big|\{a\in F: \underline{\sigma}_{\delta a}=\underline{\tau} \}\big| / |F| \quad
		  \textnormal{for all } \underline{\tau}\in \Omega^k;\\
		  \bar{H}\in \mathscr{P}(\Omega), \quad &\bar{H}(\tau)
		  \equiv
		  \big|\{e\in E: \sigma_e=\tau \} \big| / |E| \quad
		  \textnormal{for all } \tau \in \Omega.
		\end{split}
		\end{equation}
	A valid $H$ must satisfy the following compatibility equation:
	\begin{equation}\label{eq:compatibility:coloringprofile}
	       \frac{1}{d} \sum_{\utau \in \Omega^{d}}\dot{H}(\utau)\sum_{i=1}^{d}\one\{\tau_i=\tau\} = \bar{H}(\tau) = \frac{1}{k}\sum_{\utau \in \Omega^k}\hat{H}(\utau)\sum_{j=1}^{k}\one\{\tau_j = \tau \}\quad\textnormal{for all}\quad \tau \in \Omega
	\end{equation}
	The \textit{simplex of coloring profile} $\bDelta$ is the space of triples $H=(\dot{H},\hat{H},\bar{H})$ which satisfies the following conditions:
	\begin{enumerate}
	\item[$\bullet$] $\dot{H} \in \mathscr{P}(\textnormal{supp}\,\dot{\Phi}), \hat{H} \in \mathscr{P}(\textnormal{supp}\,\hat{\Phi})$ and $\bar{H} \in \mathscr{P}(\Omega)$.
	\item[$\bullet$] $\dot{H},\hat{H}$ and $\bar{H}$ satisfy \eqref{eq:compatibility:coloringprofile}.
	\item[$\bullet$] From the definition of $\bZ_{\la}$ in \eqref{eq:def:Z:lambda}, $\dot{H},\hat{H}$ and $\bar{H}$ satisfy $\max\{\bar{H}(\ff),\bar{H}(\rr)\} \leq \frac{7}{2^k}$.
	\end{enumerate}
	For $L <\infty$, we let $\bDelta^{(L)}$ be the subspace of $\bDelta$ satisfying the following extra condition:
	\begin{enumerate}[resume]
	    \item[$\bullet$] $\dot{H} \in \mathscr{P}(\textnormal{supp}\,\dot{\Phi}\cap \Omega_{L}^{d}), \hat{H} \in \mathscr{P}(\textnormal{supp}\,\hat{\Phi}\cap \Omega_L^{k})$ and $\bar{H} \in \mathscr{P}(\Omega_L)$.
	\end{enumerate}
	\end{defn}
    We remark that $\bDelta^{(L)}$ in Defintion \ref{def:empiricalssz} has an extra condition $\max\{\bar{H}(\ff),\bar{H}(\rr)\} \leq \frac{7}{2^k}$ compared to \cite[Definition 3.2]{ssz22}. In \cite{ssz22}, $\bDelta^{(L)}$ in Defintion \ref{def:empiricalssz} was denoted by $\bN_{\circ}$. Because the contribution to the first moment from $H$ such that $\max\{\bar{H}(\ff),\bar{H}(\rr)\} \geq \frac{7}{2^k}$ is exponentially small in $n$ (cf. Lemma \ref{lem:free:red}), we impose the condition $\max\{\bar{H}(\ff),\bar{H}(\rr)\} \leq \frac{7}{2^k}$ in the definition of $\bDelta^{(L)}$.

	Given a coloring profile $H\in \bDelta$, denote by $\bZ_{\lambda}^{\tr}[H]$ the contribution to $\bZ_{\lambda}^{\tr}$ from the coloring configurations whose coloring profile is $H$. That is, $\bZ_\lambda^{\tr}[H] := \sum_{\underline{\sigma}:\; H[\underline{\sigma}] = H} w^{\lit}(\underline{\sigma})^\lambda$. For $H \in \bDelta^{(L)}$, $\tZ_{\lambda}[H]$ is analogously defined. 
	%Note that $\bZ_{\lambda}^{\tr}[H], \tZ_{\lambda}[H]$ are positive only if $(n\dot{H},m\hat{H})$ is integer valued.
	In \cite{ssz22}, they showed that $\E \tZ_\lambda[H]$ for the \textit{L-truncated} coloring model can be written as the following formula, which is a result of Stirling's approximation:
	\begin{equation}\label{eq:1stmo dec by H}
	\begin{split}
	\E \tZ_\lambda[H] &= n^{O_{L}(1)} \exp\left\{n F_{\lambda,L}(H)\right\}\quad\textnormal{for}\\
	F_{\lambda,L}(H)&\equiv \bigg\langle \dot{H}, \log\Big(\frac{\dot{\Phi}^\lambda}{\dot{H}}\Big) \bigg\rangle
	+ \frac{d}{k}  \bigg\langle \hat{H}, \log\Big(\frac{\hat{\Phi}^\lambda}{\hat{H}}\Big) \bigg\rangle
	+ d \bigg\langle \bar{H}, \log\big(\bar{\Phi}^\lambda\bar{H}\big) \bigg\rangle
	\end{split}
	\end{equation}
	Unfortunately, this approach has several crucial limitations to apply to our setting:
	\begin{enumerate}
		\item Our goal is to calculate the \textit{untruncated} first moment $\E \bZ_{\lambda}$. Since the underlying spin system $\Omega$ is infinite, Stirling's approximation is inaccurate and the exponent in the polynomial correction term, $O_{L}(1)$ in the equation above, tends to infinity as $L$ tends to infinity.
		\item The contribution from frozen configurations having free cycles cannot be analyzed by the above formula.
	\end{enumerate}
	Instead, we decompose $\E \bZ_\lambda$ in terms of a different type of empirical measure. Rather than revealing information on $H$, we will record $n_\fff$, namely the number of free components $\fff\in \mathscr{F}$. For the variables, clauses and edges that are not included in a free component, we record their profile as in Definition \ref{def:empiricalssz}. To formalize this idea, define the collections $\dot{\partial}^\bullet, \hat{\partial}^\bullet$ of \textit{non-free colors} by
	\begin{equation}\label{eq:def:nonfreecol}
	\dot{\partial}^\bullet \equiv \{\rr_0,\rr_1,\bb_0,\bb_1 \}, \quad \hat{\partial}^\bullet \equiv \{\rr_0,\rr_1,\bb_0,\bb_1,\fs \},
	\end{equation}
	which are the colors that can be adjacent to a variable or a clause outside of the free components. Similarly, we define $\partial$ to be the non-free colors which can be boundary colors of free components:
		\begin{equation}\label{eq:model:bij:msg and col}
		\dot{\partial}:= \{\bb_0, \bb_1 \}, \quad \hat{\partial} := \{\fs \}, \quad \partial:=\dot{\partial}\sqcup \hat{\partial}.
		\end{equation}	
	Then, we have the following definition of \textit{free component profile} and \textit{boundary profile}.
	
	\begin{defn}[free component profile, boundary profile] \label{def:empirical:boundary}
	Given a \textsc{nae-sat} instance $\GGG$ and a valid component configuration $\csig \in \comp^E$, the \textit{boundary profile} of $\csig$ is the tuple $(B[\csig],\uh[\csig])\equiv(B,\uh) \equiv (\dot{B},\hat{B},\bar{B}, \uh)$, and the \textit{free component profile} is the sequence $(n_\fff[\csig])_{\fff\in \mathscr{F}}\equiv(n_\fff)_{\fff\in \mathscr{F}}$, defined as follows.
	\begin{enumerate}
	    \item[$\bullet$] For each $\fff\in\mathscr{F}$, $n_\fff[\csig]$ is the number of free component $\fff$ inside $(\GGG,\csig)$. Also, we denote its normalization by $p_{\fff}[\csig]:=\frac{n_{\fff}[\csig]}{n}$.
		\item[$\bullet$] $\dot{B},\hat{B},$ and $\bar{B}$ are measures on $(\dot{\partial}^\bullet)^d$, $(\hat{\partial}^\bullet)^k$ and $\hat{\partial}^\bullet$ respectively, defined by
		\begin{equation*}
		\begin{split}
		&\dot{B}(\underline{\tau})
		:=
		|\{v\in V: \csig_{\delta v}=\underline{\tau} \} | / |V| \quad
		\textnormal{for all } \underline{\tau}\in (\dot{\partial}^\bullet)^d;\\
		 &\hat{B}(\underline{\tau})
		:=
		|\{a\in F: \csig_{\delta a}=\underline{\tau} \} | / |F| \quad
		\textnormal{for all } \underline{\tau}\in (\hat{\partial}^\bullet)^k;\\
		 &\bar{B}(\tau)
		:=
		|\{e\in E: \csigma_e=\tau \} | / |E| \quad
		\textnormal{for all } {\tau}\in \hat{\partial}^\bullet.
		\end{split}
		\end{equation*}
		Hence the total mass of each $\dot{B},\hat{B},$ and $\bar{B}$ is at most $1$. Furthermore, $\uh:=(h(\circ), \{h(x)\}_{x\in \partial})$ records the total number of the free components and the total number of boundary colors adjacent to the free components, normalized by the number of variables. That is,
	\begin{equation}\label{eq:def:h}
	h(\circ) := \frac{1}{|V|}\sum_{\fff\in \mathscr{F}} n_\fff,\quad\textnormal{and}\quad h(x) := \frac{1}{|V|} \sum_{\fff\in\mathscr{F}} \eta_{\fff}(x)\,n_\fff, \quad \textnormal{for}\quad x\in \partial,
	\end{equation}
	where $\{\eta_\fff(x)\}_{x\in \partial}$ are defined in Definition \ref{def:freecomp:basic}. Note that a valid boundary profile $(B,h)$ must satisfy the following compatibility condition: for all $x\in \hat{\partial}^\bullet$,
	\begin{equation}\label{eq:def:compat:1stmo}
	\begin{split}
	\bar{B}(x)
	&=
	\frac{1}{d} \sum_{\underline{\sigma}\in(\dot{\partial}^\bullet)^d } \dot{B}(\underline{\sigma})  \sum_{i=1}^d \one\{\sigma_i =x \} + \frac{\one\{x\in \hat{\partial}\}  }{d} h(x)\\
	&=
	\frac{1}{k}
	\sum_{\underline{\sigma}\in (\hat{\partial}^\bullet)^k} \hat{B} (\underline{\sigma})
	\sum_{j=1}^k \one \{\sigma_j =x \}
	+
	\frac{\one\{x\in \dot{\partial} \}}{d}  h(x).
	\end{split}
	\end{equation}
	\end{enumerate}
	\end{defn}
    The following remark shows that the boundary profile $(B[\csig],\uh[\csig])$ is determined by the free component profile $\{n_{\fff}[\csig]\}_{\fff\in \FFF}$ if $\csig$ consists of free trees. It also introduces the notation $(n_{\ttt})_{\ttt\in \FFF_{\tr}}\sim (B,s)$.
	\begin{remark}\label{rem:compat:bdry:tree}
	\begin{enumerate}
	    \item 	If $\csig\in \comp^{E}$ does not contain any free cycles, it corresponds to a unique coloring $\sig\in \Omega^{E}$. In such a case, $h(\circ)$ can also be computed from $B$ by summing up Euler characteristics:
	\begin{equation}\label{eq:def:compat:1stmo:tree}
	 h(\circ)= 1-\langle \dot{B}, \one \rangle+ \frac{d}{k}(1-\langle \hat{B}, \one \rangle) -d(1-\langle \bar{B}, \one \rangle)
	\end{equation}
	where $\one$ denotes the all-ones vector. Thus, $\uh[\sig]:= \uh[\csig]$ is fully determined from $B[\sig]:= B[\csig]$ by \eqref{eq:def:compat:1stmo} and \eqref{eq:def:compat:1stmo:tree}. With a slight abuse of notation, we denote such relation by $\uh=\uh(B)=\big(h_x(B)\big)_{x\in \partial\sqcup\{\circ\}}$. Moreover, the free component profile is encoded by the \textit{free tree profile}, $(n_\ttt[\sig])_{\ttt\in \mathscr{F}_{\tr}}:= (n_\ttt[\csig])_{\ttt\in \mathscr{F}_{\tr}}$, since $n_\fff =0$ for $\fff\in \FFF\setminus \FFF_{\tr}$. Note that the boundary profile and free tree profile from valid coloring $\sig$ must be compatible, i.e. satisfy \eqref{eq:def:h}-\eqref{eq:def:compat:1stmo:tree}. We denote this relation by $(n_{\ttt})_{\ttt \in \mathscr{F}_{\tr}} \sim B$.
	\item 	Given a \textsc{nae-sat} instance $\GGG$ and a valid coloring configuration $\sig$, define $s[\sig]:=\frac{1}{n}\log w_{\GGG}^\lit(\sig)$. Then, by Lemmas \ref{lem:size:msg and trees} and \ref{lem:model:size:col}, we can express $s[\sig]$ by
	\begin{equation*}
	    s[\sig] =\frac{1}{n}\sum_{\ttt\in \FFF_{\tr}} n_\ttt[\sig]\log w_{\ttt}^\lit \equiv \sum_{\ttt\in \FFF_{\tr}} p_\ttt[\sig]s_\ttt^\lit,
	\end{equation*} 
	where $s_{\ttt}^{\lit} :=\log w_{\ttt}^\lit$. We write $(n_\ttt)_{\ttt \in \FFF_{\tr}} \sim (B,s)$ if $(n_\ttt)_{\ttt\in \FFF_{\tr}}\sim B$ and $\sum_{\ttt\in \FFF_{\tr}} n_{\ttt}s_{\ttt}^\lit\in [ns,ns+1)$.
	\end{enumerate}
	\end{remark}
	\begin{defn}[simplex of boundary profile]
	$\bDelta^{\textnormal{b}}$ is the space of boundary profiles $B$ with the following conditions.
		\begin{enumerate}
				\item[$\bullet$] $\dot{B}, \hat{B}$ and $\bar{B}$ are measures supported on $\textnormal{supp}~\dot{I}, \textnormal{supp}~\hat{v}$ and $\hat{\partial}^\bullet$ respectively.
				\item[$\bullet$] $\dot{B}, \hat{B}$ and $\bar{B}$ have total mass at most 1 and also satisfy the bound
				\begin{equation}\label{eq:B:red:free:small}
				\bar{B}(\{\rr_0,\rr_1\}) \le 7/2^k \quad\text{and}\quad \bar{B}(\{\rr_0,\rr_1,\bb_0,\bb_1\}) \ge 1- 7/2^k.
				\end{equation}
				\item[$\bullet$] There exists $\{h(x)\}_{x\in \partial} \in \R_{\geq 0}^{\partial}$ such that \eqref{eq:def:compat:1stmo} holds.\\
			Moreover, we denote by $\bDelta^{\textnormal{b}}_n$ the subspace of $\bDelta^{\textnormal{b}}$ satisfying the following extra condition.
				\item[$\bullet$] $\dot{B}, \hat{B}$ and $\bar{B}$ are integer multiples of $\frac{1}{n}, \frac{1}{m}$ and $\frac{1}{nd}$ respectively. That is,  \begin{equation*}
				\dot{B} \in \left(n^{-1}\mathbb{Z}_{\geq 0}\right)^{(\dot{\partial}^\bullet)^d}, \quad 
				\hat{B} \in \left(m^{-1}\mathbb{Z}_{\geq 0}\right)^{(\hat{\partial}^\bullet)^k}, \quad \textnormal{and} \quad 
				\bar{B} \in \left((nd)^{-1}\mathbb{Z}_{\geq 0}\right)^{\hat{\partial}^\bullet}.
				\end{equation*}
		\end{enumerate}
	\end{defn}
	The first step towards calculating the first moment is to give \textit{a priori} estimates that there are few large free components and cyclic free components. For a valid component configuration $\csig$, denote respectively the number of cyclic components and the number of \textit{multi-cyclic} edges by
	\begin{equation}\label{eq:def:cyclic:comp:multi:edge}
	    n_\cyc[\csig]
	    := \sum_{\fff \in \mathscr{F}, n_{\fff}[\csig]\neq 0 }\one\left\{\gamma(\fff)\geq 0\right\}~~~\textnormal{ and }~~~~e_\mult[\csig]:= \sum_{\fff \in \mathscr{F}, n_{\fff}[\csig]\neq 0}\gamma(\fff)\one\left\{\gamma(\fff)\geq 0\right\}.
	\end{equation}
	For $r>0$, let $\ee_{r}$ be the set of free component profiles obeying exponential decay of frequencies in its number of variables with rate $2^{-rk}$:
	\begin{equation}\label{eq:def:exp:decay:profile}
	    \ee_{r}:= \Big\{(n_\fff)_{\fff \in \FFF}: \sum_{\fff \in \FFF,v(\fff)=v} n_\fff \leq n2^{-rkv}, \forall v\geq 1\Big\}.
	\end{equation}
	In what follows, $\bZ^{\tr}_{\la}[(\ee_{r})^{\mathsf{c}}]$ denotes the contribution to $\bZ^{\tr}_{\la}$ from the $\sig \in \Omega^{E}$ such that $\big(n_{\ttt}[\sig]\big)_{\ttt\in \FFF_{\tr}}\notin \ee_{r}$. Other quantities are similarly defined. Proposition \ref{prop:1stmo:aprioriestimate} plays a crucial role in computing the first moment and its proof is presented in Appendix \ref{sec:app:apriori}.
	\begin{prop}\label{prop:1stmo:aprioriestimate}
	For $k\geq k_0, \lambda \in [0,1], L<\infty$ and $c\in [1, 3]$, the following holds.\footnote{$\frac{2}{3}$ in the exponent can be substituted by any $x \in (0,1)$, if we adjust $k_0$. For our purposes, $x \in (\frac{1}{2}, 1)$ will suffice.}
	\begin{enumerate}
    \item $\E \tZ_{\lambda}[(\ee_{\frac{1}{c+1}})^{\mathsf{c}}] \lesssim_{k} n^{-\frac{2}{3}c}\log n \E \tZ_{\lambda}$ and $\E \bZ^\tr_{\lambda}[(\ee_{\frac{1}{c+1}})^{\mathsf{c}}] \lesssim_{k} n^{-\frac{2}{3}c}\log n \E \bZ^\tr_{\lambda}$.
    \item $\E \bZ_\la[\exists \fff\in \FFF(\ux, \GGG),~~~~f(\fff)\geq v(\fff)+2]\lesssim_{k} n^{-2}\E\bZ_\la$.
    \item $\E \bZ_\la[(\ee_{\frac{1}{c+1}})^{\mathsf{c}} \quad\textnormal{and}\quad \forall \fff\in \FFF(\ux, \GGG),~~~~f(\fff)\leq v(\fff)+1]\lesssim_{k} n^{-\frac{2}{3}c}\log n\E \bZ_\la$.
	\end{enumerate}
	Moreover, there exists a universal constant $C$ such that for every $r,\gamma \in \Z_{\geq 0}$, the following holds.
	\begin{enumerate}[resume]
	\item $\E\bZ_{\lambda}[n_{\textnormal{cyc}}\geq r, e_{\textnormal{mult}}\geq \gamma\quad\textnormal{and}\quad \ee_{\frac{1}{4}}]\lesssim_{k} \frac{1}{r!}(\frac{Ck^2}{2^k})^{r}(\frac{C\log^{3} n}{n})^{\gamma}\E\bZ_{\lambda}^{\tr}$.
	\end{enumerate}
	\end{prop}
    We remark that in Proposition \ref{prop:1stmo:aprioriestimate}-(2), if $f(\fff)\geq v(\fff)+2$ holds, then $\fff$ has at least $3$ cycles. This is because each clause in $\fff$ has at least $2$ internal edges, so $2f(\fff)\leq e(\fff)$ holds. In Proposition \ref{prop:1stmo:aprioriestimate}-(3), the condition $f(\fff)\leq v(\fff)+1\,, \forall \fff\in \FFF(\ux, \GGG)$ is useful for the proof. Note that such condition and the condition $(n_{\fff})_{\fff\in\FFF}\in \ee_{r}$ imply that $(n_{\fff})_{\fff\in\FFF}$ satisfies an exponential decay in the number of clauses, where we replace $v(\fff)$ in \eqref{eq:def:exp:decay:profile} by $f(\fff)$, up to a multiplicative constant. As a corollary of Proposition \ref{prop:1stmo:aprioriestimate}, we have the following.
	\begin{cor}\label{cor:cyclic:contribution:1stmo}
	For $k\geq k_0$ and $\la \in [0,1]$, $\E \bZ_{\la} \asymp \E \bZ^{\tr}_\la$ holds. 
	\end{cor}

	\subsection{Optimal profiles}\label{subsec:1stmo:optimal}
		Denote by $\bZ_\lambda [B,\{n_\fff \}_{\fff \in \mathscr{F}}]$  the contribution to $\bZ_{\lambda}$ from component configuration $\sig \in \comp^{E}$ with boundary profile $B[\sig]=B$ and free component profile $\big(n_\fff[\sig]\big)_{\fff \in \mathscr{F}}=\big(n_\fff\big)_{\fff \in \mathscr{F}}$. The following proposition shows how to compute the cost of including free components inside a frozen configuration in the first moment.
	\begin{prop}\label{prop:1stmo:B nt decomp}
	For every $B\in \bDelta_n$ and $\{n_{\fff}\}_{\fff\in \mathscr{F}}\sim B$, we have
	\begin{equation}\label{eq:prod:form:1stmo}
	   \E \bZ_\lambda [B, \{n_\fff \}_{\fff\in \mathscr{F}}] = \frac{n! m!}{nd !} \frac{(nd\bar{B})!}{(n\dot{B})! (m\hat{B})!}\prod_{\sig\in (\hat{\partial}^\bullet)^k}\hat{v}(\sig)^{m\hat{B}(\sig)} \prod_{\fff\in \mathscr{F}}\left[ \frac{1}{n_\fff !} \Big(d^{e(\fff)-f(\fff)}k^{f(\fff)}J_\fff w_\fff^{\la} \Big)^{n_\fff}\right],
	\end{equation}
	where $w_{\fff}^{\la} \equiv w^{\textnormal{com}}(\fff)^{\la}$ if $\fff\in \FFF \backslash\FFF_{\tr}$ and $w_{\ttt}^{\la}\equiv w(\ttt)^{\la}$ if $\ttt\in \FFF_{\tr}$. Stirling approximation in $\frac{n! m!}{nd !} \frac{(nd\bar{B})!}{(n\dot{B})! (m\hat{B})!}$ in \eqref{eq:prod:form:1stmo} gives
	\begin{equation}\label{eq:prod:form:stir:1stmo}
		\E \bZ_\lambda [B, \{n_\fff \}_{\fff\in \mathscr{F}}] = \left(1+O_{k}\left(\frac{1}{n\kappa(B)}\right)\right)\frac{ e^{n\Psi_\circ(B)}}{p_\circ(n;B)} \prod_{\fff\in \mathscr{F}}\left[ \frac{1}{n_\fff !} \left(\left(\frac{e}{n}\right)^{\gamma(\fff)}J_\fff w_\fff^\lambda \right)^{n_\fff}\right],
	\end{equation}
	where $\kappa(B)\equiv \min_{\dot{B}(\sig)\neq 0,\hat{B}(\utau)\neq 0, \bar{B}(\sigma)\neq 0}\left\{\dot{B}(\sig),\hat{B}(\utau),\bar{B}(\sigma)\right\}$ for $B \in \bDelta$. $\Psi_\circ(B)$ and $p_{\circ}(n,B)$ are defined by
	\begin{equation}\label{eq:def:Psi:circ:p:circ:1stmo}
	\begin{split}
	\Psi_\circ(B)
	&\equiv
	\bigg\langle \dot{B}, \log\frac{1}{\dot{B}} \bigg\rangle
	+ \frac{d}{k}  \bigg\langle \hat{B}, \log\frac{\hat{v}}{\hat{B}} \bigg\rangle
	+ d \bigg\langle \bar{B}, \log{\bar{B}} \bigg\rangle;\\
	p_\circ(n,B)
	&\equiv
	\left\{\frac{\prod_{\dot{B}(\sig)\neq0}\dot{B}(\sig)\prod_{\hat{B}(\utau)\neq 0}\hat{B}(\utau)}{\prod_{\bar{B}(\sigma)\neq 0}\bar{B}(\sigma)}\right\}^{1/2}(2\pi n)^{\phi_1(B)/2}d^{\phi_2(B)/2}k^{\phi_3(B)/2},
	\end{split}
	\end{equation}
	where $\phi_1(B):= |\textnormal{supp}\dot{B}|+|\textnormal{supp}\hat{B}|-|\textnormal{supp}\bar{B}|-1,\phi_2(B):=|\textnormal{supp}\hat{B}|-|\textnormal{supp}\bar{B}|$, and $\phi_3(B):= 1-|\textnormal{supp}\hat{B}|$.
	\end{prop}
		\begin{proof}
	Given a valid component coloring $\csig \in \comp^E$ and \textsc{nae-sat} instance $\GGG$, we construct \textit{labeled} configuration $\lsig =(\lsigma_e)_{e\in E}$ by the algorithm described below. We will see that it will be useful when calculating $\E \bZ_\lambda [B, \{n_\fff \}_{\fff\in \mathscr{F}}]$.
	\begin{enumerate}[label=\textnormal{Step \arabic*:}]
	    \item If $\csigma_e \in \hat{\partial}^\bullet$, then set $\lsigma_e \equiv \csigma_e$.
	    \item If $\csigma_e \notin \hat{\partial}^\bullet$, denote by $\fff(e)$ (resp. $\fff^{\textnormal{in}}(e)$) the free component (resp. free piece) that contains $e$. Choose a spanning tree $\mathscr{T}$ of $\fff^{\textnormal{in}}(e)$. For each edge $\tilde{e}$ of $\fff^{\textnormal{in}}(e)$, add an additional label to $\tilde{e}$ by \textnormal{`tree'} if $\tilde{e}$ is contained in $\mathscr{T}$. Otherwise, add a label of \textnormal{`cycle'} to $\tilde{e}$.
	    \item Uniquely label half-edges of $\fff(e)$(including the boundary ones) by $1$ to $d$ for variable adjacent half-edges and $1$ to $k$ for clause adjacent half-edges, where labels respect the orderings of the half-edges in $\GGG$. This step depends solely on how $\fff(e)$ is embedded in $\GGG$.
	    \item Let $\fff^{\textnormal{lab}}(e)$ be the labeled component resulting from Step $2$ and $3$. Finally, set $\lsigma_e$ to be the isomorphism class of $\left(\fff^{\textnormal{lab}}(e), e\right)$, where an isomorphism is a graph isomorphism that respects all the labels of the edges, half-edges and also the marked edge $e$.
	\end{enumerate}
	Observe that $\mathscr{L}(\fff)$ (see Definition \ref{def:embedding:number}) is the set of \textit{labeled components} $\fff^{\textnormal{lab}}(e)$ corresponding to $\fff$, i.e. $\fff(e)=\fff$. We define $\mathscr{L} := \sqcup_{\fff \in \mathscr{F}} \mathscr{L}(\fff)$ and denote by $\Omega_{\textnormal{lab}}$ the set of all possible outputs $\lsigma_e$ from the algorithm above. We now gather important properties of the labeled configuration, which we detail below.
	\begin{itemize}
	    \item For a valid component configuration $\csig$ containing a free cycle, there are more than one output $\lsig$ from the algorithm above. This is because Step 2 \textit{chooses} a spanning tree. More precisely, writing $\lsig \sim \csig$ if $\lsig$ could be obtained from $\csig$ by the algorithm above,
	     \begin{equation}\label{eq:labeled:component}
	    |\{\lsig: \lsig\sim \csig\}| = \prod_{\fff \in \mathscr{F}} T_\fff^{n_{\fff}(\csig)},
	    \end{equation}
	    where $T_\fff$ denotes the number of spanning trees of $\fff^{\textnormal{in}}$.
	    \item Given $\fff^{\textnormal{lab}}\in \mathscr{L}$, let $V_{\fff^{\textnormal{lab}}},F_{\fff^{\textnormal{lab}}}$ and $E^{\frac{1}{2}}_{\fff^{\textnormal{lab}}}$ be the set of variables, clauses and half-edges of $\fff^{\textnormal{lab}}$ respectively. Observe that by Step $3$ of the algorithm uniquely determines $\lsig_{\delta v}\equiv \lsig_{\delta v}[\fff^{\textnormal{lab}}]$ and $\lsig_{\delta a}\equiv  \lsig_{\delta a}[\fff^{\textnormal{lab}}]$, for $v\in V_{\fff^{\textnormal{lab}}}$ and $a \in F_\fff^{\textnormal{lab}}$ respectively. If we denote by $e_i$ the half-edge adjacent to $v \in V_{\fff^{\textnormal{lab}}}$ with label $i$ for $1\leq i\leq d$, then $\lsig_{\delta v}=(\sigma^{\textnormal{lab},1}_{v},...,\sigma^{\textnormal{lab},d}_{v})$, where
	    \begin{equation*}
	        \sigma^{\textnormal{lab},i}_{v} \equiv
	        \begin{cases}
	        \textnormal{color of }e_i & e_i\textnormal{ is a boundary half-edge}\\
	        \textnormal{isomorphism class of } (\fff^{\textnormal{lab}},\bar{e}_i)&e_i\textnormal{ is a internal half-edge}
	        \end{cases}
	    \end{equation*}
	    In the above $\bar{e}_i$ is the unique full edge containing the half-edge $e_i$. Similarly, $\lsig_{\delta a}=(\sigma^{\textnormal{lab},1}_{a},...,\sigma^{\textnormal{lab},k}_{a})$ is defined. Note that this need not be the case for component configurations, because the order of the elements of $\csig_{\delta v}$ heavily depends on how $\fff$ is embedded in $\GGG$.
	    \item By Step 2 and Step 3 of the algorithm, if $v\neq v^\prime$ holds, where $v,v^\prime\in V_{\fff^{\textnormal{lab}}}$, and $a\neq a^\prime$ holds, where $ a, a^\prime\in F_{\fff^{\textnormal{lab}}}$, then $\lsig_{\delta v} \neq \lsig_{\delta v^{\prime}}$ and $\lsig_{\delta a} \neq \lsig_{\delta a^{\prime}}$ hold. Moreover, if $\sigma^{\textnormal{lab},i}_{v}\notin \hat{\partial}$ for some $v\in V_{\fff^{\textnormal{lab}}}, 1\leq i \leq d$, then there exists a unique $a \in F_{\fff^{\textnormal{lab}}}$ and $1\leq j \leq k$ such that $\sigma^{\textnormal{lab},i}_{v}=\sigma^{\textnormal{lab},j}_{a}$. This is because $\fff^{\textnormal{lab}}$ is a finite bipartite factor graph with labeled edges of a spanning tree and labeled half-edges.
	\end{itemize}
	Analogously to Definition \ref{def:empirical:boundary}, we can also define the boundary profile and labeled free component profile of of $\lsig$, which we denote by $B[\lsig]$ and $\left(n_{\fff^{\textnormal{lab}}}[\lsig]\right)_{\fff^{\textnormal{lab}}\in \mathscr{L}}$ respectively. Letting $w^{\textnormal{lit}}(\lsig)^{\lambda}:= w^{\textnormal{lit}}(\csig)^{\lambda}$, where $\csig$ is the unique valid component configuration such that $\lsig\sim \csig$(if there exists none, define $w^{\textnormal{lit}}(\lsig)^{\lambda}:=0$), define the partition function
	\begin{equation*}
	    \mathscr{Z}_{\lambda}[B,\left(n_{\fff^{\textnormal{lab}}}\right)_{\fff^{\textnormal{lab}}\in \mathscr{L}}]:= \sum_{\lsig\in \Omega_{\textnormal{lab}}^{E}} w^{\textnormal{lit}}(\lsig)^{\lambda} \one{\Big\{B[\lsig]=B, n_{\fff^{\textnormal{lab}}}[\lsig]=n_{\fff^{\textnormal{lab}}},\ \forall\fff^{\textnormal{lab}}\in \LLL \Big\}}.
	\end{equation*}
	By \eqref{eq:labeled:component}, we have the following relationship between $\E\bZ_{\lambda}$ and $\E\mathscr{Z}_{\lambda}$:
	\begin{equation}\label{eq:partition:label}
	    \E\bZ_{\lambda}[B,\left(n_{\fff}\right)_{\fff \mathscr{F}}]=\frac{\sum \E\mathscr{Z}_{\lambda}[B,\left(n_{\fff^{\textnormal{lab}}}\right)_{\fff^{\textnormal{lab}}\in \mathscr{L}}]}{\prod_{\fff \in \mathscr{F}} T_\fff^{n_{\fff}[\csig]}},\quad\textnormal{where the sum  is for} \sum_{\fff^{\textnormal{lab}}\in \mathscr{L}(\fff)}n_{\fff^{\textnormal{lab}}}=n_{\fff}, \forall \fff \in \mathscr{F}.
	\end{equation}
	Thus, we now aim to compute $\E \mathscr{Z}_{\lambda}[B,\left(n_{\fff^{\textnormal{lab}}}\right)_{\fff^{\textnormal{lab}}\in \mathscr{L}}]$ by a matching scheme: first, locate the spins adjacent to frozen variables and separating clauses, which have empirical counts $n\dot{B}$ and $m\hat{B}$ respectively. Next, give an ordered list of $1,...,n_\fff^{\textnormal{lab}}$ to the $n_\fff^{\textnormal{lab}}$ number of free components for each $\fff^{\textnormal{lab}}\in \mathscr{L}$. Then, for each variable $v\in V_{\fff^{\textnormal{lab}}}$ and clause $a\in F_{\fff^{\textnormal{lab}}}$ in the listed free component, locate $\lsig_{\delta v}[\fff^{\textnormal{lab}}]$ and $\lsig_{\delta a}[\fff^{\textnormal{lab}}]$. Finally, we match the half-edges between variables and clauses, having the same spin and also the same list, if they have one. There are $n_{\fff^{\textnormal{lab}}}!$ number of lists leading to the same $\lsig$, so altogether we have
	\begin{equation}\label{eq:compute:1stmo:labeled}
	    \E\mathscr{Z}_{\lambda}[B,\left(n_{\fff^{\textnormal{lab}}}\right)_{\fff^{\textnormal{lab}}\in \mathscr{L}}]=\frac{n! m!}{nd !} \frac{(nd\bar{B})!}{(n\dot{B})! (m\hat{B})!}\prod_{\sig_{\delta a}\in \hat{\partial}^{k}}\hat{v}(\sig_{\delta a})^{m\hat{B}(\sig_{\delta a})}\prod_{\fff^{\textnormal{lab}}\in \mathscr{L}}\left[\frac{1}{n_\fff^{\textnormal{lab}}!}(w_\fff^{\textnormal{lab}})^{\lambda n_\fff^{\textnormal{lab}}}\right],
	\end{equation}
	where we defined $(w_\fff^{\textnormal{lab}})^{\lambda}:= w^{\textnormal{com}}(\fff)^{\lambda}$ for the unique free component $\fff$ corresponding to $\fff^{\textnormal{lab}}$ and $\underline{x}!\equiv \prod_{i} x_i !$ for a vector $\underline{x}=(x_1,x_2,...,x_\ell)$. Therefore, pluggging in \eqref{eq:compute:1stmo:labeled} to \eqref{eq:partition:label} gives a multinomial sum and together with Lemma \ref{lem:w vs wcom} concludes the proof of \eqref{eq:prod:form:1stmo}.
	\end{proof}
	Having Proposition \ref{prop:1stmo:B nt decomp} in hand, we first aim to compute $\E\bZ_{\lambda}^{\tr}$ by summing up $\E\bZ_{\lambda}^{\tr}[B,(n_{\ttt})_{\ttt\in \FFF_{\tr}}]$: it turns out that in the summation, the major contribution comes from $B$ and $(n_{\ttt})_{\ttt\in \FFF_{\tr}}$ which are close to \textit{optimal boundary profile} $B^\star_{\lambda}$ and \textit{optimal free tree profile} $(n_{\ttt,\la}^\star)_{\ttt\in \FFF_{\tr}}$, defined in terms of the so called \textit{belief propagation}(\textsc{bp}) fixed point. To this end, we now define the \textsc{bp} functional for the coloring model, which was introduced in \cite[Section 5]{ssz22}. For more background on belief propagation, we refer to \cite[Chapter 14]{mm09}. For probability measures $\qdot,\qhat \in \PPP(\Omega_L)$, where $L<\infty$, let
	\begin{equation}\label{eq:pre:BP:1stmo}
	\begin{split}
	    &[\dot{\mathbf{B}}_{1,\lambda}(\qhat)](\sigma)\cong \bar{\Phi}(\sigma)^\lambda \sum_{\sig \in \Omega_L^{d}}\one\{\sigma_1=\sigma\}\dot{\Phi}(\sig)^{\lambda}\prod_{i=2}^{d}\qhat(\sigma_i)\\
	    &[\hat{\mathbf{B}}_{1,\lambda}(\qdot)](\sigma)\cong \bar{\Phi}(\sigma)^\lambda \sum_{\sig \in \Omega_L^{k}}\one\{\sigma_1=\sigma\}\hat{\Phi}(\sig)^{\lambda}\prod_{i=2}^{d}\qdot(\sigma_i),
	\end{split}
	\end{equation}
	where $\sigma \in \Omega_L$ and $\cong$ denotes equality up to normlization, so that the output is a probability measure. We denote by $\dot{\mathscr{Z}}\equiv\dot{\mathscr{Z}}_{\hat{q}} ,\hat{\mathscr{Z}}\equiv \hat{\mathscr{Z}}_{\dot{q}}$ the normalizing constants for \eqref{eq:pre:BP:1stmo}. Now, restrict the domain to the probability measures with \textit{one-sided} dependence, i.e. satisfying $\qdot(\sigma)=\dot{f}(\dot{\sigma})$ and $\qhat(\sigma)=\hat{f}(\hat{\sigma})$ for some $\dot{f} :\dot{\Omega}_L\to\R_{\geq 0}$ and $\hat{f} :\hat{\Omega}_L\to\R_{\geq 0}$. It can be checked that $\dot{\mathbf{B}}_{1,\lambda}, \hat{\mathbf{B}}_{1,\lambda}$ preserve the one-sided property, inducing
	\begin{equation*}
	    \dot{\textnormal{BP}}_{\lambda,L}:\PPP(\hat{\Omega}_{L}) \rightarrow \PPP(\dot{\Omega}_{L}),\quad\hat{\textnormal{BP}}_{\lambda,L}:\PPP(\dot{\Omega}_L) \rightarrow \PPP(\hat{\Omega }_L).
	\end{equation*}
    More precisely, for $\hat{q}\in \PPP(\hat{\Omega}_L)$ and $\dot{q} \in \PPP(\dot{\Omega}_L)$, define the probability measures $\dot{\textnormal{BP}}_{\la,L}(\hat{q})\in \PPP(\dot{\Omega}_L)$ and $\hat{\textnormal{BP}}_{\la,L}(\dot{q})\in \PPP(\hat{\Omega}_L)$ as follows. For $\dot{\sigma}\in \dot{\Omega}_L$ and $\hat{\sigma}\in \hat{\Omega}_L$, let
    \begin{equation}\label{eq:def:BP}
    \begin{split}
    &[\dot{\textnormal{BP}}_{\la,L}(\hat{q})](\dot{\sigma})=\big(\dot{\ZZZ}_{\hat{q}}\big)^{-1} \cdot \bar{\Phi}(\dot{\sigma}, \hat{\sigma}^\prime)^\lambda \sum_{\sig \in \Omega_L^{d}}\one\{\sigma_1=(\dot{\sigma},\hat{\sigma}^\prime)\}\dot{\Phi}(\sig)^{\lambda}\prod_{i=2}^{d}\hat{q}(\hat{\sigma}_i)\,,\\
    &[\hat{\textnormal{BP}}_{\la,L}(\dot{q})](\hat{\sigma})=\big(\hat{\ZZZ}_{\dot{q}}\big)^{-1} \cdot \bar{\Phi}(\dot{\sigma}^\prime, \hat{\sigma})^\lambda \sum_{\sig \in \Omega_L^{k}}\one\{\sigma_1=(\dot{\sigma}^\prime, \hat{\sigma})\}\hat{\Phi}(\sig)^{\lambda}\prod_{i=2}^{k}\dot{q}(\dot{\sigma}_i)\,,
    \end{split}
    \end{equation}
    where $\hat{\sigma}^\prime \in \hat{\Omega}_L$ and $\dot{\sigma}^\prime \in \dot{\Omega}_L$ are arbitrary with the only exception that when $\dot{\sigma}\in \{\rr,\bb\}$ (resp. $\hat{\sigma}\in \{\rr,\bb\}$), then we take $\hat{\sigma}^\prime = \dot{\sigma}$ (resp. $\dot{\sigma}^\prime = \hat{\sigma}$) so that the \textsc{rhs} above is non-zero. From the definition of $\dot{\Phi},\hat{\Phi}$, and $\bar{\Phi}$, it can be checked that the choices of $\hat{\sigma}^\prime \in \hat{\Omega}_L$ and $\dot{\sigma}^\prime \in \dot{\Omega}_L$ do not affect the values of the \textsc{rhs} above. The normalizing constants $\dot{\ZZZ}_{\hat{q}}$ and $\hat{\ZZZ}_{\dot{q}}$ are given by
    \begin{equation}\label{eq:BP:normalization}
 \begin{split}
 &\dot{\ZZZ}_{\hat{q}}\equiv\sum_{\dot{\sigma} \in \dot{\Omega}_L} \bar{\Phi}(\dot{\sigma}, \hat{\sigma}^\prime)^\lambda \sum_{\sig \in \Omega_L^{d}}\one\{\sigma_1=(\dot{\sigma},\hat{\sigma}^\prime)\}\dot{\Phi}(\sig)^{\lambda}\prod_{i=2}^{d}\hat{q}(\hat{\sigma}_i)\,,\\
 &\hat{\ZZZ}_{\dot{q}}\equiv \sum_{\hat{\sigma} \in \hat{\Omega}_L}\bar{\Phi}(\dot{\sigma}^\prime, \hat{\sigma})^\lambda \sum_{\sig \in \Omega_L^{k}}\one\{\sigma_1=(\dot{\sigma}^\prime, \hat{\sigma})\}\hat{\Phi}(\sig)^{\lambda}\prod_{i=2}^{k}\dot{q}(\dot{\sigma}_i)\,.
 \end{split}
 \end{equation}
Here, $\hat{\sigma}^\prime \in \hat{\Omega}_L$ and $\dot{\sigma}^\prime \in \dot{\Omega}_L$ are again arbitrary. We then define the \textit{Belief Propagation functional} by $\textnormal{BP}_{\lambda,L}:= \dot{\textnormal{BP}}_{\lambda,L}\circ \hat{\textnormal{BP}}_{\lambda,L}$. The untruncated BP map, which we denote by $\textnormal{BP}_{\lambda}:\PPP(\dot{\Omega}) \to \PPP(\dot{\Omega})$, is analogously defined, where we replace $\dot{\Omega}_L$(resp. $\hat{\Omega}_L$) with $\dot{\Omega}$(resp. $\hat{\Omega}$). Let $\mathbf{\Gamma}_C$ be the set of $\dot{q} \in \PPP(\dot{\Omega})$ such that 
	\begin{equation}\label{eq:def:bp:contract:set:1stmo}
	    \dot{q}(\dot{\sigma})=\dot{q}(\dot{\sigma}\oplus 1)\quad\text{for}\quad\dot{\sigma} \in \dot{\Omega },\quad\text{and}\quad \frac{\dot{q}(\rr)+2^k\dot{q}(\ff)}{C}\leq \dot{q}(\bb) \leq \frac{\dot{q}(\rr)}{1-C2^{-k}}.
	\end{equation}
	%where $\{\rr\}\equiv\{\rr_0,\rr_1\},\{\bb\}\equiv \{\bb_0,\bb_1\}$.
	\begin{prop}[\cite{ssz22}, Proposition 5.5]
	\label{prop:BPcontraction:1stmo}
	For $\lambda \in [0,1]$, the following holds:
	\begin{enumerate}
	    \item There exists a large enough universal constant $C$ such that the map $\textnormal{BP}\equiv\textnormal{BP}_{\lambda,L}$ has a unique fixed point $\dot{q}^\star_{\lambda,L}\in \mathbf{\Gamma}_C\cap \PPP(\dot{\Omega}_L)$. Moreover, if $\dotq \in \mathbf{\Gamma}_C\cap \PPP(\dot{\Omega}_L)$, $\textnormal{BP}\dotq \in \mathbf{\Gamma}_C\cap \PPP(\dot{\Omega}_L)$ holds with
	    \begin{equation}\label{eq:BPcontraction:1stmo}
	        ||\textnormal{BP}\dotq-\dotq^\star_{\lambda,L}||_1\lesssim k^2 2^{-k}||\dotq-\dotq^\star_{\lambda,L}||_1.
	    \end{equation}
	    The same holds for the untruncated BP, i.e. $\textnormal{BP}_{\la}$, with fixed point $\dot{q}^\star_{\lambda}\in \Gamma_C$. $\dot{q}^\star_{\la,L}$ for large enough $L$ and $\dot{q}^\star_{\la}$ have full support in their domains.
	    \item In the limit $L \to \infty$, $||\dot{q}^\star_{\lambda,L}-\dot{q}^\star_{\lambda}||_1 \to 0$.
	\end{enumerate}
	\end{prop}
	For $\dot{q} \in \PPP(\dot{\Omega})$, denote $\hat{q}\equiv \hat{\textnormal{BP}}\dot{q}$, and define $H_{\dot{q}}=(\dot{H}_{\dot{q}},\hat{H}_{\dot{q}}, \bar{H}_{\dot{q}})\in \bDelta$ by
	\begin{equation}\label{eq:H:q:1stmo}
	    \dot{H}_{\dot{q}}(\sig)=\frac{\dot{\Phi}(\sig)^{\lambda}}{\dot{\mathfrak{Z}}}\prod_{i=1}^{d}\hat{q}(\hat{\sigma}_i),\quad \hat{H}_{\dot{q}}(\sig)=\frac{\hat{\Phi}(\sig)^{\lambda}}{\hat{\mathfrak{Z}}}\prod_{i=1}^{k}\dot{q}(\dot{\sigma}_i),\quad \bar{H}_{\dot{q}}(\sigma)=\frac{\bar{\Phi}(\sigma)^{-\lambda}}{\bar{\mathfrak{Z}}}\dot{q}(\dot{\sigma})\hat{q}(\hat{\sigma}),
	\end{equation}
	where $\dot{\mathfrak{Z}}\equiv \dot{\mathfrak{Z}}_{\dot{q}},\hat{\mathfrak{Z}}\equiv\hat{\mathfrak{Z}}_{\dot{q}}$ and $\bar{\mathfrak{Z}}\equiv \bar{\mathfrak{Z}}_{\dot{q}}$ are normalizing constants.
	\begin{defn}[\cite{ssz22}, Definition 5.6]\label{def:opt:coloring:profile}
	    The \textit{optimal coloring profiles} for the truncated model and the untruncated model is the tuple $H^\star_{\lambda,L}=(\dot{H}^\star_{\lambda,L},\hat{H}^\star_{\lambda,L},\bar{H}^\star_{\lambda,L})$ and $H^\star_{\lambda}=(\dot{H}^\star_{\lambda},\hat{H}^\star_{\lambda},\bar{H}^\star_{\lambda})$, defined respectively by $ H^\star_{\lambda,L}:= H_{\dot{q}^\star_{\lambda,L}}$ and $H^\star_{\lambda}:=H_{\dot{q}^\star_{\lambda}}$.
	\end{defn}
	\begin{defn}[optimal boundary profile, free tree profile and weight]\label{def:opt:bdry:1stmo}
	The \textit{optimal boundary profile}, the \textit{optimal free tree profile} and the \textit{optimal weight} are defined by the following.
	\begin{itemize}
	    \item  The optimal boundary profile for the truncated model is the tuple $B^\star_{\lambda,L} \equiv(\dot{B}^\star_{\lambda,L},\hat{B}^\star_{\lambda,L},\bar{B}^\star_{\lambda,L})$, defined by restricting the optimal coloring profile to $(\dot{\partial}^\bullet)^d, (\hat{\partial}^\bullet)^k, \hat{\partial}^\bullet$:
	 \begin{equation}\label{eq:def:optimal:bdry}
	 \begin{split}
	     &\dot{B}^\star_{\lambda,L}(\sig):=  \dot{H}^\star_{\lambda,L}(\sig)\quad\textnormal{for}\quad \sig \in (\dot{\partial}^\bullet)^d\\
	     &\hat{B}^\star_{\lambda,L}(\sig) := \sum_{\utau \in \Omega^k, \utau_{\fs}=\sig} \hat{H}^\star_{\lambda,L}(\utau)\quad\textnormal{for}\quad \sig \in (\hat{\partial}^\bullet)^k
	     \\
	     &\bar{B}^\star_{\lambda,L}(\sigma):= \sum_{\tau \in\Omega, \tau_{\fs}= \sigma} \bar{H}^\star_{\lambda,L}(\tau)\quad\textnormal{for}\quad \sigma \in \hat{\partial}^\bullet,
	 \end{split}
	 \end{equation}
	 where $\tau_{\fs}$ is defined by the simplified coloring of $\tau\in \Omega$, where $\tau_{\fs}:=\tau$, if $\hat{\tau}\neq \fs$, and $\tau_{\fs}:=\fs$, if $\hat{\tau}=\fs$. $\utau_{\fs}$ is the coordinate-wise simplified coloring of $\utau$. The optimal boundary profile for the untruncated model is defined analogously by dropping the subscript $L$ in \eqref{eq:def:optimal:bdry}. Recalling Remark \ref{rem:compat:bdry:tree}, we denote $\uh^\star_{\lambda,L}:= \uh(B^\star_{\lambda,L})$ and $\uh^\star_{\lambda}:= \uh(B^\star_{\lambda})$.
	 
	 \item The (normalized) optimal free tree profile $(p_{\ttt, \lambda,L}^\star)_{\ttt\in \FFF_{\tr}}$ for the truncated model is defined as follows. Recall the normalizing constants $\dot{\mathscr{Z}}^\star \equiv \dot{\mathscr{Z}}_{\hat{q}^\star_{\lambda,L}}, \hat{\mathscr{Z}}\equiv \hat{\mathscr{Z}}_{\dot{q}^\star_{\lambda,L}}$ for the BP map in \eqref{eq:pre:BP:1stmo}, where $\hat{q}^\star_{\lambda,L}\equiv \hat{\textnormal{BP}}\dot{q}^\star_{\lambda,L}$, and $\bar{\mathfrak{Z}}^\star \equiv\bar{\mathfrak{Z}}_{\dot{q}^\star_{\lambda,L}}$ in \eqref{eq:H:q:1stmo}. Writing $\dot{q}^\star=\dot{q}^\star_{\lambda,L}$ and $\hat{q}^\star=\hat{q}^\star_{\lambda,L}$, define
	 \begin{equation}\label{eq:optimal:tree:1stmo}
	     p_{\ttt, \lambda,L}^\star := \frac{J_{\ttt} w_{\ttt}^\lambda}{\bar{\mathfrak{Z}}^\star (\dot{\ZZZ}^\star)^{v(\ttt)}(\hat{\ZZZ}^\star)^{f(\ttt)}}\dot{q}^\star(\bb_0)^{\eta_{\ttt}(\bb_0)+\eta_{\ttt}(\bb_1)}(2^{-\lambda}\hat{q}^\star(\fs))^{\eta_{\ttt}(\fs)},
	 \end{equation}
	 for $\ttt\in \FFF$ with $v(\ttt)\leq L$. The optimal free tree profile $(p_{\ttt, \lambda}^\star)_{\ttt\in\FFF}$ for the untruncated model is defined by the same equation \eqref{eq:optimal:tree:1stmo} with $\bar{\mathfrak{Z}}^\star, \dot{\ZZZ},\hat{\ZZZ},\dot{q}^\star$ and $\hat{q}^\star$ for the untruncated model.
	 \item The optimal weight $s^\star_{\lambda,L}$ for the $\lambda$-tilted $L$-truncated model is defined by the weight of a coloring configuration having optimal free tree profile. That is, for $p^\star_{\ttt}=p^\star_{\ttt,\lambda,L}$,
	\begin{equation}\label{eq:def:slambda}
	s^\star_{\lambda,L} :=\sum_{\ttt\in\mathscr{F}_{\textsf{tr}}} p^\star_{\ttt} \log w_\ttt^\lit=\sum_{\ttt\in\mathscr{F}_{\textsf{tr}}} p^\star_{\ttt} s_\ttt^\lit. 
	\end{equation} 
	The optimal weight $s^\star_{\lambda}$ for the untruncated model is defined by the same equation \eqref{eq:def:slambda}, but with $p^\star_{\ttt}$ for the untruncated model.
	\end{itemize}
	\end{defn}
	In Lemma \ref{lem:compat:optfr:optbd:1stmo} of Appendix \ref{sec:appendix:Properties of BP fixed point}, we show that $p^\star_{\ttt,\la,L}$ and $B^\star_{\la,L}$ are compatible in the sense that they satisfy \eqref{eq:def:h} and \eqref{eq:def:compat:1stmo}. The next proposition shows that the most of the contribution to the first moment comes from the boundary profiles and weights close to their optimal values. Its proof is done by the \textit{resampling method}, which is presented in Section \ref{subsec:resampling}.
	\begin{prop}\label{prop:maxim:1stmo}
	Fix $\lambda \in [0,1]$ and large enough $L\geq L_0(\lambda,d,k)$. For any $\delta>0$, there exists $c(\delta)=c(\delta,\lambda,L,d,k)>0$ such that for $n\geq n_0(\delta,\la,L,d,k)$,
	\begin{equation}\label{eq:prop:maxim:1stmo:truncated}
	    \E \bZ^{(L),\tr}_\lambda\left[||(B,s)-(B^\star_{\lambda,L},s^\star_{\la,L})||_1 >\delta\quad\text{and}\quad (n_{\ttt})_{\ttt\in \FFF_{\tr}}\in \ee_{\frac{1}{4}} \right]\leq e^{-c(\delta)n} \E \bZ^{(L),\tr}_\lambda.
	\end{equation}
	The same holds for the untruncated model, namely for any $\delta>0$, there exists $c(\delta)=c(\delta,\lambda,d,k)>0$ such that for $n\geq n_0(\delta,\la,d,k)$
	\begin{equation}\label{eq:prop:maxim:1stmo:untruncated}
	    \E \bZ^{\tr}_\lambda\left[||(B,s)-(B^\star_{\lambda},s^\star_{\la})||_1 >\delta\quad\text{and}\quad(n_{\ttt})_{\ttt\in \FFF_{\tr}}\in \ee_{\frac{1}{4}}\right]\leq e^{-c(\delta)n} \E \bZ^{\tr}_\lambda.
	\end{equation}
	\end{prop}
    We remark that the result \eqref{eq:prop:maxim:1stmo:truncated} for the truncated model is a consequence of \cite[Proposition 3.4]{ssz22} (see also Remark \ref{remark:resampling:diff:ssz} below). However, the untruncated model needs much more careful analysis, which is done in Section \ref{subsec:resampling}.

	We now define the \textit{optimal rescaling factor} $\utheta^\star_{\lambda,L} \equiv \utheta^\star\equiv (\theta^\star_{\circ}, \{\theta^\star_{x}\}_{x \in\partial}, \theta^\star_{s}) \in \R^{|\partial|+2}$ for the truncated model as follows. $\dot{\ZZZ}^\star, \hat{\ZZZ}^\star, \bar{\mathfrak{Z}}^\star, \dot{q}^\star$ and $\hat{q}^\star$ below are for the $\lambda$-tilted and $L$-truncated model.
	\begin{equation}\label{eq:def:opt:theta}
	\begin{split}
	    \theta^\star_{\circ}&:= \log\left(\frac{(\dot{\ZZZ}^\star)^{\frac{k}{kd-k-d}}(\hat{\ZZZ}^\star)^{\frac{d}{kd-k-d}}}{\bar{\mathfrak{Z}}^\star}\right);\quad\quad \theta^\star_{\bb_0}\equiv\theta^\star_{\bb_1} := \log\left(\frac{\dot{q}^\star(\bb_0)}{(\dot{\ZZZ}^\star)^{\frac{1}{kd-k-d}}(\hat{\ZZZ}^\star)^{\frac{d-1}{kd-k-d}}}\right);\\
	    \theta^\star_{\fs} &:= \log\left(\frac{2^{-\la}\hat{q}^\star(\fs)}{(\dot{\ZZZ}^\star)^{\frac{k-1}{kd-k-d}}(\hat{\ZZZ}^\star)^{\frac{1}{kd-k-d}}}\right);\quad\quad\quad \theta^\star_s:=0.
	\end{split}
	\end{equation}
	The optimal rescaling factor $\utheta^\star_{\lambda}$ for the untruncated model is defined by \eqref{eq:def:opt:theta} with $\dot{\ZZZ}^\star, \hat{\ZZZ}^\star, \bar{\mathfrak{Z}}^\star, \dot{q}^\star$ and $\hat{q}^\star$ for the untruncated model. The optimal rescaling factor $\utheta^\star$ was designed to satisfy
	\begin{equation}\label{eq:theta:opt:tree:prob}
	    J_{\ttt}w_{\ttt}^\lambda \exp\left(\langle\utheta^\star,\boeta_{\ttt} \rangle\right) = p^\star_{\ttt}
	\end{equation}
	for both the untruncated and the truncated model, where $\boeta_{\ttt}:= (\eta_{\ttt}(\circ),\{\eta_{\ttt}(x)\}_{x\in\partial},\eta_{\ttt}(s))$ with $\eta_{\ttt}(\circ):=1$ and $\eta_{\ttt}(s):= s_{\ttt}^\lit$. Here, $\{\eta_{\ttt}(x)\}_{x\in\partial}$ was defined in \eqref{eq:def:eta}. Hence, recalling the definition of $s^\star_{\lambda,L}$ and $s^\star_{\lambda}$ in \eqref{eq:def:slambda}, Lemma \ref{lem:compat:optfr:optbd:1stmo} shows
	\begin{equation}\label{eq:grad:psi:optimal}
	\begin{split}
	    &\nabla \psi_{\lambda,L}(\utheta^\star_{\lambda,L})
	    =(\uh^\star_{\lambda,L},s^\star_{\lambda,L})\quad\textnormal{where}\quad \psi_{\lambda,L}(\utheta) := \sum_{\ttt\in \FFF_{\tr}:v(\ttt) \leq L}J_{\ttt}w_{\ttt}^{\lambda}\exp\left(\langle \utheta, \boeta_{\ttt}\rangle\right), \quad \utheta\in\R^{|\partial|+2};\\
	    &\nabla \psi_{\lambda}(\utheta^\star_{\lambda})
	    =(\uh^\star_{\lambda},s^\star_{\lambda}),\quad\textnormal{where}\quad \psi_{\lambda}(\utheta) := \sum_{\ttt\in \FFF_{\tr}}J_{\ttt}w_{\ttt}^{\lambda}\exp\left(\langle \utheta, \boeta_{\ttt}\rangle\right), \quad \utheta\in\R^{|\partial|+2}.
	\end{split}
	\end{equation}
	We also consider an analog of \eqref{eq:grad:psi:optimal} for $\utheta^{-}\in \R^{|\partial|+1}$, where we write $\utheta=(\utheta^{-},\theta_{s})$. That is, define
	\begin{equation*}
	    \psi_{\lambda,L}^{-}(\utheta^{-}):= \psi_{\lambda,L}(\utheta^{-},0)\quad\textnormal{and}\quad\psi_{\lambda}^{-}(\utheta^{-}):= \psi_{\lambda}(\utheta^{-},0).
	\end{equation*}
	Then, since $\theta^\star_s\equiv 0$ for both the truncated and the untruncated model, we have
	\begin{equation*}
	    \nabla \psi_{\lambda,L}^{-}(\utheta^{\star,-}_{\lambda,L})=h^\star_{\lambda,L}\quad\textnormal{and}\quad \nabla \psi_{\lambda}^{-}(\utheta^{\star,-}_{\lambda})=h^\star_{\lambda,L}.
	\end{equation*}
	By perturbative analysis, we have the next lemma.
	\begin{lemma}\label{lem:exist:theta}
	For $\delta >0$, denote the $\delta$-neighborhood around $(B^\star_{\la,L},s^\star_{\la,L})$ and $B^\star_{\la,L}$ by \begin{equation}\label{eq:def:BB:delta}
	\begin{split}
	     &\BB_{\la,L}(\delta) := \{(B,s)\in \bDelta^{\textnormal{b}}\times \R_{\geq 0}:||(B,s)-(B^\star_{\lambda,L},s^\star_{\la,L})||_1 \leq \delta\};\\
	     &\BB^{-}_{\la,L}(\delta) := \{B\in \bDelta^{\textnormal{b}}:||B-B^\star_{\lambda,L}||_1 \leq \delta\}.
	\end{split}
	\end{equation}
	For sufficiently large $L$($L\geq d$ suffices), there exist $\delta_0\equiv \delta_0(\lambda,L,d,k)>0, \utheta_{\la,L}:\BB_{\la,L}(\delta_0) \to\R^{|\partial|+2}$, and $\utheta^{-}_{\la,L}:\BB^{-}_{\la,L}(\delta_0)\to\R^{|\partial|+1}$, such that the following properties hold.
	\begin{itemize}
	    \item $\utheta_{\lambda,L}(B^\star_{\lambda,L},s^\star_{\lambda,L})=\utheta^\star_{\lambda,L}$ and $\utheta^{-}_{\lambda,L}(B^\star_{\lambda,L})=\utheta^{\star,-}_{\lambda,L}$.
	    \item $\nabla \psi_{\lambda,L}\left(\utheta_{\lambda,L}(B,s)\right) = \left(\uh(B),s\right)$ and $\nabla \psi_{\lambda,L}^{-}\left(\utheta^{-}_{\lambda,L}(B)\right) = \uh(B)$.
	    \item $\utheta_{\la,L}(\cdot)$ and $\utheta^{-}_{\la,L}(\cdot)$ are differentiable in the interior of their domains.
	\end{itemize}
	The analog for the untruncated model also holds: define $\BB_{\la}(\delta)$ and $\BB^{-}_\la(\delta)$ analogously to \eqref{eq:def:BB:delta}, where the subscript $L$ is dropped. Then, there exist $\delta_0(\lambda,d,k)>0$, $\utheta_{\la}:\BB_{\la}(\delta_0)\to \R$ and $\utheta_{\la}^{-}:\BB^{-}_{\la}(\delta_0)\to \R$ such that the same properties as above hold with subscript $L$ dropped.
	\end{lemma}
	\begin{proof}
	We consider the untruncated model throughout the proof. The result for truncated model with sufficiently large $L$ follows by the exact same argument. Since $B \to \uh(B)$ is a linear projection, it is differentiable. Thus, by \eqref{eq:grad:psi:optimal} and implicit function theorem, it suffices to show that $\nabla^2\psi_{\lambda}(\utheta^\star_{\lambda})\succ 0$ and $\nabla^2\psi^{-}_{\lambda}(\utheta^{\star,-}_{\lambda})\succ 0$. Also, because $\nabla^2\psi^{-}_{\lambda}(\utheta^{\star,-}_{\lambda})$ is a submatrix of $\nabla^2\psi_{\lambda}(\utheta^\star_{\lambda})$, it suffices to show the former inequality. For $v=(v_x)_{x\in \partial \sqcup\{\circ,s\}} \in \R^{|\partial|+2}$, we can use \eqref{eq:theta:opt:tree:prob} to compute
	\begin{equation*}
	    v^{T}\nabla^2\psi_{\lambda}(\utheta^\star_{\lambda})v=\sum_{\ttt \in \FFF_{\tr}}p^\star_{\ttt}\left(\sum_{x\in \partial \sqcup\{\circ, s\}}\eta_{\ttt}(x)v_x\right)^2.
	\end{equation*}
	Note that $p^\star_{\ttt}>0$ and there exists $A\subset \FFF$ with $A=|\partial|+2$ such that $\{\boeta_{\ttt}: \ttt\in A\}$ is linearly independent (we leave it as an exercise to the reader to find such a subset of free trees). Therefore, the right hand side of the equation above is positive for $v\neq 0$.
	\end{proof}
    Later, we will use $\utheta_\la$ to compute $\E \bZ_{\la,s}^\tr$ and $\utheta^{-}_\la$ to compute $\E\bZ^\tr_\la$. The next lemma shows that the optimal free tree profile decays exponentially in the number of variables.
	\begin{lemma}\label{lem:opt:tree:decay}
	Fix any $\lambda \in [0,1]$ and sufficiently large $L$. Then,
	\begin{equation}\label{eq:lem:opt:tree:decay}
	    \sum_{\ttt\in \FFF_{\tr}:v(\ttt)=v}p^\star_{\ttt,\lambda,L} 
	    \leq 2^{-kv/2}\quad\textnormal{for}\quad 1\leq v \leq L\quad\textnormal{and}
	    \sum_{\ttt\in \FFF_{\tr}:v(\ttt)=v}p^\star_{\ttt,\lambda} 
	    \leq 2^{-kv/2}\quad\textnormal{for}\quad v\geq 1.
	    \footnote{Modifying $k_0$, $\frac{1}{2}$ can be replaced by any $x\in (0,1)$.}
	\end{equation}
	\end{lemma}
	\begin{proof}
    We consider the untruncated model throughout the proof. The result for the truncated model with sufficiently large $L$ follows by the exact same argument. Fix $\lambda \in [0,1]$ and $v_0\geq 1$. Suppose by contradiction that there exists $\eps>0$ such that $ \sum_{\ttt\in\FFF_{\tr}:v(\ttt)=v_0}p^\star_{\ttt,\lambda} 
	    >(1+\eps) 2^{-kv_0/2}$ holds. Then, recalling $\delta_{0}$ and $\utheta^{-}(B)\equiv \utheta^{-}_{\lambda}(B)$ from Lemma \ref{lem:exist:theta}, \eqref{eq:theta:opt:tree:prob} and the continuity of $B\to\utheta(B)$ show that there exists some $\delta_1<\delta_0$ such that
	\begin{equation}\label{eq:lem:contrad:tree:decay}
	    \sum_{v(\ttt)=v_0} J_{\ttt}w_{\ttt}^\lambda \exp\left(\langle \utheta^{-}(B),\boeta^{-}_{\ttt}\rangle\right)>(1+\eps)2^{-kv_0/2}
	\end{equation}
	holds for $||B-B^\star_{\lambda}||_1<\delta_1$, where we wrote $\boeta_{\ttt}=\left(\boeta_{\ttt}^{-},\eta_\ttt(s)\right)$. To begin with, Proposition \ref{prop:1stmo:B nt decomp} shows that for $B\in \bDelta_n$ with $||B-B^\star_{\lambda}||_1\leq \delta_1$, we have
	\begin{equation}\label{eq:tree:decay:technical:1}
	\begin{split}
	    &\frac{\E \bZ^{\tr}_{\lambda}\left[B\quad\text{and}\quad \sum_{\ttt \in \mathscr{F}_{\tr},v(\ttt)=v} n_{\ttt} \leq n2^{-kv/2},~~~~~\forall v\geq 1 \right]}{\E \bZ^{\tr}_{\lambda}[B]}\\
	    &= \frac{\P_{\utheta^{-}(B)}\left(\sum_{i=1}^{nh_{\circ}(B)}\boeta^{-}_{X_i}=n\uh(B)\quad\text{and}\quad \sum_{i=1}^{nh_{\circ}(B)}\one(v(X_i)=v)\leq n2^{-kv/2},~~~~~~\forall v\geq 1\right)}{\P_{\utheta^{-}(B)}\left(\sum_{i=1}^{nh_{\circ}(B)}\boeta^{-}_{X_i}=n\uh(B)\right)},
	\end{split}
	\end{equation}
	where $\uh(B)=\big(h_x(B)\big)_{x\in \partial\sqcup\{\circ\}}$ is from Remark \ref{rem:compat:bdry:tree}, and $\P_{\utheta^{-}(B)}$ is taken with respect to i.i.d. random free trees $X_1,...,X_{nh_{\circ}(B)}\in \FFF_{\tr}$ with distribution
	\begin{equation}\label{eq:def:rescaling:prob}
	    \P_{\utheta^{-}(B)}(X_i=\ttt):=\frac{J_{\ttt}w_{\ttt}^\lambda \exp\left(\langle\utheta^{-}(B),\boeta^{-}_{\ttt}\rangle\right)}{h_{\circ}(B)}.
	\end{equation}
	Note that by Lemma \ref{lem:exist:theta}, $\E_{\utheta^{-}(B)}[\boeta^{-}_{X_i}]=\frac{\uh(B)}{h_{\circ}(B)}$ holds, so local central limit theorem(\textsc{clt}) implies (e.g. see Theorem 3.1 of \cite{Borokov17})
	\begin{equation}\label{eq:tree:decay:technical:2}
	    \P_{\utheta^{-}(B)}\left(\sum_{i=1}^{nh_{\circ}(B)}\boeta^{-}_{X_i}=n\uh(B)\right)\geq C n^{-|\partial|/2},
	\end{equation}
	for $||B-B^\star_{\lambda}||_1<\delta_1$ and $C=C(\delta_1,\lambda,d,k)>0$. On the other hand, by \eqref{eq:lem:contrad:tree:decay},
	\begin{equation}\label{eq:tree:decay:technical:3}
	\begin{split}
	&\P_{\utheta^{-}(B)}\left(\sum_{i=1}^{nh_{\circ}(B)}\boeta^{-}_{X_i}=n\uh(B)\quad\text{and}\quad\sum_{i=1}^{nh_{\circ}(B)}\one\left(v(X_i)=v\right)\leq n2^{-kv/2},~~~~~~\forall v\geq 1\right)\\
	&\leq \P_{\utheta^{-}(B)}\left(\sum_{i=1}^{nh_{\circ}(B)}\one\left(v(X_i)=v_0\right)\leq n2^{-kv_0/2}\right)\leq \exp\Big(-\frac{n\eps^2 2^{-kv_0/2}}{2(1+\eps)}\Big),
	\end{split}
	\end{equation}
	where the last bound is due to Chernoff bounds for binomial random variables. However, (4) of Proposition \ref{prop:1stmo:aprioriestimate} with $c=1$, and Proposition \ref{prop:maxim:1stmo} altogether imply that
	\begin{equation}\label{eq:tree:decay:technical:0}
	   \E \bZ^{\tr}_\lambda\left[||B-B^\star_{\lambda}||_1\leq \delta_1\quad\text{and} \sum_{\ttt \in \mathscr{F}_{\tr},v(\ttt)=v} n_{\ttt} \leq n2^{-kv/2},~~~~~~\forall v\geq 1\right]\geq \left(1-O_{k}(n^{-\frac{2}{3}}\log n)\right)\E\bZ^{\tr}_\lambda.
	\end{equation}
	Combining \eqref{eq:tree:decay:technical:1}, \eqref{eq:tree:decay:technical:2} and \eqref{eq:tree:decay:technical:3} contradict \eqref{eq:tree:decay:technical:0} for large enough $n$.
	\end{proof}
	Taking advantage of the previous lemma, the next lemma shows some uniform convergence properties of $\psi_{\la,L}(\cdot), \utheta_{\la,L}(\cdot)$ and $\utheta_{\la,L}^{-}(\cdot)$ as $L \to \infty$.
	\begin{lemma}\label{lem:conv:psi:theta}
	There exists some $\eps_0 =\eps_0(\lambda,k,d)>0$ such that as $L\to\infty$,
	\begin{equation}\label{eq:lem:conv:psi}
	    \sup_{||\utheta-\utheta^\star_{\lambda}||_1\leq \eps_0}\sup_{\substack{x_1,...,x_i \in \partial\sqcup\{\circ,s\}\\0\leq i\leq 3}} \big\lvert \partial_{x_1...x_i}\psi_{\lambda,L}(\utheta)-\partial_{x_1...x_i}\psi_{\lambda}(\utheta)\big\lvert \to 0
	\end{equation}
	where $\partial_{x_1,..,x_i}$ denotes partial differentiation with respect to $\theta_{x_1},...,\theta_{x_i}$(for $i=0$, interpret it as the identity). Furthermore, there exist $\delta_{0}^\prime=\delta_{0}^\prime(\eps_0)<\delta_0$ and $L(\eps_0)$ such that if $L\geq L(\eps_0)$, then
    \begin{equation*}
    \begin{split}
    \utheta_{\lambda,L}\lvert_{\BB_{\la}(\delta_0^\prime)}\;,\; \utheta_{\la}\lvert_{\BB_{\la}(\delta_0^\prime)}&:\BB_{\la}(\delta_0^\prime)\longrightarrow \{\utheta:||\utheta-\utheta^\star_{\lambda}||<\eps_{0}\}\;,\\
    \utheta^{-}_{\lambda,L}\lvert_{\BB^{-}_{\la}(\delta_0^\prime)}\;,\;\utheta^{-}_{\la}\lvert_{\BB^{-}_{\la}(\delta_0^\prime)}&:\BB^{-}_\la(\delta_0^\prime)\longrightarrow \{\utheta^{-}:||\utheta^{-}-\utheta^{\star,-}_{\lambda}||<\eps_0\}
    \end{split}
    \end{equation*}
    are twice differentiable, and satisfy the following as $L\to \infty$:
	%Moreover, $\utheta_{\la,L}(\cdot)$ and $\utheta^{-}_{\la,L}(\cdot)$ respectively converge uniformly to  $\utheta_{\la}(\cdot)$ and $\utheta^{-}_{\la}(\cdot)$ in Sobolev-type norm:
	\begin{equation}\label{eq:lem:conv:theta}
	\begin{split}
	    &\sup_{(B,s)\in \BB_{\la}(\delta_0^\prime)}\sup_{\substack{\sigma_1,...,\sigma_i\in (\dot{\partial}^\bullet)^d\sqcup(\hat{\partial}^\bullet)^k\sqcup\hat{\partial}^\bullet\sqcup\{s\}\\0\leq i\leq 2}}\big\Vert\partial_{\sigma_1,...,\sigma_i}\utheta_{\lambda,L}(B,s)-\partial_{\sigma_1,...,\sigma_i}\utheta_{\lambda}(B,s)\big\Vert_1\to 0\;,\\
	    &\sup_{B\in \BB^{-}_{\la}(\delta_0^\prime)}\sup_{\substack{\sigma_1,...,\sigma_i\in (\dot{\partial}^\bullet)^d\sqcup(\hat{\partial}^\bullet)^k\sqcup\hat{\partial}^\bullet\\0\leq i\leq 2}}\big\Vert\partial_{\sigma_1,...,\sigma_i}\utheta^{-}_{\lambda,L}(B)-\partial_{\sigma_1,...,\sigma_i}\utheta^{-}_{\lambda}(B)\big\Vert_1\to 0.
	\end{split}
	\end{equation}
	\end{lemma}
	\begin{proof}
	We first prove \eqref{eq:lem:conv:psi}: recalling the definition of $\psi_{\lambda,L}$ and $\psi_\lambda$ in \eqref{eq:grad:psi:optimal}, we have
	\begin{equation}\label{eq:lem:conv:psi:theta:technical-1}
	    \sup_{\substack{x_1,...,x_i \in \partial\sqcup\{\circ,s\}\\0\leq i\leq 3}}\big\lvert \partial_{x_1...x_i}\psi_{\lambda,L}(\utheta)-\partial_{x_1...x_i}\psi_{\lambda}(\utheta)\big\lvert = \sum_{\ttt:v(\ttt)>L}\Big(\max_{x\in \partial\sqcup \{\circ ,s\}}\eta_{\ttt}(x)\Big)^3 J_{\ttt}w_{\ttt}^{\lambda}\exp\left(\langle \utheta, \boeta_{\ttt}\rangle\right).
	\end{equation}
	Note that for a valid free tree $\ttt$, each clause must have internal degree at least $2$, so $f(\ttt)+1\leq v(\ttt)$ holds. Thus, we can crudely bound $\max_{x\in \partial\sqcup \{\circ ,s\}}\eta_{\ttt}(x)\leq dv(\ttt)$. Moreover, recalling \eqref{eq:theta:opt:tree:prob}, we can bound for $||\utheta-\utheta^\star_{\lambda}||_1\leq \eps$,
	\begin{equation*}
	    J_{\ttt}w_{\ttt}^{\lambda}\exp\left(\langle \utheta, \boeta_{\ttt}\rangle\right)\leq p^\star_{\ttt,\lambda}\exp\Big(\eps\sum_{ x\in \partial\sqcup \{\circ ,s\}}\eta_{\ttt}(x)\Big)\leq p^\star_{\ttt,\lambda}\exp\left(5\eps dv(\ttt)\right).
	\end{equation*}
	Therefore, for any $||\utheta-\utheta^\star_{\lambda}||_1\leq \eps$, we can bound the summand in the \textsc{rhs} of \eqref{eq:lem:conv:psi:theta:technical-1} by
	\begin{equation}\label{eq:lem:conv:psi:theta:technical-2}
	    \Big(\max_{x\in \partial\sqcup \{\circ ,s\}}\eta_{\ttt}(x)\Big)^3 J_{\ttt}w_{\ttt}^{\lambda}\exp\left(\langle \utheta, \boeta_{\ttt}\rangle\right)
	    \leq d^3 v^3 \exp\left(5\eps d v\right)\sum_{v(\ttt)=v}p^\star_{\ttt,\lambda}\leq d^3 v^3 \exp\left(5\eps d v\right)2^{-kv/2},
	\end{equation}
	where the last bound is due to Lemma \ref{lem:opt:tree:decay}. Therefore, by \eqref{eq:lem:conv:psi:theta:technical-1} and \eqref{eq:lem:conv:psi:theta:technical-2}, taking $\eps_0\equiv\frac{k}{15d}$ gives the first claim \eqref{eq:lem:conv:psi}. Turning to the second claim, we make the following observations.
	\begin{itemize}
	    \item $\utheta_{\lambda,L}(B,s)$, defined in Lemma \ref{lem:exist:theta}, satisfy $\utheta_{\lambda,L}(B,s)= (\nabla \psi_{\lambda,L})^{-1}\left(h(B),s\right)$. Therefore, by inverse function theorem and chain rule,
	    \begin{equation*}
	    \left(\utheta_{\lambda,L}(B,s)\right)^\prime=\left(\nabla^2\psi_{\lambda,L}(\utheta_{\lambda,L})\right)^{-1}\left(h(B),s\right)^\prime=\det\left((\nabla^2\psi_{\lambda,L}(\utheta_{\lambda,L})\right)^{-1}\textnormal{adj}\left(\nabla^2\psi_{\lambda,L}(\utheta_{\lambda,L})\right)\cdot\left(h(B),s\right)^\prime,
	    \end{equation*}
	    where $\textnormal{adj}(A)$ denotes the adjugate matrix of $A$ and $\left(f(B,s)\right)^\prime$ denotes the Jacobian of $f$ with respect to $(B,s)$. The analog holds for $\utheta_{\lambda}(B,s), \utheta^{-}_{\lambda,L}(B)$ and $\utheta^{-}_{\lambda}(B)$.
	    \item By $(2)$ of Proposition \ref{prop:BPcontraction:1stmo}, $\utheta^\star_{\lambda,L}, B^\star_{\lambda,L}$ and $h^\star_{\lambda,L}$ converges to $\utheta^\star_{\lambda},B^\star_{\lambda}$, and $h^\star_{\lambda}$ respectively as $L\to\infty$ in $\ell^1$ distance.
	    \item In the proof of Lemma \ref{lem:exist:theta}, we have shown that $\nabla^2\psi_{\lambda}(\utheta^\star_{\lambda})\succ 0$. Hence, together with \eqref{eq:lem:conv:psi}, the following holds: for sufficiently small $\beta>0$, there exists $\eps=\eps(\beta)<\eps_0$ such that for any $\utheta$ with $||\utheta-\utheta^\star_{\la}||_1<\eps$, we have $\nabla^2\psi_{\lambda}(\utheta)\succeq \beta I$.
	\end{itemize}
	Having the above observations in hand, \eqref{eq:lem:conv:theta} is straightforward from \eqref{eq:lem:conv:psi}, thus we omit the details.
	\end{proof}
	\begin{remark}\label{rem:delta:smaller}
	$\BB_{\la}(\delta_0)$ and $\BB^{-}_{\la}(\delta_0)$ in Lemma \ref{lem:exist:theta} will play a crucial role when we compute $\E\bZ_{\la,s}^\tr$ and $\E \bZ_\la$. Indeed, by Proposition \ref{prop:maxim:1stmo}, we can neglect the contribution of $\E \bZ_{\la}[B]$ when $B$ is at least a constant distance away from $B^\star_\la$. %Since the conclusion of Lemma \ref{lem:exist:theta} still hold when we make $\delta_0(\la,d,k)$ and $\delta_0(\la,L,d,k)$ smaller,
	From now on, we will consider $\delta_0$ small enough to suit our needs when summing $\E \bZ_{\la}^{\tr}[B]$ and $\E \bZ_{\la,s}^{\tr}[B]$ over $\Vert B-B^\star_{\la}\Vert _1<\delta_0$. In particular, we take $\delta_0$ small enough so that the following holds.
	\begin{itemize}
	\item For sufficiently large $L$, set $\delta_0=\delta_0(\la,d,k)=\delta_0(\la,L,d,k)$ small enough so that \eqref{eq:lem:conv:theta} in Lemma \ref{lem:conv:psi:theta} hold for $\delta_0^\prime=\delta_0$.
	    \item Any $B\in\BB_{\la}^{-}(\delta_0)$ has full support and $\inf_{B\in\BB^{-}_{\la}(\delta_0)}\kappa(B)=:\eps(\delta_{0})>0$ holds, where $\kappa(B)$ is defined in Proposition \ref{prop:1stmo:B nt decomp}.
	    \item For $(B,s)\in \BB_{\la}(\delta_0)$ and $B\in \BB_{\la}^{-}(\delta_0)$, define $(p_{\ttt,\la}(B,s))_{\ttt\in \FFF_\tr},(p_{\ttt,\la}(B))_{\ttt\in \FFF_\tr},(p_{\ttt,\la,L}(B,s))_{v(\ttt)\leq L}$ and $(p_{\ttt,\la,L}(B))_{v(\ttt)\leq L}$ by 
	    \begin{equation}\label{eq:opt:tree:prob:B:s}
	    \begin{split}
	       p_{\ttt,\la}(B,s) &:= J_{\ttt}w_{\ttt}^\la \exp\left(\langle \utheta_{\la}(B,s), \boeta_{\ttt}\rangle\right);\qquad p_{\ttt,\la}(B) := J_{\ttt}w_{\ttt}^\la \exp\left(\langle \utheta^{-}_{\la}(B), \boeta^{-}_{\ttt}\rangle\right);\\
	        p_{\ttt,\la,L}(B,s) &:= J_{\ttt}w_{\ttt}^\la \exp\left(\langle \utheta_{\la,L}(B,s), \boeta_{\ttt}\rangle\right);\qquad p_{\ttt,\la,L}(B) := J_{\ttt}w_{\ttt}^\la \exp\left(\langle \utheta^{-}_{\la,L}(B), \boeta^{-}_{\ttt}\rangle\right).
	    \end{split}
	    \end{equation}
	    By \eqref{eq:theta:opt:tree:prob}, $p_{\ttt, \lambda}(B,s)$ for $(B,s) \in \BB_{\la}(\delta_\circ)$ can deviate from $p^{\star}_{\ttt,\la}$ at most by a factor of $e^{\delta_{0}\sum_{x\in \partial \sqcup\{\circ,s\}}\eta_{\ttt}(x)}\leq e^{\delta_0(k+d+2)v(\ttt)}$. Note that the same is true for $p_{\ttt,\la}(B), p_{\ttt,\la,L}(B,s)$ and $p_{\ttt,\la,L}(B)$. Thus, we can consider $\delta_0$ small enough so that for $(B,s) \in \BB_{\la}(\delta_0)$ and $v\geq 1$,
	    \begin{equation*}
	    \max\left\{\sum_{\ttt:v(\ttt)=v}p_{\ttt,\lambda}(B,s),\sum_{\ttt:v(\ttt)=v}p_{\ttt,\lambda}(B), \sum_{\ttt:v(\ttt)=v}p_{\ttt,\lambda,L}(B,s),\sum_{\ttt:v(\ttt)=v}p_{\ttt,\lambda,L}(B)
	    \right\}\leq 2^{-kv/3}.
	    \end{equation*} 
	\end{itemize}
	\end{remark}
	In what follows, we denote by $\proj(B)$ the projection of $B\in \bDelta^{\textnormal{b}}$ onto $\bDelta^{\textnormal{b}}_n$:
	\begin{equation*}
	    \proj(B) \in \argmin_{B^\prime \in \bDelta^{\textnormal{b}}_n}\Vert B^\prime-B\Vert_1.
	\end{equation*}
	\begin{lemma}\label{lem:exist:free:energy:1stmo}
	For $(B,s) \in \BB_{\la}(\delta_0)$, define its truncated and untruncated free energy by
	\begin{equation}\label{eq:def:free:energy:B:s}
	    F_{\la,L}(B,s):= \Psi_\circ(B)-\Big\langle \utheta_{\la,L}(B,s), \left(\uh(B),s\right)\Big\rangle\quad\textnormal{and}\quad F_{\la}(B,s):= \Psi_\circ(B)-\Big\langle \utheta_{\la}(B,s), \left(\uh(B),s\right)\Big\rangle.
	\end{equation}
	 Then, the following holds for $(B,s) \in \BB_{\la}(\delta_0)$:
	\begin{equation}\label{eq:1stmo:B:s}
	\begin{split}
	&\E \bZ^{(L),\tr}_{\lambda,s}\left[\proj(B)\right] = \exp\Big(nF_{\la,L}(B,s)+O_{k}(\log n)\Big);\\
	&\E \bZ^{(L),\tr}_{\lambda,s}\left[\proj(B) ,(n_{\ttt})_{\ttt\in \FFF_{\tr}}\in \ee_{\frac{1}{4}}\right] = \exp\Big(nF_{\la,L}(B,s)+O_{k}(\log n)\Big).
	\end{split}
	\end{equation}
	The analog of \eqref{eq:1stmo:B:s} also holds for the untruncated model. Similarly, for $B\in \BB^{-}_{\la}(\delta_0)$, define
	\begin{equation}\label{eq:def:free:energy:B}
	    F_{\la,L}(B):= \Psi_\circ(B)-\Big\langle \utheta^{-}_{\la,L}(B), \uh(B)\Big\rangle\quad\textnormal{and}\quad F_{\la}(B):= \Psi_\circ(B)-\Big\langle \utheta^{-}_{\la}(B), \uh(B)\Big\rangle
	\end{equation}
	Then, the following equations hold for $B\in \BB^{-}_{\la}(\delta_0)$:
	\begin{equation}\label{eq:1stmo:B}
	\begin{split}
	    &\E \bZ^{(L),\tr}_{\lambda}\left[\proj(B)\right]= \exp\Big(nF_{\la,L}(B)+O_{k}(\log n)\Big);\\
	    &\E \bZ^{(L),\tr}_{\lambda}\left[\proj(B), (n_{\ttt})_{\ttt\in \FFF_{\tr}}\in \ee_{\frac{1}{4}}\right]= \exp\Big(nF_{\la,L}(B)+O_{k}(\log n)\Big).
	\end{split}
	\end{equation}
	The analog of \eqref{eq:1stmo:B} also holds for the untruncated model.
	\end{lemma}
	\begin{proof}
	We only prove \eqref{eq:1stmo:B:s} for the untruncated model since the other conclusions hold with similar argument. For simplicity, denote $B_n\equiv \proj(B)$. Note that $\kappa(B_n)\gtrsim_{k} 1$ holds for $B \in \BB_{\la}^{-}(\delta_0)$ (see Remark \ref{rem:delta:smaller}), where $\kappa(B_n)$ is defined in Proposition \ref{prop:1stmo:B nt decomp}. Thus, Proposition \ref{prop:1stmo:B nt decomp} and Lemma \ref{lem:exist:theta} show
	\begin{equation}\label{eq:1stmo:B:s:express:by:event}
	\begin{split}
 	    &\E \bZ^\tr_{\la,s}[B_n] 
	    \asymp_{k}\frac{\exp\big(nF_{\la}(B,s)\big)}{p_{\circ}(n,B)\left(nh_{\circ}(B_n)\right) !}\left(\frac{nh_{\circ}(B_n)}{e}\right)^{nh_{\circ}(B_n)}\P_{\utheta(B,s)}\Big(\AAA_{\uh(B_n),s}\Big);\\
	    &\E \bZ^\tr_{\la,s}\left[B_n, (n_{\ttt})_{\ttt\in \FFF_{\tr}}\in \ee_{\frac{1}{4}}\right] 
	    \asymp_{k}\frac{\exp\big(nF_{\la}(B,s)\big)}{p_{\circ}(n,B)\left(nh_{\circ}(B_n)\right) !}\left(\frac{nh_{\circ}(B_n)}{e}\right)^{nh_{\circ}(B_n)}\P_{\utheta(B,s)}\left(\AAA^{\ee}_{\uh(B_n),s}\right),
	\end{split}
	\end{equation}
	where $\P_{\utheta(B,s)}$ is taken with respect to i.i.d. random free trees $X_1,...,X_{nh_{\circ}(B_n)}\in \FFF$ with distribution
	\begin{equation}\label{eq:def:rescaling:prob:B:s}
	    \P_{\utheta(B,s)}(X_i=\ttt):=\frac{p_{\ttt,\la}(B,s)}{h_{\circ}(B)}=\frac{J_{\ttt}w_{\ttt}^\lambda \exp\left(\langle\utheta(B,s),\boeta_{\ttt}\rangle\right)}{h_{\circ}(B)},
	\end{equation}
	and the events $\AAA_{\uh(B_n),s}$ and $\AAA^{\ee}_{\uh(B_n),s}$ are defined by
	\begin{equation}\label{eq:def:lclt:event:B:s}
	\begin{split}
	     &\AAA_{\uh(B_n),s}:= \left\{\sum_{i=1}^{nh_{\circ}(B_n)}\boeta^{-}_{X_i}=n\uh(B_n)\quad\textnormal{and}\quad \sum_{i=1}^{nh_{\circ}(B_n)}\eta_{X_i}(s)\in [ns,ns+1)\right\}\\
	     &\AAA^{\ee}_{\uh(B_n),s}:= \AAA_{\uh(B),s}\bigcap \left\{\sum_{i=1}^{nh_{\circ}(B_n)}\one\{v(X_i)=v\}\leq n2^{-kv/4},~~~~~~\forall v\geq 1 \right\}
	\end{split}
	\end{equation}
	In \eqref{eq:1stmo:B:s:express:by:event}, observe that $\frac{1}{\left(nh_{\circ}(B_n)\right)!}\left(\frac{nh_\circ(B_n)}{e}\right)^{nh_\circ(B_n)}\asymp \left(nh_{\circ}(B_n)\right)^{-1/2}$ by Stirling's approximation. Also, the degree of the monomial $p_{\circ}(n,B)$, defined in \eqref{eq:def:Psi:circ:p:circ:1stmo}, is bounded as a function of $k$, so
% 	\begin{equation}\label{eq:1stmo:B:s:technical}
% 	\begin{split}
% 	    &\E \bZ^\tr_{\la,s}[B_n]=\exp\Big(nF_{\la}(B,s)+O_k(\log n)\Big)\P_{\utheta(B,s)}\left(\AAA_{\uh(B_n),s}\right);\\
% 	   &\E \bZ^\tr_{\la,s}\left[B_n, (n_{\ttt})_{\ttt\in \FFF_{\tr}}\in \ee_{\frac{1}{4}}\right] 
% 	   =\exp\Big(nF_{\la}(B,s)+O_k(\log n)\Big)\P_{\utheta(B,s)}\left(\AAA^{\ee}_{\uh(B_n),s}\right).
% 	\end{split}
% 	\end{equation}
	our goal \eqref{eq:1stmo:B:s} is proven if we show $ \P_{\utheta(B,s)}\left(\AAA_{\uh(B_n),s}^{\ee}\right)=n^{-\Omega_k(1)}$: first, observe that $\E_{\utheta(B,s)}[\boeta_{X_i}]=\left(h_\circ(B)\right)^{-1}\left(\uh(B),s\right)$ holds by the construction of $\utheta(B,s)$ in Lemma \ref{lem:exist:theta}. Thus, local \textsc{clt} implies that $ \P_{\utheta(B,s)}\left(\AAA_{\uh(B_n),s}\right)=\Omega_{k}(n^{-(|\partial|+1)/2})$ holds. Moreover, union bound shows
	\begin{equation*}
	\begin{split}
	    \P_{\utheta(B,s)}\left(\AAA^{\ee}_{\uh(B),s}\right)&\geq \P_{\utheta(B,s)}\left(\AAA_{\uh(B),s}\right)-\sum_{v\leq \frac{4\log n}{k\log 2}}\P_{\utheta(B,s)}\Big(\sum_{i=1}^{nh_{\circ}(B_n)}\one\{v(X_i)=v\}> n2^{-kv/4}\Big)\\
	    &\quad -\P_{\utheta(B,s)}\Big(\sum_{i=1}^{nh_{\circ}(B_n)}\one\Big\{v(X_i)>\frac{4\log n}{k\log 2}\Big\}\geq 1\Big).
	\end{split}
	\end{equation*}
	Recalling Remark \ref{rem:delta:smaller}, $\sum_{v(\ttt)=v}p_{\ttt,\la}(B,s)\leq 2^{-kv/3}, v\geq 1$ holds for all $(B,s)\in \BB_{\la}(\delta_0)$. Thus, by Chernoff bound for binomial random variables, we have
	\begin{equation*}
	\begin{split}
	    \sum_{v\leq \frac{4\log n}{k\log 2}}\P_{\utheta(B,s)}\Big(\sum_{i=1}^{nh_{\circ}(B_n)}\one\{v(X_i)=v\}>n2^{-kv/4}\Big)&= \exp\big(-\Omega_{k}(n^{1/3})\big);\\
	    \P_{\utheta(B,s)}\Big(\sum_{i=1}^{nh_{\circ}(B_n)}\one\Big\{v(X_i)>\frac{4\log n}{k\log 2}\Big\}\geq 1\Big)&=\exp\big(-\Omega_{k}(n^{1/3})\big).
	\end{split}
	\end{equation*}
	Therefore, we conclude that $\P_{\utheta(B,s)}\left(\AAA^{\ee}_{\uh(B_n),s}\right)=\Omega_{k}(n^{-(|\partial|+1)/2})$.
	\end{proof}
	Observe that the uniform convergence properties of $\utheta_{\la,L}(\cdot)$ and $\utheta_{\la,L}^{-}(\cdot)$, and the convergence of $(B^\star_{\la,L},s^\star_{\la,L})$ as $L\to\infty$, which are guaranteed by Proposition \ref{prop:BPcontraction:1stmo} and Lemma \ref{lem:conv:psi:theta}, imply that
	\begin{equation}\label{eq:conv:hess:L:free:energy:1stmo}
	\begin{split}
	    &\lim_{L\to\infty}\Vert \nabla^2_{B} F_{\lambda,L}(B^\star_{\lambda,L},s^\star_{\lambda,L})-\nabla^2_{B} F_{\lambda}(B^\star_{\lambda},s^\star_{\lambda})\Vert_{\textnormal{op}}=0,\\
	    &\lim_{L\to\infty}\Vert \nabla^2 F_{\lambda,L}(B^\star_{\lambda,L})-\nabla^2 F_{\lambda}(B^\star_{\lambda})\Vert_{\textnormal{op}}=0,
	\end{split}
	\end{equation}
	where $\nabla^2_{B}$ indicates that the Hessian is taken with respect to $B$. The proposition below plays a crucial role in computing the first moment and its proof is done by the \textit{resampling method}, which is presented in Section \ref{subsec:resampling}.
	\begin{prop}\label{prop:negdef}
		For $\lambda \in [0,1]$, the following holds.
		\begin{enumerate}
		\item The unique maximizer of $F_{\lambda}(B,s)$ in $(B,s)\in\BB_{\la}(\delta_0)$ is given by $(B^\star_\lambda, s^\star_\lambda)$. Similarly, the unique maximizer of $F_{\lambda}(B)$ in $B\in\BB^{-}_{\la}(\delta_0)$ is given by $B^\star_\lambda$. The analog for the truncated model also holds.
		\item There exists a constant $\beta=\beta(k)>0$, which does not depend on $L$, such that for large enough $L\geq L_0(\la,k,d)$,
		\begin{equation}\label{eq:freeenergy:negdef:1stmo:truncated}
		    \nabla^2_{B} F_{\la,L}(B^\star_{\la,L},s^\star_{\la,L}), \nabla^2 F_{\la,L}(B^\star_{\la,L})\prec -\beta I. 
		\end{equation}
		Hence, $\nabla^2_{B} F_{\la}(B^\star_\la,s^\star_\la), \nabla^2 F_{\la}(B^\star_\la)\prec 0$ holds by \eqref{eq:conv:hess:L:free:energy:1stmo}.
		\end{enumerate}
	\end{prop}
	\begin{remark}\label{rem:truncated:negdef:SSZ}
	In \cite{ssz22}, they analyzed the truncated free energy $F_{\la,L}(H)$ of a given coloring profile $H$, explicitly defined in \eqref{eq:1stmo dec by H}. They introduced the \textit{resampling method} to show that for large enough $L\geq L_0(\lambda,k,d)$, the unique maximizer of $F_{\la,L}(H), H\in \bDelta$ is given by $H^\star_{\la,L}$ with $\nabla^2 F_{\la,L}(H^\star_{\la,L})\prec 0$ (cf. \cite[Proposition 3.4]{ssz22}). Here, note that our definition of $\bDelta$ includes the condition $\max\{\bar{H}(\ff),\bar{H}(\rr)\} \leq \frac{7}{2^k}$. Hence, Proposition \ref{prop:negdef} is a generalization of \cite{ssz22} to the untruncated model. In particular, the conclusion of (1) of Proposition \ref{prop:negdef} for the truncated model and a version of \eqref{eq:freeenergy:negdef:1stmo:truncated}, for which $\beta>0$ can depend on $L$, is a consequence of \cite{ssz22}, because $F_{\la,L}(B)$ and $F_{\la,L}(B,s)$ can be obtained by the maximum of $F_{\la,L}(H)$ under a linear constraint. That is,
	\begin{equation}\label{eq:corollary:freeenergy:SSZ:1stmo}
	   (B^\star_{\la,L},s^\star_{\la,L})=\textnormal{argmax}\big\{F_{\la,L}(B,s):B\in \bDelta^{\textnormal{b}},s\in [0,\log 2]\big\}\textnormal{ and } \nabla^2 F_{\la,L}(B^\star_{\la,L},s^\star_{\la,L})\prec 0.
	\end{equation}
	Note that in \eqref{eq:corollary:freeenergy:SSZ:1stmo}, we have assumed that $F_{\la,L}(B,s)\equiv \lim_{n\to\infty}\frac{1}{n}\log \E \bZ^{(L),\tr}_{\la,s}[B]$ for $B\in \bDelta^{\textnormal{b}},s\in [0,\log 2]$ is well-defined, which follows from \cite{ssz22}. However, the proof of \cite[Proposition 3.4]{ssz22} cannot be directly applied to the untruncated model and we need to develop new techniques to deal with the unbounded spin space as we demonstrate in Section \ref{subsec:resampling}.
	\end{remark}
	\begin{remark}\label{rem:freeenergy:s:star}
		Although we did not mention in the statement of Proposition \ref{prop:negdef}, it turns out that the maximal value of $F$ also corresponds to the 1\textsc{rsb} free energy functional from the physics computations. This was already established in \cite{ssz22} but only for the \textit{truncated} model. Based on the proof of the proposition presented in Section \ref{sec:1stmo}, the same computations done in \cite{ssz22} give the correspondence between the maximal value of $F$ and the 1\textsc{rsb} free energy functional.
	\end{remark}
	\subsection{Pinning down the leading constant}\label{subsec:1stmo:leading constant}
	In this subsection, we calculate the first moment up to a leading constant. For the purpose of proving Theorem \ref{thm1}, we only need to calculate the first moment up to a constant. However, in the companion paper \cite{nss2}, where we we prove that $(b)$ and $(c)$ of Theorem \ref{thm1} hold with probability $1-\eps$ for arbitrary $\eps>0$, we need the continuity properties of the leading constants as well as their existence. Hence, we will provide details for such in the subsection. We start by defining the crucial quantities which appeared in Theorem \ref{thm1}.
	\begin{defn}\label{def:la:s:c:star}
	    For $k\geq k_0$ and $\alpha \in (\alpha_{\textsf{cond}}, \alpha_{\textsf{sat}})$, define the constants $\la^\star \equiv \la^\star(\alpha, k), s^\star\equiv s^\star(\alpha,k)$, and $c^\star\equiv c^\star(\alpha,k)$ by
	    	\begin{equation}\label{eq:def:sstarandlambdastar}
	      \lambda^\star \equiv 
	      \sup\{\lambda\in [0,1]: F_{\la}(B^\star_\lambda) \geq \lambda s^\star_{\la} \},\quad 
	       s^\star\equiv s^\star_{\la^\star},\quad\textnormal{and}\quad c^\star\equiv (2\la^\star)^{-1},
	     \end{equation}
	 respectively. We remark that $s^\star = \textsf{f}^{1\textsf{rsb}}(\alpha)$ (see Remark \ref{rem:freeenergy:s:star}) and $\la^\star \in (0,1)$ holds for $\alpha \in (\alpha_{\textsf{cond}}, \alpha_{\textsf{sat}})$ (cf. \cite[Proposition 1.4]{ssz22}).
	\end{defn}
	\begin{thm}\label{thm:1stmo:constant} 
		Let $\lambda \in [0,1]$. For sufficiently large $L$, the constants
		\begin{equation}\label{eq:thm:1stmo:constant:la}
	{C}_1(\lambda) := \lim_{n\to\infty} \frac{ \E \bZ^{\tr}_{\lambda} }{ \exp \Big( n F_{\la}(B^\star_\lambda )\Big)} \quad\textnormal{and}\quad {C}_{1,L}(\lambda):=
		\lim_{n\to\infty} \frac{ \E \tZ_{\lambda} }{\exp \Big( n F_{\la,L}(B^\star_{\lambda,L})\Big)}
		\end{equation} 
		are well-defined and continuous in $[0,1]$. Furthermore, $\lim_{L\to\infty} {C}_{1,L} (\lambda) = {C}_1 (\lambda)$ holds.
	\end{thm}
	\begin{proof}
	As a consequence of Proposition \ref{prop:1stmo:aprioriestimate}, \ref{prop:maxim:1stmo} and \ref{prop:negdef}, we have that $\E \bZ^\tr_{\la}[||B-B^\star_{\la}||_1\leq \frac{\log n}{\sqrt{n}}] \geq \big(1-O_{k}(n^{-2}\log n)\big)\E\bZ^\tr_\la$ holds, so we restrict our attention to the case $||B-B^\star_\la||_1 \leq \frac{\log n}{\sqrt{n}}$. Note that $B^\star_\la$ has full support with finite dimension, so any $B\in \bDelta^{\textnormal{b}}$ with $||B-B^\star_\la||_1 \leq \frac{\log n}{\sqrt{n}}$ has full support with $\kappa(B)$ bounded away from zero, for large enough $n$. Hence, Proposition \ref{prop:1stmo:B nt decomp} shows
	\begin{equation}\label{eq:1stmo:B:expansion}
	    \E \bZ^\tr_\la[B] =\left(1+O_k(n^{-1})\right)\frac{\exp\left\{n F_\la(B)\right\}}{p_0(n,B)}\frac{1}{\left(nh_\circ(B)\right)!}\left(\frac{nh_\circ(B)}{e}\right)^{nh_\circ(B)}\P_{\utheta^{-}(B)}(\AAA_{\uh(B)}),
	\end{equation}
	where $\AAA_{\uh(B)}\equiv \left\{\sum_{i=1}^{nh_{\circ}(B_n)}\boeta^{-}_{X_i}=n\uh(B_n)\right\}$ and $X_1,...,X_{nh_{\circ}(B)}\in \FFF_{\tr}$ are i.i.d with distribution $\P_{\utheta^{-}(B)}(X_i=\ttt)=\big(h_\circ(B)\big)^{-1}p_{\ttt,\la}(B)$ (cf. \eqref{eq:def:rescaling:prob}). To this end, we now aim to sum \eqref{eq:1stmo:B:expansion} over $||B-B^\star_{\la}||_1\leq \frac{\log n}{\sqrt{n}}$. Henceforth, we write $g(n,B)=o_n(1)$ whenever 
	$\lim_{n\to\infty}\sup_{||B-B^\star_{\la}||_1\leq \frac{\log n}{\sqrt{n}}}\big| g(n,B) \big| =0$. By definition of $p_{\circ}(n,B)$ in \eqref{eq:def:Psi:circ:p:circ:1stmo} and Stirling's approximation in $\left(nh_\circ(B)\right)!$, we have
	\begin{equation}\label{eq:1stmo:stirling:technical-1}
	\begin{split}
	    \frac{1}{p_\circ(n,B)}\frac{1}{\left(nh_\circ(B)\right)!}\left(\frac{nh_\circ(B)}{e}\right)^{nh_\circ(B)} =
	    &\left(1+o_n(1)\right)\left(\frac{\prod_{\sigma}\bar{B}^\star_{\la}(\sigma)}{h^\star_{\la}(\circ)\prod_{\sig}\dot{B}^\star_{\la}(\sig)\prod_{\sig}\hat{B}^\star_{\la}(\sig)}\right)^{1/2}\\
	    &\times (2\pi n )^{(|\hat{\partial}^\bullet|-|\textnormal{supp} \hat{v}|-|\textnormal{supp} \dot{I}|)/2}d^{(|\hat{\partial}^\bullet|-|\textnormal{supp}\hat{v}|)/2}k^{(|\textnormal{supp}\hat{v}|-1)/2}.
	\end{split}
	\end{equation}
	$\E_{\utheta^{-}(B)}[\boeta^{-}_{X_i}]=\frac{\uh(B)}{h_{\circ}(B)}$ holds by construction of $\utheta^{-}(B)=\utheta^{-}_{\la}(B)$ in Lemma \ref{lem:exist:theta}, so local \textsc{clt} shows
	\begin{equation}\label{eq:B:1stmo:local:CLT}
	    \P_{\utheta^{-}(B)}\left(\AAA_{\uh(B)}\right)=\left(1+o_n(1)\right)(2\pi n)^{-|\partial|/2} \det\left(\big(h^\star_{\la}(\circ)\big)^{-2}\left[\nabla^2 \psi^{-}_{\la}(\utheta^{\star,-}_\la)\right]_{-\circ}\right)^{-1/2},
	\end{equation}
	where $\left[\nabla^2 \psi^{-}_{\la}(\utheta^{\star,-}_\la)\right]_{-\circ}$ denotes the $|\partial|\times|\partial|$ submatrix obtained from $\nabla^2 \psi^{-}_{\la}(\utheta^{\star,-}_\la)$ by deleting the row and column indexed with $\circ$. Hence, plugging \eqref{eq:1stmo:stirling:technical-1} and \eqref{eq:B:1stmo:local:CLT} into \eqref{eq:1stmo:B:expansion} shows
	\begin{equation}\label{eq:1stmo:B:intermediate}
	\begin{split}
	     \E \bZ^\tr_\la[B]
	     &=\left(1+o_n(1)\right)\left(\frac{\prod_{\sigma}\bar{B}^\star_{\la}(\sigma)}{h^\star_{\la}(\circ)\prod_{\sig}\dot{B}^\star_{\la}(\sig)\prod_{\sig}\hat{B}^\star_{\la}(\sig)}\right)^{1/2} d^{(|\hat{\partial}^\bullet|-|\textnormal{supp}\hat{v}|)/2}k^{(|\textnormal{supp}\hat{v}|-1)/2}\\
	     &\times (2\pi n )^{(|\hat{\partial}^\bullet|-|\textnormal{supp} \hat{v}|-|\textnormal{supp} \dot{I}|-|\partial|)/2}\det\left(\big(h^\star_{\la}(\circ)\big)^{-2}\left[\nabla^2 \psi^{-}_{\la}(\utheta^{\star,-}_\la)\right]_{-\circ}\right)^{-1/2}\exp\left(nF_\la(B)\right),
	\end{split}
	\end{equation}
	The exponent above $2\pi n$ in the equation above is $\frac{|\hat{\partial}^\bullet|-|\textnormal{supp} \hat{v}|-|\textnormal{supp} \dot{I}|-|\partial|}{2}=\frac{\textnormal{dim}(\bDelta^{\textnormal{b}})}{2}$, so we can sum
	\begin{equation}\label{eq:sum:free:energy:1stmo}
	\begin{split}
	\sum_{||B-B^\star_\la||_1 \leq \frac{\log n}{\sqrt{n}}}(2\pi n)^{\textnormal{dim}(\bDelta^{\textnormal{b}})/2}\exp\left(nF_\la(B)\right)=\det\left(-\nabla^2 F_\la(B^\star_\la)\right)^{-1/2}\left(1+o_n(1)\right),
	%&=\sum_{||B-B^\star_\la||_1 \leq \frac{\log n}{\sqrt{n}}}(2\pi n)^{\textnormal{dim}(\bDelta^{\textnormal{b}})/2}\exp\left(nF_\la(B^\star_\la)+\frac{n}{2}\Big\langle B-B^\star_\la, \nabla^2 F_\la(B^\star_\la)(B-B^\star_\la)\Big\rangle +o_n(1) \right)\\
	\end{split}
	\end{equation}
	where we used Gaussian integration in the last equality and $\det\left(-\nabla^2 F_\la(B^\star_\la)\right)\neq 0$ holds by Proposition \ref{prop:negdef}. Therefore, by \eqref{eq:1stmo:B:intermediate} and \eqref{eq:sum:free:energy:1stmo}, the first part of our goal \eqref{eq:thm:1stmo:constant:la} holds with constant
	\begin{equation*}
	\begin{split}
	    C_1(\la) := &\left(\frac{\prod_{\sigma}\bar{B}^\star_{\la}(\sigma)}{h^\star_{\la}(\circ)\prod_{\sig}\dot{B}^\star_{\la}(\sig)\prod_{\sig}\hat{B}^\star_{\la}(\sig)}\right)^{1/2} d^{(|\hat{\partial}^\bullet|-|\textnormal{supp}\hat{v}|)/2}k^{(|\textnormal{supp}\hat{v}|-1)/2}\\
	    &\times \det\left(\big(h^\star_{\la}(\circ)\big)^{-2}\left[\nabla^2 \psi^{-}_{\la}(\utheta^{\star,-}_\la)\right]_{-\circ}\right)^{-1/2}\det\left(-\nabla^2 F_\la(B^\star_\la)\right)^{-1/2}.
	\end{split}
	\end{equation*}
	The same calculations work for the truncated model holds, so the second part of \eqref{eq:thm:1stmo:constant:la} holds with constant 
	\begin{equation*}
	\begin{split}
	    C_{1,L}(\la) := &\left(\frac{\prod_{\sigma}\bar{B}^\star_{\la,L}(\sigma)}{h^\star_{\la,L}(\circ)\prod_{\sig}\dot{B}^\star_{\la,L}(\sig)\prod_{\sig}\hat{B}^\star_{\la,L}(\sig)}\right)^{1/2} d^{(|\hat{\partial}^\bullet|-|\textnormal{supp}\hat{v}|)/2}k^{(|\textnormal{supp}\hat{v}|-1)/2}\\
	    &\times \det\left(\big(h^\star_{\la,L}(\circ)\big)^{-2}\left[\nabla^2 \psi^{-}_{\la,L}(\utheta^{\star,-}_{\la,L})\right]_{-\circ}\right)^{-1/2}\det\left(-\nabla^2 F_{\la,L}(B^\star_{\la,L})\right)^{-1/2}.
	\end{split}
	\end{equation*}
	The continuity of $C_1(\la),C_1(\la,L)$ in $\la \in [0,\la^\star]$ is straightforward from their explicit forms in the equations above. Moreover, $\lim_{L\to\infty} C_{1,L}(\la)=C_1(\la)$ is guaranteed by Lemma \ref{lem:conv:psi:theta} and \eqref{eq:conv:hess:L:free:energy:1stmo}.
	\end{proof} 
	
	%By summing over $(B,h)\in \Omega_1^\star$, we conclude our computation for the first moment. For simplicity, we write $\E\bZ_s = \E\bZ_{\lambda(s),s}$ and $(B^\star_s, h^\star_s) = (B^\star_{\lambda(s)}, h^\star_{\lambda(s)})$.
	\begin{thm}\label{thm:1stmo:constant:s} 
		Let $(s_n)$ be a converging sequence whose limit is  $s^\star$, satisfying $|s_n-s^\star|\leq n^{-2/3}$. Then  the constant
		\begin{equation}\label{eq:def:C:1stmo}
		C_1(\lambda^\star, s^\star) := \lim_{n\to\infty} \frac{\sqrt{n} \E \bZ^{\tr}_{\lambda^\star, s_n} }{\exp\left(nF_{\la^\star}(B^\star_{\lambda^\star})\right) }
		\end{equation}
		is well-defined regardless of the specific choice of $s_n$.
	\end{thm}
	\begin{proof}
	We proceed in the same manner as in the proof of Theorem \ref{thm:1stmo:constant}. For simplicity, we abbreviate $B^\star\equiv B^\star_{\la^\star}$. For $||B-B^\star||_1\leq \frac{\log n}{\sqrt{n}}$, Proposition \ref{prop:1stmo:B nt decomp} shows
	\begin{equation}\label{eq:B:s:1stmo:expansion}
	\begin{split}
	    &\E \bZ^\tr_{\la^\star, s_n}[B]=\left(1+O_k(n^{-1})\right)\frac{\exp\left(n F_{\la^\star}(B,s_n)\right)}{p_{\circ}(n,B)}\frac{1}{\left(nh_\circ(B)\right)!}\left(\frac{nh_\circ(B)}{e}\right)^{nh_\circ(B)}\P_{\utheta(B,s_n)}\left(\AAA_{\uh(B),s_n}\right),\textnormal{ for}\\
	    &\AAA_{\uh(B),s}:= \left\{\sum_{i=1}^{nh_{\circ}(B)}\boeta^{-}_{X_i}=n\uh(B)\quad\textnormal{and}\quad \sum_{i=1}^{nh_{\circ}(B)}\eta_{X_i}(s)\in [ns_n,ns_n+1)\right\},
	\end{split}
	\end{equation}
	where $X_1,...,X_{nh_{\circ}(B)}\in \FFF_{\tr}$ are i.i.d with distribution $\P_{\utheta(B,s_n)}(X_i=\ttt)=\big(h_\circ(B)\big)^{-1}p_{\ttt,\la}(B,s_n)$ (cf. \eqref{eq:def:rescaling:prob:B:s}). By the construction of $\utheta_{\la}(B,s)$ in Lemma \ref{lem:exist:theta}, $\E_{\utheta_{\la}(B,s_n)}[\boeta_{X_i}]=\left(h_\circ(B)\right)^{-1}\left(\uh(B),s_n\right)$ holds, so local \textsc{clt} shows (e.g. see Theorem 3.1 of \cite{Borokov17})
	\begin{equation}\label{eq:B:s:1stmo:local:CLT}
	    \P_{\utheta(B,s_n)}\left(\AAA_{\uh(B),s_n}\right)=\left(1+o_n(1)\right)(2\pi n)^{-(|\partial|+1)/2}\det\left((h^\star_{\la^\star,\circ})^{-2}\left[\nabla^2 \psi_{\la^\star}(\utheta^{\star}_{\la^\star})\right]_{-\circ}\right)^{-1/2},
	\end{equation}
	where $[A]_{-\circ}$ denotes the matrix obtained from $A$ by deleting the row and column indexed with $\circ$, and we write $g(n,B,s_n)=o_n(1)$ whenever $\lim_{n\to\infty}\sup_{||B-B^\star_{\la}||_1\leq \frac{\log n}{\sqrt{n}}}\big| g(n,B,s_n) \big| =0$ holds. We plug \eqref{eq:B:s:1stmo:local:CLT} into \eqref{eq:B:s:1stmo:expansion} and use the Stirling's approximation as done in \eqref{eq:1stmo:stirling:technical-1} to have
	\begin{equation}\label{eq:1stmo:B:s:intermediate}
	\begin{split}
	     \E \bZ^\tr_{\la^\star,s_n}[B]
	     &=\left(1+o_n(1)\right)\left(\frac{\prod_{\sigma}\bar{B}^\star(\sigma)}{h^\star_{\la^\star}(\circ)\prod_{\sig}\dot{B^\star}(\sig)\prod_{\sig}\hat{B}^\star(\sig)}\right)^{1/2} d^{(|\hat{\partial}^\bullet|-|\textnormal{supp}\hat{v}|)/2}k^{(|\textnormal{supp}\hat{v}|-1)/2}\\
	     &\times (2\pi n )^{(|\hat{\partial}^\bullet|-|\textnormal{supp} \hat{v}|-|\textnormal{supp} \dot{I}|-|\partial|-1)/2}\det\left(\big(h^\star_{\la^\star}(\circ)\big)^{-2}\left[\nabla^2 \psi_{\la^\star}(\utheta^{\star}_{\la^\star})\right]_{-\circ}\right)^{-1/2}\exp\left(nF_{\la^\star}(B,s_n)\right).
	\end{split}
	\end{equation}
	Having Proposition \ref{prop:negdef} in mind, we Taylor expand $F_{\la^\star}(B,s_n)$ around $(B^\star,s^\star)$ to see
	\begin{equation*}
	\begin{split}
	    nF_{\la^\star}(B,s_n)
	    &=nF_{\la^\star}(B^\star)+\frac{n}{2}\Big\langle (B-B^\star, s_n-s^\star), \nabla^2 F_{\la^\star}(B^\star,s^\star)(B-B^\star, s_n-s^\star)\Big\rangle +o_n(1)\\
	    &=nF_{\la^\star}(B^\star)+\frac{n}{2}\Big\langle B-B^\star, \nabla^2_{B} F_{\la^\star}(B^\star,s^\star)(B-B^\star)\Big\rangle +o_n(1),
	\end{split}
	\end{equation*}
	where the last equality is due to $|s_n-s^\star|\leq n^{-2/3}$. Thus, we can sum \eqref{eq:1stmo:B:s:intermediate} over $||B-B^\star||_1\leq \frac{\log n}{\sqrt{n}}$ by using Gaussian integration and Proposition \ref{prop:negdef}. Therefore, our goal \eqref{eq:def:C:1stmo} holds with constant
	\begin{equation*}
	 \begin{split}
	    C_1(\la^\star,s^\star) := &\left(\frac{\prod_{\sigma}\bar{B}^\star(\sigma)}{h^\star_{\la^\star}(\circ)\prod_{\sig}\dot{B}^\star(\sig)\prod_{\sig}\hat{B}^\star(\sig)}\right)^{1/2} d^{(|\hat{\partial}^\bullet|-|\textnormal{supp}\hat{v}|)/2}k^{(|\textnormal{supp}\hat{v}|-1)/2}\\
	    &\times (2\pi)^{-1/2}\det\left(\big(h^\star_{\la^\star}(\circ)\big)^{-2}\left[\nabla^2 \psi_{\la^\star}(\utheta^{\star}_{\la^\star})\right]_{-\circ}\right)^{-1/2}\det\left(-\nabla^2_{B} F_{\la^\star}(B^\star,s^\star)\right)^{-1/2}.
	\end{split}
	\end{equation*}
	\end{proof}
	%\begin{remark}
		%Observe that we do not specify the size $s$ of valid colorings in Theorem \ref{thm:1stmo:constant}, whereas we do in Theorem \ref{thm:1stmo:constant:s}. Although our eventual goal is to study $\bZ_{\lambda^\star,s_n}$ as in Theorem \ref{thm:1stmo:constant:s}, we also require Theorem \ref{thm:1stmo:constant} since it is too difficult to study $\bZ_{\lambda^\star,s_n}$ directly via small subgraph conditioning method, discussed in Section \ref{sec:whp}). Thus, we rather work with $\bZ_\lambda$ in Section \ref{sec:whp}, and then connect it with $\bZ_{\lambda^\star,s_n}$ using \eqref{eq:sketch:corresp of constants} which will be explained later. The same approach is applied for the second moment in Section \ref{subsec:sketch:2nd mo} as well, in Theorems \ref{thm:2ndmo:constant for unic}, \ref{thm:2ndmo:constant for unic:s}, Propositions \ref{prop:2nd mo constant for tr}, \ref{prop:2nd mo constant for tr:s}, \ref{prop:2ndmo:comparison tr vs unic} and \ref{prop:2ndmo:comparison tr vs unic:s}.
	%\end{remark}
	\begin{prop}\label{prop:ratio:uni:1stmo}
		Let $(s_n)$ be a converging sequence whose limit is  $s^\star$, satisfying $|s_n-s^\star|\leq n^{-2/3}$. Then  the constant
		\begin{equation}\label{eq:def:beta:1stmo}
		\beta_1(\lambda^\star, s^\star) := \lim_{n\to\infty} \frac{\E \bZ_{\lambda^\star, s_n}  }{\E \bZ^\tr_{\la^\star, s_n} }
		\end{equation}
		is well-defined regardless of the specific choice of $s_n$.
	\end{prop}
	\begin{proof}
	Recall the definition of $n_\cyc$ and $e_\mult$ from \eqref{eq:def:cyclic:comp:multi:edge}. By Proposition \ref{prop:1stmo:aprioriestimate}, we can bound
	\begin{equation*}
	    \E \bZ_{\la^\star, s_n}[n_\cyc \geq \log n \quad\textnormal{or}\quad e_\mult \geq 1\quad \textnormal{or}\quad (n_\fff)_{\fff \in \FFF} \notin \ee_{\frac{1}{2}}]\lesssim_k \frac{\log n}{n^{2/3}}\E\bZ_{\la^\star}\lesssim_k \frac{\log n}{n^{2/3}}\E\bZ^\tr_{\la^\star},
	\end{equation*}
	where the last inequality is due to Corollary \ref{cor:cyclic:contribution:1stmo}. Furthermore, $\E \bZ^{\tr}_{\la^\star} \lesssim_k \sqrt{n}\E \bZ^{\tr}_{\la^\star, s_n}$ holds by Theorem \ref{thm:1stmo:constant} and Theorem \ref{thm:1stmo:constant:s}. Thus, we have
	\begin{equation}\label{eq:uni:apriori}
	      \E \bZ_{\la^\star, s_n}[n_\cyc \geq \log n \quad\textnormal{or}\quad e_\mult \geq 1\quad \textnormal{or}\quad (n_\fff)_{\fff \in \FFF} \notin \ee_{\frac{1}{2}}]\lesssim_k \frac{\log n}{n^{1/6}}\E\bZ^\tr_{\la^\star,s_n}.
	\end{equation}
	Having \eqref{eq:uni:apriori} in mind, we now consider the case where $n_\cyc <\log n$, $e_\mult =0$ and $(n_\fff)_{\fff \in \FFF} \in \ee_{\frac{1}{2}}$. Denote by $\E\bZ^\uni_{\la^\star, s_n}[B]$ the contribution to $\E\bZ_{\la^\star, s_n}$ from component configurations $\csig$ with $e_\mult[\csig]=0$ and $B[\csig]=B$. We now divide the set of $B\in \bDelta_n^{\textnormal{b}}$ into $||B-B^\star||_1>n^{-1/3}$ and $||B-B^\star||_1\leq n^{-1/3}$ to compute the \textsc{rhs} of \eqref{eq:def:beta:1stmo}. We will argue that the former case of $||B-B^\star||_1>n^{-1/3}$ gives a negligible contribution to $\E\bZ_{\la^\star,s_n}$ while the latter case gives the precise constant. Throughout, we assume $(n_{\fff})_{\fff \in \FFF_\uni} \in \ee_{\frac{1}{2}}$ and $(n_\fff)_{\fff\in\FFF_\uni}\sim (B,s_n)$, where $\FFF_\uni \equiv \{\fff\in \FFF:\gamma(\fff)\leq 0\}$, and $(n_\fff)_{\fff\in\FFF_\uni}\sim (B,s)$ indicates that \begin{equation}\label{eq:compat:free:profile:boundary:cyclic}
	   (n_\fff)_{\fff\in\FFF_\uni}\sim (B,s) \overset{def}{\iff} \sum_{\fff\in\FFF_\uni} n_\fff\boeta^{-}_{\fff}=n\uh(B)\quad\text{and}\sum_{\fff\in\FFF_\uni} n_\fff s_\fff^{\lit} \in [ns,ns+1)
	\end{equation}
	where $\boeta^{-}_{\fff}\equiv\big(-\gamma(\fff), \boeta^{\partial}_\fff \big)\equiv \left(-\gamma(\fff), \{\eta_\fff(x)\}_{x \in \partial}\right)$. Here, note that $\boeta^{-}_{\fff}$ for $\fff\in \FFF_{\uni}$ is a generalization of $\boeta^{-}_{\ttt}$ for $\ttt\in \FFF_{\tr}$. In particular, $v(\fff)+f(\fff)<\frac{4\log n}{k}$ holds for all $\fff \in \FFF_\uni$ such that $n_{\fff}\neq 0$, and $\sum_{\fff\in \FFF_\tr}n_\fff = nh_{\circ}(B)$.
	
	\vspace{2mm}
    \noindent \textbf{Case 1.} $B\in \bDelta_n^{\textnormal{b}}$, $||B-B^\star||_1>n^{-1/3}$.
	\vspace{2mm}
	
	We proceed by a comparison argument. For free component profile $(n_{\fff})_{\fff \in \FFF_\uni}\sim B$, we construct $\TT\left(B, (n_{\fff})_{\fff \in \FFF_\uni}\right)\equiv \left(B^\prime,(n^\prime_{\ttt})_{\ttt \in \FFF_\tr}\right)$ such that $(n^\prime_{\ttt})_{\ttt \in \FFF_\tr}\sim B^\prime$ holds and $\left(B^\prime,(n^\prime_{\ttt})_{\ttt \in \FFF_\tr}\right)$ is sufficiently close to $\left(B,(n_{\ttt})_{\ttt \in \FFF_\tr}\right)$, and $\Big| \TT^{-1}\left(B^\prime,(n^\prime_{\ttt})_{\ttt \in \FFF_\tr}\right) \Big|$ is not too large. The steps to construct $\TT$ can be found below. In what follows, we denote by $e^{\textnormal{sw}}\equiv e^{\textnormal{sw}}(B)$ the number of $\bb$ edges, either $\bb_0$ or $\bb_1$, that neighbor separating clause, which could be swapped with $\fs$ edge and still make the clause a valid separating clause, i.e.
	\begin{equation*}
	    e^{\textnormal{sw}} = \sum_{\sig:\textnormal{non-forcing}}m\hat{B}(\sig)\left(\sum_{i=1}^{k}\one(\sigma_i \in \{\bb\})-2\right).
	\end{equation*}
	Denote the number of separating, but non-forcing, clauses by $m_{\fs}\equiv m_{\fs}(B)$. By \eqref{eq:B:red:free:small}, the number of $\fs$ edges are bounded above by $\frac{7k}{2^k}m$, because the variable adjacent to $\fs$ edge must be free. Hence,
	\begin{equation}\label{eq:swap:edge:linear}
	    e^{\textnormal{sw}} \geq (k-2)m_{\fs}-\frac{7k}{2^k}m \geq (k-2)(1-\frac{14k}{2^k})m-\frac{7k}{2^k}m=\left(k-2-O(k^2 2^{-k})\right)m,
	\end{equation}
	which ensures that there are at least linear number of fraction of $\bb$ edges that could be swapped with $\fs$ and still make the separating clauses valid. This fact will be important for the Step 2 below.
	\begin{enumerate}[label=\textnormal{Step \arabic*:}]
	\item For each $\fff \in \FFF_\uni\setminus \FFF_{\tr}$, we disassemble $n_\fff$ number of $\fff$'s by cutting all internal edges $e \in E(\fff^{\textnormal{in}})$ into half and adding the color $\fs$ to all cut half-edges. In the process of cutting, we delete the information of literals. Hence, every free variable $v \in V(\fff)$ becomes a free tree with a single variable. Likewise, each non-separating clause $a \in F(\fff)$ becomes a possibly invalid separating clause with its neighborhood color all determined. The only way for $a$ to be invalid after this cutting process, is when $a$ has internal degree $k-1$ or $k$ in $\fff$ i.e. when $a$ has $k-1$ or $k$ number of $\fs$ edges after being cut.
	\item For each invalid clause $a$, we swap two of its neighboring $\fs$ edges with two of $e^{\textnormal{sw}}$ $\bb$ edges, which could be swapped with $\bb$ edge and still make the separating clauses valid. Since the total number of clauses contained in the unicylic components is no greater than $\frac{4(\log n)^2}{k}\ll e^{\textnormal{sw}}$, we can guarantee that every invalid clause can be made valid by this swapping process.
	\end{enumerate}
	Step 1 above produces a new free tree profile $(n^\prime_{\ttt})_{\ttt \in \FFF_\tr}$ while Step 1 and 2 produce a new boundary profile $B^\prime$. We define $\TT\left(B,(n_{\fff})_{\fff \in \FFF_\uni}\right)\equiv \left(B^\prime,(n^\prime_{\ttt})_{\ttt \in \FFF_\tr}\right)$. The crucial properties of $\TT$ are summarized as follows.
	\begin{itemize}
	    \item By our construction, $B^\prime \in \bDelta_n$ and $(n^\prime_{\ttt})_{\ttt\in\FFF_\tr}\sim B^\prime$.
	    \item Because we have changed $O_{k}\left((\log n)^2\right)$ number of boundary spins in the process above, $||B^\prime-B||_1\lesssim_k \frac{\log^2 n}{n}$. Moreover, $n^\prime_\ttt =n_\ttt$ if $\ttt$ is not the free tree with a single free variable, while $|n^\prime_{\ttt}-n_\ttt|\lesssim \log^2 n$ holds if $\ttt$ is the free tree with a single variable. As a result, \eqref{eq:prod:form:1stmo} in Proposition \ref{prop:1stmo:B nt decomp} shows
	    \begin{equation}\label{eq:TT:ratio:uni:tree:1stmo}
	    \E \bZ^\uni_{\la^\star}[B,(n_{\ttt})_{\ttt \in \FFF_\tr}, (n_{\fff})_{\fff \in \FFF_\uni}] = \exp\left(O_k(\log^3 n)\right)\E \bZ^\tr_{\la^\star}[\TT\left(B, (n_{\fff})_{\fff \in \FFF_\uni}\right)]
	    \end{equation}
	    
	    \item$(n^\prime_{\ttt})_{\ttt\in\FFF_{\tr}}\in \ee_{\frac{1}{4}}$, since $(n_{\fff})_{\fff \in \FFF_\uni}\in \ee_{\frac{1}{2}}$.
	    \item For any $(n^\prime_{\ttt})_{\ttt\in\FFF_\tr}\sim B^\prime$, we can upper bound $\big| \TT^{-1}(B^\prime, n^\prime_{\ttt})_{\ttt\in\FFF_\tr})\big|$ by multiplying the number of ways to choose the new single free trees, the new separating clauses, the edges to be swapped and the ways to form the unicylic components among the single free trees and separating clauses. Note that the number of ways to form cyclic components among at most $a$ variables and $b$ clauses can be crudely upper bounded by $(4da)^{kb}$ since the number of matching between the clauses and variables are at most $(da)^{kb}$ and there are choices for literals of each half-edge adjacent to clauses and also the choice to be boundary or internal edge. Hence, we can crudely bound
	    \begin{equation*}
	    \begin{split}
	        \Big| \TT^{-1}(B^\prime, n^\prime_{\ttt})_{\ttt\in\FFF_\tr})\Big|
	        &\leq \underbrace{ \left[{n\choose \frac{4(\log n)^{2}}{k}}\right]}_\text{single free trees}\underbrace{\left[{m\choose \frac{4(\log n)^{2}}{k}}\right]}_\text{separating clauses}\underbrace{\left[{mk\choose \frac{4(\log n)^{2}}{k}}\right]}_\text{swapped edges}\Big(\frac{16d(\log n)^{2}}{k}\Big)^{4(\log n)^{2}}=\exp\left(O_{k}(\log^3 n)\right)
	   \end{split}
	    \end{equation*}
	\end{itemize}
	Therefore, using the above properties of $\TT$, we can bound
	\begin{equation}\label{eq:substitute:cyclic:by:tree:bound}
	\begin{split}
	    &\E \bZ^\uni_{\la^\star,s_n}[||B-B^\star||_1>n^{-1/3}, n_\cyc<\log n, \left((n_{\fff})_{\fff \in \FFF_\uni}\right)\in \ee_{\frac{1}{2}}]
	    \\%&\leq \E \bZ^\uni_{\la^\star}[||B-B^\star||_1>n^{-1/3}, n_\cyc<\log n, \left((n_{\fff})_{\fff \in \FFF_\uni}\right)\in \ee_{\frac{1}{2}}]\\
	    &\leq\exp\left(O_k(\log^3 n)\right)\E \bZ^\tr_{\la^\star}[||B-B^\star||_1>0.5n^{-1/3}, (n_\ttt)_{\ttt\in\FFF_\tr}\in \ee_{\frac{1}{4}}].
	\end{split}
	\end{equation}
	By Proposition \ref{prop:maxim:1stmo} and \ref{prop:negdef}, we can further bound the RHS above by
	\begin{equation*}
	   \E\bZ^\tr_{\la^\star}[||B-B^\star||_1>0.5n^{-1/3}, (n_\ttt)_{\ttt\in\FFF_\tr}\in \ee_{\frac{1}{4}}]= \exp\left(-\Omega_k(n^{1/3})\right)\E\bZ^\tr_{\la^\star}.
	\end{equation*}
	By Theorem \ref{thm:1stmo:constant} and Theorem \ref{thm:1stmo:constant:s}, $\E\bZ^\tr_{\la^\star,s_n}$ differs $\E\bZ^\tr_{\la^\star}$ by a factor $n^{-\frac{1}{2}}$. As a result,
	\begin{equation}\label{eq:prop:uni:negligible}
	    \E \bZ^\uni_{\la^\star,s_n}[||B-B^\star||_1>n^{-1/3}, n_\cyc<\log n, \left((n_{\fff})_{\fff \in \FFF_\uni}\right)\in \ee_{\frac{1}{2}}] =\exp\left(-\Omega_k(n^{1/3})\right)\E\bZ^\tr_{\la^\star,s_n}
	\end{equation}
	
	\vspace{2mm}
    \noindent \textbf{Case 2.} $B\in \bDelta_n^{\textnormal{b}}$, $||B-B^\star||_1\leq n^{-1/3}$.
	\vspace{2mm}
	
	Denote by $\bZ^{\uni,\circ}_{\la^\star, s_n}$ the contribution to $\bZ^{\uni}_{\la^\star, s_n}$ where there is no free component larger than $\frac{4\log n}{k}$, i.e. $n_\fff=0$ if $v(\fff)+f(\fff)>\frac{4\log n}{k}$. In this regime, we compute $\bZ^{\uni,\circ}_{\la^\star, s_n}[B,n_\cyc=r]$ and show that it is asymptotically a constant factor of $\E \bZ^\tr_{\la^\star,s_n}[B]$, where the constant does not depend on $B$. Let $\FFF_\uni^{\circ,n}\equiv \{\fff\in \FFF_\uni: \gamma(\fff)=0,v(\fff)+f(\fff)\leq \frac{4\log n}{k}\}$ and denote $g(n,B,r,s_n)=o_n(1)$, whenever 
	$$\lim_{n\to\infty}\sup_{||B-B^\star||_1\leq n^{-1/3}}\sup_{0\leq r<\log n}|g(n,B,r,s_n)|=0.$$
	Recall the definition of $\utheta_\la(B,s)$ in Lemma \ref{lem:exist:theta}. The constant below will be crucial in the calculation: by using Lemma \ref{lem:opt:tree:decay} and finding an appropriate subtree for a given unicylic component it is not hard to see that for small enough $\delta>0$ and $||(B,s)-(B^\star,s^\star)||_1<\delta$,
	\begin{equation*}
	    \xi^\uni(B,s) := \sum_{\fff\in \FFF_\uni, \gamma(\fff)=0}J_{\fff}w_{\fff}^{\la^{\star}}\exp\left(\langle\utheta_{\la^\star}(B,s),\boeta_{\fff}\rangle\right)<\infty.
	\end{equation*}
	We denote $\xi^\uni_{n}(B,s)\equiv \sum_{\fff\in \FFF_\uni^{\circ,n}}J_{\fff}w_{\fff}^{\la^{\star}}\exp\left(\langle\utheta_{\la^\star}(B,s),\boeta_{\fff}\rangle\right)$ and it is clear from the above equation that $\xi^\uni_{n}(B,s)=(1+o_n(1))\xi^\uni(B^\star,s^\star)$.
	
	First, we can use \eqref{eq:prod:form:stir:1stmo} in Proposition \ref{prop:1stmo:B nt decomp} to compute
	\begin{equation}\label{eq:uni:B:r:expansion}
	\begin{split}
	    \E\bZ^{\uni,\circ}_{\la^\star,s_n}[B,n_\cyc=r]=\left(1+o_n(1)\right)&
	    \frac{\exp\left(nF_{\la^\star}(B,s_n)\right)}{p_\circ(n,B)}\frac{1}{\left(nh_{\circ}(B)+r\right)!}\frac{\left(nh_\circ(B)\right)^{nh_\circ(B)+r}}{e^{nh_\circ(B)}}\\
	    &\times \left(1+\frac{\xi^\uni_n(B,s_n)}{nh_0(B)}\right)^{nh_\circ(B)+r}\P^{\uni}_{r,\utheta_{\la^\star}(B,s_n)}\left(\AAA_{\uh(B),s_n,r}\right),
	\end{split}
	\end{equation}
	where $\P^{\uni}_{r,\utheta_{\la^\star}(B,s)}$ is taken with respect to i.i.d random free components $X_1,...,X_{nh_\circ(B)+r}\in \FFF_\tr \sqcup \FFF_\uni^{\circ,n}$ with distribution given below. Writing $\boeta_{\fff}\equiv(\boeta^{-}_{\fff},s_\fff^\lit)$,
	\begin{equation}\label{eq:distr:uni:rescaled}
	\begin{split}
	    \P^{\uni}_{r,\utheta_{\la^\star}(B,s)}(X_i=\ttt)&:=\frac{J_\ttt w_\ttt^{\la^\star}\exp\left(\langle\utheta_{\la^\star}(B,s), \boeta_{\ttt}\rangle\right)}{h_\circ(B)\left(1+\frac{\xi^\uni_{n}(B,s_n)}{nh_0(B)}\right)}\quad\textnormal{for}\quad \ttt \in \FFF_{\tr}\quad\textnormal{and}\\
	    \P^{\uni}_{r,\utheta_{\la^\star}(B,s)}(X_i=\fff)&:=\frac{J_\fff w_\fff^{\la^\star}\exp\left(\langle\utheta_{\la^\star}(B,s), \boeta_{\fff}\rangle\right)}{nh_\circ(B)\left(1+\frac{\xi^\uni_n(B,s_n)}{nh_0(B)}\right)}\quad\textnormal{for}\quad \fff \in \FFF_{\uni}^{\circ,n}.
	\end{split}
	\end{equation}
	$\AAA_{\uh(B),s_n,r}$ is the event regarding $X_1,...,X_{nh\circ(B)+r}$, defined by $\AAA_{\uh(B),s_n,r}\equiv \AAA^\prime_{\uh(B),s_n,r} \cap \EEE_r$, where 
	\begin{equation*}
	\begin{split}
	    \AAA^\prime_{\uh(B),s_n,r}&:=\bigg\{\sum_{i=1}^{nh_\circ(B)+r}\boeta^{\partial}_{X_i}=n\uh^{\partial}(B),\quad\textnormal{and}\quad \sum_{i=1}^{nh_\circ(B)+r}s_{X_i}^\lit \in [ns_n,ns_n+1)\bigg\}\\
	    \EEE_r &:=\bigg\{\sum_{i=1}^{nh_\circ(B)+r}\one\left(X_i \in \FFF_\uni^{\circ,n}\right)=r\bigg\}
	\end{split}
	\end{equation*}
	In the equation above, $\uh(B)\equiv\left(h_\circ(B),\uh^\partial(B)\right)$. We compute $\P^{\uni}_{r,\utheta_{\la^\star}(B,s_n)}\left(\AAA_{\uh(B),s_n,r}\right)$ by conditioning on the event where we specify the index and the type of cyclic free components. For $I=\{i_1,...,i_r\}$, where  $1\leq i_1\leq ...\leq i_r\leq nh_{\circ}(B)+r$, and $J=(\fff_1,...,\fff_r) \in (\FFF_\uni^{\circ,n})^{r}$, define the event
	\begin{equation*}
	    \EEE_{I,J}:= \left\{\textnormal{$X_{i_\ell}=\fff_\ell$ for all $1\leq \ell \leq r$ and $X_j \in \FFF_{\tr}$ for $j \notin \{i_1,....,i_r\}$}\right\}.
	\end{equation*}
	Note that conditional on $\EEE_{I,J}$, the distribution of $(X_j)_{j\notin \{i_1,....,i_r\}} \in \FFF_\tr$ under $\P^{\uni}_{r,\utheta_{\la^\star}(B,s)}$ is given by $\P_{\utheta_{\la^\star}(B,s)}$, defined in \eqref{eq:def:rescaling:prob:B:s}. Moreover, since $\fff_1,...,\fff_r \in \FFF_\uni^{\circ,n}$, $||\sum_{i=1}^{r}\boeta_{\fff_i}||_1 \lesssim_k \log^2 n \ll\sqrt{n}$, so local CLT shows
	\begin{equation*}
	    \P^{\uni}_{r,\utheta_{\la^\star}(B,s)}\left(\AAA^\prime_{\uh(B),s_n,r} \Big| ~~\EEE_{I,J}\right)=\left(1+o_n(1)\right)\P_{\utheta_{\la^\star}(B,s)}\left(\AAA_{\uh(B),s_n}\right),
	\end{equation*}
	where $\AAA_{\uh(B),s}$ is defined in \eqref{eq:B:s:1stmo:expansion}. Therefore,
	\begin{equation}\label{eq:uni:local:CLT}
	\begin{split}
	   \P^{\uni}_{r,\utheta_{\la^\star}(B,s)}\left(\AAA_{\uh(B),s_n,r}\right)
	   &=\sum_{I \subset\{1,...,nh_{\circ}(B)+r\}}~~~\sum_{J \in (\FFF_\uni^{\circ,n})^{r}} \P^{\uni}_{r,\utheta_{\la^\star}(B,s)}\left(\AAA^\prime_{\uh(B),s_n,r} \Big| ~~\EEE_{I,J}\right)\P^{\uni}_{r,\utheta_{\la^\star}(B,s)}(\EEE_{I,J})\\
	   &=\left(1+o_n(1)\right)\P_{\utheta(B,s_n)}\left(\AAA_{\uh(B),s_n}\right)\P^{\uni}_{r,\utheta_{\la^\star}(B,s)}\big(\EEE_r\big)
	\end{split}
	\end{equation}
	Since $\P^{\uni}_{r,\utheta_{\la^\star}(B,s)}\big(X_i \in \FFF_\uni^{\circ, n}\big)=\left(1+\frac{\xi^\uni_n(B,s_n)}{nh_0(B)}\right)^{-1}\frac{\xi^\uni_n(B,s_n)}{nh_0(B)}$, we can approximate $\P^{\uni}_{r,\utheta_{\la^\star}(B,s)}\big(\EEE_r\big)$ by
	\begin{equation}\label{eq:number:cycle:poisson}
	    \P^{\uni}_{r,\utheta_{\la^\star}(B,s)}\big(\EEE_r\big)=\left(1+o_n(1)\right)\P\big(Z=r\big),\quad\textnormal{where}\quad Z \sim \textnormal{Poisson}\left(\xi^{\uni}(B^\star,s^\star)\right).
	\end{equation}
	Moreover, other terms in the \textsc{rhs} of \eqref{eq:uni:B:r:expansion} can be approximated by
	\begin{equation}\label{eq:uni:stirling:approx}
	\begin{split}
	    \frac{1}{\left(nh_{\circ}(B)+r\right)!}\frac{\left(nh_\circ(B)\right)^{nh_\circ(B)+r}}{e^{nh_\circ(B)}}&=\left(1+o_n(1)\right)\frac{1}{\left(nh_{\circ}(B)\right)!}\left(\frac{nh_\circ(B)}{e}\right)^{nh_\circ(B)}\\
	    \left(1+\frac{\xi^\uni_n(B,s_n)}{nh_0(B)}\right)^{nh_\circ(B)+r}&=\left(1+o_n(1)\right)e^{\xi^\uni(B^\star,s^\star)}.
	\end{split}
	\end{equation}
	 After plugging \eqref{eq:uni:local:CLT}, \eqref{eq:number:cycle:poisson} and \eqref{eq:uni:stirling:approx} into \eqref{eq:uni:B:r:expansion}, and comparing it with \eqref{eq:B:s:1stmo:expansion}, we have
	 \begin{equation*}
	      \E\bZ^{\uni,\circ}_{\la^\star,s_n}[B,n_\cyc=r]=\left(1+o_n(1)\right)\E\bZ^\tr_{\la^\star, s_n}[B]e^{\xi^\uni(B^\star,s^\star)}\P\big(Z=r\big).
	 \end{equation*}
	 Finally, summing the above equation for $0\leq r <\log n$ shows 
	\begin{equation}\label{eq:uni:B:compare:tr}
	    \E\bZ^{\uni,\circ}_{\la^\star,s_n}[B,n_\cyc<\log n]=\left(1+o_n(1)\right)e^{\xi^\uni(B^\star,s^\star)}\E\bZ^\tr_{\la^\star, s_n}[B].
	\end{equation}
	Finally, by \eqref{eq:uni:apriori} and \eqref{eq:prop:uni:negligible}, summing \eqref{eq:uni:B:compare:tr} for $||B-B^\star||<n^{-1/3}$ shows that \eqref{eq:def:beta:1stmo} holds with $\beta_1(\la^\star,s^\star)=e^{\xi^\uni(B^\star,s^\star)}$.
    \end{proof}

	\begin{proof}[Proof of Theorem \ref{thm1}-\textnormal{(a)}]
		Fix $\eps>0$. For $\lambda^\star, s^\star$ defined in \eqref{eq:def:sstarandlambdastar}, denote $s_{\circ}(C) := s^\star - \frac{\log n}{2\lambda^\star n} +\frac{C}{n}$ for $C\in \Z$ (recall  that $s^\star = \textsf{f}^{1\textsf{rsb}}(\alpha)$). We first show that there are no clusters bigger than $e^{n s_n(C_0)}$ with probability $1-\frac{\eps}{2}$ for some $C_0=C_0(\eps,\alpha,k)$: note that Theorem \ref{thm:1stmo:constant:s} and Proposition \ref{prop:ratio:uni:1stmo} imply that for $0\leq C\leq n^{1/4}$,
		\begin{equation*}
		\E \bN_{s_{\circ}(C)}\leq e^{-n\la^\star s_{\circ}(C)}\E \bZ_{\la^\star,s_{\circ}(C)}\lesssim_{k}e^{-n\la^\star s_{\circ}(C)}\E \bZ_{\la^\star,s_{\circ}(C)}^{\tr} \lesssim_{k} \frac{1}{\sqrt{n}}e^{n\la^\star s^\star}e^{-n\la^\star s_{\circ}(C)}=e^{-\la^\star C}.
		\end{equation*}
		On the other hand, $\sum_{C\geq n^{1/4}}\bN_{s_{\circ}(C)}\leq e^{-n\la^\star s_{\circ}( n^{1/4})}\bZ_{\la^\star}$ holds. Thus, Corollary \ref{cor:cyclic:contribution:1stmo} and Theorem \ref{thm:1stmo:constant} show
		\begin{equation*}
		    \sum_{C\geq n^{1/4}}\E \bN_{s_{\circ}(C)}\leq e^{-n\la^\star s_{\circ}(n^{1/4})}\E\bZ_{\la^\star}\lesssim_{k}e^{-n\la^\star s_{\circ}(n^{1/4})}\E\bZ_{\la^\star}^{\tr}\lesssim_{k} \sqrt{n} e^{-\la^\star n^{1/4}}.
		\end{equation*}
		Consequently, Markov's inequality implies that for $C_0\in \Z$ with $C_0\leq n^{\frac{1}{5}}$,
		\begin{equation}\label{eq:mainthm:1}
		\P\Big(\sum_{C\geq C_0} 
		\bN_{s_{\circ}(C)}\geq 1
		\Big) \leq \sum_{C\geq C_0}\E \bN_{s_{\circ}(C)}\lesssim_{k} e^{-\la^\star C_0}.
		\end{equation}
		 Hence, by taking $C_0=C_0(\eps,k)$ large enough for given $\eps>0$, there are no clusters of size bigger than $e^{ns_{\circ}(C_0)}$ with probability $1-\frac{\eps}{2}$.
		 
		 Next, we upper bound $\sum_{C\leq C_0} \E \bZ_{1, s_{\circ}(C)}$: proceeding in a similar fashion as before, Theorem \ref{thm:1stmo:constant:s} and Proposition \ref{prop:ratio:uni:1stmo} imply that for $-n^{1/4}\leq C\leq C_0$,
		 \begin{equation*}
		     \E \bZ_{1,s_{\circ}(C)}\leq e^{n(1-\la^\star)s_{\circ}(C)+1} \E \bZ_{\la^\star, s_{\circ}(C)}\lesssim_{k}\frac{1}{\sqrt{n}}e^{n\la^\star s^\star}e^{n(1-\la^\star)s_{\circ}(C)}=n^{-\frac{1}{2\la^\star}}e^{ns^\star+(1-\la^\star)C}.
		 \end{equation*}
		 In the regime where $C\leq -n^{1/4}$, Corollary \ref{cor:cyclic:contribution:1stmo} and Theorem \ref{thm:1stmo:constant} show
		 \begin{equation*}
		     \sum_{C\leq -n^{1/4}}\E \bZ_{1,s_{\circ}(C)}\leq e^{n(1-\la^\star)s_{\circ}(-n^{1/4})+1}\E \bZ_{\la^\star}\lesssim_{k}e^{n(1-\la^\star)s_{\circ}(-n^{1/4})}e^{n\la^\star s^\star}=n^{-\frac{1}{2\la^\star}}e^{ns^\star}\times \sqrt{n}e^{-(1-\la^\star)n^{1/4}}.
		 \end{equation*}
		 Thus, altogether we have the following:
		 \begin{equation}\label{eq:thm1:proof:useful}
		 \sum_{C\leq C_0} \E \bZ_{1, s_{\circ}(C)}\lesssim_{k}n^{-\frac{1}{2\la^\star}}e^{ns^\star+(1-\la^\star)C_0}.
		 \end{equation}
		 Hence, Markov's inequality implies that
		 \begin{equation}\label{eq:mainthm:2}
		     \P\Big(\sum_{C\leq C_0}\bZ_{1, s_{\circ}(C)}\geq C_k \eps^{-1}n^{-\frac{1}{2\la^\star}}e^{ns^\star+(1-\la^\star)C_0}\Big)\leq \frac{\eps}{2},
		 \end{equation}
		 for some constant $C_k$, which only depends on $k$. Therefore, by \eqref{eq:mainthm:1} and \eqref{eq:mainthm:2}, we can take $C(\eps,\alpha,k)=(1-\la^\star)C_0-\log \eps +\log C_k$ in the statement of Theorem \ref{thm1}-(a) to conclude the proof.
	\end{proof}
	
	\section{The second moment}\label{sec:2ndmo}
	
		The goal of this section is to compute $\E (\bZ_{\la^\star,s^\star})^2$ up to a constant. Computing the second moment is equivalent to calculating the first moment of the pair frozen model partition function and we restrict our attention to the case where neither individual frozen configuration $\ux^i$ ($i=1,2$) contains a free cycle, so that the pair frozen model has 1-1 correspondence with pair-coloring model or pair-component model (see \eqref{eq:corr:pair-frozen:color:comp}). In Section \ref{subsec:2ndmo:independent}, we calculate the contribution from near-independence regime, where the \textit{overlap} $\zeta(\ux^1,\ux^2)$ satisfies $|\zeta(\ux^1,\ux^2)-\frac{1}{2}|< \frac{k^2}{2^{k/2}}$, and in Section \ref{subsec:2ndmo:correlated}, we calculate the contribution from the correlated regime, where $|\zeta(\ux^1,\ux^2)-\frac{1}{2}|\geq\frac{k^2}{2^{k/2}}$. To this end, we first define the overlap $\zeta(\ux^1,\ux^2)$:
	
		\begin{defn}\label{def:overlapofcol}
	 	For a pair frozen configuration $(\ux^1,\ux^2)\in (\{0,1,\ff\}^{2})^V$, the \textit{overlap} $\zeta(\ux^1,\ux^2)$ is defined as
	 	\begin{equation*}
	 	\zeta(\ux^1,\ux^2) = \frac{1}{n} d(\ux^1,\ux^2),
	 	\end{equation*}
	 	where $d(\ux^1,\ux^2)$ denotes the Hamming distance between $\ux^1$ and $\ux^2$. For a valid pair-coloring configuration $\bsig \in \Omega_2^E$, there exists a unique pair frozen configuration $(\ux^1,\ux^2)$ corresponding to $\bsig$, so $\zeta(\bsig) \equiv \zeta(\ux^1,\ux^2)$ is well defined. Similarly, $\zeta(\bcsig)$ for a valid pair component configuration $\bcsig \in \Omega_{\textnormal{com},2}^E$ is well defined.
	 \end{defn}

	\subsection{Near-independence regime}\label{subsec:2ndmo:independent}
	For $\ula\in [0,1]^2$, denote by $\bZ^{2}_{\ula,\ind}$(resp. $\bZ^{2,(L)}_{\ula,\ind}$) the contribution to $\bZ^{2}_{\ula}$(resp. $\bZ^{2,(L)}_{\ula}$) from pair-coloring $\bsig$ with $|\zeta(\bsig)-\frac{1}{2}|< \frac{k^2}{2^{k/2}}$. Moreover, denote the contribution to $\bZ^2_{\ula,\ind}$ from the pair-coloring whose union components are trees by
	\begin{equation*}
	    \bZ^{2,\tr}_{\ula,\ind}:= \sum_{\substack{\bsig\in \Omega_{2}^{E}\\ \big|\zeta(\bsig)-\frac{1}{2}\big|< \frac{k^2}{2^{k/2}}}}\bw_{\GGG}^\lit(\bsig)^{\ula}\one\Big\{\textnormal{union-free components of $(\GGG,\bsig)$ are trees and } \frac{\rr(\bsig)}{nd}\vee \frac{\ff(\bsig)}{n}\leq \frac{7}{2^k} \Big\}, 
	\end{equation*}
     where $\rr(\bsig)\equiv \rr(\sig^1)\vee\rr(\sig^2)$ denotes the maximum of the number of red edges in first and second copies, and similarly, $\ff(\bsig)\equiv \ff(\sig^1)\vee\ff(\sig^2)$. 
	
    Similarly for $\bs=(s_1,s_2) \in [0,\log 2)^{2}$, we denote by $\bZ^{2,\tr}_{\ula,\bs,\ind}$ the contribution to $\bZ^{2}_{\ula,\bs}$ from $\bsig=(\sig^1,\sig^2) \in \Omega_{2}^{E}$ whose union-free components are trees and $|\zeta(\bsig)-\frac{1}{2}\Big|< \frac{k^2}{2^{k/2}}$ holds. Also, we add the superscript $(L)$ (e.g. $\bZ^{2,(L),\tr}_{\ula,\bs,\ind}$) when considering the $L$-truncated model($L$-truncated in each of the copies).
 
	In the near-independence regime, most of the computation for the second moment will follow the same argument as the computation for the first moment from Section \ref{sec:1stmo}. Hence, we will omit the proof whenever the argument for the first moment extends to the second moment in almost identical fashion, and refer to the proof done in Section \ref{sec:1stmo}.
	
	To begin with, define the sets of \textit{non-free} pair-colors  $\dot{\partial}_2^\bullet$, $\hat{\partial}_2^\bullet$ by
	\begin{equation*}
	\dot{\partial}_2^\bullet := \{\rr_0,\rr_1,\bb_0,\bb_1\}^2; \quad \hat{\partial}_2^\bullet :=\{\rr_0,\rr_1,\bb_0,\bb_1,\fs\}^2,
	\end{equation*}
	which are the colors that can be adjacent to pair frozen variables or pair-separating clauses. Analogously to the single copy model, we define the \textit{union-free component profile} and \textit{boundary profile} for the pair model.
	
	\begin{defn} \label{def:empirical:boundary:2ndmo}
	Given a \textsc{nae-sat} instance $\GGG$ and a valid pair component configuration $\bcsig \in \pcomp^E$, the \textit{boundary profile} of $\bcsig$ is the tuple $(\bB[\bcsig],\buh[\bcsig])\equiv(\bB,\buh) \equiv (\dbB,\hbB,\bbB, \buh)$, and the \textit{union component profile} is the sequence $(n_\vvv[\bcsig])_{\vvv\in \mathscr{F}_2}\equiv(n_\vvv)_{\vvv\in \mathscr{F}_2}$, defined as follows.
	\begin{enumerate}
	    \item[$\bullet$] For each $\vvv\in\mathscr{F}_2$, let $n_\vvv$ is the number of union-free component $\vvv$ inside $(\GGG,\bcsig)$.
		\item[$\bullet$] $\dbB,\hbB,$ and $\bbB$ are measures on $(\dot{\partial}_2^\bullet)^d$, $(\hat{\partial}_2^\bullet)^k$ and $\hat{\partial}_2^\bullet$ respectively, given by
		\begin{equation*}
		\begin{split}
		&\dbB(\butau)
		:=
		|\{v\in V: \bcsig_{\delta v}=\butau \} | / |V| \quad
		\textnormal{for all } \butau\in (\dot{\partial}_2^\bullet)^d;\\
		 &\hbB(\butau)
		:=
		|\{a\in F: \bcsig_{\delta a}=\butau \} | / |F| \quad
		\textnormal{for all } \butau\in (\hat{\partial}_2^\bullet)^k;\\
		 &\bbB(\btau)
		:=
		|\{e\in E: \boldsymbol{\sigma}^{\textnormal{com}}_e=\btau \} | / |E| \quad
		\textnormal{for all } {\btau}\in \hat{\partial}_2^\bullet.
		\end{split}
		\end{equation*}
		Furthermore, $\buh=(\bh(\circ), \{\bh(\bx)\}_{\bx\in \dot{\partial}_2\sqcup\hat{\partial}_2})$ records the total number of  components and boundary colors of all union components normalized by the number of variables, where $\dot{\partial}_2$ and $\hat{\partial}_2$ are defined in \eqref{eq:def:boundarycolors:pair}:
	\begin{equation}\label{eq:def:h:2ndmo}
	\bh(\circ) := \frac{1}{|V|}\sum_{\vvv\in \mathscr{F}_2} n_\vvv,\qquad\bh(\bx) := \frac{1}{|V|} \sum_{\vvv\in\mathscr{F}_2} \eta_\vvv(\bx)\,n_\vvv, ~~~~\bx\in \dot{\partial}_2\sqcup\hat{\partial}_2
	\end{equation}
	Note that a valid boundary profile $(\bB,\buh)$ must satisfy the following compatibility condition for $\bx\in \hat{\partial}_2^\bullet$:
	\begin{equation}\label{eq:def:compat:2ndmo}
	\begin{split}
	\bbB(\bx)
	&=
	\frac{1}{d} \sum_{\bsig \in(\dot{\partial}_2^\bullet)^d } \dbB(\bsig)  \sum_{i=1}^d \one\{\bsigma_i =\bx \} + \frac{\one\{\bx\in \hat{\partial}_2\}  }{d} \bh(\bx)\\
	&=
	\frac{1}{k}
	\sum_{\bsig\in (\hat{\partial}_2^\bullet)^k} \hbB (\bsig)
	\sum_{j=1}^k \one \{\bsigma_j =\bx \}
	+
	\frac{\one\{\bx\in \dot{\partial}_2 \}}{d}  \bh(\bx).
	\end{split}
	\end{equation}
	\end{enumerate}
	\end{defn}
	\begin{remark}\label{rem:clustersize:2ndmo}
	\begin{enumerate}
	\item Henceforth, we denote $(n_{\vvv})_{\vvv \in \mathscr{F}_2} \sim \bB$ when $(n_{\vvv})_{\vvv \in \mathscr{F}_2}$ is compatible with $\bB$. That is, $\{\bh(\bx)\}_{\bx\in \dot{\partial}_2\sqcup\hat{\partial}_2}$ induced from $(n_{\vvv})_{\vvv \in \mathscr{F}_2}$ by \eqref{eq:def:h:2ndmo} satisfies the compatibility equation \eqref{eq:def:compat:2ndmo}, and $\sum_{\vvv\in \mathscr{F}_2}n_\vvv v(\vvv)=n\big(1-\langle \dbB, \one \rangle \big)$ holds.
	\item If all union-free components of $\bcsig$ are trees, i.e. $n_{\vvv}[\bcsig]=0$ for $\vvv \in \FFF_2\backslash \FFF_2^\tr$, $\bh(\circ)$ can be computed from $\bB$ by the same equation as in the first moment (see \eqref{eq:def:compat:1stmo:tree}). Together with \eqref{eq:def:compat:2ndmo}, $\buh$ corresponding to $\bB$ is well defined and we denote it by $\buh[\bB]$. 
	\item Given $\bB$, denote its marginal onto the first and the second copy by $B^1\equiv (\dot{B}^1,\hat{B}^1,\bar{B}^1)$ and $B^2\equiv(\dot{B}^2,\hat{B}^2,\bar{B}^2)$ respectively, e.g. $(\dot{B}^1,\dot{B}^2)$ are marginals of $\dbB$ onto the first and the second copy. Similarly, denote by $(n_{\ttt}^{i})_{\ttt\in \FFF_{\tr}},i=1,2$ the free tree profile in each copy induced by the union component profile $(n_{\vvv})_{\vvv\in \FFF_2}$.  
	\item Let $\bsig=(\sig^1,\sig^2)\in \Omega_2^{E}$ be the unique pair-coloring configuration corresponding to $\bcsig$ through \eqref{eq:corr:pair-frozen:color:comp}. Recalling \eqref{eq:def:size:union-comp:each:copy}, define $s_{\vvv}^{\lit, i} \equiv \log w^{\lit,i}(\vvv)$ for $\vvv \in \FFF_2$. Then, for $s_1,s_2\geq 0$ and $i=1,2$, $$w^{\lit}_{\GGG}(\sig^{i})\in [e^{ns_i},e^{ns_i+1}) \iff \sum_{\vvv\in \FFF_2}n_\vvv[\bcsig]s_{\vvv}^{\lit,i}\in [ns_i,ns_i+1).$$
	\end{enumerate}
	\end{remark}
	\begin{defn}
	$\pbDelta$ is the space of boundary profiles $\bB$ with the following conditions.
		\begin{enumerate}
				\item[$\bullet$] $\dbB,\hbB$ and $\bbB$ are measures supported on $(\textnormal{supp}~\dot{I})^2$, $\textnormal{supp}~\hat{v}_2$, and $\hat{\partial}_2^{\bullet}$ respectively, and they have total mass at most $1$. 
				\item[$\bullet$] If we denote marginals of $\bB$ by $B^1,B^2$, then \eqref{eq:B:red:free:small} holds for both $B^1$ and $B^2$.
				\item[$\bullet$] There exists $\{\bh(\bx)\}_{\bx\in \dot{\partial}_2\sqcup\hat{\partial}_2}\in \R_{\geq 0}^{|\dot{\partial}_2|+|\hat{\partial}_2|}$ such that \eqref{eq:def:compat:2ndmo} holds.\\
			Moreover, denote by $\pbDelta_n$ the set of $\bB \in \pbDelta$ satisfying the following extra condition.
				\item[$\bullet$] $\dbB, \hbB$ and $\bbB$ are integer multiples of $\frac{1}{n}, \frac{1}{m}$ and $\frac{1}{nd}$, respectively.
		\end{enumerate}
		%Note that $\pbDelta$ has dimension $\textnormal{dim}(\pbDelta)\equiv 2|\textnormal{supp} \dot{I}|+2|\textnormal{supp} \hat{v}|-| \hat{\partial}_2^\bullet|+|\partial_2|$.
	\end{defn}

	Analogously to \eqref{eq:def:cyclic:comp:multi:edge}, denote by $\prescript{}{2}{n}_\cyc=\prescript{}{2}{n}_\cyc[\bcsig]$ and $\prescript{}{2}{e}_\mult=\prescript{}{2}{e}_\mult[\bcsig]$ the number of cyclic union-free components and multicylic edges of union-free components of $\bcsig$ respectively. Also, analogously to \eqref{eq:def:exp:decay:profile}, let $\pee_{r},r>0,$ be the set of union-free component profile obeying exponential decay of frequencies in its number of variables with rate $2^{-rk}$. Proposition \ref{prop:2ndmo:aprioriestimate} is the analog of Proposition \ref{prop:1stmo:aprioriestimate} for the pair model, although its proof is technically more involved. The proof of Proposition \ref{prop:2ndmo:aprioriestimate} is presented in Appendix \ref{subsec:apriori:secmo}.
	\begin{prop}\label{prop:2ndmo:aprioriestimate}
	For $k\geq k_0, \ula \in [0,1]^{2}, L<\infty$ and $c\in [1,3]$, the following holds.
	\begin{enumerate}
    \item $\E \bZ^{2,(L),\tr}_{\ula,\ind}[(\pee_{\frac{1}{c+1}})^{\mathsf{c}}] \lesssim_{k} n^{-\frac{2}{3}c}\log n \E \bZ^{2,(L),\tr}_{\ula,\ind}$ and $\E \bZ^{2,\tr}_{\ula,\ind}[(\pee_{\frac{1}{c+1}})^{\mathsf{c}}] \lesssim_{k} n^{-\frac{2}{3}c}\log n \E \bZ^{2,\tr}_{\ula,\ind}$.
    \item $\E \bZ^{2}_{\ula,\ind}[\exists \vvv\in \FFF_2(\bsig,\GGG),~~~~f(\vvv)\geq v(\vvv)+2]\lesssim_{k} n^{-2}\E \bZ^{2}_{\ula,\ind}$.
    \item $\E \bZ^{2}_{\ula,\ind}[(\pee_{\frac{1}{c+1}})^{\mathsf{c}} \quad\textnormal{and}\quad \forall \vvv\in \FFF_2(\bsig,\GGG),~~~~f(\vvv)\leq v(\vvv)+1]\lesssim_{k} n^{-\frac{2}{3}c}\log n \E \bZ^{2}_{\ula,\ind}$.
	\end{enumerate}
	Moreover, there exists a universal constant $C$ such that for every $r,\gamma \in \Z_{\geq 0}$, the following holds.
	\begin{enumerate}[resume]
	\item $\E\bZ^{2}_{\ula,\ind}[\prescript{}{2}{n}_{\textnormal{cyc}}\geq r, \prescript{}{2}{e}_{\textnormal{mult}}\geq \gamma\textnormal{ and } \pee_{\frac{1}{4}}]\lesssim_{k} \frac{1}{r!}(\frac{Ck^2}{2^k})^{r}(\frac{C\log^{3} n}{n})^{\gamma}\E\bZ^{2,\tr}_{\ula,\ind}$.
	\end{enumerate}
	\end{prop}
	\begin{cor}\label{cor:2ndmo:treecontribution:constant}
	For $k\geq k_0$, $\ula \in [0,1]^2$, $\E \bZ^{2}_{\ula,\ind} \lesssim_k \E \bZ^{2,\tr}_{\ula,\ind}$.    
	\end{cor}
	Denote by $\bZ^{2}_{\ula} [\bB,\{n_\vvv \}_{\vvv \in \mathscr{F}_2}]$ the contribution to $\bZ^2_{\ula}$ from pair component coloring $\bcsig \in \pcomp^{E}$ with boundary profile $\bB[\bcsig]=\bB$ and union-free component profile $\{n_\vvv[\bcsig]\}_{\vvv \in \mathscr{F}_2}=\{n_\vvv\}_{\vvv \in \mathscr{F}_2}$. Then, the same proof for Proposition \ref{prop:1stmo:B nt decomp} extends to the second moment.
	\begin{prop}\label{prop:2ndmo:B nt decomp}
	For every $\bB\in \pbDelta_n$ and $\{n_{\vvv}\}_{\vvv\in \mathscr{F}_2}\sim \bB$, we have
	\begin{equation}\label{eq:prod:form:2ndmo}
	   \E \bZ^{2}_{\ula} [\bB, \{n_\uuu\}_{\vvv \in \mathscr{F}_2}] = \frac{n! m!}{nd !} \frac{(nd\bbB)!}{(n\dbB)! (m\hbB)!}\prod_{\bsig\in (\hat{\partial}_2^{\bullet})^{k}}\hat{v}_2(\bsig)^{m\hbB(\bsig)} \prod_{\vvv\in \mathscr{F}_2}\left[ \frac{1}{n_\vvv !} (d^{e(\vvv)-f(\vvv)}k^{f(\vvv)}J_\vvv \bw_\vvv^{\ula} )^{n_\vvv}\right],
	\end{equation}
	where $\bw_{\vvv}^{\ula} \equiv \bw^{\textnormal{com}}(\vvv)^{\ula}$ if $\vvv\in \FFF_2 \backslash\FFF_2^{\tr}$ and $\bw_{\uuu}^{\ula}\equiv \bw(\uuu)^{\ula}$ if $\uuu\in \FFF_2^{\tr}$. Thus, Stirling's approximation of $\frac{n! m!}{nd !} \frac{(nd\bbB)!}{(n\dbB)! (m\hbB)!}$ in \eqref{eq:prod:form:2ndmo} gives
	\begin{equation}\label{eq:prod:form:stir:2ndmo}
		\E \bZ^{2}_{\ula} [\bB, \{n_\vvv\}_{\vvv\in \mathscr{F}_2}] = \left(1+O_{k}\left(\frac{1}{n\kappa(\bB)}\right)\right)\frac{ e^{n\Psi_\circ(\bB)}}{p_\circ(n;\bB)} \prod_{\vvv\in \mathscr{F}_2}\left[ \frac{1}{n_\vvv !} \left(\left(\frac{e}{n}\right)^{\gamma(\vvv)}J_\vvv \bw_\vvv^{\ula} \right)^{n_\vvv}\right],
	\end{equation}
	where $\kappa(\bB)\equiv \min_{\dbB(\bsig)\neq 0,\hbB(\butau)\neq 0, \bbB(\bsigma)\neq 0}\left\{\dbB(\bsig),\hbB(\butau),\bbB(\bsigma)\right\}$, and $\Psi_\circ(\bB)$ and $p_{\circ}(n,\bB)$ are defined analogously to \eqref{eq:def:Psi:circ:p:circ:1stmo}, i.e. replace $B$ (resp. $\hat{v}$) by $\bB$ (resp. $\hat{v}_2$) on both sides of \eqref{eq:def:Psi:circ:p:circ:1stmo}.
	\end{prop}
	Analogously to Definition \ref{def:opt:coloring:profile} and \ref{def:opt:bdry:1stmo}, we now define the optimal coloring and boundary profiles for the pair model. To do so, we first state the BP contraction results from \cite{ssz22}: replacing $\dot{\Phi}^{\la}, \hat{\Phi}^{\la},\bar{\Phi}^{\la}$ by $\dot{\Phi}_2^{\ula}, \hat{\Phi}_2^{\ula},\bar{\Phi}_2^{\ula}$ in \eqref{eq:pre:BP:1stmo} defines

	\begin{equation*}
	    \dot{\textnormal{BP}}_{\ula,L}:\PPP\big((\hat{\Omega}_{L})^{2}\big) \rightarrow \PPP\big((\dot{\Omega}_{L})^{2}\big),\quad\hat{\textnormal{BP}}_{\ula,L}:\PPP\big((\dot{\Omega}_L)^{2}\big) \rightarrow \PPP\big((\hat{\Omega }_L)^{2}\big).
	\end{equation*}
	Then, define $\textnormal{BP}_{\ula,L} \equiv \dot{\textnormal{BP}}_{\ula,L}\circ \hat{\textnormal{BP}}_{\ula,L}$. The BP map for the untruncated model $\textnormal{BP}_{\ula}\equiv \textnormal{BP}_{\ula,\infty}$ is analogously defined. For $C>0, (c,\kappa) \in [0,1]^2$, let $\mathbf{\Gamma}(C, c,\kappa)$ be the set of $\dbq \in \PPP\big((\dot{\Omega}_L)^2\big)$ satisfying $\dbq(\boldsymbol{\dot{\sigma}})=\dbq(\boldsymbol{\dot{\sigma}}\oplus \mathbf{1})$ for $\boldsymbol{\dot{\sigma}}\in (\dot{\Omega}_L)^2$ and 
	\begin{align}\label{eq:def:bp:contract:set:2ndmo}
	   &|\dbq(\bb_0\bb_0)-\dbq(\bb_0\bb_1)|\leq (k^9/2^{ck})\dbq(\bb\bb),~~~~\textnormal{and}~~~~ \dbq(\ff\ff)+\dbq\big(\{\ff\rr,\rr\ff\}\big)/2^{k}+\dbq(\rr\rr)/4^k \leq (C/2^k)\dbq(\bb\bb);\\
	   &\dbq\big(\{\rr\ff,\ff\rr\}\big) \leq (C/2^{k\kappa})\dbq(\bb\bb)~~~~\textnormal{and}~~~~ \dot{q}(\rr\rr)\leq C 2^{k(1-\kappa)}\dbq(\bb\bb);\\
	   &\dbq(\rr_x\dsigma)\geq (1-C/2^k)\dbq(\bb_x\dsigma)~~~~\textnormal{and}~~~~\dbq(\dsigma \rr_x)\geq (1-C/2^k)\dbq(\dsigma \bb_x)~~~~\textnormal{for all}~~~x\in \{0,1\}, \dsigma \in \dot{\Omega}.
	\end{align}
	The following proposition for $\la_1=\la_2$ was shown in \cite{ssz22} and exactly the same proof works for the general case where $\la_1,\la_2 \in [0,1]$.
	\begin{prop}[\cite{ssz22}, Proposition 5.5]
	\label{prop:BPcontraction:2ndmo}
	Fix $\ula=(\la_1,\la_2) \in [0,1]^2$ and $1\leq L \leq \infty$.
	\begin{enumerate}
	    \item The map $\textnormal{BP}_{\ula,L}$ has a unique fixed point in $\mathbf{\Gamma}(1,1)$, given by $\dbq^\star_{\ula,L}:= \dot{q}^\star_{\la_1,L}\otimes \dot{q}^\star_{\la_2,L}$ with $\dot{q}^\star_{\la,L}$ as in Proposition \ref{prop:BPcontraction:1stmo}. Moreover, for $c\in [0,1]$ and $k$ sufficiently large, there is no other fixed point of $\textnormal{BP}_{\ula,L}$ in $\mathbf{\Gamma}(c,1)$: if $\dbq\in \mathbf{\Gamma}(c,1)$, then $\textnormal{BP}_{\ula,L}\dbq\in \mathbf{\Gamma}(1,1)$, with $$||\textnormal{BP}_{\ula,L}\dbq-\dbq^\star_{\ula,L}||_1=O(k^4/2^k)||\dbq-\dbq^\star_{\ula,L}||_1.$$ Hereafter, we will simply denote $\dbq^\star_{\ula}\equiv \dbq^\star_{\ula,\infty}$.
	    \item If $\dbq\in \mathbf{\Gamma}(c,0)$ with $\dbq=\textnormal{BP}_{\ula,L}\dbq$ for some $c\in (0,1]$, then $\dbq\in \mathbf{\Gamma}(c,1)$.
	\end{enumerate}
	\end{prop}
	\begin{defn}[\cite{ssz22}, Definition 5.6]
	    For $\dbq \in \PPP(\dot{\Omega}^2)$, define $\bH_{\dbq}$ analogously to \eqref{eq:H:q:1stmo} for the pair model:
	   \begin{equation}\label{eq:H:q:2ndmo}
	    \dbH_{\dbq}(\bsig)=\frac{\dot{\Phi}_2^{\ula}(\bsig)}{\dot{\mathfrak{Z}}_2}\prod_{i=1}^{d}\hbq(\hat{\bsigma}_i),\quad \hbH_{\dbq}(\bsig)=\frac{\hat{\Phi}_2^{\ula}(\bsig)}{\hat{\mathfrak{Z}}_2}\prod_{i=1}^{k}\dbq(\dot{\bsigma}_i),\quad \bbH_{\dbq}(\bsigma)=\frac{1}{\bar{\mathfrak{Z}}_2}\frac{\dbq(\dot{\bsigma})\hbq(\hat{\bsigma})}{\bar{\Phi}_2^{\ula}(\bsigma)},
	\end{equation}
	where $\dot{\mathfrak{Z}}_2\equiv \dot{\mathfrak{Z}}_{2,\dbq}, \hat{\mathfrak{Z}}_2\equiv \hat{\mathfrak{Z}}_{2,\dbq},\bar{\mathfrak{Z}}_2\equiv \bar{\mathfrak{Z}}_{2,\dbq}$ are normalizing constants and $\hbq\equiv \textnormal{BP}\dbq$.
	Then the optimal coloring profiles for the truncated pair model and the untruncated pair model are the tuples $\bH^\star_{\ula,L}:=\bH_{\dbq^\star_{\ula,L}}\equiv (\dbH^\star_{\ula,L},\hbH^\star_{\ula,L},\bbH^\star_{\ula,L})$ and $\bH^\star_{\ula}:=\bH_{\dbq^\star_{\ula}}\equiv(\dbH^\star_{\ula},\hbH^\star_{\ula},\bbH^\star_{\ula})$ respectively.
	\end{defn}
	\begin{defn}\label{def:opt:bdry:2ndmo}
	For $\ula\in [0,1]^2$, the optimal boundary profile and the optimal union-free tree profile for the pair model are defined as follows.
	\begin{itemize}
	    \item  The optimal boundary profile $\bB^\star_{\ula,L}$(resp. $\bB^\star_{\ula}$) for the truncated pair model(resp. the untruncated pair model) is defined analogously to \eqref{eq:def:optimal:bdry}, i.e. by restriction of $\bH^\star_{\ula,L}$(resp. $\bH^\star_{\ula}$) to $(\dot{\partial}^\bullet_2)^d, (\hat{\partial}_2^\bullet)^k, \hat{\partial}_2^\bullet$. Moreover, recalling Remark \ref{rem:compat:bdry:tree}, we denote $\buh^\star_{\ula,L}\equiv \buh[\bB^\star_{\ula,L}]$ and $\buh^\star_{\ula}\equiv \buh[\bB^\star_{\ula}]$.
	 \item The optimal union-free tree  profile $(\bp_{\uttt,\ula,L}^\star)_{\uttt \in \FFF_2^\tr}$ is given by a similar formula as \eqref{eq:optimal:tree:1stmo}: recalling the normalizing constant $\bar{\mathfrak{Z}}_{\dot{q}}$ for $\dot{H}_{\dot{q}}$ in \eqref{eq:H:q:1stmo}, let $\bar{\mathfrak{Z}}^\star_2 \equiv \bar{\mathfrak{Z}}_{\dot{q}^\star_{\lambda_1,L}}\cdot \bar{\mathfrak{Z}}_{\dot{q}^\star_{\lambda_2,L}} $. Similarly, let $\dot{\ZZZ}^\star_2\equiv \dot{\ZZZ}_{\hat{q}^\star_{\la_1,L}}\cdot \dot{\ZZZ}_{\hat{q}^\star_{\la_2,L}}$ and $\hat{\ZZZ}^\star_2\equiv \hat{\ZZZ}_{\dot{q}^\star_{\la_1,L}}\cdot \hat{\ZZZ}_{\dot{q}^\star_{\la_2,L}}$. Moreover, for $\bx \in \hat{\partial}_2$, define $g(\bx):= 2^{-\lambda_1 \mathds{1}\{\bx^1 = \fs \} - \lambda_2 \mathds{1}\{\bx^2 = \fs \} }.$ Then, for $\dbq^\star=\dbq^\star_{\ula,L}$ and $\hbq^\star=\hbq^\star_{\ula,L}$, we have
    \begin{equation}\label{eq:optimal:tree:2ndmo}
      \bp_{\uttt,\ula,L}^\star := \frac{J_{\uttt} \bw_{\uttt}^{\ula}}{\bar{\mathfrak{Z}}^\star_2 (\dot{\ZZZ}_2^\star)^{v(\uuu)}(\hat{\ZZZ}_2^\star)^{f(\uuu)}} \prod_{\bx \in \dot{\partial}_2} \dbq^\star(\bx)^{\eta_{\uuu}(\bx)}
     \prod_{\bx\in \hat{\partial}_2} (g(\bx) \hbq^{\star} (\bx))^{\eta_\uuu (\bx) }.
     \end{equation}
     The optimal union-free tree profile $(p_{\uuu, \ula}^\star)_{\uuu\in\FFF_2^{\tr}}$ for the untruncated model is defined by the same equation \eqref{eq:optimal:tree:2ndmo} with $\bar{\mathfrak{Z}}_2^\star, \dot{\ZZZ}_2,\hat{\ZZZ}_2,\dbq^\star$ and $\hbq^\star$ for the untruncated model.
	\end{itemize}
	\end{defn}
	In Appendix \ref{subsec:compat:pair}, we gather the compatibility results regarding the optimal union-free tree profile.
	
	The next proposition shows that the most of the contribution to the second moment comes from the boundary profiles and weights close to their optimal values. The proof is presented in Section \ref{subsec:2ndmo:resampling}.
	\begin{prop}\label{prop:maxim:2ndmo}
	For $\ula=(\la_1,\la_2) \in [0,1]^2$, denote $\bs_{\ula,L}^\star:= (s^\star_{\la_1,L},s^\star_{\la_2,L})$ and $\bs_{\ula}^\star:= (s^\star_{\la_1},s^\star_{\la_2})$. For large enough $L\geq L_0(\ula,d,k)$ and $\delta>0$, there exists $c(\delta)=c(\delta,\ula,L,d,k)>0$ such that
	\begin{equation*}
	    \E \bZ^{2,(L),\tr}_{\ula,\ind}\left[||(\bB,\bs)-(\bB^\star_{\ula,L},\bs^\star_{\ula,L})||_1 >\delta\textnormal{ and }\pee_{\frac{1}{4}}\right]\leq e^{-c(\delta)n} \E \bZ^{2,(L),\tr}_{\ula,\ind}.
	\end{equation*}
	The same holds for the untruncated model: for any $\delta>0$, there exists $c(\delta)=c(\delta,\ula,d,k)>0$ such that
	\begin{equation*}
	   \E \bZ^{2,\tr}_{\ula,\ind}\left[||(\bB,\bs)-(\bB^\star_{\ula},\bs^\star_{\ula})||_1 >\delta\textnormal{ and }\pee_{\frac{1}{4}}\right]\leq e^{-c(\delta)n} \E \bZ^{2,\tr}_{\ula,\ind}.
	\end{equation*}
	\end{prop}
	Having Proposition \ref{prop:maxim:2ndmo} in hand, we can restrict our attention to the boundary profiles and weights close to the optimal. Furthermore, at the optimal profiles, it is straightforward to see the existence of optimal rescaling factor for the pair-model $\butheta^\star_{\ula}\equiv \butheta^\star\equiv \big(\btheta^\star_{\circ},\{\btheta^\star_{\bx}\}_{\bx\in \dot{\partial}_2\sqcup\hat{\partial}_2}, \btheta^\star_{s_1},\btheta^\star_{s_2}\big)$, which is analogous to \eqref{eq:def:opt:theta}: $\dot{\ZZZ}_2^\star,\hat{\ZZZ}_2^\star,\bar{\mathfrak{Z}}_2^\star,\dbq^\star$ and $\hbq^\star$ below are for $\ula$-tilted untruncated model.
	\begin{equation}\label{eq:def:opt:theta:2ndmo}
	\begin{split}
	    \btheta^\star_{\circ}&:= \log\left(\frac{(\dot{\ZZZ}^\star_2)^{\frac{k}{kd-k-d}}(\hat{\ZZZ}^\star_2)^{\frac{d}{kd-k-d}}}{\bar{\mathfrak{Z}}^\star}\right);\quad\quad\quad\quad\quad \btheta^\star_{\bx} := \log\left(\frac{\dbq^\star(\bx)}{(\dot{\ZZZ}_2^\star)^{\frac{1}{kd-k-d}}(\hat{\ZZZ}_2^\star)^{\frac{d-1}{kd-k-d}}}\right),\quad \bx\in \dot{\partial}_2;\\
	    \btheta^\star_{\bx} &:= \log\left(\frac{g(\bx)\hbq^\star(\bx)}{(\dot{\ZZZ}_2^\star)^{\frac{k-1}{kd-k-d}}(\hat{\ZZZ}_2^\star)^{\frac{1}{kd-k-d}}}\right),\quad \bx\in \hat{\partial}_2;\quad \btheta^\star_{s_1}\equiv\btheta^\star_{s_2} :=0.
	\end{split}
	\end{equation}
	$\butheta^\star_{\ula,L}\in \R^{|\dot{\partial}_2|+|\hat{\partial}_2|+3}$ is defined by the same equation as above with $\dot{\ZZZ}_2^\star,\hat{\ZZZ}_2^\star,\bar{\mathfrak{Z}}_2^\star,\dbq^\star$ and $\hbq^\star$ for $\ula$-tilted $L-$truncated model. Then, $J_{\uuu}w_{\uuu}^{\ula}\exp\big(\langle \butheta^\star_{\ula,L}, \boeta_{\uuu} \rangle\big)=\bp_{\uuu,\ula,L}^\star$ and $J_{\uuu}w_{\uuu}^{\ula}\exp\big(\langle \butheta^\star_{\ula}, \boeta_{\uuu} \rangle\big)=\bp_{\uuu,\ula}^\star$ hold, where $$\boeta_{\uuu}\equiv \big(\eta_{\uuu}(\circ),\{\eta_{\uuu}(\bx)\}_{\bx\in \dot{\partial}_2\sqcup\hat{\partial}_2}, \eta_{\uuu}(s_1),\eta_{\uuu}(s_2)\big):= \big(1,\{\eta_{\uuu}(\bx)\}_{\bx\in \dot{\partial}_2\sqcup\hat{\partial}_2},s_{\uuu}^{\lit,1},s_{\uuu}^{\lit,2}\big).$$  By perturbative analysis as done in Lemma \ref{lem:exist:theta}, we can also guarantee the existence of appropriate rescaling factor for $\bB, \bs$ close enough to the optimal, and having Proposition \ref{prop:2ndmo:aprioriestimate} and Proposition \ref{prop:2ndmo:B nt decomp} in hand, the same arguments as in Lemma \ref{lem:opt:tree:decay}, \ref{lem:conv:psi:theta} and \ref{lem:exist:free:energy:1stmo} naturally generalize to the pair model. We summarize the results for the pair model in the next proposition, which we present without proof since they follow from the same arguments as in the single copy case.
	\begin{prop}\label{prop:exist:free:energy:2ndmo}
	For $\delta>0$, denote the neighborhood of $(\bB^\star_{\ula,L},\bs^\star_{\ula,L})$ and $\bB^\star_{\ula,L}$ by
	\begin{equation}\label{eq:def:BB:delta:2ndmo}
	\begin{split}
	     &\prescript{}{2}{\BB}_{\ula,L}(\delta) \equiv \big\{(\bB,\bs)\in \pbDelta\times \R_{\geq 0}^{2}:||(\bB,\bs)-(\bB^\star_{\ula,L},\bs^\star_{\ula,L})||_1 \leq \delta\big\};\\
	     &\prescript{}{2}{\BB}^{-}_{\ula,L}(\delta) \equiv \big\{\bB\in \pbDelta:||\bB-\bB^\star_{\ula,L}||_1 \leq \delta\big\}.
	\end{split}
	\end{equation}
	$\prescript{}{2}{\BB}_{\ula}$ and $\prescript{}{2}{\BB}^{-}_{\ula}$ for the untruncated model is analogously defined. Then, there exist $\delta_0 =\delta_0(\ula,d,k)>0$ such that the following holds.
	\begin{enumerate}
	    \item For $L$ sufficiently large, the free energy of $\bB\in\prescript{}{2}{\BB}^{-}_{\ula,L}(\delta_0)$ (resp. $(\bB,\bs)\in \prescript{}{2}{\BB}_{\ula,L}(\delta_0)$), denoted by $\bF_{\ula,L}(\bB)$(resp. $\bF_{\ula,L}(\bB, \bs)$), are well-defined quantities satisfying
	\begin{equation}\label{eq:2ndmo:freeenergy:truncated}
	\begin{split}
	&\E \bZ^{2,(L),\tr}_{\ula,\ind}[\proj(\bB)] = \exp\left(n\bF_{\ula,L}(\bB)+O_{k}(\log n)\right);\\
	&\E \bZ^{2,(L),\tr}_{\ula,\bs,\ind}[\proj(\bB)] = \exp\left(n\bF_{\ula,L}(\bB,\bs)+O_{k}(\log n)\right).
	\end{split}
	\end{equation}
	\item For the untruncated model, the free energy of $\bB\in\prescript{}{2}{\BB}^{-}_{\ula}(\delta_0)$ (resp. $(\bB,\bs)\in \prescript{}{2}{\BB}_{\ula}(\delta_0)$), denoted by $\bF_{\ula}(\bB)$ (resp. $\bF_{\ula}(\bB, \bs)$) are also well-defined and satisfy the analog of \eqref{eq:2ndmo:freeenergy:truncated}, where we drop the subscript $L$ in the equation.
	\item The free energies defined above are twice differentiable in the interior of their domains, i.e. their Hessians are well-defined.
	\item $\nabla^2 \bF_{\ula,L}(\bB^\star_{\ula,L},\bs^\star_{\ula,L})$ (resp. $\nabla^2 \bF_{\ula,L}(\bB^\star_{\ula,L})$) converge in operator norm to $\nabla^2 \bF_{\ula}(\bB^\star_{\ula},\bs^\star_{\ula})$ (resp. $\nabla^2 \bF_{\ula}(\bB^\star_{\ula})$)
	\end{enumerate}
	\end{prop}
    Moreover, because $\bB^\star_{\ula,L}$ and $(\bp^\star_{\uuu,\ula,L})_{\uuu\in\FFF_2^\tr}$ (resp. $\bB^\star_{\ula}$ and $(\bp^\star_{\uuu,\ula})_{\uuu\in\FFF_2^\tr})$ are defined in terms of the product measure $\dbq^\star_{\ula,L}=\dot{q}^\star_{\la_1,L}\otimes \dot{q}^\star_{\la_2,L}$(resp. $\dbq^\star_{\ula}=\dot{q}^\star_{\la_1}\otimes \dot{q}^\star_{\la_2}$), the following relations between the free energies in the single and the pair copy hold for $\ula=(\la_1,\la_2)\in [0,1]^2$:
	\begin{equation}\label{eq:freeenergy:product}
	\begin{split}
	    &\bF_{\ula,L}(\bB^\star_{\ula,L},s^\star_{\ula,L})=\bF_{\ula,L}(\bB^\star_{\ula,L})=F_{\la_1,L}(B^\star_{\la_1,L},s^\star_{\la_1,L})+F_{\la_2,L}(B^\star_{\la_2,L},s^\star_{\la_2,L});\\
	    &\bF_{\ula}(\bB^\star_{\ula},s^\star_{\ula})=\bF_{\ula}(\bB^\star_{\ula})=F_{\la_1}(B^\star_{\la_1},s^\star_{\la_1})+F_{\la_2}(B^\star_{\la_2},s^\star_{\la_2}).
	\end{split}
	\end{equation}
	The proof of \eqref{eq:freeenergy:product} is deferred to Appendix \ref{subsec:compat:pair} (see Lemma \ref{lem:freeenergy:product}). The next proposition shows the negative definiteness of the Hessian of the free energy for the pair model and its proof is given in Section \ref{subsec:2ndmo:resampling}.
	\begin{prop}\label{prop:negdef:2ndmo}
		For $\ula \in [0,1]^2$, the following holds.
		\begin{enumerate}
		\item The unique maximizer of $\bF_{\ula}(\bB,\bs),(\bB,\bs)\in\prescript{}{2}{\BB}_{\ula}(\delta_0)$ is given by $(\bB^\star_{\ula}, \bs^\star_{\ula})$. Similarly, the unique maximizer of $\bF_{\ula}(\bB),\bB\in\prescript{}{2}{\BB}^{-}_{\ula}(\delta_0)$ is given by $\bB^\star_{\ula}$. The analog for the truncated model also holds.
		\item  There exists a constant $\beta=\beta(k)>0$, which does not depend on $L$, such that for sufficiently large $L$,
		\begin{equation*}
		    \nabla^2_{\bB} \bF_{\ula,L}(\bB^\star_{\ula,L},\bs^\star_{\ula,L}), \nabla^2 \bF_{\ula,L}(\bB^\star_{\la,L})\prec -\beta I, 
		\end{equation*}
		where $\nabla^2_{\bB}$ denotes the Hessian with respect to $\bB$. Hence, $\nabla^2_{\bB} \bF_{\ula}(\bB^\star_{\ula},\bs^\star_{\ula}), \nabla^2 \bF_{\ula}(\bB^\star_{\la})\prec 0$ holds by Proposition \ref{prop:exist:free:energy:2ndmo}.
		\end{enumerate}
	\end{prop}
	\begin{remark}\label{rem:freeenergy:2ndmo:ssz}
	analog of Remark \ref{rem:truncated:negdef:SSZ} for the pair model also holds for the pair model. \cite{ssz22} analyzed the free energy of the truncated model in the pair model when $\la_1=\la_2$, but their argument works goes through the case where $\la_1\neq \la_2$. That is, we can conclude from \cite{ssz22} that for $\ula\in [0,1]^2$,
	\begin{equation}\label{eq:rem:freeenergy:2ndmo:ssz}
	    (\bB^\star_{\ula,L},\bs^\star_{\ula,L})=\textnormal{argmax}\big\{\bF_{\ula,L}(\bB,\bs): \bB\in \pbDelta, \bs\in [0,\log 2]^2\big\}\textnormal{ and } \nabla^2 \bF_{\ula,L}(\bB^\star_{\ula,L},\bs^\star_{\ula,L})\prec 0,
	\end{equation}
	where $\bF_{\ula,L}(\bB,\bs)\equiv \lim_{n\to\infty}\frac{1}{n}\log \E \bZ^{2,(L),\tr}_{\ula,\bs,\ind}[\bB]$ for  $\bB\in \pbDelta, \bs\in [0,\log 2]^{2}$ is well-defined.
	\end{remark}
	Having Proposition \ref{prop:maxim:2ndmo} and \ref{prop:negdef:2ndmo} in hand, the same computations done in the proof of Theorem~\ref{thm:1stmo:constant} and \ref{thm:1stmo:constant:s} extend to the pair model to show Proposition \ref{prop:2nd mo constant for tr} and \ref{prop:2nd mo constant for tr:s} below. Hereafter, for $\la \in [0,1]$ and $s\in [0,\log 2]$, we denote $\bZ^{2,\tr}_{\la}\equiv \bZ^{2,\tr}_{(\la,\la)}$ and $\bZ^{2,\tr}_{\la,s}\equiv \bZ^{2,\tr}_{(\la,\la),(s,s)}$ for simplicity. In general, we simply use the subscript $\la$(resp. $s$) instead of $(\la,\la)$(resp. $(s,s)$) for all the quantities defined in the pair model.
	
	\begin{prop}\label{prop:2nd mo constant for tr}
	For $\lambda \in [0,\lambda^\star]$, the constant
		\begin{equation*}
		C_2(\lambda) := 
		\lim_{n\to\infty} 
		\frac{ \E \bZ^{2,\tr}_{\lambda,\ind}}{\exp\big(2n\bF_{\la}(\bB^\star_{\la})\big)} 
		\end{equation*}
		is well-defined and continuous on $[0,\lambda^\star]$. For the truncated model with $L$ sufficiently large,
		\begin{equation*}
		C_{2,L}(\lambda) := 
		\lim_{n\to\infty} 
		\frac{ \E \bZ^{2,(L),\tr}_{\lambda,\ind}}{\exp\big(2n\bF_{\la,L}(\bB^\star_{\la,L})\big)} 
		\end{equation*}
		is well-defined and continuous on $[0,\lambda_{L}^\star]$. Furthermore, we have for each $\lambda\in[0,\lambda^\star]$ that
		\begin{equation*}
		\lim_{L\to\infty} C_{2,L}(\lambda) = C_2(\lambda).
		\end{equation*}
	\end{prop}
	
	\begin{prop}\label{prop:2nd mo constant for tr:s}
		Let $(s_n)$ be a converging sequence whose limit is $s^\star$, satisfying $|s_n-s^\star|\leq n^{-2/3}$. Then the constant
		\begin{equation}\label{eq:def:C2:s:untrun}
	 C_2(\lambda^\star,s^\star) \equiv 
		\lim_{n\to\infty} 
		\frac{n \E \bZ^{2,\tr}_{\lambda^\star, s_n,\ind}}{\exp\big(2n\bF_{\la^\star}(\bB^\star_{\lambda^\star})\big)} 
		\end{equation}
		is well-defined regardless of the specific choice of $(s_n)$.
		For the truncated model with $L$ sufficiently large,
			\begin{equation}\label{eq:def:C2:s:trun}
	 C_{2,L}(\lambda^\star,s^\star) \equiv 
		\lim_{n\to\infty} 
		\frac{n \E \bZ^{2,(L),\tr}_{\lambda^\star, s_n,\ind}}{\exp\big(2n\bF_{\la,L}(\bB^\star_{\lambda^\star,L})\big)} 
		\end{equation}
		is well-defined. Furthermore, we have
		\begin{equation*}
		    \lim_{L \to\infty} C_{2,L}(\la^\star,s^\star)=C_{2}(\la^\star,s^\star).
		\end{equation*}
	\end{prop}
	The lemma below establishes an algebraic relationship between the leading constants of the first and the second moment. Although it is not necessary for the proof of Theorem \ref{thm1} and \ref{thm2}, it will play a crucial role in the companion paper \cite{nss2}.
	\begin{lemma}\label{lem:relation:leading:constants}
	Recall the constants $C_1(\la^\star),C_1(\la^\star,s^\star), C_2(\la^\star)$ and $C_2(\la^\star,s^\star)$, defined in Theorems \ref{thm:1stmo:constant}, \ref{thm:1stmo:constant:s}, Propositions \ref{prop:2nd mo constant for tr} and \ref{prop:2nd mo constant for tr:s} respectively. Then, we have
	\begin{equation}\label{eq:lem:relation:leading:constants}
	    \left(\frac{C_1(\la^\star,s^\star)}{C_1(\la^\star)}\right)^{2}=\frac{C_2(\la^\star,s^\star)}{C_2(\la^\star)}
	\end{equation}
	\end{lemma}
	\begin{proof}
	First recall from Remark \ref{rem:truncated:negdef:SSZ} and \ref{rem:freeenergy:2ndmo:ssz} that $F_{\la,L}(B,s)$ for $B\in \bDelta^{\textnormal{b}}, s\in [0,\log 2]$ and $\bF_{\ula,L}(\bB,\bs)$ for $\bB \in \pbDelta, \bs\in [0,\log 2]^{2}$ are well-defined. For $\la\in [0,1], s\in [0,\log 2]$ and $\ula\in [0,1]^2, \bs\in [0,\log 2]^2$, define
	\begin{equation*}
	    F_{\la, L}^{\textnormal{max}}(s)\equiv \max_{B\in \bDelta^{\textnormal{b}}} F_{\la,L}(B,s), \quad \bF_{\ula, L}^{\textnormal{max}}(\bs)\equiv \max_{\bB\in \pbDelta} \bF_{\ula,L}\big(\bB,\bs\big).
	\end{equation*}
	Then, the same computations done in the proof of Theorem \ref{thm:1stmo:constant:s} show the following generalization: there exists some $\delta_0=\delta_0(d,k)>0$ and continuous functions $C_{i,L}(\la^\star,\cdot):(s^\star-\delta_0,s^\star+\delta_0)\to \R, i=1,2$ such that for $L$ sufficiently large enough, $C_{i,L}(\la^\star,s_{L})$ converges to $C_{i}(\la^\star,s^\star)$ as $L\to\infty$ if $(s_L)_{L\geq 1}$ converges to $s^\star$, and  satisfy
	\begin{equation}\label{eq:leadingconstant:s:1stmo:2ndmo}
	    \lim_{n\to \infty}\sup_{|s-s^\star|<\delta_0}\Big|\frac{\sqrt{n}\E\bZ_{\la^\star,s}^{(L),\tr}}{\exp\big(nF_{\la^\star,L}^{\max}(s)\big)}-C_{1,L}(\la^\star,s)\Big|=0; \quad \lim_{n\to \infty}\sup_{|s-s^\star|<\delta_0}\Big|\frac{n\E\bZ_{\la^\star,s, \ind}^{2,(L),\tr}}{\exp\big(n\bF_{\la^\star,L}^{\textnormal{max}}(s)\big)}-C_{2,L}(\la^\star,s)\Big|=0.
	\end{equation}
	To this end, we aim to show \eqref{eq:lem:relation:leading:constants} for the truncated model, namely $\left(\frac{C_{1,L}(\la^\star,s^\star_{\la^\star,L})}{C_{1,L}(\la^\star)}\right)^{2}=\frac{C_{2,L}(\la^\star,s^\star_{\la^\star,L})}{C_{2,L}(\la^\star)}$, since taking $L\to\infty$ in the equation above shows \eqref{eq:lem:relation:leading:constants}.
	
	To this end, we first compute $\frac{C_{1,L}(\la^\star,s^\star_{\la^\star,L})}{C_{1,L}(\la^\star)}$. For sufficiently large $L$ so that $s^\star_{\la^\star,L} \in (s^\star-\delta_0,s^\star+\delta_0)$, we can use Proposition \ref{prop:maxim:1stmo} and \eqref{eq:leadingconstant:s:1stmo:2ndmo} to compute
	\begin{equation*}
	    \E \bZ_{\la^\star}^{(L),\tr}=\big(1+o_n(1)\big)\sum_{s\in (s^\star-\delta_0,s^\star+\delta_0) \cap \frac{1}{n}\Z}\E \bZ_{\la^\star,s}^{(L),\tr}=\big(1+o_n(1)\big)\sum_{s\in (s^\star-\delta_0,s^\star+\delta_0) \cap \frac{1}{n}\Z}\frac{C_{1,L}(\la^\star, s)}{\sqrt{n}}\exp\big(nF_{\la^\star,L}^{\textnormal{max}}(s)\big),
	\end{equation*}
	where $o_n(1)$ denotes quantity that tends to $0$ as $n\to\infty$. Note that by Remark \ref{rem:truncated:negdef:SSZ}, $F_{\la^\star,L}(B,s)$ is uniquely maximized at $(B,s)=(B^\star_{\la^\star,L},s^\star_{\la^\star,L})$ and strictly concave around its maximizer, which shows that $F_{\la^\star,L}^{\textnormal{max}}(s)$ is uniquely maximized at $s=s^\star_{\la^\star,L}$ with $\frac{d^{2}}{ds^{2}}F_{\la^\star,L}^{\textnormal{max}}(s^\star_{\la^\star,L})<0$. Thus, using Taylor expansion of $F_{\la^\star,L}^{\textnormal{max}}(s)$ around $s^\star_{\la^\star,L}$ and Gaussian integration in the equation above show 
	\begin{equation*}
	    \E \bZ^{(L),\tr}_{\la^\star}=\big(1+o_n(1)\big)C_{1,L}(\la^\star,s^\star_{\la^\star,L})\Big(-2\pi\frac{d^{2}}{ds^{2}}F_{\la^\star,L}^{\textnormal{max}}(s^\star_{\la^\star,L})\Big)^{-1/2}\exp\big(nF^{\max}_{\la^\star,L}(s^\star_{\la^\star,L})\big).
	\end{equation*}
	Therefore, by definition of $C_{1,L}(\cdot)$ in \eqref{eq:thm:1stmo:constant:la}, we can compute
	\begin{equation}\label{eq:ratio:leadingconstant:1stmo}
	\frac{C_{1,L}(\la^\star,s^\star_{\la^\star,L})}{C_{1,L}(\la^\star)}=\lim_{n\to\infty}\frac{C_{1,L}(\la^\star,s^\star_{\la^\star,L})\exp\big(nF_{\la^\star,L}(B^\star_{\la^\star,L})\big)}{ \E \bZ^{(L),\tr}_{\la^\star}}=\Big(-2\pi\frac{d^{2}}{ds^{2}}F_{\la^\star,L}^{\textnormal{max}}(s^\star_{\la^\star,L})\Big)^{-1/2}.
	\end{equation}
	Proceeding in the same fashion for the second moment, we have
	\begin{equation}\label{eq:ratio:leadingconstant:2ndmo}
	    \frac{C_{2,L}(\la^\star,s^\star_{\la^\star,L})}{C_{2,L}(\la^\star)}=2\pi \Big(\det\big(-\nabla^2 \bF^{\max}_{\ula^\star,L}(s^\star_{\la^\star,L},s^\star_{\la^\star,L})\big)\Big)^{-1/2}
	\end{equation}
	
	To this end, we now aim to show $\det\big(-\nabla^2 \bF^{\max}_{\ula^\star,L}(s^\star_{\la^\star,L},s^\star_{\la^\star,L})\big)=\Big(\frac{d^{2}}{ds^{2}}F_{\la^\star,L}^{\textnormal{max}}(s^\star_{\la^\star,L})\Big)^{2}$, which together with \eqref{eq:ratio:leadingconstant:1stmo} and \eqref{eq:ratio:leadingconstant:2ndmo} finishes the proof. Note that by definition, $\bZ^{2,(L),\tr}_{\ula,\bs,\ind}\asymp e^{n\langle \ula, \bs \rangle }\bN^{2,(L),\tr}_{\bs,\ind}$ holds, where $\bN^{2,(L),\tr}_{\bs,\ind}$ denotes the contribution to $\bN^{2}_{\bs}\equiv \bN_{s_1}\bN_{s_2}$ from $L$-truncated pair-colorings whose union-free components are composed of trees and they are in the near-independence regime. Hence, there exists a well-defined quantity $\bF_{L}(\bB,\bs)$, which does not depend on $\ula$, such that $$\bF_{\ula,L}(\bB,\bs)=\bF_{L}(\bB,\bs)+\big\langle \ula, \bs \big\rangle. $$ The analogous equation for the free energy in the single copy $F_{\la,L}(B,s)$ also holds. Thus, if we let $\la(s)\equiv \la_{L}(s)$ to be the inverse map of $\la\to s^\star_{\la,L}$, we can express $\bF^{\max}_{\la^\star,L}(\bs)\equiv \bF^{\max}_{(\la^\star,\la^\star),L}(\bs)$ as 
	\begin{equation}\label{eq:F:max:s:sum}
	\begin{split}
	\bF^{\max}_{\la^\star,L}(s_1,s_2)&=\bF_{\la(s_1),\la(s_2),L}^{\max}(s_1,s_2)+\big(\la^\star-\la(s_1)\big)s_1+\big(\la^\star-\la(s_2)\big)s_2\\
	&=F_{\la(s_1),L}^{\max}(s_1)+F_{\la(s_2),L}^{\max}(s_2)+\big(\la^\star-\la(s_1)\big)s_1+\big(\la^\star-\la(s_2)\big)s_2\\
	&=F_{\la^\star,L}^{\max}(s_1)+F_{\la^\star,L}^{\max}(s_2), 
	\end{split}
	\end{equation}
	where the second equation is due to \eqref{eq:freeenergy:product} and \eqref{eq:rem:freeenergy:2ndmo:ssz}. The equation above certainly implies our goal $\det\big(-\nabla^2 \bF^{\max}_{\ula^\star,L}(s^\star_{\la^\star,L},s^\star_{\la^\star,L})\big)=\Big(\frac{d^{2}}{ds^{2}}F_{\la^\star,L}^{\textnormal{max}}(s^\star_{\la^\star,L})\Big)^{2}$, which concludes the proof.
	\end{proof}
	Having Proposition \ref{prop:2ndmo:aprioriestimate}, \ref{prop:maxim:2ndmo} and \ref{prop:negdef:2ndmo} in hand, the proof of Proposition \ref{prop:ratio:uni:1stmo} extends to the second moment to show the following propositions.
	\begin{prop}\label{prop:2ndmo:comparison tr vs unic}
	Let $\lambda\in[0,\lambda^\star].$ The constant 
		\begin{equation*}
		\beta_2(\lambda):=\lim_{n\to\infty} \frac{\E \bZ^{2}_{\lambda,\ind} }{\E \bZ^{2,\tr}_{\lambda,\ind}}
		\end{equation*}
		is well-defined and continuous on $[0,\lambda^\star]$. For the truncated model $L>L_0$,
		\begin{equation*}
 	    \beta_{2,L}(\lambda):=\lim_{n\to\infty} \frac{\E \bZ^{2,(L)}_{\lambda,\ind} }{\E \bZ^{2,(L)\tr}_{\lambda,\ind}}
		\end{equation*}
		is well-defined and continuous on $[0,\lambda_L^\star]$. Furthermore, we have for each $\lambda\in[0,\lambda^\star]$ that 
		\begin{equation*}
		\lim_{L\to\infty} \beta_{2,L}(\lambda) = \beta_2(\lambda).
		\end{equation*}
	\end{prop}
	\begin{prop}\label{prop:2ndmo:comparison tr vs unic:s}
			Let $(s_n)$ be a converging sequence whose limit $s^\star$,  satisfying $|s_n-s^\star|\leq n^{-2/3}$.  Then, the constant
		\begin{equation*}
		\beta_2(\lambda^\star,s^\star):=\lim_{n\to\infty} \frac{\E \bZ^{2}_{\lambda^\star,s_n,\ind} }{\E \bZ^{2,\tr}_{\lambda^\star, s_n,\ind}}
		\end{equation*}
		is well-defined regardless of the specific choice of $(s_n)$. Furthermore, for the constant $\beta_2(\lambda)$ defined in Proposition \ref{prop:2ndmo:comparison tr vs unic}, we have
		\begin{equation*}
		\beta_2(\lambda^\star,s^\star) = \beta_2(\lambda^\star).
		\end{equation*}
	\end{prop}

	\subsection{Correlated regime}
	\label{subsec:2ndmo:correlated}
	In this subsection, we study the contributions to the second moment of $\bN_s^{\tr}$ from the correlated regime, where $|\zeta(\bsig) - \frac{1}{2}| > k^2 2^{-k/2}$. 
	
	The total number of clusters in the correlated regime was studied \cite[Section 4]{dss16}. Although we have the additional restriction that the clusters should be of a certain size, the proof is similar to the arguments given in \cite{dss16}. To this end, we adopt similar notations as in \cite[Section 4]{dss16}: abbreviate $\srr:= \{0,1\}$ and partition $\srr\srr:=\{0,1\}^{2}$ into $\srr\srr^{=}:=\{00,11\}$ and $\srr\srr^{\neq}:=\{01,10\}$. Also, for $\bsig \in \Omega_2^{E}$, define
	\begin{equation*}
	    \alpha(\bsig):=\frac{\pi_{\srr\srr^{=}}}{\pi_{\srr\srr}},
	\end{equation*}
	where $\pi_{\srr\srr^{=}}$ (resp.$\pi_{\srr\srr}$) denote the fraction of $\srr\srr^{=}$ variables (resp. $\srr\srr$) variables in $\bsig$. Note that if $\big|\zeta(\bsig)-\frac{1}{2}\big|>k^{2}2^{-k/2}$, then $\big|  2\alpha(\bsig)-\frac{1}{2}\big|>2k^{2}2^{-k/2}-O(k2^{-k})>k2^{-k/2}$ holds. We divide the contributions to $\E\bN^{2}_{\bs}\equiv \E \bN^{\tr}_{s^1}\bN^{\tr}_{s^2}$ from the correlated regime into \textit{near-identical} and \textit{intermediate} regimes and write
	\begin{equation}\label{eq:def:corrN}
	\begin{split}
	\bN_{\bs, \textnormal{id}}^2& := \bN_{\bs}^2\Big[ \alpha(\bsig) \wedge ( 1-\alpha(\bsig)) \le 2^{-3k/4}\quad\textnormal{and}\quad(n_{\ttt}^{i})_{\ttt\in \FFF_{\tr}}\in \ee_{1/4}\textnormal{ for $i=1,2$} \Big];\\
	\bN_{\bs,\textnormal{int}}^2 &:= \bN_{\us}^2
	\Big[\big|2\alpha(\bsig) -1\big| \in \big[k2^{-k/2}, 1- 2^{-3k/4}\big]\quad\textnormal{and}\quad (n_{\ttt}^{i})_{\ttt\in \FFF_{\tr}}\in \ee_{1/4}\textnormal{ for $i=1,2$}\Big],
	\end{split}
	\end{equation}
	where $(n_{\ttt}^{i})_{\ttt\in \FFF_{\tr}}, i=1,2,$ is the free tree profile of $\bsig\in \Omega_2^{E}$ in $i$'th copy (cf. Remark \ref{rem:clustersize:2ndmo}).
	\begin{prop}\label{prop:2ndmo:correlated overlap}
		For any sequence $(\bs_n)_{n\geq 1} \in (0,\log 2)^{2}$, there exists a constant $\widetilde{C}=\widetilde{C}(\alpha,k)$ such that
			\begin{align*}
			&1.\;\; \E\bN_{\bs_n, \textnormal{int}}^2 \leq e^{-\Omega(nk^22^{-k})};& & & & & & & & & & & & & & & & & & & & &\\
			&2.\;\; \E\bN_{\bs_n, \textnormal{id}}^2  \leq \widetilde{C} \big( \E\bN_{s_n^1}  + \E\bN_{s_n^2} \big) + e^{-\Omega(n2^{-k/2})}.& & & & & & & & & & & & & & & & & & & & &
			\end{align*}
	Moreover, the analog of the second item also holds for the truncated model. Namely, if we define $\bN^{2,(L)}_{\bs,\id}$ to be the contribution to $\bZ^{(L),\tr}_{0,s^1}\bZ^{(L),\tr}_{0,s_2}$ from near-identical regime $\big|2\alpha(\bsig) -1\big| \in \big[k2^{-k/2}, 1- 2^{-3k/4}\big]$, then we have
	\begin{equation}\label{eq:truncated:identical:bound}
	    \E \bN^{2,(L)}_{\bs,\textnormal{id}}\lesssim_{k} \E \bZ_{0,s_n^1}+\E\bZ_{0,s_n^2}+ e^{-\Omega(n2^{-k/2})}
	\end{equation}
	\end{prop}
	We remark that although \eqref{eq:truncated:identical:bound} is not needed for the current paper, it will be used in the companion paper \cite{nss2}.
	
	The proof of Proposition \ref{prop:2ndmo:correlated overlap} is deferred to Appendix \ref{subsec:app:2ndmo:corr}. In addition to Proposition \ref{prop:2ndmo:correlated overlap}, we will need a stronger version of the second statement to establish Theorem \ref{thm2} in Section \ref{sec:overlap}. Adopting similar notations as in \cite[Lemma 4.9]{dss16}, decompose $\bN^2_{\bs}=\bN^2_{\bs}[\pi]$, where $\bN^2_{\bs}[\pi]$ denotes the contribution from empirical measure $\pi$ on $\{0,1,\ff\}^{2}$. For $i=1,2$ we write $\pi^{i}$ for the projection of $\pi$ onto the $i$'th coordinate. The following lemma is analog of \cite[Lemma 4.9]{dss16}:
	
% 	Let $\pi$ be a probability measure on $\{\srr,\ff \}^2$. For a pair-coloring $\bsig$ and its corresponding $\{\rr,\ff \}$-configuration $\underline{\omega}$, we define $\bsig \in \pi$ if for any $(\eta^1, \eta^2)\in \{\rr,\ff\}^2$ we have
% 	\begin{equation*}
% 	\frac{1}{n}\left|\left\{v: \eta_v^1 = \eta^1, \ \eta_v^2 = \eta^2  \right\} \right| := \pi(\eta^1, \eta^2).
% 	\end{equation*}
% 	Further, define $$\bN_{\us}^2[\pi] = \sum_{\bsig\in \pi}  
% 	\bN_{\us}^2[\bsig],
% 	$$
% 	and set $\Delta[\pi] := n\pi( \eta^1 \neq \eta^2)$. We also write $\pi^1, \pi^2$ to denote the marginal of $\pi$ at its first and second coordinate, respectively. Then, the following lemma is an analog of \cite[Lemma 4.9]{dss16}.

	\begin{lemma}\label{lem:identical regime decay estim}
Let $(\pi_n)$ be a sequence of probability measures on $\{0,1,\ff\}^{2}$	satisfying $\pi_n^1(\ff)\vee \pi_n^2(\ff) \le 7 \cdot 2^{-k}$ and  $\Delta\equiv \Delta[\pi_n]:=n\pi(\eta^{1}\neq \eta^{2}) \le n 2^{-k/2}$.	For any sequence $(\bs_n)_{n\geq 1} \in (0,\log 2)^{2}$, there exists a constant $\widetilde{C} = \widetilde{C}(\alpha, k)$ such that
		\begin{equation}\label{eq:2ndmo:identical:mainlem}
		\E \bN_{\bs_n,\id}^2[\pi_n] \le \widetilde{C}2^{-k\Delta/10} \Big(\E \bN_{s^1_n}[\pi^1_n] +\E  \bN_{s^2_n}[\pi^2_n] \Big) + e^{-\Omega(n2^{-k/2})}.
		\end{equation}
	\end{lemma}
	
% Proofs of Proposition \ref{prop:2ndmo:correlated overlap} and Lemma \ref{lem:identical regime decay estim} are largely based on the corresponding statements proven in \cite{dss16}, although they require more technical work since we restrict our attention to the solutions with a specific size $\us_n$. We defer the details of the proof to Appendix 	\ref{subsec:app:2ndmo:corr}.
The proof of Lemma \ref{lem:identical regime decay estim} is deferred to Appendix \ref{subsec:app:2ndmo:corr}. 
	\begin{proof}[Proof of Theorem \ref{thm1}-$(b),(c)$] Throughout, we fix $M\in \N$. From the proof of Theorem \ref{thm1}-(a), recall the notation $s_{\circ}(C):=s^\star-\frac{\log n}{2\la^\star n}+\frac{C}{n}$ for $C\in \Z$. For any fixed $C\in \Z$, Paley-Zygmund's inequality shows
	\begin{equation*}
	    \P\Big(\bN^{\tr}_{s_{\circ}(C)}\big[\ee_{\frac{1}{4}}\big]\geq \frac{1}{2}\E\bN^{\tr}_{s_{\circ}(C)}\big[\ee_{\frac{1}{4}}\big]\Big)\geq\frac{1}{4} \frac{\Big(\E \bN^{\tr}_{s_{\circ}(C)}\big[\ee_{\frac{1}{4}}\big]\Big)^{2}}{\E \Big(\bN^{\tr}_{s_{\circ}(C)}\big[\ee_{\frac{1}{4}}\big]\Big)^2}.
	\end{equation*}
	Note that Proposition \ref{prop:1stmo:aprioriestimate}, Theorem \ref{thm:1stmo:constant} and Theorem \ref{thm:1stmo:constant:s} imply $$\E \bN^{\tr}_{s_{\circ}(C)}\big[\ee_{\frac{1}{4}}\big]=\big(1-O_{k}(n^{-\frac{3}{2}}\log n)\big)\E \bN^{\tr}_{s_{\circ}(C)} \asymp_{k} e^{-\la^\star C}.$$ Thus, there exists small enough $C_0=C_0(M,\alpha,k) \in \Z$ such that $\frac{1}{2}\E \bN^{\tr}_{s_{\circ}(C_0)}\big[\ee_{\frac{1}{4}}\big]\geq M$ holds. Moreover, $\E \Big(\bN^{\tr}_{s_{\circ}(C)}\big[\ee_{\frac{1}{4}}\big]\Big)^2\lesssim_{k}e^{-2\la^\star C}+e^{-\la^\star C}$ holds by Proposition \ref{prop:2nd mo constant for tr:s}, Proposition \ref{prop:2ndmo:comparison tr vs unic:s}, and Proposition \ref{prop:2ndmo:correlated overlap}. Hence, we have
	\begin{equation}\label{eq:thm1:first:event}
	   \P\Big(\bN^{\tr}_{s_{\circ}(C_0)}\geq M\Big)\geq 5\delta
	\end{equation}
	for some $\delta=\delta(\alpha,k)$, where we used the fact $\frac{e^{-2\la^\star C}}{e^{-2\la^\star C}+e^{-\la^\star C}}$ converges to $1$ as $C\to -\infty$. Therefore, Theorem \ref{thm1}-$(c)$ holds on the event $\AAA_1:=\Big\{\bN^{\tr}_{s_n(C_0)}\geq M\Big\}$ with $C^\prime(M,\alpha,k):=-C_0(M,\alpha,k)$.
	
	We now turn to Theorem \ref{thm1}-$(b)$. We aim to find $\left(K(j)\right)_{j\geq 1} \equiv \left(K(j,\alpha,k)\right)_{j\geq 1}$ such that for $n\geq n_0(\alpha,k)$ and $j\leq \frac{\log n}{4}$, $K(j)$ largest clusters occupy $1-e^{-j}$ fraction of the solution space. If such $\left(K(j)\right)_{j\geq 1}$ exists, then we can crudely set $\tilde{K}(j)\equiv K(j)+ 2^{e^{4j}}+2^{n_0}$ and observe that $\tilde{K}(j)$ largest clusters occupy $1-e^{-j}$ fraction of the solution space for $j\geq 1$ and $n\geq 1$.
	
	First, note that by $(a)$ of Theorem \ref{thm1}, there exists $C_{\textsf{up}}=C_{\up}(\alpha,k)$ such that
	\begin{equation*}
	  \P(\AAA_2)\leq \delta,\quad\textnormal{where}\quad\AAA_2:=\Big\{Z \geq n^{-\frac{1}{2\la^\star}}e^{ns^\star+C_{\up}}\Big\}.
	\end{equation*}
	Next, recall that in the proof of Theorem \ref{thm1}-$(a)$ (cf. \eqref{eq:thm1:proof:useful}), we showed that there exists a constant $C_k$ depending only on $k$ such that for every $C\in \Z$, $\sum_{C^\prime \leq C} \E\bZ_{1, s_n(C^\prime)}\leq C_k n^{-\frac{1}{2\la^\star}}e^{n s^\star+(1-\la^\star)C}$ holds. Thus, Markov's inequality shows that we have
	\begin{equation}\label{eq:proof:thm1:technical}
	    \P\Big(\sum_{C^\prime \leq C} \bZ_{1, s_n(C^\prime)}\geq \delta^{-1}C_kn^{-\frac{1}{2\la^\star}}e^{ns^\star+\frac{(1-\la^\star)}{2}C}\Big)\leq \delta e^{\frac{(1-\la^\star)}{2}C}.
	\end{equation}
	For $j\in \N$, let $C(j)\equiv C(j,\alpha,k)<0$ be a small enough integer so that 
	\begin{equation*}
	    \frac{1-\la^\star}{2}C(j)\leq \min\bigg\{-j\log 2\;,\; -j+C_0+\log\Big(\frac{\delta}{2C_k}\Big)\bigg\}.
	\end{equation*}
	$C(j)$ was chosen so that plugging in $C=C(j)$ into \eqref{eq:proof:thm1:technical} shows
	\begin{equation*}
	     \P\Big(\AAA_{3,j}\Big)\leq 2^{-j}\delta,\quad\textnormal{where}\quad\AAA_{3,j}:=\Big\{\sum_{C \leq C(j)} \bZ_{1, s_{\circ}(C)}\geq \big(2e^{j}\big)^{-1} n^{-\frac{1}{2\la^\star}}e^{ns^\star+C_0}\Big\}.
	\end{equation*}
	Hence, union bound shows $\P(\AAA_3)\leq \delta$, where $\AAA_3:=\cup_{j\geq 1} \AAA_{3,j}$. Note that $\E \bN_{s_{\circ}(C)}\lesssim_{k} \E \bN^{\tr}_{s_{\circ}(C)}\lesssim_{k} e^{-\la^\star C}$ holds by Theorem \ref{thm:1stmo:constant:s} and Proposition \ref{prop:ratio:uni:1stmo}. Thus, we can choose $K(j)\equiv K(j,\alpha,k)$ large enough so that
	\begin{equation*}
	    K(j) \geq \delta^{-1} 2^j\sum_{C=C(j)+1}^{C_{\up}} \E \bN_{s_{\circ}(C)}. 
	\end{equation*}
	$K(j)$ was chosen so that Markov's inequality shows
	\begin{equation*}
	      \P\Big(\AAA_{4,j}\Big)\leq 2^{-j}\delta,\quad\textnormal{where}\quad\AAA_{4,j}:=\Big\{\sum_{C=C(j)+1}^{C_{\up}}\bN_{s_{\circ}(C)}\geq K(j)\Big\}.
	\end{equation*}
	Thus, union bound shows $\P(\AAA_4)\leq \delta$, where $\AAA_4:=\cup_{j\geq 1}\AAA_{4,j}$. Next, note that by Proposition \ref{prop:1stmo:aprioriestimate},
	\begin{equation*}
	    \sum_{C\leq C_{\textsf{up}}}\E \bZ_{1,s_{\circ}(C)}\big[e_{\textnormal{mult}}\geq 1\big]\leq e^{n(1-\la^\star)s_{\circ}(C_{\textsf{up}})}\E \bZ_{\la^\star}\big[e_{\textnormal{mult}}\geq 1\big]\lesssim_{k}\frac{\log^3 n}{\sqrt{n}}n^{-\frac{1}{2\la^\star}}e^{ns^\star}.
	\end{equation*}
	Thus, by Markov's inequality, we have
	\begin{equation*}
	      \P(\AAA_5)\lesssim_{k}n^{-1/4}\log^3 n,\quad\textnormal{where}\quad\AAA_5:=\Big\{ \sum_{C\leq C_{\textsf{up}}}\bZ_{1,s_{\circ}(C)}\big[e_{\textnormal{mult}}\geq 1\big]\geq (2n^{1/4})^{-1} n^{-\frac{1}{2\la^\star}}e^{ns^\star+C_0}\Big\}.
	\end{equation*}
	Hence, $\P(\AAA_5)\leq \delta$ holds for $n\geq n_0(\alpha,k)$.

	Finally, let $\AAA:=\AAA_1 \cap \big(\cap_{2\leq i \leq 5}\AAA_i^{\textsf{c}}\big)$. Then, $\P(\AAA)\geq \delta$ holds for $n\geq n_0(\alpha,k)$. Moreover, on the event $\AAA$, $Z \geq n^{-\frac{1}{2\la^\star}}e^{ns^\star+C_0}$ holds by definition of $\AAA_1$. Thus, on the event $\AAA$, definition of $\AAA_3$ and $\AAA_5$ shows that for $j\leq \frac{\log n}{4}$, we have
	\begin{equation}\label{eq:proof:main:thm:occupy}
	    \sum_{C\leq C(j)}\bZ_{1, s_{\circ}(C)}+\sum_{C\leq C_{\textsf{up}}}\bZ_{1,s_{\circ}(C)}\big[e_{\textnormal{mult}}\geq 1\big]\leq \big((2e^{j})^{-1}+(2n^{1/4})^{-1}\big)Z\leq e^{-j}Z.
	\end{equation}
	Recall that there is at most $1$ cluster that is coarsened to the frozen configurations which do not have multi-cyclic free components. Moreover, on the event $\AAA$, $\sum_{C\geq C(j)+1} \bN_{s_{\circ}(C)}\leq K(j)$ holds by definition of $\AAA_2$ and $\AAA_4$. Therefore, \eqref{eq:proof:main:thm:occupy} shows that on the event $\AAA$, the $K(j)$ largest clusters occupy $1-e^{-j}$ fraction of the solution space for $j\leq \frac{\log n}{4}$. 
	\end{proof}

	\section{The resampling method}\label{subsec:resampling}
	The goal of this section is to prove Proposition \ref{prop:maxim:1stmo}, \ref{prop:negdef}, \ref{prop:maxim:2ndmo} and \ref{prop:negdef:2ndmo} by the resampling method. In Section \ref{subsec:localupdate}, we introduce the \textit{resampling Markov chain}, which reduces the non-convex optimization of the free energy to the convex \textit{tree optimization} by local updates. Although the definitions are nearly identical to \cite[Section 4]{ssz22}, there is a slight change due to existence of the large free trees (see Remark \ref{remark:resampling:diff:ssz} below). In Section \ref{subsec:resampling:treeopt}, we analyze the \textit{tree optimization} problem. In Section \ref{subsubsec:resampling:maximizer}, we prove Proposition \ref{prop:maxim:1stmo}. In Section \ref{subsubsec:resampling:negdef}, we prove Proposition \ref{prop:negdef}. In Section \ref{subsec:2ndmo:resampling}, we prove Proposition \ref{prop:maxim:2ndmo} and \ref{prop:negdef:2ndmo}.
	
	\subsection{The resampling Markov chain}\label{subsec:localupdate}
	Throughout this section, we fix $\la \in [0,1]$ and consider an edge in $\GGG$ to be of graph distance $1$, while an half-edge has distance $\frac{1}{2}$. Moreover, we consider the coloring configuration, i.e. we do not simplify the spin $\sigma=(\dot{\sigma},\fs)$ as $\fs$.

	First, we specify the law of the sampled variables $Y \subset V(\GGG)$.
	\begin{defn}[sampling mechanism]
	\label{def:sampling-mech}
		For $\eps>0$, define the $\eps$-sampling mechanism $\P_{\eps}(Y\mid \GGG)$ by the law of the set $Y= \{v\in V(\GGG): I_v=1\}$, where i.i.d random varaibles $\{I_v\}_{v\in V(\GGG)}$ has law $I_v \sim \textnormal{Ber}(\eps)$.
	\end{defn}
	Given $\GGG$, denote the $\frac{3}{2}$ neighborhood of $Y$ by $\NNN\equiv\NNN(Y)\equiv(\NN(Y), \equiv(\NN,\uL_{\NN})$. Here, $\uL_\NN$ includes the literals at $\delta \NN$, where $\delta\NN$ denotes the half-edges hanging at $\NN$. Observe that when the $\frac{3}{2}$ neighborhood of $v\in Y$ do not intersect, $\NNN$ is composed of $|Y|\equiv \kappa$ disjoint copies of $\frac{3}{2}$ depth tree $\DD$ illustrated below.
	
		\begin{figure}[h]
		\begin{tikzpicture}[square/.style={regular polygon,regular polygon sides=4},thick,scale=0.64, every node/.style={transform shape}]
		\node[circle,draw, scale=0.8, fill=black] (v1) at (0,0) {};
		\node[square,draw, scale=0.8] (a1) at (-2,-2) {};
		\node[square,draw, scale=0.8] (a2) at (0,-2) {};
		\node[square,draw, scale=0.8] (a3) at (2,-2) {};
		\draw (v1)--(a1);
		\draw (v1)--(a2);
		\draw (v1)--(a3);
		\draw[draw=orange] (a1)--(-3,-3);
		\draw[draw=orange] (a1)--(-1.5,-3);
		\draw[draw=orange] (a2)--(-0.5,-3);
		\draw[draw=orange] (a2)--(0.5,-3);
		\draw[draw=orange] (a3)--(1.5,-3);
		\draw[draw=orange] (a3)--(3,-3);
		\end{tikzpicture}
		\caption{$\frac{3}{2}$ depth tree $\DD$. The edges in the boundary $\delta \DD$ are highlighted orange.}\label{fig:depthonetree}
	\end{figure}
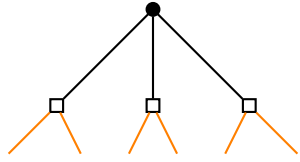
	
	Then, detaching $\NNN$ from $\GGG$ induces the cavity graph $\GGG_{\partial}\equiv (\GG_{\partial},\uL_{\GG_{\partial}})$, where $\GG_{\partial} \equiv \GG \backslash \NN$. We denote the half-edges hanging at $\GG_{\partial}$ by $\delta \GG_{\partial}$. Having sampled $Y$, we record the local statistics of spins in $\NNN(Y)$ as follows.
	\begin{defn}[sample empirical measures; \cite{ssz22}, Definition 4.1]
	\label{def:sample:empirical}
	Given an \textsc{nae-sat} instance $\GGG$ and a valid coloring $\sig\in \Omega^{E}$, let $Y\subset V(G_n), |Y|= \kappa$ be a nonempty subset of the vertices. We define $H^\sm[\GGG,Y,\sig]\equiv (\dot{H}^\sm,\hat{H}^\sm, \bar{H}^\sm)$ as follows.
	\begin{equation}\label{eq:def:H:sampled}
	\begin{split}
	    \dot{H}^\sm(\utau) &\equiv \frac{1}{\kappa}\sum_{v \in Y}\one\{\sig_{\delta v} = \utau\}\quad\textnormal{for}\quad\utau\in \Omega^{d}\\
	    \hat{H}^\sm(\utau) &\equiv \frac{1}{\kappa d}\sum_{v \in Y}\sum_{e \in \delta v}\one\{(\sig_{\delta a(e)})^{(j(e))} = \utau\}\quad\textnormal{for}\quad\utau\in \Omega^{k}\\
	    \bar{H}^\sm(\tau) &\equiv \frac{1}{\kappa d}\sum_{v \in Y}\sum_{e \in \delta v}\one\{\sigma_{e} = \tau\}\quad\textnormal{for}\quad\tau\in \Omega.
	\end{split}
	\end{equation}
	In the definition of $\hat{H}^\sm(\utau)$ above, $a(e)$ is the clause adjacent to $e$, $j(e)$ is the index of $e$ in $\delta a(e)$, and $\utau^{(j)}\equiv(\tau_j, ..., \tau_k, \tau _1,...,\tau_{j-1})$, where $\utau=(\tau_1,...,\tau_k) \in \Omega^{k}$. The use of the rotation of the indices is to distinguish the spin adjacent to $Y$ when counting $\hat{H}^\sm$. Then $H^\sm$ lies in the space $\bDelta^\sm$, defined analogously to $\bDelta$ in Definition \ref{def:empiricalssz}, except that the condition \eqref{eq:compatibility:coloringprofile} is now replaced by
	\begin{equation*}
	    \frac{1}{d} \sum_{\utau \in \Omega^{d}}\dot{H}^\sm(\utau)\sum_{i=1}^{d}\one\{\tau_i=\tau\} = \bar{H}^\sm(\tau) = \sum_{\utau \in \Omega^k}\hat{H}^\sm(\utau)\one\{\tau_1 = \tau \}, 
 	\end{equation*}
 	for every $\tau \in \Omega$. For $\kappa \in \Z_{+}$, we denote by $\bDelta^\sm_{\kappa}$ the set of $H^\sm \in \bDelta^\sm$ such that $\dot{H}^\sm,\hat{H}^\sm$ and $\bar{H}^\sm$ lies in the grid of $\frac{1}{\kappa},\frac{1}{\kappa d}$ and $\frac{1}{\kappa d}$ respectively. Moreover, denote the truncated versions of $\bDelta^\sm$ and $\bDelta^{\sm}_{\kappa}$ by $\bDelta^{\sm,(L)}$ and $\bDelta^{\sm,(L)}_{\kappa}$ respectively, where $\bDelta^{\sm,(L)}$ is the set of $H^\sm \in \bDelta^\sm$  satisfying $\textnormal{supp}{\dot{H}^\sm}\subset \Omega_L^{d}, \textnormal{supp}{\hat{H}^\sm}\subset \Omega_L^{k}, \textnormal{supp}{\bar{H}^\sm}\subset \Omega_L$. $\bDelta^{\sm,(L)}_{\kappa}$ is defined analogously.
 	
 	Furthermore, for $H\in \bDelta$, we denote $H^\sy\equiv (\dot{H},\hat{H}^\sy, \bar{H})$, where $\hat{H}^\sy$ is the average over all $k$ rotations of $\hat{H}$. Then, $H^\sy \in \bDelta^\sm$ for $H \in \bDelta$. Also, for any $H^\sm \in \bDelta^\sm$, define $\dot{h}= \dot{h}[H^\sm]\in \PPP(\dot{\Omega})$ as
 	\begin{equation}\label{eq:def:hdot}
 	    \dot{h}(\dot{\tau})\equiv \frac{1}{k-1} \sum_{\utau \in \Omega^{k}}\sum_{j=2}^{k}\one\{\dot{\tau}_{j}=\dot{\tau}\} \hat{H}^\sm(\utau).
 	\end{equation}
 	If $H^\sm =H^\sm[\GGG, Y, \sig]$, $\dot{h}=\dot{h}[H^\sm]$ is the induced empirical measure of \textit{clause-to-variable} colors on $\delta \NN$.
	\end{defn}
	Having sampled $Y$, we resample the spins and literals in $\NN(Y)$ conditioned on $\dot{h}=\dot{h}[H^\sm]$. Since the \textit{variable-to-clause} colors on $\delta \NN$ can change after resampling, we need to update the colors of the tree components intersecting $Y$, which is done by the update procedure defined below.
	
	Given $(\GGG, \sig)$ and an edge $e=(av) \in E(G_n)$, let $\dot{\ttt}(e)\equiv \dot{\ttt}_{\sig}(e)$ be the \textit{variable-to-clause} directed free tree hanging at the \textit{root edge} $e$, i.e. it is the subtree of the free tree containing $e$ obtained by deleting all the variables, clauses and edges closer to $a$ than $v$. If $v$ is frozen, we define $\dot{\ttt}(e)$ to be the single edge $e$. Given a valid coloring $\utau$ on $\dot{\ttt}(e)$ and an edge $e^\prime=(a^\prime v^\prime)$ in $\dot{\ttt}(e)$, the \textit{upward} color of $\utau$ at $e^\prime$ is defined to be $\dot{\tau}_{e^\prime}$ if $a^\prime$ is closer to $e$ than $v^\prime$ in $\dot{\ttt}(e)$ and $\hat{\tau}_{e^\prime}$ otherwise. The next lemma, which was shown for the truncated model in \cite[Lemma 4.3]{ssz22} holds also for the untruncated model without any modification of the proof.
	\begin{lemma}[\cite{ssz22}, Lemma 4.3]
	\label{lem:update}
	Given a \textsc{nae-sat} instance $\GGG$ and a valid coloring $\sig \in \Omega^{E}$, let $\dot{\ttt}(e)$ be the variable-to-clause directed tree with root edge $e$, defined above. If $\eta \in \Omega$ agrees with $\sigma$ on the upward edge $e$, i.e. $\dot{\eta}_e =\dot{\sigma}_e$, then there exists a unique valid coloring in $\dot{\ttt}(e)$, $\utau \in \Omega^{E\left(\dot{\ttt}(e)\right)}$, such that $\tau_e=\eta$ and $\utau$ agrees with $\sig\lvert_{\dot{\ttt}(e)}\equiv (\sigma_e)_{e\in E\left(\dot{\ttt}(e)\right)}$ in all the upward colors. Hence, we denote such $\utau$ by
	\begin{equation}
	\label{eqn:def:update}
	\utau = \textnormal{\textbf{update}}\left(\sig\lvert_{\dot{\ttt}(e)}, \eta;\dot{\ttt}(e)\right).
	\end{equation}
	Moreover, for a valid coloring $\utau^\prime$ in $\dot{\ttt}(e)$, define its weight by
	\begin{equation*}
	    w^\lit_{\dot{\ttt}(e)}(\utau^\prime) \equiv \prod_{v\in V\left(\dot{\ttt}(e)\right)}\left\{\dot{\Phi}(\utau^\prime_{\delta v})\prod_{e\in \delta v}\bar{\Phi}(\tau^\prime_e)\right\}\prod_{a\in F\left(\dot{\ttt}(e)\right)\backslash\{a(e)\}}\hat{\Phi}(\utau^\prime_{\partial a}).
	\end{equation*}
	Then, $w_{\dot{\ttt}}^\lit(\sig)=w_{\dot{\ttt}}^\lit(\utau)$ holds for $\utau = \textnormal{\textbf{update}}\left(\sig\lvert_{\dot{\ttt}(e)}, \eta;\dot{\ttt}(e)\right)$. Finally, if $\sig$ is a valid $L$-truncated coloring, i.e. $\sig \in \Omega_L^E$, and $\eta \in \Omega_L$, then $\textnormal{\textbf{update}}\left(\sig\lvert_{\dot{\ttt}(e)}, \eta;\dot{\ttt}(e)\right)\in \Omega_L^{E\left(\dot{\ttt}(e)\right)}$.
	\end{lemma}
	Now, we define the resampling Markov chain on tuples $(\GGG, Y, \sig)$, where $\GGG$ is a \textsc{nae-sat} instance $Y\subseteq V(\GGG)$ is a subset of variables, and $\sig$ is a valid coloring on $\GGG$. Given a coloring $\sig_{\NN}$ on $\NN$, define its weight by
	\begin{equation}\label{eq:def:wN:dep1nb}
	    w_{\NN}^\lit(\sig_{\NN};\uL_{\NN}) \equiv w^\lit_{\NNN}(\sig_{\NN})\equiv \prod_{v\in Y}\left\{\dot{\Phi}(\sig_{\partial v})\prod_{e\in \delta v}\left\{\hat{\Phi}^\lit\left((\sig\oplus \uL)_{\delta a(e)}\right)\bar{\Phi}(\sigma_e)\right\}\right\}
	\end{equation}
	In the equation above, we emphasize that we do not take product with respect to $\bar{\Phi}$ over the spins at $\delta \NN$. Given a coloring $\sig_{\partial}$ on $\GGG_{\partial}=(\GG_{\partial},\uL_{\GG_{\partial}}) \equiv (V_\partial, F_\partial,E_\partial, \uL_{\GG_{\partial}})$, let
	\begin{equation}\label{eq:def:wpartial:dep1nb}
	    w^\lit_{\partial}(\sig_{\partial})\equiv \prod_{v \in V_{\partial}}\dot{\Phi}(\sig_{\delta v})\prod_{a\in F_{\partial}}\hat{\Phi}^\lit\left((\sig \oplus \uL)_{\delta a}\right)\prod_{e\in E_{\partial}}\bar{\Phi}(\sigma_e).
	\end{equation}
	Note that $\delta \GG_{\partial}$ is included in $E_{\partial}$, so for a valid coloring $\sig$ in $\GGG$, we have
	\begin{equation*}
	    w^\lit_{\GGG}(\sig) = w^\lit_{\partial}(\sig_{\partial})w^\lit_{\NN}(\sig_{\NN};\uL_\NN)
	\end{equation*}
	With a slight abuse of notation, denote by $\dot{h}(\sig_{\delta \NN})$ the empirical measure of the colors $(\dot{\sigma}_{e})_{e \in \delta \NN}$. Then, the resampling Markov chain is defined as follows.
	
	\begin{defn}[The resampling Markov chain]
	\label{def:resampling:markovchain}
	With initial state $A_0=(\GGG, Y, \sig)$, where $\GGG$ is a \textsc{nae-sat} instance, $Y\subset V(\GGG)$ and $\sig$ is a valid coloring, we take a step in the untruncated resampling Markov chain to arrive at $A_1 = (\GGG^\prime, Y, \utau)$ by the following.
	\begin{enumerate}
	    \item[1.] If $\{\dot{\ttt}_{\sig}(e)\}_{e \in \delta \NN(Y)}$ are not disjoint, then $A_1=A_0$ with probability $1$.
	    \item[2.] Otherwise, detach $\NN = \NN(Y)$ from $G_n$ and sample new literals and spins $(\uL_{\NN}^\prime, \utau_{\NN})$ on $\NN$ from the probability measure
	    \begin{equation}\label{eq:def:resampling:markovchain:prob}
	        p\Big((\uL_{\NN}^\prime,\utau_{\NN})\Big|
	        (\uL_{\NN}, \sig_{\NN})\Big) = \frac{w_{\NN}^\lit(\utau_{\NN};\uL^\prime_{\NN})^{\lambda}\one\left\{\dot{h}(\utau_{\delta \NN}) = \dot{h}(\sig_{\delta \NN})\right\}}{Z},
	    \end{equation}
	    where $Z=Z\left(|Y|, \dot{h}(\sig_{\delta\NN})\right)$ is the normalizing constant to make $p$ a probability measure.
	    \item[3.] Form a new graph $\GGG^\prime$ by sampling a uniformly random matching between $\delta \NN$ and $\delta G_{\partial}$, conditioned on the constraint that $e \in \delta G_{\partial}$ is matched to $e^\prime \in \delta \NN$ with $\dot{\sigma}_{e}=\dot{\tau}_{e^\prime}$. Note that the number of such matchings only depends on $|Y|$ and $\dot{h}(\sig_{\delta\NN})$, which we denote by $M\left(|Y|, \dot{h}(\sig_{\delta\NN})\right)$. Also, the literals of $\GGG^\prime$ at $\delta \GG_{\partial}$ are determined by $\delta \NN$. Then, update the downward colors of the directed trees $\{\dot{\ttt}(e)\}_{e\in \delta \GG_{\partial}}$ by 
	    \begin{equation*}
	        \utau\lvert_{\dot{\ttt}(e)} \equiv \textnormal{\textbf{update}}\left(\sig\lvert_{\dot{\ttt}(e)}, \tau_e;\dot{\ttt}(e)\right),
	    \end{equation*}
	    for every $e \in \delta \GG_{\partial}$. On the rest of $\GG_\partial$, we take $\utau$ to be the same with $\sig$. 
	\end{enumerate}
	We also define the $L$-truncated resampling Markov chain by restricting the state space to the $L$-truncated colorings, i.e. the space of $(\GGG,Y,\sig)$, where $\sig$ is a valid $L$-truncated coloring. Then, note that \eqref{eq:def:resampling:markovchain:prob} is changed to 
	\begin{equation*}
	      p_L\Big((\uL_{\NN}^\prime,\utau_{\NN})\Big|
	        (\uL_{\NN}, \sig_{\NN})\Big) = \frac{w_{\NN}^\lit(\utau_{\NN};\uL^\prime_{\NN})^{\lambda}\one\left\{\dot{h}(\utau_{\delta \NN}) = \dot{h}(\sig_{\delta \NN}), \tau_{\NN} \in \Omega_L^{E(\NN)}\right\}}{Z_L},
	\end{equation*}
	where $E_{\NN}$ is the (full) edge set of $\NN$. We denote by $\pi$ and $\pi_L$ the transition probability for the untruncated and truncated resampling Markov chain respectively.
	\end{defn}
	\begin{remark}
	\label{remark:resampling:diff:ssz}
	Definition \ref{def:resampling:markovchain} is the same as the resampling Markov chain defined in \cite[Section 4]{ssz22}, except for the first item; in \cite{ssz22}, they have sampled $Y$ given $\GGG, \sig$ for the truncated model so that $\{\dot{\ttt}_{\sig}(e)\}_{e \in \delta \NN(Y)}$ are disjoint with probability $1$, so there was no need to deal with the case when they are not disjoint. However, the same approach cannot be applied for the untruncated model, because of the appearance of large trees. Instead, we show in Lemma \ref{lem:1stmo:resampling:treedisjoint} below that with good enough probability, $\{\dot{\ttt}(e)\}_{e \in \delta \NN(Y)}$ are disjoint under the sampling mechanism given in Definition \ref{def:sampling-mech}.
	\end{remark}
	One of the key features of the resampling Markov chain is that it is reversible with respect to the measure that is proportional to the weight of the colorings. Indeed the lemma below confirms this fact, whose proof is identical to \cite[Lemma 4.5]{ssz22}.
	\begin{lemma}
	\label{lem:resampling:reversingmeasure}
	Recall the $\eps$-sampling mechanism $\P_{\eps}(Y\mid \GGG)$ in Definition \ref{def:sampling-mech}. For every $\eps>0$, a reversing measure for both the untruncated and truncated resampling Markov chain is given by
	\begin{equation*}
	    \mu_{\eps}(\GGG, Y, \sig) \equiv \P(\GGG) \P_{\eps}(Y \mid \GGG) w_{\GGG}^\lit(\sig)^{\lambda}.
	\end{equation*}
	\end{lemma}

	\subsection{The tree optimization}\label{subsec:resampling:treeopt}
	The tree optimization defined below is closely related to the transition probability of the resampling Markov chain, which is made precise in Lemma \ref{lem:resampling:transitionprob} below.
	\begin{defn}[Tree optimization]
	\label{def:resampling:treeoptimization}
	Let $\HH(p)$ denote the Shannon entropy of a discrete probability measure $p$. For $H^\sm \in \bDelta^\sm$, define the following quantities:
	\begin{equation}\label{eq:def:treeop:Sigma and s}
	\begin{split}
	    \mathbf{\Sigma}^\tr(H^\sm) &\equiv \HH(\dot{H})+d\HH(\hat{H}^\sm)-d\HH(\bar{H}^\sm)+d\langle \log \hat{v}, \hat{H}^\sm \rangle\,,\\
	    \mathbf{s}^\tr(H^\sm) &\equiv \langle \log \dot{\Phi}, \dot{H}^\sm \rangle + d \langle \log \hat{\Phi}^{\textnormal{m}}, \hat{H}^\sm \rangle + d \langle \log \bar{\Phi}, \bar{H}^\sm \rangle\,.
	\end{split}
	\end{equation}
	The tree analog of $F_{\la,L}(H)$, defined in \eqref{eq:1stmo dec by H}, is then defined by
	\begin{equation}\label{eq:def:treeop:Lambda}
	    \mathbf{\Lambda}(H^\sm) \equiv \mathbf{\Sigma}^\tr(H^\sm)+\lambda \mathbf{s}^\tr(H^\sm).
	\end{equation}
	Given the boundary constraint $\dot{h}[H]=\dot{h}$, we denote the optimal $\mathbf{\Lambda}(H^\sm)$ for the truncated and untruncated models by
	\begin{equation}\label{eq:def:Lambda:op}
	\begin{split}
	    \mathbf{\Lambda}^{\textnormal{op}}(\dot{h}) &\equiv \sup\{\mathbf{\Lambda}(H^\sm): H^\sm \in \bDelta^\sm, \dot{h}[H^\sm]=\dot{h}\}\,,\\
	    \mathbf{\Lambda}^{\textnormal{op}}_{L}(\dot{h}) &\equiv \sup\{\mathbf{\Lambda}(H^\sm): H^\sm \in \bDelta^{\sm,(L)}, \dot{h}[H^\sm]=\dot{h}\}\,.
	\end{split}
	\end{equation}
	Finally, we denote their difference by
	\begin{equation}\label{eq:def:treeop:Xi}
	\begin{split}
	    \mathbf{\Xi}(H^\sm) &\equiv \mathbf{\Lambda}^{\textnormal{op}}\left(\dot{h}[H^\sm]\right)-\mathbf{\Lambda}(H^\sm)\quad\textnormal{for}\quad H^\sm \in \bDelta^{\sm}\,,\\
	    \mathbf{\Xi}_{L}(H^\sm) &\equiv \mathbf{\Lambda}^{\textnormal{op}}_{L}\left(\dot{h}[H^\sm]\right)-\mathbf{\Lambda}(H^\sm)\quad\textnormal{for}\quad H^\sm \in \bDelta^{\sm,(L)}\,.
	\end{split}
	\end{equation}
 	\end{defn}
 	We now introduce certain subsets of the full state space, which will be useful throughout this section. Given $H^\sm \in \bDelta^\sm$, $Y\subset V(\GGG)$, and $\eps>0$, let $\AAA(H^\sm,Y,\eps)$ be the set of $(\GGG,Y, \sig)$ such that the following conditions hold:
	\begin{itemize}
	\item $H^\sm[\GGG,Y,\sig]=H^\sm$.
	\item $\{\dot{\ttt}_{\sig}(e)\}_{e \in \delta \NN(Y)}$ are disjoint with $v\left (\dot{t}_{\sig}(e)\right)\leq\frac{-4\log\eps}{k\log 2}$ for all $e\in \delta \NN(Y)$, where $v\left (\dot{t}_{\sig}(e)\right)$ denote the number of variables in $\dot{t}_{\sig}(e)$. 
	\end{itemize}
	The choice of $\frac{-4\log\eps}{k\log 2}$ above is justified by Lemma \ref{lem:1stmo:resampling:treedisjoint} below. Also, for $H^\sm\in \bDelta^{\sm,(L)}$ and $Y\subset V(\GGG)$, let $\AAA_L(H^\sm,Y)$ be the set of $(\GGG,Y,\sig)$ such that the following conditions hold:
	\begin{itemize}
	    \item $\sig\in \Omega_L^{E}$ with $H^\sm[\GGG,Y,\sig]=H^\sm$.
	    \item $\{\dot{\ttt}_{\sig}(e)\}_{e \in \delta \NN(Y)}$ are disjoint.
	\end{itemize}
	\begin{lemma}
	\label{lem:resampling:transitionprob}
	Consider $Y\subset V(\GGG)$ with $\eps n/2 \leq |Y|\leq 2 \eps n$, and $\eps>0$. For $H^\sm \in \bDelta^\sm$, let $A_1=(\GGG^\prime,Y, \utau)$ be the state one-step reachable from $\AAA(H^\sm,Y,\eps)$ by the untruncated resampling Makrov chain with transition probability $\pi$. Then, for a constant $C_{k,\eps}$ depending only on $k$ and $\eps$, 
	\begin{equation}\label{eq:lem:resampling:transitionprob:untruncated}
	   \pi\left(A_1, \AAA(H^\sm,Y,\eps)\right) \leq \exp\left(-\frac{\eps n}{2}\Xi(H^\sm) + C_{k,\eps}\log n\right)
	\end{equation}
	Moreover, for $H^\sm \in \bDelta^{\sm,(L)}$, let $A_1=(\GGG^\prime,Y,\utau)$ is one-step reachable from $\AAA_L(H^\sm,Y)$ by the $L$-truncated resampling Markov chain. Then, for a constant $C_{k,L}$ depending on $k$ and $L$,
	\begin{equation*}
	     \pi_L\left(A_1, \AAA_L(H^\sm,Y)\right) \leq \exp\left(-\frac{\eps n}{2}\Xi_L(H^\sm) + C_{k,L}\log n\right)
	\end{equation*}
	\end{lemma}
	\begin{proof}
	We only provide the proof for the untruncated model, i.e. \eqref{eq:lem:resampling:transitionprob:untruncated}, since the truncated case follows by the same argument. The definition of $H^\sm$ in \eqref{eq:def:H:sampled} only depends on $\sig_{\NN}$ and $\uL_{\NN}$, so denote this relation by $H^\sm=H^\sm(\sig_{\NN},\uL_{\NN})$ with abuse of notation. Then, by definition of the resampling Markov chain, we have the following expression for the transition probability:
	\begin{equation}\label{eq:1stmo:resampling:transition:prob}
	    \pi\left(A_1,\AAA(H^\sm, Y,\eps)\right)=\frac{\sum_{H^\sm(\sig_{\NN},\uL_{\NN})=H^\sm}w^{\lit}_{\NN}(\sig_{\NN}; \uL_{\NN})^{\la}}{\sum_{H^\prime \in \bDelta^\sm_{\kappa}}\sum_{H^\sm(\utau_{\NN},\uL_{\NN}^\prime)=H^\prime}\one\{\dot{h}[H^\prime]=\dot{h}[H^\sm]\}w^{\lit}_{\NN}(\utau_{\NN}; \uL_{\NN}^\prime)^{\la}}.
	\end{equation}
	Note that by definition of $\AAA(H^\sm,Y,\eps)$, $\dot{h}[H^\sm]$ is supported on $\dsigma \in \dot{\Omega}$ with $v(\dsigma)\leq \frac{-4\log \eps}{k\log 2}$, where $v(\dsigma)$ is the number of variables in the tree $\dsigma$. Hence, if $\dot{h}[H^\prime]=\dot{h}[H^\sm]$, then $H^\prime \in \bDelta^{\sm,(\frac{-4d \log \eps}{\log 2})}_{\kappa}$. Moreover, we can compute 
	\begin{equation*}
	\sum_{H^\sm(\utau_{\NN},\uL_{\NN}^\prime)=H^\prime}w^{\lit}_{\NN}(\utau_{\NN}; \uL_{\NN}^\prime)^{\la}= \frac{\kappa!}{(\kappa\dot{H}^\prime)!}\frac{\kappa d!}{(\kappa d\hat{H}^\prime)!}2^{\kappa kd}(\kappa d \bar{H}^\prime)!\exp\big(\kappa d\langle \log \hat{v}, \hat{H}^\prime \rangle+\la \kappa s^\tr(H^\prime)\big).
	\end{equation*}
	Using the fact $\left(\frac{x}{e}\right)^{x}\leq x! \leq e\sqrt{x}\left(\frac{x}{e}\right)^{x}$ in $(\kappa\dot{H}^\prime)!, (\kappa d\hat{H}^\prime)!$ and $(\kappa d \bar{H}^\prime)!$, we have the crude bound
	\begin{multline}\label{eq:1stmo:resampling:long}
	    	(e\sqrt{\kappa})^{-|\textnormal{supp}(\dot{H}^\prime)|} (e\sqrt{\kappa d})^{-|\textnormal{supp}(\hat{H}^\prime)|}\exp\big(\kappa\Lambda(H^\prime)\big) \leq (\kappa !\kappa d! 2^{\kappa kd})^{-1}\sum_{H^\sm(\utau_{\NN},\uL_{\NN}^\prime)=H^\prime}w^{\lit}_{\NN}(\utau_{\NN}; \uL_{\NN}^\prime)^{\la} \\
	    	 \leq (e\sqrt{\kappa d})^{|\textnormal{supp}(\bar{H}^\prime)|}\exp\big(\kappa\Lambda(H^\prime)\big).
	\end{multline}
	Since $H^\prime \in \bDelta^{\sm,(\frac{-4d \log \eps}{\log 2})}_{\kappa}$ and $\kappa\leq 2\eps n$, $|\textnormal{supp}(\dot{H}^\prime)|\vee |\textnormal{supp}(\hat{H}^\prime)|\vee |\textnormal{supp}(\bar{H}^\prime)|\leq n^{C_{k,\eps}}$ for a constant depending only on $k$ and $\eps$. Thus, using the bound \eqref{eq:1stmo:resampling:long} in \eqref{eq:1stmo:resampling:transition:prob} finishes the proof of \eqref{eq:lem:resampling:transitionprob:untruncated}.
	\end{proof}
	We now gather key properties of $\Xi(\cdot)$ and $\Xi_L(\cdot)$.	Recall that when the $\frac{3}{2}$ neighborhoods of $v\in Y$ do not intersect, $\NN(Y)=\sqcup_{i=1}^{\kappa}\DD_{i}$, where $\kappa =|Y|$ and $\DD_1,...,\DD_{\kappa}$ are disjoint copies of $\DD$. For coloring $\sig_\DD$, define its weight by 
	\begin{equation*}
	w_{\DD}(\sig_\DD) \equiv \dot{\Phi}(\sig_{\delta v})\prod_{e\in \delta v}\left\{\bar{\Phi}(\sigma_e)\hat{\Phi}(\sig_{\delta a(e)})\right\},
	\end{equation*}
	where $v$ is the unique variable in $\DD$. Denote by $\Omega_{\DD}\equiv \{\sig_{\DD}:w_{\DD}(\sig_\DD)\neq 0\}$ the space of valid colorings $\sig_\DD$ on $\DD$. Given $(\GGG,Y,\sig)$, the statistics of $\sig_{\NN(Y)}$ are summarized by $\nu\equiv \nu[\GGG,Y,\sig] \in \PPP(\Omega_\DD)$, where $\nu(\sig_{\DD})$ is the fraction of $\sig_{\DD}$ among $\sig_{\DD_i},1\leq i\leq |Y|$. Then $H^\sm[\GGG,Y,\sig]$ is a linear projection of $\nu[\GGG,Y,\sig]$ and we denote this relation by $H^\sm =H^\tr(\nu)$. Then, $\bLa(H)$ in Definition \ref{def:resampling:treeoptimization} has the following characterization.
\begin{lemma}[\cite{ssz22}, Lemma 5.2]
	\label{lem:resampling:Lambda:sup}
	For $H^\sm \in \bDelta^\sm$, we have 
	\begin{equation}\label{eq:express:Lambda}
	    \bLa(H^\sm) =\sup\left\{\HH(\nu)+\la\langle \log w_{\DD}, \nu \rangle: \nu \in \PPP(\Omega_{\DD})\textnormal{ with } H^\tr(\nu)=H^\sm \right\}.
	\end{equation}
	\end{lemma}
	Hence, $\bLa^{\op}(\dot{h})$ for $\dot{h}\in \PPP(\dot{\Omega})$ and $\bLa^{\op}_{L}(\dot{h})$ for $\dot{h} \in \PPP(\Omega_{L})$ in \eqref{eq:def:Lambda:op} can be expressed as
	\begin{equation}\label{eq:express:Lambda:op:sup}
	\begin{split}
	&\bLa^\op(\dot{h})=\sup\left\{\HH(\nu)+\la\langle \log w_{\DD}, \nu \rangle: \nu \in \PPP(\Omega_{\DD})\textnormal{ with } \dot{h}\left[H^\tr(\nu)\right]=\dot{h}\right\}\,,\\
	&\bLa^\op_L(\dot{h})=\sup\left\{\HH(\nu)+\la\langle \log w_{\DD}, \nu \rangle: \nu \in \PPP(\Omega_{\DD})\textnormal{ with } \dot{h}\left[H^\tr(\nu)\right]=\dot{h}\textnormal{ and } H^\tr(\nu)\in \bDelta^{\sm, (L)} \right\}\,.
	\end{split}
	\end{equation}
	It was shown in \cite[Appendix C]{ssz22} that the optimization in the expression above for $\bLa^\op_L(\dot{h})$, which is a finite dimensional entropy maximization subject to a linear constraint, has a unique maximizer $\nu= \nu^{\op}_L(\dot{h})$, and there exists a unique $\dot{q}[\dot{h}]\equiv\dot{q}_L[\dot{h}] \in \PPP(\dot{\Omega}_L)$ such that $\nu^{\op}_{L}[\dot{h}]$ can be expressed as
	\begin{equation}\label{eq:def:nu:q:dot}
	    \nu^\op_{L}[\dot{h}]= \nu_{\dot{q}[\dot{h}]},\textnormal{ where }~~\nu_{\dot{q}}(\sig_{\DD})\equiv \frac{w_{\DD}(\sig_{\DD})^\la}{Z_{\dot{q}}}\prod_{e\in \delta \DD} \dot{q}(\dot{\sigma}_e),
	\end{equation}
	where $Z_{\dot{q}}$ is the normalizing constant. Observe that $\dot{h}=\dot{h}\big[H^\tr(\nu_{\dot{q}})\big]$ holds for $\dot{q}\equiv \dot{q}_L[\dot{h}]$. Thus the inverse function of $\dot{h} \to \dot{q}_L[\dot{h}]$ is given by $\dot{q} \to \dot{h}_{\dot{q}}$, where
	\begin{equation}\label{eq:1stmo:hdot:intermsof:qdot}
	    \dot{h}_{\dot{q}}(\dot{\sigma})\equiv \sum_{\sig \in \Omega_L^{k}}\frac{\hat{\Phi}(\sig)^\la}{Z^\prime_{\dot{q}}}\prod_{i=1}^{k-1}\dot{q}(\dot{\sigma}_i) \textnormal{BP}\dot{q}(\sigma_k)\one\{\dot{\sigma}_1=\dot{\sigma}\},
	\end{equation}
	and $Z^\prime_{\dot{q}}$ is a normalizing constant. However, if $\dot{h}$ is not finitely supported, the results from \cite[Appendix C]{ssz22} do not directly apply. We show in Appendix \ref{sec:appendix:continuity:tree:optmization} that at least when $\dot{h}$ has exponential tail, there exists a unique $\dot{q}\equiv\dot{q}[\dot{h}]$, which satisfies $\dot{h}_{\dot{q}}=\dot{h}$.
	\begin{lemma}[Proved in Appendix \ref{sec:appendix:continuity:tree:optmization}]\label{lem:exists:qdot:hdot:exp:tail}
	Suppose $\dot{h}\in \PPP(\dot{\Omega})$ satisfies $\dot{h}(\rr)\vee \dot{h}(\ff)=O(\frac{1}{2^k})$ and $\sum_{\dsigma: v(\dsigma)\geq L}\dot{h}(\dsigma)\leq 2^{-ckL}$ for all $L\geq 1$, where $v(\dsigma)$ is the number of variables in $\dsigma$ and $c>0$ is an absolute constant. Then, there exists a unique $\dot{q}\equiv \dot{q}[\dot{h}]\in \PPP(\dot{\Omega})$ such that $\dot{h}_{\dot{q}}=\dot{h}$, where $\dot{h}_{\dot{q}}$ is defined in \eqref{eq:1stmo:hdot:intermsof:qdot}.
	\end{lemma}
	We remark that other various properties regarding the tree optimization are also provided in Appendix \ref{sec:appendix:continuity:tree:optmization}, which are crucial for the proofs of Propositions \ref{prop:maxim:1stmo} and \ref{prop:negdef}. It was shown in \cite[Proposition 5.1]{ssz22} that the unique minimizer of $\Xi_L(H)$, when $H=H^\sy$, is given by $H=H^\star_{\la,L}$. Having Lemma \ref{lem:exists:qdot:hdot:exp:tail} in hand, the analogous statement for the untruncated model can be established.
	\begin{lemma}\label{lem:1stmo:unique:zero:Xi}
	 Consider $H\in \bDelta$ such that $H=H^{\textnormal{sy}}$ and $\dot{h}=\dot{h}[H]$ satisfies $\sum_{\dsigma: v(\dsigma)\geq L}\dot{h}(\dsigma)\leq 2^{-ckL}$ for all $L\geq 1$ and an absolute constant $c>0$. Then, $\Xi(H)=0$ if and only if $H=H^\star_{\la}$. Also, for $H \in \bDelta^{(L)}$ with $H=H^{\textnormal{sy}}$, $\Xi_{L}(H)=0$ if and only if $H=H^\star_{\la,L}$.
	\end{lemma}
     \begin{proof}
     Recall the definition of $\bDelta^{(L)}$ in Definition \ref{def:empiricalssz}, which is the analog of $N_{\circ}$ in \cite[Eq.~(44)]{ssz22}, imposed $\max\big\{ \bar{H}(\ff), \bar{H}(\rr)\big\}\leq \frac{7}{2^k}$. Thus, the last assertion regarding the truncated model follows from the \cite[Proposition 5.1-$a$]{ssz22}. We now consider the untruncated model. Since $\dot{h}=\dot{h}[H]$ satisfies $\dot{h}(\rr)\vee \dot{h}(\ff)=O(\frac{1}{2^k})$ (recall the condition $\max\big\{ \bar{H}(\ff), \bar{H}(\rr)\big\}\leq \frac{7}{2^k}$ in Defintion \ref{def:empiricalssz}) and $\sum_{\dsigma: v(\dsigma)\geq L}\dot{h}(\dsigma)\leq 2^{-ckL}$ for all $L\geq 1$, Lemma \ref{lem:exists:qdot:hdot:exp:tail} shows that there exists $\dot{q}\in \PPP(\dot{\Omega})$ such that $\dot{h}_{\dot{q}}=\dot{h}$. Recall $\nu_{\dot{q}}$ $Z_{\dot{q}}$ in \eqref{eq:def:nu:q:dot}. Denote $\overline{\bLa}(\mu):= \HH(\mu)+\la\langle \log w_{\DD}, \mu \rangle$ for $\mu \in \PPP(\Omega_{\DD})$. Then, we have that (e.g. see the proof of \cite[Proposition 5.1]{ssz22})
     \begin{equation*}
         \overline{\bLa}(\mu)= -\DD_{\textnormal{KL}}\big(\mu\mid \nu_{\dot{q}}\big)+\log Z_{\dot{q}}-|\delta \DD| \big\langle \log \dot{q}, \dot{h}^{\tr}\big[H^{\tr}(\mu)\big]\big\rangle.
     \end{equation*}
     In particular, if $\dot{h}^{\tr}\big[H^{\tr}(\mu)\big]=\dot{h}$, we have that $\overline{\bLa}(\mu)\leq \overline{\bLa}(\nu_{\dot{q}})$. Thus, $\nu_{\dot{q}}$ solves the optimization problem for $\bLa(\dot{h})$ in \eqref{eq:express:Lambda:op:sup}. Thus, if denote by $\mu(H)\in \PPP(\Omega_{\DD})$ the unique measure achieving the supremum in \eqref{eq:def:Lambda:op}, then 
     \begin{equation*}
         \Xi(H)=\overline{\bLa}(\nu_{\dot{q}})-\overline{\bLa}\big(\mu(H)\big)=\DD_{\textnormal{KL}}\big(\mu(H) \mid \nu_{\dot{q}}\big).
    \end{equation*}
    Now, if $H=H^\star_{\la}$, then since $\dot{q}^\star_{\la}$ is the BP fixed point, $H^{\tr}[\nu_{\dot{q}^\star_{\la}}]=H^\star_{\la}$ holds. Thus, $\nu_{\dot{q}^\star_{\la}}$ belongs to the constraint set in \eqref{eq:def:Lambda:op} and we have that $\mu(H^\star_{\la})=\nu_{\dot{q}^\star_{\la}}$. Hence, the equation above show $\Xi(H^\star_{\la})=0$.

    On the other hand, if $\Xi(H)=0$ and $H=H^{\sy}$, then we must have $\mu(H)=\nu_{\dot{q}}$, which implies that $H=H^{\tr}[\nu_{\dot{q}}]$ satisfies $H=H^{\sy}$. In the proof of \cite[Lemma 5.4]{ssz22}, it was shown that if $\Xi(H)=0$ and $H=H^{\sy}\in \bDelta$ (recall that $\bDelta$ included the condition $\max\big\{ \bar{H}(\ff), \bar{H}(\rr)\big\}\leq \frac{7}{2^k}$), then $H=H_{\dot{q}}$ for a BP fixed point $\dot{q}\in \PPP(\dot{\Omega})$. Moreover, \cite[Lemma 5.7]{ssz22} shows that this further implies that $\dot{q}=\dot{q}^\star_{\la}$\footnote{ Although \cite[Lemma 5.4, Lemma 5.7]{ssz22} are stated only for truncated model, the proof of \cite[Lemma 5.4, Lemma 5.7]{ssz22} used nothing about the truncated model, and their proof applies to the untruncated model.}, which concludes the proof.
     \end{proof}
	Henceforth, we denote $h^\star_{\la,L}\equiv \dot{h}_{\dot{q}^\star_{\la,L}}=\dot{h}[H^\star_{\la,L}]$. The lemma below shows that $\Xi_L$ has quadratic growth near its minimizer, with constant uniform in $L$. 
	\begin{lemma}\label{lem:quadratic:growth:Xi:1stmo}
	There exist constants $C_k$, which only depends on $k$, and $\eps_L>0$, which depends on $k$ and $L$, such that for $H\in \bDelta^{(L)}$ with $H=H^\sy$ and $||\dot{h}[H]-\dot{h}^\star_{\la,L}||_1<\eps_L$, we have
	\begin{equation}\label{eq:lem:quadratic:growth:Xi:1stmo}
	    \Xi_{L}(H)\geq C_k||H-H^\star_{\la,L}||_1^{2}.
	\end{equation}
	\end{lemma}
	\begin{proof}
	We follow the same route taken in the proof of \cite[Proposition 5.1]{ssz22}. The only improvement is that the constant $C_k>0$ in \eqref{eq:lem:quadratic:growth:Xi:1stmo} is uniform in $L$.
	
	For $H\in \bDelta^{(L)}$ with $H=H^\sy$, let $\dot{q} \equiv \dot{q}_L\big[\dot{h}[H]\big]$. Since $\dot{h}\to \dot{q}_L[\dot{h}]$ is continuous, and $\dot{q}^\star_{\la,L}= \dot{q}_L[\dot{h}^\star_{\la,L} ]$, we take $\eps_L$ small enough so that the following holds.
	\begin{itemize}
	    \item In Lemma \ref{lem:1stmo:BPfixedpoint:estimates}, we show that $2\dot{q}^\star_{\la,L}(\bb_0)=\dot{q}^\star_{\la,L}(\bb)> \frac{1}{2}-\frac{C}{2^k}$ for some universal constant $C>0$. Also, since $\dot{q}^\star_{\la,L}$ is the BP fixed point, $\textnormal{BP} \dot{q}^\star_{\la,L} (\bb)=\dot{q}^\star_{\la,L}(\bb)$ holds. Hence, we can take $\eps_L$ small enough so that $\min\big(\dot{q}(\bb_1),\dot{q}(\bb_0),\textnormal{BP}\dot{q}(\bb_0),\textnormal{BP}\dot{q}(\bb_1)\big)\geq\frac{1}{4}-\frac{C}{2^k}$ holds.
	    \item $\dot{q}^{\textnormal{av}}\in \Gamma_{C}$, where $\Gamma_{C}$ is defined in \eqref{eq:BPcontraction:1stmo}. Here, $\dot{q}^{\textnormal{av}}(\dot{\sigma})\equiv \frac{\dot{q}(\dot{\sigma})+\dot{q}(\dot{\sigma}\oplus 1)}{2}, \dot{\sigma} \in \dot{\Omega}_L$. Hence, by Proposition \ref{prop:BPcontraction:1stmo}, $||\textnormal{BP} \dot{q}-\dot{q}^\star_{\la,L}||_1 \lesssim \frac{k^2}{2^k}||\dot{q}-\dot{q}^\star_{\la,L}||_1$.
	\end{itemize}
	Now, denote by $\mu\equiv\mu(H)\in \PPP(\Omega_{\DD})$ and $\nu\equiv\nu_{L}\left(\dot{h}[H]\right)\in \PPP(\Omega_{\DD})$ the unique measures achieving the supremum in the \textsc{rhs} of \eqref{eq:express:Lambda} and \eqref{eq:express:Lambda:op:sup} respectively. Then, $\Xi_L[H]=\DD_{\textnormal{KL}}(\mu\mid \nu) \geq \frac{1}{2}||\mu-\nu||_1^{2}$ holds, so it suffices to show that $||\mu-\nu||_1\gtrsim_{k}||H-H^\star_{\la,L}||_1$ holds.
	
	First, denote $\nu^\star\equiv \nu_{\dot{q}^\star_{\la,L}}$, and note that $H^\tr(\cdot)$ is a linear projection with $H^\tr(\mu)=H$ and $H^\tr(\nu^\star)=H^\star_{\la,L}$. Thus,
	\begin{equation*}
	    ||H-H^\star_{\la,L}||_1\lesssim ||\mu -\nu^\star||_1\leq ||\mu-\nu||_1+||\nu-\nu^\star||_1,
	\end{equation*}
	so it remains to show $||\nu-\nu^\star||_1\lesssim_{k} ||\mu-\nu||_1$. Because $\dot{q}(\bb)\geq\frac{1}{2}-\frac{C}{2^k}$ holds, Lemma \ref{lem:nu:q:continuous:1stmo} in Appendix \ref{sec:appendix:continuity:tree:optmization} shows $||\nu-\nu^\star||_1\lesssim_{k}||\dot{q}-\dot{q}^\star_{\la,L}||_1$. Moreover, for a universal constant $C>0$, we have
	\begin{equation*}
	    (1-\frac{Ck}{2^k})||\dot{q}-\dot{q}^\star_{\la,L}||_1\leq ||\dot{q}-\dot{q}^\star_{\la,L}||_1-||\textnormal{BP}\dot{q}-\dot{q}^\star_{\la,L}||_1\leq ||\dot{q}-\textnormal{BP}\dot{q}||_1,
	\end{equation*}
	so the rest of the proof is devoted to proving $||\dot{q}-\textnormal{BP}\dot{q}||_1\lesssim_{k}||\mu-\nu||_1$. Let $K\equiv (\dot{K},\hat{K},\bar{K})\equiv H^\tr(\nu)$. Also, define $\hat{K}^\prime$ to be a rotation of $\hat{K}$: $\hat{K}^\prime(\sig)\equiv \hat{K}(\sigma_2,...,\sigma_k,\sigma_1)$. Since $H=H^\sy$, we have
	\begin{equation}\label{eq:bound:K:hat:diff:q:diff-1}
	    ||\hat{K}-\hat{K}^\prime||_1\leq ||\hat{H}-\hat{K}||_1+||\hat{H}-\hat{K}^\prime||_1=2||\hat{H}-\hat{K}||_1\lesssim ||\mu-\nu||_1.
	\end{equation}
	To this end, we aim to lower bound $||\hat{K}-\hat{K}^\prime||_1$ by $||\dot{q}-\textnormal{BP}\dot{q}||_1$. First, note that
	\begin{equation*}
	    \hat{K}(\sig)=\frac{\hat{\Phi}(\sig)^{\la}}{Z^\prime_{\dot{q}}}\textnormal{BP}\dot{q}(\dot{\sigma_1})\prod_{i=2}^{k}\dot{q}_i(\dot{\sigma}_i),
	\end{equation*}
	where $Z^\prime_{\dot{q}}$ denotes a normalizing constant. Thus, we can lower bound
	\begin{equation*}
	    ||\hat{K}-\hat{K}^\prime||_1\geq \sum_{\substack{\sig \in \Omega_L^{k}\\\sigma_2=...=\sigma_k=\bb_0 }}\frac{\hat{\Phi}(\sig)^{\la}}{Z^\prime_{\dot{q}}}\Big|\textnormal{BP}\dot{q}(\dot{\sigma_1})\dot{q}(\bb_0)-\dot{q}(\dsigma_1)\textnormal{BP}\dot{q}(\bb_0)\Big|\dot{q}(\bb_0)^{k-2}
	\end{equation*}
	Note that we can crudely bound $Z^\prime_{\dot{q}}\leq 1$ since $\hat{\Phi}(\sig)^{\la}\leq 1$. Also, for any $\dsigma_1\in \dot{\Omega}$, taking $\sigma_1=\dsigma_1$ if $\dsigma \in \{\rr,\bb\}$ and $\sigma_1=(\dsigma_1,\fs)$ if $\dsigma \in \{\ff\}$, $\sig=(\sigma_1,\bb_0,...,\bb_0)$ is valid with $\hat{\Phi}(\sig)^{\la}\geq 2^{-k+1}$. Hence,
	\begin{equation}\label{eq:bound:K:hat:diff:q:diff-2}
	    ||\hat{K}-\hat{K}^\prime||_1\geq 2^{-k+1}\dot{q}(\bb_0)^{k-2}\sum_{\dsigma\in \dot{\Omega}}\Big|\textnormal{BP}\dot{q}(\dot{\sigma_1})\dot{q}(\bb_0)-\dot{q}(\dsigma_1)\textnormal{BP}\dot{q}(\bb_0)\Big|\gtrsim 2^{-3k}||\dot{q}-\textnormal{BP}\dot{q}||_1,
	\end{equation}
	where the last inequality is due to $\textnormal{BP}\dot{q}(\bb_0)\wedge \dot{q}(\bb_0) \geq \frac{1}{4}-\frac{C}{2^k}$. Reading \eqref{eq:bound:K:hat:diff:q:diff-1} and \eqref{eq:bound:K:hat:diff:q:diff-2} together, $||\dot{q}-\textnormal{BP}\dot{q}||_1\lesssim_{k} ||\mu-\nu||_1$ holds, which concludes the proof.
	\end{proof}

	\subsection{Maximizer of the exponent}\label{subsubsec:resampling:maximizer}
	We now prove Proposition \ref{prop:maxim:1stmo}. The result for the truncated model in \eqref{eq:prop:maxim:1stmo:truncated} is straightforward from the fact that $F_{\la,L}(H)$ in \eqref{eq:1stmo dec by H} is uniquely maximized at $H^\star_{\la,L}$, which was shown in \cite[Proposition 3.4]{ssz22}(see Remark \ref{rem:truncated:negdef:SSZ}). Hence, we consider the result for the untruncated model in \eqref{eq:prop:maxim:1stmo:untruncated}.
	
	The first step is to define the set of bad variables, which should be avoided while sampling $Y$: for $v\in V(\GGG)$, let $\NN(v)$ be the $\frac{3}{2}$ neighborhood of $v$ and $\delta \NN(v)$ be the set of half-edges hanging at the boundary of $\NN(v)$. Given $(\GGG,\sig)$, define $V_{\textnormal{bad}}\equiv V_{\textnormal{bad}}(\GGG,\sig)$ by 
	\begin{equation*}
	    V_{\textnormal{bad}}\equiv \{v\in V:\exists e_1,e_2 \in \delta \NN(v)\quad\textnormal{s.t.}\quad \ttt_{\sig}(e_1)\cap \ttt_{\sig}(e_2) \neq \emptyset\}
	\end{equation*}
	The next lemma shows that there are not too many bad variables on average.
	\begin{lemma}\label{lem:1stmo:resampling:bad:variables}
	Fix $B \in \bDelta^{\textnormal{b}}_n$ and $(n_{\ttt})_{\ttt\in \FFF_{\tr}} \sim B$ such that $(n_{\ttt})_{\ttt\in \FFF_{\tr}}\in \ee_{\frac{1}{4}}$. Then, we have
	\begin{equation}\label{eq:lem:1stmo:resampling:bad:variables}
	    \sum_{\substack{(\GGG,\sig):B[\sig]=B,\\ n_{\ttt}[\sig]=n_{\ttt},\forall\ttt\in\FFF_{\tr}}}\P(\GGG)w^\lit_{\GGG}(\sig)^\la\one\big\{|V_{\textnormal{bad}}|\geq \sqrt{n}\big\}\lesssim_{k}\frac{\log n}{\sqrt{n}} \sum_{\substack{(\GGG,\sig):B[\sig]=B,\\ n_{\ttt}[\sig]=n_{\ttt},\forall\ttt\in\FFF_{\tr}}}\P(\GGG)w^\lit_{\GGG}(\sig)^\la
	\end{equation}
	\end{lemma}
	\begin{proof}
	Define the law $\P(\GGG,\sig)\equiv\frac{\P(\GGG)w^\lit_{\GGG}(\sig)^\la\one\{B[\sig]=B,n_{\ttt}(\sig)=n_{\ttt},\forall \ttt\in \FFF_{\tr}\}}{\E\Z_{\la}[B,(n_{\ttt})_{\ttt\in\FFF_{\tr}}]}$. Then, we now aim to show
	\begin{equation}\label{eq:goal:lem:bad:variables}
	    \E[|V_{\textnormal{bad}}(\GGG,\sig)|]\lesssim_{k} \log n,
	\end{equation}
	under $(\GGG,\sig)\sim \P$. Then \eqref{eq:lem:1stmo:resampling:bad:variables} is implied from \eqref{eq:goal:lem:bad:variables} by Markov's inequality. We simulate $(\GGG,\sig) \sim \P$ as follows: first, choose $\ttt^{\textnormal{lab}}\in \LLL(\ttt)$ uniformly at random for each $n_{\ttt}$ number of $\ttt$'s. Say the resulting number of $\ttt^{\textnormal{lab}}$ is $n_{\ttt^{\textnormal{lab}}}$. Second, recalling \eqref{eq:compute:1stmo:labeled}, follow the procedure given in the paragraph above \eqref{eq:compute:1stmo:labeled} to produce $\sig^{\textnormal{lab}}$ with $B[\lsig]=B$ and $n_{\ttt^{\textnormal{lab}}}(\lsig)=\ttt^{\textnormal{lab}}$ for all $\ttt^{\textnormal{lab}}$. Third, we restore $\sig$ from $\lsig$ by dropping the spurious labels on the half-edges. The proof of Proposition \ref{prop:1stmo:B nt decomp} shows that this three-step procedure gives $(\GGG,\sig) \sim \P$. In particular, the variable-adjacent half-edges colored $\sigma \in \{\bb_0,\bb_1,\fs\}$ are matched uniformly at random with the clause-adjacent half-edges colored $\sigma$.

	Since $\sig$ does not contain any cyclic free components, we can classify the bad variables as
	\begin{equation}\label{eq:V:bad:inclusion}
	    V_{\textnormal{bad}} \subset V_{\textnormal{bad}}^{1}\cup V_{\textnormal{bad}}^{2} \cup V_{\textnormal{bad}}^{3} \cup V_{\textnormal{bad}}^{4},
	\end{equation}
	where $V_{\textnormal{bad}}^{i}, i=1,2,3,4$ are defined below. Denote $a\sim v$ when $a\in F$ and $v\in V$ are connected. Then, for each $i=1,2,3,4$, $V_{\textnormal{bad}}^{i}$ is the set of $v\in V$ satisfying
	\begin{itemize}
	    \item $i=1$: $\exists$non-separating clauses $a_1,a_2\sim v$ such that $\sigma_{(a_1 v)},\sigma_{(a_2 v)}\in \{\bb\}$ and $a_1,a_2$ are contained in the same free tree.  
	    \item $i=2$: $\exists e_1,e_2 \in \delta \NN(v)$ such that $a(e_1)\neq a(e_2)$, $\hat{\sigma}_{e_1}=\hat{\sigma}_{e_2}=\fs$, and $v(e_1),v(e_2)$ are in the same free tree.
	    \item $i=3$: $\exists a\sim v, e_1,e_2\in \delta a$ such that $\hat{\sigma}_{e_1}=\hat{\sigma}_{e_2}=\fs$ and $v(e_1),v(e_2)$ are in the same free tree.
	    \item $i=4$: $\exists e_1, e_2 \in \delta \NN(v)$ such that $\hat{\sigma}_{e_1}=\fs$, $\sigma_{(a(e_2)v)}\in \{\bb\}$, $a(e_2)$ is non-separating, and $v(e_1), a(e_2)$ are in the same free tree.
	\end{itemize}
	First, we bound $\E[V_{\textnormal{bad}}^{1}]$: let $V_{\sigma}\equiv \{v\in V: \exists e \in \delta v, \sigma_e= \sigma\}$ and $E_{\sigma} \equiv \{e\in E: \sigma_e =\sigma \}$ for $\sigma \in \{\bb_0,\bb_1,\fs\}$. Note that $|V_{\sigma}|,|E_{\sigma}|$ are determined by $B$ for $\sigma \in \{\bb_0,\bb_1,\fs\}$, and $|V_\sigma|\leq |E_\sigma|$ holds. The number of boundary half-edges colored either $\bb_0$ or $\bb_1$, and adjacent to a free tree $\ttt$ is at most $k v(\ttt) \leq 4\log_{2}n$ since $(n_{\ttt^\prime})_{\ttt^\prime\in \FFF_{\tr}}\in \ee_{\frac{1}{4}}$ and $n_{\ttt}\neq 0$ implies $v(\ttt)\leq \frac{4\log_{2} n}{k}$. Thus, a union bound gives
	\begin{equation}\label{eq:bound:V:bad:1}
	    \E[V_{\textnormal{bad}}^{1}]\leq \sum_{\sigma \in\{\bb\}}\binom{d}{2}|V_\sigma|\frac{4\log _2 n}{|E_\sigma|-1}\lesssim_{k} \log n.
	\end{equation}
	Turning to bound $\E[V_{\textnormal{bad}}^{2}]$, let $F_{\sigma,\fs}\equiv \{a\in F: \exists e_1\neq e_2\in \delta a, \sigma_{e_1}=\sigma,\sigma_{e_2}=\fs\}$ for $\sigma \in \{\bb_{0},\bb_1,\fs\}$. Then $|F_{\sigma,\fs}|$ are determined by $B$ for $\sigma\in\{\bb_0,\bb_1,\fs\}$, and $|F_{\sigma,\fs}|\leq |E_\sigma| \wedge |E_{\fs}|$ holds. If $v\in V_{\textnormal{bad}}^{2}\cap V_{\sigma}$, then there exist two separating clauses $a_1,a_2 \in F_{\sigma,\fs}$ and two $\fs$ edges $e_1\in \delta a_1, e_2\in \delta a_2$ such that $v(e_1),v(e_2)$ are in the same free tree. Since the number of boundary half-edges colored $\fs$ and adjacent to a free tree $\ttt$ with $n_{\ttt}\geq 1$ is at most $4d\log_2 n/k$, union bound shows
	\begin{equation*}
	    \E[V_{\textnormal{bad}}^{2}]\leq \sum_{\sigma\in\{\bb_0,\bb_1,\fs\}} \binom{d}{2}(k-1)^2 |V_\sigma|\frac{\left((k-1)|F_{\sigma,\fs}|\right)^{2}}{|E_\sigma|(|E_\sigma|-1)}\frac{4d\log_{2}n/k}{E_{\fs}-1-2\one\{\sigma =\fs\}}\lesssim_{k} \log n.
	\end{equation*}
	For the case of $E[V_{\textnormal{bad}}^{3}]$, we bound the number of clauses $a\in F_{\fs,\fs}$, which have two neighboring $\fs$ edges connected to the same tree. Such clause has $k$ neighboring variables, so
	\begin{equation*}
	    \E[V_{\textnormal{bad}}^{3}]\leq k\binom{k}{2}|F_{\fs,\fs}|\frac{4d\log_{2}n/k}{E_{\fs}-1}\lesssim_{k}\log n.
	\end{equation*}
	Finally, we bound $\E[V_{\textnormal{bad}}^{4}]$ in a similar fashion as done in \eqref{eq:bound:V:bad:1}:
	\begin{equation}\label{eq:bound:V:bad:4}
	    \E[V_{\textnormal{bad}}^{4}]\leq \sum_{\sigma \in\{\bb\}}\binom{d}{2}|V_{\sigma}|\frac{4\log_2 n}{|E_{\sigma}|-1}\lesssim_{k}\log n.
	\end{equation}
	Therefore, \eqref{eq:V:bad:inclusion}-\eqref{eq:bound:V:bad:4} altogether finish the proof of \eqref{eq:goal:lem:bad:variables}.
	\end{proof}
	\begin{lemma}
	\label{lem:1stmo:resampling:treedisjoint}
	 Fix $0<\eps<\frac{1}{2}$ and consider $(\GGG,\sig)$ with $V_{\textnormal{bad}}(\GGG,\sig)<\sqrt{n}$ and $\big(n_{\ttt}[\sig]\big)_{\ttt\in \FFF_{\tr}}\in \ee_{\frac{1}{4}}$. Then, there exists a constant $C_k>0$, which depends only on $k$, such that if $n\geq n_{0}(\eps,k)$,
	 \begin{equation}\label{eq:lem:1stmo:resampling:treedisjoint}
	     \sum_{Y:\frac{\eps n}{2}\leq |Y|\leq 2\eps n}\P_{\eps}(Y\mid \GGG)\one\left\{\{\dot{t}_{\sig}(e)\}_{e\in \NN}\textnormal{ are disjoint, }v\left (\dot{t}_{\sig}(e)\right)\leq \frac{-4\log\eps}{k\log 2}, \forall e\in \delta \NN\right\}\gtrsim e^{-C_k n \eps^2\log(\frac{1}{\eps})}.
	 \end{equation}
	 Moreover there exists an absolute constant $C >0$ such that for all $\eta >0$ and $n\geq n_0(\eps,\eta,k)$,
	 \begin{equation}\label{eq:1stmo:resampling:sampled:close:initial}
	  \sum_{Y:\frac{\eps n}{2}\leq |Y|\leq 2\eps n}\P_{\eps}(Y\mid \GGG)\one\left\{||H^\sm[\GGG,Y,\sig]-\left(H[\sig]\right)^{\textnormal{sy}}||_1\geq \eta\right\}\leq e^{-C n \eps \eta^2 }.
	 \end{equation}
	 Thus, taking $\eta=\eps^{1/3}$ so that $\eta^{2}\gg \eps\log(\frac{1}{\eps})$ for small enough $\eps$ shows the following with respect to $\eps$-sampling mechanism: for large enough $n$, with probability at least $Ce^{-C_kn\eps^2\log(\frac{1}{\eps})}$, we have that $\{\dot{t}_{\sig}(e)\}_{e\in \NN}$ are disjoint, $v\left (\dot{t}_{\sig}(e)\right)\leq\frac{-4\log\eps}{k\log 2}$ for all $e\in \delta \NN$, and $||H^\sm[\GGG,Y,\sig]-\left(H[\sig]\right)^{\textnormal{sy}}||_1\leq \eps^{1/3}$.  
	\end{lemma}
	\begin{proof}
	We first prove \eqref{eq:lem:1stmo:resampling:treedisjoint}. To begin with, denote the conditional law of $Y$ given $|Y|=\kappa$ as $\P_{\eps,\kappa}(Y\mid \GGG)\equiv \frac{\P_\eps(Y\mid \GGG)\one\{|Y|=\kappa\}}{\P_\eps(|Y|=\kappa\mid \GGG)}$. By Hoeffding's inequality, $\P_{\eps}(\eps n/2 \leq |Y|\leq 2\eps n \mid \GGG) \geq 1-2e^{-n\eps^2/4}$ holds, so in order to prove \eqref{eq:lem:1stmo:resampling:treedisjoint}, it suffices to prove the following for $\eps n/2 \leq \kappa\leq 2\eps n$:
	\begin{equation}\label{eq:goal:lem:1stmo:resampling:treedisjoint}
	    \P_{\eps,\kappa}\left(\{\dot{t}_{\sig}(e)\}_{e\in \NN}\textnormal{ are disjoint, } v\left (\dot{t}_{\sig}(e)\right)\leq \frac{-4\log\eps}{k\log 2}, \forall e\in \delta \NN \Big| \GGG\right)\geq e^{-C_k n\eps^2\log(\frac{1}{\eps})}.
	\end{equation}
	Note that $Y=\{X_i\}_{i\leq \kappa} \sim \P_\eps(\cdot \mid \GGG,\kappa)$ is uniformly distributed among $\kappa$ variables, so sampling from $\P_\eps(\cdot \mid \GGG,\kappa)$ is equivalent to sequentially sampling $X_1,...,X_{\kappa}\in V(\GGG)$ without replacement. Define
	\begin{equation*}
	V_{\textnormal{big}}\equiv V_{\textnormal{big}}(\GGG,\sig, \eps)\equiv \{v\in V:\exists e\in \delta \NN(v)\quad\textnormal{s.t.}\quad v\left (\dot{t}_{\sig}(e)\right)>\frac{-4\log\eps}{k\log 2}\}.
	\end{equation*}
	Observe that $v\in V_{\textnormal{big}}$ implies $v$ is included in the distance $2$-neighborhood of a free tree $\ttt$ with $v(\ttt)>\frac{-4\log\eps}{k\log 2}$, where the distance is measured in graph distance. Hence,
	\begin{equation}\label{eq:bound:V:big}
	    |V_{\textnormal{big}}| \leq kd\sum_{v>\frac{-4\log\eps}{k\log 2}}\sum_{\ttt\in \FFF_{\tr}, v(\ttt)=v}n_{\ttt}(\sig)v\leq kdn\sum_{v>\frac{-4\log\eps}{k\log 2}} v2^{-kv/4}\lesssim -dn\eps \log \eps.
	\end{equation}
	Thus, if we define $V_{\textnormal{bad}}^{+}\equiv V_{\textnormal{bad}}\cup V_{\textnormal{big}}$, then  $|V_{\textnormal{bad}}^{+}|\leq \sqrt{n}+Cdn\eps \log(\frac{1}{\eps})$ holds. We now define \textit{successful} sampling as follows: given $X_1,...,X_{i-1}$, call $X_i$ a successsful sampling if it satisfies the $2$ conditions detailed below.
	\begin{enumerate}
	    \item $\NN_{+}(X_i)\cap\left(\cup_{\ell=1}^{i-1}\NN_{+}(X_\ell)\right)=\emptyset$, where $\NN_{+}(v)\equiv \NN(v) \sqcup \left(\cup_{e\in \delta \NN(v)}\dot{t}_{\sig}(e)\right), v\in V$.
	    \item $X_{i} \notin V_{\textnormal{bad}}^{+}$.
	\end{enumerate}
	Note that successful sampling of $X_1,...,X_{\kappa}$ implies that $\{\dot{t}_{\sig}(e)\}_{e\in \NN}$ are disjoint and $ v\left (\dot{t}_{\sig}(e)\right)\leq \frac{-4\log\eps}{k\log 2}$ for all $e\in \delta \NN$. To this end, we aim to lower bound the probability of a successful sampling. For $v\notin V_{\textnormal{bad}}^{+}$, the number of variables in $\NN_{+}(v)$ is at most $\frac{-4d\log\eps}{\log 2}$, so we have
	\begin{multline}
	 \prod_{i=1}^{\kappa}\P\left(X_i\textnormal{ is successful}\mid X_1,...,X_{i-1}\textnormal{ is successful}\right)\geq\prod_{i=1}^{\kappa} \left(1-\frac{\frac{-4d\log\eps}{\log 2}i+|V_{\textnormal{bad}}^{+}|}{n}\right)\\
	 \geq \prod_{i=1}^{2\eps n}\left(1-\frac{Cd\log(\frac{1}{\eps})i+Cdn\eps \log(\frac{1}{\eps})+\sqrt{n}}{n}\right)\geq e^{-C_k n\eps^2\log(\frac{1}{\eps})},
	\end{multline}
	where in the last inequality, we assumed that $n$ is large enough. Therefore, \eqref{eq:goal:lem:1stmo:resampling:treedisjoint} holds.
	
	Next, we prove \eqref{eq:1stmo:resampling:sampled:close:initial}. To do so, it suffices to prove the following for $\eps n/2 \leq \kappa \leq 2\eps n$:
	\begin{equation}\label{eq:goal:lem:1stmo:sampled:close:initial}
	    \P_{\eps,\kappa}\left(||H^\sm[\GGG,Y,\sig]-\left(H[\sig]\right)^{\textnormal{sy}}||_1\geq \eta \right)\leq e^{-C n\eps \eta^2}.
	\end{equation}
	To prove the equation above, recall the definition of $\nu[\GGG, Y, \sig]\in \PPP(\Omega_{\DD})$. Let $\mu \equiv \nu[\GGG,V(\GGG), \sig]$, i.e. $\mu(\sig_{\DD})$ is the fraction of $\sig_{\DD}$ among $\sig_{\DD_1},...,\sig_{\DD_n}$, where $\DD_1,...,\DD_n$ is all the copies of $\DD$ embedded in $\GG$. Then, observe that $H^\tr(\mu)=\left(H[\sig]\right)^{\textnormal{sy}}$ holds, so $||H^\sm[\GGG,Y,\sig]-\left(H[\sig]\right)^{\textnormal{sy}}||_1\lesssim ||\nu[\GGG,Y,\sig]-\mu||_1$ holds because $\nu \to H^\tr(\nu)$ is a projection. Hence, it suffices to show 
	\begin{equation}\label{eq:goal-2:lem:1stmo:sampled:close:initial}
	    \P_{\eps,\kappa}\left(||\nu[\GGG,Y,\sig]-\mu||_1\geq \eta \right)\leq e^{-Cn\eps \eta^2}.
	\end{equation}
	We argue \eqref{eq:goal-2:lem:1stmo:sampled:close:initial} by a standard large deviation argument: it is straightforward to compute
	\begin{equation*}
	    \P_{\eps,\kappa}\left(\nu[\GGG,Y,\sig]=\nu \right)=\frac{\prod_{\sig_{\DD}\in \Omega_{\DD}}\binom{n\mu(\sig_{\DD})}{\kappa \nu(\sig_{\DD})}}{\binom{n}{\kappa}}.
	\end{equation*}
	Using Stirling's approximation, we have $\binom{\ell}{a}\leq \exp\big(\ell \HH(\frac{a}{\ell})\big)$ for all $0\leq a \leq \ell$. Also, we can lower bound $\binom{n}{\kappa}\gtrsim \frac{\sqrt{n}}{\sqrt{\kappa}\sqrt{n-\kappa}}\exp\big(n\HH(\frac{\kappa}{n})\big)\gtrsim \frac{1}{\sqrt{n}}\exp\big(n\HH(\frac{\kappa}{n})\big)$. Thus, we can further bound
	\begin{multline}\label{eq:bound:prob:nu:Y}
	     \P_{\eps,\kappa}\left(\nu[\GGG,Y,\sig]=\nu \right)\lesssim n^{1/2}\exp\left(\sum_{\sig_{\DD}\in \Omega_{\DD}}n\mu(\sig_{\DD})\HH\bigg(\frac{\kappa \nu(\sig_{\DD})}{n\mu(\sig_{\DD})}\bigg)-n\HH\bigg(\frac{\kappa}{n}\bigg)\right)\\
	     =n^{1/2}\exp\left(-\kappa\bigg(\DD_{\textnormal{KL}}(\mu\mid \nu)+\frac{n-\kappa}{\kappa}\DD_{\textnormal{KL}}\Big(\frac{n\nu-\kappa\mu}{n-\kappa}\Big|~~~ \nu\Big)\bigg)\right)\leq n^{1/2}\exp\bigg(-\frac{\eps}{4}||\mu-\nu||_1^{2}\bigg),
	\end{multline}
	where the last inequality is due to $\DD_{\textnormal{KL}}(\mu\mid \nu)\geq \frac{1}{2}||\mu-\nu||_1^{2}$. Hence, summing up \eqref{eq:bound:prob:nu:Y} for $\nu$ with $||\nu-\mu||_1\geq \eta$ shows \eqref{eq:goal-2:lem:1stmo:sampled:close:initial} for large enough $n$, concluding the proof of \eqref{eq:1stmo:resampling:sampled:close:initial}.
	\end{proof}
    \begin{lemma}\label{lem:B:sy}
    For $B\equiv (\dot{B}, \hat{B}, \bar{B})\in \bDelta^{\textnormal{b}}$, let $B^{\sy}\equiv (\dot{B},\hat{B}^{\sy}, \bar{B})\in \bDelta^{\textnormal{b}}$, where $\hat{B}^{\sy}$ is the average over all $k$ roations of $\hat{B}$. Then, there exists a universal constant $C>0$ such that for any $\delta>0$ and $\lambda\in [0,1]$, we have for large enough $n$ that
    \begin{equation*}
    \E \bZ_{\la}^{\tr}\Big[\|B-B^{\sy}\|_1\geq \delta \Big]\equiv \sum_{B: \|B-B^{\sy}\|_1\geq \delta}\E \bZ_{\la}^{\tr}[B]\leq e^{-Cn\delta^2}\E\bZ_{\la}^{\tr}\,.
    \end{equation*}
    \end{lemma}
    \begin{proof}
     By Proposition \ref{prop:1stmo:B nt decomp}, for any $\{n_{\ttt}\}_{\ttt\in \FFF_{\tr}}\sim B$, we have
     \begin{equation*}
     \E \bZ_\lambda^{\tr} \big[B, \{n_\ttt \}_{\ttt\in \mathscr{F}_{\tr}}\big] =  n^{O_k(1)} e^{n\Psi_\circ(B)} \prod_{\ttt\in \mathscr{F}_{\tr}}\left[ \frac{1}{n_\ttt !} \left(\frac{n J_\ttt w_\ttt^\lambda}{e} \right)^{n_\ttt}\right]\,,
     \end{equation*}
     where $\Psi_{\circ}(B)$ is defined in \eqref{eq:def:Psi:circ:p:circ:1stmo}. Note that $\langle \hat{B}^{\sy}, \log \hat{v} \rangle= \langle \hat{B}, \log \hat{v} \rangle$, since $\hat{v}(\sig)$ is invariant under the permutation of the coordinates of $\sig$. Moreover, $\hat{B} \to -\langle \hat{B}, \log \hat{B} \rangle$ is a strictly concave function with concavity parameter $C>0$ for some universal constant $C>0$. Thus, 
     \begin{equation*}
     \Psi_{\circ}(B)\leq \Psi_{\circ}(B^{\sy})-C\|B-B^{\sy}\|_1^2\, 
     \end{equation*}
     Therefore, the 2 equations in the displays above conclude the proof.
    \end{proof}
    Having Lemmas \ref{lem:1stmo:resampling:bad:variables}, \ref{lem:1stmo:resampling:treedisjoint}, and \ref{lem:B:sy} in hand, we now prove Proposition \ref{prop:maxim:1stmo}.
	\begin{proof}[Proof of Proposition \ref{prop:maxim:1stmo}]
	Fix $\delta>0$ throughout the proof. We consider $\eps>0$ small enough in terms of $\delta$, to be determined below. Let $\AAA_0\equiv \AAA_0(\eps)$ be the set of $(\GGG,Y,\sig)$ which satisfy the following $3$ conditions:
	\begin{itemize}
	\item $||\left(\left(B[\sig]\right)^{\sy},s[\sig]\right)-(B^\star_{\la},s^\star_{\la})||_1>\delta$ and $\big(n_{\ttt}[\sig]\big)_{\ttt\in\FFF_{\tr}}\in \ee_{\frac{1}{4}}$.
	\item $|Y| \in [\eps n/2,2\eps n]$ and $\{\dot{\ttt}_{\sig}(e)\}_{e \in \delta \NN(Y)}$ are disjoint with $v\left(\dot{\ttt}_{\sig}(e)\right)\leq \frac{-4\log\eps}{k\log 2}$ for all $e\in \delta \NN(Y)$. 
	\item $||H^\sm[\GGG,Y,\sig]-\left(H[\sig]\right)^{\textnormal{sy}}||_1\leq \eps^{1/3}$. 
	\end{itemize}
    Here, $\left(B[\sig]\right)^{\sy}$ is defined in Lemma \ref{lem:B:sy}. Furthermore, let $\AAA_1\equiv \AAA_1(\eps)$ be the set of $A_1(\GGG,Y,\sig)$ such that $A_1$ is one-step reachable from some $A_0 \in \AAA_0$. By the reversibility of the Markov chain, stated in Lemma \ref{lem:resampling:reversingmeasure}, we have
	\begin{equation}
	\label{eqn:resampling:keyequation}
	    \mu_{\eps}(\AAA_0)=\sum_{A_0\in \AAA_0}\sum_{A_1\in \AAA_1}\mu_{\eps}(A_0)\pi(A_0,A_1) = \sum_{A_1\in \AAA_1}\sum_{A_0\in \AAA_0}\mu_{\eps}(A_1)\pi(A_1,A_0)\leq \mu_{\eps}(\AAA_1)\max_{A_1\in \AAA_1}\pi(A_1,\AAA_0).
	\end{equation}
	Observe that by Lemma \ref{lem:1stmo:resampling:treedisjoint}, we can lower bound the \textsc{lhs} of the equation above by
	\begin{equation}\label{eq:maxim:lowerbound}
	    \mu_{\eps}(\AAA_0) \gtrsim e^{-C_k n \eps^2 \log(\frac{1}{\eps})}\E \bZ^\tr_{\la}\left[||(B,s)-(B^\star_{\la},s^\star_{\la})||_1>\delta,\quad(n_{\ttt})_{\ttt\in \FFF_{\tr}}\in \ee_{\frac{1}{4}}\right].
	\end{equation}
	Turning to upper bound the \textsc{rhs} of \eqref{eqn:resampling:keyequation}, for $A_1=(\GGG^\prime,Y,\utau)\in \AAA_1$, let $\kappa^\prime\equiv |Y|$. Lemma \ref{lem:resampling:transitionprob} shows
	\begin{equation}\label{eq:maxim:1stmo:upperbound:pi}
	\begin{split}
	    \pi(A_1,\AAA_0)&\leq  \sum_{H^\sm_0\in \bDelta^\sm_{\kappa^\prime}:H^\sm_0=H^\sm[A_0], A_0\in \AAA_0}\exp\left(-\frac{\eps n}{2}\Xi(H^\sm_0)+C_{k,\eps} \log n\right)\\&\leq \exp\left(-\frac{\eps n}{2}\inf_{H^\sm_0=H^\sm[A_0], A_0\in \AAA_0}\Xi(H^\sm_0)+C_{k,\eps}^\prime \log n\right),
	\end{split}
	\end{equation}
	where the last inequality is because $H^\sm_0 \in \bDelta^{\sm,(\frac{-4d\log \eps}{\log 2})}_{\kappa^\prime}$ and $\big|\bDelta^{\sm,(\frac{-4d\log \eps}{\log 2})}_{\kappa^\prime}\big|\leq n^{C_{k,\eps}}$. To this end, we gather the key observations to lower bound $\Xi(H^\sm_0)$.
    \begin{itemize}
	\item Define $\utau=(\tau_1,\ldots \tau_d) \in \Omega^d$ to be free if $\tau_i \in \{\fF\}$ holds for all $1\leq i\leq d$. We show in Lemma \ref{lem:H:dot:to:s:lipschitz} that $|s[\sig]-s^\star_{\la}|\leq \log 2\sum_{\utau \in \Omega^{d}:\textnormal{free}}|\dot{H}(\utau)-H^\star_{\la}(\utau)|$ holds for $\dot{H}=\dot{H}[\sig]$.  Since $B[\sig]$ can be obtained by restriction of $H[\sig]$ onto frozen variables, separating clauses, and the edges adjacent to them, we have for $(\GGG,Y,\sig)\in \AAA_0$,
   \begin{equation}\label{eq:H:sy:distance:H:star:lowerbound}
   \delta< ||(\left(B[\sig]\right)^{\textnormal{sy}},s[\sig])-(B^\star_{\la},s^\star_{\la})||_1\lesssim ||\left(H[\sig]\right)^{\textnormal{sy}}-H^\star_{\la}||_1\,.
   \end{equation}
	\item For $C>0$, define 
	\begin{equation}\label{eq:def:bDelta:exp:decay}
	    \bDelta^{\textnormal{exp}}_{C}\equiv \{H^\sm\in \bDelta^\sm:\sum_{v(\dsigma)\geq L} \dot{h}[H^\sm](\dsigma)\leq 2^{-CkL},\forall L \geq 1\}.
	\end{equation}
	Note that for $(\GGG,Y,\sig)\in \AAA_0$, $\left(H[\sig]\right)^\sy \in \bDelta^{\textnormal{exp}}_{1/5}$ holds, because
	\begin{equation}\label{eq:doth:H:sy:exp:decay}
	\begin{split}
	 \sum_{v(\dot{\tau})\geq L}\dot{h}\left[\left(H[\sig]\right)^{\sy}\right](\dot{\tau})
	 &=\sum_{v(\dot{\tau})\geq L}\frac{1}{d}\sum_{\ttt\in \FFF_{\tr}}p_{\ttt}[\sig]\sum_{e\in E(\ttt)}\one\left\{\dot{\sigma}_e(\ttt)=\dot{\tau}\right\}\\
	&\leq \sum_{\ttt:v(\ttt)\geq L}v(\ttt)p_{\ttt}[\sig]\leq\sum_{v\geq L} v2^{-kv/4}\leq 2L 2^{-kL/4}\leq 2^{-kL/5},
	\end{split}
	\end{equation}
	where $\sigma_{e}(\ttt)$ is defined in \eqref{eq:def:col on freetree}. Note that by Lemma \ref{lem:opt:tree:decay}, $H^\star_{\la}\in \bDelta^{\textnormal{exp}}_{1/5}$ also holds. Since $H^\sm \to \dot{h}[H^\sm]$ is a linear projection, for $H^\sm_0=H^\sm[A_0], A_0\in \AAA_0$,
	\begin{equation*}
	    \sum_{v(\dot{\tau})\geq L}\dot{h}[H^\sm_0](\dot{\tau})\leq \sum_{v(\dot{\tau})\geq L}\dot{h}\left[\left(H[\sig]\right)^{\sy}\right](\dot{\tau}) +\eps^{1/3}\leq 2^{-kL/5}+\eps^{1/3}
	\end{equation*}
	Observe that for $L\leq \frac{-4\log\eps}{k\log 2}$, $2^{-kL/5}+\eps^{1/3}\leq 2^{-kL/20}$ holds, since $L\to 2^{-kL/20}-2^{-kL/5}$ is a decreasing function and $\eps^{1/3}+\eps^{4/5}\leq \eps^{1/5}$ for small enough $\eps$. Also, because $v\left(\dot{\ttt}_{\sig}(e)\right)\leq \frac{-4\log\eps}{k\log 2}$ for all $e\in \delta \NN(Y)$, $\sum_{v(\dot{\tau})\geq L}\dot{h}[H^\sm_0](\dot{\tau})=0$ for $L >\frac{-4\log\eps}{k\log 2}$. Therefore, we conclude that $\left(H[\sig]\right)^\sy,H^\star_{\la}, H^\sm_0\in \bDelta^{\textnormal{exp}}_{1/20}$.
	\item It is straightforward to see that $\bDelta^{\textnormal{exp}}_{1/20}$ is tight and closed, where we endow $\bDelta^{\textnormal{exp}}_{1/20}$ with the topology induced by total variation norm (or equivalently, weak convergence, since $\Omega$ is countable). Thus, $\bDelta^{\textnormal{exp}}_{1/20}$ is compact by Prokhorov's theorem. Hence, by Lemma \ref{lem:1stmo:unique:zero:Xi} and the continuity of $\Xi[H]$ on $\bDelta^{\textnormal{exp}}_{1/20}$, guaranteed by Lemma \ref{lem:Xi:continuous:1stmo} in Appendix \ref{sec:appendix:continuity:tree:optmization}, we have
	\begin{equation}\label{eq:G:positive}
	G(\eta)\equiv \inf\left\{\Xi(H):H\in \bDelta^{\textnormal{exp}}_{1/20},H=H^\sy,||H-H^\star_{\la}||_1\geq\eta \right\}>0 \quad\textnormal{for}\quad \eta>0.
	\end{equation}
	For $(\GGG,Y,\sig)\in \AAA_0$, $\left(H[\sig]\right)^\sy \in \bDelta^{\textnormal{exp}}_{1/20}$, by the previous observation. Thus, \eqref{eq:H:sy:distance:H:star:lowerbound} shows
	\begin{equation}\label{eq:lowerbound:Xi:H:sy}
	    \Xi\left(\left(H[\sig]\right)^\sy\right)\geq G(\delta)
	\end{equation}
	Moreover, since $\Xi(\cdot)$ is continuous on the compact set $\bDelta^{\textnormal{exp}}_{1/20}$, it is uniformly continuous, i.e.
	\begin{equation}\label{eq:Xi:uniform:continous}
	    \lim_{\eta\to 0}f(\eta)=0,\textnormal{ where }f(\eta) \equiv \inf\left\{\big|\Xi(H^\sm_1)-\Xi(H^\sm_2)\big|:H^\sm_1,H^\sm_2\in \bDelta^{\textnormal{exp}}_{1/20}, ||H^\sm_1-H^\sm_2||_1\leq \eta\right\}.
	\end{equation}
	\end{itemize}
	Now, because $||H^\sm_0-\left(H[\sig]\right)^\sy||\leq \eps^{1/3}$ holds for $H^\sm_0=H^\sm[\GGG,Y,\sig], (\GGG,Y,\sig)\in \AAA_0$, we have
	\begin{equation*}
	    \Xi(H^\sm_0)\geq \Xi\left(\left(H[\sig]\right)^\sy\right)-f(\eps^{1/3})\geq G(\delta)-f(\eps^{1/3}),
	\end{equation*}
	where the last inequality is due to \eqref{eq:lowerbound:Xi:H:sy}. Note that we have $G(\delta)>0$ by \eqref{eq:G:positive}. Hence, \eqref{eq:Xi:uniform:continous} shows that $f(\eps^{1/3})\leq G(\delta)/2$ for small enough $\eps$, i.e. $\eps<\eps_0(\delta)$. Therefore, by the above equation and \eqref{eq:maxim:1stmo:upperbound:pi}, for $\eps<\eps_0(\delta)$, we can upper bound
	\begin{equation}\label{eq:maxim:upperbound}
	\max_{A_1\in \AAA_1}\pi(A_1,\AAA_0)\leq \exp\Big\{-\frac{\eps G(\delta)}{4}n+C^\prime_{k,\eps}\log n\Big\}.
	\end{equation}
    Moreover, if we denote by $\rr(\sig)$ (resp. $\ff(\sig)$) the number of red edges (resp. free variables) in a coloring $\sig \in \Omega^E$, we can bound
    \begin{equation}\label{eq:maxim:upperbound:1}
        \mu_{\eps}(\AAA_1)\leq \E\bZ_{\la}^{\tr}+\E\sum_{\sig\in \Omega^E} w_{\GGG}^{\lit}(\sig)^{\la}\mathds{1}\bigg\{ \frac{\rr(\ux)}{nd} \vee \frac{\ff(\ux)}{n}\in \Big( \frac{7}{2^k}\,,\,\frac{7}{2^k}+\eps\Big) \bigg\}\leq \E\bZ_{\la}^{\tr}+e^{-cn}\lesssim_k \E\bZ_{\la}^{\tr}\,,
    \end{equation}
    where the second inequality is due to Lemma \ref{lem:free:red}. The last inequality holds because $\E\bZ_{\la}^{\tr}\asymp_k \E\bZ_{\la}^{\tr}$ (cf. Corollary \ref{cor:cyclic:contribution:1stmo}), and $\E\bZ_{\la}\gtrsim 1$ holds because there is a valid frozen configuration w.h.p. for $\alpha<\alpha_{\textsf{sat}}$. Consequently, reading \eqref{eqn:resampling:keyequation}, \eqref{eq:maxim:lowerbound}, \eqref{eq:maxim:upperbound}, and \eqref{eq:maxim:upperbound:1} altogether, we have
	\begin{equation*}
	\frac{\E \bZ^\tr_{\la}\left[||(B^{\sy},s)-(B^\star_{\la},s^\star_{\la})||_1>\delta,(n_{\ttt})_{\ttt\in \FFF_{\tr}}\in \ee_{\frac{1}{4}}\right]}{\E \bZ_{\la}^\tr} \lesssim \exp\left(-\frac{\eps G(\delta)n}{4}-C_k n \eps^2 \log\eps +C^\prime_{k,\eps}\log n\right).
	\end{equation*}
	Finally, take $\eps>0$ small enough so that $0<-\eps \log \eps<\frac{G(\delta)}{8}$ and $n$ large enough so that $C^\prime_{k,\eps}\log n \ll \frac{\eps G(\delta) n}{8}$ to conclude for some $c(\delta)>0$ and $n\geq n_0(\delta)$,
	\begin{equation*}
	    \E \bZ^\tr_{\la}\left[||(B^{\sy},s)-(B^\star_{\la},s^\star_{\la})||_1>\delta,(n_{\ttt})_{\ttt\in \FFF_{\tr}}\in \ee_{\frac{1}{4}}\right]\leq e^{-c(\delta) n}\E \bZ_{\la}^\tr\,.
	\end{equation*}
Combining with Lemma \ref{lem:B:sy} concludes the proof. 
	\end{proof}

	\subsection{Negative-definiteness of the exponent}\label{subsubsec:resampling:negdef}
	The following proposition is the crux of the proof of Proposition \ref{prop:negdef}.
	\begin{prop}\label{prop:crux:negdef:1stmo}
	For $L\geq L_{0}(k)$, there exist constants $C_1,C_2,C_3>0$, which depend on $k$ only, and $\delta_0(k,L)$, which depends on $k$ and $L$, such that the following holds: consider $B\in \bDelta^{b}$ with $B=B^\sy$ and $D\equiv ||B-B^\star_{\la,L}||_1<\delta_0(k,L)$. For $0<\eps<\eps_0(k,L,D)$, where $D\to \eps_0(k,L,D)$ is non-decreasing, we have
	\begin{equation}\label{eq:prop:crux:negdef:1stmo}
	    F_{\la,L}(B)\leq \max\left\{F_{\la,L}(B^\prime):||B^\prime-B||_1\leq C_1\eps D\right\} -C_2 \eps D^2- C_3\eps^{2}\log\eps
	\end{equation}
	\end{prop}
	The proposition above easily implies Proposition \ref{prop:negdef}.
	\begin{proof}[Proof of Proposition \ref{prop:negdef}]
	The first item is straightforward from Proposition \ref{prop:maxim:1stmo} and Lemma \ref{lem:exist:free:energy:1stmo}. Thus, we aim to prove the second item. 
	
	We first show $\nabla^{2}F_{\la,L}(B^\star_{\la,L}) \prec -\beta(k) I$. To begin with, note that $B\to \Psi_{\circ}(B)$ is strictly concave by its definition in \eqref{eq:def:Psi:circ:p:circ:1stmo}. Also, $\uh(B)=\uh(B^\sy)$ shows $\utheta(B,s)=\utheta(B^\sy,s)$. Hence, we have 
	\begin{equation}\label{eq:neg:def:Psi:circ}
	    F_{\la,L}(B)\leq F_{\la,L}(B^\sy)-\Omega_k(||B-B^\sy||_1^{2}))
	\end{equation}
	Let $\bDelta^{\textnormal{b},\sy}$ be the space of $B\in \bDelta^{\textnormal{b}}$ such that $B=B^\sy$. For $B\in \bDelta^{\textnormal{b},\sy}$ with $D\equiv ||B-B^\star_{\la,L}||_1<\delta_0(k,L)$ and $\eps<\eps_0(k,L,D)$, \eqref{eq:prop:crux:negdef:1stmo} and \eqref{eq:neg:def:Psi:circ} show 
	\begin{equation}\label{eq:crux:negdef:sy:1stmo}
	     F_{\la,L}(B)\leq \max\left\{F_{\la,L}(B^\prime):||B^\prime-B||_1\leq C_1\eps D, B^\prime \in \bDelta^{\textnormal{b},\sy}\right\}-C_2 \eps D^2- C_3\eps^{2}\log\eps.
	\end{equation}
	We proceed by making recursive use of \eqref{eq:crux:negdef:sy:1stmo}: fix $B\in \bDelta^{\textnormal{b},\sy}$ with $D\equiv ||B-B^\star_{\la,L}||_1<\frac{\delta_0(k,L)}{2}$. Suppose at time $t\in \Z_{\geq 0}$, the following holds for some $D_t \leq \frac{D}{2}$ and $a_t>0$:
	\begin{equation}\label{eq:recursive:1}
	       F_{\la,L}(B)\leq \max\left\{F_{\la,L}(B^\prime):||B^\prime-B||_1\leq D_t, B^\prime \in \bDelta^{\textnormal{b},\sy}\right\}-a_t.
	\end{equation}
	Note that for $B^\prime\in \bDelta^{\textnormal{b},\sy}$ with $||B^\prime-B||_1\leq D_t$, the triangle inequality shows $$\frac{D}{2} \leq D-D_t \leq||B^\prime-B^\star_{\la,L}||_1\leq D+D_t\leq \frac{3}{2}D<\delta_0(k,L).$$ Using \eqref{eq:crux:negdef:sy:1stmo} with $\eps=\frac{1}{C_1 K}$, where $K$ is a large enough integer so that $\frac{1}{C_1 K}<\eps_0(k,L,\frac{D}{2})\leq \eps_0(k,L,||B^\prime-B^\star_{\la,L}||_1)$, shows
	\begin{equation}\label{eq:recursive:2}
	    F_{\la,L}(B^\prime)
	    \leq \max \left\{F_{\la,L}(\tilde{B}):||\tilde{B}-B^{\prime}||_1\leq \frac{D+D_t}{K}, \tilde{B}\in \bDelta^{\textnormal{b},\sy}\right\}-\frac{C_2(D-D_t)^2}{C_1 K}
	    +\frac{C_3\log(C_1 K)}{C_1^2K^2}
	\end{equation}
	Hence, we can plug \eqref{eq:recursive:2} into \eqref{eq:recursive:1} to obtain the bound at time $t+1$. The recurrence relation for $(D_t)_{t\geq 0}$ and $(a_t)_{t\geq 0}$ is then given by
	\begin{equation*}
	    D_{t+1}=D_{t}+\frac{D+D_{t}}{K},\quad a_{t+1}=a_{t}+\frac{C_2(D-D_t)^2}{C_1 K}-\frac{C_3\log(C_1 K)}{C_1^2K^2},
	\end{equation*}
	with initial condition $D_{0}=0, a_0=0$. Solving the recurrence relation gives
	\begin{equation*}
	    D_t=\left((1+K^{-1})^{t}-1\right)D, \quad a_t=-\frac{C_3\log(C_1 K)}{C_1^2K^2}t+\sum_{i=0}^{t-1}\frac{\left(2-(1+K^{-1})^{i}\right)^{2}}{K}\frac{C_2}{C_1}D^2.
	\end{equation*}
	The terminal condition $D_t \leq \frac{D}{2}$ shows $t\leq t_{0}(K)\equiv \lfloor K\log \frac{3}{2} \rfloor$. It is straightforward to compute
	\begin{equation*}
	    \lim_{K\to\infty}a_{t_0(K)}=\frac{C_2}{C_1}D^2\lim_{K\to\infty}\sum_{i=0}^{\lfloor K \log \frac{3}{2}\rfloor -1}\frac{\left(2-(1+K^{-1})^{i}\right)^{2}}{K}=\frac{C_2}{C_1}D^2(4\log \frac{3}{2}-\frac{11}{8})
	\end{equation*}
	Consequently, for $B\in \bDelta^{\textnormal{b},\sy}$ with $||B-B^\star_{\la,L}||_1<\frac{\delta_0(k,L)}{2}$, we have 
	\begin{equation*}
	     F_{\la,L}(B)\leq  F_{\la,L}(B^\star_{\la,L})-\frac{C_2(4\log \frac{3}{2}-\frac{11}{8})}{C_1}||B-B^\star_{\la,L}||_1^{2}.
	\end{equation*}
	 For a general $B\in \bDelta^{\textnormal{b}}$ with $||B-B^\star_{\la,L}||_1<\frac{\delta_0(k,L)}{2}$, we can combine the inequality above with \eqref{eq:neg:def:Psi:circ} to show $F_{\la,L}(B)\leq  F_{\la,L}(B^\star_{\la,L})-C_k||B-B^\star_{\la,L}||_1^{2}$ for some $C_k>0$, since $C_1,C_2$ do not depend on $L$. Therefore, we conclude that $\nabla^{2}F_{\la,L}(B^\star_{\la,L})\prec -\beta(k) I$ holds for some $\beta(k)>0$.
	 
	Next, we prove $\nabla^{2}_{B}F_{\la,L}(B^\star_{\la,L},s^\star_{\la,L})\prec -\beta(k)I$. Recalling Remark \ref{rem:truncated:negdef:SSZ}, it is straightforward to see from $\E \bZ^{(L),\tr}_{\la}[B]=\sum_{0\leq s \leq \log 2}\E \bZ^{(L),\tr}_{\la,s}[B]$ that the following holds.
	\begin{equation*}
	    F_{\la,L}(B) =\max_{0\leq s \leq \log 2} F_{\la,L}(B,s)
	\end{equation*}
	Subsequently, for $B \in \bDelta^{\textnormal{b}}$ with $||B-B^\star_{\la,L}||_1<\frac{\delta_0(k,L)}{2}$, we have
	\begin{equation*}
	    F_{\la,L}(B,s^\star_{\la,L})\leq F_{\la,L}(B)\leq F_{\la,L}(B^\star_{\la,L})-C_k||B-B^\star_{\la,L}||_1^{2}=F_{\la,L}(B^\star_{\la,L},s^\star_{\la,L})-C_k||B-B^\star_{\la,L}||_1^{2},
	\end{equation*}
	for some $C_k>0$, depending only on $k$. Therefore, $\nabla^{2}_{B} F_{\la,L} (B^\star_{\la,L},s^\star_{\la,L})\prec -\beta(k)I$ holds for some $\beta(k)>0$.
	\end{proof}
	We now aim to prove Proposition \ref{prop:crux:negdef:1stmo}. The first step is to define the set of appropriate initial configurations for the $L$-truncated resampling Markov chain: given $B\in \BB_{\la}^{-}(\delta_0)$ with $B=B^\sy$ and $\eps>0$, let $\AAA_0^\prime \equiv \AAA_0^\prime(B, \eps,L)$ be the set of $(\GGG,Y,\sig)$ satisfying the following $4$ conditions.
	\begin{itemize}
	    \item $\sig \in \Omega_L^{E}$ and $B[\sig]=\proj (B)$.
	    \item $|p_{\ttt}[\sig]-p_{\ttt,\la,L}(B)| \leq n^{-1/3}$ for all $\ttt\in \FFF_{\tr}$ with $v(\ttt)\leq L$, where $p_{\ttt,\la,L}(B)$ is defined in \eqref{eq:opt:tree:prob:B:s}.
	    \item $|Y| \in [\eps n/2,2\eps n]$ and $\{\dot{\ttt}_{\sig}(e)\}_{e \in \delta \NN(Y)}$ are disjoint. 
	\item $||H^\sm[\GGG,Y,\sig]-\left(H[\sig]\right)^{\textnormal{sy}}||_1\leq \eps^{1/3}$. 
	\end{itemize}
	Also, let $\AAA_1^\prime\equiv \AAA_1^\prime(B, \eps,L)$ be the set of $A_1=(\GGG^\prime,Y,\utau)$ satisfying the following $2$ conditions.
	\begin{itemize}
	    \item $A_1$ is one-step approachable from some $A_0 \in \AAA_0^\prime$ by the $L$-truncated resampling Markov chain.
	    \item Denote $H^\sm_{1}=H^\sm(A_1)$ and $\nu = \nu^\op_{L}\left[\dot{h}[H^\sm_{1}]\right]$, where $\nu^\op_{L}[\cdot]$ is defined in \eqref{eq:def:nu:q:dot}. Then, $||H^\sm_{1}-H^\tr[\nu]||_1\leq \eps^{1/3}$.
	\end{itemize}
	Below are the lemmas regarding the properties of $\AAA_0^\prime$ and $\AAA_1^\prime$, which play crucial roles in the proof of Proposition \ref{prop:crux:negdef:1stmo}.
	\begin{lemma}\label{lem:1stmo:negdef:initialstate}
	Consider $B \in \BB_{\la}^{-}(\delta_0)$, $\eps\in(0,\frac{1}{2})$, and the set $\AAA_0^\prime=\AAA_0^\prime(B,\eps, L)$ defined above. For a constant $C_k>0$ and $n\geq n_0(k,B,\eps,L)$, we have
	\begin{equation*}
	    \mu_{\eps}(\AAA_0^\prime) \gtrsim \exp\left(nF_{\la,L}(B)-C_k n\eps^{2}\log\left(\frac{1}{\eps}\right)\right).
	\end{equation*}
	\end{lemma}
	\begin{proof}
	Recall Remark \ref{rem:delta:smaller} that $\sum_{v(\ttt)=v}p_{\ttt,\la,L}(B)\leq 2^{-kv/3}$ holds for $B \in \BB_{\la}^{-}(\delta_0)$. Thus, if $(\GGG,Y,\sig) \in \AAA_0^\prime$, the number of free trees in $\sig$ with $v$ variables for $v<L$ can be bounded by 
	\begin{equation*}
	    \sum_{v(\ttt)=v} p_{\ttt}[\sig] \leq \sum_{v(\ttt)=v} p_{\ttt,\la,L}(B)+n^{-1/3}|\{\ttt\in \FFF_{\tr}: v(\ttt)=v\}|\leq 2^{-kv/3}+n^{-1/3}C_{k,L}.
	\end{equation*}
	Hence, for large enough $n$, $\left(n_{\ttt}[\sig]\right)_{\ttt\in \FFF_{\tr}}\in \ee_{\frac{1}{4}}$ holds, so we can use Lemma \ref{lem:1stmo:resampling:bad:variables} and Lemma \ref{lem:1stmo:resampling:treedisjoint} to show that for $\eps<\frac{1}{2}$,
	\begin{equation*}
	    \mu_{\eps}(\AAA_0^\prime) \gtrsim e^{-C_k n\eps^{2}\log(\frac{1}{\eps})}\E\bZ_{\la,s}^{(L),\tr}\left[B_n, |p_{\ttt}[\sig]-p_{\ttt,\la,L}(B)| \leq n^{-1/3}, \forall\ttt\in \FFF_{\tr}\right],
	\end{equation*}
	where $B_n \equiv \proj(B)$. Hence, if suffices to show the following:
	\begin{equation}\label{eq:1stmo:tree:prob:concentration}
	    \E\bZ_{\la,s}^{(L),\tr}\left[B_n, |p_{\ttt}[\sig]-p_{\ttt,\la,L}(B)| \leq n^{-1/3}, \forall\ttt\in \FFF_{\tr}\right] \geq \exp \left(nF_{\la,L}(B)+O_k(\log n)\right)
	\end{equation}
	The proof of \eqref{eq:1stmo:tree:prob:concentration} is close in spirit to the proof of Lemma \ref{lem:exist:free:energy:1stmo}: it is straightforward to compute
	\begin{multline}\label{eq:1stmo:tree:prob:concentration:inter}
	\E\bZ_{\la,s}^{(L),\tr}\left[B_n, |p_{\ttt}[\sig]-p_{\ttt,\la,L}(B,s)| \leq n^{-1/3}, \forall\ttt\in \FFF_{\tr}\right]\\
	=\exp\left(nF_{\la,L}(B,s)+O_k(\log n)\right) \P_{\utheta_{\la,L}(B,s)}(\EEE_{\uh(B),s}),
	\end{multline}
	where $\P_{\utheta_{\la,L}(B,s)}$ is taken with respect to i.i.d. random free trees $X_1,...,X_{nh_{\circ}(B_n)}$ with distribution
	\begin{equation*}
	\P_{\utheta_{\la,L}(B,s)}(X_i=\ttt)\equiv \frac{J_\ttt w_{\ttt}^{\la}\exp\big(\langle \utheta_{\la,L}(B,s), \boeta_{\ttt}\rangle\big)}{h_{\circ}(B)}\one\{v(\ttt)\leq L\}=\frac{p_{\ttt,\la,L}(B,s)}{h_{\circ}(B)}\one\{v(\ttt)\leq L\}
	\end{equation*}
	Also, recalling the event $\AAA_{\uh(B),s}$ defined in \eqref{eq:def:lclt:event:B:s}, the event $\EEE_{\uh(B),s}$ is defined by
	\begin{equation*}
	    \EEE_{\uh(B),s}=\AAA_{\uh(B),s} \bigcap\Bigg\{\bigg|\frac{1}{n}\sum_{i=1}^{nh_{\circ}(B_n)}\one\{X_i=\ttt\}-p_{\ttt,\la,L}(B,s)\bigg|\leq n^{-1/3}, \textnormal{ for all }\ttt\in\FFF_{\tr}\textnormal{ with }v(\ttt)\leq L \Bigg\}.
	\end{equation*}
	For each $\ttt\in \FFF_{\tr}$ with $v(\ttt)\leq L$, observe that Hoeffding's inequality gives
	\begin{equation*}
	    \P_{\utheta_{\la,L}(B,s)}\Bigg(\bigg|\frac{1}{n}\sum_{i=1}^{nh_{\circ}(B_n)}\one\{X_i=\ttt\}-p_{\ttt,\la,L}(B,s)\bigg|> n^{-1/3}\Bigg)\leq \exp\left(-\Omega_{k}(n^{1/3})\right).
	\end{equation*}
	Also, we have that $\P_{\utheta_{\la,L}(B,s)}(\AAA_{\uh(B),s})=\Omega_{k}(n^{-(|\partial|+1)/2})$ from local CLT, so union bound shows
	\begin{equation*}
	\P_{\utheta_{\la,L}(B,s)}(\EEE_{\uh(B),s})\geq C_kn^{-(|\partial|+1)/2}-C_{k,L}\exp\{-\Omega_{k}(n^{1/3})\},
	\end{equation*}
	where $C_k$ depends on $k$ and $C_{k,L}$ depends on $k$ and $L$. Hence, $\P_{\utheta_{\la,L}(B,s)}(\EEE_{\uh(B),s})=\Omega_{k}(n^{-(|\partial|+1)/2})$ holds, and plugging into \eqref{eq:1stmo:tree:prob:concentration:inter} finishes the proof of \eqref{eq:1stmo:tree:prob:concentration}.
	\end{proof}
	For $B \in \BB_{\la}^{-}(\delta_0)$, define $\dot{h}^\op\equiv \dot{h}^\op_{\la,L}[B]\in \PPP(\dot{\Omega})$ by
	\begin{equation}\label{eq:def:h:op:B}
	\dot{h}^\op(\dsigma) \equiv
	\begin{cases}
	\bar{B}(\dsigma) & \dsigma \in \{\rr,\bb\}\\
	\frac{1}{d}\sum_{\ttt: v(\ttt)\leq L}p_{\ttt,\la,L}(B)\sum_{e\in E(\ttt)}\one\{\dsigma_e(\ttt)=\dsigma\}&\dsigma \in \{\ff\}
	\end{cases}
	\end{equation}
	where $p_{\ttt,\la,L}(B)$ is defined in \eqref{eq:opt:tree:prob:B:s} and $\sigma_e(\ttt)$ is defined in \eqref{eq:def:col on freetree}. The properties of $\dot{h}^\op_{\la,L}[B]$ and its connection with $\AAA_0^\prime$ are summarized by the following lemma.
	\begin{lemma}\label{lem:cont:h:op}
	For $\dot{h}^\op_{\la,L}[B], B\in \BB_{\la}^{-}(\delta_0)$, defined in \eqref{eq:def:h:op:B}, the following holds.
	\begin{enumerate}
	    \item $\dot{h}^\op_{\la,L}[B^\star_{\la,L}]=h^\star_{\la,L}$.
	    \item There exists $\delta_L>0$ and $C_k>0$ such that
	 \begin{equation}\label{eq:continuity:h:op}
	||B-B^\star_{\la,L}||_1<\delta_L \implies ||\dot{h}^\op_{\la,L}[B]-\dot{h}^\star_{\la,L}||_1\leq C_k||B-B^\star_{\la,L}||_1.
	\end{equation}
	\item There exists a constant $C_{k,L}$ such that for $(\GGG,Y,\sig)\in \AAA_0^\prime(B,\eps,L)$, 
	\begin{equation}\label{eq:dot:h:close:h:op}
	    ||\dot{h}\left[\left(H[\sig]\right)^\sy\right]-\dot{h}^\op_{\la,L}[B]||_1\leq C_{k,L} n^{-1/3}.
	\end{equation}
	Thus, $ ||\dot{h}\left[H^\sm[\GGG,Y,\sig]\right]-\dot{h}^\op_{\la,L}[B]||_1\leq \eps^{1/3}+C_{k,L} n^{-1/3}$ holds for $(\GGG,Y,\sig)\in \AAA_0^\prime(B,\eps,L)$.
	\end{enumerate}
	\end{lemma}
	\begin{proof}
	The proof of the first item is deferred to Lemma \ref{lem:1stmo:BP:compatibility} and we only prove the second item and third item. For the proof of second item, we can use the triangle inequality to bound
	\begin{multline}
	    ||\dot{h}^\op_{\la,L}[B]-\dot{h}^\op[B^\star_{\la,L}]||_1
	    \leq ||B-B^\star_{\la,L}||_1+\sum_{v=1}^{L}\sum_{\ttt:v(\ttt)=v}\frac{|E(\ttt)|}{d}\big|p^\star_{\ttt,\la,L}-p_{\ttt,\la,L}(B)\big|\\
	    =||B-B^\star_{\la,L}||_1+\sum_{v=1}^{L}\sum_{\ttt:v(\ttt)=v}\frac{|E(\ttt)|}{d}p_{\ttt,\la,L}^\star\Big|\exp\Big(\big\langle \utheta_{\la,L}^{-}(B)-\utheta_{\la,L}^{\star,-}, \boeta^{-}_{\ttt} \big\rangle\Big)-1\Big|.
	\end{multline}
	Note that $|E(\ttt)|\leq dv(\ttt)$ and Lemma \ref{lem:conv:psi:theta} shows $||\utheta_{\la,L}^{-}(B)-\utheta_{\la,L}^{\star,-}||_1\leq C_k ||B-B^\star_{\la,L}||_1$, for some $C_k$, which does not depend on $L$. Hence, by taking $\delta_L$ small enough, we have
	\begin{equation*}
	\begin{split}
	    ||\dot{h}^\op_{\la,L}[B]-\dot{h}^\star_{\la,L}||_1&\lesssim_k ||B-B^\star_{\la,L}||_1\sum_{v=1}^{L}\sum_{\ttt:v(\ttt)=v} v\Big(\max_{x\in \{\circ,\bb_0,\bb_1\}}\eta_{\ttt}(x)\Big)p_{\ttt,\la,L}^{\star}\\
	    &\lesssim_k||B-B^\star_{\la,L}||_1\sum_{v=1}^{L}v^2 2^{-kv/2}\lesssim_k||B-B^\star_{\la,L}||_1 
	\end{split}
	\end{equation*}
	where the second inequality is due to Lemma \ref{lem:opt:tree:decay}.
	
	Turning to prove the third item, first observe that for a valid coloring $\sig \in \Omega^{E}$, $\dot{h}=\dot{h}\left[\left(H[\sig]\right)^\sy\right]$ can be expressed in a way which resembles \eqref{eq:def:h:op:B}:
	\begin{equation}\label{eq:express:dot:h}
	\dot{h}(\dot{\tau})=
	\begin{cases}
	\bar{B}[\sig](\dot{\tau}) & \dot{\tau} \in \{\rr,\bb\}\\
	\frac{1}{d}\sum_{\ttt:v(\ttt)\leq L}p_{\ttt}[\sig]\sum_{e\in E(\ttt)}\one\{\dot{\sigma}_e(\ttt)=\dot{\tau}\} &\dot{\tau}\in \{\ff\}
	\end{cases}
	\end{equation}	
	Thus, for $(\GGG,Y,\sig) \in \AAA_0^\prime$, using the triangle inequality shows
	\begin{multline}
	||\dot{h}\left[\left(H[\sig]\right)^\sy\right]-\dot{h}^\op_{\la,L}[B]||_1 \leq \frac{1}{n}+\frac{1}{d}\sum_{\dsigma \in \dot{\Omega}}\sum_{\ttt:v(\ttt)\leq L}\big|p_{\ttt,\la,L}(B)-p_{\ttt}[\sig]\big|\sum_{e\in E(\ttt)}\one\{\dsigma_e(\ttt)=\dsigma\}\leq C_{k,L} n^{-1/3},
	\end{multline}
	where the last inequality holds by the definition of $\AAA_0^\prime$, and the constant $C_{k,L}$ can be taken to be $C_{k,L}\equiv L|\{\ttt:v(\ttt)\leq L\}|+1$. The last conclusion is because $||\dot{h}\left[H^\sm[\GGG,Y,\sig]\right]-\dot{h}\left[\left(H[\sig]\right)^\sy\right] \leq \eps^{1/3}$ holds, and $H^\sm \to \dot{h}[H^\sm]$ is a linear projection.
	\end{proof}
	The next lemma gives quantitative estimates on the distance taken by the $L$-truncated resampling Markov chain starting from $\AAA_0^\prime$. 
	\begin{lemma}\label{lem:1stmo:negdef:onestep}
	There exist constants $\delta_0(k,L),\eps_0(k,L)>0$, and $C_k>0$ such that the following holds: let $B\in \bDelta^{\textnormal{b}}$ with $B=B^\sy$ and $D\equiv ||B-B^\star_{\la,L}||_1\leq \delta_0(k,L)$. Also, consider $\eps \leq\eps_0(k,L)$, and the set $\AAA_1^\prime=\AAA_1^\prime(B,\eps, L)$ defined above. Then, for large enough $n$, i.e. $n\geq n_0(k,B,\eps,L)$,
	\begin{equation}\label{eq:lem:1stmo:negdef:onestep}
	   \AAA_1^\prime \subset \{(\GGG^\prime,Y,\utau):||B[\utau]-B||_1\leq C_k\eps(D+\eps^{1/3})\}
	\end{equation}
	\end{lemma}
	\begin{proof}
	For $H^\sm \in \bDelta^\sm$, let $B\equiv B[H^\sm]$ denote the empirical measure of the boundary spins induced by $H^\sm$, similar to the equation \eqref{eq:def:optimal:bdry}:
	\begin{equation}\label{eq:def:B:of:H:sm}
	 \begin{split}
	     &\dot{B}(\sig) \equiv  \dot{H}^\sm(\sig)\quad\textnormal{for}\quad \sig \in \dot{\partial}^{d}\\
	     &\hat{B}(\sig) \equiv \sum_{\utau \in \Omega^k, \utau_{\fs}=\sig} \hat{H}^\sm(\utau)\quad\textnormal{for}\quad \sig \in \hat{\partial}^{k}
	     \\
	     &\bar{B}(\sigma) \equiv \sum_{\tau \in\Omega, \tau_{\fs}= \sigma} \bar{H}^\sm(\tau)\quad\textnormal{for}\quad \sigma \in \hat{\partial}.
	 \end{split}
	 \end{equation}
	Note that $\bar{B}\left[H^\sm[\GGG,Y,\sig]\right]$ does not count the boundary spins at $\delta \NN(Y)$, but the empirical meausure of the boundary spins at $\delta \NN(Y)$ can be obtained by a linear projection of $\hat{B}\left[H^\sm[\GGG,Y,\sig]\right]$.
	
	Now, suppose $(\GGG^\prime, Y,\utau)\in \AAA_1^\prime$ is one-step approachable from $(\GGG,Y,\sig)\in \AAA_0^\prime$. For simplicity, from now and onwards, we abbreviate $H^\sm_0\equiv H^\sm[\GGG,Y,\sig]$ and $H^\sm_1 \equiv H^\sm[\GGG,Y,\utau]$. By definition of the resampling Markov chain, the changes in the boundary spins in $\sig$ (i.e. $(\sigma_e)_{\sigma_e\in \{\rr,\bb,\fs\}}$) to the boundary spins in $\utau$ occur only at $\NN(Y)$. Hence, the change from $(n\dot{B}[\sig], m\hat{B}[\sig], nd\bar{B}[\sig])$ to $(n\dot{B}[\utau], m\hat{B}[\utau], nd\bar{B}[\utau])$ can be obtained by a linear projection of the change from $\kappa B[H^\sm_0]$ to $\kappa B[H^\sm_1]$, where $\kappa\equiv |Y|\leq 2\eps n$. Hence, we have
	\begin{equation*}
	||B[\utau]-\proj(B)||_1=||B[\utau]-B[\sig]||_1\lesssim_k \eps||B[H^\sm_1]-B[H^\sm_0]||_1.
	\end{equation*}
	By definition of $\AAA_0^\prime$ and $\AAA_1^\prime$, $||H^\sm_0-\left(H[\sig]\right)^{\sy}||_1\leq \eps^{1/3}$ and $||H^\sm_1-H^\tr[\nu]||_1\leq \eps^{1/3}$ hold, where $\nu=\nu_L^\op\left[\dot{h}\left[H^\sm_1\right]\right]=\nu_L^\op\left[\dot{h}\left[H^\sm_0\right]\right]$. Also, $H^\sm \to B[H^\sm]$ is a linear projection with $B\left[\left(H[\sig]\right)^{\sy}\right]=B^\sy=B$, so we can further bound the \textsc{rhs} of the equation above by
	\begin{equation}\label{eq:upperbound:B:tau:minus:B}
	 ||B[\utau]-\proj(B)||_1\lesssim_k \eps^{4/3}+\eps||B-B\left[H^\tr[\nu]\right]||_1\leq \eps^{4/3}+\eps D+\eps||B\left[H^\tr[\nu]\right]-B^\star_{\la,L}||_1.
	\end{equation}
	We now aim to upper bound $||B\left[H^\tr[\nu]\right]-B^\star_{\la,L}||_1$: first, note that $B^\star_{\la,L}=B[H^\star_{\la,L}]=B\left[H^\tr[\nu^\star]\right]$, where $\nu^\star\equiv \nu_{\dot{q}^\star_{\la,L}}$. Also, $\mu \to H^\tr[\mu]\to B\left[H^\tr[\mu]\right]$ is a linear projection, so 
	\begin{equation}\label{eq:upperbound:B:nu:minus:B:star}
	    ||B\left[H^\tr[\nu]\right]-B^\star_{\la,L}||_1\lesssim ||\nu-\nu^\star||_1=||\nu_{\dot{q}_0}-\nu_{\dot{q}^\star_{\la,L}}||_1,
	\end{equation}
	where $\dot{q}_0\equiv \dot{q}_L\left[\dot{h}[H^\sm_0]\right]$. Moreover, Lemma \ref{lem:cont:h:op} shows that $\dot{h}[H^\sm_0]$ and $\dot{h}^\star_{\la,L}$ are close:
	\begin{equation}\label{eq:dot:h:close:h:star}
	    ||\dot{h}[H^\sm_0]-\dot{h}^\star_{\la,L}||_1\lesssim_{k}\eps^{1/3}+C_{k,L} n^{-1/3}+D.
	\end{equation}
	Hence, by taking $\eps_0(k,L),\delta_0(k,L)$ small enough and $n_0(k,B,\eps,L)$ large enough, we can guarantee that for all $\eps\leq\eps_0(k,L),\delta\leq\delta_0(k,L)$ and $n\geq n_0(k,B,\eps,L)$, $H^\sm_0$ satisfy the following $2$ conditions:
	\begin{itemize}
	    \item $||\dot{h}[H^\sm_0]-\dot{h}^\star_{\la,L}||_1<\eps_L$, where $\eps_L$ is the constant given in Proposition \ref{prop:1stmo:Lipschitz:hdot:qdot}. Thus, we have $||\dot{q}_0-\dot{q}^\star_{\la,L}||_1\lesssim_k ||\dot{h}[H^\sm_0]-\dot{h}^\star_{\la,L}||_1$.
	    \item $\dot{q}_0$ satisfy the bound $\dot{q}_0(\bb) \geq \frac{1}{2}-\frac{C}{2^k}$, where $C$ is a universal constant, so that the conclusions of Lemma \ref{lem:nu:q:continuous:1stmo} holds. In particular, we have $||\nu_{\dot{q}_0}-\nu_{\dot{q}^\star_{\la,L}}||_1\lesssim_{k}||\dot{q}_0-\dot{q}^\star_{\la,L}||_1$
	\end{itemize}
	With the $2$ conditions above, it is straightforward to bound
	\begin{equation}\label{eq:upperbound:nu:q:minus:q:star}
	||\nu_{\dot{q}_0}-\nu_{\dot{q}^\star_{\la,L}}||_1\lesssim_{k}||\dot{q}_0-\dot{q}^\star_{\la,L}||_1\lesssim_k ||\dot{h}[H^\sm_0]-\dot{h}^\star_{\la,L}||_1\lesssim_k \eps^{1/3}+C_{k,L} n^{-1/3}+D,
	\end{equation}
	where the last inequality is by \eqref{eq:dot:h:close:h:star}. Therefore, for large enough $n$, we conclude from \eqref{eq:upperbound:B:tau:minus:B}, \eqref{eq:upperbound:B:nu:minus:B:star} and \eqref{eq:upperbound:nu:q:minus:q:star} that $||B[\tau]-B||_1\leq C_k\eps(D+\eps^{1/3})$ holds, which concludes the proof.
	\end{proof}
	\begin{proof}[Proof of Proposition \ref{prop:crux:negdef:1stmo}]
	Fix $B\in \bDelta^{\textnormal{b}},B=B^\sy$ such that $0<D\equiv ||B-B^\star_{\la,L}||_1<\delta_0(k,L)$, and $\eps>0$ such that $\eps< \eps_0(k,L,D)$, where we take $\delta_0(k,L)$ and $\eps_0(k,L,D)$ so that for all such $B$ and $\eps$, the following conditions are satisfied:
	\begin{itemize}
	    \item By Lemma \ref{lem:cont:h:op}, we can ensure that for $(\GGG,Y,\sig)\in \AAA_0^\prime(B,\eps,L)$,
	    \begin{equation*}
	        ||\dot{h}\left[\left(H[\sig]\right)^\sy\right]-\dot{h}^\star_{\la,L}||_1\lesssim_k C_{k,L}n^{-1/3}+D\leq C_{k,L}n^{-1/3}+\delta_0(k,L).
	    \end{equation*}
	    Thus, by taking $\delta_0(k,L)$ small enough, $ ||\dot{h}\left[\left(H[\sig]\right)^\sy\right]-\dot{h}^\star_{\la,L}||_1<\eps_L$ holds for large enough $n$, where $\eps_L$ is the constant from Lemma \ref{lem:quadratic:growth:Xi:1stmo}. Therefore, Lemma \ref{lem:quadratic:growth:Xi:1stmo} gives 
	    \begin{equation}\label{eq:lowerbound:Xi:L:H:sy}
	        \Xi_L\left[\left(H[\sig]\right)^\sy\right]\gtrsim_{k}||\left(H[\sig]\right)^\sy-H^\star_{\la,L}||_1^{2}\geq ||\left(\proj(B)\right)^{\sy}-B^\star_{\la,L}||_1^{2}.
	    \end{equation}
	    The last inequality above is because $B\left[\left(H[\sig]\right)^\sy\right]=\left(\proj(B)\right)^{\sy}, B[H^\star_{\la,L}]=B^\star_{\la,L}$ and $H^\sm \to B[H^\sm]$ is a linear projection, where $B[H^\sm]$ is defined in \eqref{eq:def:B:of:H:sm}.
	    \item Note that $H^\sm\in \bDelta^{\sm,(L)} \to \Xi_L(H^\sm)$ is uniformly continuous, since $\bDelta^{\sm,(L)}$ is compact and it is continuous from the definition, i.e. 
	    \begin{equation*}
	        \lim_{\eta \to 0} g(\eta) \equiv \lim_{\eta \to 0}\sup_{||H_1-H_2||<\eta}\Big|\Xi_L(H_1)-\Xi_L(H_2)\Big|=0. 
	    \end{equation*}
	    Therefore, by taking $\eps_0(k,L,D)$ small enough, we can guarantee the following for $(\GGG,Y,\sig)\in \AAA_0^\prime$ with $H^\sm_0\equiv H^\sm[\GGG,Y,\sig]$:
	    \begin{equation}\label{eq:negdef:lowerbound:Xi:H:sm}
	    \Xi_L(H^\sm_0)\geq \Xi_L(\left(H[\sig]\right)^\sy)-g(\eps^{1/3})\geq C_k D^{2}-g(\eps^{1/3})\geq \frac{C_k}{2}D^{2},
	    \end{equation}
	    where the second inequality is due to \eqref{eq:lowerbound:Xi:L:H:sy} and the fact that $B=B^\sy$.
	    \item The conclusions of Lemma \ref{lem:1stmo:negdef:initialstate} and Lemma \ref{lem:1stmo:negdef:onestep} hold for large enough $n$. In particular, we take $\eps_0(k,L,D)\leq D^{3}$, so that \eqref{eq:lem:1stmo:negdef:onestep} implies that for a constant $C_1$, depending only on $k$,
	   \begin{equation}\label{eq:onestep:bound:final}
	       \AAA_1^\prime\subset \{(\GGG^\prime,Y,\utau):||B[\utau]-B||_1\leq C_1 \eps D\}.
	   \end{equation}
	\end{itemize}
	It is evident that for a fixed value of $k$ and $L$, if we define $\epsilon_0(k,L,D)$ as the largest value among those that satisfy all three conditions outlined above, then the function $D\to \epsilon_0(k,L,D)$ is non-decreasing. Also, we remark that we will eventually send $n\to \infty$, fixing $B$ and $\eps$. Note that the reversibility of Markov chain, guaranteed by Lemma \ref{lem:resampling:reversingmeasure}, shows
	\begin{equation}\label{eq:1stmo:resampling:negdef:key}
	    \mu_{\eps}(\AAA_0^\prime) \min_{A_0 \in \AAA_0^\prime}\pi_L(A_0,\AAA_1^\prime)\leq  \mu_{\eps}(\AAA_1^\prime)\max_{A_1\in \AAA_1^\prime}\pi_L(A_1,\AAA_0^\prime).
	\end{equation}
	To this end, we will first lower bound the \textsc{lhs} of the equation above and then upper bound the \textsc{rhs}. First, fix $A_0=(\GGG,Y,\sig)\in \AAA_0^\prime$ and denote $\kappa =|Y|\in [\eps n/2,2\eps n]$, and $\nu=\nu_L^\op\left[\dot{h}[H^\sm(A_0)]\right]$. Then, Lemma \ref{lem:resampling:transitionprob} shows
	\begin{equation}\label{eq:negdef:lowerbound:transitionprob}
	    \pi_L(A_0,\AAA_1^\prime)=1-\sum_{A_1\notin \AAA_1^\prime}\pi_L(A_0,A_1)\geq 1-\sum_{\substack{H^\sm \in \bDelta^{\sm,(L)}_{\kappa}\\||H^\sm-H^\tr[\nu]||_1\geq \eps^{1/3}}}\exp\left(-\frac{\eps n}{2}\Xi_L(H^\sm)+C_{k,L} \log n\right).
	\end{equation}
	For $H^\sm \in \bDelta^{\sm,(L)}_{\kappa}$, let $\mu=\mu[H^\sm]\in \PPP(\Omega_{D})$ be the unique maximizer of the optimization regarding $\Lambda(H^\sm)$ in \eqref{eq:express:Lambda}. Then, $||\mu-\nu||_1\gtrsim ||H^\sm-H^\tr[\nu]||_1$ holds, since $H^\tr[\cdot]$ is a linear projection and $H^\tr[\mu]=H^\sm$. Thus, if $||H^\sm-H^\tr[\nu]||_1\geq \eps^{1/3}$, we have
	\begin{equation}\label{eq:lowerbound:Xi:L}
	    \Xi_L(H^\sm)=\DD_{\textnormal{KL}}(\mu \mid \nu)\gtrsim ||\mu-\nu||_1^{2}\gtrsim \eps^{2/3}
	\end{equation}
	Thus, we can plug the bound \eqref{eq:lowerbound:Xi:L} into \eqref{eq:negdef:lowerbound:transitionprob} to have
	\begin{equation*}
	      \pi_L(A_0,\AAA_1^\prime)\geq 1-\exp\{-C\eps^{5/3}n+C_{k,L}\log n\}|\bDelta^{\sm,(L)}_{\kappa}|= 1-\exp\{-C\eps^{5/3}n+C_{k,L}^\prime\log n\},
	\end{equation*}
	where $C>0$ is an absolute constant and $C_{k,L}^\prime$ depends on $k$ and $L$ only. Therefore, together with Lemma \ref{lem:1stmo:negdef:initialstate}, we can lower bound the \textsc{lhs} of \eqref{eq:1stmo:resampling:negdef:key} for large enough $n$ as follows.
	\begin{equation}\label{eq:negdef:lowerbound:final}
	     \mu_{\eps}(\AAA_0^\prime) \min_{A_0 \in \AAA_0^\prime}\pi_L(A_0,\AAA_1^\prime)\gtrsim \exp\left(nF_{\la,L}(B)-C_k n\eps^{2}\log\left(\frac{1}{\eps}\right)\right)\Big(1-\exp\big(-C\eps^{5/3}n+C_{k,L}^\prime\log n\big)\Big).
	\end{equation}
	We now turn to upper bounding the \textsc{rhs} of \eqref{eq:1stmo:resampling:negdef:key}. Fix $A_1=(\GGG^\prime,Y,\utau)\in \AAA_1^\prime$ with $\kappa^\prime =|Y|$ and recall that $\Xi_{L}[H^\sm_0] \gtrsim_{k}D^{2}$ holds for $H^\sm_0=H^\sm[\GGG,Y,\sig], (\GGG,Y,\sig)\in \AAA_0^\prime$ by \eqref{eq:negdef:lowerbound:Xi:H:sm}. Thus, by Lemma \ref{lem:resampling:transitionprob},
	\begin{equation*}
	    \pi_L(A_1,\AAA_0^\prime) \leq |\bDelta^{\sm,(L)}_{\kappa^\prime}|\exp\big(-C_2 \eps n D^{2}+C_{k,L}\log n\big)=\exp\big(-C_2 \eps n D^{2}+C_{k,L}^\prime\log n\big),
	\end{equation*}
	where $C_2>0$ only depends on $k$. Therefore, by \eqref{eq:onestep:bound:final}, we can upper bound
	\begin{equation}\label{eq:negdef:upperbound:final}
	\begin{split}
	    \mu_{\eps}(\AAA_1^\prime)\max_{A_1\in \AAA_1^\prime}\pi_L(A_1,\AAA_0^\prime) 
	    &\leq \exp\big(-C_2 \eps n D^{2}+C_{k,L}^\prime\log n\big)\sum_{B^\prime \in \bDelta^{\textnormal{b}}_n:||B^\prime-B||_1\leq C_1 \eps D}\E\bZ_{\la}^{(L),\tr}[B^\prime]\\
	    &\leq \exp\left(n\left(\max_{||B^\prime-B||_1\leq C_1 \eps D}F_{\la,L}(B^\prime)-C_2\eps  D^2\right)+C_{k,L}^{\prime\prime}\log n\right),
	\end{split}
	\end{equation}
	where the last inequality is due to Lemma \ref{lem:exist:free:energy:1stmo}. Finally, reading \eqref{eq:1stmo:resampling:negdef:key}, \eqref{eq:negdef:lowerbound:final} and \eqref{eq:negdef:upperbound:final} together shows that for large enough $n$,
	\begin{equation*}
	\exp\left\{nF_{\la,L}(B)-C_k n\eps^{2}\log\left(\frac{1}{\eps}\right)\right\}\lesssim \exp\left\{n\left(\max_{||B^\prime-B||_1\leq C_1 \eps D}F_{\la,L}(B^\prime)-C_2\eps  D^2\right)+C_{k,L}^{\prime\prime}\log n\right\}.
	\end{equation*}
	 Taking $\frac{1}{n}\log$ on both sides of the inequality above and sending $n\to \infty$  finishes the proof of our goal \eqref{eq:prop:crux:negdef:1stmo}, since $C_1,C_2,C_k$ only depends on $k$.
	\end{proof}

\subsection{Resampling method in the pair model}\label{subsec:2ndmo:resampling}
	Techniques similar to the ones discussed so far can be used to prove the corresponding results for the second moment, Propositions \ref{prop:maxim:2ndmo} and \ref{prop:negdef:2ndmo}. In this section, we discuss the necessary adjustments in the procedure to apply the resampling method to the pair model. Throughout the subsection, we fix a tuple of constants $\ula = (\lambda^1, \lambda^2)$ such that $0 \le \lambda^1, \lambda^2 \le 1$.
	
To begin with,	we define the sampling empirical measures. Recall that $\Omega_2 := \Omega \times \Omega$. On a given \textsc{nae-sat} instance $\GGG$, a valid pair-coloring $\bsig \in \Omega_2^E$, and $Y\subset V(G_n)$,  the tuple  $\bH^{\sm} = \bH^{\sm}[\GGG, Y, \bsig] = (\dot{\bH}^{\sm}, \hat{\bH}^{\sm}, \bar{\bH}^{\sm})$ is defined  analogously to Definition \ref{def:sample:empirical}. Note that $\dot{\bH}^{\sm}, \hat{\bH}^{\sm}, \bar{\bH}^{\sm}$ are probability measures on $\Omega_2^d $, $\Omega_2^k$ and $\Omega_2$, respectively. The tuple $\bH^{\sy} = (\dot{\bH}, \hat{\bH}^{\sy}, \bar{\bH})$ and the probability measure $\dot{\bh}\in \PPP(\dot{\Omega}_2)$ are defined analogously to \eqref{eq:def:hdot}.

Moreover, denote the pair-coloring (resp.~pair-component coloring) on $\GGG$ by $\bsig = (\sig^1, \sig^2)$ (resp.~$\bsig^{\textnormal{com}}$), recalling the one-to-one correspondence between $\bsig$ and $\bsig^{\textnormal{com}}$.  We define $\dot{\ttt}^1(e)\equiv \dot{\ttt}^1_{\bsig}(e) $ as before, and similarly for $\dot{\ttt}^2(e)$.  Then, \eqref{eqn:def:update} can be defined for each copy of the pair model and hence we have the analog of Lemma \ref{lem:update}. 

Further, we define the pair-model analog of $\dot{\ttt}^1(e)$ as follows: For an edge $e=(av)$, let $\dot{\uuu}(e) \equiv \dot{\uuu}_{\bsig}(e)$ be the graph of the \textit{variable-to-clause} directed union-free tree hanging at the \textit{root edge} $e$, i.e., it is the subtree of the union-free tree containing $e$ obtained by deleting all the variables, clauses, and edges closer to $a$ than $v$. If $v$ is frozen in both copies, we define $\dot{\uuu}(e)$ to be the single edge $e$.  

Having defined $\bw_{\NN}^{\lit}$ and $\bw_\partial^{\lit}$ analogously to \eqref{eq:def:wN:dep1nb} and \eqref{eq:def:wpartial:dep1nb}, we define the resampling Markov chain for the pair model as Definition \ref{def:resampling:markovchain}, with one modification needed for Step 1:
 \begin{enumerate}
 	\item [$1'.$]   If the sets $\{\dot{\uuu}_{\bsig}(e) \}_{e \in \delta \NN(Y)} $ are not disjoint, then $A_1 = A_0$ with probability 1.
 \end{enumerate}
	Note that if $\{\dot{\uuu}_{\bsig}(e) \}_{e \in \delta \NN(Y) }$ is disjoint, then both $\{\dot{\ttt}^1(e) \}_{e\in \delta \NN(Y)}$ and $\{\dot{\ttt}^2(e) \}_{e\in \delta \NN(Y)}$ are collections of disjoint trees. Thus, 1' is enough to ensure that the pair-coloring $\{\dot{\bsigma}_e \}_{e\in \delta \NN(Y)}$  at the boundary is invariant after resampling. Moreover, Lemma \ref{lem:resampling:reversingmeasure} works the same and gives the reversing measure for the resampling Markov chain in the pair model.

The analog of Definition \ref{def:resampling:treeoptimization} can be stated for $\bH^{\sm} $, using $\hat{v}_2$, $\dot{\Phi}_2, \hat{\Phi}_2^{\textnormal{m}}, \bar{\Phi}_2$, $\ula$ and $\dot{\bh}$. We denote by $\Sigma_2^{\tr}$, $\bs^{\tr}_2$, $\Lambda_2$, $\Lambda^{\op}_2$, $\Lambda^{\op}_{2,L}$, $\Xi_2$ and $\Xi_{2,L}$ the corresponding quantities of (\ref{eq:def:treeop:Sigma and s}--\ref{eq:def:treeop:Xi}) for the pair model. 
Similarly to before,	the collection $\mathscr{A}(\bH^{\sm}, Y, \varepsilon)$ is defined to be the set of $(\GGG, Y, \bsig)$ such that
\begin{itemize}
	\item  $\bH^{\sm}[\GGG,Y,\bsig] = \bH^{\sm} $;

	\item $\{\dot{\uuu}_{\bsig}(e) \}_{e\in \delta \NN(Y)}$ is disjoint, and $v(\dot{\uuu}_{\bsig}(e))  \le \frac{4\log (1/\varepsilon)}{k\log 2}$ for all $e\in \delta\NN(Y)$.
\end{itemize}
The corresponding analog for the truncated	model is defined similarly. Then, it is straight-forward to see that Lemmas  \ref{lem:resampling:transitionprob} and \ref{lem:resampling:Lambda:sup} hold the same for $\Xi_2$ and $\Xi_{2,L}$ and have the same proof. Deriving the counterpart of Lemma \ref{lem:exists:qdot:hdot:exp:tail} requires extra work to adjust to the pair model, and its statement and proof are described in 
Corollary \ref{cor:2ndmo:hdot:uniqueness}. We also obtain the pair model version of Lemmas \ref{lem:1stmo:unique:zero:Xi} and \ref{lem:quadratic:growth:Xi:1stmo}  as follows. In the statement, we write $v(\dot{\bsigma}):= v(\dot{\sigma}^1) + v(\dot{\sigma}^2)$ for $\dot{\bsigma} = (\dot{\sigma}^1, \dot{\sigma}^2).$
Moreover, we consider the empirical measures that lie in the \textit{near-independent} regime, formally defined as follows: Let $\prescript{\textnormal{in}}{2}{\bDelta}$ be the space of tuples $\bH = (\dot{\bH}, \hat{\bH}, \bar{\bH})$ with $\dot{\bH} \in \mathscr{P}(\Omega_2^d)$, $\hat{\bH}\in \mathscr{P}(\Omega_2^k)$ and $\bar{\bH} \in \mathscr{P}(\Omega_2)$, such that
\begin{itemize}
    \item Each copy of $\bH$ is in $\bDelta$ (Definition \ref{def:empiricalssz})
    
    \item The \textit{overlap} $\zeta(\bH)$ satisfies $|\zeta(\bH) -\frac{1}{2} | <\frac{k^2}{2^{k/2}}$. Note that although the pair-coloring $\bsig$ is not well-defined from $\bH$, the overlap is still well-defined. 
\end{itemize}
We set $\prescript{\textnormal{in}}{2}{\bDelta}^{(L)}$ to be the $L$-truncated space of $\prescript{\textnormal{in}}{2}{\bDelta}.$

\begin{lemma}\label{lem:2ndmo:Xi:unique and quad}
	Let $\bH \in \prescript{\textnormal{in}}{2}{\bDelta}$ be $\bH = \bH^{\sy}$ and suppose $\dot{\bh} = \dot{\bh}[\bH]$  satisfies $\sum_{\dot{\bsigma}: v(\dot{\bsigma}) \ge L} \dot{\bh}(\dot{\bsigma}) \le 2^{-ckL} $ for all $L$ with an absolute constant $c>0$. Then, $\Xi_2(\bH)=0$ if and only if $\bH = \bH_{\ula}^\star$. 
	
	For the truncated model, the corresponding result holds the same with $\bH_{2,L}.$ Furthermore, there exist constants $c_k, \varepsilon_L>0$ such that for $\prescript{\textnormal{in}}{2}{\bDelta}^{(L)}$ with $\bH = \bH^{\sy}$ and $|| \dot{\bh}[\bH]- \dot{\bh}_{\ula, L}^\star ||_1 <\varepsilon_L$, we have
	\begin{equation}\label{eq:2ndmo:Xi:quad}
	\Xi_{2,L} (\bH) \ge c_k ||\bH -\bH_{\ula,L}^\star ||_1^2.
	\end{equation}
\end{lemma}
	
\begin{proof}
	The first part of the lemma can be done the same as Lemma \ref{lem:1stmo:unique:zero:Xi}, using Corollary \ref{cor:2ndmo:hdot:uniqueness} instead of Lemma \ref{lem:exists:qdot:hdot:exp:tail}. To establish \eqref{eq:2ndmo:Xi:quad}, we repeat the proof of Lemma \ref{lem:quadratic:growth:Xi:1stmo}, relying on the analogs of  Lemmas \ref{lem:nu:q:continuous:1stmo} and \ref{lem:1stmo:BPfixedpoint:estimates} for the pair model; see Appendix \ref{subsec:app:conti:2ndmo} for details.
\end{proof}

	To complete the proof of Propositions \ref{prop:maxim:2ndmo} and \ref{prop:negdef:2ndmo}, we obtained  the pair model versions of the tools introduced in Sections \ref{subsubsec:resampling:maximizer} and \ref{subsubsec:resampling:negdef}. One necessary element is the generalized version of Lemma \ref{lem:1stmo:resampling:bad:variables}: We need to show that there are not many \textit{bad} variables from the pair-model perspective. The property can be stated as follows:
	
	\begin{cor}\label{cor:2ndmo:resampling:bad:variables}
		For $v\in V(\GGG)$, let $\NN(v)$ be the $\frac{3}{2}$ neighborhood of $v$ and $\delta \NN(v)$ be the set of half-edges hanging at the boundary of $\NN(v)$. Given $(\GGG,\bsig)$, we define $V_{\textnormal{bad}}^{\shortparallel}\equiv V_{\textnormal{bad}}^{\shortparallel}(\GGG,\bsig)$
		\begin{equation*}
		V_{\textnormal{bad}}^{\shortparallel}\equiv \{v\in V:\exists e_1,e_2 \in \delta \NN(v)\quad\textnormal{s.t.}\quad \dot{\uuu}_{\bsig}(e_1)\cap \dot{\uuu}_{\bsig}(e_2) \neq \emptyset\}
		\end{equation*}
		Fix $\bB \in \bDelta^{\textnormal{b}}_n$ and $(n_{\uuu})_{\uuu\in \FFF_{2}} \sim \bB$ such that $\sum_{v(\uuu)=v, \uuu \in \FFF_{2}}n_{\uuu}(\bsig)\leq n2^{-kv/4}$ for all $v\geq 1$. Then, we have
		\begin{equation}\label{eq:lem:2ndmo:resampling:bad:variables}
		\sum_{\substack{(\GGG,\bsig):\bB[\bsig]=\bB,\\ n_{\uuu}(\bsig)=n_{\uuu},\forall\uuu\in\FFF_{2}}}\P(\GGG)\bw^\lit_\GGG(\bsig)^{\ula}\one\{|V_{\textnormal{bad}}^{\shortparallel} |\geq \sqrt{n}\}\lesssim_{k}\frac{\log n}{\sqrt{n}} \sum_{\substack{(\GGG,\bsig):\bB[\bsig]=\bB,\\ n_{\uuu}(\bsig)=n_{\uuu},\forall\uuu\in\FFF_{2}}}\P(\GGG)\bw^\lit_{\GGG}(\bsig)^{\ula}.
		\end{equation}
	\end{cor}
	
	\begin{proof}
		Relying on the same idea used in the proof of Lemma \ref{lem:1stmo:resampling:bad:variables},
		we briefly discuss the necessary changes needed in the pair model. 
		
		In Lemma \ref{lem:1stmo:resampling:bad:variables}, we divided the bad variables into four distinct categories and estimated the contribution from each of them. We can again separate $V_{\textnormal{bad}}^{\shortparallel}$ into four parts as follows. Recall the collections of boundary spins $\dot{\partial}_2$ and $\hat{\partial}_2$
		\begin{enumerate}
			\item $\exists$ non-pair-separating clauses (i.e., non-separating in at least one copy) $a_1, a_2 \sim v$ such that $\bsigma_{a_1v}, \bsigma_{a_2v} \in \dot{\partial}_2$ and $a_1, a_2$ are contained in the same free tree.
			
			\item $\exists$ $e_1, e_2\in \delta \NN(v)$ such that $a(e_1) \neq a(e_2)$ and ${\bsigma}_{e_1}, \bsigma_{e_2} \in \hat{\partial}_2$, with $v(e_1)$ and $v(e_2)$ in the same union-free tree.

			\item $\exists a\sim v$, $\exists e_1, e_2 \in \delta a$ such that $\bsigma_{e_1}, \bsigma_{e_2} \in \hat{\partial}_2$ and $v(e_1), v(e_2)$ are in the same union-free tree.
			
			\item $\exists e_1, e_2 \in \delta \NN(v)$ such that $\bsigma_{e_1} \in \hat{\partial}_2 $, $\bsigma_{a(e_2)v} \in \dot{\partial}_2$, $a(e_2)$ is non-pair-separating, and $v(e_1), a(e_2)$ are in the same union-free tree.
		\end{enumerate}
		
		Due to the assumption that gives an exponential decay of union-free tree frequencies in their sizes, we can repeat the same argument as Lemma \ref{lem:1stmo:resampling:bad:variables} to bound the size of each of the four collections described above. This implies that their sizes are all bounded by $O_k(\log n)$ in expectation, and hence we obtain the desired conclusion.
	\end{proof}
    We also recall the following estimate from \cite{ssz22}, which shows that the contribution to the second moment from pair configuration with intermediate overlap is negligible.
    \begin{lemma}[Corollary~D.2 of \cite{ssz22}]\label{lem:ssz:inter}
    On $\GGG$, let $Z^2[\zeta]$ count the number of pairs $\ubx,\ubx^\prime \in \{0,1\}^V$ of valid $\naesat$ solutions which agree on $\zeta$ fraction of variables. Then $\E Z^2[\zeta]\leq \exp\{-nk/2^k\}$ holds for all $\zeta$ in $[\exp\{-k/\log k\}, \frac{1}{2}(1-k/2^{k/2})]\cup [\frac{1}{2}(1+k2^{k/2}, 1-\exp\{-k/\log k\}]$.
    \end{lemma}
	We now are ready to discuss  the proof of Propositions \ref{prop:maxim:2ndmo} and \ref{prop:negdef:2ndmo}.
	
	\begin{proof}[Proof of Proposition \ref{prop:maxim:2ndmo}]
		
Note that Lemma \ref{lem:1stmo:resampling:treedisjoint} generalizes naturally to the pair model, as the same proof works with union-free trees using Corollary \ref{cor:2ndmo:resampling:bad:variables}. Also, by Lemma \ref{lem:free:red:2} and Lemma \ref{lem:ssz:inter}, the contribution from pair-coloring $\bsig=(\sig^1,\sig^2)$ with either $\frac{\rr(\sig^1)\vee\rr(\sig^2)}{nd} \vee \frac{\ff(\sig^1)\vee\ff(\sig^2)}{n}>\frac{7}{2^k}$ or $\big|\zeta(\bsig)-\frac{1}{2}\big|\in [\frac{k^2}{2^{k/2}}, \frac{2k^2}{2^{k/2}}]$ is negligible, so the analog of \eqref{eq:maxim:upperbound:1} for the second moment holds. Then, we obtain Proposition \ref{prop:maxim:2ndmo} following the proof of Proposition \ref{prop:maxim:1stmo}, based on the aforementioned lemmas for the pair model.
\end{proof}

\begin{proof}[Proof of Proposition \ref{prop:negdef:2ndmo}]
	If the same result as Proposition \ref{prop:crux:negdef:1stmo} holds for the pair model, then the  subsequent argument analogous to the proof of Proposition \ref{prop:negdef} gives Proposition \ref{prop:negdef:2ndmo}.  In order to reproduce the results in the proof of Proposition \ref{prop:crux:negdef:1stmo} for the pair model, let the boundary profile $\bB$ be $\bB \in \mathcal{B}_{\ula}^-(\delta_\circ)$ with $\bB = \bB^{\sy}$ and let $\AAA_0' \equiv \AAA_0'(\bB, \varepsilon, L)$ be the collection of $(\GGG, Y, \bsig, \{n_\uuu \}_{\uuu \in \FFF_2} )$ satisfying the following conditions:
	\begin{itemize}
		\item $\bsig \sim \{n_\uuu \}_{\uuu \in \FFF_2}$, $\bsig \in \Omega_{2,L}^E$ and $\bB[\bsig] = \proj(\bB)$.
		
		\item $|\frac{n_\uuu}{n} - p_{\uuu, \ula, L}(\bB) | \le n^{-1/3} \wedge 2^{-ckv(\uuu)}$ for all $\uuu \in \FFF_2^{\tr},$ and $n_\uuu \le \log^2 n$ for $\uuu \in \FFF_2\setminus \FFF_2^{\tr}$.
		
		\item  $\{\dot{\ttt}_{\bsig}^1(e) \}_{e\in \delta \NN(Y)}$ and $\{\dot{\ttt}_{\bsig}^2(e) \}_{e\in \delta \NN(Y)}$ are disjoint within each of them, and $|Y| \in [\varepsilon n/2, 2\varepsilon n]$.
		
		\item $|| \bH^{\sm}[\GGG,Y,\bsig] - (\bH[\bsig])^{\sy} ||_1 \le \varepsilon^{1/3}.$ 
	\end{itemize} 
	Also, let $\AAA_1' \equiv \AAA_1'(\bB ,\varepsilon, L)$ be the set of $A_1 = (\GGG', Y, \underline{\btau})$ such that
	\begin{itemize}
		\item $A_1$ is one-step approachable from some $A_0 \in \AAA_0'$ by the $L$-truncated resampling Markov chain.
		
		\item Denote $\bH^{\sm} = \bH^{\sm}(A_1)$ and $\nu = \nu_L^{\op} \left[\dot{\bh}[\bH^{\sm}] \right]$, where $\nu_L^{\op}[\cdot]$ is defined in \ref{eq:def:nu:q:dot}. Then, $||\bH^{\sm} - \bH^{\tr}[\nu] ||_1 \le \varepsilon^{1/3}.$
	\end{itemize}
	
	Then, the proof of Proposition \ref{prop:crux:negdef:1stmo} for the pair model goes as follows.
	\begin{itemize}
		\item The second moment analog of Lemma \ref{lem:1stmo:negdef:initialstate} for $\AAA_0'(\bB, \varepsilon,L)$ is obtained analogously, adapting the computations given in  Lemma \ref{lem:exist:free:energy:1stmo} (cf. Proposition \ref{prop:exist:free:energy:2ndmo}-(1)).
		
		\item Lemma \ref{lem:cont:h:op} holds the same for the pair model: The first item of the lemma is justified by Corollary \ref{cor:compat:2ndmo}. The other two can be obtained analogously to the single-copy case.
		
		\item Generalization of Lemma \ref{lem:1stmo:negdef:onestep} to the pair model is done by utilizing Lemma \ref{lem:2ndmo:Lipschitz:hdot:qdot} instead of Proposition \ref{prop:1stmo:Lipschitz:hdot:qdot}, in the same proof as Lemma \ref{lem:1stmo:negdef:onestep}.
		
		\item Then, we follow the same proof as Proposition \ref{prop:crux:negdef:1stmo}, using Proposition \ref{prop:exist:free:energy:2ndmo} in place of Lemma \ref{lem:exist:free:energy:1stmo}. 
	\end{itemize}
	This gives the pair model analog of Proposition \ref{prop:crux:negdef:1stmo}, and hence we conclude the proof of Proposition \ref{prop:negdef:2ndmo}.
\end{proof}

	\section{Concentration of the overlap at two values}\label{sec:overlap}
	
	In this section, we establish Theorem \ref{thm2}. For a random regular \textsc{nae-sat} instance $\GGG$, we denote by $\CCC (\GGG)$ the collection of clusters of solutions. The procedure of drawing two solutions uniformly, independently at random can be understood in the following way:
	\begin{enumerate}
		\item Pick two clusters $\CC_1, \CC_2 \in \CCC(\GGG)$ independently at random, with probability proportional to their sizes $|\CC_1|, |\CC_2|$, respectively.
		
		\item Select two solutions $\bx_1 \in \CC_1$, $\bx_2\in \CC_2$ independently and uniformly from each cluster. 
	\end{enumerate}
	The main idea to verify Theorem \ref{thm2} has already been discussed in the previous sections:
 The two randomly drawn clusters $\CC_1, \CC_2$ in step (1) will look near-uncorrelated or  near-identical (Proposition \ref{prop:2ndmo:correlated overlap} and Lemma \ref{lem:identical regime decay estim}). After some analysis to understand the second step of sampling random solutions, the former (resp.~latter) case will give us (a) (resp.~(b)) of Theorem \ref{thm2}.
 
The primary difficulty in formalizing such an idea comes from the clusters possessing a cyclic free component. Most of our efforts have been taken to understand the moments of  $\bZ_{\lambda, s}^\tr$, which only considers the contributions from clusters without cyclic free components. Although we have Proposition \ref{prop:1stmo:aprioriestimate} to control the effect of the rest, we do not have information about the typical profile of union-free components in a pair of clusters when each copy contains a cyclic free. We resolve this issue by comparing $\GGG$ with a \textit{locally rewired instance} $\GGG^\prime$ of $\GGG$, which is identical to $\GGG$ except for a small number of edges and literals.

In Section \ref{subsec:overlap:preprocess}, we find the appropriate set of \textsc{nae-sat} instances that satisfy the properties of overlaps stated in Theorem \ref{thm2}. In Section \ref{subsec:overlap:locrewire}, we describe how \textit{locally rewired instances} can be used to control clusters containing a free cycle. In Section \ref{subsec:overlap:samplesol}, we analyze the overlaps of solutions sampled from typical clusters and prove Theorem \ref{thm2}.

\subsection{Preprocessing}
\label{subsec:overlap:preprocess}
In this subsection, we specify the collection of \textsc{nae-sat} instances $\mathfrak{G}_i, i=1,2,3,$ which gives the desired properties of overlaps. We first start with $\mathfrak{G}_1$, which is related to the first moment analysis done in Section \ref{sec:1stmo}.

To begin with, we first define $p^\star\equiv p^\star(\alpha,k)$ appearing in Theorem \ref{thm2}: for $\ubx^{1},\ubx^{2}\in \{0,1\}^{\ell}$, define the Hamming distance and the overlap of $\ubx^1,\ubx^2$ by
	\begin{equation*}
		\textnormal{Ham}(\ubx^1, \ubx^2) := \sum_{i=1}^{\ell}  \one \{  \bx_i^1 \neq \bx_i^2 \}\quad\textnormal{and}\quad \textnormal{Overlap}(\ubx^1,\ubx^2):=\ell-2\textnormal{Ham}(\ubx^1,\ubx^2).
	\end{equation*}
	For each free tree $\ttt \in \FFF_{\tr}$, let $\textsf{SOL}(\ttt) \subset \{0,1 \}^{V(\ttt)}$ be the space of valid \textsc{nae-sat} solutions on $\ttt$. Then, define
	\begin{equation}\label{eq:def:overlap:ham}
	\textnormal{ham}(\ttt) := \frac{1}{|\textsf{SOL}(\ttt)|^2} \sum_{\ubx^1, \ubx^2 \in \textsf{SOL}(\ttt) }  \textnormal{Ham}(\ubx^1, \ubx^2 )\quad\textnormal{and}\quad \textnormal{overlap}(\ttt):= v(\ttt)-2\textnormal{ham}(\ttt).
	\end{equation}
	Thus, $\textnormal{ham}(\ttt)$ (resp. $\textnormal{overlap}(\ttt)$) is the average Hamming distance (resp. overlap) between two random \textsc{nae-sat} solutions on $\ttt$. Then, $p^\star = p^\star(d,k)$ is defined as follows.
	\begin{defn}\label{def:pstar}
	For $k\geq k_0$ and $\alpha \in (\alpha_{\textsf{cond}},\alpha_{\textsf{sat}})$, define
	\begin{equation}\label{eq:def:pstar:form}
	p^\star\equiv p^\star(\alpha,k) := 1 - \sum_{\ttt \in \FFF_{\tr}} v(\ttt)p^\star_{\ttt, \lambda^\star} +  \sum_{\ttt \in \FFF_{\tr}} \textnormal{overlap}(\ttt) p^\star_{\ttt, \lambda^\star}=1 -2 \sum_{\ttt \in \FFF_{\tr}} \textnormal{ham}(\ttt) p^\star_{\ttt,\lambda^\star}.
	\end{equation}
	Hence, $p^\star$ is the sum of the fraction of frozen variables and the total average overlap on free trees.
	\end{defn}
	Next, we define $\Gamma^\star_1$, the set of boundary and free component profiles of typical clusters as follows.
	\begin{defn}\label{def:Gamma1star}
		Let $\Gamma^\star_1$ be the collection of boundary and free component profiles $(B, \{n_\fff\}_{\fff\in\FFF}) $ that satisfy the following conditions:
		\begin{enumerate}
			\item $(n_{\fff})_{\fff\in\FFF}\sim B$, where $B \in \bDelta^{\textnormal{b}}$.
			\item $(n_{\fff})_{\fff\in\FFF}\in \ee_{\frac{1}{3}}$.
			
			\item $||B-B^\star_{\lambda^\star} ||_1 \leq n^{-0.45}$.
			\item $\Big| \sum_{\ttt\in\FFF_{\tr}} (n_\ttt-np^\star_{\ttt,\la^\star})\cdot\textnormal{overlap}(\ttt)\Big| \le n^{0.6}$.
			
			\item $\sum_{\fff\in\FFF\setminus \FFF_{\tr}}n_{\fff}\leq \log n$ and $n_\fff=0$ if $\gamma(\fff)\geq 0$.
			
		\end{enumerate}
	\end{defn}
	\begin{lemma}\label{cor:B:close:optimal:1stmo}
		Let $\E \bZ_{\lambda^\star} \big[(\Gamma_1^\star)^{\textsf{c}}\big]$ denote the contribution to $\E\bZ_{\la^\star}$ from $(B, \{n_\fff\}_{\fff\in\FFF})\notin \Gamma_1^{\star}$. Then, we have $$\E \bZ_{\lambda^\star} \big[(\Gamma_1^\star)^{\textsf{c}}\big] \lesssim_{k} \frac{\log^{3} n}{n} \E\bZ^{\tr}_{\la^\star}.$$
	\end{lemma}
	\begin{proof}
	By Proposition \ref{prop:1stmo:aprioriestimate}, we have
	\begin{equation}\label{eq:cor:B:close:optimal:1stmo}
	    \E \bZ_{\la^\star}\Big[\big(\ee_{\frac{1}{3}}\big)^{\textsf{c}}\textnormal{  or  }n_{\cyc}\geq \log n \textnormal{  or  }e_{\textnormal{mult}}\geq 1 \Big]\lesssim_{k} \frac{\log^{3}n}{n}\E\bZ_{\la^\star}.
	\end{equation}
	Moreover, the same argument to show \eqref{eq:prop:uni:negligible} in the proof of Proposition \ref{prop:ratio:uni:1stmo}, which makes use of the map $\TT(\cdot)$, can be used to show
	\begin{equation}\label{eq:cor:B:close:step1}
	    \E \bZ_{\la^\star}\Big[||B-B^\star_{\la^\star}||_1>n^{-0.45},~~n_{\cyc}\leq \log n ~~, e_{\textnormal{mult}}=0,~~(n_{\fff})_{\fff\in\FFF}\in \ee_{\frac{1}{3}}\Big]\leq \exp\big(-\Omega_k(n^{0.1}) \big)\E \bZ_{\la^\star}^{\tr}.
	\end{equation}
	To this end, we now consider $B\in \bDelta^{\textnormal{b}}_n$ with $||B-B^\star_{\la^\star}||_1\leq n^{-0.45}$. Fix cyclic free component profile $(n_{\fff})_{\fff\in \FFF\setminus \FFF_{\tr}}$, which satisfies $\sum_{\fff\in \FFF\setminus \FFF_{\tr}}n_{\fff}\leq \log n$ and $n_{\fff}=0$ if $\gamma(\fff)\geq 0$ or $v(\fff)>\frac{3\log n}{k \log 2}$. Let $\uh^\prime\equiv \big(h^\prime(\circ), \{h^\prime(x)\}_{x\in \partial}\big)$ record the number of free trees and total number of boundary colors adjacent to free trees. Then, $\uh^\prime$ is a function of $B$ and $(n_{\fff})_{\fff\in \FFF\setminus \FFF_{\tr}}$. For example, $h^\prime(x)=h_x(B)-n^{-1}\sum_{\fff\in \FFF\setminus \FFF_{\tr}}\eta_{\fff}(\sigma)n_{\fff}$ holds for $x\in \partial$. Note that the same proof done in Lemma \ref{lem:exist:theta} shows that there exists $(\utheta^\prime)^{-}\in \R^{|\partial|+1}$ such that $\nabla\psi^{-}_{\la^\star}\big((\utheta^\prime)^{-}\big)=\uh^\prime$. Then, proceeding in the similar fashion as the calculations done in \eqref{eq:tree:decay:technical:1} and \eqref{eq:tree:decay:technical:2}, Proposition \ref{prop:1stmo:B nt decomp} and local central limit theorem show
	\begin{equation}\label{eq:cor:B:close:technical-1} 
	\begin{split}
	    &\E \bZ_{\la^\star}\Big[B, \Big| \sum_{\ttt\in\FFF_{\tr}} (n_\ttt-np^\star_{\ttt,\la^\star})\cdot\textnormal{overlap}(\ttt)\Big|>n^{0.6},\textnormal{ and }(n_{\fff})_{\fff\in \FFF\setminus \FFF_{\tr}}\Big]\\
	    &\leq \exp\big(O_k(\log n)\big)\P_{(\utheta^\prime)^{-}}\bigg( \Big|\sum_{i=1}^{nh^\prime(\circ)}\textnormal{overlap}(X_i)-n\sum_{\ttt\in \FFF_{\tr}}p^\star_{\ttt,\la^\star}\cdot \textnormal{overlap}(\ttt) \Big|>n^{0.6}\bigg)\E\bZ_{\la^\star}\Big[B,(n_{\fff})_{\fff\in \FFF\setminus \FFF_{\tr}}\Big],
	\end{split}
	\end{equation}
	where $X_1,...,X_{nh^\prime(\circ)}\in \FFF_{\tr}$ are i.i.d with distribution $\P_{(\utheta^\prime)^{-}}(X_i=\ttt)\equiv \big(h^\prime(\circ)\big)^{-1} J_{\ttt}w_{\ttt}^{\la^\star}\exp\Big(\big\langle (\utheta^\prime)^{-}, \boeta_{\ttt}^{-} \big\rangle\Big)$ (cf. \eqref{eq:def:rescaling:prob}). Note that $||\uh^\prime-\uh(B)||_1\lesssim_{k} \log^{2} n$ holds, so $||\uh^\prime-\uh^\star_{\la^\star}||_1\lesssim n^{-0.45}$ holds. Thus, the same proof of Lemma \ref{lem:exist:theta} shows $||(\utheta^\prime)^{-}-\utheta^{\star,-}_{\la^\star}||_1 \lesssim_{k} n^{-0.45}$. Hence, we can bound
	\begin{equation*}
	\begin{split}
	&\Big|nh^\prime(\circ)\E_{(\utheta^\prime)^{-}}\big[\textnormal{overlap}(X)\big]-n\sum_{\ttt\in \FFF_{\tr}}p^\star_{\ttt,\la^\star}\cdot \textnormal{overlap}(\ttt)\Big|\\
	&\leq n\sum_{\ttt\in \FFF_{\tr}} \textnormal{overlap}(\ttt)p^\star_{\ttt,\la^\star} \bigg|\exp\Big(\Big\langle (\utheta^\prime)^{-}-\utheta^{\star,-}_{\la^\star}~,~ \boeta_{\ttt}^{-} \Big\rangle\Big)-1\bigg|\lesssim_{k}n^{0.55}\sum_{\ttt\in \FFF_{\tr}} v(\ttt)^2 p^\star_{\ttt,\la^\star}\ll n^{0.6},
	\end{split}
	\end{equation*}
	where we used Lemma \ref{lem:opt:tree:decay} in the last inequality. Thus, Chernoff bound shows 
	\begin{equation}\label{eq:cor:B:close:technical-2} 
	\P_{(\utheta^\prime)^{-}}\bigg( \Big|\sum_{i=1}^{nh^\prime(\circ)}\textnormal{overlap}(X_i)-n\sum_{\ttt\in \FFF_{\tr}}p^\star_{\ttt,\la^\star}\cdot \textnormal{overlap}(\ttt) \Big|>n^{0.6}\bigg)\leq \exp\big(-\Omega_k(n^{0.2})\big).
	\end{equation}
	Therefore, plugging in \eqref{eq:cor:B:close:technical-2} to \eqref{eq:cor:B:close:technical-1} and summing over $B$ and $(n_{\fff})_{\fff\in\FFF\setminus\FFF_{\tr}}$ shows that the contribution to $\E \bZ_{\la^\star}$ from boundary and free component profile $(B, \{n_\fff\}_{\fff\in\FFF})$ that satisfy the items (1),(2),(3),(5), but not (4) in Definition \ref{def:opt:bdry:1stmo} is bounded above by $\exp\big(-\Omega_k(n^{0.2}) \big)\E \bZ_{\la^\star}$. Consequently, together with \eqref{cor:B:close:optimal:1stmo} and \eqref{eq:cor:B:close:step1}, this concludes the proof.
	\end{proof}
	Let $\ux[\CC] \in \{0,1,\ff \}^V$ be the frozen configuration corresponding to $\CC \in \CCC(\GGG)$ from the coarsening algorithm in Definition \ref{def:frozenconfig}. With a slight abuse of notation, write $\CC \in \Gamma_1^\star$ if the boundary and free tree profile  pair induced by $\ux[\CC]$ is contained in $\Gamma_1^\star$. Then, define
	 \begin{equation}\label{eq:def:overlap:clustspace:restriction1}
	 \begin{split}
     \CCC_\bullet(\GGG) \equiv \CCC(\GGG; I_\bullet)&:= \left\{\CC \in \CCC(\GGG) :  \frac{1}{n}\log |\CC| \in I_\bullet \textnormal{ and } \CC \in \Gamma_1^{\star}\right\},\quad\textnormal{where}\\
     I_\bullet \equiv I_\bullet(n,d,k) &:= \left[ s^\star - \frac{1}{2\la^\star}\frac{\log n}{n}-\frac{\log^2 n}{n}, \ s^\star -\frac{1}{2\la^\star}\frac{\log n}{n}+ \frac{\log\log n}{n}  \right].
    \end{split}
	 \end{equation}
	Subsequently, define the collection of \textsc{nae-sat} instances $\mathfrak{G}_1 \equiv \mathfrak{G}_1(n,d,k)$ by 
	 \begin{equation}\label{eq:overlap:clustersize:restriction1}
	 \begin{split}
	 \mathfrak{G}_1:=\Big\{\GGG:\sum_{\CC \in \CCC(\GGG)} |\CC|\one\big\{\CC\notin \CCC_{\bullet}(\GGG)\big\} \leq n^{-\frac{1}{2\la^\star}-0.4} e^{ns^\star}\Big\}.
	 \end{split}
	 \end{equation}
	 \begin{lemma}\label{lem:mathfrak:G:1}
	 For $\mathfrak{G}_1$ defined above, $\P( \GGG \notin \mathfrak{G}_1) \lesssim_{k} (\log n)^{-\la^\star}$ holds.
	 \end{lemma}
	 \begin{proof}
	 Recall \eqref{eq:mainthm:1} in the proof of Theorem \ref{thm1}-(a). Using \eqref{eq:mainthm:1} with $C_0=\log\log n$ shows
	 \begin{equation*}
	     \P\Big(\exists~\CC\in\CCC(\GGG)\textnormal{ s.t. }\frac{1}{n}\log |\CC|\geq s^\star-\frac{1}{2\la^\star}\frac{\log n}{n}+\frac{\log \log n}{n}\Big)\lesssim_{k}(\log n)^{-\la^\star}.
	 \end{equation*}
	 Similarly, \eqref{eq:mainthm:2} with $C_0=-\log^{2}n$ and Markov's inequality show
	 \begin{equation*}
	     \P\bigg(\sum_{\CC\in\CCC(\GGG)}|\CC|\one\Big\{\frac{1}{n}\log|\CC|\leq s^\star-\frac{1}{2\la^\star}\frac{\log n}{n}-\frac{\log^{2}n}{n}\Big\}\geq 0.5 n^{-\frac{1}{2\la^\star}-0.4} e^{ns^\star}\bigg)\lesssim_{k}n^{0.4}e^{-(1-\la^\star)\log^{2}n}.
	 \end{equation*}
% 	 Moreover, a priori estimate in Lemma \ref{prop:1stmo:aprioriestimate} and Markov's inequality shows
% 	 \begin{equation*}
% 	 \begin{split}
% 	  &\P\bigg(\sum_{\CC\in\CCC(\GGG)}|\CC|\one\Big\{\frac{1}{n}\log|\CC|\in I_{\bullet}\textnormal{ and }\ux[\CC]\textnormal{ has a multicylic free component}\Big\}\geq \frac{1}{3}n^{-\frac{1}{2\la^\star}-0.4} e^{ns^\star}\bigg) \\
% 	  &\lesssim_{k}n^{\frac{1}{2\la^\star}+0.4} e^{-ns^\star}\sum_{s\in I_{\bullet}}\E \bZ_{1,s}\big[e_{\textnormal{mult}}\geq 1\big]\lesssim_{k}n^{0.9} e^{-n\la^\star s^\star+(1-\la^\star)\log\log n}\E\bZ_{\la^\star}\big[e_{\textnormal{mult}}\geq 1\big]\lesssim_{k} n^{-0.1}(\log n)^{4-\la^\star},
% 	 \end{split}
% 	 \end{equation*}
% 	 where we used Lemma \ref{prop:1stmo:aprioriestimate} and Theorem \ref{thm:1stmo:constant} in the last inequality. 
     Moreover, Corollary \ref{cor:B:close:optimal:1stmo} and Markov's inequality show
	 \begin{equation*}
	 \begin{split}
	 &\P\bigg(\sum_{\CC\in\CCC(\GGG)}|\CC|\one\Big\{\frac{1}{n}\log |\CC|\in I_{\bullet}\textnormal{ and }\CC\notin \Gamma_1^\star\Big\}\geq 0.5 n^{-\frac{1}{2\la^\star}-0.4} e^{ns^\star}\bigg)\\
	 &\leq 2n^{\frac{1}{2\la^\star}+0.4} e^{-ns^\star}\sum_{s\in I_{\bullet}}\E \bZ_{1,s}\big[(\Gamma_1^\star)^{\textsf{c}}\big]
	 \lesssim_{k} n^{0.9} e^{-n\la^\star s^\star+(1-\la^\star)\log\log n}\E\bZ_{\la^\star}\big[(\Gamma_1^\star)^{\textsf{c}}\big]\lesssim_{k}
	 n^{-0.1}(\log n)^{4-\la^\star},
	 \end{split}
	 \end{equation*}
	 where we used Corollary \ref{cor:B:close:optimal:1stmo} and Theorem \ref{thm:1stmo:constant} in the last inequality. Therefore, $\P( \GGG \notin \mathfrak{G}_1) \lesssim_{k} (\log n)^{-\la^\star}$ holds.
	 \end{proof}
	 	
	Next, we define $\mathfrak{G}_2$, which is related to the second moment analysis done in Section \ref{sec:2ndmo}. To do so, we start with the following definition. Recall the notion of union-free tree in Definition \ref{def:union:tree}.
	\begin{defn}\label{def:overlap:flipedcomp}
	 Let $\uuu \in \FFF_2^{\tr}$ be a union-free tree. \textbf{Flipped component} $\textsf{fl}(\uuu)\in \FFF_2^{\tr}$ of $\uuu$ is defined as follows:
	\begin{itemize}
		\item For $\bsigma = (\sigma^1, \sigma^2) \in \{\rr_0,\rr_1,\bb_0,\bb_1,\fs,\ff \}^2$, we define $\textsf{flip}(\bsigma) := (\sigma^1, \sigma^2 \oplus 1)$, where $\fs \oplus 1 = \fs$ and $\ff\oplus 1 = \ff$.
		
		\item $\textsf{fl}(\uuu)$ is defined to have the same graphical structure as $\uuu$, that is, $V(\textsf{fl}(\uuu)) = V(\uuu)$, $F(\textsf{fl}(\uuu)) = F(\uuu)$, $E(\textsf{fl}(\uuu)) = E(\uuu)$ and $\partial \textsf{fl}(\uuu)=\partial\uuu$.
		
		\item Furthermore, the label on each $e \in E(\textsf{fl}(\uuu))\sqcup \dot{\partial}\textsf{fl}(\uuu)$ is given by $(\textsf{flip} (\textsf{P}_2(\uuu,e)), \ \tL_e  )$, and the label on $e\in \hat{\partial} \textsf{fl}(\uuu)$ is given by $\textsf{flip}(\textsf{P}_2(\uuu,e)).$
	\end{itemize}
	In words, $\textsf{fl}(\uuu)$ is the union-free tree obtained by flipping the second copy of spin labels. Since \textsc{nae}-satisfiability is invariant under global spin flip, it is clear that the labeling on $\textsf{fl}(\uuu)$ is valid.
	\end{defn}
	
	Note that for any $\uuu \in \FFF_2^{\tr}, $ we have from the symmetry that $p^\star_{\uuu,\ula^\star} = p^\star_{\textsf{fl}(\uuu), \ula^\star}$. The symmetry also implies the following lemma, which explains why the overlap concentrates around $0$ in the near-independence regime.
	
	\begin{lemma}\label{lem:overlap:pairflipsymm}
		Let $\uuu  \in \FFF_2^{\tr}$ be a union-free tree, and let $(\ubx^1, \ubx^2) = \{(\bx^1_v,\bx^2_v)\}_{v\in V(\uuu)}$ be a uniformly chosen pair of  0-1 configurations on $V(\uuu)$ among all pairs of \textsc{nae-sat} solutions on $\uuu$. We denote the law of $\textnormal{Ham}(\ubx^1, \ubx^2)$ by $\PP_\uuu$. 
		Then, we have for all $0 \le h \le v(\uuu)$ that
		\begin{equation}\label{eq:def:overlap:uniontree}
		\PP_\uuu \big( \textnormal{Ham}(\ubx^1,\ubx^2) = h \big) 
		= \PP_{\textnormal{\textsf{fl}}(\uuu)}\big(\textnormal{Ham}(\ubx^{1},\ubx^{2}) = v(\uuu) - h\big).
		\end{equation}
		Hence, if we let $\textsf{SOL}(\uuu) \subset (\{0,1 \}^2)^{V(\uuu)}$ be the space of valid \textsc{nae-sat} solutions on $\uuu$ and define
	\begin{equation*}
	\textnormal{ham}(\uuu) := \frac{1}{|\textsf{SOL}(\uuu)|} \sum_{(\ubx^1, \ubx^2) \in \textsf{SOL}(\uuu) }  \textnormal{Ham}(\ubx^1, \ubx^2 )\quad\textnormal{and}\quad \textnormal{overlap}(\uuu):= v(\uuu)-2\textnormal{ham}(\uuu),
	\end{equation*}
	then $\textnormal{overlap}(\uuu)=-\textnormal{overlap}(\textsf{fl}(\uuu))$ holds.
	\end{lemma}
	
	\begin{proof}
		By the definition of $\textsf{fl}(\uuu)$, $(\ubx^1, \ubx^2 \oplus \textbf{1} )$ is a \textsc{nae-sat} solution on $\textsf{fl}(\uuu)$ if and only if $(\ubx^1, \ubx^2)$ is a \textsc{nae-sat} solution on $\uuu$. The conclusion comes directly from this symmetry.
	\end{proof}
	Having Lemma \ref{lem:overlap:pairflipsymm} in mind, we now define the set of boundary and union-free component profiles of typical pair of clusters in the near-independence regime. 
	
	\begin{defn}\label{def:Gamma1star:2ndmo}
		For $\beta>1$, denote by $\Gamma^\star_2(\beta)$ the collection of boundary and union-free component profiles $(\bB, \{n_\vvv\}_{\vvv\in \FFF_2}) $ that satisfy the following conditions:
		\begin{enumerate}
			\item $(n_{\vvv})_{\vvv\in\FFF_2}\sim \bB$, where $\bB \in \pbDelta$.
			
			\item $||\bB-\bB^{\star}_{\ula^\star} ||_1 \leq n^{-0.45}$ and $\big|\sum_{\uuu\in\FFF_2^{\tr}}n_{\uuu}\cdot\textnormal{overlap}(\uuu)\big|\leq n^{0,6}$.
			
			\item $\sum_{\vvv \in \FFF_2 \setminus \FFF_2^{\tr} } n_\vvv \le (\log n)^{\beta}$ and $n_{\vvv}=0$ if $v(\vvv)\geq (\log n)^{\beta}$.
		\end{enumerate}
	\end{defn}
    The definition above is similar to $\Gamma_1^\star$ in Definition \ref{def:Gamma1star}. However, the main difference is that we include an extra parameter $\beta>1$, which controls the number and the size of the cyclic union-free components by $(\log n)^{\beta}$. This is in order to have a better error bound as shown in the next lemma, which will in turn be crucial to apply an union bound over locally rewired instances (see Lemma \ref{lem:overlap:unionbd:localflip} below). Since the proof of Lemma \ref{cor:B:close:optimal:2ndmo} relies on the estimates from Appendix \ref{sec:app:apriori}, the proof is deferred to Appendix \ref{sec:app:apriori}.
	
	\begin{lemma}\label{cor:B:close:optimal:2ndmo}
		Let $\beta\in (1,10)$ and denote by $\bZ^{2}_{\ula^\star,\ind} \Big[\big(\Gamma_2^\star(\beta)\big)^{\textsf{c}}\Big]$ the contribution to $\bZ^{2}_{\ula^\star,\ind}$ from $(\bB, \{n_\vvv\}_{\vvv\in\FFF_2})\notin \Gamma_2^{\star}(\beta)$. Then, $$\E \bZ^{2}_{\ula^\star,\ind} \Big[\big(\Gamma_2^\star(\beta)\big)^{\textsf{c}}\Big] \leq \exp\big(-\Omega_k(\log^{\beta} n)\big)\E \bZ^{2}_{\ula^\star,\ind}.$$
	\end{lemma}
	With a slight abuse of notation, we denote $(\CC^1,\CC^2)\in \Gamma_2^\star(\beta)$ if the boundary and union-free component profile induced by $(\ux[\CC^1],\ux[\CC^2])$ is contained in $\Gamma_2^\star(\beta)$. Similarly to $\Gamma_2^\star(\beta)$, we define the typical set of pairs of clusters in the near-identical regime by 
	 \begin{equation}\label{eq:def:gamma:2:id}
	     \Gamma_{2,\id}^\star(\beta)  := \left\{ (\CC^1, \CC^2) \in (\CCC(\GGG))^2: \min  \left\{\zeta(\CC^1, \CC^2), \ \zeta(1\oplus\CC^1, \CC^2)\right\} \le \frac{\log^{\beta}n}{n} \textnormal{ and }\CC^1,\CC^2\in \Gamma_1^\star\right\},
	 \end{equation}
	 where $\zeta(\CC^1,\CC^2):=\zeta\big(\ux[\CC^1],\ux[\CC^2]\big)$ and $1\oplus \CC^1$ denotes the cluster corresponding to the frozen configuration $1\oplus \ux[\CC^1]$. Moreover, similarly to $I_{\bullet}$ defined in \eqref{eq:overlap:clustersize:restriction1}, we define the coarser interval $I_\circ$ and the set of clusters $\CCC^{\tr}_{\circ}(\GGG)$ as follows:
	 \begin{equation}\label{eq:def:CCC:circ}
	\begin{split}
	\CCC^{\tr}_\circ(\GGG) \equiv \CCC^{\tr}(\GGG;I_\circ)&:=\left\{\CC \in \CCC(\GGG) :  \frac{1}{n}\log |\CC| \in I_\circ ,~~\CC \in \Gamma_1^{\star,\tr}\right\},~~\textnormal{ where}\\
     I_\circ \equiv I_\circ(n,d,k) &:= \left[ 
	s^\star - \frac{\log^{3}n}{n}, \ s^\star + \frac{\log^{3}n}{n}\
	\right]\,.
	\end{split}
	\end{equation}
	Define $\Gamma_1^{\star,\tr}$ as the same set as $\Gamma_1^\star$ except that in $(2),(3)$, $(4)$ and $(5)$ in Definition \ref{def:opt:bdry:1stmo} are changed to
	\begin{enumerate}
	    \item[$(2)^\prime$] $(n_{\ttt})_{\ttt\in \FFF_{\tr}}\in \ee_{\frac{1}{4}}$.
	    \item[$(3)^\prime$]  $||B-B^\star_{\lambda^\star} ||_1 \leq 2n^{-0.45}$.
	    \item[$(4)^\prime$] $\Big| \sum_{\ttt\in\FFF_{\tr}} (n_\ttt-np^\star_{\ttt,\la^\star})\cdot\textnormal{overlap}(\ttt)\Big| \le 2n^{0.6}$.
	    \item[$(5)^\prime$]  $n_\fff=0$ if $\fff\in \FFF \setminus \FFF_{\tr}$.
	\end{enumerate}
	The reason for considering $I_\circ$ and $\Gamma_1^{\star,\tr}$ in addition to $I_\bullet$ and $\Gamma_1^\star$ is clarified in Lemma \ref{lem:overlap:unionbd:localflip} below: given $\GGG$ and $\CC \in \CCC_{\bullet}(\GGG)$, we will find a \textit{locally rewired instance} $\GGG^\prime$ such that the cluster corresponding to $\ux[\CC]$ for $\GGG^\prime$ is in $\CCC_{\circ}^{\tr}(\GGG^\prime)$. Then, define the number of bad pairs of clusters $\bN^2_{\circ, \textnormal{bad}}[\GGG]$ by
	\begin{equation*}
	    \bN^2_{\circ, \textnormal{bad}}[\GGG]:=\left|
	\left\{ 
	(\CC^1, \CC^2) \in (\CCC_\circ^{\tr}(\GGG))^2: (\CC^1,\CC^2)\notin \Gamma_2^\star(5)\textnormal{ and } (\CC^1,\CC^2)\notin \Gamma_{2,\textnormal{id}}^\star(5)
	\right\}
	\right|.
	\end{equation*}
Next, recall that for a \textsc{nae-sat} instance $\GGG = (V,F,E, \uL)$, we regard $E$ as an element in $S_{nd}$. Let $\pi_{E}\in S_{nd}$ be the permutation corresponding to $E$. We define the collection $\textsf{Loc}(\GGG)$ of \textit{locally rewired instances} as 
\begin{equation}\label{eq:def:overlap:locallyflipped}
\textsf{Loc}(\GGG) := \Big\{ \GGG' = (V,F,E^\prime, \uL'): \big|\{1\leq i \leq nd:\pi_{E^\prime}(i)\neq \pi_{E}(i)\}\big|\leq \log^{2}n\textnormal{ and }\tL^\prime_{e}=\tL_{e}\textnormal{ if } e\in E\cap E^\prime \Big\}.
\end{equation}
Finally, define the collection of \textsc{nae-sat} instances $\mathfrak{G}_2$ as
\begin{equation}\label{def:mathfrak:G:2}
\mathfrak{G}_2:=\Big\{\GGG:~\bN^2_{\circ, \textnormal{bad}}[\GGG^\prime]=0~\textnormal{holds for arbitrary }\GGG^\prime \in \textsf{Loc}(\GGG)\Big\}
\end{equation}
\begin{lemma}\label{lem:overlap:unionbd:localflip}
For $\mathfrak{G}_2$ defined above, $\P(\GGG \notin \mathfrak{G}_2)\lesssim_{k}\exp\big(-\Omega_{k}(\log^{5}n)\big)$ holds.
\end{lemma}
\begin{proof}
By definition of $\bN^{2}_{\circ,\textnormal{bad}}\equiv \bN^{2}_{\circ,\textnormal{bad}}[\GGG]$, we can bound
\begin{multline*}
\E \bN^{2}_{\circ,\textnormal{bad}}\leq e^{-2n\la^\star s^\star+\la^\star \log^{3}n}\E \bZ^{2}_{\ula^\star, \ind}\Big[\big(\Gamma_2^\star(5)\big)^{\textsf{c}}\Big]\\
+\sum_{\bs\in (I_{\circ})^{2}}\bigg(\E \bN^2_{\bs,\textnormal{int}}+\E\bN^2_{\bs,\id}\Big[\min\big\{\zeta(\ux^{1},\ux^{2}),\zeta(1\oplus\ux^{1},\ux^{2})\big\}\geq \frac{\log^{5}n}{n}\Big]\bigg),
\end{multline*}
where $\bN^2_{\bs,\textnormal{int}}$ and $\bN^2_{\bs,\id}$ are defined in \eqref{eq:def:corrN}. By Proposition \ref{prop:2ndmo:correlated overlap} and Lemma \ref{lem:identical regime decay estim}, the last sum in the \textsc{rhs} of the equation above can be bounded above by $e^{-\Omega(n2^{-k/2})}+e^{-\Omega(k\log^{5}n)}\sum_{\bs\in(I_{\circ})^2}\E\bN_{\bs}$. In addition, Theorem \ref{thm:1stmo:constant:s} and Proposition \ref{prop:ratio:uni:1stmo} shows $\sum_{\bs\in(I_{\circ})^2}\E\bN_{\bs}\leq e^{O_k(\log^{3}n)}$. Combining with Lemma \ref{cor:B:close:optimal:2ndmo}, we have
\begin{equation*}
    \E \bN^{2}_{\circ,\textnormal{bad}}\leq \exp\big(-\Omega_{k}(\log^{5}n)\big).
\end{equation*}
Now, note that by definition of $\textsf{Loc}(\GGG)$ in \eqref{eq:def:overlap:locallyflipped}, we have that $\GGG^\prime \in \textsf{Loc}(\GGG)$ if any only if $\GGG \in \textsf{Loc}(\GGG^\prime)$, and $\big|\textsf{Loc}(\GGG)\big|\leq \exp\big(O_k(\log^{3} n)\big)$. Thus, Markov's inequality and a union bound show that
\begin{equation*}
\P(\GGG \notin \mathfrak{G}_2)\leq \exp\big(-\Omega_{k}(\log^{5}n)\big)\,,
\end{equation*}
which concludes the proof.
\end{proof}

	We now introduce the final collection of \textsc{nae-sat} instances $\mathfrak{G}_3$. Recall that we denoted $s_{\circ}(C):=s^\star-\frac{1}{2\la^\star}\frac{\log n}{n}+\frac{C}{n}$ for $C\in\Z$. Abbreviate $\bN^2_{s,\textnormal{cor}}\equiv \bN^{2}_{(s,s),\id}+\bN^2_{(s,s),\textnormal{int}}$ for $s\in [0,\log 2)$. For $C\in \Z$ and $(C_{k,i})_{i=1,2,3,4}>0$, constants depending only on $k$, let $\mathfrak{G}_3\big(C,(C_{k,i})_{i=1,2,3,4}\big)$ be the collection of \textsc{nae-sat} instances which satisfy the following $3$ conditions:
	\begin{enumerate}
     	\item $Z\leq C_{k,1}n^{-\frac{1}{2\la^\star}}e^{ns^\star}$, where $Z$ is the total number of solutions; 
	    \item $\bN^{\tr}_{s_{\circ}(-C)}[\ee_{\frac{1}{4}}]\in [C_{k,2} e^{\la^\star C},C_{k,3} e^{\la^\star C} ]$;
	    \item $\bN^{2}_{s_{\circ}(-C),\textnormal{cor}}\leq C_{k,4}e^{\la^\star C}$.
	\end{enumerate}
	The event $\mathfrak{G}_3\big(C,(C_{k,i})_{i=1,2,3,4}\big)$ is designed based on the following aspects:
	\begin{itemize}
		\item Let $\CCC^{\tr}_{s_{\circ}(-C)}[\ee_{\frac{1}{4}}]\equiv \CCC^{\tr}_{s_{\circ}(-C)}[\ee_{\frac{1}{4}}](\GGG)$ denote the set of clusters that are represented by $\bN^{\tr}_{s_{\circ}(-C)}[\ee_{\frac{1}{4}}]$. Then, on the event $\GGG\in \mathfrak{G}_3\big(C,(C_{k,i})_{i=1,2,3,4}\big)$, the first and second condition imply that a randomly generated solution would be from a cluster in $\CCC^{\tr}_{s_{\circ}(-C)}[\ee_{\frac{1}{4}}]$ with positive probability. That is,
		\begin{equation}\label{eq:overlap:in:s:circ:positive}
		    \P\Big(\CC^1,\CC^2\in \CCC^{\tr}_{s_{\circ}(-C)}[\ee_{\frac{1}{4}}] ~\Big|~\GGG \Big)\geq (C_{k,1})^{2} (C_{k,2})^{-2}e^{-2(1-\la^\star)C}.
		\end{equation}
		\item The second and third condition imply that conditioned on the event where two independently drawn solutions are from clusters in $\CCC^{\tr}_{s_{\circ}(-C)}[\ee_{\frac{1}{4}}]$, the event of two clusters being near-independent or near-identical would both happen with positive probability for large enough $C$. That is,
		\begin{equation}\label{eq:overlap:indep:cor:positive:probability}
		\begin{split}
		    &\P\Big(\CC^1=\CC^2~\Big|~ \CC^1,\CC^2\in \CCC^{\tr}_{s_{\circ}(-C)}[\ee_{\frac{1}{4}}]\Big)\geq e^{-2}(C_{k,3})^{-1}e^{-\la^\star C};\\
		    &\P\bigg(\Big| \zeta(\CC^1,\CC^2)-\frac{1}{2}\Big|\leq k^2 2^{-k/2} ~\bigg|~ \CC^1,\CC^2\in \CCC^{\tr}_{s_{\circ}(-C)}[\ee_{\frac{1}{4}}]\bigg)\geq 1-e^2 C_{k,4}(C_{k,2})^{-2}e^{-\la^\star C},
		\end{split}
		\end{equation}
		where the factors $e^{-2}$ and $e^2$ come from the fact that the sizes of the clusters in $\CCC^{\tr}_{s_{\circ}(-C)}[\ee_{\frac{1}{4}}]$ can differ by at most a factor of $e$.
	\end{itemize}
	
	\begin{lemma}\label{lem:overlap:G:3}
		There exist constants $(C_{k,i})_{i=1,2,3,4}>0$, which only depend on $k$ and satisfy $C_{k,2}<C_{k,3}$, such that the following holds: for every $C>0$, there exists $\delta(C)\equiv \delta(C,\alpha,k)>0$ such that $\mathfrak{G}_3(C)\equiv \mathfrak{G}_3\big(C,(C_{k,i})_{i=1,2,3,4}\big)$ satisfies
		\begin{equation*}
		\P \big(\GGG \in  \mathfrak{G}_3(C)\big) \ge \delta(C).
		\end{equation*}
	\end{lemma}
	\begin{proof}
	Proceeding in the same fashion as in the proof of Theorem \ref{thm1}-(b),(c) (cf. \eqref{eq:thm1:first:event}), there exist constants $C_{k,2}$ and $\delta\equiv \delta(\alpha,k)$, which depend only on $k$, such that for every $C>0$ and $n\geq n_0(C,\alpha,k)$,
	\begin{equation}\label{eq:mathfrak:G:3:C:2}
	    \P\left(\bN^{\tr}_{s_\circ(-C)}[\ee_{\frac{1}{4}}]\geq C_{k,2} e^{\la^\star C} \right)\geq 4\delta.
	\end{equation}
	Moreover, $\E\bN^{\tr}_{s_{\circ}(-C)}\lesssim_{k} e^{\la^\star C}$ holds. Thus, if we take $C_{k,3}$ to be large enough constant depending only on $k$ so that we have $\E\bN^{\tr}_{s_{\circ}(-C)}\leq \delta C_{k,3} e^{\la^\star C}$, Markov's inequality shows
	\begin{equation}\label{eq:mathfrak:G:3:C:3}
	    \P\left(\bN^{\tr}_{s_\circ(-C)}[\ee_{\frac{1}{4}}]\geq C_{k,3}e^{\la^\star C} \right)\leq \delta.
	\end{equation}
	Also, Theorem \ref{thm1}-(a) implies that there exists a constant $C_{k,1}$ depending only on $k$ such that we have
	\begin{equation}\label{eq:mathfrak:G:3:C:1}
	    \P\left(Z \geq C_{k,1} n^{-\frac{1}{2\la^\star}}e^{ns^\star} \right)\leq \delta.
	\end{equation}
	Finally, recall that $\E \bN^{2}_{s_{\circ}(-C),\textnormal{cor}}\lesssim_{k} e^{\la^\star C}$ holds by Proposition \ref{prop:2ndmo:correlated overlap}. Thus, proceeding in the same fashion as in \eqref{eq:mathfrak:G:3:C:3}, take $C_{k,4}$ large enough so that we have
	\begin{equation}\label{eq:mathfrak:G:3:C:4}
	    \P\left(\bN^{2}_{s_{\circ}(-C),\textnormal{cor}}\geq C_{k,4}e^{\la^\star C} \right)\leq \delta.
	\end{equation}
	Therefore, \eqref{eq:mathfrak:G:3:C:2}-\eqref{eq:mathfrak:G:3:C:4} conclude the proof.
	\end{proof}

	\begin{remark}
		In the companion paper \cite{nss2}, we strengthen the second moment method to show that for any $\varepsilon>0$, there exists $C_{\eps}\equiv C(\eps,\alpha,k)>0$ and $\delta_{\eps}\equiv \delta(\eps,\alpha,k)>0$ such that $\P(\bN^{\tr}_{s_{\circ}(-C_{\eps})}\geq \delta_{\eps} \E\bN^{\tr}_{s_{\circ}(-C_{\eps})}) \ge 1-\eps$. Note that $$\P\Big(\bN^{\tr}_{s_{\circ}(-C)}\big[(\ee_{\frac{1}{4}})^{\textsf{c}}\big]\geq 1 \Big)\leq \E\bN^{\tr}_{s_{\circ}(-C)}\big[(\ee_{\frac{1}{4}})^{\textsf{c}}\big]\lesssim_{k} e^{\la^\star C} n^{-3/2}\log n,$$ where the last inequality is due to Proposition \ref{prop:1stmo:aprioriestimate} and Theorem \ref{thm:1stmo:constant}. Hence, repeating the proof of Lemma \ref{lem:overlap:G:3}, there exist constants $(C_{k,\eps,i})_{i=1,2,3,4}>0$, which only depend on $k$ and $\eps$, such that $\mathfrak{G}_3\equiv \mathfrak{G}_3\big(C_{\eps},(C_{k,\eps,i})_{i=1,2,3,4}\big)$ satisfies $\P\big(\GGG\in \mathfrak{G}_3\big)\geq 1-2\eps$. Moreover, we will see later in the proof of Theorem \ref{thm2} below that $\GGG\in \mathfrak{G}_1 \cap\mathfrak{G}_2 \cap\mathfrak{G}_3$ implies $(a),(b),(c)$ of Theorem \ref{thm2}. Therefore, the strengthened second moment method in \cite{nss2} immediately implies the strengthened version of Theorem \ref{thm2}, where we push the probability with respect to $\GGG$ to $1-\eps$. 
	\end{remark}

	\subsection{Locally rewired instances}
	\label{subsec:overlap:locrewire}
	In this subsection, we clarify how the locally flipped instances defined above are used to control the clusters in $\CCC_\bullet(\GGG)\setminus \CCC^{\tr}_{\circ}(\GGG)$. The main observation is summarized by the following lemma.
	
	\begin{lemma}\label{lem:overlap:localflip:existence}
		Let $\GGG$ be a \textsc{nae-sat} instance and let $\CC^1, \CC^2 \in \CCC_\bullet (\GGG)$ be two arbitrary clusters in $\CCC_\bullet (\GGG)$. Denote by $\ux^{1}\in \{0,1,\ff\}^{V}$ and $\ux^{2}\in \{0,1,\ff\}^{V}$ the frozen configurations corresponding to $\CC^{1}$ and $\CC^{2}$ in $\GGG$ respectively. Then, there exists a locally rewired instance $\GGG'\in \textnormal{\textsf{Loc}}(\GGG)$ such that
		\begin{equation}\label{eq:overlap:flipped clusters are tree}
		\acute{\CC}^{1},\acute{\CC}^{2} \in \CCC^{\tr}_\circ(\GGG'),
		\end{equation}
		where $\acute{\CC}^{i}$ denotes the cluster corresponding to $\ux^{i}$ in the \textsc{nae-sat} instance $\GGG^\prime$ for $i=1,2$.
	\end{lemma}
	
	\begin{proof}
	We aim to construct a modified graph $\GGG^\prime \in \textsf{Loc}(\GGG)$ that satisfies \eqref{eq:overlap:flipped clusters are tree} through an edge-swapping process, which can be described as follows. Take an edge $e=(av)\in E(\GGG)$ such that $e$ is contained in a free cycle of either $\ux^{1}$ or $\ux^{2}$. For each such $e$, we identify another edge $e_0=(a_0 v_0)$ such that $a_0$ is separating, but non-forcing, in both $\ux^{1}$ and $\ux^{2}$ with $\ux^{i}_{\delta a}\in \{0,1\}^{k}$ for $i=1,2$. Henceforth, we will call such an edge $e_0$ to be a \textbf{good} edge. Then, we delete $e=(av)$ and $e_0=(a_0v_0)$ from the graph $\GG$ and form new edges, $e^\prime=(av_0)$ and $e^\prime_0=(a_0 v)$ in $\GG^\prime$. Also, we will choose literals on $e^\prime$ and $e^\prime$ so that $\ux^{1}$ and $\ux^{2}$ are valid frozen configurations in $\GGG^\prime$. By this process, we will delete every free cycle of either $\ux^1$ or $\ux^2$ in $\GGG$, so that $\ux^1$ and $\ux^2$ do not contain any free cycles in $\GGG^\prime$.
	
	In such a process, the potentially problematic case is when $e$ is an forcing edge in either $\ux^1$ or $\ux^2$. However, we can avoid such a case by choosing an appropriate edge in the free cycle of the other copy: assume that $e=(av)$ is forcing in $\ux^2$ and is contained in a free cycle $(a_1v_1a_2v_2...a_{\ell}v_{\ell})$ of $\ux^{1}$, where $a_1=a, v_1=v$ and $v_{\ell}\neq v$. Then, instead of performing edge-swapping process with $e$ and $e_0$ mentioned above, we can perform the process with $\tilde{e}\equiv (av_{\ell})$ and $e_0$ to delete the free cycle $(a_1v_1a_2v_2...a_{\ell}v_{\ell})$. Note that $\tilde{e}$ cannot be forcing in $\ux^2$ since $e=(av)$ is forcing in $\ux^2$.
	
	To this end, choose an edge for each free cycle in either $\ux^{1}$ or $\ux^{2}$ such that it is not forcing in both $\ux^1$ and $\ux^2$. Let $E_{\textsf{cyc}}\subset E(\GGG)$ be the resulting collection of the edges. Note that by definition of $\CCC_{\bullet}(\GGG)$, the number of free cycles in $\ux^{1}$ and $\ux^2$ is at most $\log n$, so $|E_{\textsf{cyc}}|\leq 2\log n \ll \log^{2}n$. On the other hand, the number of clauses that are separating, but non-forcing for $\ux^1$ and $\ux^2$, and also have no free variables in their neighbors for both copies is at least linear in $n$. This is because the number of edges that are forcing or next to free variables in either copy is bounded above by $\frac{28k}{2^k}m\ll m$. Thus, we can choose $|E_{\textsf{cyc}}|$ number of good edges. Denote by $E_{\textsf{good}}$ the set of chosen good edges.
	
	It can indeed be verified that the edge-swapping process can be performed with every pair $(e,e_0)=\big((av),(a_0v_0)\big)\in E_{\textsf{cyc}}\times E_{\textsf{good}}$ by choosing appropriate literals on the new edges $e^\prime=(av_0)$ and $e_0=(a_0 v)$ so that $\ux^1$ and $\ux^2$ are valid frozen configurations in the resulting \textsc{nae-sat} instance $\GGG^\prime$ without any free cycles. For example, without loss of generality, assume that $(x_{v_0}^1,x_{v_0}^{2})=(0,0)$ and consider the case where $a$ is a forcing clause in $\ux^2$. By our construction, $e$ is non-forcing in $\ux^2$ and contained in a free cycle in $\ux^1$. Then, choose the literal on the new edge $(av_0)$ to be $\tL_{(av_0)}=\tL_e\oplus x_{v}^2$ so that $a$ remains forcing in $\ux^2$. Note that in the resulting $\GGG^\prime$, $a$ may be separating with respect to $\ux^1$. In such a case, the cyclic free component containing $a$ in $(\GGG,\ux^1)$ is decomposed into (possibly many) free trees in $(\GGG^\prime,\ux^1)$. Regarding the new edge $(a_0v)$, set its literal to be $\tL_{(a_0v)}=\tL_{e_0}\oplus x_{v}^{2}$ so that $a_0$ remains to be separating, but non-forcing, in both copies with respect to $\GGG^\prime$. Other cases of $a$ being either non-separating or separating, but non-forcing, in $\ux^2$ can be verified in a similar fashion.
	
	Note that in the above edge-swapping process, we have changed $2|E_{\textsf{cyc}}|\leq 4\log n$ number of edges, so $\GGG^\prime \in \textsf{Loc}(\GGG)$ holds. Moreover, we deleted all cyclic free components in $(\GGG,\ux^{i}), i=1,2$, and possibly produced more free trees in $(\GGG^\prime, \ux^{i}), i=1,2$. By definition of $\CCC_{\bullet}(\GGG)$, all cyclic free components of $\ux^1$ and $\ux^2$ were unicylic with size at most $\frac{3\log n}{k\log 2}$ and the total number of them was at most $\log n$ to begin with. Hence, the size of $\ux^i, i=1,2$ can change by at most $\exp\big(O_k(\log^{2}n)\big)\ll \exp(\log^{3}n)$. Also, the free tree profile $(n_{\ttt}^{i})_{\ttt\in \FFF_{\tr}}$ of $i$'th copy can change by $O_k(\log^{2}n)$ in $\ell^1$ distance and the analog holds for the boundary profile $B^i$ for $i=1,2$. Therefore, the conclusion \eqref{eq:overlap:flipped clusters are tree} holds for the constructed $\GGG^\prime \in \textsf{Loc}(\GGG)$.
	\end{proof}
	
% 	Now we turn back to the original problem of selecting two solutions independently and uniformly. Suppose that $\GGG\in \mathfrak{G}_1 \cap \mathfrak{G}_2$. As we discussed in the beginning of this section, we first select two clusters $\CC^1, \CC^2$ independently, proportional to their sizes. Then, by the definition of $\mathfrak{G}_1$,  the probability of selection, denoted by $\PP$, satisfies
% 	\begin{equation}\label{eq:overlap:clustersampling:restriction1}
% 	\PP ( \CC^1, \CC^2 \in \CCC^{\tr}_\bullet (\GGG) \cup \CCC^{\cyc}_\bullet(\GGG)\ | \ \GGG \in \mathfrak{G}_1 ) \ge 1- n^{-1/3}. 
% 	\end{equation}
% 	Furthermore, for such $\CC^1$ and $\CC^2 $, we can choose $\GGG' = (\GG, \uL') \in \textsf{Flip}(\GGG)$ given by the above lemma such that $\acute{\CC}^1:= \CC^1 \oplus \uL \oplus \uL'$ and $\acute{\CC}^2 := \CC^2 \oplus \uL \oplus \uL'$ are both in $\CCC^{\tr}_\circ(\GGG')$. Thus, we first  study the case when both clusters are from $\CCC^{\tr}_\circ(\GGG)$.
% 	\begin{equation*}
% 	\begin{split}
% 	\prescript{}{2}{\CCC}_\circ^{\star}(\GGG) := \left\{ (\CC^1, \CC^2) \in (\CCC^{\tr}_\circ(\GGG))^2: (B, \{n_\uuu \}_{\uuu \in \FFF_2}) \in \Gamma_2^\star(C, \ula^\star)
% 	 \right\};\\
% 	 \prescript{}{2}{\CCC}_\circ^{\textnormal{id}}(\GGG) := \left\{ (\CC^1, \CC^2) \in (\CCC^{\tr}_\circ(\GGG))^2: (B, \{n_\uuu \}_{\uuu \in \FFF_2}) \in \Gamma_2^{\textnormal{id}}(C, \ula^\star)
% 	 \right\}.
% 	\end{split}
% 	\end{equation*}
	
	\subsection{Sampling random solutions from typical clusters}
	\label{subsec:overlap:samplesol}
	Having Lemma \ref{lem:overlap:localflip:existence} in hand, we now aim to analyze the overlap of random solutions sampled from $\CC^1$ and $\CC^2$, where $\CC^1,\CC^2\in \CCC_{\circ}^{\tr}(\GGG)$. Recalling Lemma \ref{lem:overlap:unionbd:localflip}, it suffices to consider the cases where $(\CC^1,\CC^2)$ belongs to either $\Gamma_2^\star(5)$ or $\Gamma_{2,\id}^\star(5)$.
	\begin{prop}\label{prop:overlap:tree}
Let $\CC^1$ and $\CC^2$ be arbitrary clusters from $\CCC^{\tr}_\circ(\GGG).$ Suppose that $\ubx^1, \ubx^2 \in \{0,1 \}^V$ are two solutions drawn independently and uniformly from $\CC^1$ and $\CC^2$, respectively. Denote the probability with respect to this sampling by $\PP$. Then, the overlap $\rho(\ubx^1, \ubx^2)$ satisfies the following:
		\begin{equation}\label{eq:overlap:pair:twocases}
		\begin{split}
&\textnormal{If } (\CC^1, \CC^2)\in \Gamma_2^\star(5), \textnormal{ then }
		\PP \left(  \big|\rho(\ubx^1,\ubx^2)\big| \ge n^{-0.35}  \right) \le \exp\left(-\Omega_k(n^{1/4}) \right);\\
	&	\textnormal{If } (\CC^1, \CC^2)\in \Gamma_{2,\id}^\star(5), \textnormal{ then }
		\PP \left(  \Big| \big|\rho(\ubx^1,\ubx^2)\big| -p^\star \Big| \ge n^{-0.35}  \right) \le \exp\left(-\Omega_k(n^{1/4}) \right).
		\end{split}
		\end{equation}
	\end{prop}
		\begin{proof}[Proof of Proposition \ref{prop:overlap:tree}]
		First, we aim to prove the first inequality of \eqref{eq:overlap:pair:twocases}.
		Let $(\CC^1, \CC^2)\in \Gamma_2^\star(5)$ and  let $(\ux^1, \ux^2)$ be the pair frozen configuration on $\GGG$ induced by $(\CC^1, \CC^2)$. Let $(\bB, \{n_\vvv\}_{\vvv\in \FFF_2})$ denote the boundary and union-free component profile of $(\ux^1, \ux^2)$. Also, abbreviate $\srr\equiv \{0,1\}$ and define $\pi_{\srr\srr^=}, \pi_{\srr\srr^{\neq}}$ by
		\begin{equation}\label{eq:def:overlap:pifrozen}
		\pi_{\srr\srr^=} := \frac{1}{n}\sum_{v\in V} \one \{ x^1_v = x^2_v \in \{0,1\} \}, \quad \pi_{\srr\srr^{\neq}} := \frac{1}{n} \sum_{v\in V} \one \{ (x^1_v, x^2_v) = (0,1) \textnormal{ or } (1,0) \}.
		\end{equation}
		Note that filling out \textsc{nae-sat} solutions on one of the union-free components $\vvv$ embedded in $\GGG$ has no effect in filling out \textsc{nae-sat} solutions on the other union-free components. Thus, sampling \textsc{nae-sat} solutions $(\ubx^1,\ubx^2)\in \{0,1\}^{2V}$ from $(\ux^1,\ux^2)$ can be done independently for each union-free component. In particular, if we denote by $X^{\vvv}_i$ the overlap of two uniformly chosen solutions on $i$'th $\vvv$ in $(\ux^1,\ux^2)$, where $1\leq i \leq n_{\vvv}$ and $\vvv\in \FFF_2$, then $(X^{\vvv}_i)_{\vvv\in \FFF_2,i\leq n_{\vvv}}$ are independent with $\E[X^\vvv_i]=\textnormal{overlap}(\vvv)$, where $\textnormal{overlap}(\vvv)$ denotes the average overlap in $\vvv$ as in \eqref{eq:def:overlap:uniontree}. Moreover, we can express $\rho(\ubx^1,\ubx^2)$ by
		\begin{equation}\label{eq:express:overlap}
		\rho(\ubx^1,\ubx^2)=\pi_{\srr\srr^{=}}-\pi_{\srr\srr^{\neq}}+\frac{1}{n}\sum_{\vvv\in \FFF_2}\sum_{i=1}^{n_{\vvv}}X^{\vvv}_i.
		\end{equation}
		Note that since $(\bB, \{n_\vvv\}_{\vvv\in \FFF_2})\in \Gamma_2^\star(5)$, the contribution to the sum in the above equation from cyclic union-free components is small. Namely, we can crudely bound $\frac{1}{n}\sum_{\vvv\in \FFF_2\setminus \FFF_2^{\tr}}\sum_{i=1}^{n_{\vvv}}X^{\vvv}_i\leq \frac{\log^{10}n}{n}\ll n^{-0.35}$ by (3) of Definition \ref{def:opt:bdry:2ndmo}. Also, (2) of Definition \ref{def:opt:bdry:2ndmo} shows 
		\begin{equation*}
		|\pi_{\srr\srr^{=}}-\pi_{\srr\srr^{\neq}}|\leq n^{-0.45}\ll n^{-0.35}\quad\textnormal{and}\quad \frac{1}{n}\bigg|\sum_{\uuu\in \FFF_2^{\tr}}\sum_{i=1}^{n_{\uuu}}\E[X^{\uuu}_i]\bigg|\leq n^{-0.4}\ll n^{-0.35}.
		\end{equation*}
		Thus, because $X^{\uuu}_i\leq \log^{5}n$ holds, Hoeffding's bound shows that for large enough $n$,
		\begin{equation}\label{eq:overlap:pair:firstcase}
		    	\PP \left(  \big|\rho(\ubx^1,\ubx^2)\big| \ge n^{-0.35}  \right) \leq \PP\bigg( \frac{1}{n}\Big|\sum_{\uuu\in \FFF_2^{\tr}}\sum_{i=1}^{n_{\uuu}}\Big(X^{\uuu}_i-\E[X^{\uuu}_i]\Big)\Big|\geq 0.5 n^{-0.35}\bigg)\leq\exp\bigg(-\Omega_k\Big(\frac{n^{0.3}}{\log^{10} n}\Big)\bigg).
		\end{equation}

		Next, we aim to prove the second inequality of \eqref{eq:overlap:pair:twocases}. Let $(\CC^1, \CC^2)\in \Gamma_{2,\id}^\star(5)$ and  let $(\ux^1, \ux^2)$ be the pair frozen configuration on $\GGG$ induced by $(\CC^1, \CC^2)$. Let $\pi_{\srr\srr^{=}}$ and $\pi_{\srr\srr^{\neq}}$ as in \eqref{eq:def:overlap:pifrozen}, and define $\pi_{\srr\ff}$ and $\pi_{\ff\srr}$ by
		\begin{equation}\label{eq:def:overlap:pifrozen:free}
		\pi_{\srr\ff}:= \frac{1}{n} \sum_{v\in V} \one \{x_v^1\in \{0,1\} \textnormal{ and } x_v^2 = \ff \}, \quad \pi_{\ff\srr}:= \frac{1}{n} \sum_{v\in V} \one \{x_v^1=\ff \textnormal{ and } x_v^2 \in \{0,1\} \}.
		\end{equation}
			Then, the definition of $\Gamma_{2,\id}^\star(5)$ in \eqref{eq:def:gamma:2:id} shows that either one of the following holds:
	\begin{equation}\label{eq:overlap:id:restriction}
		|\pi_{\srr\srr^{\neq} } + \pi_{\srr\ff} + \pi_{\ff\srr}| \le \frac{\log^{5}n}{n},  \ \textnormal{ or } \ |\pi_{\srr\srr^=} + \pi_{\srr\ff} + \pi_{\ff\srr}| \le \frac{\log^{5}n}{n}.
	\end{equation} 
	Without loss of generality, we assume the former case and prove that the overlap concentrates around $p^\star$. It will be clear from the proof that in the latter case, the overlap concentrates around $-p^\star$.

	To this end, assume $	|\pi_{\srr\srr^{\neq} } + \pi_{\srr\ff} + \pi_{\ff\srr}| \le \frac{\log^{5}n}{n}$ and let $\big(\bB,(n_{\vvv})_{\vvv\in \FFF_{2}})\big)$ be the boundary and union-free component profile of $(\ux^1,\ux^2)$. Since $\CC^1,\CC^2\in \CCC_{\circ}^{\tr}(\GGG)$, the boundary and free tree profile of both copies $(B^i,(n_{\ttt}^{i})_{\ttt\in \FFF_{\tr}}), i=1,2,$ are contained in $\Gamma_1^{\star,\tr}$. Thus, $||B^{1}-B^\star_{\la^\star}||_1\leq 2n^{-0.45}$ implies that
	\begin{equation}\label{eq:overlap:pair:secondcase:1}
	    \Big|\pi_{\srr\srr^{=}}+\pi_{\srr\srr^{\neq}}+\pi_{\srr\ff}-\big(1-\sum_{\ttt\in \FFF_{\tr}}v_{\ttt}p^\star_{\ttt,\la^\star}\big)\Big|\leq 2n^{-0.45}.
	\end{equation}
	Hence, $\big|\pi_{\srr\srr^{=}}-\pi_{\srr\srr^{\neq}}-\big(1-\sum_{\ttt\in \FFF_{\tr}}v_{\ttt}p^\star_{\ttt,\la^\star}\big)\big|\leq 3n^{-0.45}$ holds for large enough $n$.
	
	We now argue that most of the union-free components of $(\ux^1,\ux^2)$ are union-free trees $\uuu$ such that $\uuu$ is generated from a single free tree $\ttt\in \FFF_{\tr}$ by copying $\ttt$ in the second copy and merging them together. To this end, we embed $\FFF_{\tr}\subset \FFF_2^{\tr}$ by making two copies of $\ttt\in \FFF_{\tr}$ and merging them together. Let $\uuu(\ttt)$ be the resulting union-free tree. Then, the graph of $\uuu(\ttt)$ is the same as the graph of $\ttt$ and all of the boundary colors of $\uuu(\ttt)$ are given by either $\bb_0\bb_0, \bb_1\bb_1$ or $\fs\fs$ (in the other case where $|\pi_{\srr\srr^=} + \pi_{\srr\ff} + \pi_{\ff\srr}| \le \frac{\log^{5}n}{n}$, we flip the boundary colors of $\ttt$ in one of the copy and merge them together so that the boundary colors of $\uuu(\ttt)$ are either $\bb\bb^{\neq}$ or $\fs\fs$).
	
	Indeed, for a union-free component $\vvv$ in $(\ux^1,\ux^2)$, there could be two cases where $\vvv\not\in \FFF_{\tr}$: first is when $\vvv$ indeed has only one free tree in both copies, but one of its boundary color is $\bb\bb^{\neq}$. The second case is when there exist $\ff\srr$ or $\srr\ff$ variables inside $\vvv$. In the first case, the $\bb\bb^{\neq}$ variable can neighbor at most $d$ free trees, which have size at most $\frac{3\log n}{k\log 2}$. In the second case, the number of $\ff\ff$ variables in $\vvv$ is bounded above by $\frac{3\log n}{k\log 2}$ times the number of $\ff\srr$ and $\srr\ff$ variables by matching each $\ff\ff$ variable in $\vvv$ to the $\ff\srr$ or $\srr\ff$ variable in the same free tree. Therefore, we can bound
	\begin{equation}\label{eq:overlap:pair:secondcase:2}
	    \frac{1}{n}\sum_{\vvv\in \FFF_2\setminus \FFF_{\tr}}n_{\vvv}v(\vvv)\leq \frac{3d\log n}{k\log 2}\pi_{\srr\srr^{\neq}}+\Big(\frac{3\log n}{k\log 2}+1\Big)\big(\pi_{\srr\ff}+\pi_{\ff\srr}\big)\lesssim_{k}\frac{\log^{6}n}{n}\ll n^{-0.35}\,,
	\end{equation}
 where the last inequality holds due to \eqref{eq:overlap:id:restriction}. Having \eqref{eq:overlap:pair:secondcase:1} and \eqref{eq:overlap:pair:secondcase:2} in hand, we can proceed similarly as in \eqref{eq:overlap:pair:firstcase}:  as before, denote by $X^{\vvv}_i$ the overlap of two uniformly chosen solutions on $i$'th $\vvv$ in $(\ux^1,\ux^2)$, where $1\leq i \leq n_{\vvv}$ and $\vvv\in \FFF_2$. Note that by definition of $\textnormal{overlap}(\ttt)$ in \eqref{eq:def:overlap:ham}, $\E[X_i^{\uuu(\ttt)}]=\textnormal{overlap}(\ttt)$ holds. Moreover, \eqref{eq:overlap:pair:secondcase:2} implies $\frac{1}{n}\sum_{\ttt\in \FFF_{\tr}}|n^1_{\ttt}-n_{\uuu(\ttt)}|v_{\ttt}\lesssim_{k}\frac{\log^{6}n}{n}$. Thus, since $\big| \sum_{\ttt\in\FFF_{\tr}} (n_\ttt^1-np^\star_{\ttt,\la^\star})\cdot\textnormal{overlap}(\ttt)\big| \le 2n^{0.6}$ holds from the fact $(B^1,(n_{\ttt}^1)_{\ttt\in\FFF_{\tr}})\in \Gamma_1^{\star,\tr}$, we have
	\begin{equation}\label{eq:overlap:pair:secondcase:3}
	    \bigg|\frac{1}{n}\sum_{\ttt\in \FFF_{\tr}}\sum_{i=1}^{n_{\uuu(\ttt)}}\E[X^{\uuu(\ttt)}_i]-\sum_{\ttt\in\FFF_{\tr}}p^\star_{\ttt,\la^\star}\cdot\textnormal{overlap}(\ttt)\bigg|\lesssim_{k} n^{-0.4}\ll n^{-0.35}.
	\end{equation}
	Recalling the expression of $\rho(\ubx^1,\ubx^2)$ in \eqref{eq:express:overlap}, we can use \eqref{eq:overlap:pair:secondcase:1}, \eqref{eq:overlap:pair:secondcase:2}, \eqref{eq:overlap:pair:secondcase:3}, and Hoeffding's bound to show that for large enough $n$,
	\begin{equation}
	     	\PP \left(  \big|\rho(\ubx^1,\ubx^2)-p^\star \big| \ge n^{-0.35}  \right) \leq \PP\bigg( \frac{1}{n}\Big|\sum_{\ttt\in \FFF_{\tr}}\sum_{i=1}^{n_{\uuu(\ttt)}}\Big(X^{\uuu(\ttt)}_i-\E[X^{\uuu(\ttt)}_i]\Big)\Big|\geq 0.5 n^{-0.35}\bigg)\leq\exp\bigg(-\Omega_k\Big(\frac{n^{0.3}}{\log^{2} n}\Big)\bigg),
	\end{equation}
	which concludes the proof.
	\end{proof}
	To conclude this section, we prove Theorem \ref{thm2}.
	
	\begin{proof}[Proof of Theorem \ref{thm2}]
	By Lemmas \ref{lem:mathfrak:G:1}, \ref{lem:overlap:unionbd:localflip} and \ref{lem:overlap:G:3}, it suffices to show that (a),(b),(c) of Theorem \ref{thm2} hold if $\GGG\in \mathfrak{G}_1\cap \mathfrak{G}_2\cap \mathfrak{G}_3(C)$ for some constant $C>0$. In particular, we take $C$ to be a large enough constant depending on $k$ so that $e^2 C_{k,4}(C_{k,2})^{-2}e^{-\la^\star C}<\frac{1}{2}$ holds in \eqref{eq:overlap:indep:cor:positive:probability}. To this end, we assume that $\GGG\in \mathfrak{G}_1\cap \mathfrak{G}_2\cap \mathfrak{G}_3(C)$. By definition of $\mathfrak{G}_3(C)$, $\bN^{\tr}_{s_{\circ}(-C)}\geq 1$ holds, so the total number of solutions is at least $Z\geq e^{-C}n^{-\frac{1}{2\la^\star}}e^{ns^\star}$. Thus, if we denote by $\P\big((\CC^1,\CC^2)\in \cdot \mid \GGG\big)$ the probability measure given by sampling clusters $\CC^1,\CC^2$ independently and proportionally to their sizes, the definition of $\mathfrak{G}_1$ in \eqref{eq:overlap:clustersize:restriction1} implies
	\begin{equation}\label{eq:CCC:bullet:whp}
	    \P\big(\CC^1\notin \CCC_{\bullet}(\GGG)\textnormal{ or }\CC^2\notin \CCC_{\bullet}(\GGG)\mid\GGG\big)\lesssim_{k} n^{-0.4}.
	\end{equation}
	Note that \eqref{eq:overlap:in:s:circ:positive}, \eqref{eq:overlap:indep:cor:positive:probability} and \eqref{eq:CCC:bullet:whp} imply that we have the following for some $\delta\equiv \delta(\alpha,k)$:
	\begin{equation}\label{eq:overlap:indep:cor:positive:probability:2}
	\begin{split}
	&\P\bigg(\Big|\zeta(\CC^1,\CC^2)-\frac{1}{2}\Big|\leq k^2 2^{-k/2}\textnormal{ and } \CC^1,\CC^2\in \CCC^{\tr}_{\circ}(\GGG)~\bigg|~\GGG\bigg)\geq \delta;\\
	&\P\Big(\CC^1=\CC^2\textnormal{ and } \CC^1,\CC^2\in \CCC^{\tr}_{\circ}(\GGG)~\Big|~\GGG\Big)\geq \delta.
	\end{split}
	\end{equation}
	Since $\GGG\in \mathfrak{G}_2$, $\bN^2_{\circ,\textnormal{bad}}[\GGG]=0$ holds. Thus, if $\big|\zeta(\CC^1,\CC^2)-\frac{1}{2}\big|\leq k^2 2^{-k/2}\textnormal{ and } \CC^1,\CC^2\in \CCC^{\tr}_{\circ}(\GGG)$, $(\CC^1,\CC^2)\in \Gamma_2^\star(5)$ holds. Consequently, the first inequality in \eqref{eq:overlap:indep:cor:positive:probability:2} and Proposition \ref{prop:overlap:tree} implies that (a) of Theorem \ref{thm2} holds with probability $\frac{\delta}{2}$. Similarly, the second inequality in \eqref{eq:overlap:indep:cor:positive:probability:2} implies (b) of Theorem \ref{thm2}.
	
	To show (c) of Theorem \ref{thm2}, let us consider the case where $\CC^1,\CC^2\in \CCC_{\bullet}(\GGG)$, which holds with probability $1-O_k(n^{-0.4})$ by \eqref{eq:CCC:bullet:whp}. By Lemma \ref{lem:overlap:localflip:existence} we can find $\GGG^\prime\in \textsf{Loc}(\GGG)$ such that it satisfies $\acute{\CC}^1,\acute{\CC}^2\in \CCC^{\tr}_{\circ}(\GGG^\prime)$. Here, as before, we denoted by $\acute{\CC}^{i}$ the cluster corresponding to $(\GGG^\prime,\ux^{i})$, where $\ux^{i}\in \{0,1,\ff\}^{V}$ is the frozen configuration corresponding to $\CC^{i}$ in $\GGG$ for $i=1,2$. Also, let $\ubx^i\in \{0,1\}^{V}$ and $\acute{\ubx}^i\in \{0,1\}^{V}$ be the \textsc{nae-sat} solutions independently and uniformly drawn from $\CC^i$ and $\acute{\CC}^i$ respectively. Then, because $\CC^i\in \CCC_{\bullet}(\GGG),\acute{\CC}^i \in \CCC_{\circ}^{\tr}(\GGG)$, and $\GGG^\prime \in \textsf{Loc}(\GGG)$, it is clear that there exists a coupling between $\ubx^i$ and $\acute{\ubx}^i$ which satisfies
	\begin{equation}\label{eq:coupling:two:sol}
	    \textnormal{Ham}(\ubx^i,\acute{\ubx}^i)\leq \frac{\log^{3}n}{n},\textnormal{ almost surely for $i=1,2$.}
	\end{equation}
	Indeed, note that all the free components in $(\GGG,\ux^i)$ or $(\GGG^\prime,\ux^i)$ have number of variables and clauses at most $\frac{4\log n}{k \log 2}$, and the number of edges that differ between $\GGG^\prime$ and $\GGG$ is at most $\log^{2}n$. Thus, most of the free trees in $(\GGG,\ux^i)$ remain valid free trees in $(\GGG^\prime,\ux^i)$, and the number of free variables that are not in such free components is bounded above by $O(\frac{\log^{3} n}{k})$. Hence, in order to have \eqref{eq:coupling:two:sol}, we can sample $\ubx^i$ and $\acute{\ubx}^i$ by independently drawing uniformly chosen solutions for each common free trees using the same randomness for both $\ubx^i$ and $\acute{\ubx}^i$, and for other free components in $(\GGG,\ux^i)$ or $(\GGG^\prime,\ux^i)$ independently from each other.
	
	Note that $\bN^2_{\circ,\textnormal{bad}}[\GGG^\prime]=0$ holds since $\GGG\in\mathfrak{G}_2$. Thus, when $\CC^1,\CC^2\in \CCC_{\circ}^{\tr}(\GGG)$, Proposition \ref{prop:overlap:tree} implies that $\PP\big(\min\big(\acute{\rho}_{\textnormal{abs}}, |\acute{\rho}_{\textnormal{abs}}-p^\star|\big)\geq n^{-0.35}\big)=\exp\big(-\Omega_k(n^{1/4})\big)$ holds, where $\acute{\rho}_{\textnormal{abs}}\equiv \big|\rho(\acute{\ubx}^1,\acute{\ubx}^2)\big|$. Consequently, by \eqref{eq:coupling:two:sol}, we have the following for $\rho_{\textnormal{abs}}\equiv \big|\rho(\ubx^1,\ubx^2)\big|$:
	\begin{equation}\label{eq:overlap:concentration}
	    \PP\Big(\min\big(\rho_{\textnormal{abs}}, |\rho_{\textnormal{abs}}-p^\star|\big)\geq n^{-0.35}~\Big|~\CC^1,\CC^2\in \CCC_{\bullet}(\GGG)\Big)=\exp\big(-\Omega_k(n^{1/4})\big).
	\end{equation}
	Therefore, \eqref{eq:CCC:bullet:whp} and \eqref{eq:overlap:concentration} imply (c) of Theorem \ref{thm2}, which concludes the proof.
	
% 		Let $\acute{\ubx}^1, \acute{\ubx}^2$ be  random solutions drawn independently and uniformly  from $\acute{\CC}^1, \acute{\CC}^2$, respectively. Then, Proposition \ref{prop:overlap:tree} tells us that the conclusion of Theorem \ref{thm2} holds for $\acute{\ubx}^1$ and $\acute{\ubx}^2$, since $\prescript{}{2}{\bN}_\circ^{\textnormal{int}}[\GGG'] = 0$. Thus, the following claim concludes the proof of Theorem \ref{thm2}.
% 		\begin{claim}\label{claim:overlap:coupling}
% 			Under the above setting, let $\ubx^1$ be a random solution drawn uniformly from $\CC^1$. Then, there exists a coupling between $\ubx^1$ and $\acute{\ubx}^1$ such that  
% 			\begin{equation*}
% 			\rho(\ubx^1, \acute{\ubx}^1) \ge 1- \frac{\log^{2}n}{n}, \ \textnormal{ almost surely.}
% 			\end{equation*}
% 		\end{claim}

	\end{proof}
	
% 	\begin{proof}[Proof of Claim \ref{claim:overlap:coupling}]
% Uniform random sampling of $\ubx^1$		from $\CC^1$ can be understood as the following procedure:
% \begin{enumerate}
% 	\item For all frozen variables $v$ in $\CC^1$, $\bx_v$ is endowed with the same 0-1 value as that of $v$.
	
% 	\item For each free tree $\ttt = (V(\ttt), F(\ttt), E(\ttt), \uL_{E(\ttt)})$ in $\CC^1$, sample a uniformly random valid 0-1 assignment among all the valid 0-1 configurations on $\ttt$, and assign $\{\bx_v^1 \}_{v\in V(\ttt)}$ with those values. This process is performed independently for each $\ttt$.
% \end{enumerate}
% Thus, if $\ttt$ in $\CC^1$ is still a valid free tree in $\acute{\CC}^1$ (i.e., $\ttt \oplus \uL \oplus \uL'$ is valid in $\GGG'$), we may assign same 0-1 values on both trees and have $\ubx^1_{V(\ttt)} = \acute{\ubx}^1_{V(\ttt)}.$ Clearly, this is the case for all free trees where $\uL_{E(\ttt)} = \uL'_{E(\ttt)}$. For the rest of the free trees where this property is not satisfied, we just assign 0-1 values to $\ubx^1$ and $\acute{\ubx}^1$ independently. Note that $\tL_e$ and $\tL'_e$ can differ at most at $\log^3 n$ places, and by the definition of $\CCC^{\cyc}_\bullet (\GGG), \CCC^{\tr}_\bullet(\GGG)$ and $\Gamma_1^\star(C)$, the maximal size of a free component in $\CC^1$ or $\acute{\CC}^1$ is bounded by $\log^{20}n$. Thus, our coupling satisfies $\rho (\ubx^1, \acute{\ubx}^1) \ge 1 - \frac{\log^{25}n}{n}$ almost surely.
% 	\end{proof}

	\section*{Acknowledgements}
	We thank Amir Dembo, Nike Sun and Yumeng Zhang for helpful discussions. We thank the anonymous reviewer for a careful reading and valuable feedbacks which greatly improved our paper. DN is supported by a Samsung Scholarship. AS is  supported  by NSF grants DMS-1352013 and DMS-1855527,
Simons Investigator grant and a MacArthur Fellowship. YS is partially supported by NSF grants DMS-1613091 and DMS-1954337.
	
		\bibliography{naesatref}
	
	\newpage

	\appendix
	
	\section{A priori estimates}\label{sec:app:apriori}
	The goal of this section is to prove Proposition \ref{prop:1stmo:aprioriestimate} (Section \ref{subsec:apriori:firstmo}) and Proposition \ref{prop:2ndmo:aprioriestimate} (Section \ref{subsec:apriori:secmo}). We also provide the proof of Lemma \ref{cor:B:close:optimal:2ndmo} at the end of Section \ref{subsubsec:apriori:cycles:2ndmo}.
	
	The main idea of the proof is motivated by Section 2.3 of \cite{dss16maxis}, although the computations are technically much more involved especially for the second moment due to more variety of spins. The heart of the proof lies in the comparison argument, where we compare large components with single free trees: when there are many large free components, we disassemble them into single free trees. Then, the cost of matching the large components will be much larger than the number of possible configurations for the single free trees, so we argue that the contribution to the overall partition function is small. Also, we show that the same strategy works in the case where there exists a multi-cylic free component. Throughout, we work with the projected coloring configuration, introduced in Section \ref{subsec:model:proj}.

	\subsection{The projected coloring}\label{subsec:model:proj}
	
	We introduce the notion of the projected coloring configuration. It is a simplification of the (union-)component coloring by certain projection, which we detail below. In Section \ref{subsec:apriori:firstmo} and \ref{subsec:apriori:secmo}, we will see that the main advantage of projected coloring configuration over component coloring is that the projected coloring configuration has smaller types of boundary spins, which makes it easier to work with. Throughout this section, $\rr$ denotes a new spin, which is different from $\rr_0$ and $\rr_1$. Similarly, we consider new spins $\rr\rr^{=}$ and $\rr\rr^{\neq}$, which are different from $\rr_0\rr_0,\rr_1\rr_1,\rr_0\rr_1,$ and $\rr_1\rr_0$. Also, with a slight abuse of notation, we denote $\{\rr\}:=\{\rr_0,\rr_1\}$. Similar holds for $\{\bb\}$. Define
	\begin{equation*}
	\begin{split}
	&\Omega_{\textnormal{pj}}:= \{\rr,\bb,\fs \} \cup \{(\ff,0), (\ff,1) \};
	\quad\quad\quad \Omega_{\textnormal{pj},2}^\fs:=
	\{ \rr\rr^=, \rr\rr^{\neq}, \bb\bb^=, \bb\bb^{\neq}, \rr\bb^=, \rr\bb^{\neq}, \bb\rr^=,\bb\rr^{\neq}, \fs\rr, \fs\bb, \rr\fs, \bb\fs,\fs\fs \};\\
	&\Omega_{\textnormal{pj},2}^\ff:=
	\sqcup_{x\in \{0,1\}}\{\ff\rr_x, \ff\bb_x, \rr_x\ff,\bb_x\ff\} \sqcup \{\ff\fs, \fs\ff,\ff\ff \}  ;\quad\quad\quad
	\Omega_{\textnormal{pj},2} := 	\Omega_{\textnormal{pj},2}^\fs\sqcup \left( 	\Omega_{\textnormal{pj},2}^\ff \times \{0,1 \} \right).
	\end{split}
	\end{equation*}
    To avoid confusion, we note that the superscripts $\fs$ and $\ff$ used above should not be interpreted as powers. Recalling $\Omega_{\textnormal{com}}$ in \eqref{eq:def:Omegcom1}, define the projection $\textsf{R}: \Omega_{\textnormal{com}}\to\Omega_{\textnormal{pj}}$ as
	\begin{equation*}
	\begin{split}
	\textsf{R}(\sigma^{\textnormal{com}}) := 
	\begin{cases}
	\rr &  \sigma^{\textnormal{com}} \in \{\rr \};\\
	\bb & \sigma^{\textnormal{com}} \in \{\bb \};\\
	\fs & \sigma^{\textnormal{com}} = \fs; \\
	(\ff, \tL_e) & \sigma^{\textnormal{com}} = (\fff,e),
	\end{cases}
	\end{split}
	\end{equation*}
	where $\tL_e$ denotes the literal labeled at the edge $e$ in the free component $\fff$. Recalling the definition of $\Omega_{\textnormal{com},2}$ in \eqref{eq:def:Omegcom2}, the projection
	$\textsf{R}_2:\Omega_{\textnormal{com},2} \to \Omega_{\textnormal{pj},2}$ in the pair model is defined similarly: recall the definition of $\textsf{P}$ and $\textsf{P}_2$ in \eqref{eq:def:projP}, and let
	\begin{equation*}
	\textsf{R}_2(\bsigma^{\textnormal{com}}) :=
	\begin{cases}
	\bpi & \bpi \in 	\Omega_{\textnormal{pj},2}^\fs \textnormal{ and } \bsigma^{\textnormal{com}} \in \{\bpi \};\\
	(\sigma_e, \tL_e) &
	\bsigma^{\textnormal{com}} = ( \uuu, e),\ \uuu\in \FFF_2
	\end{cases}
	\end{equation*}
	where $\sigma_e$ and $\tL_e$ denotes the spin and the literal labeled at the edge $e$ in $\uuu$ respectively.
	
	\begin{defn}[Projected coloring]
		Given $\GG = (V,F,E)$, we call $\upi\in \Omega_{\textnormal{pj}}^E$ (resp. $\bupi \in \Omega_{\textnormal{pj},2}^E$) a \textbf{projected coloring} (resp. a \textbf{pair projected coloring}). For $\lambda \in (0,1]$, its weight $w^{\textnormal{pj}}(\upi)$ (resp. $\bw^{\textnormal{pj}}(\bupi)$) is defined as
		\begin{equation}\label{eq:def:weight:projcol}
		\begin{split}
		w^{\textnormal{pj}}(\upi)^\lambda&:= \sum_{\sig^{\textnormal{com}} \in \Omega_{\textnormal{com}}^E} w(\sig^{\textnormal{com}})^\lambda \mathds{1}\{ \textsf{R}(\sig^{\textnormal{com}}) = \upi \};\\
		\bw^{\textnormal{pj}}(\bupi)^{\ula}&:= \sum_{\bsig^{\textnormal{com}} \in \Omega_{\textnormal{com},2}^E} \bw(\bsig^{\textnormal{com}})^{\ula} \mathds{1}\{ \textsf{R}_2(\bsig^{\textnormal{com}}) = \bupi \}.
		\end{split}
		\end{equation} 
		A projected coloring $\upi\in \Omega_{\textnormal{pj}}^E$ on $\GGG$ is called \textbf{valid} if $w^{\textnormal{pj}}(\upi)>0$, that is, there exist a literal assignment $\uL$ on $\GG$ and a  component coloring $\sig^{\textnormal{com}}\in\Omega_{\textnormal{com}}^E$ such that $\sig^{\textnormal{com}} $ valid on $(\GG,\uL)$ and $\textsf{R}(\sig^{\textnormal{com}}) = \upi$. Validity of a pair projected coloring is defined analogously. 
	\end{defn}
	Thus, $\bZ_{\la}$ (resp. $\bZ^{2}_{\ula}$) is the sum of $	w^{\textnormal{pj}}(\upi)^\lambda$ (resp.$\bw^{\textnormal{pj}}(\bupi)^{\ula}$) with the constraint that the number of free variables is bounded above by $7n/2^k$ and the number of $\rr$-colored edges is bounded above by $7nd/2^k$. We remark that unlike the (union-)component coloring, which has one to one correspondence with (pair-)frozen configuration, there could be many frozen configurations which has the same (pair-)projected coloring.
	\begin{defn}[Projected components]
		Let  $\upi\in \Omega_{\textnormal{pj}}^E$ be a valid  projected coloring on $\GGG$. Note that from $\upi$, we can uniquely recover if each variable in $\GGG$ is frozen or free, as well as if each clause is separating or not. From this information, let $\fff^{\textnormal{in}}$ be a free piece in  $\GGG$, whose literal information on the edges is given by $\upi$   (Definition \ref{def:freecomp:basic}). Then, a \textbf{projected component} $\mathfrak{p}$ is defined as follows.
		\begin{itemize}
			\item It is a labeled graph whose graph structure is given by the union of a free piece $\fff^{\textnormal{in}}$ and the boundary \textit{half-edges} incident to $\fff^{\textnormal{in}}$.
			
			\item Let $E(\ppp) = E(\fff^{\textnormal{in}})$ be the collection of (full) edges, and define $\dot{\partial }\ppp$ (resp. $\hat{\partial}\ppp$) to be boundary half-edges adjacent to $F(\ppp)$ (resp. $V(\ppp)$),  and $\partial \ppp=\dot{\partial }\ppp\sqcup \hat{\partial}\ppp$. Each $e\in E(\ppp)$ is labeled by $(\ff, \tL_e)$.
			
			\item Each edge $e\in \dot{\partial} \ppp$ (resp. $e\in \hat{\partial}\ppp$) is labeled by $\bb=\pi_e$ (resp. $\fs=\pi_e$). Note that the labels do not include the literal assignment. (In fact, $\upi$ does not carry information on the literals on ${\partial}\ppp$.)
		\end{itemize}
		For a valid projected coloring $\upi$, $\mathfrak{P}(\upi)$ denotes the enumeration of the projected components in $(\upi,\GG)$. Moreover, for $\ppp \in \mathfrak{P}(\upi)$, we define the inverse image $\textsf{R}^{-1}(\ppp) \subseteq \mathscr{F} $ as 
		\begin{equation*}
		\textsf{R}^{-1}(\ppp):= \{\fff\in \mathscr{F}: \fff^{\textnormal{in}} = \fff^{\textnormal{in}}(\ppp), \ \tL_e =0  \textnormal{ for all } e\in \dot{\partial}\fff  \},
		\end{equation*}
		where $\fff^{\textnormal{in}}(\ppp)$ denotes the free piece inside $\ppp$, including the edge labels on $E(\ppp)$. We note that if a free component $\fff\in \mathscr{F}$ satisfies $\tL_e=0$ for all $e\in \dot{\partial}\fff$,  then for each $a\in F(\fff)$ the spin-labels at $e\in \delta a \cap \dot{\partial}\fff$ should be either all-$\bb_0$ or all-$\bb_1$. Recalling the definition of $w^{\textnormal{com}}(\fff)$ from \eqref{eq:def:weight:freecomp:avg}, the weight of $\ppp$ is defined by
		\begin{equation}\label{eq:def:weight:proj:comp}
		w^{\textnormal{pj}}(\ppp)^\lambda := \sum_{\fff\in \textsf{R}^{-1}(\ppp) } w^{\textnormal{com}}(\fff)^\lambda.
		\end{equation}
	\end{defn}
	
	\begin{defn}[projected union components]
		Let $\bupi \in \Omega_{\textnormal{pj},2}^E$ be a valid pair projected coloring on $\GGG$, and let $\uuu^{\textnormal{in}}$ be the union-free piece in $\GGG$, determined by $\bupi$. Then, a \textbf{projected union component} $\upp$ is defined as follows.
		\begin{itemize}
			\item $\upp$ is a labeled graph whose graph structure is given by the union of a union-free piece $\uuu^{\textnormal{in}}$ and the boundary \textit{half-edges} incident to $\uuu^{\textnormal{in}}$.
			
			\item Let $E(\upp) = E(\uuu^{\textnormal{in}})$ be the collection of (full) edges, and define $ \dot{\partial}\upp $, $\hat{\partial} \upp$ and $\partial \upp$ to be the collections of boundary half-edges analogously to Definition \ref{def:freetree:basic}. Each $e\in E(\upp)$ is labeled by $  (\bpi_e,\tL_e)$.
			
			\item Each edge $e\in \dot{\partial}\upp \sqcup \hat{\partial} \upp$ is labeled by $\bpi_e \in \Omega_{\textnormal{pj,2}}^{\fs}$. Note that the label does not include the literal assignment, as $\bupi$ does not carry information on the literals on ${\partial}\upp$.
		\end{itemize}
		For a valid pair projected coloring $\bupi$, $\mathfrak{P}_2(\bupi)$ denotes the set of the projected union components in $(\bupi,\GG).$ Moreover, for $\ppp\in \mathfrak{P}_2(\bupi)$, define $\textsf{R}_2^{-1}(\upp) \subseteq \mathscr{F}_2$ to be the collection of $\uuu\in \mathscr{F}_2$ such that
		\begin{itemize}
			\item $\uuu^{\textnormal{in}}= \uuu^{\textnormal{in}}(\upp) $, where $\uuu^{\textnormal{in}}(\upp)$ denotes the union-free piece inside $\upp$. Namely, the graph structure, spin assignments and literal assignments on the edges are the same for $\uuu^{\textnormal{in}}$ and $\uuu^{\textnormal{in}}(\upp)$.
			
			\item For each $e\in \dot{\partial} \uuu$, $\tL_e=0$ and 
		\end{itemize}
		Recalling the definition of $\bw^{\textnormal{com}}(\uuu)$ from \eqref{eq:def:weight:freecomp:avg:2ndmo}, we define the weight of $\upp$ by
		\begin{equation}\label{eq:def:weight:proj:comp:2ndmo}
		\bw^{\textnormal{pj}}(\upp)^{\ula} := \sum_{\uuu\in \textsf{R}_2^{-1}(\upp) } \bw^{\textnormal{com}}(\uuu)^{\ula}.
		\end{equation}
	\end{defn}
	Recall the functions $\hat{v}(\cdot)$ on $\Omega^k$ and $\hat{v}_2(\cdot)$ on $\Omega_2^k$. Observe that for $\upi \in \{\rr,\bb,\fs\}^{k}$, for any $\sig \in \textsf{R}^{-1}(\upi)$, where $\textsf{R}^{-1}$ is acted component-wise, $\hat{v}(\sig)$ stays constant. Thus, with a slight abuse of notation, we can define $\hat{v}(\upi):=\hat{v}(\sig), \sig \in \textsf{R}^{-1}(\upi)$. Similarly, for $\bupi \in (\Omega_{\textnormal{pj},2}^{\fs})^{k}$, $\hat{v}_2(\bupi)=\hat{v}_2(\bsig), \bsig \in \textsf{R}_2^{-1}(\bupi)$ is well-defined. Then the following lemma shows the weight of (pair-)projected coloring configuration is determined by its projected (union-)components and the spins adjacent to (pair-)separating clauses.
	\begin{lemma}\label{lem:weight:projected}
		Let $\upi \in \Omega_{\textnormal{pj}}^E$ be a valid projected coloring on $\GGG$, and respectively set $n_\ff$ and $F_{\textnormal{sep}}(\upi)$ to be the number of free variables and the collection of separating clauses in $\GG$ induced by $\upi$. Then, we have
		\begin{equation}\label{eq:weightof projcol:1stmo}
		w^{\textnormal{pj}}(\upi)^\lambda = 2^{n-n_{\textnormal{\texttt{f}}}} \prod_{\ppp \in \mathfrak{P}(\upi)} w^{\textnormal{pj}}(\ppp)^\lambda \prod_{a\in F_{\textnormal{sep}}(\upi)} \hat{v}(\upi_{\delta a}).
		\end{equation}
		Also, let $\bupi\in \Omega_{\textnormal{pj},2}^E$ be a valid pair projected coloring on $\GGG$, and respectively set $n_{\ff \ff}$ and $F_{\textnormal{sep}}(\bupi)$ to be the number of union-free variables and the collection of pair-separating clauses induced by $\bupi$. Then, we have
			\begin{equation}\label{eq:weightof projcol:2ndmo}
			\bw^{\textnormal{pj}}(\bupi)^{\ula} = 2^{n-n_{\textnormal{\texttt{f}}\textnormal{\texttt{f}}}} \prod_{\upp \in \mathfrak{P}_2(\bupi)} \bw^{\textnormal{pj}}(\upp)^{\ula} \prod_{a\in F_{\textnormal{sep}}(\bupi_a )} \hat{v}_2(\bupi_a).
			\end{equation}
		\end{lemma}
	
	\begin{proof}
		We only present a proof  of \eqref{eq:weightof projcol:1stmo} since \eqref{eq:weightof projcol:2ndmo} can be verified analogously. Suppose that a component coloring $\sig^{\textnormal{com}} \in \Omega_{\textnormal{com}}^E$ satisfies $\textsf{R}(\sig^{\textnormal{com}}) = \upi$. Then, $\sig^{\textnormal{com}}$ is determined by $\upi$, the value of frozen variables (either 0 or 1), and the literals of the boundary edges $e\in {\partial}\ppp$ for each $\ppp\in \mathfrak{P}(\upi)$. We denote this relation by $\sig^{\textnormal{com}} = (\upi, \ux, \uL )$ for $\ux \in \{0,1 \}^{V_{\textnormal{fz}}}$ and $\uL\in \{0,1\}^{{\partial}E(\upi)}$, where $V_{\textnormal{fz}}(\upi)$ is the collection of frozen variables and ${\partial}E(\upi):= \cup_{\ppp \in \mathfrak{P}(\upi)} {\partial} \ppp $. Thus, we can write
		\begin{equation*}
		w^{\textnormal{pj}}(\upi)^\lambda = \sum_{\ux \in \{0,1 \}^{V_{\textnormal{fz} }(\upi) }} \sum_{\uL_{{\partial}E(\upi)}} w(\upi, \ux, \uL_{{\partial}E(\upi)} )^\lambda.
		\end{equation*}
		Observe that the inner sum
		is independent of $\ux  $, due to the $0/1$ symmetry. Moreover, from the formula of $w(\sig^{\textnormal{com}})^\lambda $ in \eqref{eq:sizeformula:freecomp:1stmo} and the definition of $w^{\textnormal{pj}}(\ppp)$ in \eqref{eq:def:weight:proj:comp}, we have
		\begin{equation*}
		\sum_{\uL_{{\partial}E(\upi)}} w(\upi, \ux, \uL_{{\partial}E(\upi)})^\lambda  = \prod_{\ppp \in \mathfrak{P}(\upi)} w^{\textnormal{pj}}(\ppp)^\lambda \prod_{a\in F_{\textnormal{sep}}(\upi)} \hat{v}(\upi_a),
		\end{equation*}
		which concludes the proof of the lemma.
	\end{proof}
	\subsection{First moment}\label{subsec:apriori:firstmo}
	Given a \textsc{nae-sat} instance $\GGG$ and a valid projected configuration $\upi\in \Omega_\pj^{E}$, let $\mathfrak{F}[\upi]$ be the free subgraph induced by $\upi$. As usual, the free subgraph is defined by the set of free variables, the set of non-separating clauses and the matching between the half-edges adjacent to them. Hence, we can encode the subgraph $\mathfrak{F}$ alone as the subset of half-edges $H_{\mathfrak{F}}$, adjacent to the free variables and non-separating clauses, and a matching $M_{\mathfrak{F}}$ on $H_{\mathfrak{F}}$. Note that the half-edges that are not matched serve as boundary half-edges and they are labeled $\bb$ if adjacent to clauses, and $\fs$ if adjacent to variables.
	
	Let $\dot{H}_{\circ}\equiv \dot{H}_{\circ}[\upi]$ encode the empirical measure of spins adjacent to frozen variables of $\upi$:
	\begin{equation*}
	    \dot{H}_{\circ}(\utau) \equiv \frac{1}{n}|\{v\in V: \upi_{\delta v}=\utau\}|\quad\textnormal{for all}\quad\utau \in \{\rr,\bb\}^{d}\backslash \{\bb\}^d
	\end{equation*}
	Denote by $\bZ_{\la}[\dot{H}_{\circ}, \mathfrak{F}]$ the contribution to $\bZ_\la$ from projected configurations $\upi$ with $\dot{H}_{\circ}[\upi]=\dot{H}_{\circ}$ and $\mathfrak{F}[\upi]=\mathfrak{F}$. Later, we will compare $\bZ_{\la}[\dot{H}_{\circ}, \mathfrak{F}]$ and $\bZ_{\la}[\dot{H}_{\circ}, \mathfrak{F}^\prime]$, where $\mathfrak{F}^\prime$ is roughly a ``disassembled" $\mathfrak{F}$. To this end, we first compute $\bZ_{\la}[\dot{H}_{\circ}, \mathfrak{F}]$. Let $w(\mathfrak{F})^{\la} \equiv \prod_{\ppp \in \mathfrak{F}}w^{\textnormal{pj}}(\ppp)^{\la}$ be the weight of $\mathfrak{F}$, where $w(\ppp)^{\la}$ is defined in \eqref{eq:def:weight:proj:comp} and $\ppp\in\mathfrak{F}$ denotes the projected component in $\mathfrak{F}$. By Lemma \ref{lem:weight:projected}, we have
	\begin{equation}\label{eq:apriori:techinical-1}
	    \E \bZ_\la[\dot{H}_\circ, \mathfrak{F}] = \E\left[\E\left[\bZ_\la[\dot{H}_\circ, \mathfrak{F}]\Big| \GG\right]\right] = 2^{n-n_\tsf}w(\mathfrak{F})^{\la}\E \Big[\sum_{\substack{\mathfrak{F}[\upi]=\mathfrak{F}\\ \dot{H}_\circ[\upi]=\dot{H}_\circ}}\prod_{a\in F_\sep}\hat{v}(\upi_{\delta a})\Big],
	\end{equation}
	where $n_\tsf$ is the number of free variables and $F_\sep$ is the set of separating clauses, which are all determined by $\kF$. Note that if $\kF[\upi]=\kF$ and $\dot{H}_\circ[\upi]=\dot{H}_\circ$, then $(\upi_{\delta v})_{v\in V}$ is fully determined, modulo choosing the location of the spins adjacent to frozen variables with empirical $\dot{H}_\circ$. Hence, if we denote $c(n_\tsf,\dot{H}_\circ)\equiv {n-n_\tsf \choose n\dot{H}_\circ}\equiv \frac{(n-n_\tsf)!}{\prod_{\sig}\left(n\dot{H}_\circ(\sig)\right)!}$, then the rightmost term of the equation above can be computed by
	\begin{equation}\label{eq:apriori:1stmo:matching}
	    \E \Big[\sum_{\substack{\mathfrak{F}[\upi]=\mathfrak{F}\\ \dot{H}_\circ[\upi]=\dot{H}_\circ}}\prod_{a\in F_\sep}\hat{v}(\upi_{\delta a})\Big] = c(n_\tsf,\dot{H}_\circ)\E\left[\prod_{a\in F_\sep}\hat{v}(\upi_{\delta a})\one\left\{A_1 \cap A_2 \cap A_3 \cap A_4\right\}\right],
	\end{equation}
	where the expectation in the \textsc{rhs} is with respect to uniform matching of $nd$ half-edges with empirical distribution determined by $\kF$ and $\dot{H}_\circ$, and
	\begin{equation*}
	\begin{split}
	    A_1 &\equiv \{\textnormal{Each clause can contain at most one red edge}\},\\
	    A_2 &\equiv \{\textnormal{Free edges, either $(\ff,0)$ or $(\ff,1)$, must be matched according to $M_{\kF}$}\},\\
	    A_3 &\equiv \{\textnormal{$\fs$ edges are not matched to clauses that contain a red edge or a free edge}\},\\
	    A_4 &\equiv \{\textnormal{Clauses without red edge nor free edge must have at least $2$ blue edges}\}.
	\end{split}
	\end{equation*}
	Let $E_x, x\in \{\rr,\bb,\fs\}$ denote the number of edges with color $x$ and $E_{\tsf}$ denote the number of edges with color either $(\ff,0)$ or $(\ff,1)$. Note that $E_x$'s are all determined by $\kF$ and $\dot{H}_\circ$. Let $m_\ns$ denote the number of non-separating clauses, determined by $\kF$. Then, it is straightforward to compute
	\begin{equation*}
	    \P(A_1 \cap A_2) = \frac{\prod_{j=0}^{E_\rr -1 }(mk-m_\ns k-jk)}{\prod_{i=0}^{E_{\rr}+E_{\tsf} -1}(nd-i)} = \frac{k^{E_\rr}(m-m_\ns)_{E_\rr}}{(nd)_{E_\rr+E_\tsf}}.
	\end{equation*}
	Let $m_{\textnormal{s}}\equiv m-m_{\ns}-E_{\rr}$ denote the number of separating, but non-forcing clauses. On the event $A_3$, all $\fs$ edges must be matched to these $m_\fs$ clauses, so
	\begin{equation*}
	    \P(A_3 \mid A_1 \cap A_2) = \frac{(km_{\tns})_{E_\fs}}{(nd-E_{\rr}-E_\tsf)_{E_\fs}}.
	\end{equation*}
	Conditional on $A_1\cap A_2\cap A_3$, $E_\fs$ edges are matched to $km_\tns$ half-edges adjacent to separating, but non-forcing clauses. Also, for $a\in F_\sep$, if $a$ is forcing $\hat{v}(\upi_{\delta a})=2^{-k+1}$. Hence, we can write
	\begin{equation}\label{eq:apriori:technical-2}
	\begin{split}
	    \E\left[\prod_{a\in F_\sep}\hat{v}(\upi_{\delta a})\one\left\{A_4\right\}\bigg| A_1\cap A_2\cap A_3\right]&= 2^{-(k-1)E_{\rr}}\E\left[\prod_{a\in F_\sep \backslash F_{\textnormal{fc}}}\hat{v}(\upi_{\delta a})\one\left\{A_4\right\}\bigg| A_1\cap A_2\cap A_3\right]\\
	    &\equiv 2^{-(k-1)E_{\rr}}f(m_{\tns},E_{\fs}),
	\end{split}
	\end{equation}
	where $F_{\textnormal{fc}}$ denotes the set of forcing clauses. Therefore, reading \eqref{eq:apriori:techinical-1}-\eqref{eq:apriori:technical-2} altogether shows
	\begin{equation}\label{eq:apriori:firstmo:expansion}
	    \E \bZ_\la[\dot{H}_\circ, \mathfrak{F}]= 2^{n-n_\tsf-(k-1)E_\rr}w(\mathfrak{F})^{\la}c(n_\tsf,\dot{H}_\circ)\frac{k^{E_\rr}(m-m_\ns)_{E_\rr}}{(nd)_{E_\rr+E_\tsf}}\frac{(km_{\tns})_{E_\fs}}{(nd-E_{\rr}-E_\tsf)_{E_\fs}}f(m_{\tns},E_{\fs})
	\end{equation}
	\subsubsection{Exponential decay of free tree frequencies}\label{subsubsec:exp:decay:1stmo} Let $a,b,\ell,A$ be non-negative integers with $a\geq 2, b\geq 1$ and let $\kF_\circ$ be a free subgraph in projected configuration, which does not have any isolated free variable nor any projected component with $a$ variables and $b$ clauses. Henceforth, we refer to a projected component with $a$ variables and $b$ clauses as an $(a,b)$-component. Let $\Omega^{a,b}_{\ell,A}(n_\tsf; \kF_\circ)$ denote the collection of free subgraphs $\kF$ such that
	\begin{itemize}
	    \item $\kF$ contains $\kF_\circ$ and has $|V(\kF)|=n_\tsf$ variables.
	    \item $\kF\backslash \kF_\circ$ contains $\ell$ $(a,b)$-components and all the other remaining components in $\kF\backslash \kF_\circ$ have a single free variable.
	    \item $(a,b)$-components have $q\equiv \ell(a+b-1)+A$ internal edges. 
	\end{itemize}
	By \eqref{eq:apriori:firstmo:expansion}, for $\kF \in \Omega^{a,b}_{\ell,A}(n_\tsf; \kF_\circ)$, $\E\bZ_\la[\dot{H}_\circ, \kF]$ is fully determined by $a,b,\ell,A,n_\tsf,\dot{H}_\circ$ and $\kF_\circ$. The lemma below is the crux of the proof of Proposition \ref{prop:1stmo:aprioriestimate} (1),(2) and (3).
	\begin{lemma}\label{lem:apriori:1stmo:comparison}
	For $k\geq k_0$, $n_\tsf \leq 7n/2^k$, $E_\rr \leq 7nd/2^k$,$m/n \in [\alpha_{\textsf{lbd}}, \alpha_{\textsf{ubd}}]$, and $n\geq n_0(k)$, the following inequality holds. For $\kF \in \Omega^{a,b}_{\ell,A}(n_\tsf; \kF_\circ)$ and $\kF^\prime \in \Omega^{a,b}_{0,0}(n_\tsf; \kF_\circ)$,
	\begin{equation*}
	    \textnormal{\textbf{R}}^{a,b}_{\ell,A}(\dot{H}_\circ,\kF_\circ) \equiv \frac{|\Omega^{a,b}_{\ell,A}(n_\tsf; \kF_\circ)|}{|\Omega^{a,b}_{0,0}(n_\tsf; \kF_\circ)|}\frac{\E \bZ_\la[\dot{H}_\circ,\kF]}{\E \bZ_\la[\dot{H}_\circ,\kF^\prime]} \lesssim_{k} \left(\frac{n}{k q}\left(\frac{Ck}{2^k}\right)^{a}(Ck)^{b}\right)^{\ell}\left(\frac{C(a\land b)k}{n}\right)^{A},
	\end{equation*}
	where $C$ is a universal constant.
	\end{lemma}
	\begin{proof}
	We first upper bound $\frac{\E \bZ_\la[\dot{H}_\circ,\kF]}{\E \bZ_\la[\dot{H}_\circ,\kF^\prime]}$ using \eqref{eq:apriori:firstmo:expansion}. As before, let $E_x, x \in \{\rr,\bb,\fs, \tsf\}$ be the number of edges colored $x$ and $m_\tns$ be the number of separating, but non-forcing, clauses, corresponding to $(\dot{H}_\circ,\kF)$. Let $E^\prime_x, x \in \{\rr,\bb,\fs, \tsf\}$ and $m^\prime_\tns$ be the same corresponding to $(\dot{H}_\circ,\kF^\prime)$. Note that $m^\prime_\tns = m_\tns+ \ell b$ and $E^\prime_\fs = E_\fs +q$ holds. In Section \ref{subsubsec:apriori:separating}, we show in Proposition \ref{prop:apriori:1stmo:separating} that the following holds for $\delta_1, \delta_2\in \Z_{\geq 0}$ in the regime stated in the statement of Lemma \ref{lem:apriori:1stmo:comparison}: there exists a universal constant $C$ such that
	\begin{equation}\label{eq:apriori:1stmo:separating}
	    \frac{f(m_\tns, E_\fs)}{f(m_\tns+\delta_1, E_\fs+\delta_2)}\lesssim_{k} e^{C(\delta_1+\delta_2)}.
	\end{equation}
	We use \eqref{eq:apriori:1stmo:separating} for $\delta_1=\ell b$ and $\delta_2 =q$. Observe that in the stated regime, $q\leq E_s+E_\tsf +E_\rr \leq 14nd/2^k$. Also, since each non-separating clauses have at least $2$ free edges, $m_\tns=m-m_\ns-E_\rr\geq (1-21k/2^{k+1})m$. Hence, together with \eqref{eq:apriori:firstmo:expansion} and \eqref{eq:apriori:1stmo:separating}, it is straightforward to bound
	\begin{equation}\label{eq:apriori:1stmo:long}
	\begin{split}
	    \frac{\E \bZ_\la[\dot{H}_\circ,\kF]}{\E \bZ_\la[\dot{H}_\circ,\kF^\prime]} &= \frac{w(\kF)^\la}{w(\kF^\prime)^\la}\frac{(m_\tns+E_\rr)_{E_\rr}}{(m_\tns+\ell b+E_\rr)_{E_\rr}}\frac{(nd)_{E_\rr+E_\tsf-q}}{(nd)_{E_\rr+E_\tsf}}\frac{\frac{(km_\tns)_{E_\fs}}{(nd-E_\rr-E_\tsf)_{E_\fs}}}{\frac{(km_\tns+k\ell b)_{E_\fs+q}}{(nd-E_\rr-E_\tsf+q)_{E_\fs+q}}}\frac{f(m_\tns,E_\fs)}{f(m_\tns+\ell b, E_\fs + q)}\\
	    &\leq \frac{w(\kF)^\la}{w(\kF^\prime)^\la}\frac{(nd)_{E_\rr+E_\tsf-q}}{(nd)_{E_\rr+E_\tsf}}
	    \frac{(nd-E_\rr-E_\tsf+q)_q}{(km_\tns+k\ell b -E_\fs)_q}\frac{f(m_\tns,E_\fs)}{f(m_\tns+\ell b, E_\fs + q)}\lesssim_{k} \frac{w(\kF)^\la}{w(\kF^\prime)^\la} \left(\frac{1}{nd}\right)^q e^{O(q)}.
	\end{split}
	\end{equation}
	To further bound the \textsc{rhs} of the equation above, note that for a projected component $\ppp$, we have $w(\ppp)^\la \leq 2^{f(\ppp)}\frac{2^{\la v(\ppp)}}{2^{k f(\ppp)}}$,
	where $f(\ppp) =|F(\ppp)|$ and $v(\ppp)=|V(\ppp)|$. This is because there are at most $2^{f(\ppp)}$ many free components corresponding to $\ppp$ in \eqref{eq:def:weight:proj:comp} by choosing $\bb_0$ or $\bb_1$ for the boundary colors adjacent to each clauses, and each of them has $\la$-tilted weight at most $\frac{2^{\la v(\ppp)}}{2^{k f(\ppp)}}$. Since a single free projected component, i.e. the unique projected component which has one variable, has $\la$-tilted weight exactly $2^\la$, we have $\frac{w(\kF)^\la}{w(\kF^\prime)^\la}=\frac{w(\kF\backslash \kF_\circ)^\la}{w(\kF^\prime\backslash \kF_\circ)^\la}\leq \frac{1}{2^{(k-1)\ell b}}$. Therefore, plugging it into \eqref{eq:apriori:1stmo:long} shows
	\begin{equation}\label{eqn:apriori:1stmo:ratio:final}
	     \frac{\E \bZ_\la[\dot{H}_\circ,\kF]}{\E \bZ_\la[\dot{H}_\circ,\kF^\prime]}\lesssim_{k} \frac{1}{2^{k\ell b}}\left(\frac{1}{nd}\right)^{q}e^{O(q)}.
	\end{equation}
	We turn now to upper bound $\frac{|\Omega^{a,b}_{\ell,A}(n_\tsf; \kF_\circ)|}{|\Omega^{a,b}_{0,0}(n_\tsf; \kF_\circ)|}$. $\kF \in \Omega^{a,b}_{\ell,A}(n_\tsf; \kF_\circ)$ is obtained as follows: first, from the $n-|V(\kF_\circ)|$ variables and $m-|F(\kF_\circ)|$ clauses, choose $T\equiv n_\tsf -|V(\kF_\circ)|$ variables and $\ell b$ clauses to belong to $\kF \backslash \kF_\circ$. From these we choose a subset of $T-\ell a$ variables to belong to single free projected components.  Next we choose $\ell a $ variables and $\ell b $ clauses to form $\ell$ $(a,b)$-components. Divide $\ell a$ variables and $\ell b$ clauses into $\ell$ groups of $a$ variables and $b$ clauses; the number of ways to do this is $(\ell a)!(\ell b)!/\ell!(a!)^\ell (b!)^\ell$. To decide the internal edges among these components, first choose an ordered list of variable-adjacent half-edges $c_1,...,c_q$ from the $\ell da$ half-edges available. Then, each $c_i$ must be matched to another half-edge $d_i$, adjacent to a clause that is in the same group as the variable adjacent to $a_i$. There are $kb$ choices of $d_i$ for each $c_i$, and $q!$ lists of ordered pairs $(c_i,d_i)_{i\leq q}$ yield the same set of internal edges. Finally, assign literal $0$ or $1$ to each one of $q$ internal edges. Therefore, altogether we have
	\begin{equation}\label{eq:apriori:Omega:upperbound}
	    |\Omega^{a,b}_{\ell,A}(n_\tsf; \kF_\circ)|\leq {n-n_\circ \choose T}{m-m_\circ \choose\ell b}{T\choose \ell a} \frac{(\ell a)!(\ell b)!}{\ell!(a!)^\ell (b!)^\ell}\frac{(2\ell da)^{q}(kb)^q}{q!},
	\end{equation}
	where $n_\circ\equiv|V(\kF_\circ)|$ and $m_\circ\equiv|F(\kF_\circ)|$. Since the above inequality is an equality when $\ell=A=0$, we can bound
	\begin{equation}\label{eq:apriori:omega:inter}
	    \frac{|\Omega^{a,b}_{\ell,A}(n_\tsf; \kF_\circ)|}{|\Omega^{a,b}_{0,0}(n_\tsf; \kF_\circ)|} \leq \frac{(m-m_\circ)_{\ell a}(T)_{\ell a}}{\ell!(a!)^\ell (b!)^\ell}\frac{(2\ell da)^{q}(kb)^q}{q!}\leq m^{\ell b}T^{\ell a}e^{\ell+a\ell+b\ell+ 2q}\frac{(\ell da)^{q}(kb)^q}{\ell^\ell a^{\ell a}b^{\ell b}q^q},
	\end{equation}
	where the last inequality is due to the bound $x! \geq (x/e)^{x}$. Recalling $q=\ell(a+b-1)+A$, we can further bound the rightmost term in the equation above by
	\begin{equation}\label{eq:apriori:factor:technical}
	    \frac{(\ell da)^{q}(kb)^q}{\ell^\ell a^{\ell a}b^{\ell b}q^q} = (dk)^q \frac{(\ell a)^{b\ell-\ell+A}(\ell b)^{a\ell-\ell+A}}{q^q \ell^A}\leq (dk)^q\frac{q^{(a+b-2)\ell+A}\left(\ell(a\land b)\right)^{A}}{q^{q}\ell^A}=(dk)^q\frac{(a\land b)^{A}}{q^{\ell}}.
	\end{equation}
	Together with the fact $T\leq 7n/2^k$, plugging \eqref{eq:apriori:factor:technical} into \eqref{eq:apriori:omega:inter} shows
	\begin{equation}\label{eq:apriori:ratio:omega}
	     \frac{|\Omega^{a,b}_{\ell,A}(n_\tsf; \kF_\circ)|}{|\Omega^{a,b}_{0,0}(n_\tsf; \kF_\circ)|} \leq e^{O(q)}n^{\ell(a+b)}d^{q+\ell b}k^{q-\ell b}\frac{(a\land b)^{A}}{2^{k\ell a}q^{\ell}}.
	\end{equation}
	Finally, recalling the bound $d\leq k2^{k-1}\log 2$ by Remark \ref{rem:k:adjusted}, we multiply \eqref{eqn:apriori:1stmo:ratio:final} and \eqref{eq:apriori:ratio:omega} to find
	\begin{equation}\label{eq:lem:apriori:1stmo:comparison:final}
	    \textnormal{\textbf{R}}^{a,b}_{\ell,A}(\dot{H}_\circ,\kF_\circ) \lesssim_{k} e^{O(q)}\frac{d^{\ell b}k^{q-\ell b}}{n^{A-\ell}}\frac{(a\land b)^{A}}{2^{k\ell (a+b)}q^{\ell}}\leq e^{O(q)} \left(\frac{n}{kq}\left(\frac{k}{2^k}\right)^{a}k^{b}\right)^{\ell}\left(\frac{(a\land b)k}{n}\right)^{A}.
	\end{equation}
	Recalling $q=\ell(a+b-1)+A$, \eqref{eq:lem:apriori:1stmo:comparison:final} concludes the proof.
	\end{proof}
	\begin{proof}[Proof of Proposition \ref{prop:1stmo:aprioriestimate} $(1)$ and $(3)$]
	We only present the proof of $(3)$ of Proposition \ref{prop:1stmo:aprioriestimate} since $(1)$ follows by the same argument. Let $\ell_{a,b}=\ell_{a,b}(\upi)$ denote the number of $(a,b)$-component in a projected configuration $\upi$. Then, 
	\begin{multline}\label{eq:apriori:1stmo:inter:1}
	   \E \bZ_\la[(\ee_{\frac{1}{c+1}})^{\mathsf{c}} \quad\textnormal{and}\quad \forall \fff,~~~~f(\fff)\leq v(\fff)+1]\leq \E \bZ_\la[\exists a \leq \frac{7n}{2^k}\quad\textnormal{s.t.}\quad \sum_{b=1}^{a+1} \ell_{a,b}>n2^{-\frac{ka}{c+1}}]\\
	    \leq \E\bZ_\la[\exists a \leq \frac{7n}{2^k},b\leq a+1 \quad\textnormal{s.t.}\quad\ell_{a,b}>n2^{-\frac{11ka}{10(c+1)}}],
	\end{multline}
	where the last inequality is because $2^{\frac{ka}{10(c+1)}}\geq a+1$ for any $a\geq 1$, given large enough $k$ and $c \leq 3$. Recalling the definition of $\textnormal{\textbf{R}}^{a,b}_{\ell,A}(\dot{H}_\circ,\kF_\circ)$ in Lemma \ref{lem:apriori:1stmo:comparison}, we can bound 
	\begin{equation*}
	    \frac{\E\bZ_\la[\exists  a \leq \frac{7n}{2^k},b\leq a+1 \quad\textnormal{s.t.}\quad\ell_{a,b}>n2^{-\frac{11ka}{10(c+1)}}]}{\E \bZ_\la} \leq \sup_{\dot{H}_\circ,\kF_\circ}\left\{ \sum_{a=1}^{ 7n/2^k}\sum_{b=1}^{a+1}\sum_{\ell\geq \ell_{\textnormal{max}}(a)}\sum_{A\geq 0}\textnormal{\textbf{R}}^{a,b}_{\ell,A}(\dot{H}_\circ,\kF_\circ)\right\}.
	\end{equation*}
	where $\ell_{\textnormal{max}}(a)\equiv \ceil{n2^{-\frac{11ka}{10(c+1)}}}$. For any $\dot{H}_\circ$ and $\kF_\circ$, Lemma \ref{lem:apriori:1stmo:comparison} shows
	\begin{equation}\label{eq:proof:apriori:1stmo:2}
	\begin{split}
	   \sum_{a=1}^{ 7n/2^k}\sum_{b=1}^{a+1}\sum_{\ell\geq \ell_{\textnormal{max}}(a)}\sum_{A\geq 0}\textnormal{\textbf{R}}^{a,b}_{\ell,A}(\dot{H}_\circ,\kF_\circ)&\lesssim_{k} \sum_{a=1}^{ 7n/2^k}\sum_{\ell\geq \ell_{\textnormal{max}}(a)}\sum_{b=1}^{a+1}\sum_{A\geq 0}\left(\frac{n}{k\ell a}\left(\frac{Ck}{2^k}\right)^{a}(Ck)^{b}\right)^{\ell}\left(\frac{Cbk}{n}\right)^{A}\\
	   &\lesssim \sum_{a=1}^{ 7n/2^k}\sum_{\ell\geq \ell_{\textnormal{max}}(a)}\left(\frac{Cn}{\ell a}\left(\frac{Ck^2}{2^k}\right)^{a}\right)^{\ell},
	\end{split}
	\end{equation}
	where the universal constant $C$ may differ in each line. Note that for $k$ large enough $(Ck^2/2^k)^{a} \leq 2^{-11ka/12}\leq (\ell_{\max}(a)/n)^{5(c+1)/6}$, so we can further bound the \textsc{rhs} of the equation above by
	\begin{equation*}
	\sum_{a=1}^{ 7n/2^k}\sum_{\ell\geq \ell_{\textnormal{max}}(a)}\left(\frac{Cn}{\ell a}\left(\frac{Ck^2}{2^k}\right)^{a}\right)^{\ell}\leq	\sum_{a=1}^{ 7n/2^k}\sum_{\ell\geq \ell_{\textnormal{max}}(a)} \left(\frac{C}{a}\left(\frac{\ell}{n}\right)^{\frac{5c-1}{6}}\right)^{\ell}\lesssim \sum_{a=1}^{ 7n/2^k}  \frac{1}{a}\left(\frac{\ell_{\textnormal{max}}(a)}{n}\right)^{(\frac{5c-1}{6})\ell_{\textnormal{max}}(a)}.
	\end{equation*}
	Note that $\frac{5}{6}c-\frac{1}{6} \geq \frac{2}{3}c$ for $c\geq 1$ and $\ell \to (\ell/n)^{2c\ell/3}, 1\leq \ell \leq n2^{-11k/10(c+1)}$ is maximized at $\ell =1$, thus
	\begin{equation}\label{eq:apriori:1stmo:inter:2}
	    \sum_{a=1}^{ 7n/2^k}  \frac{1}{a}\left(\frac{\ell_{\textnormal{max}}(a)}{n}\right)^{(\frac{5c-1}{6})\ell_{\textnormal{max}}(a)}\leq n^{-\frac{2c}{3}} \sum_{a=1}^{ 7n/2^k}  \frac{1}{a}\lesssim n^{-\frac{2c}{3}}\log n,
	\end{equation}
	which concludes the proof of Proposition \ref{prop:1stmo:aprioriestimate}-(3).
	\end{proof}
	\begin{proof}[Proof of Proposition \ref{prop:1stmo:aprioriestimate} $(2)$]
		Since each clause in a projected component has internal degree at least $2$, the number of internal edges $q$ of $\ell$ $(a,b)$-components satisfy $q\equiv \ell(a+b-1)+A \geq 2b \ell$. Thus, we have that 
    \begin{equation*}
    \frac{ \E\bZ_\la[\exists b\geq a+2,~~~~~\ell_{a,b}\geq 1]}{\E \bZ_\la}\leq \sup_{\dot{H}_\circ,\kF_\circ}\left\{\sum_{a=1}^{7n/2^k}\sum_{b=a+2}^{7km/2^k}\sum_{\ell\geq 1}\sum_{A\geq \ell(b-a+1)}\textnormal{\textbf{R}}^{a,b}_{\ell,A}(\dot{H}_\circ,\kF_\circ)\right\}\,.
    \end{equation*}
    We use Lemma \ref{lem:apriori:1stmo:comparison} to further bound the \textsc{rhs} above by
	\begin{multline}\label{eq:proof:apriori:1stmo:1}
    \frac{ \E\bZ_\la[\exists b\geq a+2,~~~~~\ell_{a,b}\geq 1]}{\E \bZ_\la}\lesssim_{k} \sum_{a=1}^{ 7n/2^k}\sum_{b=a+2}^{7km/2^k}\sum_{\ell\geq 1}\sum_{A\geq \ell(b-a+1)}\left(\frac{n}{2kb\ell}\left(\frac{Ck}{2^k}\right)^{a}(Ck)^{b}\right)^{\ell}\left(\frac{Cak}{n}\right)^{A}
    \\\leq 2\sum_{a=1}^{ 7n/2^k}\sum_{b=a+2}^{7km/2^k}\sum_{\ell\geq 1}\left(\frac{Ca}{2^{ka}}\frac{(Ck)^{2b}}{\ell b}\left(\frac{a}{n}\right)^{b-a}\right)^{\ell}\leq 4\sum_{a=1}^{ 7n/2^k}\sum_{b=a+2}^{7km/2^k}\frac{C^{2a+1}k^{2a}a}{2^{ka}b}\left(\frac{C^2k^2a}{n}\right)^{b-a},
    \end{multline}
	where the last inequality is due to $a/n\leq 7/2^k$. We can further bound
	\begin{equation*}
	    \sum_{a=1}^{ 7n/2^k}\sum_{b=a+2}^{7km/2^k}\frac{C^{2a+1}k^{2a}a}{2^{ka}b}\left(\frac{C^2k^2a}{n}\right)^{b-a}\leq 2\sum_{a=1}^{ 7n/2^k}\frac{1}{n^2}\frac{C^{2a+5}k^{2a+4}a^2}{2^{ka}}\lesssim_{k}\frac{1}{n^2},
	\end{equation*}
	concluding the proof.
	\end{proof}
	\subsubsection{Contribution from cycles}\label{subsubsec:apriori:cycles} Given a projected component $\ppp$, we find a subtree $\Psi_{\tr}(\ppp)$ of $\ppp$, which is a valid projected component, by the following algorithm.
	\begin{enumerate}[label=\textnormal{Step \arabic*:}]
	\item If any, find a clause $a \in F(\ppp)$ such that it has internal degree $2$ and deleting $a$ doesn't affect the connectivity of $\ppp$. Then, delete $a$ and all the half-edges adjacent to $a$, namely $k-2$ boundary edges and the half-edges included in the internal edges, $e_1=(av_1)$ and $e_2=(av_2)$. The half-edges of $e_1$ and $e_2$ hanging on $v_1$ and $v_2$ respectively become a boundary half-edge, so give the color $\fs$ to them.
	
	\item Repeat Step $1$ until there is no such clause.
	
	\item If any, find a tree-excess edge $e=(a^\prime v^\prime)$, i.e. an edge after deletion doesn't affect the connectivity of the graph. Cut $e$ in half to make two boundary half-edges adjacent to $a^\prime$ and $v^\prime$ respectively. The new boundary half-edge adjacent to $a^\prime$ is colored $\bb$ while the one adjacent to $v^\prime$ is colored $\fs$. Note that by Step $(1)$ and $(2)$, $a^\prime$ must have internal degree at least $2$ after $e$ is cut, which guarantees the validity of $a^\prime$.
	\item Repeat Step $3$ until there is no such edge.
	\end{enumerate}
	We make the following observations regarding $\Psi_{\tr}(\ppp)$:
	\begin{itemize}
	\item Let $\Delta(\ppp)$ be the number of clauses deleted after Step $(1)$ and $(2)$ and denote $\gamma(\ppp)=e(\ppp)-v(\ppp)-f(\ppp)$. Then, $\Delta(\ppp)\leq \gamma(\ppp)+1$ holds since deletion of $\Delta(\ppp)$ clauses and $2\Delta(\ppp)$ internal edges in Step $(1)$ and $(2)$ do not affect the connectivity of the graph.
	\item Because $\Psi_{\tr}(\ppp)$ is a tree, Step $(1)$-$(4)$ deletes $\gamma(\ppp)+\Delta(\ppp)+1$ number of internal edges.
	\item For any $\fff\in \FFF$ corresponding to $\ppp$ through \eqref{eq:def:weight:proj:comp}, $\Psi_{\tr}(\fff)$ can be defined through the same algorithm above, with a slight change of Step $(3)$: the new boundary half-edge adjacent to $a^\prime$ has the same literal information as that of $e$, and its color (either $\bb_0$ or $\bb_1$) is determined by taking the same color as the other boundary half-edges adjacent to $a^\prime$. Then, $\Psi_{\tr}(\fff)$ corresponds to $\Psi_{\tr}(\ppp)$ and $w\left(\Psi_{\tr}(\fff)\right)^{\la} \leq 2^{-k\Delta(\ppp)}w(\fff)^\la$ holds since we have obtained $\Psi_{\tr}(\fff)$ from $\fff$ by deleting $\Delta(\ppp)$ clauses and some internal edges. Also, every $\fff^\prime$ corresponding to $\Psi_{\tr}(\ppp)$ is obtained by $\fff^\prime = \Psi_{\tr}(\fff)$ for some $\fff$ corresponding to $\ppp$, so 
	\begin{equation}\label{eq:apriori:cyclic:weight:ratio}
	    w\left(\Psi_{\tr}(\ppp)\right)^\la\leq 2^{-k\Delta(\ppp)}w(\ppp)^{\la}\,.
	\end{equation}
	\end{itemize}
	
	For $\ell,r, \gamma,\Delta\geq 0$, let $\Xi_{\ell,r}^{\gamma,\Delta}$ denote the collection of free subgraphs $\kF$ such that
	\begin{itemize}
	    \item $\kF=\sqcup_{i=1}^{\ell}\ppp_i$, where $\ppp_1,...,\ppp_r$ are cyclic projected components and $\ppp_{r+1},...,\ppp_{\ell}$ are tree projected components.
	    \item $\sum_{i=1}^{r}\gamma(\ppp_i)=\gamma$ and $\sum_{i=1}^{r}\Delta(\ppp_i)=\Delta$.
	    \item $\sum_{i=1}^{\ell} v(\ppp_i)\leq 7n/2^k$ and for any $v\geq 1$, $|\{i:v(\ppp_i)=v\}| \leq n2^{-kv/4}$.
	\end{itemize}
	Define $\Psi_{\ell,r}^{\gamma,\Delta}:\Xi_{\ell,r}^{\gamma,\Delta}\to \Xi_{\ell,0}^{0,0}$ by applying $\Psi_{\tr}$ componentwise, i.e. $\Psi_{\ell,r}^{\gamma,\Delta}(\kF)\equiv \sqcup_{i=1}^{\ell} \Psi_{\tr}(\ppp_i)$, for $\kF=\sqcup_{i=1}^{\ell} \ppp_i$. Note that in order for the set $\Xi_{\ell,r}^{\gamma,\Delta}$ to be non-empty, $\Delta \leq \gamma+r$ must hold. The following lemma is the crux of the proof of Proposition \ref{prop:1stmo:aprioriestimate} $(4)$.
	\begin{lemma}\label{lem:apriori:1stmo:ratio:cycle}
		For $k\geq k_0$, $n_\tsf \leq 7n/2^k$, $E_\rr \leq 7nd/2^k$,$m/n \in [\alpha_{\textsf{lbd}}, \alpha_{\textsf{ubd}}], n\geq n_0(k), r\geq 1,\gamma\geq 1, 0 \leq \Delta \leq \gamma+r$ and $\kF^\prime \in \Xi_{\ell,0}^{0,0}$, we have
	\begin{equation*}
	    \textnormal{\textbf{S}}^{\gamma,\Delta}_{\ell,r}(\dot{H}_\circ,\kF^\prime)\equiv \sup_{\kF \in (\Psi_{\ell,r}^{\gamma,\Delta})^{-1}(\kF^\prime)}\left|(\Psi_{\ell,r}^{\gamma,\Delta})^{-1}(\kF^\prime)\right|\frac{\E\bZ_\la[\dot{H}_\circ,\kF]}{\E\bZ_\la[\dot{H}_\circ,\kF^\prime]}\lesssim_{k} \frac{1}{r!}\left(\frac{Ck^2}{2^k}\right)^{r}\left(\frac{C\log^{3} n}{n}\right)^{\gamma},
	\end{equation*}
	where $C$ is a universal constant.
	\end{lemma}
	\begin{proof}
	For any $\kF \in (\Psi_{\ell,r}^{\gamma,\Delta})^{-1}(\kF^\prime)$, $\kF$ has $\Delta$ more non-separating clauses and $\gamma+r+\Delta$ more free edges than $\kF^\prime$. Hence, using \eqref{eq:apriori:1stmo:separating} (see Proposition \ref{prop:apriori:1stmo:separating} below for the proof) with $\delta_1=\Delta$ and $\delta_2=r+\gamma+\Delta$, the same calculation as done in \eqref{eq:apriori:1stmo:long} shows
	\begin{equation}\label{eq:apriori:upperbound:cyclicmo:ratio}
	    \frac{\E\bZ_\la[\dot{H}_\circ,\kF]}{\E\bZ_\la[\dot{H}_\circ,\kF^\prime]}\lesssim_{k} \frac{w(\kF)^\la}{w(\kF^\prime)^\la}\left(\frac{1}{nd}\right)^{\gamma+r+\Delta}e^{O(\gamma+r+\Delta)}\leq \frac{1}{2^{k\Delta}}\left(\frac{1}{nd}\right)^{\gamma+r+\Delta}e^{O(\gamma+r+\Delta)},
	\end{equation}
	where the last inequality is due to \eqref{eq:apriori:cyclic:weight:ratio}. We turn to upper bound $\left|(\Psi_{\ell,r}^{\gamma,\Delta})^{-1}(\kF^\prime)\right|$. Enumerate all projected components of $\kF^\prime$ by the number of variables and suppose there exists $\ell_i$ $a_i$-components for $1\leq i \leq K$, where $a_i$-component denotes a component with $a_i$ variables. Here, we assume $\{a_i\}_{1\leq i \leq K}$ are all different. Recalling $\kF^\prime \in \Xi_{\ell,0}^{0,0}$, we make the following observations.
	\begin{itemize}
	    \item Let $b^{\max}_i$ be the maximum number of clauses among $a_i$-components, then $b^{\max}_i < a_i$ holds since $\kF^\prime$ consists of tree components and each clause in the component has internal degree at least $2$.
	    \item $\sum_{i=1}^{K}\ell_i=\ell$ and $\sum_{i=1}^{K}\ell_i a_i \leq 7n/2^k$. Moreover, $1\leq \ell_i\leq n 2^{-ka_i/4}$ for any $1\leq i \leq K$. In particular, $a_i\leq \frac{4\log_2 n}{k}$.
	\end{itemize}
	Now observe that $\kF \in (\Psi_{\ell,r}^{\gamma,\Delta})^{-1}(\kF^\prime)$ can be generated as follows. Fix some $r_i, \ell_i$ and $\gamma_i$ for $1\leq i \leq K$. We iterate the following procedure for $1\leq i \leq K$. First, choose $r_i$ components from $\ell_i$ $a_i$-components to form cyclic components. Then, form $r_i+\gamma_i-\Delta_i$ edges among boundary half-edges of $r_i$ components, corresponding to the deleted cyclic edges from Step $(3)$ and $(4)$ above. The number of possible ways to do this can be bounded above by $\frac{(rda_i)^{r_i+\gamma_i-\Delta_i}(kb^{\max}_i)^{r_i+\gamma_i-\Delta_i}}{(r_i+\gamma_i-\Delta_i)!}$, which follows from the same argument as done to show \eqref{eq:apriori:Omega:upperbound}. Next, choose half-edges $c^{1}_1, ...,c^{1}_{\Delta_i}$ among half-edges adjacent to separating clauses. Then, choose half-edges $c^{2}_j,1\leq j\leq \Delta_i$, adjacent to the clause that $c^{1}_j$ is located at. $c^{1}_j$ and $c^{2}_j$ must have partners $d^{1}_j$ and $d^{2}_j$, which are adjacent to variables in the same cylcic $a_i$-component. $\{(c^{1}_j,d^{1}_j),(c^{2}_j,d^{2}_j)\}_{1\leq j \leq \Delta_i}$ form the edges deleted in Step $(1)$ and $(2)$ above and there are $2^{\Delta_i}\Delta_i!$ lists of ordered pairs yielding the same set of $2\Delta_i$ edges. Finally assign a literal, either $0$ or $1$ to the $r_i+\gamma_i+\Delta_i$ new edges. Therefore,
	\begin{multline}\label{eq:apriori:upperbound:Psi:inverse}
	    \left|(\Psi_{\ell,r}^{\gamma,\Delta})^{-1}(\kF^\prime)\right|\leq 2^{r+\gamma+\Delta}\sum_{\sum_{i=1}^{K}r_i=r}\sum_{\sum_{i=1}^{K}\gamma_i=\gamma}\sum_{\substack{\sum_{i=1}^{K}\Delta_i=\Delta\\ 0\leq \Delta_i \leq r_i+\gamma_i}}\prod_{i=1}^{K}\Bigg\{{\ell_i\choose r_i}\frac{(r_idka_ib^{\max}_i)^{r_i+\gamma_i-\Delta_i}}{(r_i+\gamma_i-\Delta_i)!}\\
	    \times\frac{(km)^{\Delta_i}k^{\Delta_i}(r_ida_i)^{\Delta_i}(da_i)^{\Delta_i}}{2^{\Delta_i}\Delta_i!}\Bigg\}.
	\end{multline}
	We can upper bound the term inside the product by
	\begin{multline}\label{eq:apriori:upperbound:Psi:inverse-2}
	{\ell_i\choose r_i}\frac{(r_idka_ib^{\max}_i)^{r_i+\gamma_i-\Delta_i}}{(r_i+\gamma_i-\Delta_i)!}\frac{(r_i d^2k^2ma_i^2)^{\Delta_i}}{2^{\Delta_i}\Delta_i!}\leq \frac{\ell_i^{r_i}}{r_i!}\frac{(r_idka_ib^{\max}_i)^{r_i+\gamma_i}}{(r_i+\gamma_i)!}\frac{1}{\Delta_i!}\left(\frac{dkma_i(r_i+\gamma_i)}{2b^{\max}_i}\right)^{\Delta_i}\\
	\leq e^{r_i+\gamma_i+\Delta_i}\frac{(\ell_i dka_i b^{\max}_i)^{r_i}}{r_i!}(dka_ib^{\max}_i)^{\gamma_i}\left(\frac{dkma_i(r_i+\gamma_i)}{2b^{\max}_i\Delta_i}\right)^{\Delta_i},
	\end{multline}
	where we used $x!\geq (x/e)^{x}$ in the last inequality. Hence, \eqref{eq:apriori:upperbound:cyclicmo:ratio}, \eqref{eq:apriori:upperbound:Psi:inverse} and \eqref{eq:apriori:upperbound:Psi:inverse-2} altogether show
	\begin{multline}
	        \textnormal{\textbf{S}}^{\gamma,\Delta}_{\ell,r}(\dot{H}_\circ,\kF^\prime)\lesssim_{k} e^{O(r+\gamma+\Delta)}\sum_{\sum_{i=1}^{K}r_i=r}\sum_{\sum_{i=1}^{K}\gamma_i=\gamma}\sum_{\substack{\sum_{i=1}^{K}\Delta_i=\Delta\\ 0\leq \Delta_i \leq r_i+\gamma_i}}\prod_{i=1}^{K}\Bigg\{\frac{1}{r_i!}\left(\frac{k\ell_i a_ib^{\max}_i}{n}\right)^{r_i}\left(\frac{ka_ib^{\max}_i}{n}\right)^{\gamma_i}\\
	        \times\left(\frac{ka_i(r_i+\gamma_i)}{b^{\max}_i\Delta_i}\right)^{\Delta_i}\Bigg\},
	\end{multline}
	where we used $d\leq k2^k$ to bound the term involving $\Delta_i$ in the equation above. Note that fixing $c>0$, $x\to (c/x)^{x}$ is increasing for $0<x<c/e$. Since $ka_i/b^{\max}_i>k>e$, the term involving $\Delta_i$ in the equation above is maximized at $\Delta_i=r_i+\gamma_i$, in the regime $0\leq \Delta_i \leq r_i+\gamma_i$. Also, the total number of $\Delta_i$ possible is $r_i+\gamma_i+1 \leq e^{r_i+\gamma_i}$, so we can further bound the \textsc{rhs} above by
	\begin{equation}\label{eq:apriori:cycle:final-1}
	\begin{split}
	\textnormal{\textbf{S}}^{\gamma,\Delta}_{\ell,r}(\dot{H}_\circ,\kF^\prime)&\lesssim_{k} \sum_{\sum_{i=1}^{K}r_i=r}\sum_{\sum_{i=1}^{K}\gamma_i=\gamma}\prod_{i=1}^{K}\Bigg\{\frac{1}{r_i!}\left(\frac{Ck^2\ell_i a_i^2}{n}\right)^{r_i}\left(\frac{Ck^2a_i^2}{n}\right)^{\gamma_i}\Bigg\}\\
	&\leq \frac{1}{r!}\left(\frac{Ck^2\sum_{i=1}^{K}\ell_i a_i^2}{n}\right)^{r}\left(\frac{Ck^2\sum_{i=1}^{K}a_i^2}{n}\right)^{\gamma},
	\end{split}
	\end{equation}
	where $C$ is a universal constant and used $\Delta \leq r+\gamma$ in the first inequality while we used the crude bound $1\leq \gamma!/\prod_{i=1}^{K} \gamma_i!$ in the second inequality. Finally, note that we can crudely bound
	\begin{equation}\label{eq:apriori:cycle:final-2}
	    \sum_{i=1}^{K}\ell_i a_i^2\leq  \frac{28n}{2^k}+n\sum_{a\geq 5}a^2 2^{-ka/4}\leq \frac{Cn}{2^k}\quad\textnormal{and}\quad \sum_{i=1}^{K}a_i^2 \leq \sum_{a=1}^{\frac{4\log_2 n}{k}}a^2\leq \frac{C\log^{3}n}{k^3}.
	\end{equation}
	Therefore, \eqref{eq:apriori:cycle:final-1} and \eqref{eq:apriori:cycle:final-2} conclude the proof.
	\end{proof}
	\begin{proof}[Proof of Proposition \ref{prop:1stmo:aprioriestimate} (4)]
	By Lemma \ref{lem:apriori:1stmo:ratio:cycle}, we have
	\begin{multline}
	    \frac{\E\bZ_{\lambda}[n_{\textnormal{cyc}}\geq r, e_{\textnormal{mult}}\geq \gamma,~~~~~~~\textnormal{and}~~~~~~~ \ee_{\frac{1}{4}}]}{\E\bZ_\la^{\tr}}\leq \sup_{\dot{H}_\circ,\ell\geq 1,\kF^\prime \in \Xi_{\ell,0}^{0,0}}\left\{\sum_{r^\prime=r}^{\ell}\sum_{\gamma^\prime\geq\gamma}\sum_{\Delta=0}^{r^\prime+\gamma^\prime} \textnormal{\textbf{S}}^{\gamma^\prime,\Delta}_{\ell,r^\prime}(\dot{H}_\circ,\kF^\prime)\right\}\\
	    \lesssim_{k} \sum_{r^\prime\geq r}\sum_{\gamma^\prime\geq\gamma}\frac{r^\prime+\gamma^\prime+1}{r^\prime !}\left(\frac{Ck^2}{2^k}\right)^{r^\prime}\left(\frac{C\log^{3} n}{nk}\right)^{\gamma^\prime}\leq \frac{1}{r!}\left(\frac{C^\prime k^2}{2^k}\right)^r \left(\frac{C^\prime \log^{3} n}{nk}\right)^{\gamma},
	\end{multline}
	where we used $r^\prime+\gamma^\prime+1\leq e^{r^\prime+\gamma^\prime}$ in the last inequality and $C^\prime$ is an another universal constant.
	\end{proof}
	\subsubsection{Estimates on separating constraints}\label{subsubsec:apriori:separating}
	We now aim to prove \eqref{eq:apriori:1stmo:separating}. From its definition in \eqref{eq:apriori:technical-2}, $f(m_\tns, E_\fs)$ equals the expectation of the contribution of $\hat{v}$ from separating clauses, under uniform matching of $km_\tns$ half-edges, of which $E_\fs$ are $\fs$ edges and others are $\bb$. Note that $\hat{v}(\cdot)$ is completely determined by the number of $\fs$ edges, so we write $\hat{v}(x)$ for the value of $\hat{v}$ containing $x$ number of $\fs$ edges and $k-x$ number of $\bb$ edges. Writing $\xi$ to be the proportion of $\fs$ edges, we have
	\begin{equation*}
	f(m_\tns,km_\tns \xi) = \E_{\xi}\left[\prod_{i=1}^{m_\tns} \hat{v}(X_i)\bigg|\sum_{i=1}^{m_\tns} X_i= km_\tns \xi \right],
	\end{equation*}
	where $\E_\xi$ denotes the expectation with respect to i.i.d. random variables $X_1,...,X_{m_\tns}$ with $X_i \sim \textnormal{Binomial}(k,\xi)$. We calculate $f(m_\tns, km_\tns\xi)$ by introducing a rescaling factor $\gamma \in \R$. If we let $p_\xi(x)\equiv \binom{k}{x}\xi^{x}(1-\xi)^{k-x}$, then
    \begin{equation}\label{eq:apriori:1stmo:f:prior}
    \begin{split}
    f(m_\tns,km_\tns \xi)
    &=\frac{\sum_{(x_i)_{i\leq m_\tns}}\one\Big(\sum_{i=1}^{m_{\tns}}x_i=km_\tns \xi\Big)e^{-km_\tns \xi \gamma} \prod_{i=1}^{m_\tns}\hat{v}(x_i)p_{\xi}(x_i)e^{\gamma x_i}}{\P_{\xi}\Big(\sum_{i=1}^{km_\tns}X_i=km_\tns \xi\Big)}\\
    &=\frac{\P_{\gamma,\xi}\Big(\sum_{i=1}^{km_\tns}\wt{X}_i=km_\tns \xi\Big)}{\P_{\xi}\Big(\sum_{i=1}^{km_\tns}X_i=km_\tns \xi\Big)}\cdot \exp\Big\{-m_{\tns}\Big(k\xi\gamma-\Lambda_{\xi}(\gamma)\Big)\Big\}\,,
    \end{split}
    \end{equation}
    where $\Lambda_{\xi}(\gamma)\equiv \log \E_{\xi}\big[\hat{v}(X)e^{\gamma X}\big]$, $X\sim \textnormal{Binomial}(k,\xi)$, and $\P_{\gamma,\xi}$ denotes probability with respect to i.i.d. random variables $\wt{X}_1,\ldots, \wt{X}_{m_{\tns}}\sim \nu_{\gamma,\xi}$. Here, $\nu_{\gamma,\xi}\in \PPP(\{0,\ldots,k\})$ is defined by   
	\begin{equation*}
	    \nu_{\gamma,\xi}(x)=\frac{p_{\xi}(x)\hat{v}(x)\exp(\gamma x)}{\sum_{x=0}^{k-2}p_{\xi}(x)\hat{v}(x)\exp(\gamma x)},\quad 0\leq x \leq k.
	\end{equation*}
    Here, we note that $\hat{v}(x)=0$ for $x\in \{k-1,k\}$ since a valid separating clause must have at least $2$ $\bb$-colored edges. We take $\gamma=\gamma(\xi)$ so that $\partial_{\gamma} \Lambda_{\xi}(\gamma)\equiv \sum_{x}x\nu_{\gamma,\xi}(x)=k\xi$, whose existence is guaranteed by Lemma \ref{lem:apriori:1stmo:gamma:exist} below for $0<\xi\leq 10/2^k$. Then, local CLT \cite{Borokov17} shows that we can approximate 
    \begin{equation*}
         f(m_\tns,km_\tns \xi)
   \asymp \bigg(\frac{\Var_{\xi}(X_1)}{\Var_{\gamma(\xi),\xi}(\wt{X}_1)}\bigg)^{1/2}\exp\Big\{-m_{\tns}\Big(k\xi\gamma-\Lambda_{\xi}\big(\gamma(\xi)\big)\Big)\Big\}\,.
    \end{equation*}
    Lemma \ref{lem:apriori:1stmo:gamma:exist} below shows that $\gamma(\xi)$ is uniformly bounded by $O(k/2^k)$. Also, since $\hat{v}(x)\in [1/2,1]$ for $0\leq x\leq k-2$, we have $\frac{\nu_{\gamma(\xi),\xi}(x)}{p_{\xi}(x)}\asymp_k 1$ uniformly over $0\leq x\leq k-2$. Thus, we have uniformly over the regime $0<\xi\leq 10/2^k$ that
	\begin{equation}\label{eq:apriori:1stmo:f}
	    f(m_\tns,km_\tns\xi)\asymp_k, \bigg(\frac{\Var_{\xi}(X_1)}{\Var_{\xi}(X_1 \,|\, X_1\leq k-2)}\bigg)^{1/2} \exp\Big\{-m_{\tns}\Big(k\xi\gamma-\Lambda_{\xi}\big(\gamma(\xi)\big)\Big)\Big\}\asymp_k \exp\Big\{-m_{\tns}\Big(k\xi\gamma-\Lambda_{\xi}\big(\gamma(\xi)\big)\Big)\Big\}\,.
	\end{equation}
	\begin{lemma}\label{lem:apriori:1stmo:gamma:exist}
	For $X \sim \textnormal{Binomial}(k,\xi)$, define $\Lambda_{\xi}(\gamma)\equiv \log \E_{\theta}[\hat{v}(X)e^{\gamma X}]$ for $\gamma \in \R$. In the regime where $0<\xi\leq 10/2^k$, there exists a unique $\gamma(\xi)$ such that $\partial_{\gamma}\Lambda_{\xi}\left(\gamma(\xi)\right)=k\xi$ and satisfies
	\begin{equation}\label{eq:lem:apriori:1stmo:gamma:upperbound}
	    \sup_{0<\xi\leq 10/2^k}|\gamma(\xi)|\lesssim \frac{k}{2^k}\quad\textnormal{and}\quad \sup_{0<\xi\leq 10/2^k}|\Lambda_{\xi}\left(\gamma(\xi)\right)|\lesssim \frac{k}{2^k}.
	\end{equation}
	\end{lemma}
	\begin{proof}
	Uniqueness of $\gamma(\xi)$ is guaranteed by strict convexity of $\Lambda_{\xi}(\cdot)$. To see existence, note that
	\begin{equation*}
	    \partial_{\gamma}\Lambda_{\xi}(\gamma)= \frac{\E_{\mu}[X\hat{v}(X)]}{\E_{\mu}[\hat{v}(X)]}\quad\textnormal{for}\quad \mu\equiv\mu(\gamma)\equiv \frac{\xi e^{\gamma}}{1-\xi+\xi e^{\gamma}}.
	\end{equation*}
	Hence, $\partial_{\gamma}\Lambda_{\xi}(\gamma)\in [0,k]$ and for any $\eps>0$, $\Lambda_{\xi}(\gamma)+\frac{1}{2}\eps\gamma^2$ is a convex function with derivative $ \partial_{\gamma}\Lambda_{\xi}(\gamma)+\eps \gamma$ tending in norm to $\infty$ as $|\gamma|\to \infty$. By Rockafellar's theorem (see e.g. Lemma 2.3.12 of \cite{DZ10}), there exists a unique $\gamma_{\eps}$ such that $\partial_{\gamma}\Lambda_{\xi}(\gamma_\eps)+\eps\gamma_{\eps}=k\xi$. We now show that $\gamma_{\eps}$ stays in a bounded region as $\eps \to 0$. We first claim that $\mu_\eps \equiv \mu(\gamma_\eps)\leq\frac{1}{k}$ for small enough $\eps$: suppose $\mu(\gamma_\eps)>\frac{1}{k}$. Then $e^{\gamma_\eps}>\frac{1-\xi}{(k-1)\xi}$ holds, and recalling $1/2\leq \hat{v}(x)\leq 1$ for $x\leq k-2$,
	\begin{equation*}
	    \gamma_\eps = \frac{1}{\eps}\left(k\xi - \frac{\E_{\mu_\eps}[X\hat{v}(X)]}{\E_{\mu_\eps}[\hat{v}(X)]}\right)\leq \frac{1}{\eps}\left(k\xi -\frac{\E_{\mu_\eps}[X\one\{X\leq k-2\}]}{2}\right)\leq\frac{1}{\eps}\left(k\xi -\frac{1}{4k}\right)\ll 0,
	\end{equation*}
	contradicting $e^{\gamma_\eps}>\frac{1-\xi}{(k-1)\xi}$. Thus $\limsup_{\eps \to 0}\gamma_{\eps}$ must be finite. For the lower bound,
	\begin{equation*}
	     \gamma_\eps = \frac{1}{\eps}\left(k\xi - \frac{\E_{\mu_\eps}[X\hat{v}(X)]}{\E_{\mu_\eps}[\hat{v}(X)]}\right)\geq \frac{1}{\eps}\left(k\xi -2\E_{\mu_\eps}[X\mid X\leq k-2]\right)\geq\frac{1}{\eps}\left(k\xi -2\mu_{\eps}\right)\geq\frac{\xi}{\eps}\left(k -\frac{2 e^{\gamma_\eps}}{1-\xi}\right),
	\end{equation*}
	so $\liminf_{\eps\to 0}\gamma_{\eps}$ must be finite. Therefore, there exists a unique $\gamma$ such that $\partial_\gamma\Lambda_\xi(\gamma)=k\xi$ and satisfies $\mu(\gamma)\leq \frac{1}{k}$.
	
	We now turn to prove \eqref{eq:lem:apriori:1stmo:gamma:upperbound}. Observe that $\hat{v}(x)=1-h(x)-\one\{x\geq k-1\}$, where
	\begin{equation}\label{eq:apriori:def:h}
	    h(x) \equiv
	    \begin{cases}
	    \frac{k+1}{2^{k-1}} &x=0\\
	    \frac{2^{x}}{2^{k-1}}&x=1,..., k-2\\
	    0& x=k-1,k
	    \end{cases}
	\end{equation}
	If we denote by $q_{\mu}$ the law of $Y \sim \textnormal{Binomial}(k-1,\mu)$, then $xp_{\mu}(x)=k\mu q_{\mu}(x-1)$, so
	\begin{equation}\label{eq:apriori:lem:crux}
	    k\xi=\frac{\E_{\mu}[X\hat{v}(X)]}{\E_{\mu}[\hat{v}(X)]}=k\mu\left(\frac{1-\P_{\mu}^{k-1}(X\geq k-2)-\E_{\mu}^{k-1}[h(X+1)]}{1-\P_{\mu}^{k}(X\geq k-1)-\E_{\mu}^{k}[h(X)]}\right),
	\end{equation}
	where $\mu=\mu\left(\gamma(\xi)\right)$ and $\P_{\mu}^{\ell}$ denotes the law of $\textnormal{Binomial}(\ell,\mu)$. Since we have already shown that $\mu\leq \frac{1}{k}$ holds, \eqref{eq:apriori:def:h} and \eqref{eq:apriori:lem:crux} show that $\xi= \mu\left(1+O(k2^{-k})\right)$, which implies the first inequality of \eqref{eq:lem:apriori:1stmo:gamma:upperbound}. Finally, observe that
	\begin{equation*}
	\begin{split}
	    \left|\Lambda_{\xi}\left(\gamma(\xi)\right)\right|
	    &=\left|\log\E_{\mu}[\hat{v}(X)]+\log\E_{\xi}[e^{\gamma(\xi) X}]\right|\\
	    &=\left|\log\left(1-\P_{\mu}(X\geq k-1)-\E_{\mu}[h(X)]\right)+k\log(1-\xi+\xi e^{\gamma(\xi)})\right|\lesssim \frac{k}{2^k},
	\end{split}
	\end{equation*}
	concluding the proof of \eqref{eq:lem:apriori:1stmo:gamma:upperbound}.
	\end{proof}
	\begin{prop}\label{prop:apriori:1stmo:separating}
	In the regime $(1-\frac{21k}{2^{k+1}})m\leq m_\tns \leq m_\tns +\delta_1\leq m$ and $0\leq E_\fs \leq E_\fs +\delta_2 \leq \frac{7k}{2^k}m$, where $\delta_1,\delta_2$ are integers, \eqref{eq:apriori:1stmo:separating} holds.
	\end{prop}
	\begin{proof}
	We may assume $E_\fs \neq 0$ since $\frac{f(m_\tns,0)}{f(m_\tns,1)}=\frac{1-(k+1)/2^{k-1}}{1-1/2^{k-2}}\leq 1$. Let $m_\tns^\prime \equiv m_\tns+\delta_1, \xi\equiv \frac{E_\fs}{km_\tns}, \xi^\prime \equiv \frac{E_\fs+\delta_2}{km_\tns^\prime}, \gamma\equiv\gamma(\xi)$ and $\gamma^\prime\equiv \gamma(\xi^\prime)$. Note that $0<\xi,\xi^\prime \leq 10/2^k$ and $m_\tns \xi\leq m_\tns^\prime \xi^\prime$, so \eqref{eq:apriori:1stmo:f} shows
	\begin{equation}\label{eq:apriori:1stmo:prop:intermediate}
	    \frac{f(m_\tns,E_\fs)}{f(m_\tns+\delta_1,E_\fs+\delta_2)}
	    \lesssim_k \exp\Big\{m_\tns\Big(k\gamma \xi -\Lambda_{\xi}(\gamma)\Big)-m_\tns^\prime \Big(k\gamma^\prime \xi^\prime-\Lambda_{\xi^\prime}(\gamma^\prime)\Big)\Big\}
	\end{equation}
	Let $m_t\equiv m_\tns+t \delta_1,\xi_t \equiv \frac{E_\fs+t\delta_2}{km_t},\gamma_t \equiv \gamma(\xi_t)$ and $\bar{f}(t) \equiv m_t\left(k\gamma_t\xi_t-\Lambda_{\xi_t}(\gamma_t)\right)$ for $0\leq t\leq 1$. Then,
	\begin{equation}\label{eq:apriori:1stmo:compute:partial:theta}
	m_\tns\Big(k\gamma \xi -\Lambda_{\xi}(\gamma)\Big)-m_\tns^\prime \Big(k\gamma^\prime \xi^\prime-\Lambda_{\xi^\prime}(\gamma^\prime)\Big)=\bar{f}(0)-\bar{f}(1)\leq \sup_{0\leq t\leq 1}\left|\frac{d\bar{f}(t)}{dt}\right|.
	\end{equation}
	We now aim to upper bound $\left|\frac{d\bar{f}(t)}{dt}\right|$. Note that we can compute $\partial_{\xi}\Lambda_{\xi_t}(\gamma_t)$ by
	\begin{equation*}
	   \partial_{\xi}\Lambda_{\xi_t}(\gamma_t)=\xi_t^{-1}\partial_{\gamma}\Lambda_{\xi_t}(\gamma_t)-(1-\xi_t)^{-1}\left(k-\partial_{\gamma}\Lambda_{\xi_t}(\gamma_t)\right)=0,
	\end{equation*}
	where the last equality is because $\partial_{\gamma}\Lambda_{\xi_t}(\gamma_t)=k\xi_t$. Hence, $\frac{d\bar{f}(t)}{dt}$ can be computed by
	\begin{equation}\label{eq:apriori:1stmo:prop:intermediate-2}
	    \left|\frac{d\bar{f}(t)}{dt}\right|=\left|\frac{dm_t}{dt}\left(k\gamma_t\xi_t-\Lambda_{\xi_t}(\gamma_t)\right)+m_t k \gamma_t \frac{d\xi_t}{dt}\right|=\Big|\delta_2\gamma_t-\delta_1 \Lambda_{\xi_t}(\gamma_t)\Big|\lesssim \frac{k}{2^k}(\delta_1+\delta_2)\,,
	\end{equation}
	where the inequality is due to Lemma \ref{lem:apriori:1stmo:gamma:exist}. Therefore, \eqref{eq:apriori:1stmo:prop:intermediate}-\eqref{eq:apriori:1stmo:prop:intermediate-2} conclude the proof.
	\end{proof}
	\subsection{Second moment}\label{subsec:apriori:secmo}
	Given a \textsc{nae-sat} instance $\GGG$ and a pair projected coloring $\bupi$, let $\pF=\pF(\bupi)$ be the union-free subgraph of $\bupi$. Here, unlike for the first moment, we view $\pF=\pF(\bupi)$ as the labeled subgraph, which is the disjoint union of projected union components $\upp$ of $\bupi$. That is, the boundary half-edge $e$ of a projected union component is labeled by $\bpi_e\in \Omega_{\textnormal{pj},2}^\fs$, and the inner full-edge $e$ of a projected union component is labeled by $(\bpi_e,\tL_e)\in\Omega_{\textnormal{pj},2}^\ff\times\{0,1\}$. For the first moment, such labeling scheme was redundant because every boundary half-edge adjacent to a variable (resp. clause) have the spin $\fs$ (resp. $\bb$). However, there are more variety of spins of boundary half-edges in the second moment. 
 
 We encode $\pF$ by the matching $M_{\pF}$ between the half-edges adjacent to union-free variables and pair-nonseparating clauses, where the half-edges involved in the matching are labeled by $(\bpi,\tL)\in\Omega_{\textnormal{pj},2}^\ff\times\{0,1\}$ while the rest of the half-edges are labeled by $\bpi\in \Omega_{\textnormal{pj},2}^\fs$. Moreover, let $\pHdot = \pHdot(\bupi)$ and $\pHhat=\pHhat(\bupi)$ denote the empirical distribution of pair-frozen variables and pair-forcing variables of $\bupi$ respectively, i.e. 
	\begin{equation*}
	\begin{split}
	    \pHdot(\butau)&=\frac{1}{n}|\{v\in V:\bupi_{\delta v}=\butau\}|\quad\textnormal{for}\quad \butau \in \prescript{}{2}\Omega_{\circ}\\
	    \pHhat(\butau)&=\frac{1}{m}|\{a\in F:\bupi_{\delta a}=\butau\}|\quad\textnormal{for}\quad \butau \in \prescript{}{2}\Omega_{\textnormal{fc}},
	\end{split}
	\end{equation*}
	where $\prescript{}{2}\Omega_{\circ}\equiv \{\rr\rr^=, \rr\rr^{\neq}, \bb\bb^=, \bb\bb^{\neq}, \rr\bb^=, \rr\bb^{\neq}, \bb\rr^=,\bb\rr^{\neq}\}^{d}\backslash\left(\{\bb\bb^=, \bb\bb^{\neq},\bb\rr^=,\bb\rr^{\neq}\}^{d}\sqcup\{\bb\bb^=, \bb\bb^{\neq}, \rr\bb^=, \rr\bb^{\neq}\}^{d}\right)$ and
	$\prescript{}{2}\Omega_{\textnormal{fc}}\equiv \textnormal{Per}\left(\rr\rr^{=}, (\bb\bb^=)^{k-1}\right)\sqcup \textnormal{Per}\left(\rr\rr^{\neq}, (\bb\bb^{\neq})^{k-1}\right)\sqcup \textnormal{Per}\left(\rr\bb^{=},\bb\rr^=, (\bb\bb^{\neq})^{k-2}\right)\sqcup \textnormal{Per}\left(\rr\bb^{\neq},\bb\rr^{\neq}, (\bb\bb^{=})^{k-2}\right)$. Denote by $\bZ^{2}_{\ula}[\pHdot,\pHhat,\pF]$ the contribution to $\bZ^{2}_{\ula}$ from pair projected configurations $\bupi$ with $\pHdot[\bupi]=\pHdot, \pHhat[\bupi]=\pHhat$ and $\pF[\bupi]=\pF$. Observe that the overlap defined by Definition \ref{def:overlapofcol} is determined by $\pHdot$ and $\pF$, which we denote by $\zeta(\pHdot,\pF)$.
	
	We now aim to compute $\bZ^{2}_{\ula}[\pHdot,\pHhat,\pF]$ using a similar matching scheme as the one used in \eqref{eq:apriori:1stmo:matching}. Note that the total mass of $\pHhat$ determines the number of pair-forcing clauses, which we denote by $m_{\textnormal{fc}}\equiv |F_{\textnormal{fc}}|$, where $F_{\textnormal{fc}}$ is the set of pair-forcing clauses. Let $n_\tsf$ and $m-m_\ns=|F_\sep|$ be the number of union-free variables and pair-separating clauses respectively, determined by $\pF$. Moreover, let $\bw(\pF)^{\ula} \equiv \prod_{\upp \in \pF} \ppjw(\upp)^{\ula}$, where $ \ppjw(\ppp)^{\ula}$ is defined in \eqref{eq:def:weight:proj:comp:2ndmo} and $\upp\in \pF$ denotes the projected union component $\upp$ in $\pF$. Similar to \eqref{eq:apriori:techinical-1} and \eqref{eq:apriori:1stmo:matching}, we can use Lemma \ref{lem:weight:projected} to have
	\begin{equation}\label{eq:apriori:2ndmo:matching}
	    \E \bZ^{2}_{\ula}[\pHdot,\pHhat,\pF]=2^{n-n_\tsf}\bw(\pF)^{\ula} c(n_\tsf,m_\ns, \pHdot,\pHhat)\E\left[\prod_{a\in F_\sep}\hat{v}_2(\bupi_{\delta a})\one\left\{B_1 \cap B_2 \cap B_3 \cap B_4\right\}\right],
	\end{equation}
	where $c(n_\tsf,m_\ns, \pHdot,\pHhat)\equiv \binom{n-n_\tsf}{n\pHdot}\binom{m-m_\ns}{m_{\textnormal{fc}}}\binom{m_{\textnormal{fc}}}{m_{\textnormal{fc}}\pHhat}$, and the expectation in the \textsc{rhs} is with respect to uniform matching of $nd$ half-edges with empirical distribution determined by $\pF$ and $\pHdot$, and
	\begin{equation*}
	\begin{split}
	    B_1 &\equiv \{\textnormal{Free edges are matched according to $M_{\kF}$ and the half-edges adjacent to pair-forcing}\\
	    &\textnormal{clauses are matched to the half-edges adjacent to pair-frozen variables with the same color}\},\\
	    B_2 &\equiv \{\textnormal{Boundary half-edges of $\pF$ adjacent to clauses are matched to half-edges adjacent to}\\
	    &\textnormal{pair-frozen variables with the same color}\},\\
	    B_3 &\equiv \{\textnormal{Clauses, which are not pair-forcing, have at most one red edge}\},\\
	    B_4 &\equiv \{\textnormal{$\hat{v}_2(\bupi)_{\delta a} \neq 0$ for $a\in F_\sep$}\}.
	\end{split}
	\end{equation*}
	Let $\bar{H}_\circ(\cdot)\equiv d^{-1}\sum_{\butau \in \prescript{}{2}{\Omega_\circ}}\pHdot(\butau)\sum_{i=1}^{d}\one\{\btau_i=\cdot \}$ and $\bar{H}_{\textnormal{fc}}(\cdot)\equiv d^{-1}\sum_{\butau \in \prescript{}{2}{\Omega_{\textnormal{fc}}}}\pHhat(\butau)\sum_{i=1}^{k}\one\{\btau_i=\cdot \}$ be the empirical distribution of colors adjacent to pair-frozen variables and pair-forcing clauses respectively. Let $E_\tsf$ be the number of free edges, determined by $\pF$. Then, we have
	\begin{equation}\label{eq:apriori:2ndmo:B1}
	    \P(B_1)=\frac{\prod_{\bsigma\in  \prescript{}{2}{\Omega_{\textnormal{fz}}}}\left(nd\bar{H}_{\circ}(\bsigma)\right)_{nd\bar{H}_{\textnormal{fc}}(\bsigma)}}{(nd)_{km_{\textnormal{fc}}+E_{\tsf}}},
	\end{equation}
	where $\prescript{}{2}{\Omega_{\textnormal{fz}}}\equiv \{\rr\rr^=, \rr\rr^{\neq}, \bb\bb^=, \bb\bb^{\neq}, \rr\bb^=, \rr\bb^{\neq}, \bb\rr^=,\bb\rr^{\neq}\}$. Let $g_{\partial}(\bsigma)\equiv g_{\partial,\pF}(\bsigma)$ be the the number of clause-adjacent boundary edges of $\pF$ for $\bsigma \in \prescript{}{2}{\Omega_{\textnormal{fz}}}\backslash \{\rr\rr^=, \rr\rr^{\neq}\}$. It is then straightforward to compute
	\begin{equation*}
	    \P(B_2\mid B_1)=\frac{\prod_{\bsigma\in \prescript{}{2}{\Omega_{\textnormal{fz}}}\backslash \{\rr\rr^=, \rr\rr^{\neq}\} }\left(nd\bar{H}_{\circ}(\bsigma)-nd\bar{H}_{\textnormal{fc}}(\bsigma)\right)_{g_{\partial}(\bsigma)}}{(nd-km_{\textnormal{fc}}-E_\tsf)_{km_\ns-E_\tsf}}.
	\end{equation*}
	Let $g(\rr)$ denote the number of unmatched red edges conditioned on the event $B_1 \cap B_2$. Note that $$g(\rr)= \sum_{\bsigma\in \{\fs\rr,\rr\fs,\rr\bb^=, \rr\bb^{\neq}, \bb\rr^=,\bb\rr^{\neq}\}}g_{\rr}(\bsigma),$$ where $g_{\rr}(\bsigma)$ for $\bsigma\in\{\fs\rr,\rr\fs\}$ is the number of variable-adjacent boundary half-edges of $\pF$ colored $\bsigma$ and $g_{\rr}(\bsigma)$ for $\bsigma\in\{\rr\bb^=, \rr\bb^{\neq}, \bb\rr^=,\bb\rr^{\neq}\}$ is the number of unmatched  $\bsigma$-half-edges adjacent to pair-frozen variables, conditioned on the event $B_1 \cap B_2$. Denote by $m_{\tns}\equiv m-m_{\ns}-m_{\textnormal{fc}}$ the number of pair-separating, but non-pair-forcing, clauses. Then, we can compute
	\begin{equation*}
	    \P(B_3\mid B_1\cap B_2) =\frac{k^{g(\rr)}(m_{\tns})_{g(\rr)}}{(km_{\tns})_{g(\rr)}}.
	\end{equation*}
	Finally, define $\underline{E}\equiv \left(E(\bsigma)\right)_{\bsigma \in \{\bb\bb^{=},\bb\bb^{\neq},\bb\fs,\fs\bb,\fs\fs\}}$, where $E(\bsigma)$ is the number of unmatched $\bsigma$-half-edges conditioned on the event $B_1\cap B_2 \cap B_3$. We note that $\underline{E}$ is determined by $\pHdot, \pHhat$ and $\pF$. Recalling the fact if $a\in F_{\sep}$ is forcing in either copy, then $\hat{v}_2(\bupi_{\delta a})=2^{-k+1}$, we can write
	\begin{equation}\label{eq:apriori:2ndmo:B4}
	\begin{split}
	    \E\left[\prod_{a\in F_\sep}\hat{v}_2(\bupi_{\delta a})\one\left\{B_1 \cap B_2 \cap B_3 \cap B_4\right\}\right] &= 2^{-(k-1)m_{\rr}}\E\left[\prod_{a\in F_\sep\backslash F_{\textnormal{fc}}}\hat{v}_2(\bupi_{\delta a})\one\left\{B_1 \cap B_2 \cap B_3 \cap B_4\right\}\right]\\
	    &\equiv 2^{-(k-1)m_{\rr}}f_2(m_{\tns}, \mathbf{g}_{\rr},\underline{E}),
	\end{split}
	\end{equation}
	where $\mathbf{g}_{\rr}\equiv \left(g_{\rr}(\bsigma)\right)_{\bsigma \in \{\fs\rr,\rr\fs,\rr\bb^=, \rr\bb^{\neq}, \bb\rr^=,\bb\rr^{\neq}\}}$ and $m_{\rr}$ denotes the total number of clauses containing red edges, i.e. forcing in either copy, conditioned on $B_1\cap B_2\cap B_3$. Note that $m_{\rr}$ is determined by $\pHdot$ and $\pHhat$. Therefore, \eqref{eq:apriori:2ndmo:B1}-\eqref{eq:apriori:2ndmo:B4} altogether show
	\begin{multline}\label{eq:apriori:2ndmo:moment:final}
	    \E \bZ^{2}_{\ula}[\pHdot,\pHhat,\pF] =2^{n-n_\tsf-(k-1)m_{\rr}}c(n_\tsf,m_\ns, \pHdot,\pHhat)\bw(\pF)^{\ula}\frac{\prod_{\bsigma\in  \prescript{}{2}{\Omega_{\textnormal{fz}}}}\left(nd\bar{H}_{\circ}(\bsigma)\right)_{nd\bar{H}_{\textnormal{fc}}(\bsigma)}}{(nd)_{km_{\textnormal{fc}}+E_{\tsf}}}\\
	    \times \frac{\prod_{\bsigma\in \prescript{}{2}{\Omega_{\textnormal{fz}}}\backslash \{\rr\rr^=, \rr\rr^{\neq}\} }\left(nd\bar{H}_{\circ}(\bsigma)-nd\bar{H}_{\textnormal{fc}}(\bsigma)\right)_{g_{\partial}(\bsigma)}}{(nd-km_{\textnormal{fc}}-E_\tsf)_{km_\ns-E_\tsf}}\frac{k^{g(\rr)}(m_{\tns})_{g(\rr)}}{(km_{\tns})_{g(\rr)}}f_2(m_{\tns}, \mathbf{g}_{\rr},\underline{E})
	\end{multline}
	\subsubsection{Exponential decay of union-free tree frequencies}\label{subsubsec:exp:decay:2ndmo} Having \eqref{eq:apriori:2ndmo:moment:final} in hand, we proceed in the same fashion as in the first moment. Let $\pF_\circ$ be a free subgraph in pair projected coloring, which does not have any isolated union-free variable nor any projected union component with $a$ variables and $b$ clauses. Denote a projected union component with $a$ variables and $b$ clauses by union $(a,b)$-component. Let $\prescript{}{2}{\Omega}^{a,b}_{\ell,A}(n_\tsf; \kF_\circ)$ be the set of free subgraphs $\kF$ in pair projected coloring such that
	\begin{itemize}
	    \item $\kF$ contains $\kF_\circ$ and has $|V(\kF)|=n_\tsf$ variables.
	    \item $\kF\backslash \kF_\circ$ consists of $\ell$ union $(a,b)$- components with all remaining projected union components having a single free variable.
	    \item Union $(a,b)$-components have $q\equiv \ell(a+b-1)+A$ internal edges. 
	\end{itemize}
	Define $\prescript{}{2}{\Phi}_{\ell,A}^{a,b}:\prescript{}{2}{\Omega}^{a,b}_{\ell,A}(n_\tsf; \kF_\circ)\to\prescript{}{2}{\Omega}^{a,b}_{0,0}(n_\tsf; \kF_\circ)$ by the following: for $\pF \in \prescript{}{2}{\Omega}^{a,b}_{\ell,A}(n_\tsf; \kF_\circ)$, let $\upp_1,...,\upp_\ell$ be the $\ell$ union $(a,b)$-components. For each $\upp_i, 1\leq i \leq \ell$, delete all clauses of $\upp_i$ and all half-edges adjacent to $\upp_i$. Then, the variables of $\upp_i$ become isolated with $d$ half-edges adjacent to them and one of the $d$ half-edges must contain $\ff$ color in at least one copy. Change all $\ff$ to $\fs$ in each half-edge, e.g. $\ff\sigma$ for $\sigma \in \{\rr,\bb,\fs\}$ is changed to $\fs\sigma$. Hence, $\upp_i$ is changed to $a$ isolated projected union components with valid neighbor colors, which we denote by $\Phi(\upp_i)$. Note that the each isolated component of $\Phi(\upp_i)$ has boundary colors, which have marginal $\fs^{d}$ in at least one of the copy. Then, $\prescript{}{2}{\Phi}_{\ell,A}^{a,b}(\pF)$ is defined by the resulting free subgraph, i.e. $\prescript{}{2}{\Phi}_{\ell,A}^{a,b}(\pF) \equiv (\pF\backslash \cup_{i=1}^{\ell}\upp_i) \cup \left(\cup_{i=1}^{\ell}\Phi(\upp_i)\right)$. We make the following observations on $\prescript{}{2}{\Phi}_{\ell,A}^{a,b}$:
	\begin{itemize}
	    \item For each union-free variable $v$ in the free subgraph $\prescript{}{2}{\Phi}_{\ell,A}^{a,b}(\pF)$, $v$ is frozen in first copy if and only if $v$ is frozen in first copy for $\pF$. The same holds for the second copy, so we have
	    \begin{equation*}
	        \zeta(\pHdot,\pF)=\zeta\left(\pHdot,\prescript{}{2}{\Phi}_{\ell,A}^{a,b}(\pF)\right)
	    \end{equation*}
	    \item For each union $(a,b)$-component $\upp_i$ ($1\leq i \leq \ell$) in $\pF$, there are at most $2^b$ many union-free component corresponding to $\pF$, i.e. $|(\textsf{R}_2)^{-1}(\upp_i)|\leq 2^b$. Hence, $\ppjw(\upp_i)^{\ula}\leq \frac{2^b}{2^{kb}}\ppjw\left(\Phi(\upp_i)\right)^{\ula}$. Therefore,
	    \begin{equation}\label{eq:apriori:2ndmo:F:PhiF:ratio}
	        \bw(\pF)^{\ula} \leq 2^{-(k-1)\ell b}\bw\left(\prescript{}{2}{\Phi}_{\ell,A}^{a,b}(\pF)\right)^{\ula}. 
	    \end{equation}
	\end{itemize}
	
	The following lemma is an analog of Lemma \ref{lem:apriori:1stmo:comparison} for the second moment.
	\begin{lemma}\label{lem:apriori:2ndmo:comparison}
	For $k\geq k_0$, $n_\tsf \leq 14n/2^k$, $E_\rr \leq 14nd/2^k$,$m/n \in [\alpha_{\textsf{lbd}}, \alpha_{\textsf{ubd}}]$, and $n\geq n_0(k)$, the following inequality holds. For $\pF^\prime \in \prescript{}{2}{\Omega}^{a,b}_{0,0}(n_\tsf; \kF_\circ)$ with $\zeta(\pHdot, \pF^\prime)\in [\frac{1}{2}-\frac{k^2}{2^{k/2}},\frac{1}{2}+\frac{k^2}{2^{k/2}}]$
	\begin{equation*}
	\begin{split}
	    \prescript{}{2}{\textnormal{\textbf{R}}}^{a,b}_{\ell,A}(\dot{H}_\circ,\pHhat, \pF^\prime) &\equiv \sup_{\pF \in \left(\prescript{}{2}{\Phi}_{\ell,A}^{a,b}\right)^{-1}(\pF^\prime)}\left|\left(\prescript{}{2}{\Phi}_{\ell,A}^{a,b}\right)^{-1}(\pF^\prime)\right|\frac{\E \bZ^{2}_{\ula}[\dot{H}_\circ,\pHhat,\pF]}{\E \bZ^{2}_{\ula}[\dot{H}_\circ,\pHhat, \pF^\prime]}\\
	    &\lesssim_{k} \left(\frac{n}{k q}\left(\frac{Ck}{2^k}\right)^{a}(Ck)^{b}\right)^{\ell}\left(\frac{C(a\land b)k}{n}\right)^{A},
	\end{split}
	\end{equation*}
	where $C$ is a universal constant.
	\end{lemma}
	\begin{proof}
	Given $\pF \in \left(\prescript{}{2}{\Phi}_{\ell,A}^{a,b}\right)^{-1}(\pF^\prime)$, we first compute $\frac{\E \bZ^{2}_{\ula}[\dot{H}_\circ,\pHhat,\pF]}{\E \bZ^{2}_{\ula}[\dot{H}_\circ,\pHhat, \pF^\prime]}$. By \eqref{eq:apriori:2ndmo:moment:final}, we have
	\begin{equation}\label{eq:apriori:2ndmo:long}
	\begin{split}
	     &\frac{\E \bZ^{2}_{\ula}[\dot{H}_\circ,\pHhat,\pF]}{\E \bZ^{2}_{\ula}[\dot{H}_\circ,\pHhat, \pF^\prime]}\\ &\leq\underbrace{\frac{\bw(\pF)^{\ula}}{\bw(\pF^\prime)^{\ula}}}_\text{(A)}~~\underbrace{\frac{(nd)_{km_{\textnormal{fc}}+E_{\tsf}^\prime}}{(nd)_{km_{\textnormal{fc}}+E_{\tsf}}}}_\text{(B)}~~\overbrace{\frac{\frac{\prod_{\bsigma\in \prescript{}{2}{\Omega_{\textnormal{fz}}}\backslash \{\rr\rr^=, \rr\rr^{\neq}\} }\left(nd\bar{H}_{\circ}(\bsigma)-nd\bar{H}_{\textnormal{fc}}(\bsigma)\right)_{g_{\partial}(\bsigma)}}{(nd-km_{\textnormal{fc}}-E_\tsf)_{km_\ns-E_\tsf}}}{\frac{\prod_{\bsigma\in \prescript{}{2}{\Omega_{\textnormal{fz}}}\backslash \{\rr\rr^=, \rr\rr^{\neq}\} }\left(nd\bar{H}_{\circ}(\bsigma)-nd\bar{H}_{\textnormal{fc}}(\bsigma)\right)_{g_{\partial}^\prime(\bsigma)}}{(nd-km_{\textnormal{fc}}-E_\tsf^\prime)_{km_\ns^\prime-E_\tsf^\prime}}}}^\text{(C)}~~\overbrace{\frac{\frac{k^{g(\rr)}(m_{\tns})_{g(\rr)}}{(km_{\tns})_{g(\rr)}}}{\frac{k^{g^\prime(\rr)}(m_{\tns}^\prime)_{g^\prime(\rr)}}{(km_{\tns}^\prime)_{g^\prime(\rr)}}}}^\text{(D)}~~\overbrace{\frac{f_2(m_{\tns}, \mathbf{g}_{\rr},\underline{E})}{f_2(m_{\tns}^\prime, \mathbf{g}_{\rr}^\prime,\underline{E}^\prime)}}^\text{(E)},
	\end{split}
	\end{equation}
	where $E_\tsf^\prime, m_\ns^\prime, g_{\partial}^\prime, m_{\tns\tns}^\prime, g^\prime(\rr), \mathbf{g}_{\rr}^\prime$ and $\underline{E}^\prime$ correspond to $\pF^\prime$. We make the following observations:
	\begin{itemize}
	\item $E_\tsf^\prime =E_\tsf -q,m_\ns^\prime = m_\ns -\ell b$ and $m_{\tns}^\prime=m_{\tns}+\ell b$.
	\item $g_{\partial}(\bsigma)-g_{\partial}^\prime(\bsigma)\geq 0$ for $\bsigma\in \prescript{}{2}{\Omega}_{\textnormal{fz}}\backslash \{\rr\rr^{=},\rr\rr^{\neq}\}$ and $\sum_{\bsigma\in \prescript{}{2}{\Omega}_{\textnormal{fz}}\backslash \{\rr\rr^{=},\rr\rr^{\neq}\}}\left(g_{\partial}(\bsigma)-g_{\partial}^\prime(\bsigma)\right)=k\ell b -q$.
	\item $g^\prime_{\rr}(\bsigma)-g_{\rr}(\bsigma)=g_{\partial}(\bsigma)-g_{\partial}^\prime(\bsigma)$, for $\bsigma\in \{\bb\rr^{=},\bb\rr^{\neq},\rr\bb^{=},\rr\bb^{\neq}\}$. Since every non-pair-forcing clause must contain at most one red edge, $\sum_{\bsigma\in \{\bb\rr^{=},\bb\rr^{\neq},\rr\bb^{=},\rr\bb^{\neq}\}}\left(g^\prime_{\rr}(\bsigma)-g_{\rr}(\bsigma)\right)\leq \ell b$.
	\item $g^\prime_{\rr}(\bsigma)-g_{\rr}(\bsigma)\geq 0$ for $\bsigma \in\{\fs\rr,\rr\fs\}$ and $\sum_{\bsigma \in \{\fs\rr,\rr\fs\}}\left(g^\prime_{\rr}(\bsigma)-g_{\rr}(\bsigma)\right) \leq q$.
	\item $E^\prime(\bsigma)-E(\bsigma)\geq 0$ for $\bsigma \in \{\bb\bb^{=},\bb\bb^{\neq},\fs\bb,\bb\fs,\fs\fs\}$, $\sum_{\bsigma \in \{\fs\bb,\bb\fs,\fs\fs\}}\left(E^\prime(\bsigma)-E(\bsigma)\right)\leq q$ and\\%space
    $\sum_{\bsigma \in \bb\bb^{=},\bb\bb^{\neq}}\left(E^\prime(\bsigma)-E(\bsigma)\right)\leq k\ell b$
	\end{itemize}
	With above observations in mind, $\text{(A),(B),(C),(D),(E)}$ in \eqref{eq:apriori:2ndmo:long} can be bounded by the following:
	\begin{itemize}
	    \item $\textnormal{(A)}\leq 2^{-(k-1)\ell b}$ by \eqref{eq:apriori:2ndmo:F:PhiF:ratio}.
	    \item $\textnormal{(B)}=\frac{1}{(nd-km_{\textnormal{fc}}-E_\tsf^\prime)_{q}}\leq e^{O(q)}\left(\frac{1}{nd}\right)^{q}$ since $E_\tsf+km_{\textnormal{fc}}\leq \frac{28km}{2^k}$.
	    \item $\zeta(\pHdot, \pF) \in [\frac{1}{2}-\frac{k^2}{2^{k/2}},\frac{1}{2}+\frac{k^2}{2^{k/2}}]$ implies that $nd\pHdot(\bb\bb^{=}), nd\pHdot(\bb\bb^{\neq})\leq \left(\frac{1}{2}+\frac{k^2}{2^{k/2}}\right)nd$, so
	    \begin{equation*}
	    \begin{split}
	        \textnormal{(C)}=&\frac{(nd-km_{\textnormal{fc}}-E_\tsf+q)_{km_\ns^\prime-E_\tsf^\prime}}{(nd-km_{\textnormal{fc}}-E_\tsf)_{km_\ns^\prime-E_\tsf^\prime}}\frac{\prod_{\bsigma\in \prescript{}{2}{\Omega_{\textnormal{fz}}}\backslash \{\rr\rr^=, \rr\rr^{\neq}\} }\left(nd\bar{H}_{\circ}(\bsigma)-nd\bar{H}_{\textnormal{fc}}(\bsigma)-g_{\partial}^\prime(\bsigma)\right)_{g_{\partial}(\bsigma)-g_{\partial}^\prime(\bsigma)}}{(nd-km_{\textnormal{fc}}-E_\tsf)_{km_\ns^\prime-E_\tsf^\prime}}\\
	        \leq &e^{O(q)}\left(\frac{1}{2}\right)^{k\ell b -q}\leq e^{O(q)}\left(\frac{1}{2}\right)^{k\ell b}.
	    \end{split}
	    \end{equation*}
	    \item Note that $m_{\tns\tns}=m-m_\ns- m_{\textnormal{fc}}\geq (1-28k/2^k)m$, so we can bound
	    \begin{equation*}
	        \textnormal{(D)}=\prod_{i=0}^{g(\rr)-1}\frac{km_{\tns}-ki}{km_{\tns}-i}\left(\prod_{i=0}^{g^\prime(\rr)-1}\frac{km_{\tns}^\prime-ki}{km_{\tns}^\prime-i}\right)^{-1}\leq \prod_{i=g(\rr)}^{g^\prime(\rr)-1}\frac{km_{\tns}^\prime-ki}{km_{\tns}^\prime-i}\leq e^{O(q)}.
	    \end{equation*}
	    \item In Section \ref{subsubsec:apriori:separating:2ndmo}, we show in Proposition \ref{prop:apriori:2ndmo:separating} that in the stated regime,
	    \begin{equation}\label{eq:apriori:2ndmo:separating}
	        \frac{f_2(m_{\tns}, \mathbf{g}_{\rr},\underline{E})}{f_2(m_{\tns}+\delta_\circ, \mathbf{g}_{\rr}+\underline{\delta}_{\rr},\underline{E}+\underline{\delta})}\lesssim_{k}\exp\bigg\{O\left(\frac{k^{4}}{2^{k/2}}\right)\Big(||\underline{\delta}_{\rr}||_1+||\underline{\delta}||_1\Big)\bigg\},
	    \end{equation}
	    for $\delta_\circ\geq 0, \underline{\delta}_{\rr}\geq 0$ and $\underline{\delta}=\left(\delta(\bsigma)\right)_{\bsigma \in \{\bb\bb^{=},\bb\bb^{\neq},\fs\bb,\bb\fs,\fs\fs\}}$ with $\delta(\fs\bb),\delta(\bb\fs),\delta(\fs\fs)\geq 0$. Using \eqref{eq:apriori:2ndmo:separating} for $\delta_{\circ}=\ell b, \underline{\delta}_{\rr}=\mathbf{g}_{\rr}^\prime-\mathbf{g}_{\rr}$ and $\underline{\delta}=\underline{E}^\prime -\underline{E}$ shows $\textnormal{(E)}\leq e^{O(q)}$.
	\end{itemize}
	Therefore, the \textsc{rhs} of \eqref{eq:apriori:2ndmo:long} can be bounded by
	\begin{equation}\label{eq:apriori:2ndmo:ratio}
	    \frac{\E \bZ^{2}_{\ula}[\dot{H}_\circ,\pHhat,\pF]}{\E \bZ^{2}_{\ula}[\dot{H}_\circ,\pHhat, \pF^\prime]}\lesssim_{k} \frac{1}{2^{2k\ell b}}\left(\frac{1}{nd}\right)^{q}e^{O(q)}.
	\end{equation}
	We remark that compared to \eqref{eqn:apriori:1stmo:ratio:final}, the extra $2^{-k\ell b}$ term comes from matching the clause-adjacent boundary half-edges of $\ell$  union $(a,b)$ components in the near-independence regime. We turn now to upper bound $\left|\left(\prescript{}{2}{\Phi}_{\ell,A}^{a,b}\right)^{-1}(\pF^\prime)\right|$. Note that $\pF \in\left(\prescript{}{2}{\Phi}_{\ell,A}^{a,b}\right)^{-1}(\pF^\prime) $ can be obtained by the same procedure as in the procedure to obtain $\pF^\prime \in \Omega_{\ell,A}^{a,b}(n_{\tsf},\kF_\circ)$ in the paragraph above \eqref{eq:apriori:Omega:upperbound}, except that we do not choose $n_{\tsf}-|V(\pF_\circ)|$ variables among $n-|V(\pF_\circ)|$, since they are already determined by $\pF^\prime$, and we have to choose the colors of the edges of $\ell$ union-$(a,b)$ components components in $\pF$. There are at most $22$ possible choices for the colors and the literals of the $q$ inner edges. For the new boundary half-edges adjacent to $\ell b$ clauses, there are at most $2^{k-2}+2(k-2)$ choices for each $\ell b$ clauses, where the maximum number of choices comes from the clauses having $2$ internal edges with color $\ff\bb_{x}$ and $\ff \bb_{x^\prime}$, so the same calculation done in \eqref{eq:apriori:Omega:upperbound} show
	\begin{equation}\label{eq:apriori:2ndmo:Phi:inverse:bound}
	    \left|\left(\prescript{}{2}{\Phi}_{\ell,A}^{a,b}\right)^{-1}(\pF^\prime)\right|\leq {m-|F(\pF_\circ)| \choose\ell b}{n_{\tsf}-|V(\pF_\circ)| \choose \ell a} \frac{(\ell a)!(\ell b)!}{\ell!(a!)^\ell (b!)^\ell}\frac{(\ell da)^{q}(kb)^q}{q!}2^{k\ell b}e^{O(q)}.
	\end{equation}
	Observe that compared to the bound \eqref{eq:apriori:omega:inter}, \eqref{eq:apriori:2ndmo:Phi:inverse:bound} has an extra $2^{k\ell b}$ term, matching the extra $2^{-k \ell b}$ term of \eqref{eq:apriori:2ndmo:ratio} compared to \eqref{eqn:apriori:1stmo:ratio:final}. Therefore, having \eqref{eq:apriori:2ndmo:ratio} and \eqref{eq:apriori:2ndmo:Phi:inverse:bound} in hand, the same calculation done in \eqref{eq:lem:apriori:1stmo:comparison:final} concludes the proof.
	\end{proof}
	Having Lemma \ref{lem:apriori:2ndmo:comparison} in hand, the proof of Proposition \ref{prop:2ndmo:aprioriestimate} $(1),(2),(3)$ is a repeat of the proof of Proposition \ref{prop:1stmo:aprioriestimate} $(1),(2),(3)$, and hence is omitted.
	\subsubsection{Contribution from cycles}\label{subsubsec:apriori:cycles:2ndmo}
	Given a projected union component $\upp$, we find a subtree $\prescript{}{2}{\Psi}_{\tr}(\upp)$ of $\upp$, which is a valid projected union component, using a similar algorithm as the one used to define $\Psi_{\tr}(\ppp)$:
	\begin{enumerate}
	\item[\textnormal{Step $1^\prime$:}] Find and delete a clause $a \in F(\upp)$ such that it has internal degree 2 with internal edges $e_1=(av_1)$ and $e_2=(av_2)$, and deleting $a$ doesn't affect the connectivity of $\upp$. Say $e_1$ has color $\sigma_1^{1}\sigma_1^{2}$ and $e_2$ has color $\sigma_2^{1} \sigma_2^{2}$, where $\sigma_{i}^{j} \in \{\rr_0,\rr_1,\bb_0,\bb_1,\fs,\ff\}, i,j\in \{1,2\}$. The half-edges of $e_1$ and $e_2$ hanging on $v_1$ and $v_2$ respectively become boundary half-edges with color $\tau_1^{1}\tau_1^{2}$ and $\tau_2^{1}\tau_2^{2}$, where $\tau_i^{j} \in \{\rr,\bb,\fs\}, i,j \in \{1,2\}$ is obtained from $\sigma_i^{j}$ by deleting $0$ and $1$ if it has any, and substituting $\ff$ by $\fs$.
	\item[\textnormal{Step $2^\prime$:}] Repeat \textnormal{Step $1^\prime$} until there is no such clause.
	\item[\textnormal{Step $3^\prime$:}] Find a tree-excess edge $e=(a^\prime v^\prime)$, with color $\bsigma = (\sigma^{1}\sigma^{2})$ and cut $e$ in half to make two boundary half-edges adjacent to $a^\prime$ and $v^\prime$ respectively. The new boundary half-edge adjacent to $v^\prime$ is colored $\tau^{1}\tau^{2}$, where $\tau^{i}, i\in \{1,2\}$ is obtained from $\sigma^{i}$ by the same procedure as in Step $1^\prime$ above while the new boundary half-edge adjacent to $a^\prime$ is colored $\bsigma^\prime \in \{\bb\bb^{=}, \bb\bb^{\neq}\}$, where $\bsigma^\prime$ is chosen from $\{\bb\bb^{=}, \bb\bb^{\neq}\}$(may not be unique) to make the colors neighboring $a^\prime$ to be valid, e.g. if there exists a boundary edge of $a^\prime$ colored $\rr \bb^{=}$, we must take $\bsigma^\prime$ to be $\bb\bb^{\neq}$.
	\item[\textnormal{Step $4^\prime$:}] Repeat \textnormal{Step $3^\prime$} until there is no such edge.
	\end{enumerate}
	We make the following observations on $\prescript{}{2}{\Psi}_{\tr}(\upp)$:
	\begin{itemize}
	\item For a projected union component $\upp$, let $\Delta(\upp) \equiv |F(\upp)|-|F(\prescript{}{2}{\Psi}_{\tr}(\upp)|$. Then, $\Delta(\upp)\leq \gamma(\upp)+1$.
	\item $\prescript{}{2}{\Psi}_{\tr}(\upp)$ has $\gamma(\upp)+\Delta(\upp)+1$ less internal edges than $\upp$.
	\item Similar to \eqref{eq:apriori:cyclic:weight:ratio}, we have
	\begin{equation}\label{eq:apriori:2ndmo:cylcic:weight:ratio}
	\bw\left(\prescript{}{2}{\Psi}_{\tr}(\upp)\right)^{\ula}\leq 2^{-k\Delta(\upp)}\bw(\upp)^{\ula}
	\end{equation}
	\end{itemize}
	For $\ell,r, \gamma,\Delta\geq 0$, let $\prescript{}{2}{\Xi}_{\ell,r}^{\gamma,\Delta}$ denote the collection of free subgraphs $\kF$ such that
	\begin{itemize}
	    \item $\kF=\sqcup_{i=1}^{\ell}\upp_i$, where $\upp_1,...,\upp_r$ are cyclic projected union components and $\upp_{r+1},...,\upp_{\ell}$ are tree projected union components.
	    \item $\sum_{i=1}^{r}\gamma(\upp_i)=\gamma$ and $\sum_{i=1}^{r}\Delta(\upp_i)=\Delta$.
	    \item $\sum_{i=1}^{\ell} v(\upp_i)\leq 14n/2^k$ and for any $v\geq 1$, $|\{i:v(\upp_i)=v\}| \leq n2^{-kv/4}$.
	\end{itemize}
	Define $\prescript{}{2}{\Psi}_{\ell,r}^{\gamma,\Delta}:\prescript{}{2}{\Xi}_{\ell,r}^{\gamma,\Delta}\to \prescript{}{2}{\Xi}_{\ell,0}^{0,0}$ by acting $\prescript{}{2}{\Psi}_{\tr}$ componentwise. Note that in order for the set $\prescript{}{2}{\Xi}_{\ell,r}^{\gamma,\Delta}$ to be non-empty, $\Delta \leq \gamma+r$ must hold. The following lemma is an analog of Lemma \ref{lem:apriori:1stmo:ratio:cycle} for the second moment.
	\begin{lemma}\label{lem:apriori:2ndmo:ratio:cycle}
	For $k\geq k_0$, $n_\tsf \leq 14n/2^k$, $E_\rr \leq 14nd/2^k$,$m/n \in [\alpha_{\textsf{lbd}}, \alpha_{\textsf{ubd}}], n\geq n_0(k), r\geq 1,\gamma\geq 1, 0\leq \Delta \leq \gamma+r$ and $\kF^\prime \in \prescript{}{2}{\Xi}_{\ell,0}^{0,0}$, with $\zeta(\pHdot, \pF^\prime) \in [\frac{1}{2}-\frac{k^2}{2^{k/2}},\frac{1}{2}+\frac{k^2}{2^{k/2}}]$, we have 
	\begin{equation*}
	    \prescript{}{2}{\textnormal{\textbf{S}}}^{\gamma,\Delta}_{\ell,r}(\dot{H}_\circ,\pHhat, \kF^\prime)\equiv \sup_{\kF \in (\prescript{}{2}{\Psi}_{\ell,r}^{\gamma,\Delta})^{-1}(\kF^\prime)}\left|(\prescript{}{2}{\Psi}_{\ell,r}^{\gamma,\Delta})^{-1}(\kF^\prime)\right|\frac{\E\bZ_\la[\dot{H}_\circ,\pHhat, \kF]}{\E\bZ_\la[\dot{H}_\circ,\pHhat, \kF^\prime]}\lesssim_{k} \frac{1}{r!}\left(\frac{Ck^2}{2^k}\right)^{r}\left(\frac{C\log^{3} n}{n}\right)^{\gamma},
	\end{equation*}
	where $C$ is a universal constant.
	\end{lemma}
	\begin{proof}
	Fix some $\pF \in (\prescript{}{2}{\Psi}_{\ell,r}^{\gamma,\Delta})^{-1}(\kF^\prime)$. Recall \eqref{eq:apriori:2ndmo:long} and let $m_{\tns\tns}, \mathbf{g}_{\rr}, \underline{E}$ correspond to $\pF$ and let $m_{\tns\tns^\prime}, \mathbf{g}^\prime_{\rr}, \underline{E}^\prime$ correspond to $\pF^\prime$. Note that $\pF$ has $(k-1)\Delta -\gamma-r$ more clause-adjacent boundary half-edges and $\gamma+r+\Delta$ more internal edges compared to $\pF^\prime$. Moreover, observe that $m_{\tns\tns}^\prime= m_{\tns\tns}+\Delta, \mathbf{g}_{\rr} \leq \mathbf{g}^\prime_{\rr}, ||\mathbf{g}_{\rr}^\prime||_1 \leq ||\mathbf{g}_{\rr}||_1 +\gamma+r, E(\bsigma) \leq E^\prime(\bsigma), \bsigma \in \{\bb\fs,\fs\bb,\fs\fs\}$ and $||\underline{E}^\prime-\underline{E}||_1\leq k\Delta+\gamma+r$, so assuming \eqref{eq:apriori:2ndmo:separating}, similar calculations done in \eqref{eq:apriori:2ndmo:long} and \eqref{eq:apriori:2ndmo:ratio} show
	\begin{equation}\label{eq:apriori:2ndmo:cyclic:ratio}
	\frac{\E\bZ_\la[\dot{H}_\circ,\pHhat, \kF]}{\E\bZ_\la[\dot{H}_\circ,\pHhat, \kF^\prime]} \leq \frac{\bw(\pF)^{\ula}}{\bw(\pF^\prime)^{\ula}}\frac{1}{2^{k\Delta}}\left(\frac{1}{nd}\right)^{\gamma+r+\Delta}e^{O(\gamma+r+\Delta)}\leq \frac{1}{2^{2k\Delta}}\left(\frac{1}{nd}\right)^{\gamma+r+\Delta}e^{O(\gamma+r+\Delta)},
	\end{equation}
	where the last inequality is due to \eqref{eq:apriori:2ndmo:cylcic:weight:ratio}. We turn to upper bound $\left|(\prescript{}{2}{\Psi}_{\ell,r}^{\gamma,\Delta})^{-1}(\kF^\prime)\right|$. Enumerate all projected union components of $\kF^\prime$ by the number of variables and suppose there exists $\ell_i$ $a_i$-components for $1\leq i \leq K$, where $a_i$-component denotes a component with $a_i$ variables. Here, we assume $\{a_i\}_{1\leq i \leq K}$ are all different. Let $b^{\textnormal{max}}_i$ be the maximum number of clauses among $\ell_i$ $a_i$-components. Recalling \eqref{eq:apriori:upperbound:Psi:inverse}, $\left|(\prescript{}{2}{\Psi}_{\ell,r}^{\gamma,\Delta})^{-1}(\kF^\prime)\right|$ can be upper bounded by the same quantity, except there are extra choices for the colors of the internal edges and boundary edges. We can bound the number of such choices by $2^{k\Delta}e^{O(\gamma+r+\Delta)}$, so we have
	\begin{multline}\label{eq:apriori:upperbound:Psi:inverse:2ndmo}
	    \left|(\prescript{}{2}{\Psi}_{\ell,r}^{\gamma,\Delta})^{-1}(\kF^\prime)\right|\leq e^{O(r+\gamma+\Delta)}2^{k\Delta}\sum_{\sum_{i=1}^{K}r_i=r}\sum_{\sum_{i=1}^{K}\gamma_i=\gamma}\sum_{\substack{\sum_{i=1}^{K}\Delta_i=\Delta\\ 0\leq \Delta_i \leq r_i+\gamma_i}}\prod_{i=1}^{K}\Bigg\{{\ell_i\choose r_i}\frac{(r_idka_ib^{\max}_i)^{r_i+\gamma_i-\Delta_i}}{(r_i+\gamma_i-\Delta_i)!}\\
	    \times\frac{(km)^{\Delta_i}k^{\Delta_i}(r_ida_i)^{\Delta_i}(da_i)^{\Delta_i}}{2^{\Delta_i}\Delta_i!}\Bigg\}.
	\end{multline}
	Having \eqref{eq:apriori:2ndmo:cyclic:ratio} and \eqref{eq:apriori:upperbound:Psi:inverse:2ndmo} in hand, the rest of the proof is identical to the proof of Lemma \ref{lem:apriori:1stmo:ratio:cycle}.
	\end{proof}
	Having Lemma \ref{lem:apriori:2ndmo:ratio:cycle} in hand, the proof of Proposition \ref{prop:2ndmo:aprioriestimate} $(4)$ is identical to the proof of Proposition \ref{prop:1stmo:aprioriestimate} $(4)$.
	\begin{proof}[Proof of Lemma \ref{cor:B:close:optimal:2ndmo}]
	Throughout, we fix $\beta \in (1,10)$. First, Lemma \ref{lem:apriori:2ndmo:comparison} implies that we have
	\begin{equation}\label{eq:overlap:optimal:2ndmo:1}
	\E\bZ^{2}_{\ula^\star,\ind}\big[n_{\vvv}\geq 1\textnormal{ for some } \vvv\in \FFF_2\textnormal{ s.t. } f(\vvv)\geq v(\vvv)+\log^{\beta} n\big]\leq e^{-\Omega_{k}(\log^{\beta+1} n)}\E \bZ^2_{\ula^\star \ind}.
	\end{equation}
	Indeed, let $\ell_{a,b}$ be the number of union $(a,b)$-components. Then, proceeding in a similar fashion as in \eqref{eq:proof:apriori:1stmo:1}, Lemma \ref{lem:apriori:2ndmo:comparison} shows
	\begin{equation*}
	   \frac{ \E\bZ^2_{\ula^\star,\ind}[\exists b\geq a+\log^{\beta} n,~~~~~\ell_{a,b}\geq 1]}{\E \bZ^2_{\ula^\star,\ind}}\lesssim_{k}\sum_{a=1}^{ 14n/2^k}\sum_{b=a+\lceil\log^{\beta} n\rceil}^{14km/2^k}\frac{C^{2a+1}k^{2a}a}{2^{ka}b}\left(\frac{C^2k^2a}{n}\right)^{b-a}\leq e^{-\Omega_{k}(\log^{\beta+1} n)}.
	\end{equation*}
	Similarly, Lemma \ref{lem:apriori:2ndmo:comparison} also implies that
	\begin{equation}\label{eq:overlap:optimal:2ndmo:2}
	    \E \bZ^{2}_{\ula^\star,\ind}\big[n_{\vvv}\geq 1\textnormal{ for some } \vvv\textnormal{ s.t. } v(\vvv)\geq\log^{\beta} n\textnormal{ and }n_{\vvv}=0\textnormal{ if }f(\vvv)\geq v(\vvv)+\log^{\beta} n\big]\leq e^{-\Omega_{k}(\log^{\beta} n)}\E \bZ^2_{\ula^\star \ind}.
	\end{equation}
	Indeed, proceeding in a similar fashion as done in \eqref{eq:proof:apriori:1stmo:2}, Lemma \ref{lem:apriori:2ndmo:comparison} shows
	\begin{equation*}
	     \frac{ \E\bZ^2_{\ula^\star,\ind}\big[\exists a\geq \log^{\beta}n, b\leq a+\log^{\beta} n\textnormal{ s.t. }\ell_{a,b}\geq 1\big]}{\E \bZ^2_{\ula^\star,\ind}}\lesssim_{k}\sum_{a=\lceil\log^{\beta} n\rceil}^{7n/2^k}\sum_{b=1}^{a+\lceil\log^{\beta} n\rceil}\sum_{\ell\geq 1}\left(\frac{n}{k\ell a}\left(\frac{Ck}{2^k}\right)^{a}(Ck)^{b}\right)^{\ell}\leq e^{-\Omega_{k}(\log^{\beta} n)}.
	\end{equation*}
	Next, we control the number of multi-cyclic edges: let $\prescript{}{2}{e}_{\textnormal{mult}}$ denote the number of multi-cyclic edges of union-free components as before. Using similar calculations as in the proof of Lemma \ref{lem:apriori:2ndmo:ratio:cycle}, we will show
	\begin{equation}\label{eq:overlap:optimal:2ndmo:3}
	    \E \bZ^{2}_{\ula^\star,\ind}\big[\prescript{}{2}{e}_{\textnormal{mult}}\geq \log^{\beta}n\textnormal{ and }n_{\vvv}=0\textnormal{ if }v(\vvv)\geq\log^{\beta} n\big]\leq e^{-\Omega_{k}(\log^{\beta+1} n)}\E \bZ^{2}_{\ula^\star,\ind}.
	\end{equation}
	Note that in the \textsc{lhs} of the equation above, the union-free component profile is no longer guaranteed to have exponential decay $\pee_{\frac{1}{4}}$ as before. In particular the analog of \eqref{eq:apriori:cycle:final-2} is no longer guaranteed. However, having \eqref{eq:apriori:2ndmo:cyclic:ratio} in hand, we can proceed similarly up to \eqref{eq:apriori:cycle:final-1} in the proof of Lemma \ref{lem:apriori:1stmo:ratio:cycle} and instead of the bound \eqref{eq:apriori:cycle:final-2}, we can bound
	\begin{equation*}
	    \sum_{i=1}^{K}\ell_i a_i^2\leq  \Big(\sum_{i=1}^{K}\ell_i a_i\Big)\max_{1\leq i\leq K} a_i\leq \frac{14n}{2^k}\log^{\beta} n \quad\textnormal{and}\quad\sum_{i=1}^{K}a_i^2\leq \sum_{a=1}^{\lfloor \log^{\beta} n \rfloor }a_i^2\leq \log^{3\beta}n.
	\end{equation*}
	Then, we can plug the equation above into the analog of \eqref{eq:apriori:cycle:final-1} in the second moment to show
	\begin{equation*}
	\begin{split}
	   \frac{\E \bZ^{2}_{\ula^\star,\ind}\big[\prescript{}{2}{e}_{\textnormal{mult}}\geq \log^{\beta}n\textnormal{ and }n_{\vvv}=0\textnormal{ if }v(\vvv)\geq\log^{\beta} n\big]}{\E\bZ^2_{\ula^\star,\ind}}
	   &\lesssim_{k}\sum_{r\geq 0}\sum_{\gamma\geq \log^{\beta}n}\frac{1}{r!}\big(C_k\log^{\beta }n\big)^{r}\Big(\frac{C_k\log^{3\beta}n}{n}\Big)^{\gamma}\\
	   &\leq e^{-\Omega_k(\log^{\beta+1}n)},
	 \end{split}
	\end{equation*}
	which gives \eqref{eq:overlap:optimal:2ndmo:3}.
	
	Moreover, Lemma \ref{lem:apriori:2ndmo:comparison} also shows the following weaker exponential decay of union-free components with stronger error bound:
	\begin{equation}\label{eq:overlap:optimal:2ndmo:4}
	\E \bZ^2_{\ula^\star,\ind}\big[\exists v\leq \log^{\beta} n\textnormal{ s.t. } \sum_{v(\vvv)=v}n_{\vvv}\geq \max(n2^{-kv/2},n^{1/20})\textnormal{ and }\prescript{}{2}{e}_{\textnormal{mult}}\leq \log^{\beta}n\big]=e^{-\Omega_k(n^{1/20})}\E\bZ^2_{\ula^\star,\ind}.
	\end{equation}
	Indeed, note that for a union-free component $\vvv$, the number of multi-cyclic edges in $\vvv$ is at least $f(\vvv)-v(\vvv)$ since the clauses in $\vvv$ have internal degree at least $2$. Thus, $\prescript{}{2}{e}_{\textnormal{mult}}\geq \sum_{\vvv\in \FFF_2}n_{\vvv}\one\{f(\vvv)\geq v(\vvv)+1\}$ holds. Hence, we can bound the \textsc{lhs} of the equation above by
	\begin{equation*}
	\begin{split}
	    	&\E \bZ^2_{\ula^\star,\ind}\big[\exists v\leq \log^{\beta} n\textnormal{ s.t. } \sum_{v(\vvv)=v}n_{\vvv}\geq \max(n2^{-kv/2},n^{1/20})\textnormal{ and }\prescript{}{2}{e}_{\textnormal{mult}}\leq \log^{\beta}n\big]\\
	    	&\leq \E \bZ^2_{\ula^\star,\ind}\big[\exists v\leq \log^{\beta} n\textnormal{ s.t. } \sum_{v(\vvv)=v, f(\vvv)\leq v}n_{\vvv}\geq 0.5\max(n2^{-kv/2},n^{1/20})\big]
	\end{split}
	\end{equation*}
	for large enough $n$. Then, proceeding similarly as done in \eqref{eq:proof:apriori:1stmo:2}, we can use Lemma \ref{lem:apriori:2ndmo:comparison} to further bound the \textsc{rhs} of the equation above by the following: if we let $\ell_{\max}^\prime(a)\equiv \frac{0.5}{a}\max(n2^{-ka/2},n^{1/20})$, then
	\begin{equation*}
	\begin{split}
	    &\frac{\E \bZ^2_{\ula^\star,\ind}\big[\exists v\leq \log^{\beta} n\textnormal{ s.t. } \sum_{v(\vvv)=v, f(\vvv)\leq v}n_{\vvv}\geq 0.5\max(n2^{-kv/2},n^{1/20})\big]}{ \E\bZ^2_{\ula^\star,\ind}}\\
	    &\leq\sum_{a=1}^{\lfloor \log^{\beta}n \rfloor}\sum_{b=1}^{a}\sum_{\ell\geq \ell_{\max}^\prime(a)}\left(\frac{n}{k\ell a}\left(\frac{Ck}{2^k}\right)^{a}(Ck)^{b}\right)^{\ell}\lesssim_{k} \sum_{a=1}^{\lfloor \log^{\beta}n \rfloor}\left(\frac{n}{k\ell_{\max}^\prime(a)a}\left(\frac{C^2 k^2}{2^k}\right)^{a}\right)^{\ell_{\max}^\prime(a)}\leq e^{-\Omega_k(n^{1/20})},
	\end{split}
	\end{equation*}
	which gives \eqref{eq:overlap:optimal:2ndmo:4}.
	
	Next, we show that the number of cyclic union-free components are at most $\log^{\beta} n$ with $e^{-\Omega_k(\log^{\beta}n)}$ error bound: let $\prescript{}{2}{n}_{\cyc}$ denote the number of cyclic union-free components as before. Then, we will show that 
	\begin{equation}\label{eq:overlap:optimal:2ndmo:5}
	\begin{split}
	   &\E\bZ^2_{\ula^\star,\ind}\big[\prescript{}{2}{n}_{\cyc}\geq \log^{\beta} n, n_{\vvv}=0\textnormal{ if }v(\vvv)\geq\log^{\beta}n,\textnormal{ and } \sum_{v(\vvv)=v}n_{\vvv}\leq \max(n2^{-kv/2},n^{1/20})\textnormal{ for }v\leq\log^{\beta}n \big]\\
	    &\leq e^{-\Omega_{k}(\log^{\beta} n)}\E \bZ^2_{\ula^\star \ind}.
	\end{split}
	\end{equation}
	Indeed, proceeding in a similar fashion as in the proof of \eqref{eq:overlap:optimal:2ndmo:3}, instead of the bound \eqref{eq:apriori:cycle:final-2}, we can bound
	\begin{equation*}
	\begin{split}
	    &\sum_{i=1}^{K}\ell_i a_i^2
	    =\sum_{a_i\leq\frac{1.9\log n}{k\log 2}}\ell_i a_i^2+\sum_{\frac{1.9\log n}{k\log 2}<a_i<\log^{\beta} n}\ell_i a_i^2\leq\frac{14n}{2^k}+n\sum_{a\geq 2}a^2 2^{-ka/2}+n^{1/20}\sum_{a=\lceil\frac{1.9\log n}{k\log 2}\rceil}^{ \lfloor \log^{\beta}n\rfloor }a^2\leq \frac{C}{2^k}n;\\
	    &\sum_{i=1}^{K}a_i^2\leq \sum_{a=1}^{\lfloor \log^{\beta} n \rfloor }a_i^2\leq \log^{3\beta}n, 
	\end{split}
	\end{equation*}
	and plug it into the analog of \eqref{eq:apriori:cycle:final-1} in the second moment to show
	\begin{equation*}
	\begin{split}
	    &\frac{\E\bZ^2_{\ula^\star,\ind}\big[\prescript{}{2}{n}_{\cyc}\geq \log^{\beta} n, n_{\vvv}=0\textnormal{ if }v(\vvv)\geq\log^{\beta}n,\textnormal{ and } \sum_{v(\vvv)=v}n_{\vvv}\leq \max(n2^{-kv/2},n^{1/20})\textnormal{ for }v\leq\log^{\beta}n \big]}{\E\bZ^2_{\ula^\star,\ind}}\\
	    &\lesssim_{k}\sum_{r\geq \log^{\beta}n}\sum_{\gamma\geq 0}\frac{1}{r!}\left(\frac{C^\prime k^2}{2^k}\right)^{r}\left(\frac{C^\prime k^2\log^{3\beta}n}{n}\right)^{\gamma}=e^{-\Omega_k(\log^{\beta}n)}.
	\end{split}
	\end{equation*}
	Having \eqref{eq:overlap:optimal:2ndmo:1}-\eqref{eq:overlap:optimal:2ndmo:5} in hand, we can proceed similarly as in the proof of Lemma \ref{cor:B:close:optimal:1stmo} (cf. \eqref{eq:cor:B:close:step1}) to show
	\begin{multline}\label{eq:overlap:optimal:2ndmo:6}
	    \E\bZ^2_{\ula^\star,\ind}\Big[||\bB-\bB^\star_{\ula^\star}||_1\geq n^{-0.45}, \prescript{}{2}{n}_{\cyc}\leq \log^{\beta} n, n_{\vvv}=0\textnormal{ if }v(\vvv)+f(\vvv)\geq 3\log^{\beta}n,\\
	    \textnormal{ and }\forall v\leq\log^{\beta}n, \sum_{v(\vvv)=v}n_{\vvv}\leq \max(n2^{-kv/2},n^{1/20})\Big]=e^{-\Omega_k(n^{0.1})}\E\bZ^2_{\ula^\star,\ind}.
	\end{multline}
	In proving \eqref{eq:overlap:optimal:2ndmo:6}, we need an analog of $\TT(\cdot)$, defined in the proof of Proposition \ref{prop:ratio:uni:1stmo}, for the second moment. Since we are in the independence regime, there are enough $\bb\bb^{=}$ and $\bb\bb^{\neq}$ edges that can be swapped with to decompose cyclic union-free components into single union-free tree with the right boundary colors. The same argument as in the first moment can be applied to show \eqref{eq:overlap:optimal:2ndmo:6}.
	
	Finally, because of Lemma \ref{lem:overlap:pairflipsymm}, we have $\sum_{\uuu\in \FFF_2^{\tr}}p^\star_{\uuu,\ula^\star}\cdot \textnormal{overlap}(\uuu)=0$. Thus, proceeding in a similar fashion as done in \eqref{eq:cor:B:close:technical-1} and \eqref{eq:cor:B:close:technical-2} in the proof of Lemma \ref{cor:B:close:optimal:1stmo}, we can show
	\begin{equation}\label{eq:overlap:optimal:2ndmo:7}
	\begin{split}
	    &\E\bZ^2_{\ula^\star,\ind}\Big[\big|\sum_{\uuu\in \FFF_2^{\tr}}n_{\uuu}\cdot \textnormal{overlap}(\uuu)\big|\geq n^{-0.6}, ||\bB-\bB^\star_{\ula^\star}||_1\leq n^{-0.45},\prescript{}{2}{n}_{\cyc}\leq \log^{\beta} n,\textnormal{ and } n_{\vvv}=0\textnormal{ if }v(\vvv)+f(\vvv)\geq 3\log^{\beta}n \Big]\\
	    &\leq e^{-\Omega_k(n^{0.2})}\E\bZ^2_{\ula^\star,\ind}.
	\end{split}
	\end{equation}
	Therefore, \eqref{eq:overlap:optimal:2ndmo:1}-\eqref{eq:overlap:optimal:2ndmo:7} concludes the proof.
	\end{proof}
	\subsubsection{Estimates on separating constraints}\label{subsubsec:apriori:separating:2ndmo}
	We now aim to prove \eqref{eq:apriori:2ndmo:separating}. Let $\XXX\equiv \{\bb\bb^{=}, \bb\bb^{\neq},\bb\fs,\fs\bb,\fs\fs\}$ and for $\btau \in \YYY\equiv \{\fs\rr,\rr\fs,\bb\rr^{=},\bb\rr^{\neq},\rr\bb^{=},\rr\bb^{\neq}\}$, define $\Omega_{\btau}$ to be the set of $\{x_{\bsigma}\}_{\bsigma \in \XXX} \in \Z_{\geq 0}^{\XXX}$ satisfying the following:
	\begin{itemize}
	    \item $\sum_{\bsigma \in \XXX}x_{\bsigma}=k-1$.
	    \item If the number of $\bsigma$-color in $\bupi=(\bpi_1,...,\bpi_{k-1})$ is $x_{\bsigma}$ for $\bsigma \in \XXX$, then $\hat{v}_2(\bupi,\btau)\neq 0$, i.e. $\bupi$ can neighbor a clause with $\btau$-color, and $(\bupi,\btau)$ is not pair-forcing.
	\end{itemize}
	For example $\Omega_{\rr\bb^{=}}=\{\bx\in \Z_{\geq 0}^{\XXX}: \sum_{\bsigma \in \XXX}x_{\bsigma}=k-1, x_{\fs\bb}=x_{\fs\fs}=0,x_{\bb\bb^{\neq}}\leq k-3, 1\leq x_{\bb\bb^{=}}\leq k-2\}$. Similarly, define $\Omega_{\fs\fs}$ to be the set of $\{x_{\bsigma}\}_{\bsigma \in \XXX} \in \Z_{\geq 0}^{\XXX}$ satisfying the following:
	\begin{itemize}
	\item $\sum_{\bsigma \in \XXX}x_{\bsigma}=k$.
	\item If $\bupi=(\bpi_1,...,\bpi_{k})$ has $x_{\bsigma}$ number of $\bsigma$-color for $\bsigma \in \XXX$, then $\hat{v}_2(\bupi)\neq 0$.
	\end{itemize}
	Moreover, recalling the definition of $f_2(m_{\tns}, \mathbf{g}_{\rr}, \underline{E})$ in \eqref{eq:apriori:2ndmo:B4}, denote by $p_{\btau}, \btau \in \YYY,$ the fraction of $\btau$-colored clauses among $m_{\tns}$ clauses:
	\begin{equation}\label{eq:apriori:2ndmo:def:bp}
	    \bp \equiv (\bp_{\YYY}, p_{\tns\tns}),\quad\textnormal{where}\quad \bp_{\YYY}\equiv (p_{\btau})_{\btau \in \YYY} \equiv \frac{\mathbf{g}_{\rr}}{m_{\tns}}\quad\textnormal{and}\quad p_{\tns\tns} \equiv 1-\sum_{\btau \in \YYY}p_{\btau}
	\end{equation}
	Define $\kappa \equiv \kappa(\bp)\equiv k-\sum_{\btau \in \YYY}p_{\btau}$  and let $\xi_{\bsigma}, \bsigma \in \XXX,$ denotes the fraction of half-edges colored $\bsigma$ among $\kappa m_{\tns}$ half-edges:
	\begin{equation}\label{eq:apriori:2ndmo:def:bxi}
	    \bxi \equiv \left(\xi_{\bsigma}\right)_{\bsigma \in \XXX}\equiv \frac{\underline{E}}{\kappa m_{\tns}}.
	\end{equation}
	With a slight abuse of notation, we let $f_2(m_{\tns},m_{\tns}\bp,\kappa m_{\tns}\bxi)\equiv f_2(m_{\tns},m_{\tns}\bp_{\YYY},\kappa m_{\tns}\bxi)$. Note that $\hat{v}_2(\cdot)$ is determined by the number of $\bsigma$-colored edges for $\bsigma \in \XXX$, which we denote by $\hat{v}_2(\bx)$ for $\bx =\left(x_{\bsigma}\right)_{\bsigma \in \XXX}$. Thus, we can express $f_2(m_{\tns}, m_{\tns}\bp, \kappa m_{\tns} \bxi)$ as
	\begin{equation*}
	    f_2(m_{\tns}, m_{\tns}\bp, \kappa m_{\tns} \bxi)= \E_{\bxi}\left[\prod_{i=1}^{m_{\tns}p_{\tns\tns}}\hat{v}_2(X_i^{\tns\tns})\cdot \prod_{\btau \in \YYY\cup\{\tns\tns\}}\prod_{i=1}^{m_{\tns}p_{\btau}}\one\left\{X_i^{\btau} \in \Omega_{\btau}\right\}\bigg|\sum_{\btau \in \YYY\cup\{\tns\tns\}}\sum_{i=1}^{m_{\tns}p_{\btau}}X^{\btau}_i= \kappa m_{\tns} \bxi\right]\,,
	\end{equation*}
    where $X_i^{\btau}\sim \textnormal{Multinomial}(k-\one\{\btau \in \YYY\}, \bxi)$ for $\btau\in \YYY\cup\{\tns\tns\}$ and $1\leq i \leq m_{\tns}p_{\btau}$ are independent multinomial random variables, and $\E_{\bxi}$ denotes the expectation with respect to $(X_i^{\btau})_{\btau\in \YYY\cup\{\tns\tns\}, i\leq m_{\tns}p_{\btau}}$. We denote their marginal probabilities by
    \begin{equation}\label{eq:def:multinomial:prob}
	    p_{\bxi}^{\ell}(\bx)\equiv \P(X= \bx),\quad X\sim \textnormal{Multinomial}(\ell,\bxi)\,.
	\end{equation}
    Analogously to \eqref{eq:apriori:1stmo:f}, we aim to approximate $f_2(m_{\tns}, m_{\tns}\bp, \kappa m_{\tns} \bxi)$ by introducing a rescaling factor $\bgamma\equiv (\gamma_{\bsigma})_{\bsigma \in \XXX} \in \R^{\XXX}$ and using local CLT. Proceeding in the same fashion as in \eqref{eq:apriori:1stmo:f:prior}, we have 
    \begin{equation}\label{eq:apriori:2ndmo:f:prior}
    f_2(m_{\tns}, m_{\tns}\bp, \kappa m_{\tns} \bxi)=\frac{\P_{\bgamma, \bxi}\Big(\sum_{\btau \in \YYY\cup\{\tns\tns\}}\sum_{i=1}^{m_{\tns}p_{\btau}}\wt{X}^{\btau}_i= \kappa m_{\tns} \bxi\Big)}{\P_{\bxi}\Big(\sum_{\btau \in \YYY\cup\{\tns\tns\}}\sum_{i=1}^{m_{\tns}p_{\btau}}X^{\btau}_i= \kappa m_{\tns} \bxi\Big)}\exp\Big\{-m_{\tns}\Big(\kappa\langle \bgamma, \bxi \rangle - \Lambda_{\bxi}(\bp,\bgamma)\Big)\Big\}\,,
    \end{equation}
    where $\Lambda_{\bxi}(\bp,\bgamma)$ is defined by \begin{equation}
	\Lambda_{\bxi}(\bp,\bgamma) \equiv \sum_{\btau \in \YYY}p_{\btau}\log\left(\sum_{\bx \in \Omega_{\btau}}p_{\bxi}^{k-1}(\bx)e^{\langle \bgamma, \bx \rangle}\right)+p_{\tns\tns}\log\left(\sum_{\bx \in \Omega_{\tns\tns}}\hat{v}_2(\bx)p_{\bxi}^{k}(\bx)e^{\langle \bgamma, \bx \rangle}\right)\,,
    \end{equation}
    and $\P_{\bgamma,\bxi}$ denotes the probability with respect to independent random variables $(\wt{X}_i^{\btau})_{\btau\in \YYY\cup\{\tns\tns\}, i\leq m_{\tns}p_{\btau}}$ with $\wt{X}_i^{\btau}\sim \nu_{\btau}$. Here $\nu_{\btau}\equiv \nu_{\btau,\bgamma,\bxi} \in \PPP(\Omega_{\btau})$ is defined by  
    \begin{equation}\label{eq:apriori:2ndmo:def:bnu:star}
	      \nu_{\btau}(\bx) \equiv \frac{p_{\bxi}^{k-1}(\bx)e^{\langle \bgamma, \bx \rangle}\one\{\bx \in \Omega_{\btau}\}}{\sum_{\bx^\prime\in \Omega_{\btau}}p_{\bxi}^{k-1}(\bx^\prime)e^{\langle \bgamma, \bx^\prime \rangle}},~~\btau \in \YYY,\quad\textnormal{and}\quad \nu_{\tns\tns}(\bx) \equiv \frac{\hat{v}_2(\bx)p_{\bxi}^{k}(\bx)e^{\langle \bgamma, \bx \rangle}\one\{\bx\in\Omega_{\tns\tns}\}}{\sum_{\bx^\prime\in \Omega_{\btau}}\hat{v}_2(\bx^\prime)p_{\bxi}^{k}(\bx^\prime)e^{\langle \bgamma, \bx^\prime \rangle}}.
	  \end{equation}
	  To use local CLT, we take $\bgamma=\bgamma(\bp,\bxi)$ such that
   \begin{equation}\label{eq:apriori:2ndmo:partial:Lambda:xi}\nabla_{\bgamma}\Lambda_{\xi}\left(\bp,\bgamma(\bp,\bxi)\right)=\kappa\bxi
	  \end{equation}
	  whose existence is guaranteed by Lemma \ref{lem:apriori:2ndmo:gamma:exists} below. Before proceeding, we make the following observations.
	  \begin{itemize}
	  \item If $\xi_{\bsigma}=0$ for $\bsigma \in \XXX$, then $\partial_{\gamma_{\bsigma}}\Lambda_{\xi}(\bp,\bgamma^\prime)=0$ for every $\bgamma^\prime$ and $\nabla_{\bgamma} \Lambda(\bp,\bgamma^\prime)$ does not depend on $\gamma^\prime_{\bsigma}$.
	  \item If $\partial_{\gamma_{\bsigma}}\Lambda_{\xi}(\bp,\bgamma^\prime)=\kappa\bxi_{\bsigma}$ for $\bsigma \in \XXX -\{\bb\bb^{=}\}$, then $\partial_{\gamma_{\bb\bb^{=}}}\Lambda_{\xi}(\bp,\bgamma^\prime)=\kappa\bxi_{\bb\bb^{=}}$.
	  \item If $\bar{\bgamma}= (\bar{\bgamma})_{\bsigma\in\XXX}$ satisfies $\sum_{\bsigma \in \XXX} \bar{\bgamma}_{\bsigma}=0$, then $\nabla_{\bgamma}\Lambda_{\bxi}(\bp,\bgamma^\prime)=\nabla_{\bgamma}\Lambda_{\bxi}(\bp,\bgamma^\prime+\bar{\bgamma})$
	  \end{itemize}
	  Thus, in showing \eqref{eq:apriori:2ndmo:partial:Lambda:xi}, it suffices to show that there exists $\bgamma(\bp,\bxi)=\left(\gamma_{\bsigma}(\bp,\bxi)\right)_{\bsigma \in \XXX}$ such that $\gamma_{\bsigma}(\bp,\bxi)=0$ for $\xi_{\bsigma}=0$, $\gamma_{\bb\bb^{=}}(\bp,\bxi)=0$, and
	  \begin{equation}\label{eq:apriori:2ndmo:xi:prime}
	      \partial_{\gamma_{\bsigma}}\Lambda_{\xi}\left(\bp,\bgamma(\bp,\bxi)\right)=\kappa\bxi_{\bsigma},\quad\textnormal{for}\quad \bsigma \in \XXX^\prime \equiv \XXX^\prime_{\bxi}\equiv \{\bsigma\in \XXX: \xi_{\bsigma} \neq 0\}\backslash \{\bb\bb^{=}\}.
	  \end{equation}
	 Taking $\bgamma=\bgamma(\bp,\bxi)$ which satisfies \eqref{eq:apriori:2ndmo:partial:Lambda:xi} and plugging it into \eqref{eq:apriori:2ndmo:f:prior}, we have by local CLT \cite{Borokov17} that 
     \begin{equation*}
         f_2(m_{\tns}, m_{\tns}\bp, \kappa m_{\tns} \bxi)=\left(\frac{\det\Big(\sum_{\btau\in \YYY\cup\{\tns\tns\}}p_{\btau}\Cov_{\bxi}\big(X^{\btau}\big)\Big)}{\det\Big(\sum_{\btau\in \YYY\cup\{\tns\tns\}}p_{\btau}\Cov_{\bgamma(\bp,\bxi),\bxi}\big(\wt{X}^{\btau}\big)\Big)}\right)^{1/2}\exp\Big\{-m_{\tns}\Big(\kappa\big\langle \bgamma(\bp,\bxi), \bxi \big\rangle - \Lambda_{\bxi}\big(\bp,\bgamma(\bp,\bxi)\big)\Big)\Big\}\,.
     \end{equation*}
     Here, with abuse of notation, we took the convention that $\Cov_{\bxi}\big(X^{\btau}_1\big)$ and $\Cov_{\bgamma(\bp,\bxi),\bxi}\big(\wt{X}^{\btau}_1\big)$ are $|\XXX^\prime| \times |\XXX^\prime|$ matrices that encodes the covariance of $(X^{\btau}_{\bsigma})_{\bsigma\in \XXX^\prime}$ and $(\wt{X}^{\btau}_{\bsigma})_{\bsigma\in \XXX^\prime}$ respectively. That is, we delete the coordinate $\bsigma= \bb\bb^{=}$ and $\bsigma$ such that $\xi_{\bsigma}=0$. With such convention, Lemma \ref{lem:apriori:2ndmo:cov:comparable} below shows that $\det\big(\sum_{\btau\in \YYY\cup\{\tns\tns\}}p_{\btau}\Cov_{\bxi}\big(X^{\btau}\big)\big)\asymp_{k} \det\big(\sum_{\btau\in \YYY\cup\{\tns\tns\}}p_{\btau}\Cov_{\bgamma(\bp,\bxi),\bxi}\big(\wt{X}^{\btau}\big)\big)$ holds uniformly over $\bxi$ and $\bp$ such that $\xi_{\bb\bb^{=}}, \xi_{\bb\bb^{\neq}} \in [\frac{1}{2}-\frac{3k^2}{2^{k/2}}, \frac{1}{2}+\frac{3k^2}{2^{k/2}}]$ and $\sum_{\btau \in \YYY} p_{\btau} \leq \frac{15k}{2^k}$. Therefore, we have that 
     \begin{equation}\label{eq:apriori:2ndmo:f:final}
          f_2(m_{\tns}, m_{\tns}\bp, \kappa m_{\tns} \bxi)\asymp_k \exp\Big\{-m_{\tns}\Big(\kappa\big\langle \bgamma(\bp,\bxi), \bxi \big\rangle - \Lambda_{\bxi}\big(\bp,\bgamma(\bp,\bxi)\big)\Big)\Big\}.
     \end{equation}
     \begin{lemma}\label{lem:apriori:2ndmo:gamma:exists}
	  In the regime where $\xi_{\bb\bb^{=}}, \xi_{\bb\bb^{\neq}} \in [\frac{1}{2}-\frac{3k^2}{2^{k/2}}, \frac{1}{2}+\frac{3k^2}{2^{k/2}}]$ and $\sum_{\btau \in \YYY} p_{\btau} \leq \frac{15k}{2^k}$, there exists a unique $\bgamma= \bgamma(\bp, \bxi)=\left(\gamma_{\bsigma}(\bp,\bxi)\right)_{\bsigma \in \XXX}$ such that it satisfies the following.
	  \begin{enumerate}
	      \item $\gamma_{\bsigma}(\bp,\bxi)=0$ if $\xi_{\bsigma}=0$. Also, $\gamma_{\bb\bb^{=}}(\bp,\bxi)=0$.
	      \item $\partial_{\gamma_{\bsigma}}\Lambda_{\bxi}\left(\bp,\bgamma(\bp,\bxi)\right) = \kappa \xi_{\bsigma}$ for $\bsigma\in \XXX^{\prime}$, where $\XXX^\prime$ is defined in \eqref{eq:apriori:2ndmo:xi:prime}.
	      \item $||\bgamma(\bp,\bxi)||_1\lesssim \frac{k^3}{2^{k/2}}$ and $||\nabla_{\bp}\Lambda_{\xi}\left(\bp,\bgamma(\bp,\bxi)\right)||_1\lesssim \frac{k^4}{2^{k/2}}$.
	  \end{enumerate}
	  \end{lemma}
	 Before going into the proof of Lemma \ref{lem:apriori:2ndmo:gamma:exists}, we introduce the necessary notations. For a vector $\bx \in \R^{\XXX}$, denote $\bx^{-}\equiv (\bx_{\bsigma})_{\bsigma \in \XXX^\prime}$. Denote $\Lambda^{-}_{\bxi}(\bp,\bgamma^{-})\equiv \Lambda_{\bxi}(\bp,\bgamma)$, where $\bgamma$ is obtained from $\bgamma^{-}$ by concatenating $0$ for coordinates $\bsigma \in \XXX\backslash\XXX^\prime$. Then, $\Lambda^{-}_{\bxi}(\bp,\bgamma^{-})=\sum_{\btau \in \YYY\cup \{\fs\fs\}}p_{\btau}\Lambda_{\bxi}^{\btau}(\bgamma^{-})$, where 
	  \begin{equation*}
	      \Lambda^{\btau}_{\xi}(\bgamma^{-})\equiv \log\left(\sum_{\bx \in \Omega_{\btau}}p_{\bxi}^{k-1}(\bx)e^{\langle \bgamma^{-}, \bx^{-} \rangle}\right), \btau \in \YYY,\quad\textnormal{and}\quad \Lambda^{\tns\tns}_{\xi}(\bgamma^{-})\equiv\log\left(\sum_{\bx \in \Omega_{\tns\tns}}\hat{v}_2(\bx)p_{\bxi}^{k}(\bx)e^{\langle \bgamma^{-}, \bx^{-} \rangle}\right).
	  \end{equation*}
	We note that $\partial_{\fs\fs}\Lambda_{\xi}^{\btau}(\bgamma^{-})=0$ for $\btau \in \YYY$ since $x\in \Omega_{\btau},\btau \in \YYY$ implies that $x_{\fs\fs}=0$. Define
	\begin{equation}\label{eq:2ndmo:def:bmu}
	    \bmu\equiv \bmu(\bgamma^{-})\equiv \left(\frac{\xi_{\sigma}e^{\gamma_{\bsigma}}}{\sum_{\bsigma^\prime\in \XXX} \xi_{\sigma^\prime}e^{\gamma_{\bsigma^\prime}}}\right)_{\bsigma \in \XXX},\quad\textnormal{where}\quad \gamma_{\bsigma}\equiv 0\quad\textnormal{for}\quad\bsigma \in \XXX\backslash \XXX^\prime
	\end{equation}
	Then, we have the following identity for $\bsigma \in \XXX^\prime$:
	  \begin{equation}\label{eq:2ndmo:compute:partial:Lambda}
	      \partial_{\gamma_{\bsigma}}\Lambda^{\btau}_{\bxi}(\bgamma^{-})=\E_{\bmu}^{k-1}[X_{\bsigma}\mid X\in \Omega_{\btau}],\tau \in \YYY,\quad\textnormal{and}\quad \partial_{\gamma_{\bsigma}}\Lambda^{\tns\tns}_{\bxi}(\bgamma^{-})=\frac{\E_{\bmu}^{k}[X_{\bsigma}\hat{v}_2(X)\one\{X\in\Omega_{\tns\tns}\}]}{\E_{\bmu}^{k}[\hat{v}_2(X)\one\{X\in \Omega_{\tns\tns}\}]},
	  \end{equation}
	  where $\E^{\ell}_{\bmu}$ is the expectation with respect to $p^{\ell}_{\bmu}$, defined in \eqref{eq:def:multinomial:prob}. To this end, we define
	  \begin{equation}\label{eq:2ndmo:define:partial:Lambda:bb}
	      \partial_{\gamma_{\bb\bb^{=}}}\Lambda^{\btau}_{\bxi}(\bgamma^{-})\equiv \E_{\bmu}^{k-1}[X_{\bb\bb^{=}}\mid X\in \Omega_{\btau}],\tau \in \YYY,\quad\textnormal{and}\quad \partial_{\gamma_{\bb\bb^{=}}}\Lambda^{\tns\tns}_{\bxi}(\bgamma^{-})\equiv\frac{\E_{\bmu}^{k}[X_{\bb\bb^{=}}\hat{v}_2(X)\one\{X\in\Omega_{\btau}\}]}{\E_{\bmu}^{k}[\hat{v}_2(X)\one\{X\in \Omega_{\btau}\}]}.
	  \end{equation}
	  Analogously, let $\partial_{\gamma_{\bb\bb^{-}}}\Lambda^{-}_{\bxi}(\bgamma^{-})\equiv \sum_{\btau \in \YYY\cup\{\tns\tns\}}\partial_{\gamma_{\bb\bb^{-}}}\Lambda^{\btau}_{\bxi}(\bgamma^{-})$. The following lemma will be crucial for the proof of Lemma \ref{lem:apriori:2ndmo:gamma:exists}.
	  \begin{lemma}\label{lem:partial:Lambda:upper:bound}
	  $\partial_{\gamma_{\tns\tns}}\Lambda^{\tns\tns }_{\bxi}(\bgamma^{-})\leq 8k\mu_{\tns\tns}$ and $\partial_{\gamma_{\bsigma}}\Lambda^{\tns\tns }_{\bxi}(\bgamma^{-})\leq 8(k\mu_{\bsigma}+2)$ for $\bsigma \in \XXX\backslash\{\tns\tns\}$.
	\end{lemma}
	\begin{proof}  
	We first make the following observations.
	\begin{itemize}
	    \item We have the lower bound $\hat{v}_2(\bx)\geq \frac{1}{4}$ for $\bx \in \Omega_{\tns\tns}$.
	    \item $\Omega_{\tns\tns} = \Omega^{\dagger}_{\tns\tns}\backslash \{x_{\fs\fs}=k-2, x_{\fs\bb}=x_{\bb\fs}=0,x_{\bb\bb^{=}}=x_{\bb\bb^{\neq}}=1\}$, where $\bar{\Omega}_{\tns\tns}$ is defined by
	    \begin{equation*}
	    \Omega_{\tns\tns}^{\dagger}\equiv \{\bx\in \Z_{\geq 0}^{\XXX}:\sum_{\bsigma \in \XXX}x_{\bsigma}=k, x_{\fs\bb}+x_{\bb\bb^{=}}+x_{\bb\bb^{\neq}}\geq 2,\quad\textnormal{and}\quad x_{\bb\fs}+x_{\bb\bb^{=}}+x_{\bb\bb^{\neq}}\geq 2\}
	    \end{equation*}
	\end{itemize}
	With the above observations in mind, we can upper bound $\partial_{\gamma_{\bsigma}}\Lambda^{\tns\tns}_{\bxi}(\bgamma^{-})$ by
	\begin{equation*}
	    \partial_{\gamma_{\bsigma}}\Lambda^{\tns\tns}_{\bxi}(\bgamma^{-})=\frac{\E_{\bmu}^{k}[X_{\bsigma}\hat{v}_2(X)\one\{X\in\Omega_{\tns\tns}\}]}{\E_{\bmu}^{k}[\hat{v}_2(X)\one\{X\in \Omega_{\tns\tns}\}]}\leq \frac{4 \E^{k}_{\bmu}[X_{\bsigma}\one\{X \in \Omega_{\tns\tns}^{\dagger}\}]}{\P_{\bmu}^{k}(X \in \Omega_{\tns\tns}^{\dagger})-\P_{\bmu}^{k}(X_{\fs\fs}=k-2,X_{\bb\bb^{=}}=X_{\bb\bb^{\neq}}=1)}.
	\end{equation*}
	Note that $\frac{\P_{\bmu}^{k}(X_{\fs\fs}=k-2,X_{\bb\bb^{=}}=X_{\bb\bb^{\neq}}=1)}{\P_{\bmu}^{k}(X \in \Omega_{\tns\tns}^{\dagger})}\leq \frac{\P_{\bmu}^{k}(X_{\fs\fs}=k-2,X_{\bb\bb^{=}}=X_{\bb\bb^{\neq}}=1)}{\P_{\bmu}^{k}(X_{\fs\fs}=k-2,X_{\bb\bb}=2)}=\frac{2\mu_{\bb\bb^{=}}\mu_{\bb\bb^{\neq}}}{(\mu_{\bb\bb^{=}}+\mu_{\bb\bb^{\neq}})^{2}}\leq \frac{1}{2}$, where we abbreviated $X_{\bb\bb}=X_{\bb\bb^{=}}+X_{\bb\bb^{\neq}}$, so we can further bound the \textsc{rhs} of the above equation by
	\begin{equation}\label{eq:apriori:2ndmo:lem:crucial}
	     \partial_{\gamma_{\bsigma}}\Lambda^{\tns\tns}_{\bxi}(\bgamma^{-})\leq 8 \E_{\bmu}^{k}[X_{\bsigma}\mid X\in \Omega_{\tns\tns}^{\dagger}].
	\end{equation}
	To this end, we aim to upper bound $\E_{\bmu}^{k}[X_{\bsigma}\mid X\in \Omega_{\tns\tns}^{\dagger}]$. The bound for $\bsigma= \fs\fs$ is the easiest:
	\begin{equation*}
	    \E_{\bmu}^{k}[X_{\fs\fs}\mid X\in \Omega_{\tns\tns}^{\dagger}]=k \mu_{\fs\fs}\frac{1-\P_{\mu}^{k-1}(X_{\fs\bb}+X_{\bb\bb}\leq 1 \quad\textnormal{or}\quad X_{\bb\fs}+X_{\bb\bb}\leq 1)}{1-\P_{\mu}^{k}(X_{\fs\bb}+X_{\bb\bb}\leq 1 \quad\textnormal{or}\quad X_{\bb\fs}+X_{\bb\bb}\leq 1)}\leq k \mu_{\fs\fs},
	\end{equation*}
	which finishes the proof of the claim for $\bsigma=\fs\fs$. Turning to $\bsigma \in \XXX\backslash\{\fs\fs\}$, consider $X \sim \textnormal{Multinomial}(k,\bmu)$ as the $k$th step of the random walk
	\begin{equation*}
	    X_{t} \equiv (X_{t,\bsigma})_{\bsigma \in \XXX} \equiv\left(\sum_{i=1}^{t}\one\{Z_{j}=\bsigma\}\right)_{\bsigma\in \XXX},
	\end{equation*}
	where $(Z_j)_{1\leq j \leq k}$ are independent $\XXX$-valued random variables with law $\P(Z_j=\bsigma)=\mu_{\bsigma}, \bsigma \in \XXX$. Define the stopping time $\tau \equiv \inf\{t\geq 0: X_{t} \in \Omega_{\tns\tns}^{\dagger}\}$, so $\{\tau \leq k\}=\{X \in \Omega_{\tns\tns}^{\dagger}\}$. Note that $X_{\tau,\bb\bb^{=}}, X_{\tau, \bb\bb^{\neq}}\leq 2$, thus applying the Markov property shows the inequality below for $\bsigma \in \{\bb\bb\}$:
	\begin{equation*}
	    \E_{\bmu}^{k}[X_{\bsigma}\mid X\in \Omega_{\tns\tns}^{\dagger}] \leq 2+\E_{\bmu}[X_{k,\bsigma}-X_{\tau,\bsigma}\mid \tau \leq k]\leq 2+\E_{\bmu}[X_{k-\tau, \bsigma}]\leq 2+k\mu_{\bsigma}.
	\end{equation*}
	Finally, we prove the claim for $\bsigma \in \{\bb\fs, \fs\bb\}$: write $\ast$ for $\{\bb, \fs\}$, e.g. $x_{\ast\bb}=x_{\fs\bb}+x_{\bb\bb}$, and define the stopping time $\tau_{\ast\bb}\equiv \inf\{t\geq 0: X_{t,\ast\bb}\geq 2\}$ and symmetrically $\tau_{\bb\ast}$.  Then, $\tau = \tau_{\ast\bb} \vee \tau_{\bb\ast}$, so we have
	\begin{multline}\label{eq:apriori:2ndmo:lem:technical-1}
	    \E^{k}_{\bmu}[X_{\bb\fs}\mid X\in \Omega_{\tns\tns}^{\dagger}]=\frac{\E_{\bmu}[X_{\bb\fs}\one\{\tau \leq k\}]}{\P_{\bmu}(\tau \leq k)}\leq 2+\frac{\E_{\bmu}[(X_{\bb\fs}-X_{\tau_{\bb\ast},\bb\fs})\one\{\tau \leq k\}]}{\P_{\bmu}(\tau \leq k)}\\
	    =2+\sum_{\ell\leq k }\sum_{\bx}\frac{\P_{\bmu}(\tau_{\bb\ast}=\ell, X_{\ell}=x, \tilde{X}_{k-\ell,\ast\bb}\geq 2-x_{\ast\bb})}{\P_{\bmu}(\tau \leq k)}\E_{\bmu}[\tilde{X}_{k-\ell,\bb\fs}\mid \tilde{X}_{k-\ell,\ast\bb}\geq 2-x_{\ast\bb}],
	\end{multline}
	where $(\tilde{X}_{t})_{t\geq 0}$ is an independent realization of the random walk $(X_t)_{t\geq 0}$. Now observe that for any $t\geq t^\prime \geq 0$, $\E_{\bmu}[\tilde{X}_{t,\bb\fs}\mid \tilde{X}_{t,\ast\bb}=t^\prime]=\frac{(t-t^\prime)\mu_{\bb\fs}}{1-\mu_{\ast\bb}}$ is decreasing in $t^\prime$, so we have
	\begin{equation}\label{eq:multionimal:conditional:decreasing}
	\E_{\bmu}[\tilde{X}_{t, \bb\fs}\mid \tilde{X}_{t,\ast\bb}\geq l]=\frac{\sum_{t^\prime \geq l}\E_{\bmu}[\tilde{X}_{t,\bb\fs}\one\{\tilde{X}_{t,\ast\bb}=t^\prime\}]}{\sum_{t^\prime \geq l}\P_{\bmu}(\tilde{X}_{t,\ast\bb}=t^\prime)}\leq \frac{\sum_{t^\prime\geq 0}\E_{\bmu}[\tilde{X}_{t,\bb\fs}\one\{\tilde{X}_{t,\ast\bb}=t^\prime\}]}{\sum_{t^\prime \geq 0}\P_{\bmu}(\tilde{X}_{t,\ast\bb}=t^\prime)}=t\mu_{\bb\fs}.
	\end{equation}
	Therefore, applying \eqref{eq:multionimal:conditional:decreasing} to the \textsc{rhs} of \eqref{eq:apriori:2ndmo:lem:technical-1} shows
	\begin{equation*}
	     \E^{k}_{\bmu}[X_{\bb\fs}\mid X\in \Omega_{\tns\tns}^{\dagger}]\leq 2+\max_{\ell \leq k, l\leq 2}\E_{\bmu}[\tilde{X}_{k-\ell, \bb\fs}\mid \tilde{X}_{k-\ell, \ast\bb}\geq l]\leq 2+k \mu_{\bb\fs}.
	\end{equation*}
	Symmetrically $\E^{k}_{\bmu}[X_{\fs\bb}\mid X\in \Omega_{\tns\tns}^{\dagger}]\leq 2+k \mu_{\fs\bb}$ holds, which combined with \eqref{eq:apriori:2ndmo:lem:crucial} concludes the proof of the claim.
	\end{proof}
	
	\begin{proof}[Proof of Lemma \ref{lem:apriori:2ndmo:gamma:exists}]
	We only provide the proof for the case where $\XXX^\prime = \{\bb\bb^{\neq}, \fs\bb,\bb\fs, \fs\fs\}$, i.e. when $\xi_{\bb\fs}, \xi_{\fs\bb}, \xi_{\fs\fs} >0$, since other cases follow by simpler argument. Uniqueness of $\bgamma(\bp,\bxi)$ follows from the strict convexity of $\bgamma^{-} \to \Lambda^{-}_{\bxi}(\bp, \bgamma^{-})$. We now aim to show existence.  
	Take $\eps>0$ small and consider the function $\bgamma^{-} \to \Lambda^{-}_{\bxi}(\bp,\bgamma^{-})+\frac{1}{2}\eps||\bgamma^{-}||_2^{2}$. Since the norm of the derivative tends to $\infty$ as $||\bgamma^{-}||_2\to \infty$, by Rockafellar's theorem, there exists a unique $\bgamma^{-}_{\eps}\equiv (\gamma_{\eps,\bsigma})_{\bsigma \in \XXX^\prime}$ such that 
	\begin{equation}\label{eq:apriori:2ndmo:lem:gamma:eps}
	\partial_{\bgamma_{\bsigma}}\Lambda^{-}_{\bxi}(\bp,\bgamma^{-}_{\eps})+\eps \gamma_{\eps,\bsigma} = \kappa \xi_{\bsigma},\quad\bsigma \in \XXX^\prime.
	\end{equation}
	With the estimates shown in Lemma \ref{lem:partial:Lambda:upper:bound}, we show that $\bgamma^{-}_\eps$ stays in a bounded region as $\eps \to 0$. Let $\bmu_{\eps}\equiv(\bmu_{\eps,\bsigma})_{\bsigma\in \XXX}\equiv \bmu(\bgamma_{\eps}^{-})$. We first claim that $\limsup_{\eps \to 0}\gamma_{\eps,\bsigma}<\infty$ for every $\bsigma \in \XXX^\prime$. Suppose by contradiction that $A\equiv \{\bsigma \in \XXX^\prime:\limsup_{\eps \to 0}\gamma_{\eps,\bsigma}=\infty\} \neq \emptyset$. Observe that \eqref{eq:2ndmo:compute:partial:Lambda} and \eqref{eq:2ndmo:define:partial:Lambda:bb} show
	\begin{equation*}
	    \sum_{\bsigma \in \XXX} \partial_{\gamma_{\bsigma}}\Lambda_{\bxi}^{-}(\bp,\bgamma^{-}_{\eps})= \kappa = \kappa\sum_{\bsigma \in \XXX}\xi_{\bsigma}.
	\end{equation*}
	Thus, we can sum \eqref{eq:apriori:2ndmo:lem:gamma:eps} over $\bsigma \in A$ and use Lemma \ref{lem:partial:Lambda:upper:bound} to bound
	\begin{equation*}
	\begin{split}
	    \sum_{\bsigma \in A}\gamma_{\eps,\bsigma}&=\eps^{-1}(\sum_{\bsigma \in \XXX\backslash A}\partial_{\gamma_{\bsigma}}\Lambda_{\bxi}^{-}(\bp,\bgamma^{-}_{\eps})-\sum_{\bsigma \in \XXX\backslash A}\kappa \xi_{\bsigma})\\
	    &\leq \eps^{-1}\bigg(|\XXX\backslash A|\Big(16+(k-1)\sum_{\btau \in \YYY}p_{\btau}\Big)+8p_{\tns\tns}\sum_{\bsigma \in \XXX\backslash A}\mu_{\eps,\bsigma}-\sum_{\bsigma \in \XXX\backslash A}\kappa \xi_{\bsigma}\bigg).
	\end{split}
	\end{equation*}
	By our assumption, $A \neq \emptyset$, $\lim_{\eps \to 0}\mu_{\eps,\bsigma}=0$ for $\bsigma \in \XXX\backslash A$ by the definition of $\bmu$ in \eqref{eq:2ndmo:def:bmu}. Moreover $\bb\bb^{=} \in \XXX \backslash A$, because $A \subset \XXX^\prime$. Thus the \textsc{rhs} of the above equation tends to $-\infty$ as $\eps \to 0$, since $\kappa\xi_{\bb\bb^{=}}=\frac{k}{2}-O(\frac{k^2}{2^{k/2}})\gg 64$. Hence, the above equation shows $\lim_{\eps \to 0}  \sum_{\bsigma \in A}\gamma_{\eps,\bsigma}=0$, which contradicts the definition of $A$. Therefore, $\limsup_{\eps \to 0} \gamma_{\eps,\bsigma}<\infty$ for every $\bsigma \in \XXX^\prime$.
	
	Turning to the lower bound, plug in $\bsigma=\bb\bb^{\neq}$ in \eqref{eq:apriori:2ndmo:lem:gamma:eps} and use Lemma \ref{lem:partial:Lambda:upper:bound} to have
	\begin{equation*}
	   \gamma_{\eps,\bb\bb^{\neq}}=\eps^{-1}\left(\kappa\xi_{\bb\bb^{\neq}}-\partial_{\bgamma_{\bb\bb^{\neq}}}\Lambda^{-}_{\bxi}(\bp,\bgamma^{-}_{\eps})\right)\geq \eps^{-1}\left(\kappa \xi_{\bb\bb^{\neq}}-(k-1)\sum_{\btau \in \YYY}p_{\btau}-8p_{\tns\tns}(k\mu_{\eps,\bb\bb^{\neq}}+2)\right),
	\end{equation*}
	which shows $\liminf_{\eps \to 0}\gamma_{\eps,\bb\bb^{\neq}}>-\infty$, since $p_{\tns\tns}=1-\sum_{\tau\in \YYY}=1-O(\frac{k}{2^k})$. Furthermore, since $\limsup_{\eps \to 0} \gamma_{\eps,\bb\bb^{\neq}}<\infty$, the equation above shows
	\begin{equation}\label{eq:mu:eps:BB:equal:lower}
	    \liminf_{\eps\to 0}\mu_{\eps,\bb\bb^{\neq}} \geq \frac{1}{20}
	\end{equation}
	Similarly, plugging in $\bsigma=\fs\fs$ in \eqref{eq:apriori:2ndmo:lem:gamma:eps} and using Lemma \ref{lem:partial:Lambda:upper:bound} shows
	\begin{equation*}
	    \gamma_{\eps, \fs\fs}= \eps^{-1}\left(\kappa\xi_{\fs\fs}-\partial_{\bgamma_{\fs\fs}}\Lambda^{-}_{\bxi}(\bp, \bgamma^{-}_{\eps})\right)=\eps^{-1}\left(\kappa\xi_{\fs\fs}-\partial_{\bgamma_{\fs\fs}}\Lambda_{\bxi}^{\tns\tns}(\bgamma^{-}_{\eps})\right)\geq \eps^{-1}(\kappa \xi_{\fs\fs}-8kp_{\tns\tns} \mu_{\eps,\fs\fs}),
	\end{equation*}
	which shows $\liminf_{\eps \to 0}\gamma_{\eps,\fs\fs}>-\infty$, since we have assumed $\XXX^\prime = \{\bb\bb^{\neq},\bb\fs,\fs\bb,\fs\fs\}$, i.e. $\xi_{\fs\fs} \neq 0$. Turning to the case $\bsigma=\bb\fs$, note that $\partial_{\gamma_{\bb\fs}}\Lambda^{\btau}_{\bxi}(\bgamma^{-}_{\eps})=0$, if $\btau \in \{\bb\rr^{=},\bb\rr^{\neq},\fs\rr\}$, so we have
	 \begin{equation}\label{eq:2ndmo:lem:gamma:eps:BS}
	 \begin{split}
	  \gamma_{\eps,\bb\fs}
	  &=\eps^{-1}\left(\kappa\xi_{\bb\fs}-\sum_{\btau\in \{\rr\bb^{=},\rr\bb^{\neq},\rr\fs\}}p_{\btau}\E^{k-1}_{\bmu_{\eps}}[X_{\bb\fs}\mid X\in \Omega_{\btau}]- p_{\tns\tns}\frac{\E_{\bmu_{\eps}}^{k}[X_{\bb\fs}\hat{v}_2(X)\one\{X\in\Omega_{\tns\tns}\}]}{\E_{\bmu_{\eps}}^{k}[\hat{v}_2(X)\one\{X\in \Omega_{\tns\tns}\}]}\right)\\
	  &\geq \eps^{-1}\left(\kappa\xi_{\bb\fs}-\sum_{\btau\in \{\rr\bb^{=},\rr\bb^{\neq},\rr\fs\}}\frac{p_{\btau}k\mu_{\eps,\bb\fs}}{\P_{\bmu_{\eps}}(X\in \Omega_{\btau})}- \frac{p_{\tns\tns}4k\mu_{\eps,\bb\fs}}{\P_{\bmu_{\eps}}(X\in \Omega_{\tns\tns})}\right)\\
	  &\geq \eps^{-1}\left(\kappa\xi_{\bb\fs}-\sum_{\btau\in \{\rr\bb^{=},\rr\bb^{\neq},\rr\fs\}}\frac{p_{\btau}k\mu_{\eps,\bb\fs}}{\mu_{\eps,\bb\bb^{=}}^{2}\mu_{\eps,\bb\bb^{\neq}}^{k-3}}- \frac{p_{\tns\tns}4k\mu_{\eps,\bb\fs}}{\mu_{\eps,\bb\bb^{=}}^{2}\mu_{\eps,\bb\bb^{\neq}}^{k-2}}\right),
	 \end{split}
	 \end{equation}
	 where the first inequality is due to the bound $\hat{v}_2(\bx)\geq \frac{1}{4}$ for $\bx \in \Omega_{\tns\tns}$, and the second inequality is due to the fact $\{x_{\bb\bb^{=}}=2,x_{\bb\bb^{\neq}}=k-3\}\in \Omega_{\btau},\btau\in\{\rr\bb^{=},\rr\bb^{\neq},\rr\fs\}$, and $\{x_{\bb\bb^{=}}=2,x_{\bb\bb^{\neq}}=k-2\}\in \Omega_{\tns\tns}$. Now, recall \eqref{eq:mu:eps:BB:equal:lower}, and note that $\liminf_{\eps \to 0}\mu_{\eps,\bb\bb^{=}}>0$ holds, because $\limsup_{\eps \to 0}\gamma_{\eps,\bsigma}<\infty$. Hence, \eqref{eq:2ndmo:lem:gamma:eps:BS} shows that $\liminf_{\eps\to 0}\gamma_{\eps,\bb\fs}>-\infty$ since we assumed $\xi_{\bb\fs}>0$. Symmetrically, $\liminf_{\eps \to 0}\gamma_{\eps,\fs,\bb}>-\infty$ holds, which concludes the proof of the boundedness of $\bgamma_{\eps}^{-}$ as $\eps \to 0$. It follows by compactness argument that $\bgamma_{\eps}^{-}$ converges to $\bgamma^{-}=\bgamma^{-}(\bp,\bxi)$ satisfying 
	 \begin{equation}\label{eq:apriori:2ndmo:lem:gamma:minus}
	     \partial_{\bgamma_{\bsigma}}\Lambda^{-}_{\bxi}(\bp,\bgamma^{-})= \kappa \xi_{\bsigma},\quad\bsigma \in \XXX^\prime.
	 \end{equation}

	 Finally, we aim to show $||\bgamma(\bp,\bxi)||_1\lesssim\frac{k^3}{2^{k/2}}$ and $||\nabla_{\bp}\Lambda_{\xi}\left(\bp,\bgamma(\bp,\bxi)\right)||_1\lesssim \frac{k^3}{2^{k/2}}$. For simplicity, we abbreviate $\bgamma = \bgamma(\bp,\bxi)$ and $\bmu \equiv (\mu_{\bsigma})_{\bsigma\in\XXX}\equiv \bmu\left(\bgamma(\bp,\bxi)\right)$ from now on. The crucial observations are summarized as follows.
	 \begin{itemize}
	     \item \eqref{eq:mu:eps:BB:equal:lower} shows that $\mu_{\bb\bb^{\neq}}\geq \frac{1}{20}$ and symmetrically, $\mu_{\bb\bb^{=}}\geq \frac{1}{20}$, so
	     \begin{equation}\label{eq:apriori:2ndmo:BB:crude}
	         \mu_{\bb\bb^{\neq}}, \mu_{\bb\bb^{=}}\geq \frac{1}{20}\quad\textnormal{and}\quad \mu_{\bb\fs}+\mu_{\fs\bb}+\mu_{\fs\fs}\leq \frac{9}{10}
	     \end{equation}
	     \item Similar to $h(x)$ defined in \eqref{eq:apriori:def:h}, let $h_2(\bx)\equiv \left(1-\hat{v}_2(\bx)\right) \one\{\bx \in \Omega_{\tns\tns}\}$. Then, $\hat{v}_2(\bx)\one\{\bx\in \Omega_{\tns\tns}\}=1-h_2(\bx)-\one\{\bx\notin \Omega_{\tns\tns}\}$ holds and we have the following upper bound for $h_2(\bx)$:
	     \begin{equation}\label{eq:2ndmo:upperbound:h:two}
	         h_2(\bx)\leq h_2^{\textnormal{up}}(\bx)\equiv \frac{2^{x_{\fs\fs}+x_{\fs\bb}}}{2^{k-1}}+\frac{2^{x_{\fs\fs}+x_{\bb\fs}}}{2^{k-1}}
	     \end{equation}
	 \end{itemize}
	 With the above observations in mind, we now show the improved bound $\mu_{\bb\bb^{=}}, \mu_{\bb\bb^{\neq}}\geq \frac{1}{2}-O(\frac{k^2}{2^{k/2}})$. First note that by \eqref{eq:2ndmo:compute:partial:Lambda}, we have
	 \begin{equation}\label{eq:aprori:2ndmo:lem:inter-0}
	 \begin{split}
	     \partial_{\gamma_{\bb\bb^{\neq}}}\Lambda_{\bxi}^{\tns\tns}(\bgamma^{-})=k\mu_{\bb\bb^{\neq}}\frac{1-\E_{\bmu}^{k-1}[h_2(X+\one_{\bb\bb^{\neq}})]-\P_{\bmu}^{k-1}(X+\one_{\bb\bb^{\neq}}\notin \Omega_{\tns\tns})}{1-\E_{\bmu}^{k}[h_2(X)]-\P_{\bmu}^{k}(X\notin \Omega_{\tns\tns})}.
	 \end{split}
	 \end{equation}
	 Using union bound and the crude bound \eqref{eq:apriori:2ndmo:BB:crude}, we can upper bound $\P_{\bmu}^{k}(X\notin \Omega_{\tns\tns})$ by
	 \begin{equation}\label{eq:apriori:2ndmo:lem:inter-1}
	 \begin{split}
	     \P_{\bmu}^{k}(X\notin \Omega_{\tns\tns})&\leq \P_{\bmu}^{k}(X_{\bb\bb}+X_{\bb\fs}\leq 1)+\P_{\bmu}^{k}(X_{\bb\bb}+X_{\fs\bb}\leq 1)+\P_{\bmu}^{k}(X_{\fs\fs}=k-2,X_{\bb\bb^{=}}=X_{\bb\bb^{\neq}}=1)\\
	     &\lesssim k^2 0.9^{k}
	 \end{split}
	 \end{equation}
	 Similarly, $\P_{\bmu}^{k-1}(X+\one_{\bb\bb^{\neq}}\notin \Omega_{\tns\tns})\lesssim k^2 0.9^{k}$ holds. Moreover, \eqref{eq:2ndmo:upperbound:h:two} shows
	 \begin{equation}\label{eq:apriori:2ndmo:lem:inter-2}
	     \E_{\bmu}^{k}[h_2(X)]\leq \E_{\bmu}^{k}[h_2^{\textnormal{up}}(X)]=\frac{(1+\mu_{\fs\fs}+\mu_{\fs\bb})^{k}}{2^{k-1}}+\frac{(1+\mu_{\fs\fs}+\mu_{\bb\fs})^{k}}{2^{k-1}}\lesssim 0.95^{k}.
	 \end{equation}
	 Similarly, $\E_{\bmu}^{k-1}[h_2(X+\one_{\bb\bb^{\neq}})]\lesssim 0.95^{k}$ holds. Hence, plugging in the bound \eqref{eq:apriori:2ndmo:lem:inter-1} and \eqref{eq:apriori:2ndmo:lem:inter-2} into \eqref{eq:aprori:2ndmo:lem:inter-0} and using \eqref{eq:apriori:2ndmo:lem:gamma:minus} for $\bsigma=\bb\bb^{\neq}$ show
	 \begin{equation*}
	 \kappa\xi_{\bb\bb^{\neq}}=\sum_{\btau \in \YYY}p_{\btau}\partial_{\gamma_{\bb\bb^{\neq}}}\Lambda_{\xi}^{\btau}(\bgamma^{-})+p_{\tns\tns}\Lambda^{\tns\tns}_{\bxi}(\bgamma^{-})=O\left(\frac{k^2}{2^k}\right)+\frac{k\mu_{\bb\bb^{\neq}}}{1-O(0.95^k)},
	 \end{equation*}
	which implies that $\mu_{\bb\bb^{\neq}}\geq \frac{1}{2}-O(0.95^k)$. Symmetrically, $\mu_{\bb\bb^{=}}\geq \frac{1}{2}-O(0.95^k)$ holds, so $\mu_{\fs\bb}+\mu_{\bb\fs}+\mu_{\fs\fs}=O(0.95^k)$. Note that we can iterate once more, i.e. use $\mu_{\fs\bb}+\mu_{\bb\fs}+\mu_{\fs\fs}=O(0.95^k)$ to get improved bounds for \eqref{eq:apriori:2ndmo:lem:inter-1} and \eqref{eq:apriori:2ndmo:lem:inter-2}, to show that
	\begin{equation}\label{eq:mu:bb:final}
	    \mu_{\bb\bb^{=}},\mu_{\bb\bb^{\neq}}\geq \frac{1}{2}-O\left(\frac{k^2}{2^{k/2}}\right)\quad\textnormal{and}\quad \mu_{\bb\fs}+\mu_{\fs\bb}+\mu_{\fs\fs}=O\left(\frac{k^2}{2^{k/2}}\right).
	\end{equation}
	Having \eqref{eq:mu:bb:final} in hand, we claim that $\frac{\partial_{\gamma_{\bsigma}}\Lambda^{\btau}_{\bxi}(\bgamma^{-})}{(k-\one\{\btau \in \YYY\})\mu_{\bsigma}}=1+O(\frac{k^3}{2^{k/2}})$ for all $\bsigma \in \XXX$ and $\btau \in \YYY\cup\{\tns\tns\}$, except for the case when $\bsigma$ and $\btau$ are incompatible. Here, $\bsigma$ and $\btau$ are defined to be incompatible if and only if $\btau\in \{\rr\ast\}$ and $\bsigma \in \{\fs\ast\}$ or $\btau\in \{\ast\rr\}$ and $\bsigma \in \{\ast\fs\}$. For incompatible $\bsigma$ and $\btau$, $\partial_{\gamma_{\bsigma}}\Lambda_{\bxi}^{\btau}(\bgamma^{-})=0$. Otherwise, for $\btau \in \YYY$, we have
	\begin{equation}\label{eq:apriori:2ndmo:lem:partial:Lambda:tau}
	    \frac{\partial_{\gamma_{\bsigma}}\Lambda^{\btau}_{\bxi}(\bgamma^{-})}{(k-1)\mu_{\bsigma}}=\frac{1-\P_{\bmu}^{k-2}(X+\one_{\bsigma}\notin \Omega_{\btau})}{1-\P_{\bmu}^{k-1}(X\notin \Omega_{\btau})}.
	\end{equation}
	For any $\btau \in \YYY$, $\{\bx\not\in \Omega_{\btau}\}\subset \{\bx_{\ast\fs}\geq 1\}\cup\{\bx_{\fs\ast}\geq 1\}\cup \{x_{\bb\bb^{=}}\geq 1\}\cup\{x_{\bb\bb^{\neq}}\geq 1\}$, so by union bound,
	\begin{equation}\label{eq:2ndmo:super:technical-0}
	    \P_{\bmu}^{k-1}(X\notin \Omega_{\btau})\leq k\mu_{\ast \fs}+k\mu_{\fs\ast}+k(1-\mu_{\bb\bb^{=}})^{k-1}+k(1-\mu_{\bb\bb^{\neq}})^{k-1}=O\left(\frac{k^3}{2^{k/2}}\right)
	\end{equation}
	Similarly, $\P_{\bmu}^{k-2}(X+\one_{\bsigma}\notin \Omega_{\btau})=O(\frac{k^{3}}{2^{k/2}})$ holds for $\bsigma$ compatible with $\btau$. Hence, $  \frac{\partial_{\gamma_{\bsigma}}\Lambda^{\btau}_{\bxi}(\bgamma^{-})}{(k-1)\mu_{\bsigma}}=1+O(\frac{k^{3}}{2^{k/2}})$ holds by \eqref{eq:apriori:2ndmo:lem:partial:Lambda:tau}. For the case of $\btau=\{\fs\fs\}$, similar calculations done in \eqref{eq:aprori:2ndmo:lem:inter-0}-\eqref{eq:apriori:2ndmo:lem:inter-2} show $\frac{\partial_{\gamma_{\bsigma}}\Lambda^{\tns\tns}_{\bxi}(\bgamma^{-})}{k\mu_{\bsigma}}=1+O(\frac{k^{3}}{2^{k/2}})$ for all $\bsigma \in \XXX$. Hence, for all cases we have
	\begin{equation}\label{eq:apriori:2ndmo:lem:most:important}
	    \frac{\partial_{\gamma_{\bsigma}}\Lambda^{\btau}_{\bxi}(\bgamma^{-})}{(k-\one\{\btau \in \YYY\})\mu_{\bsigma}}=1+O\left(\frac{k^3}{2^{k/2}}\right),\quad\forall~\textnormal{compatible}\quad \bsigma\in \XXX\quad\textnormal{and}\quad \btau \in \YYY\cup \{\tns\tns\}.
	\end{equation}
	Therefore, by \eqref{eq:apriori:2ndmo:lem:gamma:minus} and \eqref{eq:apriori:2ndmo:lem:most:important}, we have the following for $\bsigma \in \{\ast\fs\}$.
	\begin{equation*}
	    \kappa\xi_{\bsigma}=\left((k-1)p_{\rr\ast}+kp_{\tns\tns}\right)\mu_{\bsigma}\left(1+O\left(\frac{k^3}{2^{k/2}}\right)\right),
	\end{equation*}
	which implies that $\xi_{\bsigma}=\mu_{\bsigma}\left(1+O(\frac{k^3}{2^{k/2}})\right)$. Analogously, for $\bsigma \in \{\fs\ast\}\cup \{\bb\bb^{=},\bb\bb^{\neq}\}$, the same holds. (Recall that \eqref{eq:apriori:2ndmo:lem:gamma:minus} implies that $\partial_{\gamma_{\bb\bb^{=}}}\Lambda^{-}_{\bxi}(\bp,\bgamma^{-})=\kappa \xi_{\bb\bb^{=}}$.) Therefore, we conclude that
	\begin{equation}\label{eq:apriori:2ndmo:lem:final}
	    \xi_{\bsigma}=\mu_{\bsigma}\left(1+O\left(\frac{k^3}{2^{k/2}}\right)\right),\quad\forall \bsigma \in \XXX.
	\end{equation}
	In particular, recalling $\gamma_{\bb\bb^{=}}\equiv 0$, taking $\bsigma=\bb\bb^{=}$ in the equation above shows 
	\begin{equation}\label{eq:2ndmo:super:technical}
	    \sum_{\bsigma \in \XXX}\xi_{\bsigma}e^{\gamma_{\bsigma}}=1+O\left(\frac{k^3}{2^{k/2}}\right),
	\end{equation}
	which in turn implies $|\gamma_{\bsigma}|=O(\frac{k^3}{2^{k/2}})$ for $\bsigma \in \XXX^\prime$ by \eqref{eq:lem:apriori:1stmo:comparison:final}. To conclude, note that for $\btau \in \YYY$
	\begin{equation*}
	    \partial_{p_{\btau}}\Lambda_{\bxi}(\bp,\bgamma)=\Lambda^{\btau}_{\bxi}(\bgamma^{-})=\log\left(1-\P_{\mu}^{k-1}(X\notin \Omega_{\btau})\right)+(k-1)\log(\sum_{\bsigma\in \XXX}\xi_{\bsigma}e^{\gamma_{\bsigma}})=O\left(\frac{k^4}{2^{k/2}}\right),
	\end{equation*}
	where the last bound is due to \eqref{eq:2ndmo:super:technical-0} and \eqref{eq:2ndmo:super:technical}. Similarly, for $\btau=\fs\fs$,
	\begin{equation*}
	    \partial_{p_{\tns\tns}}\Lambda_{\bxi}(\bp,\bxi)=\Lambda_{\bxi}^{\tns\tns}(\bgamma^{-})=\log\left(1-\E_{\bmu}^{k}[h_2(X)]-\P_{\bmu}^{k}(X\notin \Omega_{\tns\tns})\right)+k\log(\sum_{\bsigma\in \XXX}\xi_{\bsigma}e^{\gamma_{\bsigma}})=O\left(\frac{k^4}{2^{k/2}}\right),
	\end{equation*}
	which altogether shows $||\nabla_{\bp}\Lambda_{\xi}\left(\bp,\bgamma(\bp,\bxi)\right)||_1\lesssim \frac{k^4}{2^{k/2}}$.
	  \end{proof}
    \begin{lemma}\label{lem:apriori:2ndmo:cov:comparable}
        Recalling the notation $\XXX^\prime \equiv \{\bsigma \in \XXX\setminus \{\bb\bb^{=}\}: \xi_{\bsigma}\neq 0\}$, let $\Cov_{\bxi}(X^{\btau})$ and $\Cov_{\bgamma, \bxi}(\wt{X}^{\btau})$ respectively denote the covariance matrices of $(X^{\btau}_{\bsigma})_{\bsigma\in \XXX^\prime}$ and $(\wt{X}^{\btau}_{\bsigma})_{\bsigma\in \XXX^\prime}$, where $X^{\btau}\sim \textnormal{Multinomial}(k-\one(\btau\in \YY),\bxi)$ and $\wt{X}^{\btau}\sim \nu_{\btau,\bgamma,\bxi}$. Then, uniformly over $\bxi\equiv (\xi_{\bsigma})_{\bsigma\in \XXX},\bp\equiv (p_{\btau})_{\btau\in \YYY\cup\{\tns\tns\}}$ such that $\xi_{\bb\bb^{=}}, \xi_{\bb\bb^{\neq}} \in [\frac{1}{2}-\frac{3k^2}{2^{k/2}}, \frac{1}{2}+\frac{3k^2}{2^{k/2}}]$ and $\sum_{\btau \in \YYY} p_{\btau} \leq \frac{15k}{2^k}$, we have that 
        \begin{equation*}
            \det\bigg(\sum_{\btau\in \YYY\cup\{\tns\tns\}}p_{\btau}\Cov_{\bxi}\Big(X^{\btau}\Big)\bigg)\asymp_{k} \det\bigg(\sum_{\btau\in \YYY\cup\{\tns\tns\}}p_{\btau}\Cov_{\bgamma(\bp,\bxi),\bxi}\Big(\wt{X}^{\btau}\Big)\bigg)\,.
        \end{equation*}
    \end{lemma}
    \begin{proof}
       We only provide the proof for the case where $\XXX^\prime =\{\bb\bb^{\neq}, \bb\fs, \fs\bb,\fs\fs\}$, i.e. when $\xi_{\bb\fs},\xi_{\fs\bb},\xi_{\fs\fs}>0$ since other cases follow by simpler argument. Throughout, we treat $X^{\btau}$ as a vector in $\R^{\XXX\setminus\{\bb\bb^{=}\}}$, i.e. $X^{\btau}\equiv (X^{\btau}_{\bsigma})_{\bsigma \in \{\bb\bb^{\neq}, \bb\fs,\fs\bb,\fs\fs\}}$, and similarly $\wt{X}^{\btau}\equiv (\wt{X}^{\btau}_{\bsigma})_{\bsigma \in \{\bb\bb^{\neq}, \bb\fs,\fs\bb,\fs\fs\}}$. Also, for simplicity, we denote $\bgamma\equiv \bgamma(\bp,\bxi)$.

       First, since $\Cov_{\bxi}(X^{\btau})=\big(k-\one(\tau\in \YYY)\big)\big(\textnormal{diag}(\bxi^{-})-\bxi^{-}(\bxi^{-})^{T}\big)$, where $\bxi^{-}\equiv (\xi_{\bsigma})_{\bsigma\in \{\bb\bb^{\neq},\bb\fs,\fs\bb,\fs\fs\}}$, we have 
       \begin{equation}\label{eq:cov:easy}
           \det\bigg(\sum_{\btau\in \YYY\cup\{\tns\tns\}}p_{\btau}\Cov_{\bxi}\Big(X^{\btau}\Big)\bigg)=\det\bigg(\kappa\Big(\textnormal{diag}(\bxi^{-})-\bxi^{-}(\bxi^{-})^{T}\Big)\bigg)\asymp_k \xi_{\bb\fs}\cdot \xi_{\fs\bb}\cdot \xi_{\fs\fs}\,, 
       \end{equation}
       where the last estimate holds since it is straightforward to check using $\max_{\bsigma\in \{\bb\fs,\fs\bb,\fs\fs\}}\xi_{\bsigma}\lesssim k^2/2^{k/2}$ that the determinant of $\textnormal{diag}(\bxi^{-})-\bxi^{-}(\bxi^{-})^{T}$ is dominated by the product of its diagonal elements. To this end, we aim to show that $\det\big(\sum_{\btau\in \YYY\cup \{\tns\tns\}}p_{\tau}\Cov_{\bgamma,\bxi}(\wt{X}^{\btau})\big)\asymp_k \xi_{\bb\fs}\cdot \xi_{\fs\bb}\cdot \xi_{\fs\fs}$. To do so, we first claim that
       \begin{equation}\label{eq:cov:first:claim}
\Cov_{\bgamma,\bxi}\big(\wt{X}^{\btau}_{\bsigma_1},\wt{X}^{\btau}_{\bsigma_2}\big)\lesssim k^2\xi_{\bsigma_1}\cdot \xi_{\bsigma_2}\quad\textnormal{for}\quad \btau\in \YYY\cup\{\tns\tns\}\,,\,\,\bsigma_1,\bsigma_2\in \XXX\setminus \{\bb\bb^{=}\}\,,~~~\textnormal{and}~~~\bsigma_1\neq\bsigma_2\,.
       \end{equation}
       To show \eqref{eq:cov:first:claim}, observe that for all $\btau\in \YYY\cup\{\tns\tns\}$, $(x_{\bsigma})_{\bsigma\in \XXX}\in \Omega_{\btau}$ holds if $x_{\bb\bb^{=}}+x_{\bb\bb^{\neq}}=k-\one\{\btau\in \YYY\}$ and $x_{\bb\fs}=x_{\fs\bb}=x_{\fs\fs}=0$. Thus, it follows that for $\btau\in \YYY\cup\{\tns\tns\}$, 
       \begin{equation}\label{eq:cov:prob:lower}
       \P_{\bxi}\Big(X^{\btau}\in \Omega_{\btau}\Big)\geq (\xi_{\bb\bb^{=}}+\xi_{\bb\bb^{\neq}})^k\geq \Big(1-O(k^2 2^{-k/2})\Big)^k\gtrsim 1.
       \end{equation}
       Moreover, recall the definition of $\nu_{\btau}$ in \eqref{eq:apriori:2ndmo:def:bnu:star}, and that $\hat{v}_2(\bx)\in [1/4,1]$ holds for $\bx\in \Omega_{\tns\tns}$. Thus, it follows that for all $\btau\in \YYY\cup\{\tns\tns\}$ and $\bsigma \in \XXX\setminus \{\bb\bb^{=}\}$, 
       \begin{equation}\label{eq:first:claim:first}
           \E_{\bgamma,\bxi}\big[\wt{X}^{\btau}_{\bsigma}\big]\lesssim e^{2k||\bgamma||_1}\E_{\bxi}\big[X^{\btau}_{\bsigma}\big]\lesssim k\xi_{\bsigma}\,, 
       \end{equation}
       where we used Lemma \ref{lem:apriori:2ndmo:gamma:exists} in the last inequality. Similarly, we have that for $\bsigma_1\neq \bsigma_2$,
       \begin{equation*}
           \E_{\bgamma,\bxi}\big[\wt{X}^{\btau}_{\bsigma_1}\wt{X}^{\btau}_{\bsigma_2}\big]\lesssim e^{4k||\bgamma||_1}\E_{\bxi}\big[X^{\btau}_{\bsigma_1}X^{\btau}_{\bsigma_2}\big]\lesssim k^2\xi_{\bsigma_1}\cdot \xi_{\bsigma_2}\,.
       \end{equation*}
       Combining with \eqref{eq:first:claim:first} shows the estimate \eqref{eq:cov:first:claim}. 

       Second, we claim that for $\bsigma \in \XXX\setminus \{\bb\bb^{=}\}$, 
       \begin{equation}\label{eq:cov:second:claim}
\Var_{\bgamma,\bxi}\big(\wt{X}^{\tns\tns}_{\bsigma}\big)\gtrsim k \xi_{\bsigma}\,.
       \end{equation}
       To show \eqref{eq:cov:second:claim}, note that $\nu_{\tns\tns}(\bx) \asymp p_{\bxi}^k(\bx)$ holds uniformly over $\bx\in \Omega_{\tns\tns}$ by definition of $\nu_{\tns\tns}$ in \eqref{eq:apriori:2ndmo:def:bnu:star} and Lemma \ref{lem:apriori:2ndmo:gamma:exists}. Also, $\P_{\bxi}(X^{\tns\tns}\in \Omega_{\tns\tns})\gtrsim 1$ holds by \eqref{eq:cov:prob:lower}. Thus, we have
       \begin{equation}\label{eq:cov:var:lower}
\Var_{\bgamma,\bxi}\big(\wt{X}^{\tns\tns}_{\bsigma}\big)\gtrsim \Var_{\bxi}\big(X^{\tns\tns}_{\bsigma}\big|X\in \Omega_{\tns\tns}\big)\geq \E_{\bxi}\Big[\Big(X^{\tns\tns}_{\bsigma}-\E_{\bxi}[X^{\tns\tns}_{\bsigma}|X^{\tns\tns}\in \Omega_{\tns\tns}]\Big)^2\one\{X^{\tns\tns}\in \Omega_{\tns\tns}\}\Big]\,.
       \end{equation}
       For $\bsigma \in \{\bb\fs,\fs\bb,\fs\fs\}$, we have $\E_{\bxi}[X^{\tns\tns}_{\bsigma}|X^{\tns\tns}\in \Omega_{\tns\tns}]\lesssim k\xi_{\bsigma}\lesssim k^4 2^{-k/2}$ by \eqref{eq:cov:prob:lower}. Hence, we can further lower bound 
       \begin{equation*}
\Var_{\bgamma,\bxi}\big(\wt{X}^{\tns\tns}_{\bsigma}\big)\gtrsim \P_{\bxi}\Big(X^{\tns\tns}_{\bsigma}=1\,,\, X^{\tns\tns}_{\bb\bb^{=}}+X^{\tns\tns}_{\bb\bb^{\neq}}=k-1\Big)\gtrsim k\xi_{\bsigma}\,. 
       \end{equation*}
       For $\bsigma=\bb\bb^{\neq}$, note that $\E_{\bxi}[X^{\tns\tns}_{\bb\bb^{\neq}}|X^{\tns\tns}\in \Omega_{\tns\tns}]=\frac{k}{2}+O(\frac{k^4}{2^{k/2}})$ holds. Thus, we can further bound the RHS of \eqref{eq:cov:var:lower} by 
       \begin{equation*}
\Var_{\bgamma,\bxi}\big(\wt{X}^{\tns\tns}_{\bb\bb^{\neq}}\big)\gtrsim \E_{\bxi}\Big[\Big(X^{\tns\tns}_{\bb\bb^{\neq}}-\E_{\bxi}[X^{\tns\tns}_{\bb\bb^{\neq}}|X^{\tns\tns}\in \Omega_{\tns\tns}]\Big)^2\one\{X^{\tns\tns}_{\bb\bb^{=}}+X^{\tns\tns}_{\bb\bb^{\neq}}=k\}\Big]\geq \frac{k}{4}-O\Big(\frac{k^5}{2^{k/2}}\Big)\,,
       \end{equation*}
       which finishes the proof of the second claim \eqref{eq:cov:second:claim}. 

       Third, note that by \eqref{eq:first:claim:first}, we have for all $\btau\in \YYY\cup\{\tns\tns\}$ and $\bsigma \in\XXX\setminus\{\bb\bb^{=}\}$,
       \begin{equation}\label{eq:cov:third:claim}
           \Var_{\bgamma,\bxi}\big(\wt{X}^{\btau}_{\bsigma}\big)\leq k \E_{\bgamma,\bxi}\big[\wt{X}^{\btau}_{\bsigma}\big]\lesssim k^2 \xi_{\bsigma}\,.
       \end{equation}
       In the regime $\sum_{\btau \in \YYY} p_{\btau} \leq \frac{15k}{2^k}$, the estimates \eqref{eq:cov:first:claim}, \eqref{eq:cov:second:claim}, and \eqref{eq:cov:third:claim} show that the determinant of $\sum_{\btau\in \YYY\cup\{\tns\tns\}}p_{\btau}\Cov_{\bgamma,\bxi}\big(\wt{X}^{\btau}\big)$ is dominated by the product of its diagonal elements, which implies that 
       \begin{equation*}
           \det\bigg(\sum_{\btau\in \YYY\cup\{\tns\tns\}}p_{\btau}\Cov_{\bgamma,\bxi}\Big(\wt{X}^{\btau}\Big)\bigg)\asymp_k \xi_{\bb\fs}\cdot \xi_{\fs\bb}\cdot \xi_{\fs\fs}\,.
       \end{equation*}
       Combining with \eqref{eq:cov:easy} concludes the proof.
    \end{proof}
	Having Lemma \ref{lem:apriori:2ndmo:gamma:exists} and Lemma \ref{lem:apriori:2ndmo:cov:comparable} in hand, we now prove \eqref{eq:apriori:2ndmo:separating}.
	\begin{prop}\label{prop:apriori:2ndmo:separating}
    Consider $\mathbf{g}_{\rr}\equiv (g_{\rr}(\btau))_{\btau\in \YYY}\in \Z_{\geq 0}^{\YYY}$ and $\underline{E}\equiv (E(\bsigma))_{\bsigma\in \in \XXX}\in \Z_{\geq 0}^{\XXX}$. Further consider $\delta_\circ\geq 0, \underline{\delta}_{\rr}\equiv (\delta_{\rr}(\btau))_{\btau\in \YYY}$, and $\underline{\delta}=\left(\delta(\bsigma)\right)_{\bsigma \in \XXX}$, where all the coordinates are non-negative. Assume that $(1-\frac{28k}{2^k})m\leq m_{\tns}\leq m_{\tns}+\delta_{\circ} \leq m$ and $\sum_{\btau \in \YYY} g_{\rr}(\btau)+\delta_{\rr}(\btau)\leq \frac{14k}{2^k}m$ hold. Further assume that $E_{\bb\bb^{=}}\wedge E_{\bb\bb^{\neq}} \geq \left(\frac{1}{2}-\frac{2k^2}{2^{k/2}}\right)m$ and $\left(E_{\bb\bb^{=}}+\delta(\bb\bb^{=})\right)\wedge \left(E_{\bb\bb^{\neq}}+\delta(\bb\bb^{\neq})\right)\geq \left(\frac{1}{2}-\frac{2k^2}{2^{k/2}}\right)m$ hold. Then, the estimate \eqref{eq:apriori:2ndmo:separating} holds. 
	\end{prop}
	\begin{proof}
	We first introduce the necessary notations. Recalling the definition of $\bp$ and $\bxi$ in \eqref{eq:apriori:2ndmo:def:bp} and \eqref{eq:apriori:2ndmo:def:bxi}, let 
	\begin{equation*}
	\begin{split}
	    &m_{\tns}^\prime \equiv m_{\tns}+\delta_{\circ},\quad\bp^\prime \equiv (\bp_{\YYY}^\prime, p_{\tns\tns}^\prime),\quad\textnormal{where}\quad \bp_{\YYY}^\prime \equiv (p_{\btau}^\prime)_{\btau \in \YYY} \equiv \frac{\mathbf{g}_{\rr}+\underline{\delta}_{\rr}}{m_{\tns}^\prime}\quad\textnormal{and}\\
	    &p_{\tns\tns}^\prime \equiv 1-\sum_{\btau \in \YYY}p_{\btau}^\prime,\quad\kappa^\prime\equiv k-\sum_{\btau \in \YYY}p_{\btau}^\prime,\quad \textnormal{and}\quad \bxi^\prime \equiv \frac{\underline{E}+\underline{\delta}}{\kappa^\prime m_{\tns}^\prime}.
	\end{split}
	\end{equation*}
	Recall $\bgamma(\bp,\bxi)$ as in Lemma \ref{lem:apriori:2ndmo:gamma:exists} and abbreviate $\bgamma\equiv (\gamma_{\bsigma})_{\bsigma\in \XXX}\equiv \bgamma(\bp,\bxi)$ and $\bgamma^\prime\equiv (\gamma^\prime_{\bsigma})_{\bsigma\in \XXX} \equiv \bgamma(\bp^\prime,\bxi^\prime)$ for simplicity. Recalling the estimate \eqref{eq:apriori:2ndmo:f:final}, which follows from Lemmas \ref{lem:apriori:2ndmo:gamma:exists} and \ref{lem:apriori:2ndmo:cov:comparable}, we have that
	\begin{multline}\label{eq:apriori:2ndmo:prop:firststep}
	      \frac{f_2(m_{\tns},m_{\tns}\bp, m_{\tns}\bxi)}{f_2(m_{\tns}^\prime,m_{\tns}^\prime\bp^\prime, m_{\tns}^\prime\bxi^\prime)} \lesssim \exp\bigg\{m_{\tns}^\prime\Big(\kappa^\prime \langle \bgamma^\prime, \xi^\prime \rangle - \Lambda_{\bxi^\prime}(\bp^\prime,\bgamma^\prime)\Big)-m_{\tns}\Big(\kappa\langle \bgamma, \xi \rangle - \Lambda_{\bxi}(\bp,\bgamma)\Big)\bigg\}.
	\end{multline}
    For $0\leq t \leq 1$, define
	\begin{equation*}
	\begin{split}
	&m_t \equiv m_{\tns}+t \delta_{\circ},\quad \bp_t\equiv (\bp_{\YYY,t},p_{\tns\tns, t}),\quad\textnormal{where}\quad \bp_{\YYY,t}\equiv (p_{\btau,t})_{\btau \in \YYY}\equiv \frac{\mathbf{g}_{\rr}+t\underline{\delta}_{\rr}}{m_t}\quad\textnormal{and}\\
	&p_{\tns\tns,t}\equiv 1- \sum_{\btau \in \YYY} p_{\btau, t},\quad \kappa_t\equiv k-\sum_{\btau \in \YYY}p_{\btau, t},\quad\bxi_{t}\equiv \frac{\underline{E}+t\underline{\delta}}{\kappa_{t}m_t},\quad\textnormal{and}\quad \bgamma_t\equiv (\gamma_{\bsigma,t})_{\bsigma \in \XXX}\equiv \bgamma(\bp_{t},\bxi_{t}).
	\end{split}
	\end{equation*}
	Further, let $\bar{f}_2(t)\equiv m_t\left(\kappa_t \langle \bgamma_t,\bxi_t \rangle- \Lambda_{\bxi_t}(\bp_t,\bgamma_t)\right)$. Since $\bar{f}_2(t)=m_t \sup_{\bgamma}\left\{\langle \bgamma, \kappa_t\bxi_t\rangle -\Lambda_{\xi_t}(\bp_t,\bgamma)\right\}$, $\bar{f}_2(t)$ is continuous in $[0,1]$ and differentiable in $(0,1)$. Thus, we can bound
	\begin{equation}\label{eq:2ndmo:prop:upperbound:by:fprime}
	\bigg|m_{\tns}^\prime\Big(\kappa^\prime \langle \bgamma^\prime, \xi^\prime \rangle - \Lambda_{\bxi^\prime}(\bp^\prime,\bgamma^\prime)\Big)-m_{\tns}\Big(\kappa\langle \bgamma, \xi \rangle - \Lambda_{\bxi}(\bp,\bgamma)\Big)\bigg| = |\bar{f}_2(1)-\bar{f}_2(0)|\leq \sup_{0\leq t\leq 1}\bigg|\frac{d\bar{f}_2(t)}{dt}\bigg|.
	\end{equation}
	To this end, we compute $\frac{d\bar{f}_2(t)}{dt}$. Since $\nabla_{\bgamma} \Lambda_{\bxi_t}(\bp_t, \bgamma_t)=\kappa_t\bxi_t$,
	\begin{equation*}
	    \frac{d\bar{f}_2(t)}{dt}= \langle \bgamma_t, \underline{\delta}\rangle -\delta_{\circ}\Lambda_{\bxi_t}(\bp_t,\bgamma_t)-m_t\bigg\langle\frac{d\bxi_t}{dt}\,,\,\nabla_{\bxi}\Lambda_{\bxi_t}(\bp_t,\bgamma_t)\bigg\rangle-m_t\bigg\langle \frac{d\bp(t)}{dt}\,,\,\nabla_{\bp}\Lambda_{\bxi_t}(\bp_t,\bxi_t)\bigg\rangle.
	\end{equation*}
	Similar calculations as in \eqref{eq:apriori:1stmo:compute:partial:theta} show $\nabla_{\bxi}\Lambda_{\bxi_t}(\bp_t,\bgamma_t)=0$. Also, $m_t\frac{d\bp(t)}{dt}=\underline{\delta}_{\rr}-\delta_{\circ}\bp_t$ and $\langle \bp_t, \nabla_{\bp}\Lambda_{\bxi_t}(\bp_t,\bxi_t) \rangle =\Lambda_{\bxi_t}(\bp_t,\bgamma_t)$. Hence, Lemma \ref{lem:apriori:2ndmo:gamma:exists} shows
	\begin{equation}\label{eq:2ndmo:prop:fprime:bound}
	    \bigg|\frac{d\bar{f}_2(t)}{dt}\bigg| =\bigg|\langle \bgamma_t, \underline{\delta}\rangle-\Big\langle \underline{\delta}_{\rr}, \nabla_{\bp}\Lambda_{\bxi_t}(\bp_t,\bxi_t) \Big\rangle \bigg|\lesssim \frac{k^4}{2^{k/2}}\Big(||\underline{\delta}||_1+||\underline{\delta}_{\rr}||_1\Big).
	\end{equation}
	Therefore, by \eqref{eq:apriori:2ndmo:prop:firststep}, \eqref{eq:2ndmo:prop:upperbound:by:fprime}, and \eqref{eq:2ndmo:prop:fprime:bound}, we have 
	\begin{equation*}
	\frac{f_2(m_{\tns},m_{\tns}\bp, m_{\tns}\bxi)}{f_2(m_{\tns}^\prime,m_{\tns}^\prime\bp^\prime, m_{\tns}^\prime\bxi^\prime)} \lesssim_{k} \exp\left\{O\left(\frac{k^4}{2^{k/2}}\right)\Big(||\underline{\delta}||_1+||\underline{\delta}_{\rr}||_1\Big)\right\},
	\end{equation*}
	which concludes the proof of \eqref{eq:apriori:2ndmo:separating}.
	\end{proof}

		\section{Compatibility properties }
	\label{sec:appendix:Properties of BP fixed point}
	
	In this section, we establish compatibility properties of the BP fixed point, which were used in Sections \ref{sec:1stmo} and \ref{sec:2ndmo}. In Section \ref{subsec:app:compat:BP:1stmo}, we consider the single copy model, corresponding to results in Section \ref{sec:1stmo}, and in Section \ref{subsec:compat:pair}, we consider the pair copy model, corresponding to results in Section \ref{sec:2ndmo}. Throughout, we let $\textnormal{per}(\underline{z}):= \{(z_{\pi(1)},\ldots, z_{\pi(\ell)}):\pi\in S_{\ell} \}$ be the set of permutations for a vector $\underline{z}=(z_1,\ldots,z_{\ell})$.
	
	\subsection{Compatibility of the BP fixed point in the single-copy model}\label{subsec:app:compat:BP:1stmo}
	
	For a free tree $\ttt\in \FFF_{\tr}$, recall the definition of the coloring $\sig(\ttt)$, defined in \eqref{eq:def:col on freetree}. The following lemma is the crux of the compatibility results for the single-copy model.
	\begin{lemma}\label{lem:compat:opt:H:opt:p:tree}
	For $\sig \in \Omega^{\ell}, \ell \geq 1$, define $\langle \sig \rangle$ similarly to \eqref{eq:def:spin multi index} by $\langle \sig \rangle (\sigma) \equiv \sum_{i=1}^{\ell}\one\{\sigma=\sigma_i\}, \forall \sigma\in \Omega$, i.e. $\langle \sig \rangle$ is the empirical count of the spins $\{\sigma_1,...,\sigma_\ell\}$. If $\sig \in \Omega_L^{k}$ is non-separating, we have
	\begin{equation}\label{eq:lem:compat:opt:H:opt:p:tree:clause}
	    \frac{d}{k}\binom{k}{\langle \sig \rangle} \hat{H}^\star_{\la,L}(\sig)=\sum_{\ttt\in \FFF_{\tr}}p^\star_{\ttt,\la,L}\big|\{a\in F(\ttt):\sig_{\delta a}(\ttt)\in \textnormal{per}(\sig)\}\big|
	\end{equation}
	Moreover, for free $\sig \in \Omega_L^{d}$ , i.e. $\dsigma_i \in \{\ff\},\forall 1\leq i \leq d$, and $\sigma \in \Omega_L\cap\{\ff\}$, we have
	\begin{equation}\label{eq:lem:compat:opt:H:opt:p:tree:variable}
	\begin{split}
	    \binom{d}{\langle \sig \rangle}\dot{H}^\star_{\la,L}(\sig)&=\sum_{\ttt\in \FFF_{\tr}}p^\star_{\ttt,\la,L}\big|\{v\in V(\ttt):\sig_{\delta v}(\ttt)\in \textnormal{per}(\sig)\}\big|\\
	    d\bar{H}^\star_{\la,L}(\sigma)&=\sum_{\ttt\in \FFF_{\tr}}p^\star_{\ttt,\la,L}\big|\{e\in E(\ttt):\sigma_e(\ttt)=\sigma\}\big|
	\end{split}
	\end{equation}
	The analogs hold for the untruncated model, where we drop the subscript $L$ in the equations above. 
	\end{lemma}
	\begin{proof}
	We only prove \eqref{eq:lem:compat:opt:H:opt:p:tree:clause} since \eqref{eq:lem:compat:opt:H:opt:p:tree:variable} and the analog for the truncated model follow by a similar argument. For simplicity, denote $\dot{q}^\star=\dot{q}^\star_{\la,L}$ and $\hat{q}^\star =\hat{\textnormal{BP}}_{\la,L}\dot{q}^\star_{\la,L}$. Then, $\dot{q}^\star=\dot{\textnormal{BP}}_{\la,L}\hat{q}^\star$ holds since $\dot{q}^\star$ is the BP fixed point. Thus, recalling the normalizing constant $\hat{\mathfrak{Z}}^\star=\hat{\mathfrak{Z}}^\star_{\dot{q}^\star_{\la,L}}$ for $\hat{H}^\star_{\la,L}$, we have
	\begin{equation}\label{eq:expand:H:op:lambda:L}
	\hat{H}^\star_{\la,L}(\sig)=(\hat{\mathfrak{Z}}^\star)^{-1}\hat{\Phi}(\sig)^\la \prod_{i=1}^{k}\dot{q}^\star(\dsigma_i)=(\hat{\mathfrak{Z}}^\star)^{-1}(\dot{\ZZZ}^\star)^{-k}\hat{\Phi}(\sig)^\la \prod_{i=1}^{k}\left\{\sum_{\utau\in \Omega_L^{d}:\dot{\tau}_1=\dsigma_i}\bar{\Phi}(\tau_1)^{\la}\dot{\Phi}(\utau)^\la\prod_{j=2}^{d}\hat{q}^\star(\hat{\tau}_j)\right\},
	\end{equation}
	where $\dot{\ZZZ}^\star=\dot{\ZZZ}_{\hat{q}^\star_{\la,L}}$ is the normalizing constant for $\dot{\textnormal{BP}}_{\la,L}\hat{q}^\star_{\la,L}$. Observe that we can further expand the \textsc{rhs} of the equation above by $\hat{q}^\star=\hat{\textnormal{BP}}_{\la,L}\dot{q}^\star$. We can iterate this procedure using the relationship $\dot{q}^\star=\dot{\textnormal{BP}}_{\la,L}\hat{q}^\star, \hat{q}^\star=\hat{\textnormal{BP}}_{\la,L}\dot{q}^\star$ until $\hat{H}^\star_{\la,L}(\sig)$ is expressed as a polynomial of $\dot{q}^\star(\bb_0)=\dot{q}^\star(\bb_1)$ and $\hat{q}^\star(\fs)$. Note that the degrees of $\dot{q}^\star(\bb_0)$ and $\hat{q}^\star(\fs)$ are determined by $\sig$ by summing up the clause-adjacent and variable-adjacent boundary half-edges in $\dsigma_1,...,\dsigma_k$ respectively. To this end, we now aim to compute the coefficient in front of the monomial of $\dot{q}^\star(\bb_0)$ and $\hat{q}^\star(\fs)$, when we expand $\hat{H}^\star_{\la,L}(\sig)$.

	To begin with, we view $\sig$ as joining the trees $\dsigma_1,...,\dsigma_k$ at a root clause $ a_{0}$ to form a tree $T$. Denote the set of variables and the clauses of $T$ by $V(T)$ and $F(T)$ respectively. Note that viewing $a_0$ as a root, every $v\in V(T)$ and $a\in F(T)\backslash \{a_0\}$ has a parent edge in its neighbor $\delta v$ and $\delta a$, which we denote by $e_0(v)$ and $e_0(a)$ respectively. We call elements of $\delta v\backslash \{e_0(v)\}$ and $\delta a\backslash \{e_0(a)\}$ children edges. Then, we make the following crucial observations.
	\begin{itemize}
	    \item Given $v\in V(T)$ and $\dsigma_{e_0(v)}\in \{\ff\}$, there exists a unique set of clause-to-variables colorings $\{\hat{\sigma}_e\}_{e\in \delta v \backslash e_0(v)}$, which are compatible with $\dsigma_{e_0(v)}$. That is, if $\utau \in \Omega_{L}^{d}$ and $\dot{\Phi}(\utau)\neq 0$ with $\dot{\tau}_1=\dsigma_{e_0(v)}$, then $\{\tau_2,...,\tau_d\}$ is fully determined as a multiset. This is since there is a unique $\{\hat{\sigma}_e\}_{e\in \delta v \backslash e_0(v)}$ such that $\dot{T}\left(\{\hat{\sigma}_e\}_{e\in \delta v \backslash e_0(v)}\right)=\dsigma_{e_0(v)}$, where $\dot{T}$ is defined in Definition \eqref{def:model:msg config}.
	    \item The same need not hold for $a\in F(T)\backslash \{a_0\}$ and $\hat{\sigma}_{e_0(a)}\in \{\ff\}$: there could be many valid coloring for children edges of $a$, which are compatible with the parent edge coloring $\hat{\sigma}_{e_0(a)}$. This is because of the nature of the iteration in \eqref{eq:def:localeq:msg}, where if $\{\dsigma_e\}_{e\in \delta a\backslash e_0(a)}$ is compatible with $\hat{\sigma}_{e_0(a)}$, then $\{\dsigma_e\oplus \uL\}_{e\in \delta a\backslash e_0(a)}$ is also compatible for $\uL\in \{0,1\}^{d-1}$.
	    \item Given a set of choices for the colorings of the children edges of $a\in F(T)\backslash \{a_0\}$, there exists a unique free tree $\ttt$ that corresponds to these choices. Moreover, it is not hard to see that after fixing a free tree $\ttt$, the number of choices for the colorings of the children edges of $v\in V(T)$ and $a\in F(T)\backslash \{a_0\}$ which give rise to $\ttt$ is given by \begin{equation}\label{eq:compat:arranging:tree:1stmo}
	        \prod_{v\in V(\ttt)}\frac{1}{d}\binom{d}{\langle \sig_{\delta v}\rangle}\prod_{\substack{a\in F(\ttt)\\a\neq a_0}}\frac{1}{k}\binom{k}{\langle \sig_{\delta a}\rangle}\Big|\{a\in F(\ttt):\sig_{\delta a}(\ttt)=\sig\}\Big|=\frac{k}{d} \frac{J_{\ttt}}{\binom{k}{\langle \sig \rangle}}\Big|\{a\in F(\ttt):\sig_{\delta a}(\ttt)\in \textnormal{per}(\sig)\}\Big|,
	    \end{equation}
	    where $\langle \sig_{\delta v} \rangle$ and $\langle \sig_{\delta a}\rangle$ are defined in \eqref{eq:def:spin multi index}.
	\end{itemize}
	With the above observations and the paragraph below \eqref{eq:expand:H:op:lambda:L} in mind, we can compute
	\begin{equation*}
	\begin{split}
	    \hat{H}^\star_{\la,L}(\sig)&=\sum_{\ttt\in \FFF_{\tr}}\bigg\{\frac{k}{d} \frac{J_{\ttt}}{\binom{k}{\langle \sig \rangle}}\Big|\{a\in F(\ttt):\sig_{\delta a}(\ttt)\in \textnormal{per}(\sig)\}\Big|(\hat{\mathfrak{Z}}^\star)^{-1}(\dot{\ZZZ}^\star)^{-|V(\ttt)|}(\hat{\ZZZ}^\star)^{-(|F(\ttt)|-1)}\\
	    &\quad\quad\quad\quad\times\prod_{v\in V(\ttt)}\dot{\Phi}\left(\sig_{\delta v}(\ttt)\right)^{\la} \prod_{a\in F(\ttt)}\hat{\Phi}\left(\sig_{\delta a}(\ttt)\right)^{\la} \prod_{e\in E(\ttt)}\bar{\Phi}\left(\sig_{e}(\ttt)\right)^{\la}\dot{q}^\star(\bb_0)^{|\dot{\partial}\ttt|}\hat{q}^\star(\fs)^{|\hat{\partial}\ttt|}\bigg\}\\
	    &=\sum_{\ttt\in \FFF_{\tr}}\frac{k}{d} \frac{p^\star_{\ttt,\la,L}}{\binom{k}{\langle \sig \rangle}}\Big|\{a\in F(\ttt):\sig_{\delta a}(\ttt)\in \textnormal{per}(\sig)\}\Big|,
	\end{split}
	\end{equation*}
	where the last equality holds because $(\hat{\mathfrak{Z}}^\star)^{-1}\hat{\ZZZ}^\star=(\bar{\mathfrak{Z}}^\star)^{-1}$ and $\dot{\Phi}(\sigma)=2$ holds if $\hat{\sigma}=\fs$. This finishes the proof of \eqref{eq:lem:compat:opt:H:opt:p:tree:clause}.
	\end{proof}
	\begin{lemma}\label{lem:compat:optfr:optbd:1stmo}
	$B^\star_{\lambda,L}$ and $(p_{\ttt, \lambda,L}^\star)_{v(\ttt)\leq L}$, defined in Definition \ref{def:opt:bdry:1stmo}, are compatible. Namely, for $x\in \{\circ, \bb_0,\bb_1,\fs\}$,
	\begin{equation}\label{eq:lem:compat:optfr:optbd:1stmo}
	\sum_{\ttt:v(\ttt)\leq L} p_{\ttt, \lambda,L}^\star \eta_{\ttt}(x)=h^\star_{\lambda,L}(x),
	\end{equation}
	where $\eta_{\ttt}(\circ)\equiv 1$. The same holds for the untruncated model. 
	\end{lemma}
	\begin{proof}
	We consider the truncated model throughout the proof. The result for untruncated model will follow by the same argument. In what follows, we will often omit the subscripts $\lambda$ and $L$ for simplicity. Note that it suffices to prove \eqref{eq:lem:compat:optfr:optbd:1stmo} for $x \in \{\circ,\bb_0,\fs\}$ since $x=\bb_1$ case follows from $x=\bb_0$ case: Define $\ttt \oplus 1 \in \FFF$ to be the free tree obtained from $\ttt$ by flipping the boundary literals and colors adjacent to clauses (inner literals are the same). Then, $J_{\ttt\oplus 1}=J_{\ttt}$, so $p_{\ttt\oplus1}^\star =p_{\ttt}$. Hence,
	\begin{equation*}
	   \sum_{\ttt \in \FFF} p_{\ttt}^\star \eta_{\ttt}(\bb_1)=\sum_{\ttt \in \FFF} p_{\ttt\oplus 1}^\star \eta_{\ttt\oplus 1}(\bb_0)=\sum_{\ttt \in \FFF} p_{\ttt}^\star \eta_{\ttt}(\bb_0),
	\end{equation*}
	and by $0,1$ symmetry of the BP fixed point (see \eqref{eq:def:bp:contract:set:1stmo}), $h^\star(\bb_0)=h^\star(\bb_1)$. We now divide cases.
	
	First, we deal with the case where $x=\bb_0$. Observe that for $\dot{q}^\star=\dot{q}^\star_{\la,L}$, $\hat{q}^\star = \hat{\textnormal{BP}}\dot{q}^\star_{\la,L}$, and the normalizing constant $\hat{\mathfrak{Z}}^\star=\hat{\mathfrak{Z}}_{\dot{q}^\star_{\la,L}}$ for $H^\star=H^\star_{\la,L}$, we can compute
	\begin{equation}\label{eq:compat:1stmo:technical-1}
	    \sum_{\sig \in \Omega_L^{k}}\hat{H}^\star(\sig)\one\{\sigma_1=\bb_0\}=\frac{\dot{q}^\star(\bb_0)}{\hat{\mathfrak{Z}}^\star}\sum_{\sig\in \Omega_L^{k}, \sigma_1=\bb_0}\hat{\Phi}(\sig)^\la\prod_{i=2}^{k}\dot{q}^\star(\sigma_i)=\frac{\dot{q}^\star(\bb_0)}{\hat{\mathfrak{Z}}^\star}\hat{\ZZZ}^\star \hat{q}^\star(\bb_0)=\bar{B}^\star(\bb_0),
	\end{equation}
	where $\hat{\ZZZ}^\star=\hat{\ZZZ}_{\dot{q}^\star_{\la,L}}$ is the normalizing constant for $\hat{\textnormal{BP}}\dot{q}^\star_{\la,L}$ , and the last equality is due to $(\hat{\mathfrak{Z}}^\star)^{-1}\hat{\ZZZ}^\star=(\bar{\mathfrak{Z}}^\star)^{-1}$. On the other hand, recalling the definition of $\hat{B}^\star_{\la,L}$ in \eqref{eq:def:optimal:bdry}, we can compute the contribution from separating $\sig\in \Omega_L^{k},\sigma_1=\bb_0$ by
	\begin{equation}\label{eq:compat:1stmo:technical-2}
	   \sum_{\sig\in \Omega_L^{k}:\textnormal{separating $\sig$}}\hat{H}^\star(\sig)\one\{\sigma_1=\bb_0\}=\sum_{\utau \in \hat{\partial}^{k}}\hat{B}^\star(\utau)\one\{\tau_1=\bb_0\}.
	\end{equation}
	Thus, by \eqref{eq:compat:1stmo:technical-1}, \eqref{eq:compat:1stmo:technical-2}, and the definition of $h^\star(\bb_0)$, we have
	\begin{equation}\label{eq:compat:1stmo:technical-3}
	\frac{1}{d}h^\star(\bb_0)=\bar{B}^\star(\bb_0)-\sum_{\utau \in \hat{\partial}^{k}}\hat{B}^\star(\utau)\one\{\tau_1=\bb_0\}= \sum_{\sig\in \Omega_L^{k}:\textnormal{non-separating $\sig$}}\hat{H}^\star(\sig)\one\{\sigma_1=\bb_0\}.
	\end{equation}
	Now, put an equivalence relation on $\sig \in \Omega_L^{k}$ by $\sig_1\sim \sig_2$ if and only if $\sig_2$ can be obtained by permuting $\sig_1$. Note that if $\sig_1\sim\sig_2$ and $\sig_1$ is non-separating, then $\sig_2$ is also non-separating with $H^\star(\sig_1)=H^\star(\sig_2)$. Also, for $\sigma_{\sim} \in \Omega_L^{k}/\sim$, the number of $\sig\in \sigma_{\sim}$ with $\sigma_1=\bb_0$ is given by $\frac{\eta_{\sigma_{\sim}}(\bb_0)}{k}\binom{k}{\langle \sigma_{\sim}\rangle}$, where $\eta_{\sigma_{\sim}}(\bb_0)$ counts the number of $\bb_0$ in $\sigma_{\sim}$ and $\langle \sigma_{\sim} \rangle$ is the empirical count of the spins in $\sigma_{\sim}$. Hence, by \eqref{eq:compat:1stmo:technical-3}, we have
	\begin{equation}\label{eq:compat:1stmo:bb:final}
	\begin{split}
	    h^\star(\bb_0)
	    &=d\sum_{\sigma_{\sim}\in \Omega_L^{k}/\sim:\textnormal{non-separating}}\hat{H}^\star(\sigma_{\sim})\frac{\eta_{\sigma_{\sim}}(\bb_0)}{k}\binom{k}{\langle \sigma_{\sim}\rangle}\\
	    &=\sum_{\sigma_{\sim}\in \Omega_L^{k}/\sim:\textnormal{non-separating}}\sum_{\ttt\in \FFF_{\tr}}p^\star_{\ttt}\big|\{a\in F(\ttt): \sig_{\delta a}(\ttt) \in \sigma_{\sim}\}\big|\eta_{\sigma_{\sim}}(\bb_0)=\sum_{\ttt\in \FFF_{\tr}:v(\ttt)\leq L}p^\star_{\ttt}\eta_{\ttt}(\bb_0),
	\end{split}
	\end{equation}
	where the second equality is due to Lemma \ref{lem:compat:opt:H:opt:p:tree}. This finishes the proof of \eqref{eq:lem:compat:optfr:optbd:1stmo} for $x=\bb_0$.
	
	Turning to the second case of $x=\fs$, by definition of $h^\star(\fs)$,
	\begin{equation}\label{eq:compat:1stmo:technical-4}
	    h^\star(\fs)=d\sum_{\sig\in \hat{\partial}^{k}}\hat{B}^\star(\sig)\one\{\sigma_1=\fs\}=d\sum_{\sig\in \Omega_L^{k}}\hat{H}^\star(\sig)\one\{\hat{\sigma}_1=\fs\}.
	\end{equation}
	Having \eqref{eq:compat:1stmo:technical-4} in hand, the same computations done in \eqref{eq:compat:1stmo:bb:final}, which were based on Lemma \ref{lem:compat:opt:H:opt:p:tree}, finish the proof for the case of $x=\fs$.
	
	Finally, we deal with the case where $x=\circ$. By definition of $\dot{B}^\star$ in \eqref{eq:def:optimal:bdry},
	\begin{equation*}
	    1-\langle \dot{B}^\star, \one \rangle =\sum_{\sig\in \Omega_L^{d}:\dsigma_i \in \{\ff\},\forall 1\leq i \leq d}\dot{H}^\star(\sig)=\sum_{\ttt\in \FFF_{\tr}:v(\ttt)\leq L}p_{\ttt}^\star v(\ttt),
	\end{equation*}
	where the last equality is due to Lemma \ref{lem:compat:opt:H:opt:p:tree}. Proceeding in the same fashion, we have
	\begin{equation*}
	    1-\langle \hat{B}^\star, \one \rangle=\frac{k}{d}\sum_{\ttt\in \FFF_{\tr}:v(\ttt)\leq L}p^\star_{\ttt}f(\ttt)\quad\textnormal{and}\quad 1-\langle \bar{B}^\star, \one \rangle =\frac{1}{d}\sum_{\ttt\in \FFF_{\tr}:v(\ttt)\leq L}p^\star_{\ttt}e(\ttt).
	\end{equation*}
	Therefore, by definition of $h^\star(\circ)$ given in \eqref{eq:def:compat:1stmo:tree},
	\begin{equation*}
	    h^\star(\circ)=\sum_{\ttt\in \FFF_{\tr}:v(\ttt)\leq L}p^\star_{\ttt}\left(v(\ttt)+f(\ttt)-e(\ttt)\right)=\sum_{\ttt\in \FFF_{\tr}:v(\ttt)\leq L}p^\star_{\ttt},
	\end{equation*}
	which concludes the proof for the case $x=\circ$.
	\end{proof}
	\begin{lemma}\label{lem:1stmo:BP:compatibility}
	Recall $\dot{h}^\star_{\la,L} \equiv \dot{h}_{\dot{q}^\star_{\la,L}}$ from \eqref{eq:1stmo:hdot:intermsof:qdot}. Then, we have
	\begin{equation}\label{eq:lem:1stmo:BP:compatibility}
	\dot{h}^\star_{\la,L}(\dsigma)=
	\begin{cases}
	B^\star_{\la,L}(\dot{\sigma}) &\dsigma \in \{\rr,\bb\}\\
	\frac{1}{d}\sum_{\ttt:v(\ttt)\leq L}p^\star_{\ttt,\la,L}\sum_{e\in E(\ttt)}\one\{\dot{\sigma}_e(\ttt)=\dsigma\} &\dsigma \in \{\ff\}.
	\end{cases}
	\end{equation}
	\end{lemma}
	\begin{proof}
	Note that $\dot{h}^\star_{\la,L}=\dot{h}[H^\star_{\la,L}]$ from their definitions. Thus, $\dot{h}^\star_{\la,L}(\dsigma)=B^\star_{\la,L}(\dsigma)$ holds for $\dsigma \in \{\rr,\bb\}$. For the case of $\dsigma \in \{\ff\}$, we can proceed in a similar fashion as done in \eqref{eq:compat:1stmo:bb:final} to compute
	\begin{equation*}
	\begin{split}
	    \dot{h}^\star_{\la,L}(\dsigma)&=\sum_{\utau\in \Omega_L^{k}}\hat{H}^\star_{\la,L}(\utau)\one\{\dot{\tau}_1=\dsigma\}=\sum_{\sigma_{\sim}\in \Omega_L^{k}/\sim:\textnormal{non-separating}}\hat{H}^\star_{\la,L}(\sigma_{\sim})\frac{\eta_{\sigma_{\sim}}(\dsigma)}{k}\binom{k}{\langle \sigma_{\sim}\rangle}\\
	    &=\sum_{\ttt\in \FFF_{\tr}:v(\ttt)\leq L}p^\star_{\ttt,\la,L}\sum_{e\in E(\ttt)}\one\{\dot{\sigma}_e(\ttt)=\dsigma\}.
	\end{split}
	\end{equation*}
	Here, $\eta_{\sigma_{\sim}}(\dsigma)$ denotes the number of variable-to-clause spins in $\sigma_{\sim}$ which equal $\dsigma$, and we used Lemma \ref{lem:compat:opt:H:opt:p:tree} in the final equality.
	\end{proof}

	\subsection{Compatibility in the pair-copy model}\label{subsec:compat:pair}

Fix a tuple of constants $\ula =(\lambda^1, \lambda^2)$ such that $\lambda^1, \lambda^2 \in [0,1]$, and let $\uuu$ denote a union-free tree. The density of $\uuu$ at optimality $\bp^\star_{\uuu, \ula, L}$ is given by \eqref{eq:optimal:tree:2ndmo}. Note that although $\dot{\bq}^\star = \dot{q}_{\lambda^1, L}^\star \otimes \dot{q}_{\lambda^2, L}^\star$ is a probability measure on the truncated space $\Omega_L^2$, the size of $\uuu$ in \eqref{eq:optimal:tree:2ndmo} does not need to be bounded.
We state the compatibility result for the pair model as follows, which is an analog of Lemmas \ref{lem:compat:optfr:optbd:1stmo} and \ref{lem:1stmo:BP:compatibility} combined.

\begin{cor}\label{cor:compat:2ndmo}
	$\buh_{\ula,L}^\star$ and $(\bp^\star_{\uuu, \ula,L})$ are compatible in the sense that for any $\bx \in \{\circ\}\sqcup \dot{\partial}_2 \sqcup \hat{\partial}_2$,
	\begin{equation*}
	\sum_{\uuu\in \FFF_2^{\tr}} \bp_{\uuu, \ula,L}^\star \eta_\uuu(\bx) = \bh_{\ula,L}^\star (\bx).
	\end{equation*} 
	Moreover, let $\dbh^\star_{\ula,L}:=\dbh[\dbq^\star_{\ula,L}]$, where $\dbh[\dbq]$ can be written explicitly by \eqref{eq:2ndmo:hdot:intermsof:qdot}. Then, we have
	\begin{equation*}
		\dbh^\star_{\ula,L}(\dot{\bsigma})=
		\begin{cases}
		\bB^\star_{\ula,L}(\dot{\bsigma}) &\dot{\bsigma} \in \{\rr,\bb \}^2\\
		\frac{1}{d}\sum_{\uuu} {\bp}^\star_{\uuu,\ula,L}\sum_{e\in E(\uuu)}\one\Big\{\dot{\pi}\big((\dot{\bsigma}_e^{\textsf{u}}(\uuu)\big)=\dot{\bsigma} \Big\} &\dot{\bsigma} \in \dot{\Omega}_L^2 \setminus \{\rr, \bb \}^2.
		\end{cases}
	\end{equation*}
	The analogs hold for the untruncated model, where we drop the subscript $L$ in the equation above.
	\end{cor}

The lemma below, which is an analog of Lemma \ref{lem:compat:opt:H:opt:p:tree} for the pair model, can be proven by a similar argument as in the proof of Lemma \ref{lem:compat:opt:H:opt:p:tree}.

\begin{lemma}\label{cor:compat:opt:H:2ndmo}
For $\ell\geq 1$, let $\pi(\bsigma^{\textsf{u}}_1,...,\bsigma^{\textsf{u}}_{\ell}):=\big(\pi(\bsigma^{\textsf{u}}_i)\big)_{i\leq \ell}$ for $\bsigma^{\textsf{u}}_{i}\in \Omega^{\textsf{u}}, i\leq \ell$. For pair-non-separating $\bsig \in \Omega_{2,L}^{k}$, we have the following. 
	\begin{equation}\label{eq:lem:compat:opt:H:opt:p:tree:clause:2ndmo}
	\frac{d}{k}\binom{k}{\langle \bsig \rangle} \hat{\bH}^\star_{\ula,L}(\bsig)=\sum_{\uuu\in \FFF_{2}^{\tr}}\bp^\star_{\uuu,\ula,L}\Big|\Big\{a\in F(\uuu):\pi\big(\bsig_{\delta a}^{\textsf{u}}(\uuu)\big)\in \textnormal{per}(\bsig)\Big\}\Big|
	\end{equation}
	Moreover, for union-free $\bsig \in \Omega_{2,L}^{d}$, i.e. either $\dsigma_i^1 \in \{\ff\},\forall 1\leq i \leq d$ or $\dsigma_i^2 \in \{\ff\},\forall 1\leq i \leq d$,  we have
	\begin{equation}\label{eq:lem:compat:opt:H:opt:p:tree:variable:2ndmo}
	\begin{split}
	\binom{d}{\langle \bsig \rangle}\dot{\bH}^\star_{\ula,L}(\bsig)=\sum_{\uuu\in \FFF_2^{\tr}}\bp^\star_{\uuu,\la,L}\Big|\Big\{v\in V(\ttt):\pi\big(\bsig_{\delta v}^{\textsf{u}}(\uuu)\big)\in \textnormal{per}(\bsig)\Big\}\Big|
	\end{split}
	\end{equation}
Finally, for  $\bsigma \in \Omega_{2,L}$ such that $\sigma^1 \in \{\ff \} $ or $\sigma^2 \in \{\ff \}$, we have
\begin{equation*}
d\bar{\bH}^\star_{\ula,L}(\bsigma)=\sum_{\uuu\in \FFF_2^{\tr}}\bp^\star_{\uuu,\ula,L}\Big|\Big\{e\in E(\ttt):\pi\big(\bsigma_e^{\textsf{u}}(\uuu)\big)=\bsigma\Big\}\Big|.
\end{equation*}
	The analogs hold for the untruncated model, where we drop subscript $L$ in the equations above. 
\end{lemma}

\begin{proof}[Proof of Corollary \ref{cor:compat:2ndmo}]
	The proof follows the same as that of Lemmas \ref{lem:compat:optfr:optbd:1stmo} and \ref{lem:1stmo:BP:compatibility}, where we use Lemma \ref{cor:compat:opt:H:2ndmo} in the places where Lemma \ref{lem:compat:opt:H:opt:p:tree} is used.
\end{proof}
We conclude this section by proving \eqref{eq:freeenergy:product}:
\begin{lemma}\label{lem:freeenergy:product}
    For $\ula\in [0,1]^2$, we have
    \begin{align}
	    &\bF_{\ula,L}(\bB^\star_{\ula,L},s^\star_{\ula,L})=F_{\la_1,L}(B^\star_{\la_1,L},s^\star_{\la_1,L})+F_{\la_2,L}(B^\star_{\la_2,L},s^\star_{\la_2,L}); \label{eq:freeenergy:product:truncated}\\
	    &\bF_{\ula}(\bB^\star_{\ula},s^\star_{\ula})=F_{\la_1}(B^\star_{\la_1},s^\star_{\la_1})+F_{\la_2}(B^\star_{\la_2},s^\star_{\la_2}) \label{eq:freeenergy:product:untruncated}.
	\end{align}
\end{lemma}
\begin{proof}
    We only prove \eqref{eq:freeenergy:product:untruncated}, since the proof of \eqref{eq:freeenergy:product:truncated} follows by the same argument. Recalling the normalizing constants $\dot{\mathfrak{Z}}_2, \hat{\mathfrak{Z}}_2$ and $\bar{\mathfrak{Z}}_2$ for $\dbH_{\dbq},\hbH_{\dbq}$ and $\bbH_{\dbq}$ in \eqref{eq:H:q:2ndmo}, denote $\dot{\mathfrak{Z}}_2^\star\equiv \dot{\mathfrak{Z}}_{2,\dbq^\star_{\ula}}, \hat{\mathfrak{Z}}_2^\star\equiv \hat{\mathfrak{Z}}_{2,\dbq^\star_{\ula}}$ and $\bar{\mathfrak{Z}}_2^\star\equiv \bar{\mathfrak{Z}}_{2,\dbq^\star_{\ula}}$. Also, we abbreviate $\dbq^\star\equiv \dbq^\star_{\ula}$ and $\hbq^\star\equiv \hbq^\star_{\ula}$ for simplicity. Then, we first aim to prove
    \begin{equation}\label{eq:compute:sec:mo:free:energy}
        \bF_{\ula}(\bB^\star_{\ula},s^\star_{\ula})=\log \dot{\mathfrak{Z}}_2^\star+\alpha \log \hat{\mathfrak{Z}}_2^\star-d\log  \bar{\mathfrak{Z}}_2^\star.
    \end{equation}
    Recalling the definition of $\bF_{\ula}(\cdot)$ from Proposition \ref{prop:exist:free:energy:2ndmo}, it is not hard to see that
    \begin{equation*}
        \bF_{\ula}(\bB^\star_{\ula},s^\star_{\ula})\equiv \Psi_{\circ}(\bB^\star_{\ula})-\big\langle \butheta^\star_{\ula}, (\buh^\star_{\ula},\bs^\star_{\ula})\big\rangle.
    \end{equation*}
    To this end, we first calculate $\Psi_{\circ}(\bB^\star_{\ula})$. Define $\dot{q}^\star_{\la}(\fs):=\dot{q}^\star_{\la}(\ff)$ (note that apriori, $\dot{q}_{\la}(\fs)$ is not defined since $\fs\notin \dot{\Omega}$) and for $(\sigma^1,\sigma^2)$ such that $\sigma^1=\fs$ or $\sigma^2=\fs$, define $\dbq^\star(\sigma^1,\sigma^2)=\dot{q}^\star_{\la_1}(\sigma^1)\dot{q}^\star_{\la_2}(\sigma^2)$. Then, it is straightforward to see from Definition \ref{def:opt:bdry:2ndmo} that the following holds:
    \begin{equation}\label{eq:opt:bdry:2ndmo}
    \begin{split}
        \dbB^\star_{\ula}(\bsig)&\equiv \frac{\prod_{i=1}^{d} \hbq^\star(\bsigma_i)}{\dot{\mathfrak{Z}}_2^\star}\quad\textnormal{for}\quad\bsig \in (\dot{\partial}_2^{\bullet})^{d}\quad\textnormal{if}\quad\dot{\Phi}_2(\bsig)\neq 0;\\
        \hbB^\star_{\ula}(\bsig)&\equiv \frac{\hat{v}_2(\bsig)\prod_{i=1}^{k}\dbq^\star(\bsigma_i)}{\hat{\mathfrak{Z}}_2^\star}\quad\textnormal{for}\quad\bsig \in (\hat{\partial}_2^{\bullet})^{k};\\
        \bbB^\star_{\ula}(\bsigma)&\equiv \frac{g(\bsigma)\dbq^\star(\bsigma)\hbq^\star(\bsigma)}{\bar{\mathfrak{Z}}_2^\star}\quad\textnormal{for}\quad\bsigma\in \hat{\partial}_2^{\bullet},
    \end{split}
    \end{equation}
    where $g(\sigma^1,\sigma^2):=2^{-\la_1\one\{\sigma^1=\fs\}-\la_2\one\{\sigma^2=\fs\}}$. Having \eqref{eq:opt:bdry:2ndmo} in hand and recalling $\big(\bh^\star_{\ula}(\bx)\big)_{\bx\in \dot{\partial}_2\sqcup\hat{\partial}_2}$, which can be computed from $\bB^\star_{\ula}$ by \eqref{eq:def:compat:2ndmo}, it is straightforward to compute
    \begin{equation}\label{eq:compute:Psi:B:star}
    \begin{split}
        \Psi_{\circ}(\bB^\star_{\ula})=&\sum_{\bx \in \dot{\partial}_2}\bh^\star_{\ula}(\bx)\log \dbq^\star(\bx)+\sum_{\bx \in \hat{\partial}_2}\bh^\star_{\ula}(\bx)\log\big(g(\bx)\hbq^\star(\bx)\big)\\
        &\quad\quad\quad+\big\langle \dbB^\star_{\ula},\one \big\rangle \log \dot{\mathfrak{Z}}_2^\star+\alpha \big\langle \hbB^\star_{\ula},\one \big\rangle \log \hat{\mathfrak{Z}}_2^\star-d\big\langle \bbB^\star_{\ula},\one \big\rangle \log \bar{\mathfrak{Z}}_2^\star.
    \end{split}
    \end{equation}
    Moreover, from the definition of $\butheta^\star$ in \eqref{eq:def:opt:theta:2ndmo}, we can compute
    \begin{equation}\label{eq:compute:theta:star:dot:h:star}
    \begin{split}
        \big\langle \butheta^\star_{\ula}, (\buh^\star_{\ula},\bs^\star_{\ula})\big\rangle&=\sum_{\bx \in \dot{\partial}_2}\bh^\star_{\ula}(\bx)\log \dbq^\star(\bx)+\sum_{\bx \in \hat{\partial}_2}\bh^\star_{\ula}(\bx)\log\big(g(\bx)\hbq^\star(\bx)\big)-\big(1-\big\langle \dbB^\star_{\ula},\one \big\rangle\big)\log \dot{\ZZZ}^\star_2\\
        &-\alpha\big(1-\big\langle \hbB^\star_{\ula},\one \big\rangle\big)\log \hat{\ZZZ}^\star_2-\Big(1-\big\langle \dbB^\star_{\ula},\one \big\rangle+\alpha\big((1-\big\langle \hbB^\star_{\ula},\one \big\rangle\big)-d\big(1-\big\langle \bbB^\star_{\ula},\one \big\rangle\big)\Big)\log \bar{\mathfrak{Z}}^\star_2,
    \end{split}
    \end{equation}
    where we used the fact $\sum_{\bx\in\dot{\partial}_2}\bh^\star_{\ula}(\bx)=d\big(\big\langle \bbB^\star_{\ula},\one\big\rangle-\big\langle \hbB^\star_{\ula},\one\big\rangle\big)$ and $\sum_{\bx\in\hat{\partial}_2}\bh^\star_{\ula}(\bx)=d\big(\big\langle \bbB^\star_{\ula},\one\big\rangle-\big\langle \dbB^\star_{\ula},\one\big\rangle\big)$, which can be obtained from \eqref{eq:def:compat:2ndmo}. Since $\dot{\ZZZ}^\star_2=(\bar{\mathfrak{Z}}^\star_2)^{-1}\dot{\mathfrak{Z}}^\star_2$ and $\hat{\ZZZ}^\star_2=(\bar{\mathfrak{Z}}^\star_2)^{-1}\hat{\mathfrak{Z}}^\star_2$ holds, we can substract \eqref{eq:compute:theta:star:dot:h:star} from \eqref{eq:compute:Psi:B:star} to see that our goal \eqref{eq:compute:sec:mo:free:energy} holds.
    
    Then, the same argument shows that the analog of \eqref{eq:compute:sec:mo:free:energy} holds for the first moment:
    \begin{equation}\label{eq:compute:first:mo:free:energy}
    F_{\la}(B^\star_{\la},s^\star_{\la})=\log \dot{\mathfrak{Z}}^\star+\alpha \log \hat{\mathfrak{Z}}^\star-d\log  \bar{\mathfrak{Z}}^\star,
    \end{equation}
    where $\dot{\mathfrak{Z}}^\star\equiv \dot{\mathfrak{Z}}_{\dot{q}^\star_{\la}},\hat{\mathfrak{Z}}^\star\equiv \hat{\mathfrak{Z}}_{\dot{q}^\star_{\la}}, \bar{\mathfrak{Z}}^\star\equiv \bar{\mathfrak{Z}}_{\dot{q}^\star_{\la}}$ are normalizing constants for $\dot{H}^\star_{\la}, \hat{H}^\star_{\la}, \bar{H}^\star_{\la}$ in \eqref{eq:H:q:1stmo}. Then, \eqref{eq:compute:sec:mo:free:energy} and \eqref{eq:compute:first:mo:free:energy} finish the proof of \eqref{eq:freeenergy:product:untruncated} since $\dot{\mathfrak{Z}}_{2,\dbq^\star_{\ula}}=\dot{\mathfrak{Z}}_{\dot{q}^\star_{\la_1}}\dot{\mathfrak{Z}}_{\dot{q}^\star_{\la_2}}, \hat{\mathfrak{Z}}_{2,\dbq^\star_{\ula}}=\hat{\mathfrak{Z}}_{\dot{q}^\star_{\la_1}}\hat{\mathfrak{Z}}_{\dot{q}^\star_{\la_2}}$ and $\bar{\mathfrak{Z}}_{2,\dbq^\star_{\ula}}=\bar{\mathfrak{Z}}_{\dot{q}^\star_{\la_1}}\bar{\mathfrak{Z}}_{\dot{q}^\star_{\la_2}}$ hold.
\end{proof}

% We conclude the section by pointing out a compatibility property between $\{\bp^\star_{\uuu, \ula, L}\}$ and $s^\star_L$, which is a direct consequence of Corollary \ref{cor:compat:opt:H:2ndmo} and the fact that $\bH^\star_{\ula, L}  = H^\star_{\lambda^1, L} \otimes H^\star_{\lambda^2, L}$.

% \begin{cor}
% 	Let $\ula^\star = (\lambda^\star, \lambda^\star)$, and recall the definition of $\bw^{\lit,l}(\uuu)$ for a union-free tree $\uuu \in \FFF_2^{\tr}$ \eqref{eq:def:size:union-comp:each:copy}. Then, for $l=1,2$, we have
% 	\begin{equation*}
% 	\sum_{\uuu \in \FFF_2^{\tr}} \bp^\star_{\uuu, \ula^\star, L} \log \bw^{\lit, l}(\uuu) = s^\star_L.
% 	\end{equation*}
% \end{cor}

	\section{Continuity of tree optimization}
	\label{sec:appendix:continuity:tree:optmization}
	In this section, we gather continuity properties which were used in Section \ref{subsec:resampling}.
	
	\subsection{Continuity in the single-copy model} We first show that $s[\sig]$ for $\sig \in \Omega^{E}$ is a Lipschitz function with respect to $\dot{H}[\sig]$, which was used in the proof of Proposition \ref{prop:maxim:1stmo}. Recall that we define $\utau=(\tau_1,\ldots \tau_d) \in \Omega^d$ to be free if $\tau_i \in \{\fF\}$ holds for all $1\leq i\leq d$.
	\begin{lemma}\label{lem:H:dot:to:s:lipschitz}
	Given a valid \naesat instance $\GGG$ and a valid coloring $\sig\in \Omega^{E}$ on $\GGG$, let $\dot{H}=\dot{H}[\sig]$. Then, we have
	\begin{equation}\label{eq:lem:H:dot:to:s:lipschitz}
	    \big| s[\sig]-s^\star_{\la}\big| \leq \log 2 \sum_{\utau\in \Omega^{d}:\textnormal{free}}\big|\dot{H}(\utau)-\dot{H}^\star_{\la}(\utau)\big|.
	\end{equation}
	\end{lemma}
	\begin{proof}
	As before, put an equivalence relation on $\Omega^{d}$ by $\sig_1\sim \sig_2$ if and only if $\sig_2$ can be obtained from $\sig_1$ by a permutation. Then, for $\sigma_{\sim}\in \Omega^{d}/ \sim$ and $\dot{H}=\dot{H}[\sig]$, we have 
	\begin{equation}\label{eq:H:dot:relation:p:tree}
	    \sum_{\utau \in \sigma_{\sim}}\dot{H}(\utau)=\sum_{\ttt\in \FFF_{\tr}}p_{\ttt}[\sig]\big|\{v\in V(\ttt):\sig_{\delta v}(\ttt)\in \sigma_{\sim}\}\big|
	\end{equation}
	It is not hard to see that for $\ttt_1,\ttt_2\in \FFF_{\tr}$ with $\big|\{v\in V(\ttt_1):\sig_{\delta v}(\ttt_1)\in \sigma_{\sim}\}\big|, \big|\{v\in V(\ttt_2):\sig_{\delta v}(\ttt_2)\in \sigma_{\sim}\}\big|\geq 1$, $s_{\ttt_1}^{\lit}=s_{\ttt_2}^\lit$ holds. This is because we can determine such $\ttt$ from $\sigma_{\sim}$ by choosing colorings of the children edges, described in the proof of Lemma \ref{lem:compat:opt:H:opt:p:tree}, and in such a process, $s_{\ttt}^\lit$ stays constant. To this end, for $s\in \log \Z\equiv \{\log n: n\in \Z\}$,  define
	\begin{equation*}
	\begin{split}    
	    \Omega_{\sim}(s)&\equiv \{\sigma_{\sim} \in \Omega^{d}/\sim: \exists \ttt\in \FFF_{\tr}(s)\textnormal{ s.t. } \sig_{\delta v}(\ttt)\in \sigma_{\sim}\textnormal{ for some } v\in V(\ttt)\},\quad\textnormal{where}\\
	    \FFF_{\tr}(s) &\equiv \{\ttt\in \FFF_{\tr}:s_{\ttt}^{\lit}=s\}. 
	\end{split}    
	\end{equation*}
	By the observation above, $\big\{\Omega_{\sim}(s)\big\}_{s\in \log \Z}$ are disjoint, and \eqref{eq:H:dot:relation:p:tree} shows
	\begin{equation}\label{eq:H:dot:relation:p:tree:2}
	    \sum_{\sigma_{\sim}\in \Omega_{\sim}(s)}\sum_{\utau \in \sigma_{\sim}}\dot{H}(\utau)=\sum_{\ttt\in \FFF_{\tr}(s)}p_{\ttt}[\sig]v(\ttt)
	\end{equation}
	On the other hand, recalling Remark \ref{rem:compat:bdry:tree}, we have
	\begin{equation}\label{lem:H:dot:to:s:lipschitz:inter1}
	    \big|s[\sig]-s^\star_{\la}\big|=\Big|\sum_{s\in \log \Z}s\sum_{\ttt\in \FFF_{\tr}(s)}\big(p_{\ttt}[\sig]-p^\star_{\ttt,\la}\big) \Big|\leq \log 2 \sum_{s\in \log \Z}\Big|\sum_{\ttt\in \FFF_{\tr}(s)}v_{\ttt}\big(p_{\ttt}[\sig]-p^\star_{\ttt,\la}\big)\Big|,
	\end{equation}
	where the last inequality holds due to the triangle inequality and the fact that $s_{\ttt}^\lit \leq v(\ttt)\log 2$. Since the analog of \eqref{eq:H:dot:relation:p:tree} holds for $\dot{H}^\star_{\la}$ and $p^\star_{\ttt,\la}$ by Lemma \ref{lem:compat:opt:H:opt:p:tree}, \eqref{eq:H:dot:relation:p:tree:2} shows we can compute
	\begin{multline}\label{lem:H:dot:to:s:lipschitz:inter2}
	   \sum_{s\in \log \Z}\Big|\sum_{\ttt\in \FFF_{\tr}(s)}v_{\ttt}\big(p_{\ttt}[\sig]-p^\star_{\ttt,\la}\big)\Big|
	   =\sum_{s\in \log \Z}\Big|\sum_{\sigma_{\sim} \in \Omega_{\sim}(s)}\sum_{\utau \in \sigma_{\sim}}\big(\dot{H}(\utau)-\dot{H}^\star_{\la}(\utau)\big) \Big|\\
	   \leq \sum_{s\in \log \Z}\sum_{\sigma_{\sim} \in \Omega_{\sim}(s)}\sum_{\utau \in \sigma_{\sim}}\big|\dot{H}(\utau)-\dot{H}^\star_{\la}(\utau)\big|=\sum_{\utau\in \Omega^{d}:\textnormal{free}}\big|\dot{H}(\utau)-\dot{H}^\star_{\la}(\utau)\big|.
	\end{multline}
	Therefore, \eqref{lem:H:dot:to:s:lipschitz:inter1} and \eqref{lem:H:dot:to:s:lipschitz:inter2} finish the proof of \eqref{eq:lem:H:dot:to:s:lipschitz}.
	\end{proof}
	Recall the definition of the measure $\nu_{\dot{q}}\in \PPP(\Omega_{\DD})$ in \eqref{eq:def:nu:q:dot}. The next lemma shows that $\dot{q}\to\nu_{\dot{q}}$ is Lipschitz continuous in total variation distance under suitable condition.
	\begin{lemma}\label{lem:nu:q:continuous:1stmo}
	 Suppose $\dot{q}_1,\dot{q}_2\in \PPP(\dot{\Omega})$ satisfy $\dot{q}_1(\bb),\dot{q}_2(\bb) \geq C_k$, for constant $C_k>0$, which only depends on $k$. Then, there exists another constant $C_k^\prime>0$, which only depends on $k$ such that 
	\begin{equation}\label{eq:lem:nu:q:continuous:1stmo}
	    ||\nu_{\dot{q}_1}-\nu_{\dot{q}_2}||_1\leq C_k^\prime||\dot{q}_1-\dot{q}_2||_1
	\end{equation}
	\end{lemma}
	\begin{proof}
	First, we claim the bound $w_{\DD}(\sig_{\DD})^{\la}\leq 2^{\la}$: if $\sig_{\delta v}\in \{\bb,\rr\}^{d}$, we have $$w_{\DD}(\sig_{\DD})^{\la}=\prod_{e\in \delta v}\hat{\Phi}(\sig_{\delta a(e)})^{\la}\leq 1.$$ Otherwise $\{\sigma_e\}_{e\in \delta v} \subset \{\ff\}$, if we assume $w_{\DD}(\sig_{\DD})\neq 0$. Thus, we can use Lemma \ref{lem:decompose:Phi:hat} to obtain
	\begin{equation}\label{eq:compute:w:DD}
	    w_{\DD}(\sig_{\DD})^{\la}=\dot{\Phi}(\sig_{\delta v})^{\la}\prod_{e \in \delta v}\left\{\hat{z}[\hat{\sigma}_e]^{\la}\hat{v}(\sig_{\delta a(e)})\right\}=\bigg(\sum_{\bx \in \{0,1\}}\prod_{e \in \delta v}\Big\{1-\prod_{e^\prime \in \delta a(e)\backslash e}\dot{\mm}[\dot{\sigma}_{e^\prime}](\bx)\Big\}\bigg)^{\la}\prod_{e \in \delta v}\hat{v}(\sig_{\delta a(e)}),
	\end{equation}
	where the second equality is due to the definition of $\dot{\varphi}$ in \eqref{eq:def:phi} and the definition of $\hat{\mm}$ in \eqref{eq:def:bethe:bpmsg:hat}. Hence, we have $w_{\DD}(\sig_{\DD})^{\la}\leq 2^{\la}$ for all the cases.
	
	Next, we lower bound the the normalizing constant $Z_{\dot{q}}$ for $\nu_{\dot{q}}$ by using the stated bound for $\dot{q}(\bb)$. For $(\dot{\tau}_e)_{e\in \delta \DD} \in \{\bb\}^{(k-1)d}$, consider the unique coloring $\sig_{\DD}^{\fs} \equiv \sig_{\DD}^{\fs}[(\dot{\tau}_e)_{e\in \delta \DD}]\in\Omega_{\DD}$, which is valid, i.e. $w_{\DD}(\sig_{\DD}^{\fs})\neq 0$, and satisfy the following $2$ conditions:
	\begin{itemize}
	    \item For $e \in \delta \DD$, $\dot{\sigma}_e=\dot{\tau}_e$.
	    \item For $e\in \delta v$, $\hat{\sigma}_e=\fs$. Hence, $a(e)$ is a separating clause and $v$ is free variable. 
	\end{itemize}
	Using \eqref{eq:compute:w:DD}, it is straightforward to compute $w_{\DD}(\sig_{\DD}^{\fs})^{\la}=2^{\la}(1-2^{-k+2})^{d}$ for any $\sig_{\DD}^{\fs}=\sig_{\DD}^{\fs}[(\dot{\tau}_e)_{e\in \delta \DD}]$. Thus, for $\dot{q}\in \PPP(\dot{\Omega})$ with $\dot{q}(\bb)\geq C_k$, we have
	\begin{equation}\label{eq:lowerbound:Z:q:dot}
	\begin{split}
	    Z_{\dot{q}}\equiv \sum_{\sig_{\DD}\in \Omega_{\DD}}w_{\DD}(\sig_{\DD})^{\la}\prod_{e\in \delta \DD}\dot{q}(\dot{\sigma}_e)&\geq \sum_{(\dot{\tau}_e)_{e\in \delta \DD} \in \{\bb\}^{(k-1)d}} w_{\DD}\left(\sig_{\DD}^{\fs}[(\dot{\tau}_e)_{e\in \delta \DD}]\right)^{\la}\prod_{e\in \delta \DD}\dot{q}(\dot{\tau}_e)\\
	    &=2^{\la}(1-2^{-k+2})^{d}\dot{q}(\bb)^{(k-1)d}\gtrsim_{k}2^{\la}.
	\end{split}
	\end{equation}
	
	Finally, we prove our goal \eqref{eq:lem:nu:q:continuous:1stmo}: by the triangle inequality, we can bound
	\begin{equation*}
	\begin{split}
	    ||\nu_{\dot{q}_1}-\nu_{\dot{q}_2}||_1&\leq \sum_{\sig_{\DD}\in \Omega_{\DD}}\frac{w_{\DD}(\sig_{\DD})^{\la}}{Z_{\dot{q}_1}}\bigg|\prod_{e\in \delta \DD}\dot{q}_2(\dot{\sigma}_e)-\prod_{e\in \delta \DD}\dot{q}_1(\dot{\sigma}_e)\bigg|+\frac{|Z_{\dot{q}_2}-Z_{\dot{q}_1}|}{Z_{\dot{q}_1}}\\
	    &\leq \sum_{\sig_{\DD}\in \Omega_{\DD}}\frac{2w_{\DD}(\sig_{\DD})^{\la}}{Z_{\dot{q}_1}}\bigg|\prod_{e\in \delta \DD}\dot{q}_2(\dot{\sigma}_e)-\prod_{e\in \delta \DD}\dot{q}_1(\dot{\sigma}_e)\bigg|\lesssim_{k}\sum_{\sig_{\DD}\in \Omega_{\DD}}\bigg|\prod_{e\in \delta \DD}\dot{q}_2(\dot{\sigma}_e)-\prod_{e\in \delta \DD}\dot{q}_1(\dot{\sigma}_e)\bigg|,
	\end{split}
	\end{equation*}
	where the final inequality is due to the bounds $w_{\DD}(\sig_{\DD})^{\la}\leq 2^{\la}$ and \eqref{eq:lowerbound:Z:q:dot}. Using the triangle inequality once more on the \textsc{rhs} of the equation above, we have
	\begin{equation*}
	  ||\nu_{\dot{q}_1}-\nu_{\dot{q}_2}||_1\lesssim_{k}\sum_{\sig_{\DD}\in \Omega_{\DD}}\bigg|\prod_{e\in \delta \DD}\dot{q}_2(\dot{\sigma}_e)-\prod_{e\in \delta \DD}\dot{q}_1(\dot{\sigma}_e)\bigg|\leq (k-1)d||\dot{q}_2-\dot{q}_1||_1,
	\end{equation*}
	which concludes the proof.
	\end{proof}
	The next lemma plays an important role in proving Lemma \ref{lem:exists:qdot:hdot:exp:tail}.
	\begin{lemma}\label{lem:hdot:exptail:q:b:lowerbound}
	Consider $\dot{h}\in \PPP(\dot{\Omega}_L)$, which satisfies $\dot{h}(\rr)\vee \dot{h}(\ff)\leq \frac{10}{2^k}$ and denote $\dot{q}=\dot{q}_L[\dot{h}]$. Then, there exists a constant $C_k>0$, which only depends on $k$, such that $\dot{q}(\bb) \geq C_k$.
	\end{lemma}
	\begin{proof}
	We first show $\dot{q}(\bb)\geq \dot{q}(\ff)$ by crude estimates: suppose by contradiction that $\dot{q}(\bb)<\dot{q}(\ff)$ holds. Denote $\nu=\nu_{\dot{q}}\in \PPP(\Omega_{\DD})$, where $\nu_{\dot{q}}$ is defined in \eqref{eq:def:nu:q:dot}. Also, denote by $\E_{\nu}$ the expectation taken with respect to $\nu$. Since $\dot{h}=\dot{h}\left[H^\tr[\nu]\right]$,
	\begin{equation}\label{eq:compute:dot:h:ff:by:nu}
	    \E_{\nu}\bigg[\sum_{e\in \delta \DD} \one\{\dot{\sigma}_e\in \ff\}\bigg]=\dot{h}(\ff)(k-1)d \leq \frac{10kd}{2^k}\leq 10k^2.
	\end{equation}
	To compute the \textsc{lhs} of the equation above, label the clauses in $\DD$ by $a_1,..,a_{d}$ and let $e_i=(a_i v), 1\leq i \leq d$, where $v$ is the unique variable in $\DD$. We divide cases into where $v$ is free or frozen: for the case where $v$ is free, we condition on the number of non-separating clauses in $\DD$ and the spins adjacent to them. Fix $1\leq i_1<i_2<...<i_{\ell}\leq d$ and $\utau_1,...,\utau_\ell \in \Omega_L^{k}$, which are non-separating. Then, since non-separating clauses have at least $2$ free spins adjacent to them and $\sig_{a_i}, 1\leq i \leq d$ are independent conditional on $\sig_{\delta v}$,
	\begin{equation}\label{eq:conditional:v:free:long}
	\begin{split}
	    &\E_{\nu}\bigg[\sum_{e\in \delta \DD} \one\{\dot{\sigma}_e\in \ff\}\bigg| \sig_{a_{i_j}}=\utau_{j}, 1\leq j \leq \ell,\textnormal{ and }\hat{\sigma}_{e_i}=\fs\textnormal{ if } i \notin \{i_1,...,i_{\ell}\}\bigg]\\
	&\geq \ell+\sum_{i \notin \{i_1,...,i_{\ell}\}}\E_{\nu}\bigg[\sum_{e\in \delta a_i\backslash e_i}\one\{\dot{\sigma}_e\in \ff\}\bigg| \hat{\sigma}_{e_i}=\fs\bigg]=\ell+(d-\ell)\frac{\sum_{j=0}^{k-3}j\binom{k-1}{j}(1-2^{-k+j+2})\dot{q}(\ff)^{j}\dot{q}(\bb)^{k-1-j}}{\sum_{j=0}^{k-3}\binom{k-1}{j}(1-2^{-k+j+2})\dot{q}(\ff)^{j}\dot{q}(\bb)^{k-1-j}}\geq d,
	\end{split}
	\end{equation}
	where the last inequality holds because we assumed $\dot{q}(\ff)>\dot{q}(\bb)$. Thus, we have
	\begin{equation}\label{eq:conditional:v:free:final}
	    \E_{\nu}\bigg[\sum_{e\in \delta \DD} \one\{\dot{\sigma}_e\in \ff\}\bigg|\textnormal{$v$ is free}\bigg]\geq d
	\end{equation}
	Turning to the case where $v$ is frozen, let $X_{\rr}^{\delta}$ be the number of clauses in $\DD$ which have a red edge among $\delta \DD$ and let $X_{\rr}^{\textnormal{in}}$ be the number of red edges among $e_1,...,e_d$. By Markov's inequality,
	\begin{equation}\label{eq:X:rr:delta:small}
	    \P_{\nu}\left(X_{\rr}^{\delta}\geq \frac{d}{3}\right)\leq \frac{3\E_{\nu}[X_{\rr}^{\delta}]}{d}=3(k-1)\dot{h}(\rr)\leq \frac{30k}{2^k}.
	\end{equation}
	Note that $v$ is frozen if and only if $X_{\rr}^{\textnormal{in}}\geq 1$, so again by Markov's inequality,
	\begin{equation}\label{eq:X:rr:in:intermediate}
	    \P_{\nu}\bigg(X_{\rr}^{\textnormal{in}}\geq \frac{d}{3} \bigg|\textnormal{$v$ is frozen}\bigg)\leq \frac{3\E_{\nu}[X_{\rr}^{\textnormal{in}}\mid X_{\rr}^{\textnormal{in}}\geq 1]}{d}=\frac{3}{d}\frac{\sum_{j=1}^{d}j\binom{d}{j}\left(2^{-k+1}\dot{q}(\bb)^{k-1}\right)^{j}A^{d-j}}{\left(2^{-k+1}\dot{q}(\bb)^{k-1}+A\right)^{d}-A^{d}},
	\end{equation}
	where $A\equiv \sum_{\sig\in \Omega_L^{k}, \sigma_1\in \{\bb\}}\hat{\Phi}(\sig)^{\la}\prod_{i=2}^{k}\dot{q}(\dsigma_i)$. Since $\hat{v}(\sig)\geq 1/2$ for valid separating $\sig$, we can lower bound $A$ by the contribution from separating $\sig$ as
	\begin{equation*}
	    A\geq \frac{1}{2}\left(\left(\dot{q}(\bb)+\dot{q}(\ff)\right)^{k-1}-\dot{q}(\ff)^{k-1}\right)\geq \frac{2^{k-1}-1}{2}\dot{q}(\bb)^{k-1}\geq 2^{k-3}\dot{q}(\bb)^{k-1}
	\end{equation*}
	Hence, we can use the inequality above to further bound the \textsc{rhs} of \eqref{eq:X:rr:in:intermediate} by
	\begin{equation}\label{eq:X:rr:in:small}
	\begin{split}
	    \P_{\nu}\bigg(X_{\rr}^{\textnormal{in}}\geq \frac{d}{3} \bigg|\textnormal{$v$ is frozen}\bigg)
	    &\leq 3\frac{2^{-k+1}\dot{q}(\bb)^{k-1}\left(2^{-k+1}\dot{q}(\bb)^{k-1}+A\right)^{d-1}}{\left(2^{-k+1}\dot{q}(\bb)^{k-1}+A\right)^{d}-A^{d}}\leq\frac{3}{d}\left(\frac{2^{-k+1}\dot{q}(\bb)^{k-1}+A}{A}\right)^{d-1}\\
	    &\lesssim \frac{1}{d}\lesssim \frac{1}{k2^k}
	\end{split}
	\end{equation}
	Having \eqref{eq:X:rr:delta:small} and \eqref{eq:X:rr:in:small} in mind, we condition on the event where $X_{\rr}^{\delta}=\ell_1\leq \frac{d}{3}$ and $X_{\rr}^{\textnormal{in}}=\ell_2\in [1,\frac{d}{3}]$. Similar to the calculations done in \eqref{eq:conditional:v:free:long}, we can lower bound
	\begin{equation}\label{eq:conditional:v:frozen:final}
	\begin{split}
	    &\E_{\nu}\bigg[\sum_{e\in \delta \DD}\one\{\dsigma_e\in \ff\}\bigg| X_{\rr}^{\delta}=\ell_1,X_{\rr}^{\textnormal{in}}=\ell_2\bigg]\\
	    &\geq (d-\ell_1-\ell_2)\frac{\sum_{\sig\in\{\bb,\ff\}^{k},\sigma_1\in \{\bb\}}\hat{\Phi}(\sig)^{\la}\prod_{i=2}\dot{q}(\dsigma_i)\sum_{j=2}^{k}\one\{\dsigma_i\in \{\ff\}\}}{\sum_{\sig\in\{\bb,\ff\}^{k},\sigma_1\in \{\bb\}}\hat{\Phi}(\sig)^{\la}\prod_{i=2}^{k}\dot{q}(\dsigma_i)}\geq \frac{d}{3},
	\end{split}
	\end{equation}
	where the last inequality holds because $\dot{q}(\ff)>\dot{q}(\bb)$ and $\ell_1,\ell_2\leq \frac{d}{3}$. Therefore, \eqref{eq:conditional:v:free:final} and \eqref{eq:conditional:v:frozen:final} show
	\begin{equation}\label{eq:nu:free:too:many}
	    \E_{\nu}\bigg[\sum_{e\in \delta \DD} \one\{\dot{\sigma}_e\in \ff\}\bigg]\geq \frac{d}{3}\P_{\nu}\left(X_{\rr}^{\delta}\leq \frac{d}{3}, X_{\rr}^{\textnormal{in}}\leq \frac{d}{3}\right)\geq \frac{d}{3}\left(1-\frac{C k}{2^k}\right),
	\end{equation}
	where the last inequality is due to \eqref{eq:X:rr:delta:small}, \eqref{eq:X:rr:in:small} and $\P_{\nu}(X_{\rr}^{\delta}\geq \frac{d}{3})\leq \P_{\nu}(X_{\rr}^{\delta}\geq \frac{d}{3}\mid\textnormal{$v$ is frozen})$. Hence, in the regime of $d\geq k 2^{k}$, \eqref{eq:compute:dot:h:ff:by:nu} contradicts \eqref{eq:nu:free:too:many} for large $k$, so we conclude that $\dot{q}(\bb)\geq \dot{q}(\ff)$.
	
	Next, we show by rough estimates that $\dot{q}(\rr)\leq 2^{5k}\dot{q}(\bb)$ holds. Suppose by contradiction that $\dot{q}(\rr)\geq 2^{5k}\dot{q}(\bb)$ holds. Recalling \eqref{eq:X:rr:delta:small}, we have
	\begin{equation}\label{eq:upperbound:prob:X:rr:delta:close:d}
	    \frac{k}{2^k}\gtrsim \P_{\nu}\left(X_{\rr}^{\delta}=d-1\right)=\frac{d(k-1)^{d-1}2^{-(k-1)d}\dot{q}(\rr)^{d-1}\dot{q}(\bb)^{(k-2)d+1}}{\sum_{\sig\in \Omega_L^{k}}w_{\DD}(\sig_{\DD})^{\la}\prod_{e\in \delta \DD}\dot{q}(\dsigma_e)}
	\end{equation}
	We now upper bound the denominator in the \textsc{rhs} of the equation above by specifying the number of $\rr$ edge in $\delta\DD$. Recalling the fact $w_{\DD}(\sig_{\DD})^{\la}\leq 2^{\la}$ from the proof of Lemma \ref{lem:nu:q:continuous:1stmo}, for $0\leq \ell \leq d-1$,
	\begin{equation*}
	\sum_{\substack{\sig\in \Omega_L^{k}\\ |\{e\in \delta \DD:\dsigma_e\in \{\rr\}\}|=\ell}}w_{\DD}(\sig_{\DD})^{\la}\prod_{e\in \delta \DD}\dot{q}(\dsigma_e)\leq 2^{\la} \binom{d}{\ell}\left((k-1)2^{-k+1}\dot{q}(\rr)\dot{q}(\bb)^{k-1}\right)^{\ell}\left(\dot{q}(\bb)+\dot{q}(\ff)\right)^{(d-\ell)(k-1)}.
	\end{equation*}
	Hence, using the bound $\dot{q}(\bb)\geq \dot{q}(\ff)$ and $\binom{d}{\ell}\leq d^{d-\ell}$, we have
	\begin{equation*}
	\begin{split}
	\frac{\sum_{\sig\in \Omega_L^{k}}w_{\DD}(\sig_{\DD})^{\la}\prod_{e\in \delta \DD}\dot{q}(\dsigma_e)}{d(k-1)^{d-1}2^{-(k-1)d}\dot{q}(\rr)^{d-1}\dot{q}(\bb)^{(k-2)d+1}}
	&\leq 2\sum_{\ell=1}^{d-1}\frac{2^{2(k-1)(d-\ell)}d^{d-\ell-1}}{(k-1)^{d-\ell-1}}\left(\frac{\dot{q}(\bb)}{\dot{q}(\rr)}\right)^{d-\ell-1}\\
	&\leq 2^{2k}\sum_{\ell=0}^{d-1}\left(\frac{2^{2k-2}d\dot{q}(\bb)}{(k-1)\dot{q}(\rr)}\right)^{\ell} \lesssim \frac{d}{k2^{k}}\lesssim 1,
	\end{split}
	\end{equation*}
	which contradicts \eqref{eq:upperbound:prob:X:rr:delta:close:d} for large $k$. Therefore, we conclude that $\dot{q}(\rr)\leq 2^{5k}\dot{q}(\bb)$ holds, which together with $\dot{q}(\bb)\geq \dot{q}(\ff)$ shows $\dot{q}(\bb)\gtrsim 2^{-5k}$.
	\end{proof}
	Having Lemma \ref{lem:hdot:exptail:q:b:lowerbound} in hand, we prove the following Lemma, which implies Lemma \ref{lem:exists:qdot:hdot:exp:tail}.
	\begin{lemma}\label{lem:exists:qdot:hdot:exp:tail:improved}
	Suppose $\dot{h}\in \PPP(\dot{\Omega})$ satisfies $\dot{h}(\rr)\vee \dot{h}(\ff)\leq \frac{9}{2^k}$ and $\sum_{\dsigma: v(\dsigma)\geq L}\dot{h}(\dsigma)\leq 2^{-ckL}$ for all $L\geq 1$, where $c>0$ is an absolute constant. Then, there exists a unique $\dot{q}\equiv \dot{q}[\dot{h}]\in \PPP(\dot{\Omega})$ such that $\dot{h}_{\dot{q}}=\dot{h}$. Moreover, there exists a constant $C_k$ and $C_k^\prime$ such that $\dot{q}(\bb) \geq C_k$ and $\sum_{v(\dsigma)\geq L}\dot{q}(\dsigma)\leq C_k^\prime 2^{-ckL}$.
	\end{lemma}
	\begin{proof}
	Define $\dot{h}_L\in \PPP(\dot{\Omega}_L)$, the $L$-truncated version of $\dot{h}$, as follows.
	\begin{equation*}
	    \dot{h}_L(\dsigma) \equiv \frac{\dot{h}(\dsigma)\one\{v(\dsigma)\leq L\}}{\sum_{\dot{\tau}\in \dot{\Omega}_L}\dot{h}(\dot{\tau})}.
	\end{equation*}
	Also, denote $\dot{q}_L \equiv \dot{q}_L[\dot{h}_L]$. We first argue that $\left\{\dot{q}_L\right\}_{L\geq 1}$ is tight. Consider $L$ large enough so that $\dot{h}_L(\dsigma)\leq \frac{10}{9}\dot{h}(\dsigma)$ holds for all $\dsigma \in \dot{\Omega}$. In particular, $\dot{h}_L(\rr) \vee \dot{h}_L(\ff) \leq \frac{10}{2^k}$, so Lemma \ref{lem:hdot:exptail:q:b:lowerbound} implies that $\dot{q}_L(\bb)\geq C_k>0$ for all $L$ large enough. Fix $T\leq L$ and denote $\nu_L = \nu_{\dot{q}_L}$. Then, since $\dot{h}_L = \dot{h}[\nu_L]$,
	\begin{equation*}
	    \frac{10}{9}2^{-ckT}\geq \sum_{v(\dot{\tau})\geq T}\dot{h}_L(\dot{\tau}) \geq \sum_{v(\dot{\tau})\geq T} \sum_{\sig_{\DD}\in \Omega_{\DD}}\nu_L[\sig_{\DD}]\one\left\{\dsigma_1=\dot{\tau}, \dsigma_2,...,\dsigma_{(k-1)d}\in \{\bb\}\right\},
	\end{equation*}
	where we identified $\delta \DD\equiv \{1,2,...,(k-1)d\}$. Similar to $\sig_{\DD}^{\fs}[(\dot{\tau}_e)_{e\in \delta \DD}]$ considered in the proof of Lemma \ref{lem:nu:q:continuous:1stmo}, we can consider $\sig_{\DD}$ with $\hat{\sigma}_e=\fs$ for $e\in \delta v$ to further lower bound the \textsc{rhs} of the equation above by
	\begin{equation*}
	     \frac{10}{9}2^{-ckT}\geq (Z_{\dot{q}_L})^{-1}\sum_{v(\dot{\tau})\geq T}2^{\la}(1-2^{-k+2})^{d-1}(1-2^{-k+3})\dot{q}_L(\bb)^{(k-1)d-1}\dot{q}_L(\dot{\tau}),
	\end{equation*}
	where $Z_{\dot{q}_L}$ is the normalizing constant for $\nu_L$. Using the fact $w_{\DD}(\sig_{\DD})^\la \leq 2^{\la}$, it is straightforward to upper bound $Z_{\dot{q}_L}\leq 2^{\la}$, so the equation above and $\dot{q}_L(\bb)\gtrsim_k 1$ show
	\begin{equation}\label{eq:dot:q:L:exp:tight}
	    \sum_{v(\dot{\tau})\geq T}\dot{q}_L(\dot{\tau})\lesssim_{k}2^{-ckT}.
	\end{equation}
	Thus, $\{\dot{q}_L\}_{L\geq 1}$ is tight, so by Prokhorov's theorem, there exists a subsequence $\{L_i\}_{i\geq 1}$ and $\dot{q}\in \PPP(\dot{\Omega})$ such that $\dot{q}_{L_i}$ converges to $\dot{q}$ in total variation distance. In particular, $\dot{q}(\bb)\geq C_k$, where $C_k$ is the constant from Lemma \ref{lem:hdot:exptail:q:b:lowerbound}, and \eqref{eq:dot:q:L:exp:tight} shows that $\sum_{v(\dot{\tau})\geq T}\dot{q}(\dot{\tau})\lesssim_{k}2^{-ckT}$ for all $T\geq 1$. We now argue that $\dot{h}_{\dot{q}}=\dot{h}$. Note that $\dot{h}_{\dot{q}}=\dot{h}\left[H^\tr[\nu_{\dot{q}}]\right]$ holds by definition, and $\nu \to \dot{h}\left[H^\tr[\nu]\right]$ is a linear projection. Hence, Lemma \ref{lem:nu:q:continuous:1stmo} shows 
	\begin{equation}\label{eq:bound:dot:h:by:dot:q}
	    ||\dot{h}_{\dot{q}}-\dot{h}_L||_1\lesssim ||\nu_{\dot{q}}-\nu_{\dot{q}_L}||_1\lesssim_k ||\dot{q}-\dot{q}_L||_1.
	\end{equation}
	Therefore, $\lim_{L\to\infty}||\dot{h}_{\dot{q}}-\dot{h}_L||_1=0$ and since $\lim_{L\to\infty}||\dot{h}_L-\dot{h}||_1=0$ by the exponential decay of the tail of $\dot{h}$, we conclude that $\dot{h}_{\dot{q}}=\dot{h}$ holds.
	
	What remains to be proven is the uniqueness of $\dot{q}$ satisfying $\dot{h}_{\dot{q}}=\dot{h}$. Suppose we have $\dot{h}_{\dot{q}_1}=\dot{h}_{\dot{q}_2}=\dot{h}$. Then for both $i=1,2$, $\nu_{\dot{q}_i}$ achieves the supremum in \eqref{eq:express:Lambda:op:sup}, since for any $\nu\in \PPP(\Omega_{\DD})$ with $\dot{h}\left[H^\tr[\nu]\right]=\dot{h}$, 
	\begin{equation*}
	\HH(\nu_{\dot{q}_i})+\la\left\langle \log w_{\DD}, \nu_{\dot{q}_i} \right\rangle - \HH(\nu)-\la\left\langle \log w_{\DD}, \nu\right\rangle =\DD_{\textnormal{KL}}(\nu||\nu_{\dot{q}_i})\geq 0.
	\end{equation*}
	On the other hand, the optimization in \eqref{eq:express:Lambda:op:sup} with respect to $\nu$ is strictly concave, so there exists a unique maximizer. Thus, $\nu_{\dot{q}_1}=\nu_{\dot{q}_2}$. Also, $\dot{q}_1(\bb),\dot{q}_2(\bb)>0$, since otherwise $\dot{h}(\bb)=0$. Having $\nu_{\dot{q}_1}=\nu_{\dot{q}_2}$ with $\dot{q}_1(\bb),\dot{q}_2(\bb)>0$, it is straightforward to see that $\dot{q}_1=\dot{q}_2$, which concludes the proof.
	\end{proof}
	\begin{lemma}\label{lem:Xi:continuous:1stmo}
	Recall the definition of $\bDelta^{\textnormal{exp}}_{C}$ in \eqref{eq:def:bDelta:exp:decay} and endow $\bDelta^{\textnormal{exp}}_{C}$ with topology induced by total variation distance. Then for any $C>0$, $\bXi:\bDelta^{\textnormal{exp}}_{C}\to \R_{\geq 0}$ is continuous.
	\end{lemma}
	\begin{proof}
	Note that $\bLa(H)$ is continuous from its definition, so it suffices to prove that $\dot{h}\to \bLa^\op(\dot{h})$ is continuous among $\dot{h}$ satisfying $\sum_{v(\dsigma)\geq L}\dot{h}(\dsigma)\leq 2^{-CkL}, L \geq 1$ and $\dot{h}(\rr)\vee\dot{h}(\ff) \leq \frac{9}{2^k}$. 
	
	Suppose $\{\dot{h}_n\}_{n\geq 1}$ satisfy such conditions with $\lim_{n\to\infty}||\dot{h}_n-\dot{h}||_1=0$. Denote $\dot{q}_n=\dot{q}[\dot{h}_n]$ and $\dot{q}=\dot{q}[\dot{h}]$ whose existence is guaranteed by Lemma \ref{lem:exists:qdot:hdot:exp:tail:improved}. We first show that $\lim_{n\to\infty}||\dot{q}_n-\dot{q}||_1=0$: note that Lemma \ref{lem:exists:qdot:hdot:exp:tail:improved} again shows that for a constant $C_k,C_k^\prime>0$,
	\begin{equation}\label{eq:q:n:tail:small}
	    \dot{q}(\bb) \geq C_k\quad\textnormal{ and }\sum_{v(\dsigma)\geq L}\dot{q}_n(\dsigma) \leq C_{k}^\prime 2^{-CkL}\textnormal{ for all }L \geq 1.
	\end{equation}
	Thus, any subsequence of $\{\dot{q}_n\}_{n\geq 1}$ admits a further subsequence converging to some limit $\dot{q}^\prime$ by Prokhorov's theorem. By the same argument as done in \eqref{eq:bound:dot:h:by:dot:q}, $\dot{h}_{\dot{q}^\prime}=\dot{h}=\dot{h}_{\dot{q}}$ holds, so the uniqueness of such $\dot{q}$ guaranteed by Lemma~\ref{lem:exists:qdot:hdot:exp:tail:improved} shows that $\dot{q}^\prime=\dot{q}$ holds. Therefore, $\lim_{n\to\infty}||\dot{q}_n-\dot{q}||_1=0$.
	
	Now, we aim to prove our goal $\lim_{n\to\infty}\bLa^\op(\dot{h}_n)=\bLa^\op(\dot{h})$. It is straightforward to compute
	\begin{equation*}
	    \bLa^\op(\dot{h}_n)=\log Z_{\dot{q}_n} -\langle \dot{h}_n, \log \dot{q}_n \rangle,
	\end{equation*}
	where $Z_{\dot{q}_n}$ is the normalizing constant for $\nu_{\dot{q}_n}$. Also, it is straightforward to see from $w_{\DD}(\sig_{\DD})^\la \leq 2^{\la}$ that $\lim_{n\to\infty}Z_{\dot{q}_n}=Z_{\dot{q}}$ holds. To this end, we aim to prove $\lim_{n\to \infty} \langle \dot{h}_n,\log \dot{q}_n \rangle = \langle \dot{h},\log \dot{q}\rangle$ for the rest of the proof. Denote $\nu_n =\nu_{\dot{q}_n}$. Then, $\dot{h}\left[H^\tr[\nu_n]\right]=\dot{h}_n$, so 
	\begin{equation*}
	    \langle \dot{h}_n, \log \dot{q}_n \rangle = (Z_{\dot{q}_n})^{-1}\sum_{\sig_{\DD}\in \Omega_{\DD}}w_{\DD}(\sig_{\DD})^{\la}\dot{q}_n(\dsigma_1)\log \dot{q}_n(\dsigma_1)\prod_{i=2}^{(k-1)d}\dot{q}_n(\dsigma_i),
	\end{equation*}
	where we identified $\delta \DD\equiv \{1,2,...,(k-1)d\}$. The analog holds for $\dot{q}$, so it suffices to show
	\begin{equation*}
	    \lim_{n\to \infty}\sum_{\sig_{\DD}\in \Omega_{\DD}}w_{\DD}(\sig_{\DD})^{\la}\dot{q}_n(\dsigma_1)\log \dot{q}_n(\dsigma_1)\prod_{i=2}^{(k-1)d}\dot{q}_n(\dsigma_i)=\sum_{\sig_{\DD}\in \Omega_{\DD}}w_{\DD}(\sig_{\DD})^{\la}\dot{q}(\dsigma_1)\log \dot{q}(\dsigma_1)\prod_{i=2}^{(k-1)d}\dot{q}(\dsigma_i).
	\end{equation*}
	Observe that by the triangle inequality and the bound $w_{\DD}(\sig_{\DD})^{\la}\leq 2^{\la}$, we have
	\begin{equation*}
	\begin{split}
	    &\bigg|\sum_{\sig_{\DD}\in \Omega_{\DD}}w_{\DD}(\sig_{\DD})^{\la}\dot{q}_n(\dsigma_1)\log \dot{q}_n(\dsigma_1)\prod_{i=2}^{(k-1)d}\dot{q}_n(\dsigma_i)-\sum_{\sig_{\DD}\in \Omega_{\DD}}w_{\DD}(\sig_{\DD})^{\la}\dot{q}(\dsigma_1)\log \dot{q}(\dsigma_1)\prod_{i=2}^{(k-1)d}\dot{q}(\dsigma_i)\bigg|\\
	    &\lesssim ||\dot{q}_n \log \dot{q}_n -\dot{q} \log \dot{q}||_1+(kd-d-1)||\dot{q}\log\dot{q}||_1||\dot{q}_n-\dot{q}||_1,
	\end{split}
	\end{equation*}
	where we abbreviated $\dot{q}\log\dot{q}\equiv \{\dot{q}(\dsigma)\log\dot{q}(\dsigma)\}_{\dsigma\in \dot{\Omega}}$. Therefore, we now aim to prove
	\begin{equation}\label{eq:q:logq:convergent}
	   \lim_{n\to \infty}\sum_{\dsigma \in \dot{\Omega}}\big|\dot{q}_n(\dsigma)\log \dot{q}_n(\dsigma)-\dot{q}(\dsigma)\log \dot{q}(\dsigma)\big|=0\textnormal{ and } \sum_{\dsigma \in \dot{\Omega}}\dot{q}(\dsigma)\log\dot{q}(\dsigma)<\infty
	\end{equation}
	To prove the equation above, note that $x \to x^2\log(x^2)$ has bounded derivative in $[0,1]$, so 
	\begin{equation}\label{eq:bound:q:logq:by:sqrt:q}
	    ||\dot{q}_n \log \dot{q}_n -\dot{q} \log \dot{q}||_1 \lesssim \sum_{\dsigma \in \dot{\Omega}}\big|\sqrt{\dot{q}_n(\dsigma)}-\sqrt{\dot{q}(\dsigma)}\big|\textnormal{ and }||\dot{q}\log\dot{q}||_1\lesssim \sum_{\dsigma \in \dot{\Omega}}\sqrt{\dot{q}(\dsigma)}
	\end{equation}
	Observe that using Cauchy Schwartz, we have the following tail estimates:
	\begin{equation}\label{eq:tail:sqrt:q:n:intermediate}
	    \sum_{v(\dsigma)\geq L} \sqrt{\dot{q}(\dsigma)}\leq \sum_{T=L}^{\infty}\Big(\sum_{v(\dsigma)=T}\dot{q}(\dsigma)\Big)^{1/2}\Big|\{\dsigma:v(\dsigma)=T\}\Big|^{1/2}\lesssim_{k}\sum_{T=L}^{\infty} 2^{-CkT/2}\Big|\{\dsigma:v(\dsigma)=T\}\Big|^{1/2},
	\end{equation}
	where the last inequality is due to \eqref{eq:q:n:tail:small}. Note that we can upper bound $|\{\dsigma:v(\dsigma)=T\}|$ as follows. $\dsigma \in \dot{\Omega}$ is fully determined by specifying the underlying graph and the color of the clause-adjacent boundary half-edges, either $\bb_0$ or $\bb_1$. If $v(\dsigma)=T$, then $f(\dsigma)\leq T$, where $f(\dsigma)$ is the number of clauses in $\dsigma$, because each clause has internal degree at least $2$ in the tree $\dsigma$. The number of isomorphism class of graphs with $K$ vertices is at most $4^K$(see \cite[Section 7.5]{FlajoletSedgewick}), so we can bound $ \Big|\{\dsigma:v(\dsigma)=T\}\Big|\leq 4^{2T}2^{T}=32^T$.
	Plugging in this bound to \eqref{eq:tail:sqrt:q:n:intermediate} shows
	\begin{equation}\label{eq:tail:sqrt:q:n:final}
	    \sum_{v(\dsigma)\geq L} \sqrt{\dot{q}(\dsigma)}\lesssim_{k}\sum_{T=L}^{\infty} 2^{-(\frac{Ck}{2}-32\log 2)T}\lesssim_{k}2^{-C^\prime kL},
	\end{equation}
	where we assumed $k$ is large enough. Hence, the second claim of \eqref{eq:q:logq:convergent} holds. Also, the analog also holds for $\dot{q}_n$. Thus,
	\begin{equation*}
	    \limsup_{n \to \infty}\sum_{\dsigma\in \dot{\Omega}}\big|\sqrt{\dot{q}_n(\dsigma)}-\sqrt{\dot{q}(\dsigma)}\big|\leq C_k 2^{-C^\prime k L}+\limsup_{n \to \infty}\sum_{\dsigma:v(\dsigma)\leq L}\big|\sqrt{\dot{q}_n(\dsigma)}-\sqrt{\dot{q}(\dsigma)}\big|=C_k 2^{-C^\prime k L},
	\end{equation*}
	and sending $L\to \infty$ shows $\lim_{n\to\infty}\sum_{\dsigma\in \dot{\Omega}}\big|\sqrt{\dot{q}_n(\dsigma)}-\sqrt{\dot{q}(\dsigma)}\big|=0$. Therefore, together with \eqref{eq:bound:q:logq:by:sqrt:q}, this finishes the proof of \eqref{eq:q:logq:convergent}.
	\end{proof}
	The next lemma gives some estimates on the values of BP fixed point $\dot{q}^\star_{\la,L}$, which can be read off from \cite[Appendix A]{ssz22}. It will be important for Proposition \ref{prop:1stmo:Lipschitz:hdot:qdot} below.
		\begin{lemma}\label{lem:1stmo:BPfixedpoint:estimates}
		For some absolute constant $C>0$, the following holds for $\la\in [0,1]$ and $L \geq 1$:
		\begin{enumerate}
		\item $\dot{q}^{\star}_{\la,L}(\rr) \in (\frac{1}{2},\frac{1}{2}+\frac{C}{2^k}]$.
		\item $\dot{q}^{\star}_{\la,L}(\bb) \in [\frac{1}{2}-\frac{C}{2^k},\frac{1}{2})$.
		\item $\dot{q}^{\star}_{\la,L}(\ff) \leq \frac{C}{2^k}$.
		\end{enumerate}
		
	\end{lemma}
		\begin{proof}
	   The lemma follows as a consequence of computations done in \cite{ssz22}. From Proposition \ref{prop:BPcontraction:1stmo}, recall that $\dot{q}^{\star}_{\la, L} \in \mathbf{\Gamma}_{C'}$ for some absolute constant $C'>0$ and $\dot{q}^{\star}_{\la, L} = \BP[\dot{q}^{\star}_{\la, L}]$. Thus, it suffices to obtain the desired conclusion for $\BP[\dot{q}^{\star}_{\la, L}]$. From \cite[Lemma A.4]{ssz22}, there exists an absolute constant $C>0$ such that
	   \begin{equation*}
	   \frac{1}{2} - \frac{C}{2^k} \le \dot{q}^{\star}_{\la, L} (\rr) , \dot{q}^{\star}_{\la, L}(\bb) \le \frac{1}{2} + \frac{C}{2^k} , \quad \dot{q}^{\star}_{\lambda, L} (\ff) \le \frac{C}{2^k}.
	   \end{equation*}
	   
	   To obtain the conclusion, it suffices to show that $\dot{q}^{\star}_{\lambda, L} \ge \frac{1}{2}$. Recall Definition \ref{def:msg:m}, and view $\dot{\mm}[\dot{\sigma}]$ for a coloring spin $\dot \sigma$  as the definition using the equivalence of the message configurations and the colorings. We write
	   \begin{equation*}
[	   \dot{\mm}^\lambda \dot{q}](\ff) := \sum_{\dot{\sigma} \in \ff} \dot{\mm}[\dot{\sigma}](1)^\lambda \dot{q}(\dot{\sigma}).
	   \end{equation*}
	   Then, from the definition of $\BP$, one can observe that a fixed point $\dot{q}$ of $\BP$ must satisfy
	   \begin{equation*}
	   \dot{q}(\rr) = \dot{q}(\bb) + [ \dot{\mm}^\lambda \dot{q}](\ff) + [(1-\dot{\mm})^\lambda \dot{q}](\ff),
	   \end{equation*}
	   which is the assumption of \cite[Lemma B.2]{ssz22}. (For details, we refer to Appendices A and B of \cite{ssz22}.) Then, we can conculde the proof from the fact that $ [ \dot{\mm}^\lambda \dot{q}](\ff) + [(1-\dot{\mm})^\lambda \dot{q}](\ff) \ge \dot{q}(\ff)$.
	\end{proof}
	The next proposition played a crucial role in the proof of Lemma \ref{lem:1stmo:negdef:onestep} and Proposition \ref{prop:negdef}.
	\begin{prop}
	\label{prop:1stmo:Lipschitz:hdot:qdot}
	Fix $k\geq k_{0}$. Recall that for $\dot{h}\in \PPP(\dot{\Omega}_L)$, $\dot{q}_L[\dot{h}]\equiv \dot{q}_{\la,L}[\dot{h}]\in \PPP(\dot{\Omega}_L)$ is determined by \eqref{eq:def:nu:q:dot} with inverse function $\dot{q} \to \dot{h}_{\dot{q}}$ in \eqref{eq:1stmo:hdot:intermsof:qdot}, and denote $\dot{h}^\star_{L}\equiv \dot{h}^\star_{\la,L}, \dot{q}^\star_{L} \equiv \dot{q}^\star_{\la,L}$. Then, there exists $\eps_L>0$ and a constant $C_k$, which may depend on $k$ but not on $L$, such that
	\begin{equation}\label{eq:lem:1stmo:Lipschitz:hdot:qdot}
	    ||\dot{h}-\dot{h}^\star_{L}||_1<\eps_L, \dot{h}\in \PPP(\dot{\Omega}_L) \implies ||\dot{q}_L[\dot{h}]-\dot{q}^\star_L||_1\leq C_k ||\dot{h}-\dot{h}^\star_L||_1
	\end{equation}
	\end{prop}
	\begin{proof}
	Throughout the proof, we denote by $C>0$ a universal constant. Lemma \ref{lem:1stmo:BPfixedpoint:estimates} in Appendix \ref{sec:appendix:Properties of BP fixed point} shows that $\dot{q}^\star_{L}(\rr)=\frac{1}{2}+O(\frac{1}{2^k}), \dot{q}^\star_{L}(\bb)=\frac{1}{2}-O(\frac{1}{2^k})$ and $\dot{q}^\star_{L}(\ff)=O(\frac{1}{2^k})$. Note that $\dot{h}\to\dot{q}_L[\dot{h}]$ is continuous(cf. \cite[Appendix C]{ssz22}), so we take $\eps_L>0$ small enough so that the following holds for all $||\dot{h}-\dot{h}^\star_L||_1<\eps_L$:
	\begin{itemize}
	    \item $\dot{q}_L[\dot{h}]^{\textnormal{av}}\in \Gamma\equiv \Gamma_{C}$, where $\Gamma_{C}$ is defined in \eqref{eq:BPcontraction:1stmo}. Here, $\dot{q}^{\textnormal{av}}\in \PPP(\dot{\Omega}_L)$ is defined by $\dot{q}^{\textnormal{av}}(\dot{\sigma})\equiv \frac{\dot{q}(\dot{\sigma})+\dot{q}(\dot{\sigma}\oplus 1)}{2}, \dot{\sigma} \in \dot{\Omega}_L$. Hence, by Proposition \ref{prop:BPcontraction:1stmo}, $||\textnormal{BP} \dot{q}_L[\dot{h}]-\dot{q}^\star_{L}||_1 \lesssim \frac{k^2}{2^k}||\dot{q}_L[\dot{h}]-\dot{q}^\star_{L}||_1$.
	    \item $\textnormal{BP} \dot{q}_L[\dot{h}](\rr), \dot{q}_L[\dot{h}](\rr)\in [\frac{1}{2},\frac{1}{2}+\frac{C}{2^k}]$ and $\textnormal{BP} \dot{q}_L[\dot{h}](\bb), \dot{q}_L[\dot{h}](\bb)\in [\frac{1}{2}-\frac{C}{2^k}, \frac{1}{2}]$.
	    \item $\textnormal{BP} \dot{q}_L[\dot{h}](\ff), \dot{q}_L[\dot{h}](\ff)\leq \frac{C}{2^k}$.
	\end{itemize}
	For $\dot{h} \in \PPP(\dot{\Omega}_L)$, define $\dot{q}^\circ[\dot{h}]\equiv \dot{q}^\circ_{L}[\dot{h}]\in \PPP(\dot{\Omega}_L)$ by 
	\begin{equation}\label{eq:def:qdot:Z:circ}
	    \dot{q}^\circ_{L}[\dot{h}](\dot{\sigma}) \equiv \frac{1}{Z^{\circ}_{\dot{h}}}\frac{\dot{h}(\dot{\sigma})}{\dot{h}^\star_L(\dot{\sigma})}\dot{q}^\star_L(\dot{\sigma}),\dot{\sigma}\in \dot{\Omega}_L,\textnormal{  where  } Z_{\dot{h}}^\circ \equiv \sum_{\dot{\sigma}\in \dot{\Omega}_L}\frac{\dot{h}(\dot{\sigma})}{\dot{h}^\star_L(\dot{\sigma})}\dot{q}^\star_L(\dot{\sigma})
	\end{equation}
	For a signed measure $a$ on $\dot{\Omega}_L$, define the norm $||a||_{\ff}\equiv \sum_{\dot{\sigma} \in\{\rr,\bb\}}|a(\dsigma)|+2^{k}\sum_{\dsigma \in \{\ff\}}|a(\dsigma)|$. Then, we claim the two inequalities stated below. For $\eps_L>0$ small enough and $||\dot{h}-\dot{h}^\star_L||_1<\eps_L,\dot{h}\in \PPP(\dot{\Omega}_L)$.
	\begin{align}
	    ||\dot{q}^\circ_L[\dot{h}]-\dot{q}^\star_L||_{\ff} &\leq C2^{2k} ||\dot{h}-\dot{h}^\star_L||_1, \label{eq:goal-1:lem:1stmo:Lipschitz:hdot:qdot}\\
	    ||\dot{q}^\circ_L[\dot{h}]-\dot{q}_L[\dot{h}]||_{\ff}&\leq C\frac{k^2}{2^k}||\dot{q}_L[\dot{h}]-\dot{q}^\star_L||_{\ff},\label{eq:goal-2:lem:1stmo:Lipschitz:hdot:qdot}
	\end{align}
	The two inequalities above imply \eqref{eq:lem:1stmo:Lipschitz:hdot:qdot} by the following: for $||\dot{h}-\dot{h}^\star_L||_1<\eps_L,\dot{h}\in \PPP(\dot{\Omega}_L)$,
	\begin{equation*}
	    \left(1-C\frac{k^2}{2^k}\right)||\dot{q}_L[\dot{h}]-\dot{q}^\star_L||_{\ff}\leq ||\dot{q}_L[\dot{h}]-\dot{q}^\star_L||_{\ff}-||\dot{q}^\circ_L[\dot{h}]-\dot{q}_L[\dot{h}]||_{\ff}\leq  ||\dot{q}^\circ_L[\dot{h}]-\dot{q}^\star_L||_{\ff}\leq  C2^{2k}||\dot{h}-\dot{h}^\star_L||_1,
	\end{equation*}
	so that for $k$ large enough, $||\dot{q}_L[\dot{h}]-\dot{q}^\star_L||_{1}\leq ||\dot{q}_L[\dot{h}]-\dot{q}^\star_L||_{\ff}\lesssim 2^{2k}||\dot{h}-\dot{h}^\star_L||_1$.
	
	Hence, it suffices to prove \eqref{eq:goal-1:lem:1stmo:Lipschitz:hdot:qdot} and \eqref{eq:goal-2:lem:1stmo:Lipschitz:hdot:qdot} for $||\dot{h}-\dot{h}^\star_L||_1<\eps_L,\dot{h}\in \PPP(\dot{\Omega}_L)$. The proof of \eqref{eq:goal-1:lem:1stmo:Lipschitz:hdot:qdot} is easier: dropping the subscript $L$ for simplicity, we have
	\begin{equation}\label{eq:goal-1:Lipschitz:1ststep}
	\begin{split}
	    ||\dot{q}^\circ[\dot{h}]-\dot{q}^\star||_{\ff}&=\sum_{\dsigma}\frac{2^{k\one\left\{\dsigma \in \{\ff\}\right\}}}{Z^\circ_{\dot{h}}}\frac{\dot{q}^\star(\dsigma)}{\dot{h}^\star(\dsigma)}\Big|Z^\circ_{\dot{h}}\dot{h}^\star(\dsigma)-\dot{h}(\dsigma)\Big|
	    \\&\leq \sum_{\dsigma}\frac{2^{k\one\left\{\dsigma \in \{\ff\}\right\}}}{Z^\circ_{\dot{h}}}\frac{\dot{q}^\star(\dsigma)}{\dot{h}^\star(\dsigma)}\Big|\dot{h}^\star(\dsigma)-\dot{h}(\dsigma)\Big|+\frac{\big|Z^\circ_{\dot{h}}-1\big|}{Z^\circ_{\dot{h}}}\sum_{\dsigma}\dot{q}^\star(\dsigma)2^{k\one\left\{\dsigma \in \{\ff\}\right\}}. 
	\end{split}
	\end{equation}
	We first upper bound $\frac{\dot{q}^\star(\dsigma)}{\dot{h}^\star(\dsigma)}$ in the \textsc{rhs} of the equation above: recall \eqref{eq:1stmo:hdot:intermsof:qdot} and take $\sig \in (\dot{\sigma}, \bb^{k-1})$ for $\dsigma \in \{\rr,\bb\}$ and $\sig \in (\dot{\sigma}\fs, \bb^{k-1})$ for $\dsigma \in \{\ff\}$ in the sum of \eqref{eq:1stmo:hdot:intermsof:qdot} to lower bound $\dot{h}^\star(\dsigma)$ by
	\begin{equation}\label{eq:lowerbound:hdot:star}
	    \dot{h}^\star(\dsigma)\geq \frac{\dot{q}^\star(\dsigma)}{Z^\prime_{\dot{q}^\star}}\frac{1}{2^k}\dot{q}^\star(\bb)^{k-1} \gtrsim\frac{\dot{q}^\star(\dsigma)}{ Z^\prime_{\dot{q}^\star}}\frac{1}{2^{2k}},
	\end{equation}
	where the last inequality is because $\dot{q}^\star(\bb)=\frac{1}{2}-O(\frac{1}{2^k})$. Also, because $\dot{q}^\star(\rr)=\frac{1}{2}+O(\frac{1}{2^k}), \dot{q}^\star(\ff) =O(\frac{1}{2^k})$, it is not hard to see that $Z^\prime_{\dot{q}^\star}=\sum_{\sig\in \Omega_L^k}\hat{\Phi}(\sig)^\la \prod_{i=1}^{k}\dot{q}^\star(\dsigma_i) \asymp \frac{1}{2^k}$, where the main contribution comes from $\sig \in \bb^{k}$. Hence, \eqref{eq:lowerbound:hdot:star} shows 
	\begin{equation}\label{eq:upperbound:hdot:qdot:ratio}
	    \sup_{\dsigma \in \dot{\Omega}_L}\frac{\dot{q}^\star(\dsigma)}{\dot{h}^\star(\dsigma)}\lesssim 2^k
	\end{equation}
	Using the equation above, we can also estimate $Z^\circ_{\dot{h}}$, defined in \eqref{eq:def:qdot:Z:circ}, by
	\begin{equation}\label{eq:1stmo:estimate:Z:circ}
	    |Z^\circ_{\dot{h}}-1|=\bigg|\sum_{\dsigma \in \dot{\Omega}_L} \Big(\frac{\dot{h}(\dsigma)}{\dot{h}^\star(\dsigma) }-1\Big)\dot{q}^\star(\dsigma)\bigg|\leq \sup_{\dsigma \in \dot{\Omega}_L}\frac{\dot{q}^\star(\dsigma)}{\dot{h}^\star(\dsigma)} ||\dot{h}-\dot{h}^\star||_1\lesssim 2^k||\dot{h}-\dot{h}^\star||_1,
	\end{equation}
	so taking $\eps_L$ small enough, $Z^\circ_{\dot{h}}\geq \frac{1}{2}$ for $||\dot{h}-\dot{h}^\star||_1<\eps_L$. Therefore, plugging \eqref{eq:upperbound:hdot:qdot:ratio} and \eqref{eq:1stmo:estimate:Z:circ} into the \textsc{rhs} of \eqref{eq:goal-1:Lipschitz:1ststep} shows our first claim \eqref{eq:goal-1:lem:1stmo:Lipschitz:hdot:qdot}.
	
	Turning to the second claim \eqref{eq:goal-2:lem:1stmo:Lipschitz:hdot:qdot}, for $\dot{q} \in \PPP(\dot{\Omega}_L)$, define the positive measure $\mu_{\dot{q}}$ on $\dot{\Omega}_L$ by
	\begin{equation}\label{eq:def:mu:qdot}
	\mu_{\dot{q}}(\dsigma)\equiv \sum_{\sig \in \Omega_L^k, \dsigma_1=\dsigma}\hat{\Phi}(\sig)^\la \prod_{i=2}^{k-1}\dot{q}(\dsigma_i) \textnormal{BP}\dot{q}(\dsigma_k),\quad\textnormal{for}\quad \dsigma \in \dot{\Omega}_L.
	\end{equation}
	Then, \eqref{eq:1stmo:hdot:intermsof:qdot} shows that $\dot{h}_L[\dot{q}](\dsigma)=\frac{\dot{q}(\dsigma)}{Z^\prime_{\dot{q}}}\mu_{\dot{q}}(\dsigma)$ for $\dot{q} \in \PPP(\dot{\Omega}_L)$, so plugging it into \eqref{eq:def:qdot:Z:circ} shows
	\begin{multline}\label{eq:1stmo:Lipschitz:very:long}
	    ||\dot{q}^\circ[\dot{h}]-\dot{q}[\dot{h}]||_{\ff}
	    =\sum_{\dsigma\in \dot{\Omega}_L}\frac{2^{k\one\left\{\dsigma \in \{\ff\}\right\}}\dot{q}[\dot{h}](\dsigma)}{Z^\circ_{\dot{h}}}\bigg|\sum_{\dot{\tau}\in \dot{\Omega}_L}\dot{q}[\dot{h}](\dot{\tau})\frac{\mu_{\dot{q}[\dot{h}]}(\dot{\tau})}{\mu_{\dot{q}^\star}(\dot{\tau})}-\frac{\mu_{\dot{q}[\dot{h}]}(\dsigma)}{\mu_{\dot{q}^\star}(\dsigma)}\bigg|\\
	    \leq \sum_{\dsigma\in \dot{\Omega}_L}\sum_{\dot{\tau}\in \dot{\Omega}_L}\frac{2^{k\one\left\{\dsigma \in \{\ff\}\right\}}\dot{q}[\dot{h}](\dsigma)\dot{q}[\dot{h}](\dot{\tau})}{Z^\circ_{\dot{h}}}\bigg|\frac{\mu_{\dot{q}[\dot{h}]}(\dot{\tau})}{\mu_{\dot{q}^\star}(\dot{\tau})}-\frac{\mu_{\dot{q}[\dot{h}]}(\dsigma)}{\mu_{\dot{q}^\star}(\dsigma)}\bigg|\lesssim \sup_{\dsigma\in \dot{\Omega}_L, \dot{\tau}\in \dot{\Omega}_L}\bigg|\frac{\mu_{\dot{q}[\dot{h}]}(\dot{\tau})}{\mu_{\dot{q}^\star}(\dot{\tau})}-\frac{\mu_{\dot{q}[\dot{h}]}(\dsigma)}{\mu_{\dot{q}^\star}(\dsigma)}\bigg|,
	\end{multline}
	where the first inequality is due to the triangle inequality, and the second inequality is due to \eqref{eq:1stmo:estimate:Z:circ} and the bound $\dot{q}[\dot{h}](\ff)=O(\frac{1}{2^k})$. We now claim that for $||\dot{h}-\dot{h}^\star_L||_1<\eps_L,\dot{h}\in \PPP(\dot{\Omega}_L)$,
	\begin{equation}\label{eq:finalgoal:lem:1stmo:Lipschitz:hdot:qdot}
	    \sup_{\dsigma\in \dot{\Omega}_L, \dot{\tau}\in \dot{\Omega}_L}\bigg|\frac{\mu_{\dot{q}[\dot{h}]}(\dot{\tau})}{\mu_{\dot{q}^\star}(\dot{\tau})}-\frac{\mu_{\dot{q}[\dot{h}]}(\dsigma)}{\mu_{\dot{q}^\star}(\dsigma)}\bigg| \leq C\frac{k^2}{2^k}||\dot{q}[\dot{h}]-\dot{q}^\star||_{\ff}
	\end{equation}
	It is clear from \eqref{eq:1stmo:Lipschitz:very:long} that \eqref{eq:finalgoal:lem:1stmo:Lipschitz:hdot:qdot} implies our second claim \eqref{eq:goal-2:lem:1stmo:Lipschitz:hdot:qdot}. Thus the rest of the proof is devoted to proving \eqref{eq:finalgoal:lem:1stmo:Lipschitz:hdot:qdot}. Henceforth, we denote $\dot{q}=\dot{q}_L[\dot{h}]$ for simplicity. Note that $\dot{q}$ satisfy $||\textnormal{BP}\dot{q}-\dot{q}^\star||_1\lesssim \frac{k^2}{2^k}||\dot{q}-\dot{q}^\star||_1$, $\textnormal{BP}\dot{q}(\rr),\dot{q}(\rr)=\frac{1}{2}+O(\frac{1}{2^k}), \textnormal{BP}\dot{q}(\bb), \dot{q}(\bb)=\frac{1}{2}-O(\frac{1}{2^k})$, and $\textnormal{BP}\dot{q}(\ff),\dot{q}(\ff)=O(\frac{1}{2^k})$.
	
	First, observe that it suffices to prove \eqref{eq:finalgoal:lem:1stmo:Lipschitz:hdot:qdot} for $\dot{\tau}=\rr_0$, by a triangle inequality. Also, since $\mu_{\dot{q}}(\dsigma)=\mu_{\dot{q}}(\dsigma\oplus 1)$, we may assume $\dsigma \in \{\bb,\ff\}$. Next, lower bounding $\mu_{\dot{q}^\star}(\dsigma)$ in the similar fashion as in \eqref{eq:lowerbound:hdot:star}, i.e. taking $(\sigma_2,...,\sigma_k)\in \bb^{k-1}$ in the sum of \eqref{eq:def:mu:qdot}, shows $\mu_{\dot{q}^\star}(\dsigma)\gtrsim \frac{1}{2^k}$ for $\dsigma \in \{\bb,\ff\}$, so
	\begin{equation}\label{eq:Lipschitz:finalgoal:firststep}
	    \bigg|\frac{\mu_{\dot{q}}(\dot{\rr_0})}{\mu_{\dot{q}^\star}(\rr_0)}-\frac{\mu_{\dot{q}}(\dsigma)}{\mu_{\dot{q}^\star}(\dsigma)}\bigg|= \bigg|\frac{\dot{q}(\bb)^{k-2}\textnormal{BP}\dot{q}(\bb)}{\dot{q}^\star(\bb)^{k-1}}-\frac{\mu_{\dot{q}}(\dsigma)}{\mu_{\dot{q}^\star}(\dsigma)}\bigg|\lesssim 2^{2k}\bigg|\dot{q}(\bb)^{k-2}\textnormal{BP}\dot{q}(\bb)\mu_{\dot{q}}(\dsigma)-\dot{q}^\star(\bb)^{k-1}\mu_{\dot{q}}(\dsigma)\bigg|.
	\end{equation}
	We now aim to show $|\dot{q}(\bb)^{k-2}\textnormal{BP}\dot{q}(\bb)\mu_{\dot{q}}(\dsigma)-\dot{q}^\star(\bb)^{k-1}\mu_{\dot{q}}(\dsigma)| \lesssim \frac{k^2}{2^{3k}}||\dot{q}-\dot{q}^\star||_{\ff}$. Note the following:
	\begin{equation}\label{eq:1stmo:Lipschitz:f:expression}
	\begin{split}
	    &\bigg|\dot{q}(\bb)^{k-2}\textnormal{BP}\dot{q}(\bb)\mu_{\dot{q}}(\dsigma)-\dot{q}^\star(\bb)^{k-1}\mu_{\dot{q}}(\dsigma)\bigg|= \bigg|\sum_{\sig\in \Omega_L^k,\dot{\sigma}_1=\dot{\sigma}}G(\sig)\bigg|,\quad\textnormal{where}\\
	    &G(\sig) \equiv \hat{\Phi}(\sig)^\la\Big(\prod_{i=2}^{k-1}\dot{q}(\dsigma_i)\textnormal{BP}\dot{q}(\dsigma_k)\dot{q}^\star(\bb)^{k-1}-\prod_{i=2}^{k}\dot{q}^\star(\dsigma_i)\dot{q}(\bb)^{k-2}\textnormal{BP}\dot{q}(\bb)\Big)
	\end{split}
	\end{equation}
	The crucial observation is that writing $\sig =(\sigma_1, \sig^{-1})$, the contribution of $\sig^{-1}\in \bb^{k-1}$ to the sum in \eqref{eq:1stmo:Lipschitz:f:expression} is zero, i.e. $\sum_{\dsigma_1=\dsigma, \sig^{-1}\in \bb^{k-1}}G(\sig)=0$. To this end, we deal with the case when $\sig^{-1}\notin \bb^{k-1}$ and divide the sum in \eqref{eq:1stmo:Lipschitz:f:expression} into the following $4$ cases. Let $D_i\equiv D_i(\dsigma),i=1,2,3,4$ be defined by
	\begin{equation*}
	\begin{split}
	    D_1 &\equiv \{\sig \in \Omega_L^k:\dsigma_1=\dsigma\textnormal{ and }\sig^{-1}\in \textnormal{per}(\sigma \bb^{k-2})\textnormal{ for some }\sigma\textnormal{ with }\hat{\sigma}=\fs\}\\
	    D_2 &\equiv \{\sig \in \Omega_L^k:\dsigma_1=\dsigma\textnormal{ and }\exists 2\leq i \leq k, \sigma_i \in \{\rr\}\}\\
	    D_3 &\equiv \{\sig \in \Omega_L^k:\dsigma_1=\dsigma, \exists 2\leq i<j\leq k, \dot{\sigma}_i,\dot{\sigma}_j \in \{\ff\}, \textnormal{ and }\sig\textnormal{ is separating.}\}\\
	    D_4 &\equiv \{\sig \in \Omega_L^k:\dsigma_1=\dsigma\textnormal{ and }\sig\textnormal{ is non-separating.}\}
	\end{split}
	\end{equation*}
	Let $f_i(\dsigma) \equiv \Big|\sum_{\sig \in D_i}G(\sig)\Big|, 1\leq i \leq 4$. Then, the triangle inequality shows
	\begin{equation}\label{eq:Lipschitz:finalgoal:secondstep}
	    \Big|\dot{q}(\bb)^{k-2}\textnormal{BP}\dot{q}(\bb)\mu_{\dot{q}}(\dsigma)-\dot{q}^\star(\bb)^{k-1}\mu_{\dot{q}}(\dsigma)\Big| \leq f_1(\dsigma)+f_2(\dsigma)+f_3(\dsigma)+f_4(\dsigma).
	\end{equation}
	To this end, for $\dsigma \in \{\bb,\ff\}$, we show $f_i(\dsigma) \lesssim \frac{k^2}{2^{3k}}||\dot{q}-\dot{q}^\star||_{\ff}$ separately for $1\leq i \leq 4$. First, using the bound $\hat{\Phi}(\sig)^\la \leq 1$ and a triangle inequality, it is straightforward to bound
	\begin{equation}\label{eq:bound:f:1:first}
	    f_1(\dsigma)\leq (k-2) \dot{q}(\bb)^{k-3}\dot{q}^\star(\bb)^{k-2}\textnormal{BP}\dot{q}(\bb)\Big|\dot{q}(\ff)\dot{q}^\star(\bb)-\dot{q}^\star(\ff)\dot{q}(\bb)\Big|
	    +\dot{q}(\bb)^{k-2}\dot{q}^\star(\bb)^{k-2}\Big|\textnormal{BP}\dot{q}(\ff)\dot{q}^\star(\bb)-\dot{q}^\star(\ff)\textnormal{BP}\dot{q}(\bb)\Big|,
	\end{equation}
	Using the elementary fact $|ab-a^\prime b^\prime|\leq |a-a^\prime|b^\prime+a^\prime|b-b^\prime|$ and the bound $\dot{q}(\bb), \dot{q}^\star(\bb)= \frac{1}{2}-O(\frac{1}{2^k})$, we can further bound the \textsc{rhs} of the equation above by
	\begin{equation}\label{eq:bound:f:1:final}
	    f_1(\dsigma)\lesssim \frac{k}{2^{2k}}|\dot{q}(\ff)-\dot{q}^\star(\ff)|+\frac{k}{2^{2k}}\dot{q}^{\star}(\ff)|\dot{q}(\bb)-\dot{q}^\star(\bb)|+\frac{1}{2^{2k}}||\textnormal{BP}\dot{q}-\dot{q}^\star||_1\lesssim \frac{k^2}{2^{3k}}||\dot{q}-\dot{q}^\star||_{\ff},
	\end{equation}
	where the final inequality is due to $||\textnormal{BP}\dot{q}-\dot{q}^\star||_1\lesssim \frac{k^2}{2^{k}}||\dot{q}-\dot{q}^\star||_1$, $\dot{q}^\star(\ff)=O(\frac{1}{2^k})$ and the fact that we have weighted $\ff$ spins by $2^k$ in the definition of $||\cdot||_{\ff}$.
	
	Second, we bound $f_2(\dsigma)$. Note that $\hat{\Phi}(\sig)^{\la}=2^{-k+1}$ when $\sig$ is valid and has a $\rr$ spin. Proceeding in a similar fashion as in \eqref{eq:bound:f:1:first} and \eqref{eq:bound:f:1:final}, we can bound
	\begin{multline}
	    f_2(\dsigma)\leq \frac{k-2}{2^{k-1}} \dot{q}(\bb)^{k-3}\dot{q}^\star(\bb)^{k-2}\textnormal{BP}\dot{q}(\bb)\Big|\dot{q}(\rr)\dot{q}^\star(\bb)-\dot{q}^\star(\rr)\dot{q}(\bb)\Big|\\
	    +\frac{1}{2^{k-1}}\dot{q}(\bb)^{k-2}\dot{q}^\star(\bb)^{k-2}\Big|\textnormal{BP}\dot{q}(\rr)\dot{q}^\star(\bb)-\dot{q}^\star(\rr)\textnormal{BP}\dot{q}(\bb)\Big|\lesssim \frac{k}{2^{3k}}||\dot{q}-\dot{q}^\star||_{\ff}.
	\end{multline}
	
	To bound $f_3(\dsigma)$ and $f_4(\dsigma)$, the following elementary inequality will be useful: given finite sets $\XXX_{1},...,\XXX_{\ell}$ and positive measures $\mu_{i},\nu_{i}$ on $\XXX_{i}, 1\leq i \leq \ell$, the triangle inequality shows
	\begin{equation}\label{eq:Lipschitz:useful:ineq}
	\begin{split}
	    \sum_{\ux\in \prod_{i=1}^{\ell}\XXX_i} \bigg|\prod_{i=1}^{\ell}\mu_i(x_i)-\prod_{i=1}^{\ell}\nu_i(x_i) \bigg|
	    &\leq \sum_{\ux\in \prod_{i=1}^{\ell}\XXX_i} \sum_{i=1}^{\ell}\bigg(\prod_{j<i}\nu_j(x_j)\prod_{j>i}\mu_j(x_j)\bigg)\bigg|\mu_i(x_i)-\nu_i(x_i)\bigg|\\
	    &=\sum_{i=1}^{\ell}\bigg(\prod_{j<i}||\nu_j||_1\prod_{j>i}||\mu_j||_1\bigg)||\mu_i-\nu_i||_1.
	\end{split}
	\end{equation}
	Note that for separating $\sig=(\sigma_1,...,\sigma_k)$, either $\sigma_i\in \{\bb\}$ or $\sigma_i=(\dot{\sigma},\fs)$ with $\dsigma \in \{\ff\}$. Thus, to bound $f_3(\dsigma)$, we can split the sum $\sum_{\sig \in D_3}$ by the location of free spins and use \eqref{eq:Lipschitz:useful:ineq} with $\ell=2k-2$. Recalling $\hat{\Phi}(\sig)^\la \leq 1$, $\dot{q}(\ff),\textnormal{BP}\dot{q}(\ff), \dot{q}^{\star}(\ff)\leq \frac{C}{2^k}$ and $\dot{q}(\bb),\textnormal{BP}\dot{q}(\bb), \dot{q}^{\star}(\bb)\leq \frac{1}{2}$, we can bound
	\begin{equation}\label{eq:bound:f:3}
	\begin{split}
	   f_3(\dsigma) &\leq \sum_{i=2}^{k-2} \binom{k-1}{i}\left(i\left(\frac{C}{2^k}\right)^{i-1}\frac{1}{2^{2k-i-2}}+(2k-2-i)\left(\frac{C}{2^k}\right)^{i}\frac{1}{2^{2k-i-3}}\right) ||\dot{q}-\dot{q}^\star||_1\vee||\textnormal{BP}\dot{q}-\dot{q}^\star||_1\\
	   &\lesssim \frac{k^2}{2^{3k}}||\dot{q}-\dot{q}^\star||_1\leq\frac{k^2}{2^{3k}}||\dot{q}-\dot{q}^\star||_{\ff} 
	\end{split}
	\end{equation}
	For non-separating $\sig=(\sigma_1,...,\sigma_k)$, suppose there are $i$ free spins among $\sigma_1,...,\sigma_k$ and the rest are in $\{\bb\}$. Then, by Lemma \ref{lem:decompose:Phi:hat},
	\begin{equation}\label{eq:bound:f:4:ingredient-1}
	    \hat{\Phi}(\sig)^\la = \hat{\Phi}^{\textnormal{m}}(\sig)^{\la} \hat{v}(\sig)\leq 2^{\la}\hat{v}(\sig)\leq 2^{\la}\frac{2}{2^{k-i}},
	\end{equation}
	where the last inequality is because there are $2$ choices for the literals on the edges colored $\bb$ since $\sig$ is non-separating. Also, note that for $\dsigma_1,...,\dsigma_k \in \dot{\Omega}_L$,
	\begin{equation}\label{eq:bound:f:4:ingredient-2}
	    \big|\{\utau \in \Omega_L^k:\hat{\Phi}(\utau)\neq 0\textnormal{ and } \dot{\tau}_j=\dsigma_j, 1\leq j \leq k\}\big|\leq 2^{i+1},
	\end{equation}
	since the literals uniquely define $\utau$ if $\dot{\tau}_j$'s are determined and there are $2^{i+1}$ number of choices of literals. With \eqref{eq:bound:f:4:ingredient-1} and \eqref{eq:bound:f:4:ingredient-2} in hand, we can bound $f_4(\dsigma)$ in the similar fashion as in \eqref{eq:bound:f:3}: separating clauses have at least $2$ free spins, so we can bound
	\begin{equation}\label{eq:bound:f:4}
	\begin{split}
	   f_4(\dsigma) &\lesssim \sum_{i=1}^{k-1}\frac{2^{2i}}{2^k} \binom{k-1}{i}\left(i\left(\frac{C}{2^k}\right)^{i-1}\frac{1}{2^{2k-i-2}}+(2k-2-i)\left(\frac{C}{2^k}\right)^{i}\frac{1}{2^{2k-i-3}}\right) ||\dot{q}-\dot{q}^\star||_1\\
	   &\lesssim \frac{k}{2^{3k}}||\dot{q}-\dot{q}^\star||_1\leq\frac{k^2}{2^{3k}}||\dot{q}-\dot{q}^\star||_{\ff} .
	\end{split}
	\end{equation}
	Therefore, $f_i(\dsigma)\lesssim \frac{k^2}{2^{3k}} ||\dot{q}-\dot{q}^\star||_{\ff}$ for $i=1,2,3,4$ holds and together with \eqref{eq:Lipschitz:finalgoal:firststep} and \eqref{eq:Lipschitz:finalgoal:secondstep}, this concludes the proof of our final goal \eqref{eq:finalgoal:lem:1stmo:Lipschitz:hdot:qdot}.
	\end{proof}

	\subsection{Continuity in the pair model}
	\label{subsec:app:conti:2ndmo}
		In this subsection, we derive the analogs of the results in the previous subsection corresponding to the pair model. It is obvious that Lemmas \ref{lem:H:dot:to:s:lipschitz}, 
	\ref{lem:nu:q:continuous:1stmo} and \ref{lem:1stmo:BPfixedpoint:estimates} hold the same for the pair model without any modification. 
	
	The counterpart of Lemma \ref{lem:hdot:exptail:q:b:lowerbound} can be derived by analogous approach as well, but we give the precise statement and briefly discuss the necessary adjustments for its proof.
	
	\begin{cor}\label{cor:hdot:exptail:q:b:lowerbd}
		Let $\dot{\bh} \in \PPP(\dot{\Omega}_L^2)$ satisfy $\dot{\bh}(\{ \dot{\bsigma}: \dot{\sigma}^1 \textnormal{ or } \dot{\sigma}^2 \in \{\rr, \ff \} \}) \le C2^{-k}$, and write $\dot{\bq}=\dot{\bq}_L[\dot{\bh}]$. Then, there exists a constant $C_k>0$ depending only on $k$, such that
		\begin{equation*}
		\dot{\bq}(\bb\bb^=) + \dot{\bq}(\bb\bb^{\neq}) \ge C_k.
		\end{equation*}
	\end{cor}
	
	\begin{proof}
		The proof of Lemma \ref{lem:hdot:exptail:q:b:lowerbound} consisted of two separate parts where we showed $\dot{q}(\bb) \ge \dot{q}(\ff)$ and $\dot{q}(\bb) \ge 2^{-5k} \dot{q}(\rr)$. We take a similar approach, aiming to establish
		\begin{equation*}
\begin{split}
		\dot{\bq}(\bb\bb^=) + \dot{\bq}(\bb\bb^{\neq}) &\ge \dot{\bq}(\bb\ff) + \dot{\bq}(\ff\bb) + \dot{\bq}(\ff\ff);\\
		\dot{\bq}(\bb\bb^=) + \dot{\bq}(\bb\bb^{\neq}) &\ge 2^{-5k}  \dot{\bq} ( \{ \dot{\bsigma}: \dot{\sigma}^1 \textnormal{ or } \dot{\sigma}^2 \in \{\rr \} \}).
\end{split}
		\end{equation*}
		The first inequality can be obtained by studying $\dot{\bh}(\{ \dot{\bsigma}: \dot{\sigma}^1 \textnormal{ or } \dot{\sigma}^2 \in \{\ff \} \})$, in the same way it is done in Lemma \ref{lem:hdot:exptail:q:b:lowerbound}. For the second one, we study the number of half-edges in $\delta \mathcal{D}$ that are forcing in at least one copy, which is a natural counterpart of $X_\rr^\delta$. Further details are omitted due to similarity.
	\end{proof}

\begin{proof}[Details of the proof of Corollary \ref{cor:hdot:exptail:q:b:lowerbd}]

	Define $\gggg := \{\bb,\ff \}$ and $\pppp := \{\rr, \bb \}.$ We use this to simplify the pair-coloring as well; for instance, we write $\gggg\bb := \{\bb\bb^=, \bb\bb^{\neq}, \ff\bb \}$. Also, let $\rrrr := \{0,1 \}$.
		
		Proof is done in two steps:
		\begin{enumerate}
			\item Show $ \max\{ \dot{\bq}(\bb \ff ), \dot{\bq} (\ff\bb), \dot{\bq}(\ff\ff) \} \le \frac{1}{1000}\dot{\bq}(\gggg\gggg)$.
			
			\item Show $(1-\dot{\bq}(\gggg\gggg)) \le 2^{1000k} \dot{\bq}(\gggg\gggg).$
		\end{enumerate}
		
		\vspace{3mm}
		
		{\bf Part 1.} ~ We deal with $\dot{\bq}(\bb\ff)$, $\dot{\bq}(\ff\bb)$ and $\dot{\bq}(\ff\ff)$ separately. Dealing with the first two is done analogously, and working with $\dot{\bq}(\ff\ff)$ is simpler. We present the details for $\dot{\bq}(\bb\ff)$ and also some comments regarding what changes are needed for $\dot{\bq}(\ff\ff)$.
		
		Suppose that $\dot{\bq}(\bb\ff) \ge \frac{1}{1000}\dot{\bq}(\gggg\gggg).$ Define $\bb\bb := \{\bb\bb^=, \bb\bb^{\neq} \}.$ This means that 
		\begin{equation}\label{eq:qbfweight}
		\begin{split}
		\dot{\bq}(\gggg\gggg)=\dot{\bq}(\bb\bb) + \dot{\bq}(\ff\bb) + \dot{\bq}(\bb\ff) + \dot{\bq}(\ff\ff) &\ge \frac{1000}{999} (\dot{\bq}(\bb\bb)+ \dot{\bq}(\ff\bb));\\
\dot{\bq}(\bb\gggg)=		\dot{\bq}(\bb\bb) + \dot{\bq}(\bb\ff) &\ge \frac{1000}{999}\dot{\bq}
(\bb\bb).		\end{split}
		\end{equation}
		Using this along with the assumption 
		$\dot{\bh}(\{ \dot{\bsigma}: \dot{\sigma}^1 \textnormal{ or } \dot{\sigma}^2 \in \{\rr\} \}) \le C2^{-k}$, we will deduce contradiction by showing $\dot{\bh}(\{ \dot{\bsigma}: \dot{\sigma}^1 \textnormal{ or } \dot{\sigma}^2 \in \{ \ff \} \}) >> C2^{-k}$.
		
	With a little abuse of notation, we denote by $\bsigma_v$ the pair-frozen model spin at $v$. To use the same approach as the first moment, we look at the cases when $\bsigma_v \in \{\ff\ff\},  \{\rrrr \ff \}, \{\ff\rrrr \}, \{\rrrr\rrrr^=, \rrrr\rrrr^{\neq} \}$. Note that the cases $\rrrr\ff$ and $\ff\rrrr$ should be treated differently. Recall the definition of $\nu=\nu_{\dot{\bq}}$ for the pair model. Also, we write $a_1,\ldots, a_d \sim v$ , $e_i = (va_i)$, $i=1,\ldots, d$, and $\{e_{ij}\}_{j=2}^k$ the half-edges at $\delta \DD$ adjacent to $a_i$.
		
		\vspace{3mm}
		{\bf Part 1: Case 1.} ~ $\bsigma_v = \ff\ff$.
		
		In this case, we always have $\{\bsig_{e_{ij}} \}_j \in (\gggg\gggg)^{k-1}$ for all $i$. Due to \eqref{eq:qbfweight} and from the same argument as before, we get
		\begin{equation}
		\E_\nu [\#\ff \textnormal{ in the 2nd copy at }\delta \DD | \bsigma_v = \ff\ff  ] \gtrsim d,
		\end{equation}
		since on a $\gggg\gggg$ color at $\delta\DD$
  we are likely to see $\bb\ff$ with probability bounded from below by an absolute constant. Full detail can be written using the binomial expansion similarly as \eqref{eq:conditional:v:free:long}.

		\vspace{3mm}
		{\bf Part 1: Case 2.} ~ $\bsigma_v \in \rrrr\ff$.
		
		In this case, $\bsigma_{e_i} \in \{\bb\ff, \rr\ff \}.$ We treat the two cases separately, and also divide $\bb\ff$ into two types as follows:
		\begin{enumerate}
			\item [$\bb\ff1$.]  $\bsigma_{e_i} \in  \bb\ff$ and $(\bsigma_{e_{ij}} ) \in (\gggg\gggg)^{k-1}$. In this case, the clause weight is $\hat{v}_2 (\bsig_{a_i})  \asymp 1$.
			
			\item [$\bb\ff2$.] $\bsigma_{e_i} \in  \bb\ff$ and $(\bsigma_{e_{ij}} ) \in (\rr\gggg, (\bb\gggg)^{k-2})$. Here, and throughout the proof, the notation $(\rr\gggg, (\bb\gggg)^{k-2})$ means the configurations that belong to $(\rr\gggg, \bb\gggg, \ldots, \bb\gggg)$ and its permutations.  In this case, the clause weight is $\hat{v}_2 (\bsig_{a_i})  = 2^{-k+1}$.
			
			\item [$\rr\ff$.] $\bsigma_{e_i} \in \rr\ff$ and $(\bsigma_{e_{ij}} ) \in (\bb\gggg)^{k-1}$.  In this case, the clause weight is $\hat{v}_2 (\bsig_{a_i})  = 2^{-k+1}$.
		\end{enumerate}
		On the edges $e_i$'s, suppose that we condition on the number of each type: $(\bb\ff1, \bb\ff2, \rr\ff) := (\#\{i: \bsigma_{e_i} \in \bb\ff1 \}, \#\{i: \bsigma_{e_i} \in \bb\ff2 \}, \#\{i: \bsigma_{e_i} \in \rr\ff \}) = (x_1, x_2, x_3)$. Because of \eqref{eq:qbfweight}, we have
		\begin{equation}\label{eq:eq111}
		\E_\nu [\#\ff \textnormal{ in the 2nd copy at } \delta \DD | \bsigma_v \in  \rrrr\ff, (x_1, x_2, x_3)] \gtrsim d,
		\end{equation} 
		since for all three types, in $\gggg\gggg$ (resp. $\bb\gggg$), we are likely to see $\gggg\ff$ (resp. $\bb\ff$) with probability bounded below by an absolute constant.
		
		\vspace{3mm}
		
		{\bf Part 1: Case 3.} $\bsigma_v \in \ff\rrrr$.
		
		Although this case should be dealt differently from the previous one, the classification of $e_i$ can be done in the same way:
		\begin{enumerate}
				\item [$\ff\bb1$.]  $\bsigma_{e_i} \in  \ff\bb$ and $(\bsigma_{e_{ij}} ) \in (\gggg\gggg)^{k-1}$. In this case, the clause weight is $\hat{v}_2 (\bsig_{a_i})  \asymp 1$.
				
				\item [$\ff\bb2$.] $\bsigma_{e_i} \in  \ff\bb$ and $(\bsigma_{e_{ij}} ) \in (\gggg\rr, (\gggg\bb)^{k-2})$.  In this case, the clause weight is $\hat{v}_2 (\bsig_{a_i})  = 2^{-k+1}$.
				
				\item [$\ff\rr$.] $\bsigma_{e_i} \in \ff\rr$ and $(\bsigma_{e_{ij}} ) \in (\gggg\bb)^{k-1}$.  In this case, the clause weight is $\hat{v}_2 (\bsig_{a_i})  = 2^{-k+1}$.
		\end{enumerate}
		Suppose that we condition on $(\bb\ff1, \bb\ff2, \rr\ff) = (x_1, x_2, x_3)$. Then, we have 
			\begin{equation}
			\E_\nu [\#\ff \textnormal{ in the 2nd copy at } \delta \DD | \bsigma_v \in \ff\rrrr, (x_1, x_2, x_3)] \gtrsim x_1,
			\end{equation} 
		by the same reason as above. In this case, we cannot gain anything about $x_2$ and $x_3$. However, if we just condition on $\#\ff\bb2 = x_2$ and $\#\ff \rr \ge 1$ (which is necessary to have $\bsigma_v \in \ff\rrrr$), we see that
		\begin{equation}
		\E_\nu [\# \ff\bb1 |  \bsigma_v \in \ff\rrrr, \#\ff\bb2 = x_2, \#\ff\rr \ge 1] \gtrsim d-1-x_2,
		\end{equation}
		since $\dot{\bq}(\gggg\gggg )\ge \dot{\bq}(\gggg\bb)$ and the clause weight of $\ff\bb1$ is bigger than that of $\ff\rr$.
		
		Finally, we know from the assumption that it is rare to have many $\ff\bb2$, due to the bound on $\#\rr$ at $\delta\DD$:
		\begin{equation}\label{eq:part1case3eq1}
		\E_\nu [ \bsigma_v \in \ff\rrrr,  \ \#\ff\bb2 \ge \frac{d}{2} ] \lesssim 2^{-2k}.
		\end{equation}
		
		Combining the three observations, we get
		\begin{equation}
		\E_\nu[\# \ff \textnormal{ in the 2nd copy at }\delta \DD ; \, \bsigma_v \in \ff\rrrr] \gtrsim d \P_\nu(\bsigma_v \in \ff\rrrr) - C2^{-2k},
		\end{equation}
		where the $C$ in the right is from \eqref{eq:part1case3eq1}.
		
	\vspace{3mm}
	
	{\bf Part 1: Case 4.} ~ $\bsigma_v \in \rrrr\rrrr^=$.
	
	The case $\bsigma_v \in \rrrr\rrrr^{\neq}$ can be done in the same way; it is clear from the proof below and hence those details are omitted.
	
	There are 10 types to classify $e_i$'s in this case
	\begin{enumerate}
			\item [$\rr\rr$.]  $\bsigma_{e_i} \in  \rr\rr^=$ and $(\bsigma_{e_{ij}} ) \in (\bb\bb^=)^{k-1}$. In this case, the clause weight is $\hat{v}_2 (\bsig_{a_i})  =2^{-k+1}$.
			
			\item [$\rr\bb1$.]  $\bsigma_{e_i} \in  \rr\bb^=$ and $(\bsigma_{e_{ij}} ) \in (\bb\gggg)^{k-1}$. In this case, the clause weight is $\hat{v}_2 (\bsig_{a_i})  = 2^{-k+1}$.
			
			\item [$\rr\bb2$.]  $\bsigma_{e_i} \in  \rr\bb^=$ and $(\bsigma_{e_{ij}} ) \in (\bb\rr^=, (\bb\bb^{\neq})^{k-2})$. In this case, the clause weight is $\hat{v}_2 (\bsig_{a_i})  = 2^{-k+1}$.
			
			\item [$\bb\rr1$.]  $\bsigma_{e_i} \in  \bb\rr^=$ and $(\bsigma_{e_{ij}} ) \in (\gggg\bb)^{k-1}$. In this case, the clause weight is $\hat{v}_2 (\bsig_{a_i}) = 2^{-k+1}$.
			
			\item [$\bb\rr2$.]  $\bsigma_{e_i} \in  \bb\rr^=$ and $(\bsigma_{e_{ij}} ) \in (\rr\bb^=, (\bb\bb^{\neq})^{k-2})$. In this case, the clause weight is $\hat{v}_2 (\bsig_{a_i})  = 2^{-k+1}$.
			
			\item [$\bb\bb1$.]  $\bsigma_{e_i} \in  \bb\bb^=$ and $(\bsigma_{e_{ij}} ) \in (\gggg\gggg)^{k-1}$. In this case, the clause weight is $\hat{v}_2 (\bsig_{a_i})  \asymp 1$.
			
			\item [$\bb\bb2$.]  $\bsigma_{e_i} \in  \bb\bb^=$ and $(\bsigma_{e_{ij}} ) \in (\rr\rr^=, (\bb\bb^=)^{k-2})$. In this case, the clause weight is $\hat{v}_2 (\bsig_{a_i})  = 2^{-k+1}$.
			
				\item [$\bb\bb3$.]  $\bsigma_{e_i} \in  \bb\bb^=$ and $(\bsigma_{e_{ij}} ) \in (\rr\bb^{\neq}, \bb\rr^{\neq}, (\bb\bb^=)^{k-3})$. In this case, the clause weight is $\hat{v}_2 (\bsig_{a_i})  = 2^{-k+1}$.
				
					\item [$\bb\bb4$.]  $\bsigma_{e_i} \in  \bb\bb^=$ and $(\bsigma_{e_{ij}} ) \in (\rr\bb^{\neq}, (\bb\gggg)^{k-2})$. In this case, the clause weight is $\hat{v}_2 (\bsig_{a_i})  = 2^{-k+1}$.
					
				\item [$\bb\bb5$.]  $\bsigma_{e_i} \in  \bb\bb^=$ and $(\bsigma_{e_{ij}} ) \in (\bb\rr^{\neq},  (\gggg\bb)^{k-2})$. In this case, the clause weight is $\hat{v}_2 (\bsig_{a_i})  = 2^{-k+1}$.
	\end{enumerate}
	Let us call $(\rr\rr, \rr\bb1, \rr\bb2, \bb\rr1, \bb\rr2, \bb\bb1, \ldots, \bb\bb5) = (x_1, \ldots, x_{10})$. Conditioned on $(x_3, x_5, x_7, x_8, x_9, x_{10})$, which are the numbers of the types that have an $\rr$ at $\delta \DD$, we are more likely to have the types $\rr\bb1$ (resp. $\bb\bb1$) than $\rr\rr$  (resp. $\bb\rr1$). Thus, we have
	\begin{equation}\label{eq:eq112}
	\E_\nu [x_2+x_6 | \bsigma_v \in \rrrr\rrrr^=,  (x_3, x_5, x_7, x_8, x_9, x_{10})] \gtrsim d-x_{\rr}-1.
	\end{equation}
	This expression works for the corner cases such as $x_{\rr}:= x_3+x_5+x_7+x_8+x_9+x_{10} = d$, or    $x_{\rr} =d-1$ and $x_3\cdot x_5 = 0$, since the right hand side becomes nonpositive. Also, by assumption, we have
	\begin{equation}\label{qwer}
	\P_\nu (\bsigma_v \in \rrrr\rrrr^=, x_{\rr} \ge \frac{d}{2}) \lesssim 2^{-2k}.
	\end{equation}
	Moreover, similarly as before, \eqref{eq:qbfweight} implies that
	\begin{equation}\label{wer}
	\E_\nu [ \# \ff \textnormal{ in the 2nd copy at }\delta\DD | \bsigma_v \in \rrrr\rrrr^=,  x_2+x_6] \gtrsim x_2+x_6.
	\end{equation}
	Thus, we combine these to obtain that
	\begin{equation}\label{rewq}
	\E_\nu [ \# \ff \textnormal{ in the 2nd copy at }\delta\DD; \ \bsigma_v \in \rrrr\rrrr^=] \gtrsim d\P_\nu(\bsigma_v\in \rrrr\rrrr^=) - C2^{-2k}.
	\end{equation}
	
	\vspace{3mm}
	
	{\bf Part1: Conclusion.} ~ Combining the conclusions of the four cases, we get
	\begin{equation}
	\E_\nu [\# \ff \textnormal{ in the 2nd copy at } \delta\DD] \gtrsim d-C2^{-2k},
	\end{equation}
	which gives
	\begin{equation}
	\dot{\bh}(\sigma^2\in \{\ff \} ) \gtrsim k^{-1}.
	\end{equation}
	Therefore, we deduce a contradiction and hence we must have $\dot{\bq}(\bb\ff) \le \frac{1}{1000} \dot{\bq}(\gggg\gggg).$
	
	\vspace{3mm}
	
	{\bf Part 1: Comments for the case } $\dot{\bq}(\ff\ff) \le \frac{1}{1000} \dot{\bq}(\gggg\gggg)$.
	
	Case 1: $\bsigma_v = \ff\ff$ and Case3: $\bsigma_v \in \ff\rrrr$ can be carried out precisely the same as the above analysis. 
	
	In Case 2: $\bsigma_v \in \rrrr\ff$, \eqref{eq:eq111} is no longer true under the assumption $\dot{\bq}(\ff\ff) \ge \frac{1}{1000}\dot{\bq}(\gggg\gggg).$ However, we can still say 
	\begin{equation}
\E_\nu [ \#\ff \textnormal{ in the 2nd copy at } \delta \DD| \bsigma_v \in \rrrr\ff, (x_1, x_2, x_3)] \gtrsim x_1,
	\end{equation}
	and conditional on $x_2$ we can get a lower bound on the expected $x_1$ analogously to \eqref{eq:eq112}, and then upper bound the probability of having too large $x_2$ by \eqref{qwer}.
	
	For Case 4: $\bsigma_v \in \rrrr\rrrr^=$, instead of \eqref{eq:eq112} we have
	\begin{equation}
	\E_\nu[x_6| \bsigma_v \in \rrrr\rrrr^=, (x_3,x_5,x_7,x_8, x_9, x_{10})] \gtrsim d-x_{\rr}-1,
	\end{equation}
	and instead of \eqref{wer} we have
	\begin{equation}
	\E_\nu [ \# \ff \textnormal{ in the 2nd copy at }\delta\DD | \bsigma_v \in \rrrr\rrrr^=,  x_6] \gtrsim x_6.
	\end{equation}
	The other arguments are the same and hence we get \eqref{rewq}.

	\vspace{3mm}
	
	{\bf Part 2. } ~ So far, we proved that $\dot{\bq}(\bb\bb) \ge \frac{99}{100} \dot{\bq}(\gggg\gggg).$ Suppose that $1- \dot{\bq}(\gggg\gggg) \ge 2^{1000k} \dot{\bq}(\gggg\gggg).$
	
	We combine all the picture from the previous argument. Let $Z$ denote the partition function for $\nu$, i.e.,
	\begin{equation}
	Z= \sum_{\bsig_{\delta \DD}} \bw(\bsig_{\delta \DD})^{\ula} \prod_{e\in \delta \DD} \dot{\bq}(\bsigma_e).
	\end{equation}

 	To begin with, note that 
 	\begin{equation}
 	\P_\nu(\bsigma_{v} = \ff\ff) \le  \frac{2^{\lambda^1+\lambda^2}}{Z} (\dot{\bq}(\gggg\gggg)^{k-1} )^{d},
 	\end{equation}
 	since the  weight satisfies $ \bw(\bsig_{\delta \DD})^{\ula}  \le 2^{\lambda^1+\lambda^2}$, and the leaves $\delta \DD$ cannot have any reds when $\bsigma_v = \ff\ff$.
 	
 	For the case $\bsigma_v \in \rrrr\ff$, recall the analysis from Part 1: Case 2, and note that we must have at least one $\rr\ff$-type edge $e_i$. Thus, we can write 
 	\begin{equation}
 \begin{split}
 	\P_\nu(\bsigma_v \in \rrrr\ff) \ge \frac{1}{Z} \left[ \left\{ \frac{1}{C}\dot{\bq}(\gggg\gggg)^{k-1} + \frac{1}{2^{k-1}} \dot{\bq}(\pppp\gggg) \dot{\bq}(\bb\gggg)^{k-2}  \right\}^{d} - \left\{ \frac{1}{C}\dot{\bq}(\gggg\gggg)^{k-1} + \frac{1}{2^{k-1}}  \dot{\bq}(\bb\gggg)^{k-1}  \right\}^{d} \right],
 \end{split}
 	\end{equation}
 	where $C>0$ is an absolute constant coming from the fact that the clause weight of a separating clause is $\gtrsim 1$. We can rewrite this as
 	\begin{equation}
 	\P_\nu(\bsigma_v \in \rrrr\ff) \ge \frac{\dot{\bq}(\rr\gggg) \dot{\bq}(\bb\gggg)^{k-2} }{2^{k-1}Z} \left\{ \frac{1}{C}\dot{\bq}(\gggg\gggg)^{k-1} + \frac{\dot{\bq}(\pppp\gggg) \dot{\bq}(\bb\gggg)^{k-2}}{2^{k-1}}   \right\}^{d-1} ,
 	\end{equation}
 	by just fixing $e_1$ to be type $\rr\ff$.
 	
 	Similarly, for $\bsigma_v \in \ff\rrrr,$ we have
 	\begin{equation}
 	\P_\nu(\bsigma_v \in \ff\rrrr) \ge \frac{\dot{\bq}(\gggg\rr) \dot{\bq}(\gggg\bb)^{k-2} }{2^{k-1}Z} \left\{ \frac{1}{C}\dot{\bq}(\gggg\gggg)^{k-1} + \frac{\dot{\bq}(\gggg\pppp) \dot{\bq}(\gggg\bb)^{k-2}}{2^{k-1}}   \right\}^{d-1} .
 	\end{equation}
 	
 	Finally, for $\bsigma_v \in \rrrr\rrrr^=$, 
		we just count the case where $e_1$ is fixed to be type $\rr\rr$. Then, we see that
		\begin{equation}
		\P_\nu(\bsigma_v \in \rrrr\rrrr^=) \ge \frac{\dot{\bq}(\rr\rr^=)\dot{\bq}(\bb\bb^=)^{k-2}}{2^{k-1}Z} \left\{ \frac{1}{C} \dot{\bq}(\gggg\gggg)^{k-1} + \frac{\dot{\bq}(\gggg\pppp) \dot{\bq}(\gggg\bb)^{k-2}}{2^{k-1}}+
		\frac{\dot{\bq}(\pppp\gggg) \dot{\bq}(\bb\gggg)^{k-2}}{2^{k-1}}+ \frac{\dot{\bq}(\rr\rr^=)\dot{\bq}(\bb\bb^=)^{k-2}}{2^{k-1}}  \right\}^{d-1}
		\end{equation}
		Similarly,
			\begin{equation}
			\P_\nu(\bsigma_v \in \rrrr\rrrr^{\neq}) \ge \frac{\dot{\bq}(\rr\rr^{\neq})\dot{\bq}(\bb\bb^{\neq})^{k-2}}{2^{k-1}Z} \left\{ \frac{1}{C} \dot{\bq}(\gggg\gggg)^{k-1} + \frac{\dot{\bq}(\gggg\pppp) \dot{\bq}(\gggg\bb)^{k-2}}{2^{k-1}}+
			\frac{\dot{\bq}(\pppp\gggg) \dot{\bq}(\bb\gggg)^{k-2}}{2^{k-1}}+ \frac{\dot{\bq}(\rr\rr^{\neq})\dot{\bq}(\bb\bb^{\neq})^{k-2}}{2^{k-1}}  \right\}^{d-1}.
			\end{equation}
			
	To conclude the proof, without loss of generality we assume $\dot{\bq}(\bb\bb^=)\ge \dot{\bq}(\bb\bb^{\neq})$, and we divide into two cases where $\frac{\dot{\bq}(\bb\bb^{\neq})}{\dot{\bq}(\bb\bb^=)} $		is bigger than $\frac{1}{1000}$ or not.
	
	\vspace{3mm}
	
	{\bf Part 2: Case 1.} $\frac{\dot{\bq}(\bb\bb^{\neq})}{\dot{\bq}(\bb\bb^=)}  \ge \frac{1}{1000}$.
	
	In this case, combined with the assumption of Part 2 that $1-\dot{\bq}(\gggg\gggg) \ge 2^{1000k} \dot{\bq}(\gggg\gggg)$ and the conclusion from Part 1 that $\dot{\bq}(\bb\bb) \ge \frac{99}{100}\dot{\bq}(\gggg\gggg)$, one of the following must hold true:
	\begin{equation}\label{qer}
	\begin{split}
	&\dot{\bq}(\gggg\gggg)^{k-1} \ll \frac{ \dot{\bq}(\rr\gggg) \dot{\bq}(\bb\gggg)^{k-2}}{2^{k-1}}; \quad\textnormal{or}\quad
	\dot{\bq}(\gggg\gggg)^{k-1} \ll \frac{\dot{\bq}(\gggg\rr) \dot{\bq}(\gggg\bb)^{k-2}}{2^{k-1}}; \quad\textnormal{or}  \\
	&\dot{\bq}(\gggg\gggg)^{k-1} \ll \frac{\dot{\bq}(\rr\rr^=)\dot{\bq}(\bb\bb^=)^{k-2}}{2^{k-1}};\quad\textnormal{or}\quad \dot{\bq}(\gggg\gggg)^{k-1}\ll \frac{\dot{\bq}(\rr\rr^{\neq})\dot{\bq}(\bb\bb^{\neq})^{k-2}}{2^{k-1}}.
	\end{split}
	\end{equation}
	This is because at least one of $\dot{\bq}(\rr\gggg), \dot{\bq}(\gggg\rr), \dot{\bq}(\rr\rr^=), \dot{\bq}(\rr\rr^{\neq})$ should be greater than $2^{999k}\dot{\bq}(\gggg\gggg).$ This will give us that $\P_\nu (\bsigma_v = \ff\ff)\ll 1$.
	Moreover, observe that the above four equations also come from the types (in Part 1 of the proof) $\ff\bb2$, $\bb\ff2$, $\bb\bb2$ ($\bb\bb2^=$ and $\bb\bb2^{\neq}$), respectively.		These are the clauses that have a red leaf on $\delta \DD$, and this means that given $\bsigma_v \in \{\rrrr\ff, \ff\rrrr, \rrrr\rrrr \}$, we are more likely to have  neighboring clauses that have a red leaf than those who do not. The details can be carried out analogously to that of the first moment, and hence we get
	\begin{equation}
	\E_\nu[ \#\rr \textnormal{ in either copy at }\delta\DD ] \gtrsim d,
	\end{equation}
	contradicting $\dot{\bh}(\{ \dot{\bsigma}: \dot{\sigma}^1 \textnormal{ or } \dot{\sigma}^2 \in \{\rr \} \}) \le C2^{-k}$.
	
	\vspace{3mm}
	
	{\bf Part 2: Case 2.} $\frac{\dot{\bq}(\bb\bb^{\neq})}{\dot{\bq}(\bb\bb^=)}  < \frac{1}{1000}$.
	
	In this case, note that we can have the case where none of \eqref{qer} is true. For instance, $\dot{\bq}(\rr\rr^{\neq})$ is very big, all other red-including $\dot{\bq}$-weights are very small, and $\dot{\bq}(\bb\bb^{\neq})$ is extremely small. However, in such a case, on $\delta\DD$, $\bb\bb^=$ is going to be selected about 1000 times more often than $\bb\bb^{\neq}$ when both colors are valid. To be precise, in all types of $e_i$ mentioned in Part 1 except $\rr\rr$, $\rr\bb 2$, $\bb\rr 2$, $\bb\bb 2$ and $\bb\bb 3$, both $\bb\bb^=$ and $\bb\bb^{\neq}$ are valid choices for $\gggg\gggg, \gggg\bb, \bb\gggg$, and in such a situation $\bb\bb^=$ is going to be selected 1000 times more often.		Thus, if we have
	\begin{equation}
	\begin{split}
	\E_\nu[\#\rr\rr + \#\rr\bb 2 + \#\bb\rr 2 +\#\bb\bb 2 + \#\bb\bb 3] \ll d, \textnormal{ and} \\
	\dot{\bh}(\{ \dot{\bsigma}: \dot{\sigma}^1 \textnormal{ or } \dot{\sigma}^2 \in \{\rr, \ff \} \}) \le C2^{-k},
	\end{split}
	\end{equation}
	then we must have $	\dot{\bh}(\bb\bb^=) - \dot{\bh} (\bb\bb^{\neq} ) \gtrsim 1$ which is a contradiction. On the other hand, suppose that we had
	\begin{equation}
	\E_\nu[\#\rr\rr + \#\rr\bb 2 + \#\bb\rr 2 +\#\bb\bb 2 + \#\bb\bb 3] \gtrsim d.
	\end{equation}
	Note that the clauses $\rr\bb 2$, $\bb\rr 2$, $\bb\bb 2$ and $\bb\bb 3$ all carry at least one red color on $\delta \DD$. Also, the $\rr\rr^{\neq}$-types are less likely to be seen than $\bb\bb 2^{\neq}$-type since we are assuming $\dot{\bq}(\rr\rr^{\neq}) >> \dot{\bq}(\bb\bb^{\neq})$. Therefore, 
	\begin{equation}
	\E_\nu[\#\rr\rr + \#\rr\bb 2 + \#\bb\rr 2 +\#\bb\bb 2 + \#\bb\bb 3] \gtrsim d
	\end{equation}
	must imply
	\begin{equation}
	\E_\nu[\#\rr \textnormal{ on }\delta\DD] \gtrsim d.
	\end{equation}
	This will violate the assumption $	\dot{\bh}(\{ \dot{\bsigma}: \dot{\sigma}^1 \textnormal{ or } \dot{\sigma}^2 \in \{\rr \} \}) \le C2^{-k}$.
\end{proof}

	In the proofs of Lemmas \ref{lem:exists:qdot:hdot:exp:tail:improved} and \ref{lem:Xi:continuous:1stmo}, it is straight-forward to see that the techniques used in their proofs are generic in the sense that they do not rely on the specific properties of $\dot{q}$ except $\dot{q}(\bb) \ge C_k$ obtained from Lemma \ref{lem:hdot:exptail:q:b:lowerbound}. Thus, we can extend those results analogously to the case of the pair model, utilizing Corollary \ref{cor:hdot:exptail:q:b:lowerbd} instead of Lemma \ref{lem:hdot:exptail:q:b:lowerbound}. We can state the results as follows.
	
	\begin{cor}\label{cor:2ndmo:hdot:uniqueness}
		Suppose $\dot{\bh} \in \PPP(\dot{\Omega}^2)$ satisfies $\dot{\bh}(\{ \dot{\bsigma}: \dot{\sigma}^1 \textnormal{ or } \dot{\sigma}^2 \in \{\rr,\ff \}  \}) \le c^{-1}2^{-k}$ and $\sum_{\dot{\bsigma}: v(\dot{\bsigma}) \ge L } \dot{\bh}(\dot{\bsigma}) \le 2^{-ckL}$ for all $L\ge 1$, where $c>0$ is an  absolute constant. Then, there exists a unique $\dot{\bq} = \dot{\bq}[\dot{\bh}] \in \PPP(\dot{\Omega}^2)$ such that $\dot{\bh}_{\dot{\bq}} = \dot{\bh}$. Moreover, there exists a constant $c_k$ such that $\dot{\bq}(\bb\bb^=) + \dot{\bq}(\bb\bb^{\neq}) \ge c_k$ and $\sum_{v(\dot{\bsigma}) \ge L } \dot{\bq}(\dot{\bsigma}) \le c_k^{-1} 2^{-ckL}$.
		
		Further, for any $C>0$, $\Xi_2: \boldsymbol{\Delta}_C^{\textnormal{exp}} \to \R_{\ge 0}$ is continuous.
	\end{cor}

	We derive an analog of Proposition \ref{prop:1stmo:Lipschitz:hdot:qdot} for the pair model. Let $\dot{\bq}$ be a probability measure on $\Omega_{2,L}$, and we give the pair-model version of  \eqref{eq:1stmo:hdot:intermsof:qdot} as follows. 
	
	\begin{equation}\label{eq:2ndmo:hdot:intermsof:qdot}
	\dot{\bh}_L[\dot{\bq}](\dot{\sigma})\equiv \sum_{\bsig \in \Omega_{2,L}^{k}}\frac{\hat{\Phi}_2(\bsig)^\la}{Z^\prime_{\dot{\bq}}}\prod_{i=1}^{k-1}\dot{\bq}(\dot{\bsigma}_i) \textnormal{BP}\dot{\bq}(\dot{\bsigma}_k)\one\{\dot{\bsigma}_1=\dot{\bsigma}\}.
	\end{equation}
	Moreover, for a probability measure $\dot{\bh} $ on $\Omega_{2,L}$, we define $\dot{\bh}^{\textnormal{av}}$ to be $\dot{\bh}^{\textnormal{av}}(\bsig) = \frac{1}{2} (\dot{ \bh} (\bsigma) + \dot{\bh}(\bsigma \oplus 1 )$.
Our goal is to show the following.
	\begin{lemma}
		\label{lem:2ndmo:Lipschitz:hdot:qdot}
		Fix $k\geq k_{0}$. Recall the BP fixed point $\dot{\bq}^\star_{L}= \dot{\bq}^\star_{\ula, L} \dot{q}^\star_{\la^1,L} \otimes \dot{q}^\star_{\la^2,L}$ in Proposition \ref{prop:BPcontraction:1stmo} and let $\dot{\bh}^\star_{L}\equiv \dot{\bh}_L[\dot{\bq}^\star_{L}]$. Then, there exists $\eps_L>0$ and a constant $C_k$, which may depend on $k$ but not on $L$, such that
		\begin{equation}\label{eq:lem:2ndmo:Lipschitz:hdot:qdot}
		\dot{\bh} = \dot{\bh}^{\textnormal{av}}, \,  ||\dot{\bh}-\dot{\bh}^\star_{L}||_1<\eps_L, \dot{\bh}\in \PPP(\dot{\Omega}_{2,L}) \implies ||\dot{\bq}_L[\dot{\bh}]-\dot{\bq}^\star_L||_1\leq C_k ||\dot{\bh}-\dot{\bh}^\star_L||_1.
		\end{equation}
	\end{lemma}
	
	The proof relies on that of Proposition \ref{prop:1stmo:Lipschitz:hdot:qdot}, but we need an extra argument to take care of the discrepency between the spins $\bb\bb^= $ versus $\bb\bb^{\neq}$. For the rest of this subsection, we write $\dot{\bq} = \dot{\bq}_L[\dot{\bh}]$, $\dot{\bq}^\star = \dot{\bq}^\star_L$ and $\dot{\bh}^\star = \dot{\bh}^\star_L$ for convenience. 
	Since we assume that $\dot{\bh} = \dot{\bh}^{\textnormal{av}}$, we have $\dot{\bq} = \dot{\bq}^{\textnormal{av}}$. Thus, from now on, we view $\dot{\bh}, \dot{\bh}^{\star}_L, \dot{\bq}$ and $\dot{\bq}^{\star}_L$ as probability measures on the projected color space \begin{equation*}
	\dot{\Omega}_{\textnormal{pj}, L}:= \Omega_{\pj}^{\fs} \sqcup \left\{ \dot{\Omega}_{2,L} \setminus \{\rr, \bb \}^2 \right\};\qquad \Omega_{\pj}^{\fs}:= \{\rr\rr^=, \rr\rr^{\neq}, \bb\bb^=, \bb\bb^{\neq}, \rr\bb^=, \rr\bb^{\neq}, \bb\rr^=,\bb\rr^{\neq} \}.
	\end{equation*}
	Moreover, we write $C>0$ to denote an absolute constant that does not depend on $k$, $L$.

	For a signed measure $\ba $ on $\Omega_{2,L}$, we define the $||\cdot ||_\ff$-norm as before, by
	\begin{equation*}
	||\ba||_\ff := \sum_{\bsigma \in \{\rr,\bb \}^2 }|\ba(\bsigma)| + \sum_{\bsigma \notin \{\rr,\bb \}^2} |\ba(\bsigma)| 2^k.
	\end{equation*}
	We also define 
	\begin{equation*}
	\dot{\bq}^\circ(\dot{\bsigma}) \equiv \frac{1}{Z^{\circ}_{\dot{\bh}}}\frac{\dot{\bh}(\dot{\bsigma})}{\dot{\bh}^\star(\dot{\bsigma})}\dot{\bq}^\star(\dot{\bsigma}),\dot{\bsigma}\in \dot{\Omega}_{\textnormal{pj}, L},\textnormal{  where  } Z_{\dot{\bh}}^\circ \equiv \sum_{\dot{\bsigma}\in \dot{\Omega}_{\pj, L}}\frac{\dot{\bh}(\dot{\bsigma})}{\dot{\bh}^\star(\dot{\bsigma})}\dot{\bq}^\star(\dot{\bsigma}).
	\end{equation*}
	
	Then, the proof of Lemma \ref{lem:2ndmo:Lipschitz:hdot:qdot} will be obtained from the following three steps.
	\begin{lemma}\label{lem:2ndmo:Lip:aux}
		Under the setting of Lemma \ref{lem:2ndmo:Lipschitz:hdot:qdot}, we have
		\begin{equation}\label{eq:2ndmo:Lip:aux}
		\begin{split}
		&||\dot{\bq}^\circ - \dot{\bq}^{\star} ||_\ff \le C_k || \dot{\bh} - \dot{\bh}^{\star} ||_1 ;\\
		&|| \dot{\bq}^\circ - \dot{\bq} ||_\ff \le \frac{k^2}{2^k} || \dot{\bq} - \dot{\bq}^\star||_\ff + Ck | \dot{\bq}(\bb\bb^=) - \dot{\bq}(\bb\bb^{\neq})|.
		\end{split}
		\end{equation}
	\end{lemma}
	
	\begin{lemma}\label{lem:2ndmo:Lip:aux2}
		Under the setting of Lemma \ref{lem:2ndmo:Lipschitz:hdot:qdot}, we have
		\begin{equation*}
		 C || \dot{\bh} - \dot{\bh}^{\star} || _1 \ge |\dot{\bq}(\bb\bb^=) - \dot{\bq}(\bb\bb^{\neq})| - \frac{k^2}{2^k} || \dot{\bq} - \dot{\bq}^\star||_\ff.
		\end{equation*}
	\end{lemma}

	\begin{proof}[Proof of Lemma \ref{lem:2ndmo:Lipschitz:hdot:qdot}]
		Summing the first two inequalities in Lemma \ref{lem:2ndmo:Lip:aux} gives
		\begin{equation*}
		C_k || \dot{\bh} -\dot{\bh}^\star||_1 \ge \left( 1- \frac{k^2}{2^k} \right) || \dot{\bq} - \dot{\bq}^\star||_\ff -Ck | \dot{\bq}(\bb\bb^=) - \dot{\bq}(\bb\bb^{\neq})|.
		\end{equation*}
		Thus, we conclude the proof by combining with Lemma \ref{lem:2ndmo:Lip:aux2}. 
	\end{proof}
	
	In the proof of Lemmas \ref{lem:2ndmo:Lip:aux} and \ref{lem:2ndmo:Lip:aux2}, we will assume that $\dot{\bq}$ is very close to $\dot{\bq}^\star$ as in the beginning of the proof of Proposition \ref{prop:1stmo:Lipschitz:hdot:qdot}. This is possible since the map $\dot{\bh} \mapsto \dot{\bq}$ is continuous \cite[Appendix C]{ssz22}. More specifically, we take $\eps_L >0$ small enough so that the following holds for all $||\dot{\bh} - \dot{\bh}^\star||_1 < \eps_L$:
	\begin{itemize}
		\item $\dot{\bq} \in \Gamma$ where $\Gamma$ is defined in \eqref{eq:BPcontraction:1stmo}. Hence, by Proposition \ref{prop:BPcontraction:1stmo}, $||\BP_2[\dot{\bq}] \dot{\bq}^\star ||_1 \lesssim \frac{k^2}{2^k} ||  \dot{\bq}-\dot{\bq}^\star ||_1 .$
		
		\item For $\bsigma \in \Omega_{\pj}^{\fs}$, 
		%we let $\sigma^1, \sigma^2$  represent each coordinate of $\bsigma$; for instance, if $\bsigma = \rr\bb^=$, then $\sigma^1 = \rr$, $\sigma^2 = \bb$. Then, we have for each $\bsigma \in \Omega_{\pj}^{\fs}$ 
		we have
		\begin{equation}\label{eq:2ndmo:Lip:assumption}
		\left|\dot{\bq}(\bsigma) - \dot{\bq}^\star(\bsigma)\right| \le \frac{C}{2^k}, \quad \left|\BP_2\dot{\bq}(\bsigma) - \dot{\bq}^\star(\bsigma)\right| \le \frac{C}{2^k}.
		\end{equation}
	\end{itemize}

	\begin{proof}[Proof of Lemma \ref{lem:2ndmo:Lip:aux}]
		We first remark that the first inequality in \eqref{eq:2ndmo:Lip:aux} follows analogously from the argument \eqref{eq:goal-1:Lipschitz:1ststep}--\eqref{eq:1stmo:estimate:Z:circ}. The second inequality follows similarly from the proof of \eqref{eq:goal-2:lem:1stmo:Lipschitz:hdot:qdot}, by estimating the quantity
\begin{equation}\label{eq:goal-2:lem:2ndmo:Lip}
\sup_{\dot{\bsigma}\in \dot{\Omega}_{\pj,L}, \dot{\btau}\in \dot{\Omega}_{\pj, L}}\bigg|\frac{\boldsymbol{\mu}_{\dot{\bq}}(\dot{\btau})}{\boldsymbol{\mu}_{\dot{\bq}^\star}(\dot{\btau})}-\frac{\boldsymbol{\mu}_{\dot{\bq}}(\dot{\bsigma})}{\boldsymbol{\mu}_{\dot{\bq}^\star}(\dot{\bsigma})}\bigg|,
\end{equation}		
where we defined
			\begin{equation}\label{eq:def:mu:qdot:2nd}
			\boldsymbol{\mu}_{\dot{\bq}}(\dot{\bsigma})\equiv \sum_{\bsig \in \Omega_{\pj, L}^k, \dot{\bsigma}_1=\dot{\bsigma}}\hat{\Phi}_2(\bsig)^\la \prod_{i=2}^{k-1}\dot{\bq}(\dot{\bsigma}_i) \textnormal{BP}_2\dot{\bq}(\dot{\bsigma}_k),\quad\textnormal{for}\quad \dot{\bsigma} \in \dot{\Omega}_{\pj, L}.
			\end{equation}
		It is not difficult to see that if $\dot{\bsigma}, \dot{\btau} \notin  \{\rr\rr^=, \rr\rr^{\neq} \}$, then
		\begin{equation}\label{eq:goal-2:2ndmo:aux1}
		\bigg|\frac{\boldsymbol{\mu}_{\dot{\bq}}(\dot{\btau})}{\boldsymbol{\mu}_{\dot{\bq}^\star}(\dot{\btau})}-\frac{\boldsymbol{\mu}_{\dot{\bq}}(\dot{\bsigma})}{\boldsymbol{\mu}_{\dot{\bq}^\star}(\dot{\bsigma})}\bigg| \le \frac{k^2}{2^k} || \dot{\bq} - \dot{\bq}^\star||_\ff,
		\end{equation}
		based on the same argument as that from Proposition \ref{prop:1stmo:Lipschitz:hdot:qdot}. The only difference is that in the pair model, we do not have the same cancellation property as \eqref{eq:1stmo:Lipschitz:f:expression} and the explanation below it. However, if $\dot{\bsigma}, \dot{\bsigma } \notin  \{\rr\rr^=, \rr\rr^{\neq} \}$, then except $O(k2^{-k})$ of the total contribution from the case $\btau^{-1}, \bsig^{-1} \in \{\bb\bb^=, \bb\bb^{\neq} \}^{k-1}$ gets cancelled out by the same argument, and hence we get the upper bound \eqref{eq:goal-2:2ndmo:aux1}.
		
		On the other hand, if $\dot{\bsigma}  $ or  $\dot{\btau} \in  \{\rr\rr^=, \rr\rr^{\neq} \}$, we can obtain that
			\begin{equation}\label{eq:goal-2:2ndmo:aux2}
			\bigg|\frac{\boldsymbol{\mu}_{\dot{\bq}}(\dot{\btau})}{\boldsymbol{\mu}_{\dot{\bq}^\star}(\dot{\btau})}-\frac{\boldsymbol{\mu}_{\dot{\bq}}(\dot{\bsigma})}{\boldsymbol{\mu}_{\dot{\bq}^\star}(\dot{\bsigma})}\bigg| \le \frac{k^2}{2^k} || \dot{\bq} - \dot{\bq}^\star||_\ff + Ck | \dot{\bq}(\bb\bb^=) - \dot{\bq}(\bb\bb^{\neq})|,
			\end{equation}
		by performing the same analysis as \eqref{eq:Lipschitz:finalgoal:firststep} and the analysis below, using \eqref{eq:Lipschitz:useful:ineq}. Note that  if $\btau = \rr\rr^=$, then $\underline{\btau} = (\rr\rr^=, (\bb\bb^=)^{k-1})$ is the only configuration that can contribute to \eqref{eq:def:mu:qdot:2nd}.
		\end{proof}
		
		The remaining goal is to establish Lemma \ref{lem:2ndmo:Lip:aux2}. For a collection of $k$ probability measures $\underline{\dot{\bq}} = ( \dot{\bq}_1 , \ldots, \dot{\bq}_k)$ on $\dot{\Omega}_{\pj, L}$, and $j,l\in[k]$, we define the probability measures $\dot{\bh}_j[\underline{\dot{\bq}}]$ on  $\dot{\Omega}_{\pj, L}$ and $\dot{\bh}_{j,l}[\underline{\dot{\bq}}]$ on  $\dot{\Omega}_{\pj, L}^2$ as
		\begin{equation}\label{eq:2ndmo:Lip:hdot:generalq}
		 \begin{split}
		&  \dot{\bh}_j(\dot{\bsigma}) = \dot{\bh}_j[\underline{\dot{\bq}}](\dot{\bsigma})
		  := 
		  \frac{\dot{\bq}_j(\dot{\bsigma})}{Z^j_{\underline{\dot{\bq}}}} \sum_{\bsig \in \Omega_{\pj, L}^k, \dot{\bsigma}_j=\dot{\bsigma}}\hat{\Phi}_2(\bsig)^\la \prod_{i\neq j}\dot{\bq}(\dot{\bsigma}_i) ,\quad\textnormal{for}\quad \dot{\bsigma} \in \dot{\Omega}_{\pj, L};\\
		 & \dot{\bh}_{j,l}[\underline{\dot{\bq}}](\dot{\bsigma}, \dot{\bsigma}')
	  := 
		  \frac{\dot{\bq}_j(\dot{\bsigma}) \dot{\bq}_l(\dot{\bsigma}')}{Z^{j,l}_{\underline{\dot{\bq}}}} \sum_{\substack{\bsig \in \Omega_{\pj, L}^k, \\ \dot{\bsigma}_j=\dot{\bsigma}, \dot{\bsigma}_l = \dot{\bsigma}' }}\hat{\Phi}_2(\bsig)^\la \prod_{i\neq j,l}\dot{\bq}(\dot{\bsigma}_i) ,\quad\textnormal{for}\quad \dot{\bsigma}, \dot{\bsigma}' \in \dot{\Omega}_{\pj, L},
		 \end{split}
		\end{equation}
		where $Z_{\underline{\dot{\bq}}}^j, Z_{\underline{\dot{\bq}}}^{j,l}$ are the normalizing constants. We compute how much $\dot{\bh}_j$ changes as we vary the input $\dot{\bq}_l$. For $\delta\in \mathbb{R}$, $j,l \in[k]$ and $\dot{\btau} \in \dot{\Omega}_{\pj,L}$, let $\dot{\bq}_l^{\dot{\btau}:\delta}  $ be the measure defined as $\dot{\bq}_l^{\dot{\btau}:\delta} (\dot{\bsigma}) = \dot{\bq}_l(\dot{\bsigma})$ for all $\dot{\bsigma} \neq \dot{\btau}$, and
		\begin{equation*}
		\dot{\bq}_l^{\dot{\btau}:\delta} (\btau) = \dot{\bq}_l(\dot{\btau}) + \delta.
		\end{equation*}
		We also write $\underline{\dot{\bq}}^{l, \dot{\btau}:\delta} = (\dot{\bq}_1,\ldots, \dot{\bq}_l^{\dot{\btau}:\delta} , \ldots, \dot{\bq}_k)$, that is, switching the $l$-th coordinate of $\underline{\bq}$ to $\dot{\bq}_l^{\dot{\btau}:\delta} $. Although $\dot{\bq}_l^{\dot{\btau}:\delta} $ is not a probability measure anymore, we can define $\dot{\bh}_j[\underline{\dot{\bq}}^{l, \dot{\btau}:\delta} ]$ the same as above. We define the derivative
		\begin{equation*}
		\frac{\partial \dot{\bh}_j(\dot{\bsigma})}{\partial \dot{\bq}_l(\dot{\btau} )} := \lim_{\delta \to 0} \frac{1}{\delta} \left[ \dot{\bh}_j[\underline{\dot{\bq}}^{l, \dot{\btau}:\delta} ] (\dot{\bsigma})  - \dot{\bh}_j[\underline{\dot{\bq}}](\dot{\bsigma})\right].
		\end{equation*}
		Then, we have the following estimates on the derivatives of $\dot{\bh}_j$.
		
		\begin{lemma}\label{lem:2ndmo:Lip:deriv}
			Let $\dot{\bq}_1, \ldots, \dot{\bq}_k$ be the $k$ probability measures as above, where each of them satisfies \eqref{eq:2ndmo:Lip:assumption}, and let $\dot{\bh}_1 = \dot{\bh}[\underline{\dot{\bq}}]$ as above. Then, for $l \neq 1$, we have for all $\dot{\bsigma},\dot{\btau} \in \dot{\Omega}_{\pj, L}$ that
			\begin{equation}\label{eq:2ndmo:Lip:deriv:diff}
				\frac{\partial \dot{\bh}_1(\dot{\bsigma})}{\partial \dot{\bq}_l(\dot{\btau} )} = \frac{\dot{\bh}_{1,j}(\dot{\bsigma}, \dot{\btau}) - \dot{\bh}_1(\dot{\bsigma}) \dot{\bh}_l(\dot{\btau}) }{\dot{\bq}_l(\dot{\btau})} = O(k 2^{-k}).
			\end{equation}
			When $l = 1$, we have
			\begin{equation}\label{eq:2ndmo:Lip:deriv:same}
				\frac{\partial \dot{\bh}_1(\dot{\bsigma})}{\partial \dot{\bq}_1(\dot{\btau} )} = \frac{\dot{\bh}_1(\dot{\bsigma}) \mathds{1}\{\dot{\bsigma} = \dot{\btau} \} }{\dot{\bq}_1(\dot{\btau}) } - \frac{ \dot{\bh}_1(\dot{\bsigma}) \dot{\bh}_1(\dot{\btau}) }{ \dot{\bq}_1(\dot{\btau}) }.
			\end{equation}
			In particular, we have
			\begin{equation}\label{eq:2ndmo:Lip:deriv:same:estim}
				\frac{\partial \dot{\bh}_1(\bb\bb^= )}{\partial \dot{\bq}_1(\dot{\btau} )} =
				\begin{cases}
				2 + O(k2^{-k}) &  \dot{\btau} = \bb\bb^=;\\
				-2 + O(k2^{-k}) & \dot{\btau} = \bb\bb^{\neq};\\
				O(1) & \dot{\btau} \notin \Omega_{\pj}^{\fs};\\
				O(k2^{-k}) & \textnormal{otherwise}.
				\end{cases}
			\end{equation}
		\end{lemma} 
		
		\begin{proof}
			The first identity of \eqref{eq:2ndmo:Lip:deriv:diff} and \eqref{eq:2ndmo:Lip:deriv:same} can directly be obtained from differentiating \eqref{eq:2ndmo:Lip:hdot:generalq}. For $j\neq 1$ and $\dot{\bsigma}, \dot{\btau} \notin \{\rr\rr^=, \rr\rr^{\neq} \}$, we note that 
			$(1-O(k2^{-k}))$ of the contribution to $\dot{\bh}_{1,l}(\dot{\bsigma}, \dot{\btau})$ comes from  $\dot{\bsig} \in \{ (\dot{\bsigma}, \dot{\btau} ) \} \times \{ \bb\bb^=, \bb\bb^{\neq} \}^{k-2}$, due to the assumption \eqref{eq:2ndmo:Lip:assumption}. Thus, in such a case, we have
			\begin{equation*}
			\dot{\bh}_{1,l}(\dot{\bsigma}, \dot{\btau}) = \left(1+ O(k2^{-k}) \right) \dot{\bh}_{1}(\dot{\bsigma}) \dot{\bh}_{l}( \dot{\btau}).
			\end{equation*}
			Due to the same reason it is straight-forward to see $\dot{\bh}_l(\dot{\btau}) \le C \dot{\bq}_l(\dot{\btau})$, and hence this gives the second identity of \eqref{eq:2ndmo:Lip:deriv:diff} for $\dot{\bsigma}, \dot{\btau} \notin \{\rr\rr^=, \rr\rr^{\neq} \}$. Establishing the identity for the case $\dot{\bsigma}$ or $\dot{\btau} \in  \{\rr\rr^=, \rr\rr^{\neq} \}$ is more straight-forward and we omit the details. \eqref{eq:2ndmo:Lip:deriv:same:estim} also follows from the same idea, applied to the formula \eqref{eq:2ndmo:Lip:deriv:same}. We note that 
			\begin{equation*}
			\dot{\bh}_1(\bb\bb^=) = \frac{1}{2} + O(k2^{-k}), \quad \dot{\bq}_1(\bb\bb^=) = \frac{1}{8} + O(2^{-k}) = \dot{\bq}_1(\bb\bb^{\neq}),
			\end{equation*}
			which gives the first two estimates of \eqref{eq:2ndmo:Lip:deriv:same:estim}. We leave the rest of the details to the interested reader.
		\end{proof}
		
		\begin{proof}[Proof of Lemma \ref{lem:2ndmo:Lip:aux2}]
			It suffices to show that
			\begin{equation*}
			C | \dot{\bh}(\bb\bb^=) - \dot{\bh}^\star(\bb\bb^=)| \ge | \dot{\bq}(\bb\bb^=) - \dot{\bq}(\bb\bb^{\neq})| - \frac{k^2}{2^k} || \dot{\bq} - \dot{\bq}^\star ||_\ff.
			\end{equation*}
			Observe that $\dot{\bh}$ can be written as follows using the above notation:
			\begin{equation*}
			\dot{\bh} = \dot{\bh}_1[ \dot{\bq}, \ldots, \dot{\bq}, \BP\dot{\bq}].
			\end{equation*} 
			Hence, from the derivatives of $\dot{\bh}$ and the mean value theorem,  we rewrite as
			\begin{equation}\label{eq:2ndmo:Lip:hdotdiff:expansion:deriv}
			  \dot{\bh}(\bb\bb^=) - \dot{\bh}^\star(\bb\bb^=)
			 =
			 \sum_{l=1}^k \sum_{\dot{\btau} \in\dot{\Omega}_{\pj,L} } \frac{\partial \dot{\bh}_1(\bb\bb^=)}{ \partial \dot{\bq}_j(\dot{\btau}) } ( \dot{\bq}_j(\dot{\btau}) - \dot{\bq}^\star(\dot{\btau}) ),
			\end{equation}
			where $\dot{\bq}_j = \dot{\bq}$ for $j \le k-1$ and $\dot{\bq}_k = \BP\dot{\bq}$. Here, when applying the mean value theorem, in principle we need to be precise on which point $\underline{\dot{\bq}}$ we evaluate the derivatives. However, each derivative has the same size scale for all $\underline{\dot{\bq}}$ with \eqref{eq:2ndmo:Lip:assumption} as given in Lemma \ref{lem:2ndmo:Lip:deriv}, we slightly abuse the notation as above.
			
			Then, we can estimate the \textsc{rhs} of \eqref{eq:2ndmo:Lip:hdotdiff:expansion:deriv} using the previous lemma, by
			\begin{equation*}
\begin{split}
			|\dot{\bh}(\bb\bb^=) - \dot{\bh}^\star(\bb\bb^{=}) | \ge& 2| (\dot{\bq}_1(\bb\bb^=) - \dot{\bq}^\star(\bb\bb^= )) - ( \dot{\bq}_1(\bb\bb^{\neq}) - \dot{\bq}^\star(\bb\bb^{\neq} ) ) | - \frac{k}{2^k} \sum_{j=2}^k || \dot{\bq}_j - \dot{\bq}^\star ||_1 - \frac{k}{2^k} || \dot{\bq}_1 -\dot{\bq}^\star ||_\ff \\
			\ge& 2 |\dot{\bq}(\bb\bb^=) -\dot{\bq}(\bb\bb^{\neq}) | - \frac{k^2}{2^k} || \dot{\bq} - \dot{\bq}^\star ||_\ff,
\end{split}
			\end{equation*}
			concluding the proof.
					\end{proof}

\section{The second moment in the correlated regime}\label{subsec:app:2ndmo:corr}

\iffalse	
\begin{defn}
	We define $\Gamma_2^{\textsf{int}}$ to be the collection of $(B,\{n_\uuu \} ) \in \Gamma_2^\star$ whose overlap is of intermediate size, namely, 
	$$|\zeta(B,\{n_\uuu\}) - 1/2| \in \left(\frac{k^2}{2^{k/2}}, \frac{1}{2}- 2^{-\frac{3k}{4} }\right).$$
	Moreover, $\Gamma_2^{\textsf{id}}$ denotes the collection of $(B,\{n_\uuu \} ) \in \Gamma_2^\star$ with near-identical overlap, namely,
	$$|\zeta(B,\{n_\uuu\}) - 1/2|  \geq \frac{1}{2}- 2^{-\frac{3k}{4} }.$$
\end{defn}
\fi

In this section, we provide the proof of Proposition \ref{prop:2ndmo:correlated overlap} and Lemma \ref{lem:identical regime decay estim}. Throughout the proof, note that for any $\lambda \in [0,1]$,
	\begin{equation*}
	\E \bZ_\lambda \le \E\bZ_1 = 2^n \left(1-2^{-k-1}\right)^m = \exp\left(O(n/2^k) \right).
	\end{equation*}
	Moreover, $\bZ_0$ corresponds to the total number of clusters (without the size restriction). We also define $\E \bZ^2_{0, \textnormal{id}}$ and $\E \bZ^2_{0, \textnormal{int}}$ analogously to \eqref{eq:def:corrN}. 
	
	We begin with establishing the first statement of Proposition \ref{prop:2ndmo:correlated overlap}.

\begin{proof}[Proof of Proposition \ref{prop:2ndmo:correlated overlap}, Part 1] Note that for any $\bs\in [0,\log2)^2$, we have $\bN_{\bs, \textnormal{int}}^2\leq \bZ_{0,\textnormal{int}}^2 $. Thus, it suffices to show $\E\bZ_{0,\textnormal{int}}^2\leq e^{-\Omega(nk^2 2^{-k})}$. From the proof of \cite[Proposition 1.1]{dss16}, we have
	\begin{equation*}
	\E \bZ_{0,\textnormal{int}}^2 \le  \sup_{2^{-\frac{3k}{4}} \le  \zeta \le 1- k^22^{-\frac{k}{2} }} \exp \left(n\left(\Phi + \bar{\textnormal{\textbf{a}}}(\zeta) + O(k2^{-k})\right) \right),
	\end{equation*}
	where $\Phi$ and $\bar{\textnormal{\textbf{a}}}(\zeta)$ be defined as in its proof (we use $\zeta$ instead of $\alpha$ in \cite{dss16}).  In particular, $\Phi = \Phi(d) := \log 2 + \frac{d}{k} \log(1-2^{-k-1}) = O(2^{-k})$. In the proof of \cite[Proposition 1.1]{dss16}, they showed that
	\begin{equation*}
	\sup\left\{ \bar{\textnormal{\textbf{a}}}(\zeta) -\Phi : 2^{-\frac{3k}{4}} \le  \zeta\le k^{-\frac{4}{5}} \right\} \lesssim -k 2^{-\frac{3k}{4}}.
	\end{equation*}
	Further, in the same proof, we have $ \bar{\textnormal{\textbf{a}}}''(\zeta)<-3$ on $\frac{\log^2 k}{k} \le \zeta \le 1-\frac{\log^2 k}{k}$ with $\bar{\textnormal{\textbf{a}}}(\frac{1}{2}) = \Phi$ gives that
	\begin{equation*}
	\sup\left\{ \bar{\textnormal{\textbf{a}}}(\zeta) -\Phi : k^22^{-\frac{k}{2}} \le | 2\zeta - 1 | \le 1-2^{-\frac{3k}{4}} \right\} \lesssim -k^2 2^{-k}.
	\end{equation*}
	This concludes the proof for the intermediate regime.
\end{proof}

Since the second statement of Proposition \ref{prop:2ndmo:correlated overlap} is a direct consequence of Lemma \ref{lem:identical regime decay estim}, we focus on establishing the latter. Moreover, the truncated model will follow the same proof as in the untruncated model, so we focus on the untruncated model. The conclusion will be obtained as a consequence of Lemmas 4.8 and 4.9 of \cite{dss16}. These two lemmas have shown that $\E \bZ_{0,\textnormal{id}}^2 \lesssim \E \bZ_0$. Our conclusion will follow by observing that the argument from those lemmas can be applied analogously to $\E \bN_{\us, \textnormal{id}}^2$. Without loss of generality, we work with the case where $\pi(\srr\srr^{\neq}) < \pi(\srr\srr^=)$.

In \cite{dss16}, \eqref{eq:2ndmo:identical:mainlem} (without restricting the sizes to be $\us_n$) followed from the estimate on the following type of quantity:
\begin{equation}\label{eq:2ndmo:identical:goal}
\P \left(\left. {\ux}^2 \textnormal{ is valid} \right| \, {\ux}^1 \textnormal{ is valid}\, \right).
\end{equation}
Note that the argument given in \cite{dss16} holds for any $\ux^1 \in \{0,1,\ff \}^V$, i.e., the frozen configuration of the first copy, as long as its number of free variables is equal to the prescribed amount. However, their proof relies on the uniform random matching of half-edges, which is not directly generalizable to our case: we also have a prescribed size $s^1$ of $\ux^1$ that prevents us from exploiting the randomness of a  uniform perfect matching.

Therefore, our approach is to not only condition on $\ux^1 \in \{0,1,\ff \}^V$ being a valid frozen configuration, but also prescribe the connections between the free variables so that its size is equal to $s^1$. This will be done in the same fashion as Proposition \ref{prop:1stmo:B nt decomp}. However, we get rid of the constraints on the size of the second copy; without loss of generality we can assume $\pi_n(\ff\srr) \le \pi_n(\srr\ff)$, and in such a case  it suffices to show 
\begin{equation*}
\E \bN_{s_n^1}^2[\pi_n] :=	\sum_{s_n^2}\E \bN_{(s_n^1, s_n^2)}^2[\pi_n] \le \widetilde{C} 2^{-k\Delta/10} \left(\E \bN_{s_n^1}[\pi_n^1] \right) + e^{-cn2^{-k/2}},
\end{equation*}
instead of \eqref{eq:2ndmo:identical:mainlem}. 
Following this plan, we fix the $\{\srr,\ff \}^2$-configuration $(\underline{\eta}^1,\underline{\eta}^2)\in\{\srr,\ff \}^{2V}$ of both copies, but without a restriction on the size of the second copy. 

For a formal discussion, we introduce the notion of \textit{marked free component} and \textit{marked profile} as follows.

\begin{defn}
	Let $\fff \in \FFF$ be a free component (Definition \ref{def:freecomp:basic}).  A \textbf{marked free component}  is a pair $\underline{\fff} := (\fff, \mm)$ where $\mm \in \{\srr, \ff \}^{V(\fff)}$ illustrates an additional $\{\srr,\ff \}$-labeling on the variables of $\fff$. $\mm \in \{\srr, \ff \}^{V(\fff)}$ describes the prescribed frozen configuration of the second copy on $\fff$ which is the free component of the first copy, where $\mm_v=\srr, v\in V(\fff)$ means that $v$ is a frozen variable in the second copy. We denote the space of marked free components by $\FFF_{\textnormal{m}}$, and let $\FFF_{\textnormal{m}}^{\tr}$ be its subspace of marked free components whose graphical structure is a tree. For a marked free component $\underline{\ttt} = (\ttt, \mm)\in \FFF_{\textnormal{m}}^{\tr}$, we denote its free component part (resp.~marking on the component) by $\fff(\underline{\ttt}):= \ttt$ (resp.~$\mm(\underline{\ttt}):= \mm$). Moreover, let $v_\srr(\underline{\ttt})$ and $v_\ff(\underline{\ttt})$ denote the number of variables in $\fff(\underline{\ttt})$ that are marked as $\srr$ and $\ff$ by $\mm(\underline{\ttt})$, respectively.
	
	Let $\pi$ be a probability measure on $\{\srr\srr^=, \srr\srr^{\neq}, \srr\ff,\ff\srr, \ff\ff \}$. A \textbf{marked profile} is a tuple $(\pi, \{n_{\underline{\ttt}} \}_{\FFF_{\textnormal{m}}^{\tr}} )$ satisfying the compatibility condition given by
	\begin{equation*}
	\begin{split}
% 	\pi_{\srr\rr^=} + \pi_{\srr\rr^{\neq}} + \pi_{\rr\ff} +\frac{1}{n} \sum_{\underline{\ttt}\in \FFF_{\textnormal{m}}^{\tr} } v(\ttt) n_{\underline{\ttt}}  = 1;\\
	\pi_{\ff\srr} = \frac{1}{n} \sum_{\underline{\ttt}} v_{\srr}(\underline{\ttt}) n_{\underline{\ttt}}; \quad \pi_{\ff\ff} = \frac{1}{n} \sum_{\underline{\ttt}} v_\ff(\underline{\ttt}) n_{\underline{\ttt}}.
	\end{split}
	\end{equation*}
\end{defn}

For a marked free component $\underline{\fff} = (\fff, \mm)$, let $\textsf{p}(\underline{\fff}) := \fff$. For a marked profile $(\pi, \{n_{\underline{\ttt}} \})$, we define
\begin{equation*}
\Delta_{\ttt} = \Delta_{\ttt}(\{n_{\underline{\ttt}} \}) := \sum_{\underline{\ttt}: \textsf{p}(\underline{\ttt}) = \ttt} n_{\underline{\ttt}} \mathds{1}\{ v_{\srr}(\underline{\ttt}) \ge 1 \},
\end{equation*}
that is, the number of free trees $\ttt$ in the first copy that have at least one frozen variable in their second copy. Then, it is clear that
\begin{equation*}
\Delta[\pi ] = n(\pi_{\srr\srr^{\neq}} + \pi_{\srr\ff} + \pi_{\ff\srr}) \le n(\pi_{\srr\srr^{\neq}} + \pi_{\srr\ff}) + \sum_{\ttt} v(\ttt) \Delta_\ttt =: \hat{\Delta}[\pi, \{n_{\underline{\ttt}} \}].
\end{equation*}
We will show a stronger version of Lemma \ref{lem:identical regime decay estim}, in terms of $\hat{\Delta}[\pi, \{n_{\underline{\ttt}} \} ]$ instead of $\Delta[\pi]$. To this end, we carry out our analysis under a fixed (first-copy) free tree profile $\{n_\ttt \}$ and the number of trees $\{\Delta_\ttt \}$ that contain a frozen second-copy variable. For $\pi$, we define $\pi_{\srr \bullet} := (\pi_{\srr\srr^=}, \pi_{\srr\srr^{\neq}},  \pi_{\srr\ff})$. For a given tuple $(\pi_{\srr \bullet}, \{n_\ttt \} , \{\Delta_\ttt \})$, we then have $\hat{\Delta} = \hat{\Delta}[\pi_{\srr \bullet}, \{\Delta_\ttt \}] $ given as above. For a marked profile $(\pi, \{n_{\underline{\ttt}} \})$, we write $(\pi, \{n_{\underline{\ttt}} \}) \sim (\pi_{\srr \bullet}, \{n_\ttt \}, \{\Delta_\ttt \})$ to indicate its compatibility with $\pi_{\srr \bullet}$ and $\{\Delta_\ttt \} = \{\Delta_{\ttt} (\{n_{\underline{\ttt}} \}) \}$.

\begin{lemma}\label{lem:identical:aux}
	Let $\{n_\ttt \}$ be an arbitrarily given free tree profile satisfying \eqref{eq:def:exp:decay:profile}, and let $\pi_{\srr \bullet}$ be a measure on $\{\srr\srr^=, \srr\srr^{\neq}, \srr\ff \}$ whose total mass is equal to $1- \frac{1}{n}\sum_{\ttt} n_\ttt$ and $\pi_{\srr\ff} \ge \frac{1}{n} \sum_{\ttt} n_\ttt$ (i.e., $\pi_{\srr \ff} \ge \pi_{\ff\srr}$). Moreover, let $\{\Delta_\ttt \}$ be an arbitrarily given tuple of numbers such that $\Delta_\ttt \le n_\ttt$ and $\hat{\Delta}= \hat{\Delta}[\pi_{\srr\bullet}, \{\Delta_\ttt \}] \le n/2^{k/2}$. Also, define $s_n^1 =\frac{1}{n} \sum_{\underline{\ttt}}  n_{\underline{\ttt}}\log w^{\lit}(\ttt)$. Then, we have
	\begin{equation*}
	\E \bN_{s_n^1}^2[\pi_{\srr \bullet},\{n_{{\ttt}} \}, \{\Delta_\ttt \}] \le \widetilde{C}2^{-ck\hat{\Delta}} \E \bN_{s_n^1}[\pi^1] + e^{-cn/2^{k/2}},
	\end{equation*}
	where $\widetilde{C}, c>0$ are absolute constants independent of $k$.
\end{lemma}

\begin{proof}
    	We begin with  introducing some notations to utilize the results of \cite{dss16}. Let $\underline{\eta}^1$ denote a $\{\srr,\ff \}$-configuration on $\GGG$, and for $1\le j\le k$ let $m\nu_j$ count the number of clauses adjacent to exactly $j$ $\underline{\eta}^1$-free variables. Moreover, let $m_{\textnormal{f}}$ denote the number of $\ux$-forcing clauses, and let $\gamma$ denote the fraction of frozen variables which are $\underline{\eta}^1$-forced at most $\sqrt{k}$ times. We define the event $\Omega_B$ as
	\begin{equation}\label{eq:def:identicalregime:Omegas}
	\begin{split}
	\Omega_B :=\left\{ \left| 1- \frac{m_{\textnormal{f}}}{mk2^{-k+1} } \right| \le 2^{-k/8} \right\} \bigcap \left\{ \gamma \le \frac{k^2}{2^{k/2}} \right\}.
	\end{split}
	\end{equation}
	Let $n_\ff$ denote the number of free variables.  From \cite[Lemma 4.8]{dss16}, we have for any $s$ that
	\begin{equation*}
	\E \left[ \bN_{s}[n_\ff = n\beta];\, (\Omega_B)^{\textsf{c}} \right] \le \E \left[\bZ_0[n_\ff = n\beta] ;\, (\Omega_B)^{\textsf{c}} \right] \le (\E \bN_{s} + 1) \exp \left(-5nk^2 2^{-k} \right).
	\end{equation*}
	The second inequality is due to $\E\bZ_0 \le \exp(O(n2^{-k}))$. Here, we remark that \cite[Lemma 4.8]{dss16} also have a similar bound for the event $\Omega_A:= \{1-\nu_0-\nu_1\leq k^3\beta^2\}$, which is used in the proof of \cite[Lemma 4.9]{dss16}. However, since we impose exponential decay on the free trees such bound for the event $\Omega_A$ is not necessary as seen below.
	
	Let $\underline{\omega} = (\underline{\eta}^1, \underline{\eta}^2) $ denote a pair-$\{\srr, \ff \}$ configuration, with the given empirical measure $\pi$. For $\omega \in \{\srr ,\ff \}^2$, write $V_\omega := \{v\in V: \omega_v = \omega \}$. We also decompose the event $\Omega_B$ into disjoint events $\Omega_{B, \bx}$, where $\bx := (\nu_0, \nu_1, m_{\textnormal{f}}, \gamma)$ is the tuple of quantities defined above with respect to $\ux^1$. Let $p$ denote the fraction of frozen variables in $\ux^1$, and let $\eps$ be the constant satisfying $p\eps = \pi(\srr\srr^{\neq}) + \pi(\srr\ff)$. Define $F_\delta$ to be the event that exactly $np\eps\delta$ variables in $V_{\srr\srr^{\neq}}\cup V_{\srr\ff}$ are $\underline{\eta}^1$-forced $\le \sqrt{k}$ times, and let $m_{\textnormal{for}}^1$ denote the number of $\underline{\eta}^1$-forcing clauses.
	
	Our goal is to investigate the probability \eqref{eq:2ndmo:identical:goal} in terms of the marked free profile. Namely, let $(\pi_{\srr \bullet}, \{n_{{\ttt}} \}, \{\Delta_\ttt\} )$  be given as the assumption, such that $s_n^1 =\frac{1}{n} \sum_{\underline{\ttt}}  n_{\underline{\ttt}}\log w^{\lit}(\ttt)$. Then, we consider the probability
	\begin{equation*}
	\P \left(\left. {\underline{\eta}}^2 \textnormal{ is valid} \right| \,  (\pi_{\srr \bullet}, \{n_{{\ttt}} \}, \{\Delta_\ttt\} ), \, \Omega_{B,\bx}, \, F_\delta, \, m_{\textnormal{for}}^1 \right),
	\end{equation*}
	where the probability is taken over uniform random matching of the half-edges, which is equivalent to studying the partition function $\E \bZ_{\ula}^{\tr}$ with $\ula = (\lambda, 0)$.
	
	Let $Q_\srr$ denote the event that each $\underline{\eta}^1$-forcing clause is incident to at least one other $\{\srr\srr^{\neq}, \srr\ff \}$-variable. Further, let $Q_\ff$ denote the event that for each free tree $\ttt$, there are exactly $\Delta_\ttt$ of them having at least one $\underline{\eta}^2$-forced variable in it. Note that $\{\underline{\eta}^2 \textnormal{ is valid } \} \subset Q_\srr \cap Q_\ff$. Our main goal is to control the probability of $Q_\ff$ conditioned on $\{\underline{\eta}^1 \textnormal{ valid}, \, (\pi_{\srr \bullet}, \{n_{{\ttt}} \}, \{\Delta_\ttt\} ), \, \Omega_{B,\bx}, \, Q_\srr, \, F_\delta, \, m_{\textnormal{for}}^1  \}.$ The probability for $Q_\srr $ can be estimated analogously to \cite[Lemma 4.9]{dss16}.
	
	The events $Q_\srr$, $\Omega_{B,\bx}$, $F_\delta$ are only related to the matching between $\underline{\eta}^1$-forced variables and $m_{\textnormal{for}}^1$ clauses that are $\underline{\eta}^1$-forcing, given these events the matching between the (variable-adjacent) boundary half-edges of the free trees and $m-m_{\textnormal{for}}^1$ clauses that are non-$\underline{\eta}^1$-forcing is made uniformly at random. Moreover, to have event $Q_\ff$, for each $\ttt $, there must be $\Delta_\ttt$ trees among $n_{\ttt}$ of them that have a $\ff\srr$ variable. Note that a free tree which has a $\ff\srr$ variable must be connected to at least one external clause which is $\underline{\eta}^2$-forcing. Such a clause is separating, but non-forcing, in $\underline{\eta}^1$, and forcing in $\underline{\eta}^2$. Thus, the clause must be adjacent to at least one $\srr\srr^{\neq}$-variable, because one of the adjacent rigid variables in $\underline{\eta}^1$ must flip, otherwise it would be non-forcing in $\underline{\eta}^2$. Hence, we bound the conditional probability of $Q_\ff$ as follows:
	\begin{itemize}
		\item For each free tree that has an $\ff\srr$-variable, we select one $\srr\srr^{\neq}$ variable and compute the probability that the two are connected by a clause.
		
		\item Such a clause must be $\underline{\eta}^2$-forcing too. The conditional probability to have a  literal assignment that forces the $\ff\srr$-variable given that it's valid is at most $2^{-k+2}$, since the probability of having valid literal assignments for separating clauses joining a free tree is at least $\frac{1}{2}$.
	\end{itemize}
	Thus, writing $\Delta_{\srr\srr^{\neq}}:= n\pi_{\srr\srr^{\neq}}$, the conditional probability of $Q_\ff$ satisfies
	\begin{equation*}
	\P(Q_\ff \, | \, \underline{\eta}^1 \textnormal{ valid}, \, (\pi_{\srr \bullet}, \{n_{{\ttt}} \}, \{\Delta_\ttt\} ), \, \Omega_{B,\bx}, \, Q_\srr, \, F_\delta, \, m_{\textnormal{for}}^1) 
	\le	\left( 				\prod_{\ttt} {n_\ttt \choose \Delta_\ttt} \right)  \prod_{\ttt} \left(\Delta_{\srr\srr^{\neq }} \frac{ v(\ttt) kd}{n} 2^{-k+2} \right)^{\Delta_\ttt} .
	\end{equation*}
	Abbreviating $\Delta_\ff := \sum_{\ttt} \Delta_\ttt$ and using the bound $n_\ttt \le n 2^{-ckv(\ttt)}$, the above is upper  bounded by
	\begin{equation}\label{eq:identicalreg:freetreedecay}
	\begin{split}
	&	\exp\left(\Delta_\ff \log \Delta_{\srr\srr^{\neq}} -\sum_{\ttt} ckv(\ttt)\Delta_\ttt + \sum_{\ttt} \Delta_\ttt \log \left(\frac{v(\ttt)k^2}{\Delta_\ttt} \right) + O(\Delta_\ff) \right) \\
	&=
	\exp \left( \sum_{\ttt} \Delta_\ttt \log \left( \frac{\Delta_{\srr\srr^{\neq}} v(\ttt)k^2}{2^{ckv(\ttt)/2} \Delta_\ttt } \right)  - \sum_{\ttt} \frac{c}{2}kv(\ttt)\Delta_\ttt \right) \le \exp\left(-\sum_{\ttt} \frac{ck}{2} v(\ttt)\Delta_\ttt  + O(\Delta_{\srr\srr^{\neq}}) \right),
	\end{split}
	\end{equation}
	where the last inequality followed by the fact that $x \log (\frac{a}{x}) \le \frac{a}{e}$. In particular, this holds for all $m_{\textnormal{for}}^1$ and hence we can remove the conditioning on $m_{\textnormal{for}}^1$.
	
	We combine this bound with the bound on the conditional probability of $Q_\srr$ derived in \cite{dss16}. We set $p $ to be the total mass of $\pi_{\srr\bullet}$, let $\gamma$ be as \eqref{eq:def:identicalregime:Omegas}, and let $np\varepsilon := n(\pi_{\srr\srr^{\neq}} + \pi_{\srr\ff})$. Define the constant $$\textbf{c}_\srr^{\pi_{\srr\bullet}, \bx, \delta}:= 2^{np\varepsilon} {np\gamma \choose np\varepsilon\delta} {np(1-\gamma) \choose np\varepsilon(1-\delta)}
	$$
	be the number of choices of placing $\{\srr\srr^{\neq}, \srr\ff \}$-variables, in such a way that  $np\varepsilon\delta$ of them are $\underline{\eta}^1$-forced at most $\sqrt{k}$ times from their adjacent clauses. Here, $2^{np\varepsilon}$ is an upper bound on the number of ways to assign either $\srr\srr^{\neq}$ or $\srr\ff$. Then, we have
	\begin{equation*}
 \begin{split}
	&\E[\bN_{s_n^1}^2[\pi_{\srr \bullet}, \{n_{\ttt} \}, \{\Delta_\ttt \} ] ]\\
 &\le
	\sum_{\bx}\E [ \bN_{s_n^1}[\pi^1, \{n_\ttt \}]; \, \Omega_{B,\bx}] \sum_{\delta:\, np\eps\delta\in \Z \cap [0,np\eps]}
	\textbf{c}_{\srr}^{\pi_{\srr\bullet}, \bx,\delta} \P\left(\left. Q_{\srr} \cap Q_{\ff} \right| \, (\pi_{\srr \bullet}, \{n_{{\ttt}} \}, \{\Delta_\ttt\} ), \, \Omega_{B,\bx}, \, F_\delta \right).
 \end{split}
	\end{equation*}
	Following the proof of \cite[Lemma 4.9]{dss16} and equation (40) therein gives that
	\begin{equation*}
	\sum_{\delta:\, np\eps\delta\in \Z \cap [0,np\eps]} \textbf{c}_\srr^{\pi_{\srr\bullet}, \bx, \delta} \P(Q_\srr |  \, (\pi_{\srr \bullet}, \{n_{{\ttt}} \}, \{\Delta_\ttt\} ), \, \Omega_{B,\bx}, \, F_\delta ) \le \exp\left(-ckn(\pi_{\srr\srr^{\neq}} + \pi_{\srr\ff}) \right),
	\end{equation*}
	where $c>0$ is an absolute constant. Thus, combining with \eqref{eq:identicalreg:freetreedecay} gives
	\begin{equation*}
	\E[\bN_{s_n^1}^2[\pi_{\srr \bullet}, \{n_{\ttt} \}, \{\Delta_\ttt \} ] ] \le 
	\E [ \bN_{s_n^1}[\pi^1, \{n_\ttt \}]] \exp(-ck\hat{\Delta}) + e^{-cn/2^{k/2}},
	\end{equation*}
	where the second term in the \textsc{rhs} is from the contributions of $\Omega_B^c$. This concludes the proof of the lemma.
\end{proof}

\begin{proof}[Proof of Lemma \ref{lem:identical regime decay estim}]
	The proof follows directly from Lemma \ref{lem:identical:aux} by summing over all $\{\Delta_\ttt \}$ that gives $\hat{\Delta}[\pi_{\srr \bullet}, \{\Delta_\ttt \}] = \hat{\Delta}$. For each $v>0$, there are at most $(Ck)^v$ distinct free trees of size $v(\ttt) = v$ for a universal constant $C< \infty$. Thus, to bound the total number of choices of $\{\Delta_\ttt \}$, we first count the number of solutions $\{b_v \}_{v=1}^{\hat{\Delta}}$ such that $\sum_{v=1}^{\hat{\Delta}} vb_v = \hat{\Delta}$, and for each $\{b_v \}_{v=1}^{\hat{\Delta}}$ the number of choices of $\{{\Delta}_\ttt \}$ satisfying $\sum_{\ttt: v(\ttt) = v} \Delta_\ttt = b_v$ is bounded by $\prod_{v=1}^{\hat{\Delta}} ((Ck)^v)^{b_v} = (Ck)^{\hat{\Delta}}$. The number of solutions $\{b_v \}$ can crudely be bounded by $\prod_{v=1}^{\hat{\Delta}} \left(\frac{\hat{\Delta}}{v} + 1 \right) = {2\hat{\Delta} \choose \hat{\Delta }} \le 4^{\hat{\Delta}}$. Hence, the total number of choices of $\{\Delta_\ttt \}$ is bounded by $(4Ck)^{\hat{\Delta}}$, which can be absorbed by the $e^{-ck\hat{\Delta}}$-decay.
\end{proof}

\end{document}